\documentclass{surv-l}
\usepackage{amssymb}
\usepackage{appendix}

\makeindex

\newcommand{\dom}{\operatorname{dom}}

\newcommand{\Int}{\operatorname{Int}}
\newcommand{\specrad}{r_{\rm spec}}

\newcommand{\nspace}[2]{\ensuremath{{\mathbb{#1}}^{#2}}}
\newcommand{\module}[1]{\ensuremath{\mathcal{#1}}}
\newcommand{\vecspace}[1]{\ensuremath{\mathcal{#1}}}
\newcommand{\ring}{\ensuremath{\mathcal{R}}}
\newcommand{\field}{\ensuremath{\mathbb{K}}}
\newcommand{\ncspace}[1]{\ensuremath{#1}_{{\rm nc}}}
\newcommand{\ncspacej}[2]{\ensuremath{#1}_{{#2},{\rm nc}}}
\newcommand{\ncspaced}[2]{(\ensuremath{#1}^{{#2}})_{\rm nc}}
\newcommand{\mat}[2]{\ensuremath{{#1}^{#2\times #2}}}
\newcommand{\mattuple}[3]{\ensuremath{\left({#1}^{#2\times #2}\right)^{#3}}}
\newcommand{\rmat}[3]{\ensuremath{{#1}}^{#2\times #3}}
\newcommand{\rmattuple}[4]{\ensuremath{\left({#1}^{#2\times #3}\right)^{#4}}}
\newcommand{\free}{\boldsymbol{\mathcal{G}}}

\newcommand{\brho}{\boldsymbol{\rho}}
\newcommand{\trans}{\top}
\newcommand{\mtrans}{\boldsymbol{\top}}
\newcommand{\cdotop}{\cdot_\text{op}}
\newcommand{\id}{\operatorname{id}}
\newcommand{\rad}{\operatorname{rad}}
\newcommand{\clos}{\operatorname{clos}}
\newcommand{\col}{\operatornamewithlimits{col}}
\newcommand{\row}{\operatornamewithlimits{row}}
\newcommand{\diag}{\operatornamewithlimits{diag}}
\newcommand{\tclass}[1]{\ensuremath{\mathcal{T}^{#1}}}
\newcommand{\tensa}[1]{\mathbf{T}(\ensuremath{#1})}
\newcommand{\spa}{\operatornamewithlimits{span}}
\newcommand{\cb}{\operatorname{\rm cb}}

\newcommand{\cG}{\vecspace{G}}

\newcommand{\cU}{\vecspace{U}}

\newcommand{\cY}{\vecspace{Y}}

\theoremstyle{plain}
\newtheorem{thm}{Theorem}[chapter]
\newtheorem{cor}[thm]{Corollary}
\newtheorem{lem}[thm]{Lemma}
\newtheorem{prop}[thm]{Proposition}

\theoremstyle{definition}

\theoremstyle{remark}
\newtheorem{rem}[thm]{Remark}

\newtheorem{ex}[thm]{Example}

\numberwithin{equation}{chapter}

\numberwithin{section}{chapter}

\begin{document}

\frontmatter

\title{Foundations of Free Noncommutative Function
Theory}

\author{Dmitry S.
Kaliuzhnyi-Verbovetskyi}
\address{Department of Mathematics \\
Drexel University\\
3141 Chestnut Str.\\
 Philadelphia, PA, 19104}
\email{dmitryk@math.drexel.edu}

\author{Victor Vinnikov}
\address{Department of Mathematics \\
Ben-Gurion University of the Negev\\
 Beer-Sheva, Israel, 84105}
\email{vinnikov@math.bgu.ac.il}

\thanks{This work was partially supported by the US--Israel
Binational Science foundation (BSF) grant 2010432. The first
author was partially supported by the US National Science
Foundation (NSF) grant DMS 0901628 and by the Center of Advanced
Studies in Mathematics of Ben-Gurion University of the Negev. The
second author was partially supported by the Israel Science
Foundation (ISF). Part of this work was carried out during the
stays at Banff International Research Station (BIRS) under the
program Research in Teams, and at Mathematisches
Forschungsinstitut Oberwolfach (MFO) under the program Research in
Pairs.}



\date{}

\maketitle

\tableofcontents

\mainmatter

\chapter{Introduction}\label{sec:intro}

The goal of this work is to develop, in a systematic way and in a
full natural generality, the foundations of a theory of functions
of free\footnote{We consider only the case of free noncommuting
variables, namely a free algebra or more generally the tensor
algebra of a module; we will therefore say simply
``noncommutative'' instead of ``free noncommutative''.}
noncommuting variables. This theory offers a unified treatment for
many free noncommutative objects appearing in various branches of
mathematics.

Analytic functions of $d$ noncommuting variables originate in the
pioneering work of J. L. Taylor on noncommutative spectral theory
\cite{T1,T2}. The underlying idea is that a function of $d$
noncommuting variables is a function on $d$-tuples of square
matrices of all sizes that respects simultaneous intertwinings (or
equivalently --- as we will show --- direct sums and simultaneous
similarities). Taylor showed that such functions admit a good
differential (more precisely, difference-differential) calculus,
all the way to the noncommutative counterpart of the classical
(Brook) Taylor formula. Of course a $d$-tuple of matrices (say
over ${\mathbb C}$) \index{${\mathbb C}$} is the same thing as a
matrix over ${\mathbb C}^d$, \index{$\mathbb{C}^d$} so we can view
a noncommutative function as defined on square matrices of all
sizes over a given vector space. This puts noncommutative function
theory in the framework of operator spaces \cite{ER,Pa,Pi}.
Also, noncommutative functions equipped with the difference-differential operator
form an infinitesimal bialgebra \cite{Rota,Agui}.
\footnote{
More precisely, to use our terminology,
we have to consider noncommutative functions with values in the noncommutative space over an algebra
with a directional noncommutative difference-differential operator as a comultiplication,
and it is a topological version of the bialgebra concept where the range of the comultiplication
is a completed tensor product.
See Section \ref{subsub:prod} for the Leibnitz rule, and Section \ref{subsec:integra}
for the coassociativity of the comultiplication.
We will not pursue the infinitesimal bialgrebra viewpoint explicitly.}
The theory has been pushed forward by Voiculescu \cite{Voi00,Voi04,Voi09},
with an eye towards applications in free probability
\cite{Voi85,Voi86,Voi95,VoiDyNi}.
We mention also the work of Hadwin \cite{Ha78} and Hadwin--Kaonga--Mathes \cite{HaKaMa03},
of Popescu \cite{Po06,Po10,Po12,Po13}, of Helton--Klep--McCullough \cite{HKMcCS,HKMcC1,HKMcC2,HKMcC3},
and of Muhly--Solel \cite{MS,MS3}.
The (already non-trivial) case of functions of a single noncommutative variable\footnote{
See Remark \ref{rem:ncfun_1var}.}
was considered by Schanuel \cite{Sch} (see also Schanuel--Zame \cite{SchZa})
and by Niemiec \cite{Nie}.

In a purely algebraic setting, polynomials and rational functions
in $d$ noncommuting indeterminates and their evaluations on
$d$-tuples of matrices of an arbitrary fixed size (over a
commutative ring $\ring$) \index{$\ring$} are central objects in
the theory of polynomial and rational identities; see, e.g.,
\cite{Row80,Form}.
A deep and detailed study of the ring of noncommutative polynomials and the skew field
of noncommutative rational functions has been pursued in the work of P. M. Cohn \cite{Co71,Co06}.
The noncommutative difference-differential operator in the setting of noncommutative polynomials
is well known as the universal derivation on the free algebra, see \cite{Lew,BerDicks,DicksLew}.

In systems and control, noncommutative rational functions and
formal power series appear naturally as recognizable formal power
series of the theory of automata and formal languages
\cite{Kle,Schutz61,Schutz62b,Fliess70,Fliess74a,Fliess74b,BR} and
as transfer functions of multidimensional systems with evolution
along the free monoid \cite{Cuntz2,BGM1,BGM2,BGM3,AK0,BK-V}. In
particular, transfer functions of conservative noncommutative
multidimensional systems are characterized as formal power series
whose values on a certain class of $d$-tuples of operators are
contractive (matrix evaluations actually suffice --- see
\cite{AK}).
Such classes of formal power series appear
as the noncommutative generalization of the classical Schur class
of contractive analytic functions on the unit disc \cite{AA,Helton-scat,BC}
in the operator model theory for row contractions and more general noncommuting operator tuples
\cite{Popescu-model1,Popescu-model2,Popescu-CLT1,Popescu-CLT2,Popescu-OTncdom}
and in the representation theory of the Cuntz algebra \cite{BJ,DP},
see also the
generalized Hardy algebras associated to a $W^*$-correspondence \cite{MS1,MS2,BBFtH}.

Coming from a different direction, it turns out that most
optimization problems appearing in systems and control are
dimension-independent, i.e., the natural variables are matrices,
and the problem involves rational expressions in these matrix
variables which have therefore the same form independent of matrix
sizes; see \cite{H03}. This leads to exploring such techniques as
Linear Matrix Inequalities (LMIs) --- see, e.g.,
\cite{NN,N06,SIG97} --- in the context of noncommutative convexity
and noncommutative real semialgebraic geometry, where one
considers polynomials and rational functions in $d$ noncommuting
indeterminates evaluated on $d$-tuples of matrices over
$\mathbb{R}$ \cite{HSOS,HMcCDeg2,HMcCV,HMcCPV,HKMcC-freecon,HMcC-last}.

A key feature of noncommutative functions that we establish in
this work is very strong analyticity under very mild assumptions.
In an algebraic setting, this means that a noncommutative function
which is polynomial in matrix entries when it is evaluated on
$n\times n$ matrices, $n=1,2,\ldots$, of bounded degree, is a
noncommutative polynomial. In an analytic setting, local
boundedness implies the existence of a convergent noncommutative
power series expansion.

Difference-differential calculus for noncommutative rational
functions, its relation to matrix evaluations, and applications
were considered in \cite{KVV2,KVV3}.

Recent papers \cite{PV,BPV1,BPV2} used the results of the present
work on noncommutative function theory to study noncommutative
infinite divisibility and limit theorems in operator-valued free
probability. For instance, in the scalar-valued case a measure is
free infinitely divisible if and only if its so-called
$R$-transform has positive imaginary part on the complex upper
half-plane (\cite{BeVoi93}, see also \cite{NS}
--- this is the free analogue of the classical Levi--Hin\c cin
Theorem). One of the main results of \cite{PV} is a similar
statement in the operator-valued case except that the
$R$-transform of an operator-valued distribution is a
noncommutative function.

In a recent paper \cite{AKV}, a general fixed point
theorem for noncommutative functions has been proved, and, in
particular, the corresponding variation of the Banach contraction
mapping theorem has been obtained. This result was applied then to
prove the existence and uniqueness theorem for ODEs in
noncommutative spaces. In addition, a noncommutative version of
the principle of nested closed sets has been established.

The recent papers \cite{AgMcC1,AgMcC2,AgY} made further progress in noncommutative function theory.
Specifically, \cite{AgMcC1} established an analogue of the realization theorem of \cite{AT} and \cite{BB}
for noncommutative functions on a domain defined by a matrix noncommutative polynomial
(the result was originally established for special cases in \cite{BGM2} in the framework of noncommutative power series,
see Section \ref{subsec:realize} below for a further discussion).
\cite{AgMcC1} used this realization result to establish noncommutative analogues of the Oka--Weil approximation theorem
and of the Carleson corona theorem.
\cite{AgMcC2} applied these ideas to a noncommutative version of the Nevanlinna--Pick interpolation problem,
and \cite{AgY} gave an application to symmetric functions of two noncommuting variables.

We proceed now to give some motivating examples of noncommutative
functions followed by their definition; we then discuss the
difference-differential calculus and present some of the main
results of the theory, and we finish the introduction with a detailed
overview.
We postpone a review of and a comparison to some of the earlier work on the subject
to the short chapter at the end of the book.

\section*{Acknowledgements} We wish to thank Mihai Putinar for
directing our attention to the work of J. L. Taylor, and
Shibananda Biswas for a careful reading of the manuscript and
valuable suggestions. It is also a pleasure to thank J. A. Ball,
S. Belinschi, M. Dritschel, J. W. Helton, I. Klep, S. McCullough,
P. Muhly, M. Popa, and B. Solel for fruitful discussions.
Last but not least, we would like to thank the three anonymous referees for their many helpful remarks.

\section{Noncommutative (nc) functions: examples and genesis}
\label{subsec:ncpoly}

Let $\ring$ be a unital commutative ring, and let $\ring\langle
x_1,\ldots,x_d\rangle$ \index{$\ring\langle
x_1,\ldots,x_d\rangle$} be the ring of nc polynomials (the free
associative algebra) over $\ring$. Here $x_1,\ldots,x_d$ are nc
indeterminates, and $f \in \ring\langle x_1,\ldots,x_d\rangle$ is
of the form
\begin{equation} \label{eq:ncpoly}
f = \sum_{w \in \free_d} f_w x^w,
\end{equation}
where $\free_d$ \index{$\free_d$} denotes the free monoid on
$d$ generators (letters) $g_1,\ldots,g_d$ with identity
$\emptyset$ (the {empty word}), \index{$\emptyset$} $f_w \in
\ring$, $x^w$ are nc monomials in $x_1,\ldots,x_d$
($x^w=x_{j_1}\cdots x_{j_m}$ \index{$x^w$} for $w=g_{j_1}\cdots
g_{j_m}\in\free_d$ and $x^\emptyset =1$), and the sum is finite.
$f$ can be evaluated in an obvious way on $d$-tuples of square
matrices of all sizes over $\ring$: for $X=(X_1,\ldots,X_d) \in
\mattuple{\ring}{n}{d}$, \index{$\mattuple{\ring}{n}{d}$}
\begin{equation} \label{eq:ncpolyeval}
f(X) =\sum_{w \in \free_d}f_w X^w = \sum_{w \in \free_d} X^w f_w
\in \mat{\ring}{n}. \index{$\mat{\ring}{n}$}
\end{equation}
\index{$X^w$}

We can also consider nc formal power series or nc rational
functions. The ring $\ring\langle\langle
x_1,\ldots,x_d\rangle\rangle$ \index{$\ring\langle\langle
x_1,\ldots,x_d\rangle\rangle$} of nc formal power series over
$\ring$ is the (formal) completion of the ring of nc polynomials;
$f \in \ring\langle\langle x_1,\ldots,x_d\rangle\rangle$ is of the
same form as in \eqref{eq:ncpoly}, except that the sum is in
general infinite. There are two settings in which we can define
the evaluation of $f$ on $d$-tuples of square matrices:
\begin{itemize}
\item Assume that $X=(X_1,\ldots,X_d) \in \mattuple{\ring}{n}{d}$
is a jointly nilpotent $d$-tuple, i.e., $X^w = 0$ for all $w \in
\free_d$ with $|w| \ge k$ for some $k$, where $|w|$ \index{$\vert
w\vert$} denotes the length of the word $w$; when $\ring=\field$
\index{$\field$} is a field, this simply means that $X$ is jointly
similar to a $d$-tuple of strictly upper triangular matrices. Then
we can define $f(X)$ as in \eqref{eq:ncpolyeval}, since the sum is
actually finite. \item Assume that $\ring=\field$ is the field of
real or complex numbers and that $f$ has a positive nc multiradius
of convergence, i.e., there exists a $d$-tuple
$\rho=(\rho_1,\ldots,\rho_d)$ of strictly positive numbers such
that
\begin{equation*}
\operatornamewithlimits{limsup}_{k \to \infty}
\sqrt[k]{\sum_{|w|=k} |f_w| \rho^w} \leq 1.
\end{equation*}
Then we can define $f(X)$ as in \eqref{eq:ncpolyeval}, where the
infinite series converges absolutely and uniformly on any \emph{nc
polydisc} \index{nc polydisc}
\begin{equation*}
\coprod_{n=1}^\infty \left\{X \in \left({\field}^{n \times
n}\right)^d \colon \|X_j\| < r_j,\ j=1,\ldots,d\right\}
\end{equation*}
of multiradius $r=(r_1,\ldots,r_d)$ with $r_j < \rho_j$,
$j=1,\ldots,d$.
\end{itemize}

The skew field of nc rational functions over a field $\field$ is
the universal skew field of fractions of the ring of nc
polynomials over $\field$. This involves some non-trivial details
since unlike the commutative case, a nc rational function does not
admit a canonical coprime fraction representation; see
\cite{Am66,Be70,Co71a,Co72} for some of the original
constructions, and \cite[Chapter 8]{Row80} and \cite{Co71,Co06}
for good expositions and background. The following quick
description follows \cite{KVV2,KVV3}, to which we refer for both
details and further references. We first define (scalar) nc
rational expressions by starting with nc polynomials and then
applying successive arithmetic operations --- addition,
multiplication, and inversion. A nc rational expression $r$ can be
evaluated on a $d$-tuple $X$ of $n \times n$ matrices in its {\em
domain of regularity}, $\dom{r}$, \index{$\dom{r}$} which is
defined as the set of all $d$-tuples of square matrices of all
sizes such that all the inverses involved in the calculation of
$r(X)$ exist. (We assume that $\dom{r} \neq \emptyset$, in other
words, when forming nc rational expressions we never invert an
expression that is nowhere invertible.) Two nc rational
expressions $r_1$ and $r_2$ are called {\em equivalent} if
$\dom{r_1} \cap \dom{r_2} \neq \emptyset$ and $r_1(Z) = r_2(Z)$
for all $d$-tuples $Z \in \dom{r_1} \cap \dom{r_2}$. We define a
{\em nc rational function} ${\mathfrak r}$ to be an equivalence
class of nc rational expressions; notice that it has a
well-defined evaluation on $\bigcup_{r \in {\mathfrak r}} \dom{r}$
(in fact, on a somewhat larger set called the extended domain of
regularity of ${\mathfrak r}$ \index{$\mathfrak{r}$}).

We notice that in all these cases the evaluation of a formal
algebraic object $f$ (a nc polynomial, formal power series, or
rational function) on $d$-tuples of matrices possesses two key
properties.
\begin{itemize}
\item $f$ \emph{respects direct sums}: \index{respecting direct
sums} $f(X \oplus Y) = f(X) \oplus f(Y)$, where
\begin{equation*}
X \oplus Y =  (X_1 \oplus Y_1,\ldots,X_d \oplus Y_d) = \left(
\begin{bmatrix} X_1 & 0 \\ 0 & Y_1 \end{bmatrix}, \ldots,
\begin{bmatrix} X_d & 0 \\ 0 & Y_d \end{bmatrix} \right)
\end{equation*}
(we assume here that $X=(X_1,\ldots,X_d)$, $Y=(Y_1,\ldots,Y_d)$
are such that $f(X)$, $f(Y)$ are both defined). \item $f$
\emph{respects simultaneous similarities}: \index{respecting
simultaneous similarities} $f(T X T^{-1}) = T f(X) T^{-1}$, where
\begin{equation*}
T X T^{-1} =  (T X_1 T^{-1}, \ldots, T X_d T^{-1})
\end{equation*}
(we assume here that $X=(X_1,\ldots,X_d)$ and $T$ are such that
$f(X)$ and $f(T X T^{-1})$ are both defined).
\end{itemize}

More generally, one can consider a $p \times q$ matrix nc
polynomial $f$ with coefficients $f_w \in \rmat{\field}{p}{q}$ and
with evaluation
\begin{equation} \label{eq:ncpolyeval-mv}
f(X) = \sum_{w \in \free_d} X^w \otimes f_w \in \mat{\field}{n}
\otimes \rmat{\field}{p}{q} \cong \rmat{\field}{np}{nq}
\end{equation}
for $X \in \mattuple{\field}{n}{d}$, and similarly for
 nc formal power series with matrix coefficients and matrix nc rational
functions.\footnote{Notice that, unlike in \cite{KVV2} and
\cite{KVV3}, we write the coefficients $f_w$ on the right.} It is
still true that the evaluation of $f$ on matrices respects direct
sums and simultaneous similarities (where for similarities we
replace $T f(X) T^{-1}$ by $(T \otimes I_p) f(X) (T \otimes
I_q)^{-1}$). In the case $\field={\mathbb C}$ or
$\field=\mathbb{R}$, one can also consider operator nc polynomials
and formal power series.

{\em Quasideterminants} \cite{GRet1,GRet2,GGRetW} and {\em nc
symmetric functions} \cite{GKLLRT} are important examples of nc
rational functions, whereas {\em formal Baker--Campbell--Hausdorff
series} \cite{Dyn} are an important example of nc formal power
series. Let us also mention here {\em nc continued fractions}
\cite{Wed}.

\section{NC sets, nc functions, and nc difference-differential calculus}
\label{subsec:ncfun}

Both for the sake of potential applications and for the sake of
developing the theory in its natural generality, it turns out that
the proper setting for the theory of nc functions is that of
matrices of all sizes over a given vector space or a given module.
In the special case when the module is $\ring^d$, $n \times n$
matrices over $\ring^d$ can be identified with $d$-tuples of $n
\times n$ matrices over $\ring$, and we recover nc functions of
$d$ variables, key examples of which appeared in Section
\ref{subsec:ncpoly}.

Let $\module{M}$ \index{$\module{M}$} be a module over a unital
commutative ring $\ring$; we call
\begin{equation*}
\ncspace{\module{M}} = \coprod_{n=1}^\infty \mat{\module{M}}{n}
\end{equation*}
\index{$\mat{\module{M}}{n}$}\index{$\ncspace{\module{M}}$}the
\emph{nc space over $\module{M}$}. \index{nc space} A subset
$\Omega \subseteq \ncspace{\module{M}} $ \index{$\Omega$} is
called a \emph{nc set} \index{nc set} if it is closed under direct
sums, i.e., we have
\begin{equation*}
X \oplus Y = \begin{bmatrix} X & 0 \\ 0 & Y \end{bmatrix} \in
\Omega_{n+m}
\end{equation*}
\index{$X\oplus Y$}for all $n, m \in {\mathbb N}$ and all $X \in
\Omega_n$, \index{$\Omega_n$} $Y \in \Omega_m$, where we denote
$\Omega_n = \Omega \cap \mat{\module{M}}{n}$. NC sets are the only
reasonable domains for nc functions, but additional conditions on
the domain are needed for the development of the nc
difference-differential calculus. Essentially, we need the domain
to be closed under formation of upper-triangular block matrices
with an arbitrary upper corner block, but this is too stringent a
requirement (e.g., this is false for nc polydiscs or nc balls ---
see Section \ref{subsubsec:os} below). The proper notion turns out
to be as follows: a nc set $\Omega \subseteq \ncspace{\module{M}}$
is called \emph{right admissible} \index{right admissible nc set}
\footnote{``Upper admissible'' could have been a more appropriate
terminology, however we stick with ``right admissible'' since it
is related to the right difference-differential operator, see
below. A similar comment applies to the analogous notion of ``left
admissible'' which could have been called ``lower admissible''.}
if for all $X \in \Omega_n$, $Y \in \Omega_m$ and all $Z \in
\rmat{\module{M}}{n}{m}$ there exists an invertible $r\in\ring$
such that
\begin{equation*}
\begin{bmatrix} X & r Z \\ 0 & Y \end{bmatrix} \in \Omega_{n+m}.
\end{equation*}

Our primary examples of right admissible nc sets are as follows:
\subsection{}\label{subsub:Nilp} The set $\Omega={\rm Nilp}(\module{M})$ \index{${\rm Nilp}(\module{M})$} of
nilpotent matrices over a module $\module{M}$. Here $X \in
\mat{\module{M}}{n}$ is called nilpotent if $X^{\odot k} = 0$
\index{$X^{\odot k}$} for some $k$, where $X^{\odot k}$ denotes
the power of $X$ as a matrix over the tensor algebra
\begin{equation*}
{\mathbf T}(\module{M}) = \bigoplus_{j=0}^\infty
\module{M}^{\otimes j}
\end{equation*}
\index{${\mathbf T}(\module{M})$}of $\module{M}$ (this is often
called the ``faux'' product in operator space theory, when
$\module{M}=\vecspace{V}$ is an operator space); in the case where
$\ring=\field$ is a field, this means that there exists an
invertible $T \in \mat{\field}{n}$ such that $T X T^{-1}$ is
strictly upper triangular.
\subsection{}\label{subsub:queen} Assume that $\vecspace{V}$ is a Banach space (so
$\field={\mathbb C}$ or $\field=\mathbb{R}$) and that so are the
spaces $\mat{\vecspace{V}}{n}$, $n=2,3,\ldots$. We usually need
the topologies on $\mat{\vecspace{V}}{n}$ to be compatible in the
following sense: we require that the corresponding system of
matrix norms $\|\cdot\|_n$ \index{$\Vert\cdot\Vert_n$} is
\emph{admissible}, \index{admissible system of matrix norms} i.e.,
for every $n,m\in\mathbb{N}$ there exist $C_1(n,m)$, $C_1'(n,m)>0$
    such that for all $X\in\mat{\vecspace{V}}{n}$ and
    $Y\in\mat{\vecspace{V}}{m}$,
    \begin{multline}\label{eq:intr-dirsums-norms}
C_1(n,m)^{-1}\max\{\|X\|_n,\| Y\|_m\}\le\| X\oplus Y\|_{n+m}\\
\le
C_1'(n,m)\max\{\|X\|_n,\| Y\|_m\},
    \end{multline}
and
   for every $n\in\mathbb{N}$ there exists $C_2(n)>0$ such that for all
    $X\in\mat{\vecspace{V}}{n}$ and
    $S,T\in\mat{\field}{n}$,
    \begin{equation}\label{eq:intr-simprod-norms}
\| SXT\|_{n}\le C_2(n)\|S\|\,\|X\|_n\|T\|,
    \end{equation}
   where $\|\cdot\|$ denotes the operator norm of
   $\mat{\field}{n}$ with respect to the standard Euclidean
   norm of $\field^n$. If $\Omega\subseteq\ncspace{\vecspace{V}}$ is open in the
sense that $\Omega_n \subseteq {\mathcal V}^{n \times n}$ is open
for all $n$, then $\Omega$ is right admissible.
\subsection{} \label{subsubsec:os}
Assume that ${\mathcal V}$ is an \emph{operator space},
\index{operator space} see, e.g., \cite{ER,Pa,Pi,Ru}. Recall that
this means (by Ruan's Theorem) that there exists a system of norms
$\|\cdot\|_n$ on ${\mathcal V}^{n \times n}$, $n=1,2,\ldots$,
satisfying
\begin{align}\label{eq:os-dirsum-norms}
& \|X \oplus Y\|_{n+m} = \max\{\|X\|_n,\|Y\|_m\}
\quad\text{for all } X \in {\mathcal V}^{n \times n},\, Y \in {\mathcal V}^{m \times m},\\
\intertext{and} & \|T X S\|_n \leq \|T\| \|X\|_n \|S\|
\quad\text{for all } X \in {\mathcal V}^{n \times n},\, T,S \in
{\mathbb C}^{n \times n}. \label{eq:os-simprod-norms}
\end{align}
Clearly, this system of norms is admissible, with the constants in
\eqref{eq:intr-dirsums-norms} and \eqref{eq:intr-simprod-norms}
satisfying
$$C_1(n,m)=C_1'(n,m)=C_2(n)=1,\quad n,m\in\mathbb{N}.$$  For $Y\in\mat{\vecspace{V}}{s}$ and
$r>0$, define a \emph{nc ball centered at $Y$ of radius $r$}
\index{nc ball} as
\begin{equation*}
B_{\mathrm{nc}}(Y,r)=\coprod_{m=1}^\infty
B\Big(\bigoplus_{\alpha=1}^mY,r\Big) =\coprod_{m=1}^\infty\Big\{
X\in\mat{\vecspace{V}}{sm}\colon \Big\|
X-\bigoplus_{\alpha=1}^mY\Big\|_{sm}<r\Big\}.
\end{equation*}
\index{$B_{\mathrm{nc}}(Y,r)$}NC balls form a basis for a
topology on $\ncspace{\vecspace{V}}$, that we call the
\emph{uniformly-open topology},\index{uniformly-open topology}
which is weaker than the disjoint union topology of Section
\ref{subsub:queen}. In particular, every uniformly-open nc set is
right admissible.

Let $\module{M}$ and $\module{N}$ be modules over a unital
commutative ring $\ring$, and let $\Omega \subseteq
\ncspace{\module{M}}$ be a nc set. A function $f \colon \Omega \to
\ncspace{\module{N}}$ with $f(\Omega_n) \subseteq
\mat{\module{N}}{n}$ is called a \emph{nc function} \index{nc
function} if:
\begin{itemize}
\item $f$ \emph{respects direct sums}: \index{respecting direct
sums} $f(X \oplus Y) = f(X) \oplus f(Y)$ for all $X \in \Omega_n$,
$Y\in \Omega_m$. \item $f$ \emph{respects similarities}:
\index{respecting similarities} $f(T X T^{-1}) = T f(X) T^{-1}$
for all $X \in \Omega_n$ and invertible $T \in \mat{\ring}{n}$
such that $T X T^{-1} \in \Omega_n$.
\end{itemize}
It turns out that these two conditions are equivalent to a single
one: $f$ \emph{respects intertwinings},\index{respecting
intertwinings} namely if $X S = S Y$ then $f(X) S = S f(Y)$, where
$X \in \Omega_n$, $Y \in \Omega_m$, and $S \in
\rmat{\ring}{n}{m}$. This condition originates in the pioneering
work of Taylor \cite{T2}. We denote the module of nc functions on
$\Omega$ with values in $\ncspace{\module{N}}$ by
$\tclass{}(\Omega;\ncspace{\module{N}})$.

The main idea behind the nc difference-differential calculus is to
evaluate a nc function on block upper triangular matrices. Let $f$
be a nc function on a right admissible nc set $\Omega \subseteq
\ncspace{\module{M}}$ with values in $\ncspace{\module{N}}$. Let
$X\in\Omega_n$, $Y\in\Omega_m$, and $Z\in\rmat{\module{M}}{n}{m}$,
and let $r\in\ring$ be invertible and such that
{\small $\left[\begin{matrix} X & r Z\\
0 & Y
\end{matrix}\right]\in\Omega_{n+m}$.} Then it turns out that
\begin{equation*}
f\left(\begin{bmatrix} X & r Z\\
0 & Y
\end{bmatrix}\right)=\begin{bmatrix} f(X) & \ \Delta_R f(X,Y)(rZ)\\
0 & f(Y)
\end{bmatrix},
\end{equation*}
where the mapping $Z\mapsto\Delta_Rf(X,Y)(Z)$
\index{$\Delta_Rf(X,Y)(Z)$} from $\rmat{\module{M}}{n}{m}$ to
$\rmat{\module{N}}{n}{m}$ is $\ring$-linear.

We will call $\Delta=\Delta_R$ \index{$\Delta_R$} the \emph{right
nc difference-differential operator}.\index{right nc
difference-differential operator} (The left nc
difference-differential operator $\Delta_L$ \index{$\Delta_L$}
\index{left nc difference-differential operator} can be defined
analogously via evaluations on block lower triangular matrices.)
Its main property is that for all $n,m\in\mathbb{N}$,
$X\in\Omega_n$, $Y\in\Omega_m$, and $S\in\rmat{\ring}{m}{n}$ one
has
\begin{equation}\label{eq:gen-fin-dif}
Sf(X)-f(Y)S = \Delta f(Y,X)(SX-YS).
\end{equation}
In particular, when $n=m$ and $S=I_n$, we have the following
formula of  finite differences:
\begin{equation}\label{eq:fin-dif}
f(X)-f(Y) = \Delta f(Y,X)(X-Y)\quad (=\Delta f(X,Y)(X-Y)).
\end{equation}
Thus, the linear mapping $\Delta f (Y,Y)(\cdot)$ plays the role of
a nc  differential.

Let $\ring=\field$ be a field of real  or complex numbers. Setting
$X=Y+tZ$ (with $t\in\field$), we obtain from \eqref{eq:fin-dif}
that
$$f(Y+tZ)-f(Y)=t\Delta f(Y,Y+tZ)(Z).$$
Under appropriate continuity conditions, it follows that $\Delta
f(Y,Y)(Z)$ is the directional derivative of $f$ at $Y$ in the
direction $Z$.

In the case of $\module{M}=\ring^d$, \eqref{eq:fin-dif} turns into
\begin{equation}\label{eq:elagr}
f(X)-f(Y)=\sum_{i=1}^d\Delta_i f(Y,X)(X_i - Y_i),\quad
X,Y\in\Omega_n,
\end{equation}
where $\Delta_i f(Y,X)(C)=\Delta f(Y,X)(0,\ldots,0,C,0,\ldots,0)$,
and $C\in\mat{\ring}{n}$ is at the $i$-th position. The linear
mapping $\Delta_i f(Y,Y)(\cdot)$ plays the role of an $i$-th
partial nc differential at the point $Y$.

For $\module{M}_0$, $\module{M}_1$, $\module{N}_0$, $\module{N}_1$
modules over a unital commutative ring $\ring$, and
$\Omega^{(0)}\subseteq \ncspacej{\module{M}}{0}$,
$\Omega^{(1)}\subseteq \ncspacej{\module{M}}{1}$ nc sets, we
define a \emph{nc function of order 1} to be a function $f$ on
$\Omega^{(0)}\times\Omega^{(1)}$ so that for $X^0 \in
\Omega^{(0)}_{n_0}$ and $X^1 \in \Omega^{(1)}_{n_1}$, $f(X^0,X^1)
\colon \rmat{\module{N}_1}{n_0}{n_1} \to
\rmat{\module{N}_0}{n_0}{n_1}$ is a linear mapping, and so that
$f$ respects, in a natural way, direct sums and similarities in
each argument. A typical nc function of order 1 is $f(X,Y)(Z) =
f_0(X) (Z f_1(Y))$, where $f_0 \in
\tclass{}(\Omega^{(0)};\ncspacej{\module{N}}{0})$ and $f_1 \in
\tclass{}(\Omega^{(1)};\ncspacej{\module{N}^*}{1})$. We denote the
class of nc functions of order 1 by
$\tclass{1}(\Omega^{(0)},\Omega^{(1)};\ncspacej{\module{N}}{0},\ncspacej{\module{N}}{1})$.

It turns out that for $f \in
\tclass{}(\Omega;\ncspace{\module{N}})$
\index{$\mathcal{T}(\Omega;\ncspace{\module{N}})$} one has $\Delta
f \in
\tclass{1}(\Omega,\Omega;\ncspace{\module{N}},\ncspace{\module{M}})$.
More generally, one can define \emph{nc functions of order $k$},
\index{nc function of order $k$}
$$\tclass{k}(\Omega^{(0)},\ldots,\Omega^{(k)};\ncspacej{\module{N}}{0},\ldots,\ncspacej{\module{N}}{k}),$$
\index{$\tclass{k}(\Omega^{(0)},\ldots,\Omega^{(k)};\ncspacej{\module{N}}{0},\ldots,\ncspacej{\module{N}}{k})$}where
$\Omega^{(0)}\subseteq\ncspacej{\module{M}}{0}$, \ldots,
$\Omega^{(k)}\subseteq\ncspacej{\module{M}}{k}$, to be functions
of $k+1$ arguments in
$\Omega^{(0)}_{n_0},\ldots,\Omega^{(k)}_{n_k}$ whose values are
$k$-linear mappings
$$
\rmat{\module{N}_1}{n_0}{n_1} \times \cdots\times
\rmat{\module{N}_k}{n_{k-1}}{n_k} \longrightarrow
\rmat{\module{N}_0}{n_0}{n_k},
$$
and that respect, in a natural way, direct sums and similarities
in each argument. There is a difference-differential operator
$\Delta \colon \tclass{k} \to \tclass{k+1}$, so that by iteration
we obtain for $f \in \tclass{}(\Omega;\ncspace{\module{N}})$ that
$\Delta^\ell f \in
\tclass{\ell}(\Omega,\ldots,\Omega;\ncspace{\module{M}},\ncspace{\module{N}},\ldots,\ncspace{\module{N}})$.
(We view usual nc functions as nc functions of order 0:
$\tclass{}(\Omega;\ncspace{\module{N}})=\tclass{0}(\Omega;\ncspace{\module{N}})$.)
Rather than using iterations, we can also calculate $\Delta^\ell
f(X^0,\ldots,X^\ell)(Z^1,\ldots,Z^\ell)$ directly, by evaluating
$f$ on a block matrix having $X^0,\ldots,X^\ell$ on the main
diagonal, $Z^1,\ldots,Z^\ell$ just above the main diagonal, and
all the other block entries zero, and taking the $(1,\ell+1)$-th
block entry.

The difference-differential operators $\Delta$ (and $\Delta_i$)
are obviously linear. They also satisfy a version of the chain
rule (for a composition of nc functions), and --- in the case when
a product operation is defined --- a version of the Leibnitz rule.

Using these higher order difference-differential operators,
iterating the first order finite difference formula
\eqref{eq:fin-dif}, and using the fact that direct sums are
respected, leads to the following nc counterpart of the classical
Taylor formula (see, e.g., \cite{Sh} for the commutative
multivariable version), that we call the \emph{Taylor--Taylor
formula} (or simply the \emph{TT formula}),\index{Taylor--Taylor
(TT) formula} in honor of Brook Taylor and of Joseph L. Taylor.
Let $f\colon \Omega \to\ncspace{\module{N}}$ be a nc function on a
right admissible nc set $\Omega\subseteq\ncspace{\module{M}}$, let
$s \in {\mathbb N}$, and let $Y \in \Omega_s$; then for all
$m\in\mathbb{N}$, $X \in \Omega_{ms}$, and  $N=0,1,\ldots$ one has
\begin{multline}\label{eq:tt}
f(X) = \sum_{\ell=0}^N \Big(X - \bigoplus_{\alpha=1}^m
Y\Big)^{\odot_s\,\ell} \Delta^\ell f(\underset{\ell+1\ {\rm times}}{\underbrace{Y,\ldots,Y}}) \\
+ \Big(X - \bigoplus_{\alpha=1}^m Y\Big)^{\odot_s\,N+1}
\Delta^{N+1}f(\underset{N+1\ {\rm
times}}{\underbrace{Y,\ldots,Y}},X).
\end{multline}
Here $(X - \bigoplus_{\alpha=1}^m Y)^{\odot_s\,\ell}$
\index{$X^{\odot_s\,\ell}$} denotes the $\ell$-th power of $X -
\bigoplus_{\alpha=1}^m Y$ viewed as a $m \times m$ matrix over the
tensor algebra ${\mathbf T}(\mat{\module{M}}{s})$; this power is a
$m \times m$ matrix over $\mattuple{\module{M}}{s}{\otimes \ell}$,
to which the
 mapping $\Delta^\ell f(Y,\ldots,Y)$ viewed as a linear mapping on
$\mattuple{\module{M}}{s}{\otimes \ell}$ is applied entrywise,
yielding a $m \times m$ matrix over $\mat{\module{N}}{s}$, i.e.,
an element of $\mat{\module{N}}{ms}$. The remainder term has a
similar meaning.

The TT formula \eqref{eq:tt} makes it natural to consider the
infinite TT series of $f$ around $Y \in \Omega_s$,
\begin{equation} \label{eq:ttseries}
\sum_{\ell=0}^\infty \Big(X - \bigoplus_{\alpha=1}^m
Y\Big)^{\odot_s\,\ell} \Delta^\ell f(Y,\ldots,Y).
\end{equation}
In the case $\module{M}=\nspace{\ring}{d}$ and $Y=\mu \in
\nspace{\field}{d}$ is a scalar point (so $s=1$),
\eqref{eq:ttseries} can be rewritten as an ordinary nc power
series using the $d$-tuple of partial nc difference-differential
operators $\Delta=(\Delta_1,\ldots,\Delta_d)$ (with an obvious
abuse of notation),
\begin{equation} \label{eq:ttseries_free}
\sum_{w \in \free_d} (X - \mu I_m)^w \Delta^{w^\top}\!
f(\mu,\ldots,\mu),
\end{equation}
where $\Delta^{w^\top}=\Delta_{i_\ell}\cdots\Delta_{i_1}$ for a
word $w=g_{i_1}\cdots g_{i_\ell} \in \free_d$. There is a version
of \eqref{eq:ttseries_free} for a general matrix center $Y$:
\begin{equation} \label{eq:matr-ttseries_free}
\sum_{w \in \free_d} \Big(X -
\bigoplus_{\alpha=1}^mY\Big)^{\odot_sw} \Delta^{w^\top}\!
f(Y,\ldots,Y).
\end{equation}
\index{$X^{\odot_s\,w}$}

 For
$f\in\tclass{}(\Omega;\ncspace{\module{N}})$ and $Y\in\Omega_s$,
the sequence of $\ell$-linear mappings $f_\ell:=\Delta^\ell
f(Y,\ldots,Y)\colon
\left(\mat{\module{M}}{s}\right)^\ell\to\mat{\module{N}}{s}$,
$\ell=0,1,\ldots$, satisfies the following conditions:
\begin{equation}\label{eq:LAC_0}
Sf_0-f_0S=f_1(SY-YS),
\end{equation}
and for $\ell=1,\ldots$,
\begin{equation}\label{eq:LAC-ell_0}
Sf_\ell(Z^1,\ldots,Z^\ell)-f_\ell(SZ^1,Z^2,\ldots,Z^\ell)=f_{\ell+1}(SY-YS,Z^1,\ldots,Z^\ell),
\end{equation}
\begin{multline}\label{eq:LAC_ell_j}
f_\ell(Z^1,\ldots,Z^{j-1},Z^jS,Z^{j+1},\ldots,Z^\ell)-f_\ell(Z^1,\ldots,Z^j,SZ^{j+1},Z^{j+2},\ldots,Z^\ell)\\
=f_{\ell+1}(Z^1,\ldots,Z^j,SY-YS,Z^{j+1},\ldots,Z^\ell),
\end{multline}
\begin{equation}\label{eq:LAC_ell_ell}
f_\ell(Z^1,\ldots,Z^{\ell-1},Z^\ell S)-f_\ell(Z^1,\ldots,Z^\ell)S
=f_{\ell+1}(Z^1,\ldots,Z^\ell,SY-YS),
\end{equation}
for every $S\in\mat{\ring}{s}$. (Notice, that in the case $s=1$,
conditions \eqref{eq:LAC_0}--\eqref{eq:LAC_ell_ell} are trivial.)
Conversely, given a sequence of $\ell$-linear mappings
$f_\ell\colon
\left(\mat{\module{M}}{s}\right)^\ell\to\mat{\module{N}}{s}$,
$\ell=0,1,\ldots$, satisfying conditions
\eqref{eq:LAC_0}--\eqref{eq:LAC_ell_ell}, the sum of the series
\begin{equation} \label{eq:series}
f(X)=\sum_{\ell=0}^\infty \Big(X - \bigoplus_{\alpha=1}^m
Y\Big)^{\odot_s\,\ell} f_\ell,
\end{equation}
whenever it makes sense, defines a nc function. Similar direct and
inverse statement can also  be  formulated in the case where
$\module{M}=\ring^d$ for the  series along $\free_d$.

\section{Applications of the Taylor--Taylor formula}
\label{subsec:appl-tt}

The Taylor--Taylor formula is the main tool for the study of local
behavior of nc functions. Here are some sample results.

\subsection{} Any nc function on ${\rm Nilp}_d(\ring):={\rm
Nilp}(\ring^d)$ \index{${\rm Nilp}_d(\ring)$} (the set of jointly
nilpotent $d$-tuples of matrices over a unital commutative ring
$\ring$) is given by a nc
 power series
$$f(X)=\sum_{w\in\free_d}X^w\Delta^{w^\top}f(0,\ldots,0),$$
where the sum is finite.
 More generally, let $Y\in\mattuple{\ring}{s}{d}$. Denote by ${\rm
Nilp}_d(\ring,Y)$ \index{${\rm Nilp}_d(\ring,Y)$} the set of
$d$-tuples of $sm\times sm$ matrices over $\ring$, $m=1,2,\ldots$,
such that $X-\bigoplus_{\alpha=1}^mY\in {\rm Nilp}_d(\ring)$. Then
any nc function on ${\rm Nilp}_d(\ring,Y)$ can be written as in
\eqref{eq:matr-ttseries_free} where, again, the sum is finite.

\subsection{} \label{subsubsec:poly} Let $\field$ be
an infinite field, and let $f$ be a nc function on
$\ncspaced{\field}{d}$ such that for every $n$ each matrix entry
of $f(X_1,\ldots,X_d)$ is a polynomial in matrix entries of
$X_1,\ldots,X_d$ of a uniformly (in $n$) bounded degree. Then $f$
is a nc polynomial.

The condition of uniform boundedness of the degrees is necessary.
However, we can get a sharp result without this condition as well:
namely, $f$ belongs to the completion of the ring of nc
polynomials with respect to the decreasing sequence of ideals
$\left\{\boldsymbol{\mathcal I}_n\right\}_{n=1}^\infty$, where
$\boldsymbol{\mathcal I}_n$ denotes the ideal of identities
\cite{Row80} for $n \times n$ matrices (over $\field$ and in $d$
variables).

\subsection{} Let $\field={\mathbb C}$, let $\vecspace{V}$ and $\vecspace{W}$
be Banach spaces equipped with admissible systems of matrix norms
over $\vecspace{V}$ and over $\vecspace{W}$ --- see Section
\ref{subsub:queen}, and let $\Omega \subseteq
\ncspace{\vecspace{V}}$ be an open nc set. If a nc function $f$ on
$\Omega$ is locally bounded, then it is analytic, i.e.,
$f|_{\Omega_n}$ is an analytic function on $\Omega_n \subseteq
\mat{\vecspace{V}}{n}$ with values in $\mat{\vecspace{W}}{n}$
(see, e.g., \cite{HiPh,Mu} for the general theory of analytic
functions between Banach spaces), and the TT series
\eqref{eq:ttseries} converges to $f$ locally uniformly. More
precisely, if $f$ is bounded on $B(\bigoplus_{\alpha=1}^m
Y,\delta)$, then the series \eqref{eq:ttseries} converges
absolutely and uniformly on $B(\bigoplus_{\alpha=1}^m Y,r)$ for
all $r<\delta$. In the case $\vecspace{V}=\nspace{{\mathbb C}}{d}$
and $\vecspace{W}={\mathbb C}$, analyticity simply means that for
every $n$ each matrix entry of $f(X_1,\ldots,X_d)$ is an analytic
function of the matrix entries of $X_1,\ldots,X_d$; in this case
it is enough to require that for every $n$, $f|_{\Omega_n}$ is
locally bounded on slices, that is, for any fixed $d$-tuples of
matrices $X=(X_1,\ldots,X_d)\in\Omega_s$ and
$Z=(Z_1,\ldots,Z_d)\in\mattuple{\mathbb{C}}{s}{d}$, $f(X+tZ)$ is
bounded for all $t\colon |t|<\epsilon$, with some $\epsilon>0$. In
fact, if a nc function $f\colon\Omega\to\ncspace{\vecspace{W}}$
(for a general vector space $\vecspace{V}$ and a Banach space
$\vecspace{W}$ as before) is locally bounded on slices, then $f$
is G\^{a}teaux differentiable and its TT series converges.

\subsection{} Let $\vecspace{V}$ and $\vecspace{W}$ be operator spaces, and let
$\Omega \subseteq \ncspace{\vecspace{V}}$ be a uniformly-open nc
set (see Section \ref{subsubsec:os}). If a nc function $f$ on
$\Omega$ is locally bounded in the uniformly-open topology, then
the TT series \eqref{eq:ttseries} converges to $f$ locally
uniformly in that topology. More precisely, if $f$ is bounded on
$B_{\mathrm{nc}}(Y,\delta)$, then the series converges absolutely
and uniformly on $B_{\mathrm{nc}}(Y,r)$ for all $r<\delta$. $f$ is
called a uniformly analytic nc function.

In the case where $\vecspace{V}=\nspace{{\mathbb C}}{d}$ with some
operator space structure given by a sequence of norms
$\|\cdot\|_n$, it is further true that the series
\eqref{eq:ttseries_free}
--- without grouping together terms of the same degree into
homogeneous nc polynomials
 --- converges absolutely and uniformly on somewhat smaller
sets, namely on every \emph{open nc diamond about $Y$}, \index{nc
diamond}
\begin{equation*}
\diamondsuit_{\mathrm{nc}}(Y,r):=\coprod_{m=1}^\infty\Big\{X\in
\Omega_{sm}\colon
\sum_{j=1}^d\|e_j\|_1\,\Big\|X_j-\bigoplus_{\alpha=1}^mY_j\Big\|<r\Big\}
\end{equation*}
\index{$\diamondsuit_{\mathrm{nc}}(Y,r)$}with $r<\delta$; here $e_1,\ldots,e_d$ denote the standard basis
for $\nspace{{\mathbb C}}{d}$.

\subsection{} \label{subsec:realize}
TT series expansions allow us to reformulate some results that
were originally established for nc power series as results about
nc functions. This includes, in particular, realizations of nc
power series as transfer functions of  noncommutative
multidimensional systems (systems with evolution along the free
monoid $\free_d$\footnote{
Referred to often as systems with evolution along the free semigroup.}). These systems were first studied in
\cite{Cuntz2} in the conservative infinite-dimensional setting, in
the context of operator model theory for row contractions due
mostly to Popescu
\cite{Popescu-model1,Popescu-model2,Popescu-CLT1,Popescu-CLT2} and
of representation theory of the Cuntz algebra \cite{BJ,DP}. A
comprehensive study of nc realization theory, in both
finite-dimensional and infinite-dimensional setting, appears in
\cite{BGM1,BGM2,BGM3}; these papers give a unified framework of
structured nc multidimensional linear systems for different kinds
of realization formulae. We also mention the paper \cite{BK-V}
where an even more general class of nc systems (given though in a
frequency domain) was described and the corresponding dilation
theory was developed.

We give (a somewhat imprecise version of) the main conservative
realization theorem from \cite{BGM2} restated as a theorem about
nc functions. Let $Q_1,\ldots,Q_d$ be $p \times q$ complex
matrices with a certain additional structure that comes from a
bipartite graph. Let $\vecspace{U}$ and $\vecspace{Y}$ be Hilbert
spaces, and let $\mathcal{L}(\vecspace{U},\vecspace{Y})$
\index{$\mathcal{L}(\vecspace{U},\vecspace{Y})$} be the space of
bounded linear operators from $\vecspace{U}$ to $\vecspace{Y}$.
The corresponding {\em nc Schur--Agler class} consists of
contraction-valued $f \in \tclass{}(B_{\rm
nc}(0,1);\ncspace{\mathcal{L}(\vecspace{U},\vecspace{Y})})$, where
$B_{\rm nc}(0,1)\subseteq\ncspaced{\mathbb{C}}{d}$, and
$\mathbb{C}^d$ is equipped with the operator space structure
defined by the system of matrix norms
$$\|X\|_n=\|Q(X)\|_{\mathcal{L}(\mathbb{C}^n\otimes\mathbb{C}^q,\mathbb{C}^n\otimes\mathbb{C}^p)},$$
where $Q(X):=X_1 \otimes Q_1 + \cdots + X_d \otimes Q_d$. Then the
following are equivalent:
\begin{enumerate}
\item $f$ is in the nc Schur--Agler class. \item There exists an
{\em Agler decomposition}
$$
I_{\mathbb{C}^n\otimes \vecspace{Y}} - f(X) f(Y)^* = H(X)
(I_{\mathbb{C}^n\otimes\mathbb{C}^p\otimes\vecspace{G}} - (Q(X)
\otimes I_{\vecspace{G}})(Q(Y)^* \otimes I_{\vecspace{G}})) H(Y)^*
$$
for an auxiliary Hilbert space $\vecspace{G}$ and some $H \in
\tclass{}(B_{\rm
nc}(0,1);\ncspace{\mathcal{L}(\mathbb{C}^p\otimes\vecspace{G},\vecspace{Y})})$
which is bounded on every nc ball $B_{\rm nc}(0,r)$ of radius
$r<1$. \item There exists a conservative realization: $$f(X) =
I_n\otimes D + (I_n\otimes
C)(I_{\mathbb{C}^n\otimes\mathbb{C}^p\otimes\vecspace{G}}-Q(X)(I_n\otimes
A))^{-1} Q(X) (I_n\otimes B).$$ Here the colligation matrix
$$
\begin{bmatrix} A & B \\ C & D \end{bmatrix} \colon
\begin{bmatrix} \nspace{{\mathbb C}}{p} \otimes \cG \\ \cU \end{bmatrix} \to
\begin{bmatrix} \nspace{{\mathbb C}}{q} \otimes \cG \\ \cY \end{bmatrix}
$$
is unitary.
\end{enumerate}
(The statement in \cite{BGM2} is considerably more precise in
identifying the state spaces $\nspace{{\mathbb C}}{p} \otimes \cG$
and $\nspace{{\mathbb C}}{q} \otimes \cG$ of the unitary
colligation in terms of the bipartite graph generating the
matrices $Q_1,\ldots,Q_d$.) This is the nc version of the
commutative multidimensional realization theorem of \cite{AT} and
\cite{BB}
 for the linear nc function $Q$ defining the domain of
$f$, which in turn generalizes Agler's seminal result on the
polydisc \cite{Ag}, see also \cite{BT}. Such generalized
transfer-function realizations occur also in the context of the
generalized Hardy algebras of Muhly--Solel (elements of which are
viewed as functions on the unit ball of representations of the
algebra---see \cite{MS1,MS2, BBFtH}) and of the generalized
Schur--Agler class associated with an admissible class of test
functions (see \cite{DMMcC, DMcC}).

We note that the formulation of the realization theorem for the nc
Schur--Agler class given above suggests generalizations to an
arbitrary finite- or infinite-dimensional operator space (instead
of $\mathbb{C}^d$ equipped with the system of norms associated
with the linear nc function $Q$) and to more general nc domains (a
nc counterpart of the commutative domains of \cite{AT} and
\cite{BB}, where the defining nc function $Q$ is a nc polynomial).
In the finite-dimensional case such a generalization has been carried out recently in \cite{AgMcC1};
we refer to both this paper and \cite{AgMcC2,AgY} for applications.

\section{An overview}\label{subsec:road-map}
In Chapter \ref{sec:ncfun}, we introduce nc functions and their
difference-differential calculus. In Section
\ref{subsec:ncfundef}, we define nc spaces, nc sets, and nc
functions, and prove that respecting direct sums and similarities
is equivalent to respecting intertwinings (Proposition
\ref{prop:simint}). In Section \ref{subsec:difdif}, we introduce
the right and left difference-differential operators $\Delta_R$
and $\Delta_L$ via evaluation of nc functions on block upper
triangular matrices and then taking the upper right (left) corner
block of the value, and show that $\Delta_Rf(X,Y)(Z)$ and
$\Delta_Lf(X,Y)(Z)$ are linear in $Z$ (Propositions
\ref{prop:homog} and \ref{prop:add}). We proceed to formulate
basic nc calculus rules in Section \ref{subsec:calc1} and to
establish the first-order finite difference formulae in Section
\ref{subsec:Lagrange}. Then we prove in Section
\ref{subsec:proper} that $\Delta_Rf(X,Y)(Z)$ and
$\Delta_Lf(X,Y)(Z)$ respect direct sums and similarities in $X$
and in $Y$ (in the terminology of Chapter \ref{sec:difdif_k},
$\Delta_Rf$ and $\Delta_Lf$ are nc functions of order 1). We
conclude Chapter \ref{sec:ncfun} by discussing in Section
\ref{subsec:dir_difdif} directional and (in the special case
$\module{M}=\ring^d$) partial nc difference-differential
operators.

In Chapter \ref{sec:difdif_k}, we first introduce higher order nc
functions, and then extend the difference-differential operators
to these functions, which leads to a higher order calculus for
ordinary and higher order nc functions. In Section
\ref{subsec:higher}, we define nc functions of order $k$ and
introduce the classes $\tclass{k}$ of those functions (so that the
original nc functions belong to the class $\tclass{0}$). The
values of nc functions of order $k$ are $k$-linear mappings of
matrices over modules, which can also be interpreted as elements
of tensor products of $k+1$ matrix modules (Remarks
\ref{rem:natur_map} and \ref{rem:tensor_values}). This allows us
to define a natural mapping
$\tclass{k}\otimes\tclass{\ell}\to\tclass{k+\ell+1}$,
$k,\ell=0,1,\ldots$, and consequently --- a mapping from the
tensor product of $k+1$ classes $\tclass{0}$ (perhaps over
different operator spaces) to $\tclass{k}$ (Remark
\ref{rem:tensor_prod}) and interpret a higher order nc function in
terms of nc functions of order $0$. We then extend in Section
\ref{subsec:difdif-k} the right difference-differential operator
$\Delta_R$ to a mapping of classes $\tclass{k}\to\tclass{k+1}$,
$k=0,1,\ldots$. In this definition,  the last argument of
$f(X^0,\ldots,X^k)$, where $f$ is a nc function of order $k$, is
taken in the form of a block matrix
$X^k=\begin{bmatrix} X^{k\prime} & Z\\
0 & X^{k\prime\prime}\end{bmatrix}$, and then we take the upper
right corner block of the matrix value. The extended operator
$\Delta_R$ is also linear as a function of $Z$ (Propositions
\ref{prop:homog-k} and \ref{prop:add-k}). Starting from this
point, we concentrate on the right version of calculus, since the
left version can be developed analogously or simply obtained from
the right one by symmetry.

The iterated operator
$\Delta_R^\ell\colon\tclass{k}\to\tclass{k+\ell}$ can also be
computed directly, by evaluating a nc function on block upper
bidiagonal matrices and then taking the upper right corner block
of its matrix value (Therems \ref{thm:bidiag} and
\ref{thm:bidiag_k}). The definition of the extended operator
$\Delta_R$ becomes more transparent when we interpret higher order
nc functions using tensor products of nc functions of order $0$
(Remark \ref{rem:delta_tensor_prod}).

We can also take $X^j$ in $f(X^0,\ldots,X^k)$ in the block upper
triangular form, which leads to operators
${}_j\Delta_R\colon\tclass{k}\to\tclass{k+1}$, $j=0,1,\ldots$ (so
that $\Delta_R={}_k\Delta_R$), that also extend the original
operator $\Delta_R\colon\tclass{0}\to\tclass{1}$ (Remark
\ref{rem:j-delta}). In Section \ref{subsec:Lagrange_k}, we
establish the first order difference formulae for higher order nc
functions (Theorem \ref{thm:Lagrange_k}) and its generalized
version  for the case of matrices $X$ and $Y$ where we evaluate
the higher order nc function in these formulae, of possibly
different size (Theorem \ref{thm:gen-Lagrange_k}). Similar
formulae are valid for operators ${}_j\Delta_R$ (Remark
\ref{rem:j-Lagrange_k}).

In Section \ref{subsec:integra}, we show that for a nc function
$f\in\tclass{k}$ which is $k$ times integrable, i.e., can be
represented as $\Delta_R^kg$ with $g\in\tclass{0}$, one has
${}_i\Delta_R\,{}_j\Delta_Rf={}_j\Delta_R\,{}_i\Delta_Rf$, for all
$i,j=0,\ldots,k$. Finally, we discuss in Section
\ref{subsec:dir_difdif-k} higher order directional nc
difference-differential operators; in particular, when the
underlying module is $\ring^d$, we discuss higher order partial
difference-differential operators $\Delta_R^w$, $w\in\free_d$.

In Chapter \ref{sec:TT}, we establish the Taylor--Taylor formula
(Theorem \ref{thm:TT} or more general Theorem
\ref{thm:tt-power-gen}; in the case $\module{M}=\ring^d$
--- Corollary \ref{cor:part_TT} and, in tensor product
interpretation of $\Delta_R^{w^\top}f$, Theorem
\ref{thm:tt-power}). We show in Remark \ref{rem:lost_abbey} that
the coefficients in the TT formula, the multilinear mappings
$f_\ell=\Delta_Rf(Y,\ldots,Y)$, $\ell=0,1,\ldots$, satisfy
conditions \eqref{eq:LAC_0}--\eqref{eq:LAC_ell_ell} that we
mentioned in Section \ref{subsec:ncfun}. The counterpart of these
condition for the sequence $f_w=\Delta_R^{w^\top}f(Y,\ldots,Y)$,
$w\in\free_d$, in the case $\module{M}=\ring^d$ is established in
Remark \ref{rem:lost_abbey_semigr}.

As a first application of the TT formula, we describe in Chapter
\ref{sec:nilp} nc functions on the set ${\rm Nilp}(\module{M})$
(or, more generally, on ${\rm Nilp}(\module{M};Y)$) of nilpotent
matrices (or nilpotent matrices about $Y$) over a module
$\module{M}$ --- see Section \ref{subsub:Nilp} for the definition.
In Section \ref{subsec: from_ncfun_to_ncps} we show that every
such a nc function is a sum of its TT series (Theorems
\ref{thm:nc_nilp-gen}, \ref{thm:nc_nilp}, \ref{thm:nc_nilp-gen-Y},
and \ref{thm:nc_nilp-Y}). Then we show that the coefficients of a
nc power series representing the nc function are uniquely
determined. In fact, this series is unique and is equal to the
corresponding TT series (Theorems \ref{thm:ncps-nilp-gen-unique}
and \ref{thm:ncps-nilp-unique}).
In Section
\ref{subsec:from_ncps_to_ncfun} we show that, conversely, given a
sequence of multilinear mappings
$f_\ell\colon\mattuple{\module{M}}{s}{\ell}\to\mat{\module{N}}{s}$,
$\ell=0,1,\ldots$, satisfying conditions
\eqref{eq:LAC_0}--\eqref{eq:LAC_ell_ell}, the sum of the nc power
series $\sum_{\ell=0}^\infty
(X-\bigoplus_{\alpha=1}^mY)^{\odot_s\ell}f_\ell$  is a nc function
on ${\rm Nilp}(\module{M};Y)$ with values in
$\ncspace{\module{N}}$ (Theorem \ref{thm:sufficiency_of_LAC}), and
a version of this result, Theorem
\ref{thm:sufficiency_of_LAC-semigr}, for the case
$\module{M}=\ring^d$ and a sequence of multilinear mappings
$f_w\colon\mattuple{\module{M}}{s}{\ell}\to\mat{\module{N}}{s}$,
$w\in\free_d$, satisfying the conditions analogous to
\eqref{eq:LAC_0}--\eqref{eq:LAC_ell_ell}. These results are
algebraic in their nature, since the sum of the corresponding
series is finite at every point of ${\rm Nilp}(\module{M};Y)$.

In Chapter \ref{sec:alg}, we give some other algebraic
applications of the TT formula. We prove that a nc function on
$\ncspace{(\field^d)}$, with $\field$  an infinite field, which is
polynomial in matrix entries when evaluated on $n\times n$
matrices, $n=1,2,\ldots$, of bounded degree, is necessarily a nc
polynomial (Theorem \ref{thm:king-poly}). The degree boundedness
condition is essential and cannot be omitted (Example
\ref{ex:degrees_unbdd}). However, if we do not require this
condition, then the following is true: an arbitrary nc function on
$\ncspace{(\field^d)}$, which is polynomial of degree $M_n$ in
matrix entries when evaluated on $n\times n$ matrices, is a sum of
a nc polynomial and an infinite series of homogeneous polynomials
$f_j$ vanishing on all $d$-tuples of $n\times n$ matrices for
$j>M_n$
 (Theorem \ref{thm:queen-poly}). We also obtain a
generalization of Theorem \ref{thm:king-poly} to nc functions on
(much more general) nc spaces, which are polynomial on slices
(Theorem \ref{thm:king-poly-gen}).

In Chapter \ref{sec:conv} we study analytic nc functions and the
convergence of their TT series. We consider three different
topologies on a nc space $\ncspace{\vecspace{V}}$ (over a complex
vector space $\vecspace{V}$), or three types of convergence of nc
power series, and, correspondingly, three types of analyticity of
nc functions: finitely open topology and analyticity on slices,
norm topology and analyticity on $n\times n$ matrices over
$\vecspace{V}$ for every $n=1,2,\ldots$, and uniformly-open
topology and uniform (in $n$) analyticity. The main feature of
analytic nc functions established in Chapter \ref{sec:conv} is
that local boundedness implies analyticity, in each of the three
settings. In the first part of Section \ref{subsec:analytic}, we
introduce the finitely open topology on $\ncspace{\vecspace{V}}$
and prove that a nc function which is locally bounded on slices is
analytic on slices and discuss the convergence of the
corresponding TT series (Theorem \ref{thm:g-queen}). In the second
part of Section \ref{subsec:analytic}, we show that a nc function
which is locally bounded (with respect to the norm topology on
$\mat{\vecspace{V}}{n}$, for every $n$) is analytic, with the TT
series convergent uniformly and absolutely on certain open
complete circular nc sets about the center $Y\in\Omega_s$ (Theorem
\ref{thm:f-queen}) or on open balls centered at $Y$ (Corollary
\ref{cor:balls}). We also establish the uniqueness of the
convergent nc power series expansions, in both of the above
topologies (Theorem \ref{thm:ncps-unique}). In Section
\ref{subsec:unif-open-top}, we introduce and study the
uniformly-open topology on a nc space $\ncspace{\vecspace{V}}$
over an operator space $\vecspace{V}$. In Section
\ref{subsec:u-analytic}, we show that a uniformly locally bounded
nc function is uniformly analytic, and its TT series converges
uniformly and absolutely on certain uniformly-open complete
circular nc sets about $Y$ (Theorem \ref{thm:king-conv}) or on
certain uniformly-open matrix circular nc sets (Theorem
\ref{thm:king-conv-nc}) or on uniformly-open nc balls (Corollary
\ref{cor:nc-balls}). We also prove a version of this result for
the case $\vecspace{V}=\mathbb{C}^d$ and the TT series convergent
along $\free_d$ (Theorem \ref{thm:king-semigr} and Corollary
\ref{cor:nc-balls-semigr}). In Section
\ref{subsec:analytic-higher} we introduce and study analytic
higher order nc functions, also in the three different settings.
The main results are analogous to those in  Sections
\ref{subsec:analytic} and \ref{subsec:u-analytic}.

In Chapter \ref{sec:nc-ps}, we study the convergence of nc power
series in the three different topologies, as in Chapter
\ref{sec:conv}, to an (analytic in the corresponding sense) nc
function. The results are the converse of the results of Sections
\ref{subsec:analytic} and \ref{subsec:u-analytic}. We discuss the
convergence of nc power series in the finitely open topology
(resp., in the norm topology and in the uniformly-open topology)
in Section \ref{subsec:conv-fin-open} (resp.,
\ref{subsec:conv-norm} and \ref{subsec:conv-unif-open}). In each
topology, we have Cauchy--Hadamard type estimates for radii of
convergence and sharp estimates for the size of convergence
domains of various shapes. We also characterize the maximal nc
sets where the sum of the series is an analytic on slices (resp.,
analytic and uniformly analytic) nc function and give many examples
illustrating the differences between various types of convergence
domains.

In Chapter \ref{sec:dirsum-ext}, we define and study so-called
direct summands extensions of nc sets and nc functions. If a nc
set $\Omega$ is invariant under similarities, then one can extend
$\Omega$ to a larger nc set, $\Omega_{\rm d.s.e.}$, which
contains, together with a matrix which is decomposable into a
direct sum of matrices, every direct summand of the decomposition.
This extension preserves many properties of $\Omega$ (Proposition
\ref{prop:dirsum-ext-properties}). We then can extend a nc
function $f$ on $\Omega$ to a nc function $f_{\rm d.s.e.}$ on
$\Omega_{\rm d.s.e.}$, which inherits most important properties of
$f$, in particular, analyticity (Proposition
\ref{prop:ncfun-dirsum-ext}). Similarly, we define and establish
the properties of the direct summands extensions of higher order
nc functions (Proposition \ref{prop:higher-ncfun-dirsum-ext}). We
also define and study direct summands extensions of sequences
$f_\ell$, $\ell=0,1,\ldots$ (resp., $f_w$, $w\in\free_d$) of
multilinear mappings satisfying conditions
\eqref{eq:LAC_0}--\eqref{eq:LAC_ell_ell} (resp., their
counterparts in the case $\module{M}=\ring^d$) in Proposition
\ref{prop:LAC-dirsum-ext} (resp., in Proposition
\ref{prop:LAC-dirsum-ext-semigr}).

As we mentioned in Section \ref{subsec:ncfun}, we need a nc set
$\Omega$ to be right (left) admissible in order to define right
(left) difference-differential operators via evaluations of nc
functions $f$ on $\Omega$ at block upper (lower) triangular
matrices. However,  for a (say, upper) block triangular matrix
$\begin{bmatrix} X & Z\\
 0 & Y
\end{bmatrix}$ we need to scale $Z$, so that $\begin{bmatrix} X &
rZ\\
 0 & Y
\end{bmatrix}\in\Omega$. Then we can define $\Delta_Rf(X,Y)(rZ)$
and extend this definition to arbitrary $Z$  by linearity of
$\Delta_Rf(X,Y)(\cdot)$. Although the properties of $\Delta_R$
($\Delta_L$) can be established using these scalings, this creates
cumbersome technicalities in the proofs. To bypass those
technicalities, we define the so-called similarity invariant
envelope $\widetilde{\Omega}$ of $\Omega$, which is the smallest
nc set containing $\Omega$ and invariant under similarities. Then
it turns out that $\widetilde{\Omega}$ is also invariant under
formation of block triangular matrices (Proposition
\ref{prop:adm-env}). Our next step is to extend a nc function $f$
on $\Omega$ to a nc function $\widetilde{f}$ on
$\widetilde{\Omega}$. Such an extension is unique (Proposition
\ref{prop:ncfun-ext}). A similar extension is constructed for
higher order nc functions (Proposition
\ref{prop:higher-ncfun-ext}). It turns out that
$(\Delta_R\widetilde{f})|_{\Omega\times\Omega}=\Delta_Rf$ (Remark
\ref{rem:alternative}), which allows us to prove various
statements about the difference-differential operators throughout
the monograph assuming, without loss of generality, that the
underlying nc set $\Omega$ is similarity invariant. The
 results on similarity invariant envelopes and the corresponding
extensions of nc functions are presented in Appendix \ref{app},
which is written in collaboration with Shibananda Biswas.

\chapter{NC functions and their difference-differential
calculus}\label{sec:ncfun} \section{Definition of nc
functions}\label{subsec:ncfundef} Let $\ring$ \index{$\ring$} be a
commutative ring with identity. For $\module{M}$
\index{$\module{M}$} a module over $\ring$, we define the \emph{nc
space over $\module{M}$} \index{nc space}
\begin{equation*}
\ncspace{\module{M}}=\coprod_{n=1}^\infty\mat{\module{M}}{n}.
\end{equation*}
\index{$\mat{\module{M}}{n}$} \index{$\ncspace{\module{M}}$}For
$X\in\mat{\module{M}}{n}$ and $Y\in\mat{\module{M}}{m}$ we define
their \emph{direct sum} \index{direct sum of matrices}
\begin{equation*}
X\oplus Y=\begin{bmatrix}
X & 0\\
0 & Y\end{bmatrix}\in \mat{\module{M}}{(n+m)}.
\end{equation*}
\index{$X\oplus Y$}Notice that matrices over $\ring$ act from the
right and from the left on matrices over $\module{M}$  by the
standard rules of matrix multiplication and by the action of
$\ring$ on $\module{M}$: if $X\in\rmat{\module{M}}{p}{q}$ and
$T\in\rmat{\ring}{r}{p}$, $S\in\rmat{\ring}{q}{s}$, then
\begin{equation*}
TX  \in \rmat{\module{M}}{r}{q},\quad XS  \in
\rmat{\module{M}}{p}{s}.
\end{equation*}
A subset $\Omega\subseteq\ncspace{\module{M}}$ \index{$\Omega$} is
called a \emph{nc set} \index{nc set} if it is closed under direct
sums; explicitly, denoting $\Omega_n = \Omega \cap
\mat{\module{M}}{n}$, \index{$\Omega_n$} we have $X \oplus Y \in
\Omega_{n+m}$ for all $X \in \Omega_n$, $Y \in \Omega_m$. We will
sometimes use the convention that $\rmat{\module{M}}{n}{0}$,
\index{$\rmat{\module{M}}{n}{0}$} $\rmat{\module{M}}{0}{n}$,
\index{$\rmat{\module{M}}{0}{n}$} and
$\Omega_0=\mat{\module{M}}{0}$ \index{$\Omega_0$}
\index{$\mat{\module{M}}{0}$} consist each of a single (zero)
element, the empty matrix ``of appropriate size", and that $X
\oplus Y=Y \oplus X=X\in\Omega_n$ for $n\in\mathbb{N}$,
$X\in\Omega_n$, and $Y\in\Omega_0$.

In the case of $\module{M}=\ring^d$, \index{$\ring^d$} we identify
matrices over $\module{M}$ with $d$-tuples of matrices over
$\ring$:
\begin{equation*}
\rmat{\left(\ring^d\right)}{p}{q}\cong\rmattuple{\ring}{p}{q}{d}.
\end{equation*}
Under this identification, for $d$-tuples
$X=(X_1,\ldots,X_d)\in\mattuple{\ring}{n}{d}$ and
$Y=(Y_1,\ldots,Y_d)\in\mattuple{\ring}{m}{d}$,
\begin{equation*}
X\oplus Y=\left(\begin{bmatrix}
X_1 & 0\\
0 & Y_1\end{bmatrix},\ldots,\begin{bmatrix}
X_d & 0\\
0 & Y_d\end{bmatrix}\right)\in \mattuple{\ring}{(n+m)}{d};
\end{equation*}
and for a $d$-tuple
$X=(X_1,\ldots,X_d)\in\rmattuple{\ring}{p}{q}{d}$ and matrices
$T\in\rmat{\ring}{r}{p}$, $S\in\rmat{\ring}{q}{s}$,
\begin{equation*}
TX = \left(TX_1,\ldots,TX_d\right) \in
\rmattuple{\ring}{r}{q}{d},\quad XS =
\left(X_1S,\ldots,X_dS\right) \in \rmattuple{\ring}{p}{s}{d}.
\end{equation*}

Let $\module{M}$ and $\module{N}$ be modules over $\ring$, and let
$\Omega\subseteq\ncspace{\module{M}}$ be a nc set. A mapping $f
\colon \Omega \to \ncspace{\module{N}}$, with $f(\Omega_n)
\subseteq \mat{\module{N}}{n}$, $n=0,1,\ldots$, is called a
\emph{nc function} \index{nc function} if $f$ satisfies the
following two conditions:
\begin{itemize}
\item $f$ \emph{respects direct sums}: \index{respecting direct
sums}
\begin{equation} \label{eq:dirsums}
f(X \oplus Y) = f(X) \oplus f(Y)
\end{equation}
for all $X,Y \in\Omega$. \item $f$ \emph{respects similarities}:
\index{respecting similarities} if $X \in \Omega_n$ and $S \in
\mat{\ring}{n}$ is invertible with $SXS^{-1} \in \Omega_n$, then
\begin{equation} \label{eq:simsim}
f(SXS^{-1}) = S f(X) S^{-1}.
\end{equation}
\end{itemize}

The two conditions in the definition of a nc function can be
actually replaced by a single one.
\begin{prop}\label{prop:simint}
A mapping $f \colon \Omega \to \ncspace{\module{N}}$, with
$f(\Omega_n) \subseteq \mat{\module{N}}{n}$, $n=0,1,\ldots$,
respects direct sums and similarities if and only if it
\emph{respects intertwinings}: \index{respecting intertwinings}
for any $X\in\Omega_n$, $Y\in\Omega_m$, and $T\in
\rmat{\ring}{n}{m}$, one has
\begin{equation}\label{eq:simint}
XT=TY\Longrightarrow f(X)T=Tf(Y).
\end{equation}
\end{prop}
Condition \eqref{eq:simint} was used by J. L. Taylor in \cite{T2}.
\begin{proof}[Proof of Proposition \ref{prop:simint}]
Assume that $f$ respects direct sums and similarities. If
$X\in\Omega_n$, $Y\in\Omega_m$, and $T\in \rmat{\ring}{n}{m}$ are
such that $XT=TY$, then
\begin{equation*}
\begin{bmatrix}
X & 0\\
0 & Y
\end{bmatrix}=\begin{bmatrix}
I_n & -T\\
0 & I_m
\end{bmatrix}\begin{bmatrix}
X & 0\\
0 & Y
\end{bmatrix}\begin{bmatrix}
I_n & T\\
0 & I_m
\end{bmatrix}.
\end{equation*}
Since \begin{equation*}
\begin{bmatrix}
I_n & -T\\
0 & I_m
\end{bmatrix}=\begin{bmatrix}
I_n & T\\
0 & I_m
\end{bmatrix}^{-1},
\end{equation*}
we can apply first \eqref{eq:simsim} and then \eqref{eq:dirsums}
to get
\begin{equation*}
\begin{bmatrix}
f(X) & 0\\
0 & f(Y)
\end{bmatrix}=\begin{bmatrix}
I_n & -T\\
0 & I_m
\end{bmatrix}\begin{bmatrix}
f(X) & 0\\
0 & f(Y)
\end{bmatrix}\begin{bmatrix}
I_n & T\\
0 & I_m
\end{bmatrix}.
\end{equation*}
Equating the $(1,2)$ block entries on the two sides of this
equality, we get $0=f(X)T-Tf(Y)$, i.e., \eqref{eq:simint} holds.

Conversely, assume that $f$ respects  intertwinings. It is then
clear that $f$ respects  similarities. To prove that $f$ respects
direct sums, let $X\in\Omega_n$, $Y\in\Omega_m$. Obviously,
\begin{equation*}
\begin{bmatrix}
X & 0\\
0 & Y
\end{bmatrix}\begin{bmatrix}
I_n \\
0
\end{bmatrix}=\begin{bmatrix}
I_n \\
0
\end{bmatrix}X.
\end{equation*}
Therefore
\begin{equation}\label{eq:int1}
f\left(\begin{bmatrix}
X & 0\\
0 & Y
\end{bmatrix}\right)\begin{bmatrix}
I_n \\
0
\end{bmatrix}=\begin{bmatrix}
I_n \\
0
\end{bmatrix}f(X).
\end{equation}
Denote
\begin{equation*}
f\left(\begin{bmatrix}
X & 0\\
0 & Y
\end{bmatrix}\right)=
\begin{bmatrix}
A & B\\
C & D
\end{bmatrix}.
\end{equation*}
Then \eqref{eq:int1} becomes
\begin{equation*}
\begin{bmatrix}
A\\
C
\end{bmatrix}=\begin{bmatrix}
f(X)\\
0
\end{bmatrix}.
\end{equation*}
In the same way the identity
\begin{equation*}
\begin{bmatrix}
X & 0\\
0 & Y
\end{bmatrix}\begin{bmatrix}
0\\
I_m
\end{bmatrix}=\begin{bmatrix}
0 \\
I_m
\end{bmatrix}Y
\end{equation*}
implies that
\begin{equation*}
\begin{bmatrix}
B\\
D
\end{bmatrix}=\begin{bmatrix}
0\\
f(Y)
\end{bmatrix}.
\end{equation*}
Thus $A=f(X)$, $B=0$, $C=0$, $D=f(Y)$, i.e., \eqref{eq:dirsums}
holds.
\end{proof}

\section{NC difference-differential
operators}\label{subsec:difdif} Our next goal is to introduce
right and left difference-differential operators, $\Delta_R$
\index{$\Delta_R$} and $\Delta_L$, \index{$\Delta_L$} on nc
functions. This will be done by evaluating a nc function $f$ on
block upper (respectively, lower) triangular matrices. The idea is
to show that, for $X\in\Omega_n$ and $Y\in\Omega_m$,
\begin{equation}\label{eq:uptr}
f\left(\begin{bmatrix} X & Z\\
0 & Y
\end{bmatrix}\right)=\begin{bmatrix} f(X) & \Delta_R f(X,Y)(Z)\\
0 & f(Y)
\end{bmatrix}.
\end{equation}
Here $Z\mapsto \Delta_R f(X,Y)(Z)$ \index{$\Delta_R f(X,Y)(Z)$} is
a $\ring$-linear operator from $\rmat{\module{M}}{n}{m}$ to
$\rmat{\module{N}}{n}{m}$ which plays the role of a right nc
difference operator when $Y\neq X$, and of a right nc differential
operator when $Y= X$. Analogously,
\begin{equation}\label{eq:ltr}
f\left(\begin{bmatrix} X & 0\\
Z & Y
\end{bmatrix}\right)=\begin{bmatrix} f(X) & 0\\
\Delta_L f(X,Y)(Z) & f(Y)
\end{bmatrix},
\end{equation}
where $Z\mapsto \Delta_L f(X,Y)(Z)$ \index{$\Delta_L f(X,Y)(Z)$}
is a $\ring$-linear operator from $\rmat{\module{M}}{m}{n}$ to
$\rmat{\module{N}}{m}{n}$ which plays the role of a left nc
difference operator when $Y\neq X$, and  of a left nc differential
operator when $Y=X$. The assumption
that $\begin{bmatrix} X & Z\\
0 & Y
\end{bmatrix}\in\Omega_{n+m}$ for all $X\in\Omega_n$,
$Y\in\Omega_m$, and $Z\in\rmat{\module{M}}{n}{m}$ (respectively, that $\begin{bmatrix} X & 0\\
Z & Y
\end{bmatrix}\in\Omega_{n+m}$ for all $X\in\Omega_n$,
$Y\in\Omega_m$, and $Z\in\rmat{\module{M}}{m}{n}$) would be
however too strong, and we will proceed with a weaker assumption
on a nc set $\Omega$, which is naturally suggested by the
following proposition.
\begin{prop}\label{prop:delta_r}
Let $f$ be a nc function on a nc set $\Omega$. Let $X\in\Omega_n$,
$Y\in\Omega_m$, and $Z\in\rmat{\module{M}}{n}{m}$ be such that
$\begin{bmatrix} X & Z\\
0 & Y
\end{bmatrix}\in\Omega_{n+m}$. Then \eqref{eq:uptr} holds, where
the off-diagonal block entry $\Delta_R f(X,Y)(Z)$  is determined
uniquely by \eqref{eq:uptr} and has
the following property. If $r\in\ring$ is such that $\begin{bmatrix} X & rZ\\
0 & Y
\end{bmatrix}\in\Omega_{n+m}$, then
\begin{equation}\label{eq:homog_r}
\Delta_R f(X,Y)(rZ)=r\Delta_R f(X,Y)(Z).
\end{equation}
\end{prop}
\begin{proof}
Denote \begin{equation*}
f\left(\begin{bmatrix} X & Z\\
0 & Y
\end{bmatrix}\right)=\begin{bmatrix} A & \Delta_R f(X,Y)(Z)\\
C & D
\end{bmatrix}.
\end{equation*}
We have
\begin{equation*}
\begin{bmatrix} X & Z\\
0 & Y
\end{bmatrix}\begin{bmatrix} I_n\\
0
\end{bmatrix}=\begin{bmatrix} I_n\\
0
\end{bmatrix}X.
\end{equation*}
By Proposition \ref{prop:simint},
\begin{equation*}
f\left(\begin{bmatrix} X & Z\\
0 & Y
\end{bmatrix}\right)\begin{bmatrix} I_n\\
0
\end{bmatrix}=\begin{bmatrix} I_n\\
0
\end{bmatrix}f(X).
\end{equation*}
This implies that $A=f(X)$ and $C=0$. We have also
\begin{equation*}
Y\begin{bmatrix} 0 & I_m
\end{bmatrix}=\begin{bmatrix} 0 & I_m
\end{bmatrix}\begin{bmatrix} X & Z\\
0 & Y
\end{bmatrix}.
\end{equation*}
By Proposition \ref{prop:simint},
\begin{equation*}
f(Y)\begin{bmatrix} 0 & I_m
\end{bmatrix}=\begin{bmatrix} 0 & I_m
\end{bmatrix}f\left(\begin{bmatrix} X & Z\\
0 & Y
\end{bmatrix}\right),
\end{equation*}
which implies that $D=f(Y)$ and that again $C=0$. Thus,
\eqref{eq:uptr} holds.

Next, we have
\begin{equation*}
\begin{bmatrix} rI_n & 0\\
0 & I_m
\end{bmatrix}\begin{bmatrix} X & Z\\
0 & Y
\end{bmatrix}=\begin{bmatrix} X & rZ\\
0 & Y
\end{bmatrix}\begin{bmatrix} rI_n & 0\\
0 & I_m
\end{bmatrix}.
\end{equation*}
By Proposition \ref{prop:simint},
\begin{equation*}
\begin{bmatrix} rI_n & 0\\
0 & I_m
\end{bmatrix}f\left(\begin{bmatrix} X & Z\\
0 & Y
\end{bmatrix}\right)=f\left(\begin{bmatrix} X & rZ\\
0 & Y
\end{bmatrix}\right)\begin{bmatrix} rI_n & 0\\
0 & I_m
\end{bmatrix},
\end{equation*}
which is by \eqref{eq:uptr} equivalent to
\begin{equation*}
\begin{bmatrix} rI_n & 0\\
0 & I_m
\end{bmatrix}\begin{bmatrix} f(X) & \Delta f(X,Y)(Z)\\
0 & f(Y)
\end{bmatrix}=\begin{bmatrix} f(X) & \Delta f (X,Y)(rZ)\\
0 & f(Y)
\end{bmatrix}\begin{bmatrix} rI_n & 0\\
0 & I_m
\end{bmatrix}.
\end{equation*}
Writing out the $(1,2)$ block in the product matrix on both sides
of this equality, we obtain \eqref{eq:homog_r}.
\end{proof}
The ``left" counterpart of Proposition \ref{prop:delta_r} (which
has a similar proof) is the following.
\begin{prop}\label{prop:delta_l}
Let $f$ be a nc function on a nc set $\Omega$. Let $X\in\Omega_n$,
$Y\in\Omega_m$, and $Z\in\rmat{\module{M}}{m}{n}$ be such that
$\begin{bmatrix} X & 0\\
Z & Y
\end{bmatrix}\in\Omega_{n+m}$. Then \eqref{eq:ltr} holds, where
the off-diagonal block entry $\Delta_L f(X,Y)(Z)$  is determined
uniquely by \eqref{eq:ltr} and has
the following property. If $r\in\ring$ is such that $\begin{bmatrix} X & 0\\
rZ & Y
\end{bmatrix}\in\Omega_{n+m}$, then
\begin{equation}\label{eq:homog_l}
\Delta_L f(X,Y)(rZ)=r\Delta_L f(X,Y)(Z).
\end{equation}
\end{prop}

We will say that the nc set $\Omega\subseteq\ncspace{\module{M}}$
is \emph{right admissible} \index{right admissible nc set} if for
every $X\in\Omega_n$, $Y\in\Omega_m$, and
$Z\in\rmat{\module{M}}{n}{m}$ there exists an invertible
$r\in\ring$ such that $\begin{bmatrix} X & rZ\\
0 & Y
\end{bmatrix}\in\Omega_{n+m}$. We will say that the nc set $\Omega\subseteq\ncspace{\module{M}}$ is
\emph{left admissible} \index{left admissible nc set} if for every
$X\in\Omega_n$, $Y\in\Omega_m$, and $Z\in\rmat{\module{M}}{m}{n}$
there exists an invertible
$r\in\ring$ such that $\begin{bmatrix} X & 0\\
rZ & Y
\end{bmatrix}\in\Omega_{n+m}$.

Let the nc set $\Omega$ be right admissible, and let $f$ be a nc
function on $\Omega$. Then for every $X\in\Omega_n$,
$Y\in\Omega_m$, and $Z\in\rmat{\module{M}}{n}{m}$, and for an
invertible
$r\in\ring$ such that $\begin{bmatrix} X & rZ\\
0 & Y
\end{bmatrix}\in\Omega_{n+m}$ we define first $\Delta_Rf(X,Y)(rZ)$
as the $(1,2)$ block of the matrix $f\left(\begin{bmatrix} X & rZ\\
0 & Y
\end{bmatrix}\right)$ (see \eqref{eq:uptr}), and then define
\begin{equation}\label{eq:def_delta_r}
\Delta_Rf(X,Y)(Z)=r^{-1}\Delta_Rf(X,Y)(rZ).
\end{equation}

Similarly, let the nc set $\Omega$ be left admissible, and let $f$
be a nc function on $\Omega$. Then for every $X\in\Omega_n$,
$Y\in\Omega_m$, and $Z\in\rmat{\module{M}}{m}{n}$, and for an
invertible
$r\in\ring$ such that $\begin{bmatrix} X & 0\\
rZ & Y
\end{bmatrix}\in\Omega_{n+m}$ we define first $\Delta_Lf(X,Y)(rZ)$
as the $(2,1)$ block of the matrix $f\left(\begin{bmatrix} X & 0\\
rZ & Y
\end{bmatrix}\right)$ (see \eqref{eq:ltr}), and then define
\begin{equation}\label{eq:def_delta_l}
\Delta_Lf(X,Y)(Z)=r^{-1}\Delta_Lf(X,Y)(rZ).
\end{equation}

By Proposition \ref{prop:delta_r} (respectively, by Proposition
\ref{prop:delta_l}), the right-hand side of \eqref{eq:def_delta_r}
(respectively, of \eqref{eq:def_delta_l}) is independent of $r$,
hence the left-hand side is defined unambiguously.
 Moreover, we obtain the following.

\begin{prop}\label{prop:homog} Let $f$ be a nc function on a nc
set $\Omega$. The mappings $Z\mapsto\Delta_Rf(X,Y)(Z)$ and
$Z\mapsto\Delta_Lf(X,Y)(Z)$ are \emph{homogeneous}
\index{homogeneous mapping} as operators acting from
$\rmat{\module{M}}{n}{m}$ to $\rmat{\module{N}}{n}{m}$ and from
$\rmat{\module{M}}{m}{n}$ to $\rmat{\module{N}}{m}{n}$,
respectively. In other words, if $\Omega$ is right admissible then
for every $X\in\Omega_n$, $Y\in\Omega_m$, and
$Z\in\rmat{\module{M}}{n}{m}$, and every $r\in\ring$,
\eqref{eq:homog_r} holds; respectively, if $\Omega$ is left
admissible then for every $X\in\Omega_n$, $Y\in\Omega_m$, and
$Z\in\rmat{\module{M}}{m}{n}$, and every $r\in\ring$,
\eqref{eq:homog_l} holds.
\end{prop}
\begin{proof}
Let us prove one of these homogeneity properties --- the proof of
the other one is analogous. Let $\Omega$ be right admissible,
$X\in\Omega_n$, $Y\in\Omega_m$,
 $Z\in\rmat{\module{M}}{n}{m}$, and $r\in\ring$. Then there
 exist invertible $r_1,r_2\in\ring$ such that $\begin{bmatrix} X & r_1Z\\
0 & Y
\end{bmatrix}\in\Omega_{n+m}$ and $\begin{bmatrix} X & r_2rZ\\
0 & Y
\end{bmatrix}\in\Omega_{n+m}$. Then
\begin{equation*}
\Delta_Rf(X,Y)(Z)=r_1^{-1}\Delta_Rf(X,Y)(r_1Z)
\end{equation*}
 and
\begin{equation*}
\Delta_Rf(X,Y)(rZ)=r_2^{-1}\Delta_Rf(X,Y)(r_2rZ).
\end{equation*}
Moreover, by Proposition \ref{prop:delta_r},
\begin{equation*}
\Delta_Rf(X,Y)(r_2rZ)=r_2rr_1^{-1}\Delta_Rf(X,Y)(r_1Z).
\end{equation*}
These three equalities imply \eqref{eq:homog_r}.
\end{proof}

\begin{rem}\label{rem:alternative}
The proof of Proposition \ref{prop:homog} can be simplified with a
use of the similarity invariant envelope \index{similarity
invariant envelope} $\widetilde{\Omega}$
\index{$\widetilde{\Omega}$} of a right (resp., left) admissible
nc set $\Omega$ and the canonical extension \index{canonical
extension of a nc function}
$\widetilde{f}\colon\widetilde{\Omega}\to\ncspace{\vecspace{N}}$
\index{$\widetilde{f}$} of a nc function
$f\colon\Omega\to\ncspace{\vecspace{N}}$; see Appendix \ref{app}
for the definitions. We will show this for a right admissible nc
set $\Omega$. By Proposition \ref{prop:adm-env}, $\begin{bmatrix}
\widetilde{X}
& Z\\
0 & \widetilde{Y}
\end{bmatrix}\in\widetilde{\Omega}_{n+m}$ for all $\widetilde{X}\in\widetilde{\Omega}_n$,
$\widetilde{Y}\in\widetilde{\Omega}_m$, and
$Z\in\rmat{\module{M}}{n}{m}$, and by Proposition
\ref{prop:ncfun-ext}, $\widetilde{f}$ is a nc function. Therefore,
a formula analogous to \eqref{eq:uptr} defines $\Delta_R
\widetilde{f}(\widetilde{X},\widetilde{Y})(Z)$ for all
$n,m\in\mathbb{N}$, $\widetilde{X}\in\widetilde{\Omega}_n$,
$\widetilde{Y}\in\widetilde{\Omega}_m$, and
$Z\in\rmat{\module{M}}{n}{m}$, and by Proposition
\ref{prop:delta_r} applied to $\widetilde{f}$, one has
$$\Delta_R\widetilde{f}(\widetilde{X},\widetilde{Y})(rZ)=r\Delta_R\widetilde{f}(\widetilde{X},\widetilde{Y})(Z)$$
for every $r\in\ring$. In particular, if $X\in\Omega_n$,
$Y\in\Omega_m$, $Z\in\rmat{\module{M}}{n}{m}$, and
$r\in\ring$ is such that $\begin{bmatrix} X & rZ\\
0 & Y\end{bmatrix}\in\Omega_{n+m}$, then
$$\Delta_R{f}({X},{Y})(rZ)=\Delta_R\widetilde{f}({X},{Y})(rZ)=r\Delta_R\widetilde{f}({X},{Y})(Z).$$
If, in addition, $r$ is invertible, then from
\eqref{eq:def_delta_r} we obtain
\begin{equation}\label{eq:delta=delta_ext}
\Delta_Rf(X,Y)(Z)=\Delta_R\widetilde{f}({X},{Y})(Z).
\end{equation}
 Consequently,
the mapping $Z\mapsto\Delta_Rf(X,Y)(Z)$ is homogeneous, and the
statement of Proposition \ref{prop:homog} for a right admissible
nc set $\Omega$ follows.  An analogous argument works for a left
admissible nc set $\Omega$.

We see that the homogeneity of the mapping
$Z\mapsto\Delta_Rf(X,Y)(Z)$ follows directly from the homogeneity
of the mapping
$Z\mapsto\Delta_R\widetilde{f}(\widetilde{X},\widetilde{Y})(Z)$,
and the proof for the latter is less involved, since it does not
require manipulating with appropriate invertible elements $r\in
\ring$ so that the relevant block triangular matrices get into the
nc set $\Omega$. In the sequel, we will often use similarity
envelopes of nc sets and canonical extensions of nc functions to
simplify the proofs.
\end{rem}

\begin{prop}\label{prop:add}
Let $f$ be a nc function on a nc set $\Omega$. The mappings
$Z\mapsto\Delta_Rf(X,Y)(Z)$ and $Z\mapsto\Delta_Lf(X,Y)(Z)$ are
\emph{additive} \index{additive mapping} as operators acting from
$\rmat{\module{M}}{n}{m}$ to $\rmat{\module{N}}{n}{m}$ and from
$\rmat{\module{M}}{m}{n}$ to $\rmat{\module{N}}{m}{n}$,
respectively. In other words, if $\Omega$ is right admissible,
then for every $X\in\Omega_n$, $Y\in\Omega_m$, and $Z_I,Z_{I\!
I}\in\rmat{\module{M}}{n}{m}$,
\begin{equation}\label{eq:add_r}
\Delta_Rf(X,Y)(Z_I+Z_{I\!
I})=\Delta_Rf(X,Y)(Z_I)+\Delta_Rf(X,Y)(Z_{I\! I});
\end{equation}
respectively, if $\Omega$ is left admissible then for every
$X\in\Omega_n$, $Y\in\Omega_m$, and $Z_I,Z_{I\!
I}\in\rmat{\module{M}}{m}{n}$,
\begin{equation}\label{eq:add_l}
\Delta_Lf(X,Y)(Z_I+Z_{I\!
I})=\Delta_Lf(X,Y)(Z_I)+\Delta_Lf(X,Y)(Z_{I\! I}).
\end{equation}
\end{prop}
\begin{proof}
Let us prove one of these additivity properties --- the proof of
the other one is analogous. Let $\Omega$ be right admissible,
$X\in\Omega_n$, $Y\in\Omega_m$,
 $Z_I,Z_{I\! I}\in\rmat{\module{M}}{n}{m}$. By \eqref{eq:delta=delta_ext}, we may assume, without loss of
 generality,
  that $\Omega$ is similarity
 invariant, i.e., $\widetilde{\Omega}=\Omega$; see Remark
 \ref{rem:alternative} and Appendix \ref{app}.
  Then, by Proposition \ref{prop:adm-env}, $$\begin{bmatrix} X & Z_I\\
0 & Y
\end{bmatrix}\in\Omega_{n+m},\quad \begin{bmatrix} X & Z_{I\! I}\\
0 & Y
\end{bmatrix}\in\Omega_{n+m},\quad  \begin{bmatrix} X & Z_I+Z_{I\! I}\\
0 & Y
\end{bmatrix}\in\Omega_{n+m},$$ and $$\begin{bmatrix} X & 0 & Z_I\\
0 & X & Z_{I\! I}\\
0 & 0 & Y
\end{bmatrix}\in\Omega_{2n+m}$$ (for the latter, we use the fact that if $X\in\Omega_n$ then
$\begin{bmatrix} X & 0\\
0 & X
\end{bmatrix}\in\Omega_{2n}$).
We have
\begin{equation*}
\begin{bmatrix}  I_n & 0 & 0\\
0 & 0 & I_m
\end{bmatrix}
\begin{bmatrix} X & 0 & Z_I\\
0 & X & Z_{I\! I}\\
0 & 0 & Y
\end{bmatrix}
=\begin{bmatrix} X & Z_I\\
0 & Y
\end{bmatrix}\begin{bmatrix} I_n & 0 & 0\\
0 & 0 & I_m
\end{bmatrix}.
\end{equation*}
By Proposition \ref{prop:simint},
\begin{equation*}
\begin{bmatrix} I_n & 0 & 0\\
0 & 0 & I_m
\end{bmatrix}
f\left(\begin{bmatrix} X & 0 & Z_I\\
0 & X & Z_{I\! I}\\
0 & 0 & Y
\end{bmatrix}\right)
=f\left(\begin{bmatrix} X & Z_I\\
0 & Y
\end{bmatrix}\right)\begin{bmatrix} I_n & 0 & 0\\
0 & 0 & I_m
\end{bmatrix}.
\end{equation*}
Clearly, we have
\begin{equation*}
f\left(\begin{bmatrix} X & 0 & Z_I\\
0 & X & Z_{I\! I}\\
0 & 0 & Y
\end{bmatrix}\right)=\begin{bmatrix} f(X) & 0 & A\\
0 & f(X) & B\\
0 & 0 & f(Y)
\end{bmatrix},
\end{equation*}
with some matrices $A$ and $B$, and also we have
\begin{equation*}
f\left(\begin{bmatrix} X & Z_I\\
0 & Y
\end{bmatrix}\right)=\begin{bmatrix} f(X) & \Delta_Rf(X,Y)(Z_I)\\
0 & f(Y)
\end{bmatrix}.
\end{equation*}
Therefore,
\begin{equation*}
\begin{bmatrix} I_n & 0 & 0\\
0 & 0 & I_m
\end{bmatrix}
\begin{bmatrix} f(X) & 0 & A\\
0 & f(X) & B\\
0 & 0 & f(Y)
\end{bmatrix}
=\begin{bmatrix} f(X) & \Delta_Rf(X,Y)(Z_I)\\
0 & f(Y)
\end{bmatrix}\begin{bmatrix} I_n & 0 & 0\\
0 & 0 & I_m
\end{bmatrix}.
\end{equation*}
Comparing the $(1,3)$ block entries in a product matrix on both
sides, we get
\begin{equation*}
A=\Delta_Rf(X,Y)(Z_I).
\end{equation*}
Similarly, we proceed with the intertwining
\begin{equation*}
\begin{bmatrix} 0  & I_n & 0\\
0 & 0 & I_m
\end{bmatrix}
\begin{bmatrix} X & 0 & Z_I\\
0 & X & Z_{I\! I}\\
0 & 0 & Y
\end{bmatrix}
=\begin{bmatrix} X & Z_{I\! I}\\
0 & Y
\end{bmatrix}\begin{bmatrix} 0 & I_n & 0\\
0 & 0 & I_m
\end{bmatrix}.
\end{equation*}
By Proposition \ref{prop:simint},
\begin{equation*}
\begin{bmatrix} 0 & I_n & 0\\
0 & 0 & I_m
\end{bmatrix}
f\left(\begin{bmatrix} X & 0 & Z_I\\
0 & X & Z_{I\! I}\\
0 & 0 & Y
\end{bmatrix}\right)
=f\left(\begin{bmatrix} X & Z_{I\! I}\\
0 & Y
\end{bmatrix}\right)\begin{bmatrix} 0 & I_n & 0\\
0 & 0 & I_m
\end{bmatrix},
\end{equation*}
or equivalently,
\begin{equation*}
\begin{bmatrix} 0 & I_n & 0\\
0 & 0 & I_m
\end{bmatrix}
\begin{bmatrix} f(X) & 0 & A\\
0 & f(X) & B\\
0 & 0 & f(Y)
\end{bmatrix}
=\begin{bmatrix} f(X) & \Delta_Rf(X,Y)(Z_{I\! I})\\
0 & f(Y)
\end{bmatrix}\begin{bmatrix} 0 & I_n & 0\\
0 & 0 & I_m
\end{bmatrix}.
\end{equation*}
Comparing the $(1,3)$ block entries in a product matrix on both
sides, we get
\begin{equation*}
B=\Delta_Rf(X,Y)(Z_{I\! I}).
\end{equation*}
Next, we use the intertwining relation
\begin{equation*}
\begin{bmatrix}I_n  & I_n & 0\\
0 & 0 & I_m
\end{bmatrix}
\begin{bmatrix} X & 0 & Z_I\\
0 & X & Z_{I\! I}\\
0 & 0 & Y
\end{bmatrix}
=\begin{bmatrix} X & Z_I+Z_{I\! I}\\
0 & Y
\end{bmatrix}\begin{bmatrix} I_n & I_n & 0\\
0 & 0 & I_m
\end{bmatrix}.
\end{equation*}
By Proposition \ref{prop:simint},
\begin{equation*}
\begin{bmatrix} I_n & I_n & 0\\
0 & 0 & I_m
\end{bmatrix}
f\left(\begin{bmatrix} X & 0 & Z_I\\
0 & X & Z_{I\! I}\\
0 & 0 & Y
\end{bmatrix}\right)
=f\left(\begin{bmatrix} X & Z_I+Z_{I\! I}\\
0 & Y
\end{bmatrix}\right)\begin{bmatrix} I_n & I_n & 0\\
0 & 0 & I_m
\end{bmatrix},
\end{equation*}
or equivalently,
\begin{align*}
\begin{bmatrix} I_n & I_n & 0\\
0 & 0 & I_m
\end{bmatrix}
&\begin{bmatrix} f(X) & 0 & \Delta_Rf(X,Y)(Z_I)\\
0 & f(X) & \Delta_Rf(X,Y)(Z_{I\! I})\\
0 & 0 & f(Y)
\end{bmatrix}\\
&=\begin{bmatrix} f(X) & \Delta_Rf(X,Y)(Z_I+Z_{I\! I})\\
0 & f(Y)
\end{bmatrix}\begin{bmatrix} I_n & I_n & 0\\
0 & 0 & I_m
\end{bmatrix}.
\end{align*}
Comparing the $(1,3)$ block entries in a product matrix on both
sides, we get
\begin{equation*}
\Delta_Rf(X,Y)(Z_I)+\Delta_Rf(X,Y)(Z_{I\!
I})=\Delta_Rf(X,Y)(Z_I+Z_{I\! I}),
\end{equation*}
i.e., \eqref{eq:add_r} is true.
\end{proof}

The transposition of matrices induces naturally a transposition
operation on nc functions and enables us to relate right and left
nc difference-differential operators.
\begin{prop}\label{prop:trans}
\begin{enumerate}
    \item If $\Omega\subseteq\ncspace{\module{M}}$ is a nc set,
    then so is $$\Omega^\top:=\{X^\top\colon
    X\in\Omega\}\subseteq\ncspace{\module{M}}.$$
    \index{$\Omega^\top$}
    \item If $f\colon \Omega\to\ncspace{\module{N}}$ is a nc
    function, then so is $f^\top\colon
    \Omega^\top\to\ncspace{\module{N}}$, \index{$f^\top$} where $f^\top(X):=f\left(X^\top\right)^\top$.
    \item If $\Omega$ is
    right (respectively, left) admissible, then $\Omega^\top$ is left (respectively,
    right) admissible. In this case,
    $$\Delta_Lf^\top(X,Y)(Z)=\left(\Delta_Rf\left(X^\top,Y^\top\right)
    \left(Z^\top\right)\right)^\top$$ (respectively,
    $\Delta_Rf^\top(X,Y)(Z)=\left(\Delta_Lf\left(X^\top,Y^\top\right)\left(Z^\top\right)\right)^\top$).
\end{enumerate}
\end{prop}
The proof is straightforward. Proposition \ref{prop:trans} allows
us to deduce results for $\Delta_L$ from the corresponding results
for $\Delta_R$ and vice versa, using a transposition argument.

In the case where a nc set $\Omega$ is both right and left
admissible, there is also a simple symmetry relation between right
and left nc difference-differential operators.
\begin{prop}\label{prop:RLsym}
Let $f\colon\Omega\to\ncspace{\module{N}}$ be a nc function on a
nc set $\Omega\subseteq\ncspace{\module{M}}$ which is both right
and left admissible. Then
\begin{equation*}
\Delta_Rf(X,Y)(Z)=\Delta_Lf(Y,X)(Z)
\end{equation*}
for all $n,m\in\mathbb{N}$, $X\in\Omega_n$, $Y\in\Omega_m$, and
$Z\in\rmat{\module{M}}{n}{m}$.
\end{prop}
\begin{proof}
 By \eqref{eq:delta=delta_ext} and its left analogue, we may assume, without loss of
 generality,
  that $\Omega$ is similarity
 invariant, i.e., $\widetilde{\Omega}=\Omega$; see Remark
 \ref{rem:alternative} and Appendix \ref{app}.
  Then by Proposition \ref{prop:adm-env},   for arbitrary
$n,m\in\mathbb{N}$, $X\in\Omega_n$, $Y\in\Omega_m$, and
$Z\in\rmat{\module{M}}{n}{m}$ one has
$$\begin{bmatrix}
X & Z\\
0 & Y \end{bmatrix},\ \begin{bmatrix}
Y & 0\\
Z & X \end{bmatrix}\in\Omega_{n+m}.$$

Clearly,
$$\begin{bmatrix}
X & Z\\
0 & Y \end{bmatrix}
\begin{bmatrix}
0 & I_n\\
I_m & 0 \end{bmatrix}=\begin{bmatrix}
0 & I_n\\
I_m & 0\end{bmatrix}\begin{bmatrix}
Y & 0\\
Z & X \end{bmatrix}.
$$
By Proposition \ref{prop:simint},
$$f\left(\begin{bmatrix}
X & Z\\
0 & Y \end{bmatrix}\right)
\begin{bmatrix}
0 & I_n\\
I_m & 0 \end{bmatrix}=\begin{bmatrix}
0 & I_n\\
I_m & 0\end{bmatrix}f\left(\begin{bmatrix}
Y & 0\\
Z & X \end{bmatrix}\right),
$$
i.e., \begin{equation*}\begin{bmatrix}
f(X) & \Delta_Rf(X,Y)(Z)\\
0 & f(Y) \end{bmatrix}
\begin{bmatrix}
0 & I_n\\
I_m & 0 \end{bmatrix}  =\begin{bmatrix}
0 & I_n\\
I_m & 0\end{bmatrix}\begin{bmatrix}
f(Y) & 0\\
\Delta_Lf(Y,X)(Z) & f(X) \end{bmatrix}.
\end{equation*}
Comparing the (1,1) block entries in the matrix products on the
left-hand side and on the right-hand side, we obtain
\begin{equation*}
\Delta_Rf(X,Y)(Z)=\Delta_Lf(Y,X)(Z).
\end{equation*}
\end{proof}

\section{Basic rules of nc difference-differential
calculus}\label{subsec:calc1} We proceed to identify basic
building blocks which allow one to compute $\Delta_Rf$ and
$\Delta_Lf$ much as in the classical commutative calculus; here
and below, whenever it does not cause confusion, we often write
$\Delta_Rf$ and $\Delta_Lf$ without arguments.
\subsection{} \label{subsub:const} If $f\colon
\ncspace{\module{M}} \to \ncspace{\module{N}}$ is a \emph{constant
nc function}, \index{constant nc function} i.e., there is a
$c\in\module{N}$ such that $f(X)=cI_n$ ($=\diag[c,\ldots,c]$) for
all $n\in\mathbb{N}$ and all $X\in\mat{\module{M}}{n}$,
 then $\Delta_R f(X,Y)(Z)=0$ and $\Delta_L f(X,Y)(Z)= 0$ for all $X,Y,Z$ of appropriate sizes.
\begin{proof} Trivial.
\end{proof}
\subsection{} \label{subsub:sum}  If $f,g\colon \Omega \to
\ncspace{\module{N}}$ are nc functions on a right (respectively,
left) admissible nc set $\Omega\subseteq\ncspace{\module{M}}$ and
$a,b\in\ring$, then
  $\Delta_R(af+bg)=a\Delta_R f+b\Delta_R g$ (respectively, $\Delta_L(af+bg)=a\Delta_L f+b\Delta_L g$).
\begin{proof} Trivial.
\end{proof}
\subsection{} \label{subsub:coord} Let
$l\colon\module{M}\to\module{N}$ be a $\ring$-linear mapping. We
extend $l$  to a linear mapping from $\rmat{\module{M}}{n}{m}$ to
$\rmat{\module{N}}{n}{m}$, $n,m\in\mathbb{N}$ by
$l([\mu_{ij}])=[l(\mu_{ij})]$.
 Then $l$ is a nc function
from $\ncspace{\module{M}}$ to $\ncspace{\module{N}}$, and
$\Delta_R l(X,Y)(Z)=l(Z)$ and $\Delta_L l(X,Y)(Z)= l(Z)$ for all
$X,Y,Z$ of appropriate sizes. In particular, if $l_j\colon
\ncspace{(\ring^d)}\to\ncspace{\ring}$ is the \emph{$j$-th
coordinate nc function}, \index{coordinate nc function} i.e.,
$$l_j(X)=l_j(X_1,\ldots,X_d)=X_j,$$ then $\Delta_Rl_j(X,Y)(Z)=Z_j$ and
$\Delta_Ll_j(X,Y)(Z)=Z_j$ for all $d$-tuples $X,Y,Z$ of matrices
over $\ring$ of appropriate sizes.
\begin{proof} Trivial.
\end{proof}
\subsection{} \label{subsub:prod} Let $f\colon \Omega\to
\ncspace{\module{N}}$, $g\colon\Omega\to \ncspace{\module{O}}$ be
nc functions on a right (respectively, left) admissible nc set
$\Omega\subseteq \ncspace{\module{M}}$. Assume that we are given a
product operation $(x,y)\mapsto x\cdot y$ on
$\module{N}\times\module{O}$ with values in a module $\module{P}$
over $\ring$ which is left and right distributive and commutes
with the action of $\ring$ (i.e., we are given a $\ring$-linear
map from  $\module{N}\otimes_\ring\module{O}$ to $\module{P}$). We
extend the product operation to matrices over $\module{N}$ and
over $\module{O}$ of appropriate sizes. It is easy to check that
$f\cdot g\colon \Omega\to\ncspace{\module{P}}$ is a nc function.

Then $$\Delta_R(f\cdot g)(X,Y)(Z)=f(X)\cdot\Delta_R
g(X,Y)(Z)+\Delta_R f(X,Y)(Z)\cdot g(Y)$$ for all
$n,m\in\mathbb{N}$, $X\in\Omega_n$, $Y\in\Omega_m$, and
$Z\in\rmat{\module{M}}{n}{m}$ (respectively,
$$\Delta_L(f\cdot g)(X,Y)(Z)=\Delta_L
f(X,Y)(Z)\cdot g(X)+ f(Y)\cdot\Delta_L g(X,Y)(Z)$$ for all
$n,m\in\mathbb{N}$, $X\in\Omega_n$, $Y\in\Omega_m$, and
$Z\in\rmat{\module{M}}{m}{n}$).
\begin{proof}
We shall now give a proof of the first equality. By
\eqref{eq:delta=delta_ext}, we may assume, without loss of
 generality,
  that $\Omega$ is similarity
 invariant, i.e., $\widetilde{\Omega}=\Omega$; see Remark
 \ref{rem:alternative} and Appendix \ref{app}.
  Then by Proposition \ref{prop:adm-env},
for arbitrary $X\in\Omega_n$, $Y\in\Omega_m$, and
$Z\in\rmat{\module{M}}{n}{m}$ one has
$\begin{bmatrix} X & Z\\
0 & Y
\end{bmatrix}\in\Omega_{n+m}$.
On one hand, \begin{equation}\label{eq:1hand} (f\cdot g)
\left(\begin{bmatrix}
X & Z\\
0 & Y \end{bmatrix}\right)=\begin{bmatrix}
(f\cdot g) (X) & \Delta_R(f\cdot g)(X,Y)(Z)\\
0 &  (f\cdot g)(Y) \end{bmatrix}. \end{equation}
 On the other
hand,
\begin{multline}\label{eq:2hand}
f\left(\begin{bmatrix}
X & Z\\
0 & Y \end{bmatrix}\right)\cdot g\left(\begin{bmatrix}
X & Z\\
0 & Y \end{bmatrix}\right)\\
=\begin{bmatrix} f (X) & \Delta_Rf(X,Y)(Z)\\
0 &  f(Y) \end{bmatrix}\cdot\begin{bmatrix} g (X) & \Delta_Rg(X,Y)(Z)\\
0 &  g(Y) \end{bmatrix}\\
=\begin{bmatrix} f(X)\cdot g(X) & f(X)\cdot\Delta_R
g(X,Y)(Z)+\Delta_R f(X,Y)(Z)\cdot
g(Y)\\
0 &  f(Y)\cdot g(Y) \end{bmatrix}.
\end{multline}
Comparing the $(1,2)$ blocks on the right hand sides of
\eqref{eq:1hand} and \eqref{eq:2hand}, we obtain the required
formula for $\Delta_R(f\cdot g)(X,Y)(Z)$.

The formula for $\Delta_L(f\cdot g)(X,Y)(Z)$ is proved
analogously. In the case where $\Omega$ is both right and left
admissible, we can also use a symmetry argument (Proposition
\ref{prop:RLsym}).
\end{proof}
\begin{rem}\label{rem:fg_trans}
We can define also an opposite product operation $(y,x)\mapsto
y\cdotop x$ \index{$y\cdotop x$} on $\module{O}\times\module{N}$
with values in $\module{P}$ by $y\cdotop x:=x\cdot y$ (i.e., we
are using the canonical isomorphism
$\module{N}\otimes_\ring\module{O}\cong\module{O}\otimes_\ring\module{N}$).
Extending the product operation ``$\cdotop$" to matrices over
$\module{O}$ and over $\module{N}$ of appropriate sizes, we
clearly have $(X\cdot Y)^\top=X^\top\cdotop Y^\top$, hence also
$(f\cdot g)^\top=f^\top\cdotop g^\top$. \index{$f\cdotop g$} In
particular, we can use Proposition \ref{prop:trans} to derive the
formula for $\Delta_L(f\cdot g)(X,Y)(Z)$ from the formula for
$\Delta_R\left(f^\top\cdotop
g^\top\right)\left(X^\top,Y^\top\right)\left(Z^\top\right)$.
\end{rem}
\subsection{}\label{subsub:inv} Let $f\colon \Omega\to
\ncspace{\module{A}}$ be a nc function on a right (respectively,
left) admissible nc set $\Omega\subseteq \ncspace{\module{M}}$,
where $\module{A}$ is a unital algebra over $\ring$. Let
$$\Omega^\text{inv}:=\coprod_{n=1}^\infty\{ X\in\Omega_n\colon
f(X)\ \text{is invertible in\ } \mat{\module{A}}{n}\}.$$ Then
$\Omega^\text{inv}$ \index{$\Omega^\text{inv}$} is a right
(respectively, left) admissible nc set,
$f^{-1}\colon\Omega^\text{inv}\to\ncspace{\module{A}}$ defined by
$f^{-1}(X):=f(X)^{-1}$ is a nc function, and
 $$\Delta_Rf^{-1}(X,Y)(Z)=-f(X)^{-1}\Delta_Rf(X,Y)(Z)f(Y)^{-1}$$
for all $n,m\in\mathbb{N}$, $X\in\Omega^\text{inv}_n$,
$Y\in\Omega^\text{inv}_m$, and $Z\in\rmat{\module{A}}{n}{m}$
(respectively,
$$\Delta_Lf^{-1}(X,Y)(Z)=-f(Y)^{-1}\Delta_Lf(X,Y)(Z)f(X)^{-1}$$
for all $n,m\in\mathbb{N}$, $X\in\Omega^\text{inv}_n$,
$Y\in\Omega^\text{inv}_m$, and $Z\in\rmat{\module{A}}{m}{n}$).
\begin{proof}
Observe that $\Omega^\text{inv}$ is a nc set since $f$ respects
direct sums.

We show next that if $\Omega$ is right admissible, then so is
$\Omega^\text{inv}$. Recall first that a block upper triangular
matrix
$\begin{bmatrix} A & C\\
0 & B
\end{bmatrix}$ is invertible if and only if both $A$ and $B$ are
invertible, and in this case
\begin{equation}\label{eq:invuptr}
\begin{bmatrix} A & C\\
0 & B
\end{bmatrix}^{-1}=\begin{bmatrix} A^{-1} & -A^{-1}CB^{-1}\\
0 & B^{-1}
\end{bmatrix}.
\end{equation}
Now, let $X\in\Omega^\text{inv}_n\subseteq\Omega_n$,
$Y\in\Omega^\text{inv}_m\subseteq\Omega_m$, and
$Z\in\rmat{\module{M}}{n}{m}$. Then there exists an invertible
$r\in\ring$ such that $\begin{bmatrix} X & rZ\\
0 & Y
\end{bmatrix}\in\Omega_{n+m}$. We have
\begin{equation*}
f\left(\begin{bmatrix} X & rZ\\
0 & Y
\end{bmatrix}\right)=\begin{bmatrix} f(X) & \Delta_R f(X,Y)(rZ)\\
0 & f(Y)
\end{bmatrix}.
\end{equation*} Since $f(X)$ and $f(Y)$ are invertible, so is the
right-hand side. Therefore, $\begin{bmatrix} X & rZ\\
0 & Y
\end{bmatrix}\!\in\!\Omega^\text{inv}_{n+m}$, and $\Omega^\text{inv}$
is right admissible.

By \eqref{eq:invuptr},
$$\Delta_Rf^{-1}(X,Y)(rZ)=-f(X)^{-1}\Delta_Rf(X,Y)(rZ)f(Y)^{-1},$$
which implies the required formula for $\Delta_Rf^{-1}(X,Y)(Z)$ by
linearity.

The ``left" statements can be proved analogously, or can be
deduced from the ``right" statements by a transposition argument
(see Proposition \ref{prop:trans}) using the fact that
$\left(f^{-1}\right)^\top=\left(f^\top\right)^{-1}$. In the case
where $\Omega$ is both right and left admissible, we can also use
a symmetry argument (Proposition \ref{prop:RLsym}).

Alternatively, we could also use equality
\eqref{eq:delta=delta_ext} and its left analogue, and Proposition
\ref{prop:adm-env} which guarantees that the similarity envelope
$\widetilde{\Omega}$ of the right (resp., left) admissible nc set
$\Omega$ is right (resp., left) admissible; see Remark
 \ref{rem:alternative} and Appendix \ref{app}. Then, similar to the argument
 in the proof above, so is $\widetilde{\Omega}^{\rm
 inv}=\widetilde{\Omega^{\rm inv}}$, and then the rest of the
 proof above works with $r=1$.
\end{proof}
\subsection{}\label{subsub:compos} Let $f\colon \Omega \to
\ncspace{\module{N}}$ and $g\colon \Lambda\to\ncspace{\module{O}}$
be nc functions on right (respectively, left) admissible nc sets
$\Omega\subseteq\ncspace{\module{M}}$ and
$\Lambda\subseteq\ncspace{\module{N}}$, such that
$f(\Omega)\subseteq\Lambda$. Then the composition $g\circ f\colon
\Omega \to \ncspace{\module{O}}$ is a nc function, and
$$\Delta_R(g\circ f)(X,Y)(Z)=\Delta_R g(f(X),f(Y))(\Delta_R f(X,Y)(Z))$$
for all $n,m\in\mathbb{N}$, $X\in\Omega_n$, $Y\in\Omega_m$, and
$Z\in\rmat{\module{M}}{n}{m}$ (respectively, $$\Delta_L(g\circ
f)(X,Y)(Z)=\Delta_L g(f(X),f(Y))(\Delta_L f(X,Y)(Z))$$ for all
$n,m\in\mathbb{N}$, $X\in\Omega_n$, $Y\in\Omega_m$, and
$Z\in\rmat{\module{M}}{m}{n}$).
\begin{proof}
It is easy to see that $g\circ f$ is a nc function.

By \eqref{eq:delta=delta_ext} applied to $f$, $g$, and $g\circ f$,
we may assume, without loss of
 generality,
  that $\Omega$ is similarity
 invariant, i.e., $\widetilde{\Omega}=\Omega$; see Remark
 \ref{rem:alternative} and Appendix \ref{app}.
  Then, by Proposition \ref{prop:adm-env}, for
  arbitrary $n,m\in\mathbb{N}$, $X\in\Omega_n$, $Y\in\Omega_m$,
  and $Z\in\rmat{\module{M}}{n}{m}$ one has
$\begin{bmatrix} X & Z\\
0 & Y
\end{bmatrix}\in\Omega_{n+m}$.
On one hand, \begin{equation}\label{eq:1h} (g\circ f)
\left(\begin{bmatrix}
X & Z\\
0 & Y \end{bmatrix}\right)=\begin{bmatrix}
(g\circ f) (X) & \Delta_R(g\circ f)(X,Y)(Z)\\
0 &  (g\circ f)(Y) \end{bmatrix}. \end{equation}
 On the other
hand,
\begin{multline}\label{eq:2h}
g\left(f\left(\begin{bmatrix}
X & Z\\
0 & Y \end{bmatrix}\right)\right)
=g\left(\begin{bmatrix} f (X) & \Delta_Rf(X,Y)(Z)\\
0 &  f(Y) \end{bmatrix}\right)\\
=\begin{bmatrix} g(f(X)) & \Delta_R g(f(X),f(Y))(\Delta_Rf(X,Y)(Z))\\
0 &  g(f(Y)) \end{bmatrix}.
\end{multline}
Comparing the $(1,2)$ blocks on the right hand sides of
\eqref{eq:1h} and \eqref{eq:2h}, we obtain the required formula
for $\Delta_R(g\circ f)(X,Y)(Z)$.

The formula for $\Delta_L(g\circ f)(X,Y)(Z)$ can be proved
analogously, or by a transposition argument, or (when $\Omega$ is
both right and left admissible) by a symmetry argument.
\end{proof}

\section{First order difference formulae}\label{subsec:Lagrange}
We establish now formulae which justify the claim made in Section
\ref{subsec:difdif} that $\Delta_R$ and $\Delta_L$ play the role
of nc finite difference operators of the first order. We will see
in the sequel that these formulae are nc analogues of zeroth order
Brook Taylor expansions.
\begin{thm}\label{thm:Lagrange}
Let $f\colon \Omega \to\ncspace{\module{N}}$ be a nc function on a
right admissible nc set $\Omega$. Then for all $n\in\mathbb{N}$,
and arbitrary $X,Y\in\Omega_n$,
\begin{align}\label{eq:RightLagr}
f(X)-f(Y) &= \Delta_Rf(Y,X)(X-Y)\\
\intertext{and}
 f(X)-f(Y)  &= \Delta_Rf(X,Y)(X-Y). \label{eq:RightLagr'}
\end{align}

Analogously, let $f\colon \Omega \to\ncspace{\module{N}}$ be a nc
function on a left admissible nc set $\Omega$. Then for all
$n\in\mathbb{N}$, and arbitrary $X,Y\in\Omega_n$,
\begin{align}\label{eq:LeftLagr}
f(X)-f(Y) &= \Delta_Lf(Y,X)(X-Y)\\
\intertext{and}
 f(X)-f(Y)      &= \Delta_Lf(X,Y)(X-Y). \label{eq:LeftLagr'}
\end{align}
\end{thm}
In fact, Theorem \ref{thm:Lagrange} is a special case of the following more general finite difference formulae.
\begin{thm}\label{thm:gen-Lagrange}
Let $f\colon \Omega \to\ncspace{\module{N}}$ be a nc function on a right admissible nc set $\Omega$. Then for all
$n,m\in\mathbb{N}$, and arbitrary $X\in\Omega_n$, $Y\in\Omega_m$ and $S \in \rmat{\ring}{m}{n}$,
\begin{equation}\label{eq:gen-RightLagr}
Sf(X)-f(Y)S = \Delta_Rf(Y,X)(SX-YS).
\end{equation}

Analogously, let $f\colon \Omega \to\ncspace{\module{N}}$ be a nc function on a left admissible nc set $\Omega$.
Then for all $n,m\in\mathbb{N}$, and arbitrary $X\in\Omega_n$, $Y\in\Omega_m$ and $S \in \rmat{\ring}{m}{n}$,
\begin{equation}\label{eq:gen-LeftLagr}
Sf(X)-f(Y)S = \Delta_Lf(X,Y)(SX-YS).
\end{equation}

\end{thm}
\begin{proof}
We prove the ``right" statement first. By
\eqref{eq:delta=delta_ext} and Proposition \ref{prop:ncfun-ext},
we may assume, without loss of
 generality,
  that $\Omega$ is similarity
 invariant, i.e., $\widetilde{\Omega}=\Omega$; see Remark
 \ref{rem:alternative} and Appendix \ref{app}.
  Then, by Proposition \ref{prop:adm-env}, for
  arbitrary $n,m\in\mathbb{N}$, $X\in\Omega_n$, $Y\in\Omega_m$,
  and $S\in\rmat{\ring}{m}{n}$ one has
$\begin{bmatrix} Y & SX-YS\\
0 & X
\end{bmatrix}\in\Omega_{m+n}$. Clearly,
$$\begin{bmatrix} Y & SX-YS\\
0 & X
\end{bmatrix}\begin{bmatrix}
S\\
I_n
\end{bmatrix}=\begin{bmatrix}
S\\
I_n
\end{bmatrix}X.
$$
Therefore,
$$f\left(\begin{bmatrix} Y & SX-YS\\
0 & X
\end{bmatrix}\right)\begin{bmatrix}
S\\
I_n
\end{bmatrix}=\begin{bmatrix}
S\\
I_n
\end{bmatrix}f(X),
$$
i.e.,
$$\begin{bmatrix} f(Y) & \Delta_Rf(Y,X)(SX-YS)\\
0 & f(X)
\end{bmatrix}\begin{bmatrix}
S\\
I_n
\end{bmatrix}=\begin{bmatrix}
S\\
I_n
\end{bmatrix}f(X).
$$
Comparing the (1,1) block entries in the matrix products in the
right-hand side and in the left-hand side, we obtain
$$f(Y)S+\Delta_Rf(Y,X)(SX-YS))=Sf(X),$$
which yields \eqref{eq:gen-RightLagr}.

The ``left" statements can be proved analogously using the intertwining
$$\begin{bmatrix}
S & -I_m
\end{bmatrix}\begin{bmatrix}
X & 0\\
SX-YS & Y
\end{bmatrix}=Y\begin{bmatrix}
S & -I_m
\end{bmatrix},$$
 or by a transposition argument, or (in the case where $\Omega$ is both right and left admissible) by a symmetry
argument.
\end{proof}
\begin{rem}
\eqref{eq:RightLagr} and \eqref{eq:RightLagr'} in Theorem \ref{thm:Lagrange} are both obtained from
\eqref{eq:gen-RightLagr} by writing $X-Y=X\cdot I_n-I_n\cdot Y$ and $X-Y=I_n\cdot X-Y\cdot I_n$ respectively. We
note that \eqref{eq:RightLagr} and \eqref{eq:RightLagr'} are obtained from each other by interchanging $X$ and
$Y$. In particular, we have
$$\Delta_Rf(Y,X)(X-Y)=\Delta_Rf(X,Y)(X-Y).$$
This follows also directly from the intertwining
$$\begin{bmatrix} Y & X-Y\\
0 & X
\end{bmatrix}\begin{bmatrix}
I_n & 0\\
I_n & I_n
\end{bmatrix}=\begin{bmatrix}
I_n & 0\\
I_n & I_n
\end{bmatrix}\begin{bmatrix} X & X-Y\\
0 & Y
\end{bmatrix},$$
using, if necessary, the similarity envelope $\widetilde{\Omega}$
of the nc set $\Omega$ and the canonical extension $\widetilde{f}$
of the nc function $f$ by comparing the (1,1) block entries in the
matrix products. Notice that comparing the (1,2) block entries
yields \eqref{eq:RightLagr}, and comparing the (2,1) block entries
yields \eqref{eq:RightLagr'}.

An analogous remark applies to \eqref{eq:LeftLagr} and
\eqref{eq:LeftLagr'}, and the resulting equality
$$\Delta_Lf(Y,X)(X-Y)=\Delta_Lf(X,Y)(X-Y).$$
\end{rem}
\begin{rem}\label{rem:newdef_ncfun}
Clearly, each of the formulae \eqref{eq:gen-RightLagr} and \eqref{eq:gen-LeftLagr} generalizes the condition
\eqref{eq:simint}.
\end{rem}

\begin{rem}\label{rem:deriv}
Let $\ring=\mathbb{R}$ or $\ring=\mathbb{C}$. Setting $X=Y+tZ$
(with $t\in\mathbb{R}$ or $t\in\mathbb{C}$), we obtain from
Theorem \ref{thm:Lagrange} that
$$f(Y+tZ)-f(Y)=t\Delta_Rf(Y,Y+tZ)(Z).$$
Under appropriate continuity conditions, it follows that
$\Delta_Rf(Y,Y)(Z)$ is the directional derivative of $f$ at $Y$ in
the direction $Z$, i.e., $\Delta_Rf(Y,Y)$ is the differential of
$f$ at $Y$.
\end{rem}

\section{Properties of $\Delta_Rf(X,Y)$ and $\Delta_Lf(X,Y)$ as
functions of $X$ and $Y$}\label{subsec:proper} A nc function $f$
possesses certain key properties: it respects direct sums,
similarities, and intertwinings. We will show now that these
properties are inherited, in an appropriate form, by
$\Delta_Rf(X,Y)(\cdot)$ and by $\Delta_Lf(X,Y)(\cdot)$ as
functions of $X$ and $Y$.
\begin{prop}\label{prop:Rproper}
Let $f\colon\Omega\to\ncspace{\module{N}}$ be a nc function on a
right admissible nc set $\Omega\subseteq\ncspace{\module{M}}$.
Then:
\begin{equation}\tag{1X}\label{eq:R1X} \Delta_Rf\left(X'\oplus X'',Y\right)
\left(\col\left[Z',Z''\right]\right)
=\col\left[\Delta_Rf\left(X',Y\right)\left(Z'\right),
 \Delta_Rf\left(X'',Y\right)\left(Z''\right)\right]
\end{equation}
 for $n',n'',m\in\mathbb{N}$, $X'\in\Omega_{n'}$,
$X''\in\Omega_{n''}$, $Y\in\Omega_{m}$,
$Z'\in\rmat{\module{M}}{n'}{m}$,
$Z''\in\rmat{\module{M}}{n''}{m}$;
\begin{equation}\tag{1Y}\label{eq:R1Y}
\Delta_Rf\left(X,Y'\oplus
Y''\right)\left(\row\left[Z',Z''\right]\right)=\row\left[\Delta_Rf\left(X,Y'\right)\left(Z'\right),
 \Delta_Rf\left(X,Y''\right)\left(Z''\right)\right]
\end{equation}
 for $n,m',m''\in\mathbb{N}$, $X\in\Omega_{n}$,
$Y'\in\Omega_{m'}$, $Y''\in\Omega_{m''}$,
$Z'\in\rmat{\module{M}}{n}{m'}$,
$Z''\in\rmat{\module{M}}{n}{m''}$;
\begin{equation}\tag{2X}\label{eq:R2X}
\Delta_Rf\left(TXT^{-1},Y\right)(TZ)=T\,\Delta_Rf(X,Y)(Z)
\end{equation}
 for
$n,m\in\mathbb{N}$, $X\in\Omega_n$, $Y\in\Omega_m$,
$Z\in\rmat{\module{M}}{n}{m}$, and an invertible
$T\in\mat{\ring}{n}$ such $TXT^{-1}\in\Omega_n$;
\begin{equation}\tag{2Y}\label{eq:R2Y}
\Delta_Rf\left(X,SYS^{-1}\right)\left(ZS^{-1}\right)=\Delta_Rf(X,Y)(Z)\,S^{-1}
\end{equation}
 for
$n,m\in\mathbb{N}$, $X\in\Omega_n$, $Y\in\Omega_m$,
$Z\in\rmat{\module{M}}{n}{m}$, and an invertible
$S\in\mat{\ring}{m}$ such that $SYS^{-1}\in\Omega_m$;
\begin{equation}\tag{3X}\label{eq:R3X} \text{If}\quad
TX=\tilde{X}T\quad \text{then}\quad
T\,\Delta_Rf(X,Y)(Z)=\Delta_Rf(\tilde{X},Y)(TZ)
\end{equation}
 for
$n,\tilde{n},m\in\mathbb{N}$, $X\in\Omega_n$,
$\tilde{X}\in\Omega_{\tilde{n}}$, $Y\in\Omega_m$,
 $Z\in\rmat{\module{M}}{n}{m}$,
and $T\in\rmat{\ring}{\tilde{n}}{n}$;
\begin{equation}\tag{3Y}\label{eq:R3Y}
\text{If}\quad YS=S\tilde{Y}\quad \text{then}\quad
\Delta_Rf(X,Y)(Z)\,S=\Delta_Rf(X,\tilde{Y})(ZS)
\end{equation}
 for
$n,m,\tilde{m}\in\mathbb{N}$, $X\in\Omega_n$,
 $Y\in\Omega_m$,
$\tilde{Y}\in\Omega_{\tilde{m}}$, $Z\in\rmat{\module{M}}{n}{m}$,
and $S\in\rmat{\ring}{m}{\tilde{m}}$.
\end{prop}
\begin{proof}
By \eqref{eq:delta=delta_ext}, we may assume, without loss of
 generality,
  that $\Omega$ is similarity
 invariant, i.e., $\widetilde{\Omega}=\Omega$; see Remark
 \ref{rem:alternative} and Appendix \ref{app}.
  Then, by Proposition \ref{prop:adm-env}, for
  arbitrary $n',n'',m\in\mathbb{N}$, $X'\in\Omega_{n'}$, $X''\in\Omega_{n''}$, $Y\in\Omega_m$,
   $Z'\in\rmat{\ring}{n'}{m}$, and $Z''\in\rmat{\ring}{n''}{m}$,
$$\begin{bmatrix} X' & Z'\\
0 & Y
\end{bmatrix}\in\Omega_{n'+m},\ \begin{bmatrix} X'' & Z''\\
0 & Y
\end{bmatrix}\in\Omega_{n''+m},  \ \begin{bmatrix} X' & 0 & Z'\\
0 & X'' & Z''\\
0 & 0 & Y
\end{bmatrix}\in\Omega_{n'+n''+m}.$$
We have
\begin{equation*}
\begin{bmatrix}  I_{n'} & 0 & 0\\
0 & 0 & I_m\\
0 & I_{n''} & 0\\
0 & 0 & I_m
\end{bmatrix}
\begin{bmatrix} X' & 0 & Z'\\
0 & X'' & Z''\\
0 & 0 & Y
\end{bmatrix}
=\begin{bmatrix} X' & Z' & 0 & 0\\
0 & Y & 0 & 0\\
0 & 0 & X'' & Z''\\
0 & 0 & 0 & Y
\end{bmatrix}\begin{bmatrix}  I_{n'} & 0 & 0\\
0 & 0 & I_m\\
0 & I_{n''} & 0\\
0 & 0 & I_m
\end{bmatrix}.
\end{equation*}
By Proposition \ref{prop:simint},
\begin{multline*}
\begin{bmatrix}  I_{n'} & 0 & 0\\
0 & 0 & I_m\\
0 & I_{n''} & 0\\
0 & 0 & I_m
\end{bmatrix}
f\left(\begin{bmatrix} X' & 0 & Z'\\
0 & X'' & Z''\\
0 & 0 & Y
\end{bmatrix}\right)\\
=f\left(\begin{bmatrix} X' & Z' & 0 & 0\\
0 & Y & 0 & 0\\
0 & 0 & X'' & Z''\\
0 & 0 & 0 & Y
\end{bmatrix}\right)\begin{bmatrix}  I_{n'} & 0 & 0\\
0 & 0 & I_m\\
0 & I_{n''} & 0\\
0 & 0 & I_m
\end{bmatrix},
\end{multline*}
i.e.,
\begin{multline*}
\begin{bmatrix}  I_{n'} & 0 & 0\\
0 & 0 & I_m\\
0 & I_{n''} & 0\\
0 & 0 & I_m
\end{bmatrix}
\begin{bmatrix} \begin{matrix}
f(X')\\
0
\end{matrix} & \begin{matrix} 0\\
 f(X'')\end{matrix} & \Delta_Rf(X'\oplus
X'',Y)(\col[Z',Z''])\\
0 & 0  & f(Y)
\end{bmatrix}\\
=\begin{bmatrix} f(X') & \Delta_Rf(X',Y)(Z') & 0 & 0\\
0 & f(Y) & 0 & 0\\
0 & 0 & f(X'') & \Delta_Rf(X'',Y)(Z'')\\
0 & 0 & 0 & f(Y)
\end{bmatrix}
 \begin{bmatrix}  I_{n'} & 0 & 0\\
0 & 0 & I_m\\
0 & I_{n''} & 0\\
0 & 0 & I_m
\end{bmatrix}.
\end{multline*}
Comparing the columns formed by the (1,3) and the (3,3) block
entries in the matrix products in the left-hand side and in the
right-hand side, we obtain \eqref{eq:R1X}.

The statement \eqref{eq:R1Y} is proved analogously, using the
intertwining
\begin{equation*}
\begin{bmatrix} X & Z' & Z''\\
0 & Y' & 0\\
0 & 0 & Y''
\end{bmatrix}\!\!\begin{bmatrix} I_n  & 0 & I_n & 0\\
0 & I_{m'} & 0 & 0\\
0 & 0 & 0 & I_{m''} \\
\end{bmatrix}
=\begin{bmatrix} I_n  & 0 & I_n & 0\\
0 & I_{m'} & 0 & 0\\
0 & 0 & 0 & I_{m''} \\
\end{bmatrix}\!\!\begin{bmatrix} X & Z' & 0 & 0\\
0 & Y' & 0 & 0\\
0 & 0 & X & Z''\\
0 & 0 & 0 & Y''
\end{bmatrix}.
\end{equation*}

Next, for arbitrary $n,m\in\mathbb{N}$, $X\in\Omega_n$, $Y\in\Omega_m$, $Z\in\rmat{\module{M}}{n}{m}$, and an
invertible $T\in\mat{\ring}{n}$ we have  $\begin{bmatrix}  X & Z \\
0 & Y\end{bmatrix},\ \begin{bmatrix}  TXT^{-1} & TZ \\
0 & Y\end{bmatrix}\in\Omega_{n+m}$. Clearly,
$$\begin{bmatrix}  TXT^{-1} & TZ \\
0 & Y\end{bmatrix}=\begin{bmatrix} T & 0 \\
0 & I_m\end{bmatrix}\begin{bmatrix}  X & Z \\
0 & Y\end{bmatrix}\begin{bmatrix} T & 0 \\
0 & I_m\end{bmatrix}^{-1}.$$ Then
$$f\left(\begin{bmatrix}  TXT^{-1} & TZ \\
0 & Y\end{bmatrix}\right)=\begin{bmatrix}  T & 0 \\
0 & I_m\end{bmatrix}f\left(\begin{bmatrix}  X & Z \\
0 & Y\end{bmatrix}\right)\begin{bmatrix}  T & 0 \\
0 & I_m\end{bmatrix}^{-1},$$ i.e.,
\begin{multline*}\begin{bmatrix}  f\left(TXT^{-1}\right) & \Delta_Rf\left(TXT^{-1},Y\right)(TZ) \\
0 & f(Y)\end{bmatrix}\\
=\begin{bmatrix}  T & 0 \\
0 & I_m\end{bmatrix}\begin{bmatrix}  f(X) & \Delta_Rf(X,Y)(Z) \\
0 & f(Y)\end{bmatrix}\begin{bmatrix}  T & 0 \\
0 & I_m\end{bmatrix}^{-1}.
\end{multline*}
Comparing the (1,2) block entries in the matrix products in the
left-hand side and in the right-hand side,  we obtain
\eqref{eq:R2X}. The statement \eqref{eq:R2Y} is obtained
analogously, using the similarity
$$\begin{bmatrix}  X & ZS^{-1} \\
0 & SYS^{-1}\end{bmatrix}=\begin{bmatrix}  I_n & 0 \\
0 & S\end{bmatrix}\begin{bmatrix}  X & Z \\
0 & Y\end{bmatrix}\begin{bmatrix}  I_n & 0 \\
0 & S\end{bmatrix}^{-1}.$$

The statements \eqref{eq:R3X} and \eqref{eq:R3Y} are obtained
using the intertwinings
$$\begin{bmatrix}  T & 0 \\
0 & I_m\end{bmatrix}\begin{bmatrix}  X & Z \\
0 & Y\end{bmatrix}=\begin{bmatrix}  \tilde{X} & TZ \\
0 & Y\end{bmatrix}\begin{bmatrix}  T & 0 \\
0 & I_m\end{bmatrix}$$ and
$$\begin{bmatrix}  X & Z \\
0 & Y\end{bmatrix}\begin{bmatrix}  I_n & 0 \\
0 & S\end{bmatrix}=\begin{bmatrix}  I_n & 0 \\
0 & S\end{bmatrix}\begin{bmatrix}  X & ZS \\
0 & \tilde{Y}\end{bmatrix}.$$
\end{proof}
\begin{rem}\label{rem:Rpropequiv}
Much as in Proposition \ref{prop:simint}, the respect of
intertwinings (\eqref{eq:R3X} \& \eqref{eq:R3Y}) is equivalent to
the respect of direct sums (\eqref{eq:R1X} \& \eqref{eq:R1Y})
together with the respect of similarities (\eqref{eq:R2X} \&
\eqref{eq:R2Y}). We will show this equivalence later in
Proposition \ref{prop:simint_k}.
\end{rem}

Combining the X and  the Y properties in Proposition
\ref{prop:Rproper}, one obtains an equivalent joint formulation.
\begin{prop}\label{prop:Rproper'}
Let $f\colon\Omega\to\ncspace{\module{N}}$ be a nc function on a
right admissible nc set $\Omega\subseteq\ncspace{\module{M}}$.
Then:
\begin{multline}\label{eq:Rsum}
\Delta_Rf\left(X'\oplus X'',Y'\oplus
Y''\right)\left(\begin{bmatrix} Z^{\prime,\prime} &
Z^{\prime,\prime\prime}\\
Z^{\prime\prime,\prime} & Z^{\prime\prime,\prime\prime}
\end{bmatrix}\right)\\
=\begin{bmatrix}\Delta_Rf\left(X',Y'\right)\left(Z^{\prime,\prime}\right) &
\Delta_Rf\left(X',Y''\right)\left(Z^{\prime,\prime\prime}\right)\\
\Delta_Rf\left(X'',Y'\right)\left(Z^{\prime\prime,\prime}\right) &
\Delta_Rf\left(X'',Y''\right)\left(Z^{\prime\prime,\prime\prime}\right)
\end{bmatrix}
\end{multline}
for $n',m'\in\mathbb{N}$, $n'',m''\in\mathbb{Z}_+$,
$X'\in\Omega_{n'}$, $X''\in\Omega_{n''}$, $Y'\in\Omega_{m'}$,
$Y''\in\Omega_{m''}$, $\begin{bmatrix} Z^{\prime,\prime} &
Z^{\prime,\prime\prime}\\
Z^{\prime\prime,\prime} & Z^{\prime\prime,\prime\prime}
\end{bmatrix}\in\rmat{\module{M}}{(n'+n'')}{(m'+m'')}$, with block
entries of appropriate sizes, and if either $n''$ or $m''$ is $0$
then the corresponding block entry is void;
\begin{equation}\label{eq:Rsim}
\Delta_Rf\left(TXT^{-1},SYS^{-1}\right)\left(TZS^{-1}\right)=T\,\Delta_Rf(X,Y)(Z)\,S^{-1}
\end{equation}
for $n,m\in\mathbb{N}$, $X\in\Omega_n$, $Y\in\Omega_m$,
$Z\in\rmat{\module{M}}{n}{m}$, and invertible
$T\in\mat{\ring}{n}$, $S\in\mat{\ring}{m}$ such that
$TXT^{-1}\in\Omega_n$, $SYS^{-1}\in\Omega_m$;
\begin{equation}\label{eq:Rint}
TX=\tilde{X}T,\, YS=S\tilde{Y} \Longrightarrow
T\,\Delta_Rf(X,Y)(Z)\,S=\Delta_Rf\left(\tilde{X},\tilde{Y}\right)(TZS)
\end{equation}
for $n,\tilde{n},m,\tilde{m}\in\mathbb{N}$, $X\in\Omega_n$,
$\tilde{X}\in\Omega_{\tilde{n}}$, $Y\in\Omega_m$,
$\tilde{Y}\in\Omega_{\tilde{m}}$, $Z\in\rmat{\module{M}}{n}{m}$,
and $T\in\rmat{\ring}{\tilde{n}}{n}$,
$S\in\rmat{\ring}{m}{\tilde{m}}$.
\end{prop}
In fact, Proposition \ref{prop:Rproper'} is a special case of
Proposition \ref{prop:joint_k}, which will be proved later.

The ``left" counterparts of Propositions \ref{prop:Rproper} and
\ref{prop:Rproper'} are as follows. They can either be proved
analogously or deduced from the corresponding ``right" statements
by a transposition argument, or (in the case where the nc set
$\Omega$ is both right and left admissible) by a symmetry
argument.
\begin{prop}\label{prop:Lproper}
Let $f\colon\Omega\to\ncspace{\module{N}}$ be a nc function on a
left admissible nc set $\Omega\subseteq\ncspace{\module{M}}$.
Then:
\begin{equation}\tag{1X}\label{eq:L1X}
\Delta_Lf(X'\oplus
X'',Y)(\row[Z',Z''])=\row[\Delta_Lf(X',Y)(Z'),
 \Delta_Lf(X'',Y)(Z'')]
\end{equation}
 for $n',n'',m\in\mathbb{N}$, $X'\in\Omega_{n'}$,
$X''\in\Omega_{n''}$, $Y\in\Omega_{m}$,
$Z'\in\rmat{\module{M}}{m}{n'}$,
$Z''\in\rmat{\module{M}}{m}{n''}$;
\begin{equation}\tag{1Y}\label{eq:L1Y}
\Delta_Lf(X,Y'\oplus
Y'')(\col[Z',Z''])=\col[\Delta_Lf(X,Y')(Z'),
 \Delta_Lf(X,Y'')(Z'')]
\end{equation}
 for $n,m',m''\in\mathbb{N}$, $X\in\Omega_{n}$,
$Y'\in\Omega_{m'}$, $Y''\in\Omega_{m''}$,
$Z'\in\rmat{\module{M}}{m'}{n}$,
$Z''\in\rmat{\module{M}}{m''}{n}$;
\begin{equation}\tag{2X}\label{eq:L2X}
\Delta_Lf\left(TXT^{-1},Y\right)\left(ZT^{-1}\right)=\Delta_Lf(X,Y)(Z)\,T^{-1}
\end{equation}
 for
$n,m\in\mathbb{N}$, $X\in\Omega_n$, $Y\in\Omega_m$,
$Z\in\rmat{\module{M}}{m}{n}$, and an invertible
$T\in\mat{\ring}{n}$ such $TXT^{-1}\in\Omega_n$;
\begin{equation}\tag{2Y}\label{eq:L2Y}
\Delta_Lf\left(X,SYS^{-1}\right)(SZ)=S\,\Delta_Lf(X,Y)(Z)
\end{equation}
 for
$n,m\in\mathbb{N}$, $X\in\Omega_n$, $Y\in\Omega_m$,
$Z\in\rmat{\module{M}}{m}{n}$, and an invertible
$S\in\mat{\ring}{m}$ such that $SYS^{-1}\in\Omega_m$;
\begin{equation}\tag{3X}\label{eq:L3X}
\text{If}\quad XT=T\tilde{X}\quad \text{then}\quad
\Delta_Lf(X,Y)(Z)\,T=\Delta_Lf(\tilde{X},Y)(ZT)
\end{equation}
 for
$n,\tilde{n},m\in\mathbb{N}$, $X\in\Omega_n$,
$\tilde{X}\in\Omega_{\tilde{n}}$, $Y\in\Omega_m$,
 $Z\in\rmat{\module{M}}{m}{n}$,
and $T\in\rmat{\ring}{n}{\tilde{n}}$;
\begin{equation}\tag{3Y}\label{eq:L3Y}
\text{If}\quad SY=\tilde{Y}S\quad \text{then}\quad
S\,\Delta_Lf(X,Y)(Z)=\Delta_Lf(X,\tilde{Y})(SZ)
\end{equation}
 for
$n,m,\tilde{m}\in\mathbb{N}$, $X\in\Omega_n$,
 $Y\in\Omega_m$,
$\tilde{Y}\in\Omega_{\tilde{m}}$, $Z\in\rmat{\module{M}}{m}{n}$,
and $S\in\rmat{\ring}{\tilde{m}}{m}$.
\end{prop}
\begin{prop}\label{prop:Lproper'}
Let $f\colon\Omega\to\ncspace{\module{N}}$ be a nc function on a
left admissible nc set $\Omega\subseteq\ncspace{\module{M}}$.
Then:
\begin{multline}\label{eq:Lsum}
\Delta_Lf(X'\oplus X'',Y'\oplus Y'')\left(\begin{bmatrix}
Z^{\prime,\prime} &
Z^{\prime,\prime\prime}\\
Z^{\prime\prime,\prime} & Z^{\prime\prime,\prime\prime}
\end{bmatrix}\right)\\
=\begin{bmatrix}\Delta_Lf(X',Y')(Z^{\prime,\prime}) & \Delta_Lf(X'',Y')(Z^{\prime,\prime\prime})\\
\Delta_Lf(X',Y'')(Z^{\prime\prime,\prime}) &
\Delta_Lf(X'',Y'')(Z^{\prime\prime,\prime\prime})
\end{bmatrix}
\end{multline}
for $n',m'\in\mathbb{N}$, $n'',m''\in\mathbb{Z}_+$,
$X'\in\Omega_{n'}$, $X''\in\Omega_{n''}$, $Y'\in\Omega_{m'}$,
$Y''\in\Omega_{m''}$, $\begin{bmatrix} Z^{\prime,\prime} &
Z^{\prime,\prime\prime}\\
Z^{\prime\prime,\prime} & Z^{\prime\prime,\prime\prime}
\end{bmatrix}\in\rmat{\module{M}}{(m'+m'')}{(n'+n'')}$, with block
entries of appropriate sizes, and if either $n''$ or $m''$ is $0$
then the corresponding block entry is void;
\begin{equation}\label{eq:Lsim}
\Delta_Lf\left(TXT^{-1},SYS^{-1}\right)\left(SZT^{-1}\right)=S\,\Delta_Lf(X,Y)(Z)\,T^{-1}
\end{equation}
for $n,m\in\mathbb{N}$, $X\in\Omega_n$, $Y\in\Omega_m$,
$Z\in\rmat{\module{M}}{m}{n}$, and invertible
$T\in\mat{\ring}{n}$, $S\in\mat{\ring}{m}$ such that
$TXT^{-1}\in\Omega_n$, $SYS^{-1}\in\Omega_m$;
\begin{equation}\label{eq:Lint}
XT=T\tilde{X},\, SY=\tilde{Y}S \Longrightarrow
S\,\Delta_Lf(X,Y)(Z)\,T=\Delta_Lf(\tilde{X},\tilde{Y})(SZT)
\end{equation}
for $n,\tilde{n},m,\tilde{m}\in\mathbb{N}$, $X\in\Omega_n$,
$\tilde{X}\in\Omega_{\tilde{n}}$, $Y\in\Omega_m$,
$\tilde{Y}\in\Omega_{\tilde{m}}$, $Z\in\rmat{\module{M}}{m}{n}$,
and $T\in\rmat{\ring}{n}{\tilde{n}}$,
$S\in\rmat{\ring}{\tilde{m}}{m}$.
\end{prop}

\section{Directional nc difference-differential
operators}\label{subsec:dir_difdif} Alongside the full
differ\-ence-differential operators introduced above, we will also
consider their \emph{directional} versions. \index{directional nc
difference-differential operator} Let
$\Omega\subseteq\ncspace{\module{M}}$ be a right (resp., left)
admissible nc set, and let $f\colon\Omega\to\ncspace{\module{N}}$
be a nc function. Given $\mu\in\module{M}$, we define for all
$n,m\in\mathbb{N}$, $X\in\Omega_n$, $Y\in\Omega_m$ the linear
mapping
$\Delta_{R,\mu}f(X,Y)\colon\rmat{\ring}{n}{m}\to\rmat{\module{N}}{n}{m}$
(resp.,
$\Delta_{L,\mu}f(X,Y)\colon\rmat{\ring}{m}{n}\to\rmat{\module{N}}{m}{n}$)
by \index{$\Delta_{R,\mu}f(X,Y)$} \index{$\Delta_{L,\mu}f(X,Y)$}
\begin{equation}\label{eq:r_dir}
 \Delta_{R,\mu}f(X,Y)(A)=\Delta_Rf(X,Y)(A\mu)
 \end{equation}
 and, respectively,
\begin{equation}\label{eq:l_dir}
 \Delta_{L,\mu}f(X,Y)(A)=\Delta_Lf(X,Y)(A\mu).
 \end{equation}
Here for a matrix $A$ over $\ring$ and $\mu\in\module{M}$, the
product $A\mu$ is a matrix over  $\module{M}$ of the same size and
$(A\mu)_{ij}=A_{ij}\mu$ for all $i$ and $j$.

In the special case where $\module{M}=\ring^d$, we define the
\emph{$j$-th right (resp.,left) partial nc difference-differential
operator} \index{right partial nc difference-differential
operator} \index{left partial nc difference-differential operator}
by $\Delta_{R,j}=\Delta_{R,e_j}$ (resp.,
$\Delta_{L,j}=\Delta_{L,e_j}$) where $e_j$ \index{$e_j$} is the
$j$-th standard basis vector in $\ring^d$, $j=1,\ldots, d$.
\index{$\Delta_{R,j}$} \index{$\Delta_{L,j}$}
 By linearity, we have
\begin{equation}\label{eq:r_difdecomp}
\Delta_Rf(X,Y)(Z)=\sum_{j=1}^d\Delta_{R,j}f(X,Y)(Z_j)
\end{equation}
and
\begin{equation}\label{eq:l_difdecomp}
\Delta_Lf(X,Y)(Z)=\sum_{j=1}^d\Delta_{L,j}f(X,Y)(Z_j).
\end{equation}

For the basic rules of nc calculus, \ref{subsub:const},
\ref{subsub:sum}, \ref{subsub:prod}, and \ref{subsub:inv} have
exactly the same form for the directional nc
difference-differential operators. In the setting of
\ref{subsub:coord}, $$\Delta_{R,\mu}l(X,Y)(A)=Al(\mu)\quad
\text{and}\quad \Delta_{L,\mu}l(X,Y)(A)=Al(\mu);$$ in particular,
in the case where $\module{M}=\ring^d$ we have
$$\Delta_{R,i}l_j(X,Y)(A)=\delta_{ij}A\quad \text{and}\quad
\Delta_{L,i}l_j(X,Y)(A)=\delta_{ij}A.$$ A version of
\ref{subsub:compos} for the right and left partial nc
difference-differential operators in the case where
$\module{M}=\ring^d$ and $\module{N}=\ring^\ell$ is
$$\Delta_{R,i}(g\circ
f)(X,Y)(A)=\sum_{j=1}^\ell\Delta_{R,j}g(f(X),f(Y))(\Delta_{R,i}f_j(X,Y)(A))$$
and
$$\Delta_{L,i}(g\circ
f)(X,Y)(A)=\sum_{j=1}^\ell\Delta_{L,j}g(f(X),f(Y))(\Delta_{L,i}f_j(X,Y)(A))$$
for $i=1,\ldots,d$,
 where $f(X)=f_1(X)e_1+\cdots+f_\ell(X)e_\ell$ with the
component nc functions $f_j\colon\Omega\to\ncspace{\ring}$.

The first order difference formulae of Theorem \ref{thm:Lagrange}
in the case where $\module{M}=\ring^d$ can be written in terms of
partial nc difference-differential operators. It follows from
\eqref{eq:RightLagr}--\eqref{eq:RightLagr'} and
\eqref{eq:r_difdecomp} that
\begin{equation*}
f(X)-f(Y)=\sum_{j=1}^d\Delta_{R,j}f(X,Y)(X_j-Y_j)=\sum_{j=1}^d\Delta_{R,j}f(Y,X)(X_j-Y_j).
\end{equation*}
Similarly, it follows from
\eqref{eq:LeftLagr}--\eqref{eq:LeftLagr'} and
\eqref{eq:l_difdecomp} that
\begin{equation*}
f(X)-f(Y)=\sum_{j=1}^d\Delta_{L,j}f(X,Y)(X_j-Y_j)=\sum_{j=1}^d\Delta_{L,j}f(Y,X)(X_j-Y_j).
\end{equation*}

For the properties of $\Delta_{R,\mu}f(X,Y)$ and
$\Delta_{L,\mu}f(X,Y)$ as functions of $X$ and $Y$, all the
statements of Section \ref{subsec:proper} have exactly the same
form.

\chapter{Higher order nc functions and their difference-differential calculus}\label{sec:difdif_k}
Our next goal is to define higher order differentials of a nc
function (more precisely, higher order nc difference-differential
operators). As we saw in Chapter \ref{sec:ncfun}, for a nc
function $f$ on a nc set $\Omega\subseteq\ncspace{\module{M}}$ to
a nc space $\ncspace{\module{N}}$, $\Delta_Rf(X,Y)(\cdot)$ is a
function of two arguments $X\in\Omega_n$ and $Y\in\Omega_m$ (for
all $n$ and $m$) with values linear mappings
$\rmat{\module{M}}{n}{m}\to\rmat{\module{N}}{n}{m}$. We also saw
that $\Delta_Rf(X,Y)(\cdot)$ respects (in an appropriate way)
direct sums, similarities, and intertwinings. This motivates the
appearance of classes
$\tclass{k}=\tclass{k}(\Omega^{(0)},\ldots,\Omega^{(k)};\ncspacej{\module{N}}{0},\ldots,\ncspacej{\module{N}}{k})$
of nc functions of order $k$ ($k=0,1,\ldots$) together with the
right nc difference-differential operators
\begin{multline*}
\Delta_R\colon\tclass{k}(\Omega^{(0)},\ldots,\Omega^{(k)};\ncspacej{\module{N}}{0},\ldots,\ncspacej{\module{N}}{k})\\
\to\tclass{k+1}(\Omega^{(0)},\ldots,\Omega^{(k)},\Omega^{(k)};\ncspacej{\module{N}}{0},\ldots,
\ncspacej{\module{N}}{k},\ncspacej{\module{M}}{k}).
\end{multline*} A nc function of order $k$ is a function of $k+1$
arguments, in nc sets
$\Omega^{(0)}\subseteq\ncspacej{\module{M}}{0}$, \ldots,
$\Omega^{(k)}\subseteq\ncspacej{\module{M}}{k}$, with values
certain $k$-linear mappings, which respects (in an appropriate
way) direct sums, similarities, and intertwinings. Higher order nc
difference-differential operators will be obtained by iterating
$\Delta_R$.

(From now on, we will concentrate on right nc
difference-differential operators. Of course, one can construct an
analogous ``left" theory.)

\section{Higher order nc functions}\label{subsec:higher} Let
$\module{M}_0$, \ldots, $\module{M}_k$, $\module{N}_0$, \ldots,
$\module{N}_k$ be modules over ring $\ring$. Let
$\Omega^{(j)}\subseteq\ncspacej{\module{M}}{j}$ be a nc set,
$j=0,\ldots,k$.  We consider
 functions $f$ on $\Omega^{(0)}\times\cdots\times\Omega^{(k)}$ with
 $$f(\Omega^{(0)}_{n_0},\ldots,\Omega^{(k)}_{n_k})\subseteq\hom_\ring
 \left(\rmat{\module{N}_1}{n_0}{n_1}\otimes\cdots\otimes
 \rmat{\module{N}_k}{n_{k-1}}{n_k},\rmat{\module{N}_0}{n_0}{n_k}\right).$$
In other words, for $X^0\in\Omega^{(0)}_{n_0}$, \ldots,
$X^k\in\Omega^{(k)}_{n_k}$, $f\left(X^0,\ldots,X^k\right)$ is a
$k$-linear mapping over $\ring$ from
$\rmat{\module{N}_1}{n_0}{n_1}\times\cdots\times
 \rmat{\module{N}_k}{n_{k-1}}{n_k}$ to $\rmat{\module{N}_0}{n_0}{n_k}.$
We will call such a function $f$ a \emph{nc function of order $k$}
\index{nc function of order $k$} if $f$ satisfies the following
two conditions:
\begin{itemize}
\item $f$ \emph{respects direct sums}: \index{respecting direct
sums} if $n_0,\ldots,n_k\in\mathbb{N}$,
$X^0\in\Omega^{(0)}_{n_0}$, \ldots, $X^k\in\Omega^{(k)}_{n_k}$,
$Z^1\in\rmat{\module{N}_1}{n_0}{n_1}$, \ldots,
$Z^k\in\rmat{\module{N}_k}{n_{k-1}}{n_k}$, then
\begin{multline}\tag{1X$^0$}\label{eq:1X^0} f(X^{0\,\prime}\oplus
X^{0\,\prime\prime},X^1,\ldots,X^k)(\col\,[Z^{1\,\prime},
Z^{1\,\prime\prime}],Z^2,\ldots,Z^k)\\
=\col\,[f(X^{0\,\prime},X^1,\ldots,X^k)
(Z^{1\,\prime},Z^2,\ldots,Z^k),\\
f(X^{0\,\prime\prime},X^1,\ldots,X^k)
(Z^{1\,\prime\prime},Z^2,\ldots,Z^k)]
\end{multline}
for $n_{0}^{\prime},n_{0}^{\prime\prime}\in\mathbb{N}$,
$X^{0\,\prime}\!\in\!\Omega^{(0)}_{n_{0}^{\prime}}$,
$X^{0\,\prime\prime}\!\in\!\Omega^{(0)}_{n_{0}^{\prime\prime}}$,
$Z^{1\,\prime}\!\in\!\rmat{\module{N}_1}{n_{0}^{\prime}}{n_1}$,
$Z^{1\,\prime\prime}\!\in\!\rmat{\module{N}_1}{n_{0}^{\prime\prime}}{n_1}$,
 where $Z^2,\ldots,Z^k$ do
not show up when $k=1$;
\begin{multline}\tag{1X$^j$}\label{eq:1X^j} f(X^0,\ldots,
X^{j-1},X^{j\,\prime}\oplus
X^{j\,\prime\prime},X^{j+1},\ldots,X^k)\\
\hspace{1cm}(Z^1,\ldots,Z^{j-1},\row\,[Z^{j\,\prime},
Z^{j\,\prime\prime}],\col\,[Z^{(j+1)\,\prime},
Z^{(j+1)\,\prime\prime}], Z^{j+2},\ldots,Z^k)\\
=f(X^0,\ldots,X^{j-1},X^{j\,\prime},X^{j+1},\ldots,X^k)\\
\hfill(Z^1,\ldots,Z^{j-1},Z^{j\,\prime},
 Z^{(j+1)\,\prime}, Z^{j+2},\ldots,Z^k\\
+f(X^0,\ldots,X^{j-1},X^{j\,\prime\prime},X^{j+1},\ldots,X^k)\\
 (Z^1,\ldots,Z^{j-1},Z^{j\,\prime\prime},
 Z^{(j+1)\,\prime\prime}, Z^{j+2},\ldots,Z^k)
\end{multline}
for every $j\in\{ 1,\ldots,k-1\}$, and for
$n_{j}^{\prime},n_{j}^{\prime\prime}\in\mathbb{N}$,
$X^{j\,\prime}\in\Omega^{(j)}_{n_{j}^{\prime}}$,
$X^{j\,\prime\prime}\in\Omega^{(j)}_{n_{j}^{\prime\prime}}$,
$Z^{j\,\prime}\in\rmat{\module{N}_j}{n_{j-1}}{n_{j}^{\prime}}$,
$Z^{j\,\prime\prime}\in\rmat{\module{N}_j}{n_{j-1}}{n_{j}^{\prime\prime}}$,
$Z^{(j+1)\,\prime}\in\rmat{(\module{N}_{j+1})}{n_{j}^{\prime}}{n_{j+1}}$,
$Z^{(j+1)\,\prime\prime}\in\rmat{(\module{N}_{j+1})}{n_{j}^{\prime\prime}}{n_{j+1}}$,
 where
$Z^1$, \ldots, $Z^{j-1}$ do not show up for $j=1$, and
$Z^{j+2},\ldots,Z^k$ do not show up for $j=k-1$;
\begin{multline}\tag{1X$^k$}\label{eq:1X^k}
f(X^0,\ldots,X^{k-1},X^{k\,\prime}\oplus
X^{k\,\prime\prime})(Z^1,\ldots,Z^{k-1},\row\,[Z^{k\,\prime},
Z^{k\,\prime\prime}])\\
=\row\,[f(X^0,\ldots,X^{k-1},X^{k\,\prime})
(Z^1,\ldots,Z^{k-1},Z^{k\,\prime},\\
 f(X^0,\ldots,X^{k-1},X^{k\,\prime\prime})
(Z^1,\ldots,Z^{k-1},Z^{k\,\prime\prime})]
\end{multline}
for $n_{k}^{\prime},n_{k}^{\prime\prime}\in\mathbb{N}$,
$X^{k\,\prime}\in\Omega^{(k)}_{n_{k}^{\prime}}$,
$X^{k\,\prime\prime}\in\Omega^{(k)}_{n_{k}^{\prime\prime}}$,
$Z^{k\,\prime}\in\rmat{\module{N}_k}{n_{k-1}}{n_{k}^{\prime}}$,
$Z^{k\,\prime\prime}\in\rmat{\module{N}_k}{n_{k-1}}{n_{k}^{\prime\prime}}$,
where $Z^1,\ldots,Z^{k-1}$ do not show up when $k=1$. \vspace{2mm}

\item $f$ \emph{respects similarities}: \index{respecting
similarities} if $n_0,\ldots,n_k\in\mathbb{N}$,
$X^0\in\Omega^{(0)}_{n_0}$, \ldots, $X^k\in\Omega^{(k)}_{n_k}$,
$Z^1\in\rmat{\module{N}_1}{n_0}{n_1}$, \ldots,
$Z^k\in\rmat{\module{N}_k}{n_{k-1}}{n_k}$,  then
\begin{multline}\tag{2X$^0$}\label{eq:2X^0}
f(S_0X^0S_0^{-1},X^1,\ldots,X^k)(S_0Z^1,Z^2,\ldots,Z^k)\\
=S_0f(X^0,\ldots,X^k)(Z^1,\ldots,Z^k)
\end{multline}
for an invertible $S_0\in\mat{\ring}{n_0}$ such that
$S_0X^0S_0^{-1}\!\in\!\Omega^{(0)}_{n_0}$,
 where $Z^2,\ldots,Z^k$ do
not show up when $k=1$;
\begin{multline}\tag{2X$^j$}\label{eq:2X^j}
f(X^0,\ldots,X^{j-1},S_jX^jS_j^{-1},X^{j+1},\ldots,X^k)\\
(Z^1,\ldots,Z^{j-1},Z^jS_j^{-1},S_jZ^{j+1},Z^{j+2},\ldots,Z^k)\\
=f(X^0,\ldots,X^k)(Z^1,\ldots,Z^k)
\end{multline}
for every $j\in\{ 1,\ldots,k-1\}$ and an invertible
$S_j\in\mat{\ring}{n_j}$ such that
$S_jX^jS_j^{-1}\in\Omega^{(j)}_{n_j}$, where $Z^1,\ldots,Z^{j-1}$
do not show up for $j=1$, and $Z^{j+2},\ldots,Z^k$ do not show up
for $j=k-1$;
\begin{multline}\tag{2X$^k$}\label{eq:2X^k}
f(X^0,\ldots,X^{k-1},S_kX^kS_k^{-1})(Z^1,\ldots,Z^{k-1},Z^kS_k^{-1})\\
=f(X^0,\ldots,X^k)(Z^1,\ldots,Z^k)S_k^{-1}
\end{multline}
for an invertible $S_k\in\mat{\ring}{n_k}$ such that
$S_kX^kS_k^{-1}\in\Omega^{(k)}_{n_k}$, where $Z^1$, \ldots,
$Z^{k-1}$ do not show up when $k=1$.
\end{itemize}

We define \emph{nc functions of order zero} to be simply nc
functions as in Section \ref{subsec:ncfundef}. We denote by
$\tclass{k}=\tclass{k}(\Omega^{(0)},\ldots,\Omega^{(k)};\ncspacej{\module{N}}{0},\ldots,\ncspacej{\module{N}}{k})$
\index{$\tclass{k}(\Omega^{(0)},\ldots,\Omega^{(k)};\ncspacej{\module{N}}{0},\ldots,\ncspacej{\module{N}}{k})$}
the class of nc functions of order $k$, \index{nc function of
order $k$} $k=0,1,\ldots$. In the case where
$\module{M}_0=\cdots=\module{M}_k=:\module{M}$ and
$\Omega^{(0)}=\cdots=\Omega^{(k)}=:\Omega$, we will write
$\tclass{k}=\tclass{k}(\Omega;\ncspacej{\module{N}}{0},\ldots,\ncspacej{\module{N}}{k})$.
We will provide later a slightly different interpretation of nc
functions of order $k$ using tensor products rather than
multilinear mappings for their values, which shows why the
definition of higher order nc functions is natural and generalizes
nc functions of order zero.

 Similarly to the class $\tclass{0}$ (see Proposition
\ref{prop:simint}), the two conditions in the definition of a nc
function of order $k$ can be replaced by a single one.
\begin{prop}\label{prop:simint_k}
Let $\Omega^{(j)}\subseteq\ncspacej{\module{M}}{j}$,
$j=0,\ldots,k$, be nc sets.
 A function $f$ on
$\Omega^{(0)}\times\cdots\times\Omega^{(k)}$ with
 $$f(\Omega^{(0)}_{n_0},\ldots,\Omega^{(k)}_{n_k})\subseteq\hom_\ring
 (\rmat{\module{N}_1}{n_0}{n_1}\otimes\cdots\otimes
 \rmat{\module{N}_k}{n_{k-1}}{n_k},\rmat{\module{N}_0}{n_0}{n_k})$$
 respects direct sums and similarities, i.e.,
 $f\in\tclass{k}(\Omega^{(0)},\ldots,\Omega^{(k)};
 \ncspacej{\module{N}}{0},\ldots,\ncspacej{\module{N}}{k})$, if and only if $f$
 \emph{respects intertwinings}: \index{respecting intertwinings}
if $n_0,\ldots,n_k\in\mathbb{N}$, $X^0\in\Omega^{(0)}_{n_0}$,
 \ldots,
$X^k\in\Omega^{(k)}_{n_k}$, $Z^1\in\rmat{\module{N}_1}{n_0}{n_1}$,
\ldots, $Z^k\in\rmat{\module{N}_k}{n_{k-1}}{n_k}$,  then:
\begin{multline}\tag{3X$^0$}\label{eq:3X^0}
\text{if}\quad T_0X^0=\tilde{X}^0T_0,\quad \text{then}\\
T_0f(X^0,\ldots,X^k)(Z^1,\ldots,Z^k)
=f(\tilde{X}^0,X^1,\ldots,X^k)(T_0Z^1,Z^2,\ldots,Z^k)
\end{multline}
for $\tilde{n}_0\in\mathbb{N}$,
$\tilde{X}^0\in\Omega^{(0)}_{\tilde{n}_0}$,
 and
$T_0\in\rmat{\ring}{\tilde{n}_0}{n_0}$,
 where $Z^2,\ldots,Z^k$ do
not show up when $k=1$;
\begin{multline}\tag{3X$^j$}\label{eq:3X^j}
\text{if}\quad T_jX^j=\tilde{X}^jT_j,\quad \text{then}\\
f(X^0,\ldots,X^k)
(Z^1,\ldots,Z^{j-1},\tilde{Z}^jT_j,Z^{j+1},Z^{j+2},\ldots,Z^k)\\
=f(X^0,\ldots,X^{j-1},\tilde{X}^j,X^{j+1},\ldots,X^k)
(Z^1,\ldots,Z^{j-1},\tilde{Z}^{j},T_jZ^{j+1},Z^{j+2},\ldots,Z^k)
\end{multline}
for every $j\in\{ 1,\ldots,k-1\}$, and for
$\tilde{n}_j\in\mathbb{N}$,
$\tilde{X}^j\in\Omega^{(j)}_{\tilde{n}_j}$,
$\tilde{Z}^j\in\rmat{\module{N}_j}{n_{j-1}}{\tilde{n}_j}$,
$T_j\in\rmat{\ring}{\tilde{n}_j}{n_j}$,
 where $Z^1,\ldots,Z^{j-1}$ do
not show up for $j=1$, and $Z^{j+2},\ldots,Z^k$ do not show up for
$j=k-1$;
\begin{multline}\tag{3X$^k$}\label{eq:3X^k} \text{if}\quad X^kT_k=T_k\tilde{X}^k,\quad \text{then}\\
f(X^0,\ldots,X^k)(Z^1,\ldots,Z^k)T_k
=f(X^0,\ldots,X^{k-1},\tilde{X}^k)(Z^1,\ldots,Z^{k-1},Z^kT_k)
\end{multline}
for $\tilde{n}_k\in\mathbb{N}$,
$\tilde{X}^k\in\Omega^{(k)}_{\tilde{n}_k}$, and
$T_k\in\rmat{\ring}{n_k}{\tilde{n}_k}$,
 where $Z^1,\ldots,Z^{k-1}$ do
not show up when $k=1$.

Moreover, conditions \eqref{eq:1X^0} and \eqref{eq:2X^0} in the
definition of a nc function of order $k$ together are equivalent
to condition \eqref{eq:3X^0}, and similarly, \eqref{eq:1X^j} $\&$
\eqref{eq:2X^j} $\Longleftrightarrow$ \eqref{eq:3X^j}
($j=1,\ldots,k-1$), \eqref{eq:1X^k} $\&$ \eqref{eq:2X^k}
$\Longleftrightarrow$ \eqref{eq:3X^k}.
\end{prop}
\begin{proof}
We will provide a detailed proof of \eqref{eq:1X^0} $\&$
\eqref{eq:2X^0} $\Longleftrightarrow$ \eqref{eq:3X^0}. Other
equivalences in the proposition can be proved analogously.

Assume first that conditions \eqref{eq:1X^0} and \eqref{eq:2X^0}
in the definition of a nc function of order $k$ are fulfilled. Let
$n_0,\tilde{n}_0,n_1,\ldots,n_k\in\mathbb{N}$,
$X^0\in\Omega^{(0)}_{n_0}$,
$\tilde{X}^0\in\Omega^{(0)}_{\tilde{n}_0}$,
$X^1\in\Omega^{(1)}_{n_1}$, \ldots, $X^k\in\Omega^{(k)}_{n_k}$,
$Z^1\in\rmat{\module{N}_1}{n_0}{n_1}$,
$\tilde{Z}^1\in\rmat{\module{N}_1}{\tilde{n}_0}{n_1}$,
$Z^2\in\rmat{\module{N}_2}{n_1}{n_2}$, \ldots,
$Z^k\in\rmat{\module{N}_k}{n_{k-1}}{n_k}$, and let
$T_0\in\rmat{\ring}{\tilde{n}_0}{n_0}$ be such that
$T_0X^0=\tilde{X}^0T_0$. By \eqref{eq:1X^0}, one has
\begin{multline*} f(X^{0}\oplus
\tilde{X}^{0},X^1,\ldots,X^k)(\col\,[Z^{1},
\tilde{Z}^{1}],Z^2,\ldots,Z^k)\\
=\col\,[f(X^{0},\ldots,X^k)(Z^{1},\ldots,Z^k),
f(\tilde{X}^{0},X^1,\ldots,X^k)(\tilde{Z}^{1},Z^2,\ldots,Z^k)].
\end{multline*}
On the other hand, using the identity
\begin{equation*}
\begin{bmatrix}
X^0 & 0\\
0 & \tilde{X}^0
\end{bmatrix}=\begin{bmatrix}
I_{n_0} & 0\\
T_0 & I_{\tilde{n}_0}
\end{bmatrix}\begin{bmatrix}
X^0 & 0\\
0 & \tilde{X}^0
\end{bmatrix}\begin{bmatrix}
I_{n_0} & 0\\
-T_0 & I_{\tilde{n}_0}
\end{bmatrix},
\end{equation*}
and conditions \eqref{eq:1X^0} and \eqref{eq:2X^0}, we obtain
\begin{multline*}
f(X^{0}\oplus \tilde{X}^{0},X^1,\ldots,X^k)(\col\,[Z^{1},
\tilde{Z}^{1}],Z^2,\ldots,Z^k)\\
=f\left(\begin{bmatrix}
I_{n_0} & 0\\
T_0 & I_{\tilde{n}_0}
\end{bmatrix}\begin{bmatrix}
X^0 & 0\\
0 & \tilde{X}^0
\end{bmatrix}\begin{bmatrix}
I_{n_0} & 0\\
-T_0 & I_{\tilde{n}_0}
\end{bmatrix},X^1,\ldots,X^k\right)\\
\hfill\left(\begin{bmatrix}
I_{n_0} & 0\\
T_0 & I_{\tilde{n}_0}
\end{bmatrix}\begin{bmatrix}
I_{n_0} & 0\\
-T_0 & I_{\tilde{n}_0}\end{bmatrix}\col[Z^{1},
\tilde{Z}^{1}],Z^2,\ldots,Z^k\right)\\
=\begin{bmatrix}
I_{n_0} & 0\\
T_0 & I_{\tilde{n}_0}
\end{bmatrix}f(X^{0}\oplus
\tilde{X}^{0},X^1,\ldots,X^k)(\col,[Z^{1},
-T_0Z^1+\tilde{Z}^{1}],Z^2,\ldots,Z^k)\\
=\begin{bmatrix}
I_{n_0} & 0\\
T_0 & I_{\tilde{n}_0}
\end{bmatrix}\col\,[f(X^0,\ldots,X^k)(Z^1,\ldots,Z^k),\hfill\\
\hfill f(\tilde{X}^0,X^1,\ldots,X^k)(-T_0Z^1+\tilde{Z}^1,Z^2,\ldots,Z^k)]\\
=\col\,[f(X^0,\ldots,X^k)(Z^1,\ldots,Z^k),\hfill\\
\hfill T_0f(X^0,\ldots,X^k)(Z^1,\ldots,Z^k)
+f(\tilde{X}^0,\ldots,X^k)(-T_0Z^1+\tilde{Z}^1,Z^2,\ldots,Z^k)].
\end{multline*}
The comparison of the (2,1) blocks in the two block column
expressions for
$$f(X^{0}\oplus \tilde{X}^{0},X^1,\ldots,X^k)(\col\,[Z^{1},
\tilde{Z}^{1}],Z^2,\ldots,Z^k)$$ gives
\begin{multline*}
f(\tilde{X}^0,X^1,\ldots,X^k)(\tilde{Z}^1,Z^2,\ldots,Z^k)\\
=T_0f(X^0,\ldots,X^k)(Z^1,\ldots,Z^k)
+f(\tilde{X}^0,X^1,\ldots,X^k)(-T_0Z^1+\tilde{Z}^1,Z^2,\ldots,Z^k).
\end{multline*}
Using linearity in $Z^1$ of $f(X^0,\ldots,X^k)(Z^1,\ldots,Z^k)$,
we obtain
\begin{equation*}
0=T_0f(X^0,\ldots,X^k)(Z^1,\ldots,Z^k)
-f(\tilde{X}^0,X^1,\ldots,X^k)(T_0Z^1,Z^2,\ldots,Z^k),
\end{equation*}
and \eqref{eq:3X^0} follows.

Conversely, assume that \eqref{eq:3X^0} holds. Let
$n^\prime_0,n^{\prime\prime}_0,n_1,\ldots,n_k\in\mathbb{N}$,
$X^{0\prime}\in\Omega^{(0)}_{n^\prime_0}$,
$X^{0\prime\prime}\in\Omega^{(0)}_{n^{\prime\prime}_0}$,
$X^1\in\Omega^{(1)}_{n_1}$, \ldots, $X^k\in\Omega^{(k)}_{n_k}$,
$Z^{1\prime}\in\rmat{\module{N}_1}{n^\prime_0}{n_1}$,
$Z^{1\prime\prime}\in\rmat{\module{N}_1}{n^{\prime\prime}_0}{n_1}$,
$Z^2\in\rmat{\module{N}_2}{n_1}{n_2}$, \ldots,
$Z^k\in\rmat{\module{N}_k}{n_{k-1}}{n_k}$. The intertwinings
\begin{equation*}
\col\,[I_{n^\prime_0},0]X^{0\prime}=(X^{0\prime}\oplus
X^{0\prime\prime})\col\,[I_{n^\prime_0},0]
\end{equation*}
and
\begin{equation*}
\col\,[0,I_{n^{\prime\prime}_0}]X^{0\prime\prime}=(X^{0\prime}\oplus
X^{0\prime\prime})\col\,[0,I_{n^{\prime\prime}_0}]
\end{equation*}
imply, respectively,
\begin{multline*}
\col\,[I_{n^\prime_0},0]f(X^{0\prime},X^1,\ldots,X^k)(Z^{1\prime},Z^2,\ldots,Z^k)\\
=f(X^{0\prime}\oplus
X^{0\prime\prime},X^1,\ldots,X^k)(\col\,[I_{n^\prime_0},0]Z^{1\prime},Z^2,\ldots,Z^k)
\end{multline*}
and
\begin{multline*}
\col\,[0,I_{n^{\prime\prime}_0}]f(X^{0\prime\prime},X^1,\ldots,X^k)
(Z^{1\prime\prime},Z^2,\ldots,Z^k)\\
=f(X^{0\prime}\oplus X^{0\prime\prime},X^1,\ldots,X^k)
(\col\,[0,I_{n^{\prime\prime}_0}]Z^{1\prime\prime},Z^2,\ldots,Z^k).
\end{multline*}
Summing up these two equalities, we obtain \eqref{eq:1X^0} by
linearity.

If $n_0\in\mathbb{N}$, $X^0\in\Omega^{(0)}_{n_0}$,
$Z^1\in\rmat{\module{N}_1}{n_0}{n_1}$, and
$S_0\in\mat{\ring}{n_0}$ is invertible, then, in virtue of
\eqref{eq:3X^0}, the intertwining
$S_0X^0=\left(S_0X^0S_0^{-1}\right)S_0$ implies \eqref{eq:2X^0}.
\end{proof}

Similarly to Proposition \ref{prop:Rproper'}, each of the
conditions of respecting direct sums, similarities and
intertwinings admits an equivalent joint formulation.
\begin{prop}\label{prop:joint_k}
\begin{enumerate} \item Conditions \eqref{eq:1X^0}, \eqref{eq:1X^j} for $j=1,\ldots,k-1$,
and \eqref{eq:1X^k} are equivalent to
\begin{multline}\label{eq:dirsums_k}
f(X^{0\prime}\oplus X^{0\prime\prime},\ldots,X^{k\prime}\oplus
X^{k\prime\prime})\left(\begin{bmatrix} Z^{1\prime,\prime} &
Z^{1\prime,\prime\prime}\\
Z^{1\prime\prime,\prime} & Z^{1\prime\prime,\prime\prime}
\end{bmatrix},\ldots,\begin{bmatrix} Z^{k\prime,\prime} &
Z^{k\prime,\prime\prime}\\
Z^{k\prime\prime,\prime} & Z^{k\prime\prime,\prime\prime}
\end{bmatrix}\right)\\
=\begin{bmatrix} f^{\prime,\prime} &
f^{\prime,\prime\prime}\\
f^{\prime\prime,\prime} & f^{\prime\prime,\prime\prime}
\end{bmatrix},
\end{multline}
where, for $\alpha,\beta\in\{^\prime,^{\prime\prime}\}$,
\begin{equation}\label{eq:f_entries}
f^{\alpha,\beta}=\sum_{
  \alpha_0,\ldots,\alpha_k\in \{^\prime,^{\prime\prime}\}\colon
  \alpha_0=\alpha, \alpha_k=\beta}
f(X^{0\alpha_0},\ldots,X^{k\alpha_k})
(Z^{1\alpha_0,\alpha_1},\ldots,Z^{k\alpha_{k-1},\alpha_k}).
\end{equation}
Here $n_j^\prime\in\mathbb{N}$,
$n_j^{\prime\prime}\in\mathbb{Z}_+$,
$X^{j\alpha}\in\Omega^{(j)}_{n_j^\alpha}$ for $j=0,\ldots,k$,
$\alpha\in\{^\prime,^{\prime\prime}\}$,
$Z^{j\alpha,\beta}\in\rmat{\module{N}_j}{n_{j-1}^\alpha}{n_j^\beta}$
for $j=1,\ldots,k$, $\alpha,\beta\in\{^\prime,^{\prime\prime}\}$,
the block entry $f^{\alpha,\beta}$ is void if either $n_0^\alpha$
or $n_k^\beta$ is $0$, and a summand in \eqref{eq:f_entries} is
$0$ if at least one of $n_j^{\alpha_j}$, $j=1,\ldots,k-1$, is $0$.
\item Conditions \eqref{eq:2X^0}, \eqref{eq:2X^j} for
$j=1,\ldots,k-1$, and \eqref{eq:2X^k} are equivalent to
\begin{multline}\label{eq:sim_k}
f(S_0X^0S_0^{-1},\ldots,S_kX^kS_k^{-1})(S_0Z^1S_1^{-1},\ldots,S_{k-1}Z^kS_k^{-1})\\
=S_0f(X^0,\ldots,X^k)(Z^1,\ldots,Z^k)S_k^{-1}
\end{multline}
for $n_j\in\mathbb{N}$, $X^j\in\Omega^{(j)}_{n_j}$, and invertible
$S_j\in\mat{\ring}{n_j}$ such that
$S_jX^jS_j^{-1}\in\Omega^{(j)}_{n_j}$, $j=0,\ldots,k$, and for
$Z^j\in\rmat{\module{N}_j}{n_{j-1}}{n_j}$, $j=1,\ldots,k$.
    \item Conditions \eqref{eq:3X^0}, \eqref{eq:3X^j} for
$j=1,\ldots,k-1$, and \eqref{eq:3X^k} are equivalent to
\begin{multline}\label{eq:int_k}
\text{If}\ \  T_jX^j=\tilde{X}^jT_j,\ j=0,\ldots,k,\ \text{then}\\
T_0f(X^0,\ldots,X^k)(Z^1T_1,\ldots,Z^kT_k)\\
=f(\tilde{X}^0,\ldots,\tilde{X}^k)(T_0Z^1,\ldots,T_{k-1}Z^k)T_k
\end{multline}
for $n_j,\tilde{n}_j\in\mathbb{N}$, $X^j\in\Omega^{(j)}_{n_j}$,
$\tilde{X}^j\in\Omega^{(j)}_{\tilde{n}_j}$,
$T_j\in\rmat{\ring}{\tilde{n}_j}{n_j}$, $j=0,\ldots,k$, and for
$Z^j\in\rmat{\module{N}_j}{n_{j-1}}{\tilde{n}_j}$, $j=1,\ldots,k$.
\end{enumerate}
\end{prop}
\begin{proof}
(1). Suppose that \eqref{eq:1X^0}--\eqref{eq:1X^k} hold. Then
\begin{multline*}
f(X^{0\prime}\oplus X^{0\prime\prime},\ldots,X^{k\prime}\oplus
X^{k\prime\prime})\left(\begin{bmatrix} Z^{1\prime,\prime} &
Z^{1\prime,\prime\prime}\\
Z^{1\prime\prime,\prime} & Z^{1\prime\prime,\prime\prime}
\end{bmatrix},\ldots,\begin{bmatrix} Z^{k\prime,\prime} &
Z^{k\prime,\prime\prime}\\
Z^{k\prime\prime,\prime} & Z^{k\prime\prime,\prime\prime}\\
\end{bmatrix}\right)\\
=\col_{\alpha_0\in\{\prime,\prime\prime\}
}\Big[f(X^{0\alpha_0},X^{1\prime}\oplus
X^{1\prime\prime},\ldots,X^{k\prime}\oplus
X^{k\prime\prime})\hfill\\
\left(\row\,[Z^{1\alpha_0,\prime},Z^{1\alpha_0,\prime\prime}],\begin{bmatrix}
Z^{2\prime,\prime} &
Z^{2\prime,\prime\prime}\\
Z^{2\prime\prime,\prime} & Z^{2\prime\prime,\prime\prime}
\end{bmatrix},\ldots,\begin{bmatrix} Z^{k\prime,\prime} &
Z^{k\prime,\prime\prime}\\
Z^{k\prime\prime,\prime} & Z^{k\prime\prime,\prime\prime}
\end{bmatrix}\right)\Big]\\
=\col_{\alpha_0\in\{\prime,\prime\prime\}
}\Big[\sum_{\alpha_1\in\{\prime,\prime\prime\} }f(X^{0\alpha_0},
X^{1\alpha_1},X^{2\prime}\oplus
X^{2\prime\prime},\ldots,X^{k\prime}\oplus
X^{k\prime\prime})\hfill\\
\left(Z^{1\alpha_0,\alpha_1},
\row\,[Z^{2\alpha_1,\prime},Z^{2\alpha_1,\prime\prime}],\begin{bmatrix}
Z^{3\prime,\prime} &
Z^{3\prime,\prime\prime}\\
Z^{3\prime\prime,\prime} & Z^{3\prime\prime,\prime\prime}
\end{bmatrix},\ldots,\begin{bmatrix} Z^{k\prime,\prime} &
Z^{k\prime,\prime\prime}\\
Z^{k\prime\prime,\prime} & Z^{k\prime\prime,\prime\prime}
\end{bmatrix}\right)\Big]\\
=\col_{\alpha_0\in\{\prime,\prime\prime\}
}\Big[\sum_{\alpha_1,\alpha_2\in\{\prime,\prime\prime\}
}f(X^{0\alpha_0}, X^{1\alpha_1},X^{2\alpha_2},X^{3\prime}\oplus
X^{3\prime\prime}, \ldots,X^{k\prime}\oplus
X^{k\prime\prime})\hfill\\
\left(Z^{1\alpha_0,\alpha_1},Z^{2\alpha_1,\alpha_2},
\row[Z^{3\alpha_2,\prime},
Z^{3\alpha_2,\prime\prime}],\begin{bmatrix} Z^{4\prime,\prime} &
Z^{4\prime,\prime\prime}\\
Z^{4\prime\prime,\prime} & Z^{4\prime\prime,\prime\prime}
\end{bmatrix},\ldots,\begin{bmatrix} Z^{k\prime,\prime} &
Z^{k\prime,\prime\prime}\\
Z^{k\prime\prime,\prime} & Z^{k\prime\prime,\prime\prime}
\end{bmatrix}\right)\Big]\\
=\ldots =\col_{\alpha_0\in\{\prime,\prime\prime\}
}\Big[\sum_{\alpha_1,\ldots,\alpha_{k-1}\in\{\prime,\prime\prime\}
}f(X^{0\alpha_0},\ldots, X^{(k-1)\alpha_{k-1}},X^{k\prime}\oplus
X^{k\prime\prime})\\
\hfill(Z^{1\alpha_0,\alpha_1},\ldots,Z^{(k-1)\alpha_{k-2},\alpha_{k-1}},
\row\,[Z^{k\alpha_{k-1},\prime},
Z^{k\alpha_{k-1},\prime\prime}])\Big]\\
=\col_{\alpha_0\in\{\prime,\prime\prime\}
}\Big[\sum_{\alpha_1,\ldots,\alpha_{k-1}\in\{\prime,\prime\prime\}
}\row_{\alpha_k\in\{\prime,\prime\prime\}}
[f(X^{0\alpha_0},\ldots,
X^{k\alpha_k})(Z^{1\alpha_0,\alpha_1},\ldots,Z^{k\alpha_{k-1},\alpha_k})]
\Big]\\
=\begin{bmatrix} f^{\prime,\prime} &
f^{\prime,\prime\prime}\\
f^{\prime\prime,\prime} & f^{\prime\prime,\prime\prime}
\end{bmatrix},
\end{multline*}
where $f^{\alpha\beta}$ are as in \eqref{eq:f_entries}.

Conversely, if \eqref{eq:dirsums_k} holds, then each of
\eqref{eq:1X^0}--\eqref{eq:1X^k} is obtained as a special case
when all $n_j^{\prime\prime}$'s but one equal $0$.

(2). Suppose that \eqref{eq:2X^0}--\eqref{eq:2X^k} hold. Then
\begin{multline*}
f(S_0X^0S_0^{-1},\ldots,S_kX^kS_k^{-1})(S_0Z^1S_1^{-1},\ldots,S_{k-1}Z^kS_k^{-1})\\
=S_0f(X^0,S_1X^1S_1^{-1},\ldots,S_kX^kS_k^{-1})
(Z^1S_1^{-1},S_1Z^2S_2^{-1},\ldots,S_{k-1}Z^kS_k^{-1})\\
=S_0f(X^0,X^1,S_2X^2S_2^{-1},\ldots,S_kX^kS_k^{-1})
(Z^1,Z^2S_2^{-1},S_2Z^3S_3^{-1},\ldots,S_{k-1}Z^kS_k^{-1})\\
=\ldots =S_0f(X^0,\ldots,X^{k-1},S_kX^kS_k^{-1})(Z^1,\ldots, Z^{k-1},Z^kS_k^{-1})\\
=S_0f(X^0,\ldots,X^k)(Z^1,\ldots, Z^k)S_k^{-1}.
\end{multline*}

Conversely, if \eqref{eq:sim_k} holds, then each of
\eqref{eq:2X^0}--\eqref{eq:2X^k} is obtained as a special case
when $S_j=I_{n_j}$ for all $j$'s but one.

(3). Suppose that \eqref{eq:3X^0}--\eqref{eq:3X^k} hold, and
$T_jX^j=\tilde{X}^jT_j$, $j=0,\ldots,k$. Then
\begin{multline*}
T_0f(X^0,\ldots,X^k)(Z^1T_1,\ldots,Z^kT_k)\\
=f(\tilde{X}^0,X^1,\ldots,X^k)(T_0Z^1T_1,Z^2T_2,\ldots,Z^kT_k)\\
=f(\tilde{X}^0,\tilde{X}^1,X^2,\ldots,X^k)
(T_0Z^1,T_1Z^2T_2,Z^3T_3,\ldots,Z^kT_k)
\end{multline*}
\begin{multline*}
=\ldots =f(\tilde{X}^0,\ldots,\tilde{X}^{k-1},X^k)
(T_0Z^1,\ldots,T_{k-2}Z^{k-1},T_{k-1}Z^kT_k)\\
=f(\tilde{X}^0,\ldots,\tilde{X}^{k})(T_0Z^1,\ldots,T_{k-1}Z^k)T_k.\\
\end{multline*}

Conversely, if \eqref{eq:int_k} holds, then each of
\eqref{eq:3X^0}--\eqref{eq:3X^k} is obtained as a special case
when $\tilde{n}_j=n_j$ and $T_j=I_{n_j}$ for all $j$'s but one.
\end{proof}

Formulae \eqref{eq:dirsums_k}--\eqref{eq:f_entries} admit a
generalization to the case of direct sums of more than two
summands as follows. Let
$$X^j=\bigoplus_{\alpha=1}^{m_j}X^{j\alpha},\quad
Z^j=[Z^{j\alpha,\beta}]_{\alpha=1,\ldots,m_{j-1},\,\beta=1,\ldots,m_j},$$
where
 $n_j^\alpha\in\mathbb{Z}_+$, $X^{j\alpha}\in\Omega^{(j)}_{n_j^\alpha}$ for $j=0,\ldots,k$,
$\alpha=1,\ldots,m_j$;
$Z^{j\alpha,\beta}\in\rmat{\module{N}_j}{n_{j-1}^\alpha}{n_j^\beta}$
for $j=1,\ldots,k$,
$\alpha=1,\ldots,m_{j-1},\,\beta=1,\ldots,m_j$. Then
\begin{equation}\label{eq:dirsums_k-mult}
f(X^0,\ldots,X^k)(Z^1,\ldots,Z^k)=[f^{\alpha,\beta}]_{\alpha=1,\ldots,m_{0},\,\beta=1,\ldots,m_k},
\end{equation}
where
\begin{equation}\label{eq:f_entries-mult}
f^{\alpha,\beta}=\sum_{
  \alpha_j=1,\ldots,m_j\colon
  \alpha_0=\alpha, \alpha_k=\beta}
f(X^{0\alpha_0},\ldots,X^{k\alpha_k})
(Z^{1\alpha_0,\alpha_1},\ldots,Z^{k\alpha_{k-1},\alpha_k}).
\end{equation}
Here the block entry $f^{\alpha,\beta}$ is void if either
$n_0^\alpha$ or $n_k^\beta$ is $0$, and a summand in
\eqref{eq:f_entries-mult} is $0$ if at least one of
$n_j^{\alpha_j}$, $j=1,\ldots,k-1$, is $0$. The proof is similar
to the proof of Proposition \ref{prop:joint_k}(1).

The case where for each $j$ the diagonal blocks $X^{j\alpha}$ are
all equal will be of special importance (see Chapter
\ref{sec:TT}). Some notations are in order. For
$Z^j\in\rmat{\module{N}_j}{n_{j-1}}{n_j}$, $n_j=m_js_j$, $j=1,
\ldots, k$, we define \index{{$Z^1{}_{s_0,s_2}\odot_{s_1}Z^2$}}
$$Z^1{}_{s_0,s_2}\!\!\odot_{s_1}\cdots\,{}_{s_{k-2},s_k}\!\!\odot_{s_{k-1}}Z^k\in
\rmat{\left(\rmat{\module{N}_1}{s_0}{s_1}\otimes\cdots\otimes\rmat{\module{N}_k}{s_{k-1}}{s_k}\right)}{m_0}{m_k}$$
 viewing $Z^j$ as a
$m_{j-1}\times m_j$ matrix over
 $\rmat{\module{N}_j}{s_{j-1}}{s_j}$ and applying the standard rules of matrix
 multiplication, i.e.,
\begin{equation}\label{eq:tensa_prod}
(Z^1{}_{s_0,s_2}\!\!\odot_{s_1}\cdots\,{}_{s_{k-2},s_k}\!\!\odot_{s_{k-1}}Z^k)^{\alpha,\beta}=\sum_{
  \alpha_j=1,\ldots,m_j\colon
  \alpha_0=\alpha, \alpha_k=\beta}
Z^{1\alpha_0,\alpha_1}\otimes\cdots\otimes
Z^{k\alpha_{k-1},\alpha_k},
\end{equation}
where
$Z^j=[Z^{j\alpha,\beta}]_{\alpha=1,\ldots,m_{j-1},\,\beta=1,\ldots,m_j}$.
This is an iteration of a product of matrices over two modules
that are endowed with a product operation into a third module, as
in \ref{subsub:prod}. In case two factors have square blocks
(necessarily of the same size), we omit the left subscript
indices, i.e., we write $Z^1\odot_sZ^2$ \index{$Z^1\odot_sZ^2$}
instead of $Z^1{}_{s,s}\!\!\odot_{s}Z^2$.
  For $s=1$ we omit the right subscript,
i.e., we write $Z^1{}_{s_0,s_2}\!\!\odot Z^2$
\index{$Z^1{}_{s_0,s_2}\odot Z^2$} instead of
$Z^1{}_{s_0,s_2}\!\!\odot_1 Z^2$, and we write
$Z^1\odot\,\cdots\,\odot Z^k$ \index{$Z^1\odot Z^2$} instead of
$Z^1\odot_1\cdots\,\odot_1Z^k$; see, e.g., \cite[Page 86]{Pi} or
\cite[Page 240]{Pa}.
\begin{prop}\label{prop:diag_tensa_k}
Let
$f\in\tclass{k}(\Omega^{(0)},\ldots,\Omega^{(k)};\ncspacej{\module{N}}{0},\ldots,\ncspacej{\module{N}}{k})$
and $n_j=m_js_j$, $j=0,\ldots,k$.
\begin{enumerate}
\item For $Y^j\in\Omega^{(j)}_{s_j}$,
$X^j=\bigoplus_{\alpha=1}^{m_j}Y^j$, $j=0, \ldots, k$, and for
$Z^j\in\rmat{\module{N}_j}{n_{j-1}}{n_j}$, $j=1, \ldots, k$,
\begin{equation}\label{eq:diag_tensa_k}
f(X^0,\ldots,X^k)(Z^1,\ldots,Z^k)=Z^1{}_{s_0,s_2}\!\!\odot_{s_1}\cdots\,{}_{s_{k-2},s_k}\!\!\odot_{s_{k-1}}Z^k
f(Y^0,\ldots,Y^k).
\end{equation}
Here the linear mapping
$$f(Y^0,\ldots,Y^k)\colon\rmat{\module{N}_1}{s_0}{s_1}\otimes\cdots\otimes\rmat{\module{N}_k}{s_{k-1}}{s_k}
\longrightarrow \rmat{\module{N}_0}{s_0}{s_k}$$ is acting on the
matrix
$Z^1{}_{s_0,s_2}\!\!\odot_{s_1}\cdots\,{}_{s_{k-2},s_k}\!\!\odot_{s_{k-1}}Z^k$
entrywise. \item For $Y^j\in\Omega^{(j)}_{s_j}$,
$X^j=\bigoplus_{\alpha=1}^{m_j}Y^j$, $j=0, \ldots, k-1$,
$X^k\in\Omega^{(k)}_{n_k}$, and for
$Z^j\in\rmat{\module{N}_j}{n_{j-1}}{n_j}$, $j=1, \ldots, k$,
\begin{equation}\label{eq:diag_tensa_k'}
f(X^0,\ldots,X^k)(Z^1,\ldots,Z^k)=Z^1{}_{s_0,s_2}\!\!\odot_{s_1}\cdots\,{}_{s_{k-2},s_k}\!\!\odot_{s_{k-1}}Z^k
f(Y^0,\ldots,Y^{k-1},X^k).
\end{equation}
Here the linear mapping
\begin{multline*}
f(Y^0,\ldots,Y^{k-1},X^k)\colon\rmat{\module{N}_1}{s_0}{s_1}
\otimes\cdots\otimes\rmat{(\module{N}_{k-1})}{s_{k-2}}{s_{k-1}}\otimes
\rmat{\module{N}_k}{s_{k-1}}{m_ks_k}\\ \cong\,
\rmat{\left(\rmat{\module{N}_1}{s_0}{s_1}\otimes\cdots\otimes\rmat{\module{N}_k}{s_{k-1}}{s_k}\right)}{1}{m_k}
\longrightarrow
\rmat{\module{N}_0}{s_0}{m_ks_k}\,\cong\,\rmat{(\rmat{\module{N}_0}{s_0}{s_k})}{1}{m_k}
\end{multline*}
is acting on the rows of the matrix
$Z^1{}_{s_0,s_2}\!\!\odot_{s_1}\cdots\,{}_{s_{k-2},s_k}\!\!\odot_{s_{k-1}}Z^k$;
i.e., $f(Y^0,\ldots,Y^{k-1},X^k)$ is a $m_k\times m_k$ matrix of
linear mappings
$$\rmat{\module{N}_1}{s_0}{s_1}\otimes\cdots\otimes\rmat{\module{N}_k}{s_{k-1}}{s_k}\longrightarrow
\rmat{\module{N}_0}{s_0}{s_k}$$ that multiplies the matrix
$Z^1{}_{s_0,s_2}\!\!\odot_{s_1}\cdots\,{}_{s_{k-2},s_k}\!\!\odot_{s_{k-1}}Z^k$
on the right.
\end{enumerate}
\end{prop}
\begin{proof}
Follows immediately from
\eqref{eq:dirsums_k-mult}--\eqref{eq:f_entries-mult}.
\end{proof}

\begin{rem}\label{rem:natur_map}
We notice that there is a natural mapping
\begin{multline}\label{eq:natural}
\mat{\module{N}_0}{n_0}\otimes\mat{\module{N}_1^*}{n_1}
\otimes\cdots\otimes\mat{\module{N}_k^*}{n_k}\\
\longrightarrow\hom_\ring\left(\rmat{\module{N}_1}{n_0}{n_1}
\otimes\cdots\otimes\rmat{\module{N}_k}{n_{k-1}}{n_k},
\rmat{\module{N}_0}{n_0}{n_k}\right)
\end{multline}
for $k$ a positive integer, defined on an elementary tensor
$Y^0\otimes Y^1\otimes\cdots\otimes Y^k$ by
\begin{equation}\label{eq:natur_map}
(Z^1,\ldots,Z^k)\longmapsto Y^0(Z^1Y^1)\cdots (Z^kY^k),
\end{equation}
with $Z^jY^j\in\rmat{\ring}{n_{j-1}}{n_j}$ for
$Z^j\in\rmat{\module{N}_j}{n_{j-1}}{n_j}$, $j=1,\ldots,k$.
 If the modules
$\module{N}_1$, \ldots, $\module{N}_k$ are free and of finite rank
 over $\ring$ and the module $\module{N}_0$ is free (in particular,
 when $\ring$ is a field and $\module{N}_1$, \ldots, $\module{N}_k$ are finite-dimensional
 vector spaces),  the natural mapping
 \eqref{eq:natural}--\eqref{eq:natur_map} is easily seen to be an
 isomorphism.
\end{rem}
\begin{rem}\label{rem:tensor_values}
In the case where the natural mapping
\eqref{eq:natural}--\eqref{eq:natur_map} is an isomorphism, we can
view a nc function of
 order $k$ as a function $f$ on $\Omega^{(0)}\times\cdots\times\Omega^{(k)}$ with
\begin{equation}\label{eq:tensor_values}
 f(\Omega^{(0)}_{n_0},\ldots,\Omega^{(k)}_{n_k})
 \subseteq\mat{\module{N}_0}{n_0}\otimes\mat{\module{N}_1^*}{n_1}
 \otimes\cdots\otimes\mat{\module{N}_k^*}{n_k}.
\end{equation}
Direct sums and similarities on $\mat{\module{N}_0}{n_0}$,
$\mat{\module{N}_1^*}{n_1}$, \ldots, $\mat{\module{N}_k^*}{n_k}$
induce $i$-th factor direct sums and similarities on
$\mat{\module{N}_0}{n_0}\otimes\mat{\module{N}_1^*}{n_1}\otimes\cdots\otimes\mat{\module{N}_k^*}{n_k}$
for the corresponding $i=0,\ldots, k$. With this interpretation,
 conditions \eqref{eq:1X^0}--\eqref{eq:1X^k} and
\eqref{eq:2X^0}--\eqref{eq:2X^k} mean that  direct sums and
similarities in each argument of $f$ result in direct sums and
similarities in the corresponding factor in the value of $f$. This
explains the form of the conditions of respecting direct sums and
similarities in the definition of nc functions of order $k$ and
also puts the cases $k=0$ and $k>0$ on an equal footing.
\end{rem}
\begin{rem}\label{rem:tensor_prod}
There is also a natural mapping
\begin{multline}\label{eq:tensor_prod}
\tclass{k}(\Omega^{(0)},\ldots,\Omega^{(k)};\ncspacej{\module{N}}{0},\ldots,\ncspacej{\module{N}}{k})\\
\otimes
\tclass{\ell}(\Omega^{(k+1)},\ldots,\Omega^{(k+\ell+1)};\ncspacej{\module{L}^*}{0},\ncspacej{\module{L}}{1},
\ldots,\ncspacej{\module{L}}{\ell})\\
\longrightarrow\tclass{k+\ell
+1}(\Omega^{(0)},\ldots,\Omega^{(k+\ell+1)};\ncspacej{\module{N}}{0},\ldots,\ncspacej{\module{N}}{k},
\ncspacej{\module{L}}{0},\ldots,\ncspacej{\module{L}}{\ell})
\end{multline}
defined on elementary tensors as
$$f\otimes g\longmapsto h,$$
where
\begin{equation}\label{eq:prod_tensor}
h(X^0,\ldots,X^k,U^0,\ldots,U^\ell )=f(X^0,\ldots,X^k)\otimes
g(U^0,\ldots,U^\ell),
\end{equation}
interpreting the values of higher order nc functions as elements
of tensor products (see Remark \ref{rem:tensor_values}), or
\begin{multline}\label{eq:prod_fun}
h(X^0,\ldots,X^k,U^0,\ldots,U^\ell
)(Z^1,\ldots,Z^k,W^0,W^1,\ldots,
W^\ell )\\
=f(X^0,\ldots,X^k)(Z^1,\ldots,Z^k)\Big(W^0 g(U^0,\ldots,U^\ell
)(W^1,\ldots, W^\ell)\Big),
\end{multline}
interpreting the values of higher order nc functions as
multilinear mappings. Here $X^0\in\Omega^{(0)}_{n_0}$, \ldots,
$X^k\in\Omega^{(k)}_{n_k}$, $U^0\in\Omega^{(k+1)}_{m_0}$, \ldots,
$U^\ell\in\Omega^{(k+\ell+1)}_{m_\ell}$,
$Z^1\in\rmat{\module{N}_1}{n_0}{n_1}$, \ldots,
$Z^k\in\rmat{\module{N}_k}{n_{k-1}}{n_k}$,
$W^0\in\rmat{\module{L}_0}{n_k}{m_0}$,
$W^1\in\rmat{\module{L}_1}{m_0}{m_1}$, \ldots,
$W^\ell\in\rmat{\module{L}_\ell}{m_{\ell -1}}{m_\ell}$.

Iterating this construction, we obtain a natural mapping
\begin{multline}\label{eq:natural_fun}
\tclass{0}(\Omega^{(0)};\ncspacej{\module{N}}{0})\otimes\tclass{0}(\Omega^{(1)};\ncspacej{\module{N}^*}{1})
\otimes\cdots\otimes
\tclass{0}(\Omega^{(k)};\ncspacej{\module{N}^*}{k})\\
\hfill\longrightarrow\tclass{k}(\Omega^{(0)},\ldots,\Omega^{(k)};\ncspacej{\module{N}}{0},\ldots,\ncspacej{\module{N}}{k}),
\end{multline}
defined on elementary tensors as
$$f^0\otimes \cdots\otimes f^k\longmapsto h,$$
where
\begin{equation}\label{eq:natur_map_fun_tensor}
h(X^0,X^1,\ldots,X^k)=f^0(X^0)\otimes f^1(X^1)\otimes\cdots\otimes
f^k(X^k),
\end{equation}
interpreting the values of higher order nc functions as elements
of tensor products, or
\begin{equation}\label{eq:natur_map_fun_forms}
h(X^0,X^1,\ldots,X^k)(Z^1,\ldots,Z^k)
=f^0(X^0)\Big(Z^1f^1(X^1)\Big)\cdots\Big(Z^kf^k(X^k)\Big),
\end{equation}
interpreting the values of higher order nc functions as
multilinear mappings. Here $X^0\in\Omega^{(0)}_{n_0}$, \ldots,
$X^k\in\Omega^{(k)}_{n_k}$, $Z^1\in\rmat{\module{N}_1}{n_0}{n_1}$,
\ldots, $Z^k\in\rmat{\module{N}_k}{n_{k-1}}{n_k}$.

It is clear that the natural mappings
\eqref{eq:tensor_prod}\&\eqref{eq:prod_tensor} (or
\eqref{eq:tensor_prod}\&\eqref{eq:prod_fun}) and, in particular,
\eqref{eq:natural_fun}\&\eqref{eq:natur_map_fun_tensor} (or
\eqref{eq:natural_fun}\&\eqref{eq:natur_map_fun_forms}) are
embeddings; we will sometimes abuse the notation writing
$h=f\otimes g$ and $h=f^0\otimes\cdots\otimes f^k$ for
\eqref{eq:prod_tensor} (or \eqref{eq:prod_fun}) and
\eqref{eq:natur_map_fun_tensor} (or
\eqref{eq:natur_map_fun_forms}), respectively.\label{EXTERMINATE}
\end{rem}

\section{Higher order nc difference-differential
operators}\label{subsec:difdif-k} Proposition \ref{prop:Rproper}
tells exactly that, for $f\in\tclass{0}$, one has
$\Delta_Rf\in\tclass{1}$. We will extend $\Delta_R$ to an operator
from $\tclass{k}$ to $\tclass{k+1}$ for all $k$. Similarly to the
case $k=0$, it will be done by evaluating a nc function of order
$k$ ($>0$) on a $(k+1)$-tuple of square matrices with one of the
arguments block upper triangular. We start with the following
analogue of Proposition \ref{prop:delta_r}. It will be shown
later, using the tensor product interpretation of the values of
higher order nc functions, that this is indeed a natural analogue.
\begin{prop}\label{prop:delta_r-k}
Let $\module{M}_j$, $\module{N}_j$ be modules, and let
$\Omega^{(j)}\subseteq\ncspacej{\module{M}}{j}$ be a nc set,
$j=0,\ldots,k$. Let
$f\in\tclass{k}(\Omega^{(0)},\ldots,\Omega^{(k)};\ncspacej{\module{N}}{0},\ldots,\ncspacej{\module{N}}{k})$.
Let $X^0\in\Omega^{(0)}_{n_0}$, \ldots,
$X^{k-1}\in\Omega^{(k-1)}_{n_{k-1}}$,
$X^{k\prime}\in\Omega^{(k)}_{n_{k}^{\prime}}$,
$X^{k\prime\prime}\in\Omega^{(k)}_{n_{k}^{\prime\prime}}$,
$Z^1\in\rmat{\module{N}_1}{n_0}{n_1}$, \ldots,
$Z^{k-1}\in\rmat{(\module{N}_{k-1})}{n_{k-2}}{n_{k-1}}$,
$Z^{k\prime}\in\rmat{\module{N}_k}{n_{k-1}}{n_{k}^\prime}$,
$Z^{k\prime\prime}\in\rmat{\module{N}_k}{n_{k-1}}{n_{k}^{\prime\prime}}$.
Let $Z\in\rmat{\module{M}_k}{n_k^\prime}{n_k^{\prime\prime}}$ be
such that $\begin{bmatrix} X^{k\prime} & Z\\
0 & X^{k\prime\prime}
\end{bmatrix}\in\Omega^{(k)}_{n_k^\prime+n_k^{\prime\prime}}$. Then
\begin{multline}\label{eq:uptr-k}
 f\left(X^0,\ldots,X^{k-1},\begin{bmatrix} X^{k\prime} & Z\\
0 & X^{k\prime\prime}
\end{bmatrix}\right)(Z^1,\ldots,Z^{k-1},\row
\,[Z^{k\prime},Z^{k\prime\prime}])\\
=\row\,\Big[f(X^0,\ldots,X^{k-1},X^{k\prime})(Z^1,\ldots,Z^{k-1},Z^{k\prime}),\\
\hfill \Delta_Rf(X^0,\ldots,X^{k-1},X^{k\prime},X^{k\prime\prime})
 (Z^1,\ldots,Z^{k-1},Z^{k\prime},Z)\\
  +
f(X^0,\ldots,X^{k-1},X^{k\prime\prime})
(Z^1,\ldots,Z^{k-1},Z^{k\prime\prime})\Big].
\end{multline}
Here $\Delta_Rf=(X^0,\ldots,X^{k-1},X^{k\prime},X^{k\prime\prime})
(Z^1,\ldots,Z^{k-1},Z^{k\prime},Z)$ is determined unique\-ly by
\eqref{eq:uptr-k}, is independent of $Z^{k\prime\prime}$, and has
the following property. If
$r\in\ring$ is such that $\begin{bmatrix} X^{k\prime} & rZ\\
0 & X^{k\prime\prime}
\end{bmatrix}\in\Omega^{(k)}_{n_k^\prime+n_k^{\prime\prime}}$, then
\begin{multline}\label{eq:homog_r-k}
\Delta_Rf(X^0,\ldots,X^{k-1},X^{k\prime},X^{k\prime\prime})
(Z^1,\ldots,Z^{k-1},Z^{k\prime},rZ)\\
= r\Delta_Rf(X^0,\ldots,X^{k-1},X^{k\prime},X^{k\prime\prime})
(Z^1,\ldots,Z^{k-1},Z^{k\prime},Z).
\end{multline}
For $k=1$, the matrices $Z^1,\ldots,Z^{k-1}$ do not show up in
\eqref{eq:uptr-k} and \eqref{eq:homog_r-k}.
\end{prop}
\begin{proof}
Denote
$$f\left(X^0,\ldots,X^{k-1},\begin{bmatrix} X^{k\prime} & Z\\
0 & X^{k\prime\prime}
\end{bmatrix}\right)\left(Z^1,\ldots,Z^{k-1},\row
\left[Z^{k\prime},Z^{k\prime\prime}\right]\right)=:\row[A,B],$$
where $A\in\rmat{\module{N}_0}{n_0}{n_k^\prime}$,
$B\in\rmat{\module{N}_0}{n_0}{n_k^{\prime\prime}}$. We have
\begin{equation*}
\begin{bmatrix} X^{k\prime} & Z\\
0 & X^{k\prime\prime}
\end{bmatrix}\begin{bmatrix} I_{n_k^\prime}\\
0
\end{bmatrix}=\begin{bmatrix} I_{n_k^\prime}\\
0
\end{bmatrix}X^{k\prime}.
\end{equation*}
By \eqref{eq:3X^k},
\begin{equation*}
A=f(X^0,\ldots,X^{k-1},X^{k\prime})(Z^1,\ldots,Z^{k-1},Z^{k\prime}).
\end{equation*}
 We have also
\begin{equation*}
X^{k\prime\prime}\begin{bmatrix} 0 & I_{n_k^{\prime\prime}}
\end{bmatrix}=\begin{bmatrix} 0 & I_{n_k^{\prime\prime}}
\end{bmatrix}\begin{bmatrix} X^{k\prime} & Z\\
0 & X^{k\prime\prime}
\end{bmatrix}.
\end{equation*}
By \eqref{eq:3X^k} and the additivity of
$$f(X^0,\ldots,X^{k-1},X^{k\prime})(Z^1,\ldots,Z^{k-1},Z^{k})$$
as a function of $Z^k$,
\begin{multline*}
\row\,[0,f(X^0,\ldots,X^{k-1},X^{k\prime\prime})
(Z^1,\ldots,Z^{k-1},Z^{k\prime\prime})]\\
=f\left(X^0,\ldots,X^{k-1},\begin{bmatrix} X^{k\prime} & Z\\
0 & X^{k\prime\prime}\end{bmatrix}\right)(Z^1,\ldots,Z^{k-1},
\row\,[0,Z^{k\prime\prime}])\\
=\row\,[A,B]-f\left(X^0,\ldots,X^{k-1},\begin{bmatrix} X^{k\prime} & Z\\
0 &
X^{k\prime\prime}\end{bmatrix}\right)(Z^1,\ldots,Z^{k-1},\row\,[Z^{k\prime},0]),
\end{multline*}
which implies that
\begin{multline*}
B=f\left(X^0,\ldots,X^{k-1},\begin{bmatrix} X^{k\prime} & Z\\
0 & X^{k\prime\prime}\end{bmatrix}\right)(Z^1,\ldots,Z^{k-1},
\row\,[Z^{k\prime},0])\begin{bmatrix} 0\\
I_{n_k^{\prime\prime}}
\end{bmatrix}\\
+f(X^0,\ldots,X^{k-1},X^{k\prime\prime})(Z^1,\ldots,Z^{k-1},Z^{k\prime\prime}).
\end{multline*}
Therefore, we obtain \eqref{eq:uptr-k} with
\begin{multline}\label{eq:delta_r_k}
\Delta_Rf\left(X^0,\ldots,X^{k-1},X^{k\prime},X^{k\prime\prime}\right)
\left(Z^1,\ldots,Z^{k-1},Z^{k\prime},Z\right)\\
=f\left(X^0,\ldots,X^{k-1},\begin{bmatrix} X^{k\prime} & Z\\
0 &
X^{k\prime\prime}\end{bmatrix}\right)\left(Z^1,\ldots,Z^{k-1},
\row\left[Z^{k\prime},0\right]\right)\begin{bmatrix} 0\\
I_{n_k^{\prime\prime}}
\end{bmatrix}.
\end{multline}
We observe that \eqref{eq:delta_r_k} implies that for any
$r\in\ring$ one has
\begin{multline}\label{eq:delta_homog_k}
\Delta_Rf(X^0,\ldots,X^{k-1},X^{k\prime},X^{k\prime\prime})
(Z^1,\ldots,Z^{k-1},rZ^{k\prime},Z)\\
=r\Delta_Rf(X^0,\ldots,X^{k-1},X^{k\prime},X^{k\prime\prime})
(Z^1,\ldots,Z^{k-1},Z^{k\prime},Z).
\end{multline}

Next, for $r\in\ring$ such that $\begin{bmatrix} X^{k\prime} & rZ\\
0 & X^{k\prime\prime}
\end{bmatrix}\in\Omega^{(k)}_{n_k^\prime+n_k^{\prime\prime}}$  we have
\begin{equation*}
\begin{bmatrix} X^{k\prime} & rZ\\
0 & X^{k\prime\prime}
\end{bmatrix}\begin{bmatrix} rI_{n_k^\prime} & 0\\
0 & I_{n_k^{\prime\prime}}
\end{bmatrix}=\begin{bmatrix} rI_{n_k^\prime} & 0\\
0 & I_{n_k^{\prime\prime}}
\end{bmatrix}\begin{bmatrix} X^{k\prime} & Z\\
0 & X^{k\prime\prime}
\end{bmatrix}.
\end{equation*}
By \eqref{eq:3X^k},
\begin{multline*}
f\left(X^0,\ldots,X^{k-1},\begin{bmatrix} X^{k\prime} & rZ\\
0 & X^{k\prime\prime}\end{bmatrix}\right)(Z^1,\ldots,Z^{k-1},
\row\,[Z^{k\prime},Z^{k\prime\prime}])
\begin{bmatrix}
rI_{n_k^{\prime}} & 0\\
0 & I_{n_k^{\prime\prime}}
\end{bmatrix}\\
=f\left(X^0,\ldots,X^{k-1},\begin{bmatrix} X^{k\prime} & Z\\
0 & X^{k\prime\prime}\end{bmatrix}\right)(Z^1,\ldots,Z^{k-1},\row
\,[rZ^{k\prime},Z^{k\prime\prime}]).
\end{multline*}
Multiplying both parts on the right by $\begin{bmatrix}
 0\\
I_{n_k^{\prime\prime}}
\end{bmatrix}$ and subtracting
$$f(X^0,\ldots,X^{k-1},X^{k\prime\prime})(Z^1,\ldots,Z^{k-1},Z^{k\prime\prime}),$$
we obtain in virtue of \eqref{eq:uptr-k} that
\begin{multline*}
\Delta_Rf(X^0,\ldots,X^{k-1},X^{k\prime},X^{k\prime\prime})
(Z^1,\ldots,Z^{k-1},Z^{k\prime},rZ)\\
=\Delta_Rf(X^0,\ldots,X^{k-1},X^{k\prime},X^{k\prime\prime})
=(Z^1,\ldots,Z^{k-1},rZ^{k\prime},Z).
\end{multline*}
Then applying \eqref{eq:delta_homog_k} we obtain
\eqref{eq:homog_r-k}.
\end{proof}

Assume now that the nc set
$\Omega^{(k)}\subseteq\ncspacej{\module{M}}{k}$ is right
admissible, and let
$$f\in\tclass{k}(\Omega^{(0)},\ldots,\Omega^{(k)};
\ncspacej{\module{N}}{0},\ldots,\ncspacej{\module{N}}{k}).$$ Then
for every $X^0\in\Omega^{(0)}_{n_0}$, \ldots,
$X^{k-1}\in\Omega^{(k-1)}_{n_{k-1}}$,
$X^{k\prime}\in\Omega^{(k)}_{n_{k}^{\prime}}$,
$X^{k\prime\prime}\in\Omega^{(k)}_{n_{k}^{\prime\prime}}$,
$Z^1\in\rmat{\module{N}_1}{n_0}{n_1}$, \ldots,
$Z^{k-1}\in\rmat{(\module{N}_{k-1})}{n_{k-2}}{n_{k-1}}$,
$Z^{k\prime}\in\rmat{\module{N}_k}{n_{k-1}}{n_{k}^\prime}$,
$Z\in\rmat{\module{M}_k}{n_k^\prime}{n_k^{\prime\prime}}$, and for
an invertible $r\in\ring$
such that $\begin{bmatrix} X^{k\prime} & rZ\\
0 & X^{k\prime\prime}
\end{bmatrix}\in\Omega^{(k)}_{n_k^\prime+n_k^{\prime\prime}}$, we define
first
$$\Delta_Rf(X^0,\ldots,X^{k-1},X^{k\prime},X^{k\prime\prime})
(Z^1,\ldots,Z^{k-1},Z^{k\prime},rZ)$$ by \eqref{eq:uptr-k}, with
$Z$ replaced by $rZ$ (for an arbitrary choice of
$Z^{k\prime\prime}\in\rmat{\module{N}_k}{n_{k-1}}{n_{k}^{\prime\prime}}$),
and then define
\begin{multline}\label{eq:def_delta_r-k}
\Delta_Rf(X^0,\ldots,X^{k-1},X^{k\prime},X^{k\prime\prime})
(Z^1,\ldots,Z^{k-1},Z^{k\prime},Z)\\
=
r^{-1}\Delta_Rf(X^0,\ldots,X^{k-1},X^{k\prime},X^{k\prime\prime})
(Z^1,\ldots,Z^{k-1},Z^{k\prime},rZ).
\end{multline}
By Proposition \ref{prop:delta_r-k}, the right-hand side of
\eqref{eq:def_delta_r-k} is independent of $r$, hence the
left-hand side is defined unambiguously. Using Proposition
\ref{prop:higher-ncfun-ext}, we can  establish a higher order
version of \eqref{eq:delta=delta_ext},
\begin{multline}\label{eq:delta=delta_ext-k}
\Delta_Rf(X^0,\ldots,X^{k-1},X^{k\prime},X^{k\prime\prime})
(Z^1,\ldots,Z^{k-1},Z^{k\prime},Z)\\
=
\Delta_R\widetilde{f}(X^0,\ldots,X^{k-1},X^{k\prime},X^{k\prime\prime})
(Z^1,\ldots,Z^{k-1},Z^{k\prime},Z).
\end{multline}
 We have the
following analogue of Proposition \ref{prop:homog}, with
essentially the same proof.
\begin{prop}\label{prop:homog-k}
Let
$f\in\tclass{k}(\Omega^{(0)},\ldots,\Omega^{(k)};\ncspacej{\module{N}}{0},\ldots,\ncspacej{\module{N}}{k})$,
with $\Omega^{(k)}\subseteq\ncspacej{\module{M}}{k}$ a right
admissible nc set. The mapping
$$Z\longmapsto\Delta_Rf(X^0,\ldots,X^{k-1},X^{k\prime},X^{k\prime\prime})
(Z^1,\ldots,Z^{k-1},Z^{k\prime},Z)$$ is homogeneous as an operator
from $\rmat{\module{M}_k}{n_k^\prime}{n_k^{\prime\prime}}$ to
$\rmat{\module{N}_0}{n_0}{n_k^{\prime\prime}}$, i.e.,
\eqref{eq:homog_r-k} holds for every $X^0\in\Omega^{(0)}_{n_0}$,
\ldots, $X^{k-1}\in\Omega^{(k-1)}_{n_{k-1}}$,
$X^{k\prime}\in\Omega^{(k)}_{n_{k}^{\prime}}$,
$X^{k\prime\prime}\in\Omega^{(k)}_{n_{k}^{\prime\prime}}$,
$Z^1\in\rmat{\module{N}_1}{n_0}{n_1}$, \ldots,
$Z^{k-1}\in\rmat{(\module{N}_{k-1})}{n_{k-2}}{n_{k-1}}$,
$Z^{k\prime}\in\rmat{\module{N}_k}{n_{k-1}}{n_{k}^\prime}$,
$Z\in\rmat{\module{M}_k}{n_k^\prime}{n_k^{\prime\prime}}$, and
every $r\in\ring$.
\end{prop}
We proceed with establishing an analogue of Proposition
\ref{prop:add}.
\begin{prop}\label{prop:add-k}
Let
$f\in\tclass{k}(\Omega^{(0)},\ldots,\Omega^{(k)};\ncspacej{\module{N}}{0},\ldots,\ncspacej{\module{N}}{k}),$
with $\Omega^{(k)}\subseteq\ncspacej{\module{M}}{k}$ a right
admissible nc set. The mapping
$$Z\longmapsto\Delta_Rf(X^0,\ldots,X^{k-1},X^{k\prime},X^{k\prime\prime})
(Z^1,\ldots,Z^{k-1},Z^{k\prime},Z)$$ is additive as an operator
from $\rmat{\module{M}_k}{n_k^\prime}{n_k^{\prime\prime}}$ to
$\rmat{\module{N}_0}{n_0}{n_k^{\prime\prime}}$, i.e.,
\begin{multline}\label{eq:add_r-k}
\Delta_Rf(X^0,\ldots,X^{k-1},X^{k\prime},X^{k\prime\prime})
(Z^1,\ldots,Z^{k-1},Z^{k\prime},Z_I+Z_{I\!
I})\\
=\Delta_Rf(X^0,\ldots,X^{k-1},X^{k\prime},X^{k\prime\prime})
(Z^1,\ldots,Z^{k-1},Z^{k\prime},Z_I)\\
+\Delta_Rf(X^0,\ldots,X^{k-1},X^{k\prime},X^{k\prime\prime})
(Z^1,\ldots,Z^{k-1},Z^{k\prime},Z_{I\! I})
\end{multline}
for every $X^0\in\Omega^{(0)}_{n_0}$, \ldots,
$X^{k-1}\in\Omega^{(k-1)}_{n_{k-1}}$,
$X^{k\prime}\in\Omega^{(k)}_{n_{k}^{\prime}}$,
$X^{k\prime\prime}\in\Omega^{(k)}_{n_{k}^{\prime\prime}}$,
$Z^1\in\rmat{\module{N}_1}{n_0}{n_1}$, \ldots,
$Z^{k-1}\in\rmat{(\module{N}_{k-1})}{n_{k-2}}{n_{k-1}}$,
$Z^{k\prime}\in\rmat{\module{N}_k}{n_{k-1}}{n_{k}^\prime}$, and
$Z_I,Z_{I\!
I}\in\rmat{\module{M}_k}{n_k^\prime}{n_k^{\prime\prime}}$.
\end{prop}
\begin{proof}
Let $X^0\in\Omega^{(0)}_{n_0}$, \ldots,
$X^{k-1}\in\Omega^{(k-1)}_{n_{k-1}}$,
$X^{k\prime}\in\Omega^{(k)}_{n_{k}^{\prime}}$,
$X^{k\prime\prime}\in\Omega^{(k)}_{n_{k}^{\prime\prime}}$,
$Z^1\in\rmat{\module{N}_1}{n_0}{n_1}$, \ldots,
$Z^{k-1}\in\rmat{(\module{N}_{k-1})}{n_{k-2}}{n_{k-1}}$,
$Z^{k\prime}\in\rmat{\module{N}_k}{n_{k-1}}{n_{k}^\prime}$, and
$Z_I,Z_{I\!
I}\in\rmat{\module{M}_k}{n_k^\prime}{n_k^{\prime\prime}}$ be
arbitrary.  By \eqref{eq:delta=delta_ext-k}, we may assume,
without loss of
 generality,
  that $\Omega^{(k)}$ is similarity
 invariant, i.e., $\widetilde{\Omega}^{(k)}=\Omega^{(k)}$; see Remark
 \ref{rem:alternative} and Appendix \ref{app}.
  Then by Proposition \ref{prop:adm-env},   $$\begin{bmatrix} X^{k\prime} & Z_I\\
0 & X^{k\prime\prime}
\end{bmatrix}\in\Omega^{(k)}_{n_{k}^\prime+n_k^{\prime\prime}},\ \begin{bmatrix} X^{k\prime} & Z_{I\! I}\\
0 & X^{k\prime\prime}
\end{bmatrix}\in\Omega^{(k)}_{n_{k}^\prime+n_k^{\prime\prime}},\
\begin{bmatrix} X^{k\prime} & (Z_I+Z_{I\! I})\\
0 & X^{k\prime\prime}
\end{bmatrix}\in\Omega^{(k)}_{n_{k}^\prime+n_k^{\prime\prime}},$$ and
$$\begin{bmatrix} X^{k\prime} & 0 & Z_I\\
0 & X^{k\prime} & Z_{I\! I}\\
0 & 0 & X^{k\prime\prime}
\end{bmatrix}\in\Omega^{(k)}_{2n_{k}^\prime+n_k^{\prime\prime}}$$ (for the latter, we use the fact that if
$X^{k\prime}\in\Omega^{(k)}_{n_{k}^\prime}$ then
$\begin{bmatrix} X^{k\prime} & 0\\
0 & X^{k\prime}
\end{bmatrix}\in\Omega^{(k)}_{2n_{k}^\prime}$).
We have
\begin{equation*}
\begin{bmatrix} X^{k\prime} & Z_I\\
0 & X^{k\prime\prime}
\end{bmatrix}\begin{bmatrix} I_{n_k^\prime} & 0 & 0\\
0 & 0 & I_{n_k^{\prime\prime}}
\end{bmatrix}=\begin{bmatrix}  I_{n_k^\prime} & 0 & 0\\
0 & 0 & I_{n_k^{\prime\prime}}
\end{bmatrix}
\begin{bmatrix} X^{k\prime} & 0 & Z_I\\
0 & X^{k\prime} & Z_{I\! I}\\
0 & 0 & X^{k\prime\prime}
\end{bmatrix}.
\end{equation*}
By \eqref{eq:3X^k},
\begin{multline*}
f\left(X^0,\ldots,X^{k-1},\begin{bmatrix} X^{k\prime} & Z_I\\
0 & X^{k\prime\prime}
\end{bmatrix}\right)(Z^1,\ldots,Z^{k-1},\row\,[Z^{k\prime},0])
\begin{bmatrix} I_{n_k^\prime} & 0 & 0\\
0 & 0 & I_{n_k^{\prime\prime}}
\end{bmatrix}
\\
=f\left(X^0,\ldots,X^{k-1},\begin{bmatrix} X^{k\prime} & 0 & Z_I\\
0 & X^{k\prime} & Z_{I\! I}\\
0 & 0 & X^{k\prime\prime}
\end{bmatrix}\right)(Z^1,\ldots,Z^{k-1},\row\,[Z^{k\prime},0,0]).
\end{multline*}
Multiplying both parts by $\col[0,0,I_{n_k^{\prime\prime}}]$ and
taking into account \eqref{eq:delta_r_k}, we obtain
\begin{multline}\label{eq:I}
\Delta_Rf(X^0,\ldots,X^{k-1},X^{k\prime},X^{k\prime\prime})
(Z^1,\ldots,Z^{k-1},Z^{k\prime},Z_I)\\
=\Delta_Rf(X^0,\ldots,X^{k-1},X^{k\prime}\oplus
X^{k\prime},X^{k\prime\prime})
(Z^1,\ldots,Z^{k-1},\row\,[Z^{k\prime},0],\col\,[Z_I,Z_{I\! I}]).
\end{multline}
Similarly, using the intertwining relation
\begin{equation*}
\begin{bmatrix} X^{k\prime} & Z_{I\! I}\\
0 & X^{k\prime\prime}
\end{bmatrix}\begin{bmatrix} 0 & I_{n_k^\prime} & 0\\
0 & 0 & I_{n_k^{\prime\prime}}
\end{bmatrix}=\begin{bmatrix} 0  & I_{n_k^\prime} & 0\\
0 & 0 & I_{n_k^{\prime\prime}}
\end{bmatrix}
\begin{bmatrix} X^{k\prime} & 0 & Z_I\\
0 & X^{k\prime} & Z_{I\! I}\\
0 & 0 & X^{k\prime\prime}
\end{bmatrix},
\end{equation*}
we obtain
\begin{multline}\label{eq:II}
\Delta_Rf(X^0,\ldots,X^{k-1},X^{k\prime},X^{k\prime\prime})
(Z^1,\ldots,Z^{k-1},Z^{k\prime},Z_{I\! I})\\
\!=\!\Delta_Rf(X^0,\ldots,X^{k-1},X^{k\prime}\oplus
X^{k\prime},X^{k\prime\prime})(Z^1,\ldots,Z^{k-1},\row\,[0,Z^{k\prime}],\col\,[Z_I,Z_{I\!
I}]).
\end{multline}
Next, using the intertwining relation
\begin{equation*}
\begin{bmatrix} X^{k\prime} & Z_I+Z_{I\! I}\\
0 & X^{k\prime\prime}
\end{bmatrix}\begin{bmatrix} I_{n_k^\prime} & I_{n_k^\prime} & 0\\
0 & 0 & I_{n_k^{\prime\prime}}
\end{bmatrix}
=\begin{bmatrix} I_{n_k^\prime}  & I_{n_k^\prime} & 0\\
0 & 0 & I_{n_k^{\prime\prime}}
\end{bmatrix}
\begin{bmatrix} X^{k\prime} & 0 & Z_I\\
0 & X^{k\prime} & Z_{I\! I}\\
0 & 0 & X^{k\prime\prime}
\end{bmatrix},
\end{equation*}
we obtain
\begin{multline}\label{eq:I+II}
\Delta_Rf(X^0,\ldots,X^{k-1},X^{k\prime},X^{k\prime\prime})
(Z^1,\ldots,Z^{k-1},Z^{k\prime},Z_I+Z_{I\! I})\\
\!=\!\Delta_Rf(X^0,\ldots,X^{k-1},X^{k\prime}\oplus
X^{k\prime},X^{k\prime\prime})
(Z^1,\ldots,Z^{k-1},\row\,[Z^{k\prime},Z^{k\prime}],\col\,[Z_I,Z_{I\!
I}]).
\end{multline}
We observe that \eqref{eq:delta_r_k} implies the additivity of
$$\Delta_Rf(X^0,\ldots,X^{k-1},X^{k\prime},X^{k\prime\prime})
(Z^1,\ldots,Z^{k-1},Z^{k\prime},Z)$$ as a function of
$Z^{k\prime}$. Therefore, the sum of the right-hand sides of
\eqref{eq:I} and \eqref{eq:II} equals the right-hand side of
\eqref{eq:I+II}, and \eqref{eq:add_r-k} follows.
\end{proof}
\begin{thm}\label{thm:dif-op}
Let
$f\in\tclass{k}(\Omega^{(0)},\ldots,\Omega^{(k)};\ncspacej{\module{N}}{0},\ldots,\ncspacej{\module{N}}{k})$,
where $\Omega^{(k)}\subseteq\ncspacej{\module{M}}{k}$ is a right
admissible nc set. For every $X^0\in\Omega^{(0)}_{n_0}$, \ldots,
$X^k\in\Omega^{(k)}_{n_k}$, $X^{k+1}\in\Omega^{(k)}_{n_{k+1}}$,
the mapping
\begin{equation*}
(Z^1,\ldots,Z^{k+1})\longmapsto
\Delta_Rf(X^0,\ldots,X^{k+1})(Z^1,\ldots,Z^{k+1})
\end{equation*}
 from
$\rmat{\module{N}_1}{n_0}{n_1}\times\cdots\times\rmat{\module{N}_k}{n_{k-1}}{n_k}\times
\rmat{\module{M}_k}{n_k}{n_{k+1}}$ to
$\rmat{\module{N}_0}{n_0}{n_{k+1}}$ is $(k+1)$-linear over
$\ring$. Furthermore,
$\Delta_Rf\in\tclass{k+1}(\Omega^{(0)},\ldots,\Omega^{(k)},\Omega^{(k)};
\ncspacej{\module{N}}{0},\ldots,\ncspacej{\module{N}}{k},
\ncspacej{\module{M}}{k})$.
\end{thm}
\begin{proof}
The fact that
$\Delta_Rf\left(X^0,\ldots,X^{k+1}\right)\left(Z^1,\ldots,Z^{k+1}\right)$
is linear as a function of $Z^j$, for every $j=1,\ldots,k$,
follows from \eqref{eq:delta_r_k}. The linearity in $Z^{k+1}$
follows from Propositions \ref{prop:homog-k} and \ref{prop:add-k}.
We also observe that \eqref{eq:delta_r_k} implies that the
function
$\Delta_Rf\left(X^0,\ldots,X^{k+1}\right)\left(Z^1,\ldots,Z^{k+1}\right)$
respects direct sums and similarities in variables
$X^0,\ldots,X^{k-1}$. Thus, it only remains to show that  the
function
$\Delta_Rf\left(X^0,\ldots,X^{k+1}\right)\left(Z^1,\ldots,Z^{k+1}\right)$
respects direct sums and similarities, or equivalently (by
Proposition \ref{prop:simint_k}), respects intertwinings, in
variables $X^{k}$ and $X^{k+1}$.

Let $X^0\in\Omega^{(0)}_{n_0}$, \ldots,
$X^k\in\Omega^{(k)}_{n_k}$, $X^{k+1}\in\Omega^{(k)}_{n_{k+1}}$,
$\widetilde{X}^{k}\in\Omega^{(k)}_{\widetilde{n}_{k}}$,
$\widetilde{X}^{k+1}\in\Omega^{(k)}_{\widetilde{n}_{k+1}}$,
$Z^1\in\rmat{\module{N}_1}{n_0}{n_1}$, \ldots,
$Z^{k-1}\in\rmat{(\module{N}_{k-1})}{n_{k-2}}{n_{k-1}}$,
$Z^{k}\in\rmat{\module{N}_k}{n_{k-1}}{\widetilde{n}_{k}}$, and
$Z^{k+1}\in\rmat{\module{M}_k}{n_k}{\widetilde{n}_{k+1}}$ be
arbitrary, and let $T_k\in\rmat{\ring}{\widetilde{n}_k}{n_k}$ and
$T_{k+1}\in\rmat{\ring}{\widetilde{n}_{k+1}}{n_{k+1}}$ satisfy the
identities $T_kX^k=\widetilde{X}^kT_k$ and
$T_{k+1}X^{k+1}=\widetilde{X}^{k+1}T_{k+1}$.
 By
\eqref{eq:delta=delta_ext-k}, we may assume, without loss of
 generality,
  that $\Omega^{(k)}$ is similarity
 invariant, i.e., $\widetilde{\Omega}^{(k)}=\Omega^{(k)}$; see Remark
 \ref{rem:alternative} and Appendix \ref{app}.
  Then by Proposition \ref{prop:adm-env},   $$\begin{bmatrix} X^{k} & Z^{k+1}T_{k+1}\\
0 & X^{k+1}
\end{bmatrix}\in\Omega^{(k)}_{n_{k}+n_{k+1}},\ \begin{bmatrix} \widetilde{X}^{k+1} & T_kZ^{k+1}\\
0 & \widetilde{X}^{k+1}
\end{bmatrix}\in\Omega^{(k)}_{\widetilde{n}_{k}+\widetilde{n}_{k+1}}.$$
We have
$$\begin{bmatrix} T_k & 0\\
0 & T_{k+1}
\end{bmatrix}\begin{bmatrix} X^{k} & Z^{k+1}T_{k+1}\\
0 & X^{k+1}
\end{bmatrix}=\begin{bmatrix} \widetilde{X}^{k+1} & T_kZ^{k+1}\\
0 & \widetilde{X}^{k+1}
\end{bmatrix}\begin{bmatrix} T_k & 0\\
0 & T_{k+1}
\end{bmatrix}.$$
Using \eqref{eq:delta_r_k} and applying \eqref{eq:3X^k} to $f$, we
obtain
\begin{multline*}
\Delta_Rf(X^0,\ldots,X^{k-1},\widetilde{X}^{k},
\widetilde{X}^{k+1})(Z^1,\ldots,Z^{k},T_kZ^{k+1})T_{k+1}
\\
=f\left(X^0,\ldots,X^{k-1},\begin{bmatrix} \widetilde{X}^{k} & T_kZ^{k+1}\\
0 & \widetilde{X}^{k+1}
\end{bmatrix}\right)
(Z^1,\ldots,Z^{k-1},\row\,[Z^{k},0])\begin{bmatrix} 0\\
I_{n_{k+1}}
\end{bmatrix}T_{k+1}\\
=f\left(X^0,\ldots,X^{k-1},\begin{bmatrix} \widetilde{X}^{k} & T_kZ^{k+1}\\
0 & \widetilde{X}^{k+1}
\end{bmatrix}\right)
\left(Z^1,\ldots,Z^{k-1},\row[Z^{k},0]\right) \begin{bmatrix} T_k & 0\\
0 & T_{k+1}
\end{bmatrix}\\ \times\begin{bmatrix} 0\\
I_{n_{k+1}}
\end{bmatrix}\\
=f\left(X^0,\ldots,X^{k-1},\begin{bmatrix} {X}^{k} & Z^{k+1}T_{k+1}\\
0 & {X}^{k+1}
\end{bmatrix}\right)
\left(Z^1,\ldots,Z^{k-1},\row[Z^{k},0]\begin{bmatrix} T_k & 0\\
0 & T_{k+1}
\end{bmatrix}\right)\\ \times \begin{bmatrix} 0\\
I_{n_{k+1}}
\end{bmatrix}\\
=f(X^0,\ldots,X^{k+1})
(Z^1,\ldots,Z^{k-1},Z^{k}T_k,Z^{k+1}T_{k+1}),
\end{multline*}
which is a special case of  \eqref{eq:3X^k} applied to
$\Delta_Rf$.
\end{proof}

We have therefore the right nc difference-differential operator
\begin{multline*}
\Delta_R\colon\tclass{k}(\Omega^{(0)},\ldots,\Omega^{(k)};\ncspacej{\module{N}}{0},\ldots,\ncspacej{\module{N}}{k})\\
\longrightarrow\tclass{k+1}(\Omega^{(0)},\ldots,\Omega^{(k)},\Omega^{(k)};\ncspacej{\module{N}}{0},
\ldots,\ncspacej{\module{N}}{k},\ncspacej{\module{M}}{k}),
\end{multline*}
for $k=0,1,\ldots$. Iterating this operator $\ell$ times, we
obtain the \emph{$\ell$-th order right nc difference-differential
operator} \index{higher order right nc difference-differential
operator}  \begin{multline*} \Delta_R^\ell\colon
\tclass{k}(\Omega^{(0)},\ldots,\Omega^{(k)};\ncspacej{\module{N}}{0},\ldots,\ncspacej{\module{N}}{k})\\
\longrightarrow\tclass{k+\ell}(\Omega^{(0)},\ldots,\Omega^{(k)},\underset{\ell\
\text{times}}{\underbrace{\Omega^{(k)},\ldots,\Omega^{(k)}}};
\ncspacej{\module{N}}{0},\ldots,\ncspacej{\module{N}}{k},
\underset{\ell\
\text{times}}{\underbrace{\ncspacej{\module{M}}{k},\ldots,\ncspacej{\module{M}}{k}}}),
\end{multline*} for $ k=0,1,\ldots$. According to our definition,
$\Delta_R^\ell f$ is calculated iteratively by evaluating nc
functions of increasing orders on $2\times 2$ block upper
triangular matrices at each step. It turns out that $\Delta_R^\ell
f$ can also  be calculated in a single step by evaluating $f$ on
$(\ell +1)\times (\ell +1)$ block upper bidiagonal matrices. We
formulate the result separately for the case $k=0$ and for the
case $k>0$.

\begin{thm}\label{thm:bidiag}
Let $f\in\tclass{0}(\Omega;\ncspace{\module{N}})$, with
$\Omega\subseteq\ncspace{\module{M}}$ a right admissible nc set.
For every $X^0\in\Omega_{n_0}$, \ldots,
$X^{\ell}\in\Omega_{n_{\ell}}$,
$Z^1\in\rmat{\module{M}}{n_0}{n_1}$, \ldots,
$Z^\ell\in\rmat{\module{M}}{n_{\ell -1}}{n_\ell}$ such that
$$\begin{bmatrix}
X^0 & Z^1 & 0 & \cdots & 0\\
0   & X^1 & \ddots & \ddots & \vdots\\
\vdots & \ddots & \ddots & \ddots & 0 \\
\vdots &  & \ddots & X^{\ell  -  1} & Z^\ell\\
 0 & \cdots & \cdots & 0 & X^\ell
\end{bmatrix}\in\Omega_{n_0+\cdots +n_\ell},
$$
one has {\small\begin{multline}\label{eq:bidiag}
f\left(\begin{bmatrix}
X^0 & Z^1 & 0 & \cdots & 0\\
0   & X^1 & \ddots & \ddots & \vdots\\
\vdots & \ddots & \ddots & \ddots & 0 \\
\vdots &  & \ddots & X^{\ell  -  1} & Z^\ell\\
 0 & \cdots & \cdots & 0 & X^\ell
\end{bmatrix}\right)\\
\,=\!\!\begin{bmatrix} f(X^0) & \!\!\Delta_Rf(X^0,X^1)(Z^1) &
\cdots & \cdots &
\!\!\Delta_R^\ell f(X^0,\ldots,X^\ell)(Z^1,\ldots,Z^\ell)\\
0   & f(X^1) & \ddots & & \!\!\Delta_R^{\ell -1} f(X^1,\ldots,X^\ell)(Z^2,\ldots,Z^\ell)\\
\vdots & \ddots & \ddots & \ddots & \vdots \\
\vdots &  & \ddots &  f(X^{\ell  -  1}) & \!\!\Delta_Rf(X^{\ell -1},X^\ell)(Z^\ell)\\
 0 & \cdots & \cdots &  0 & \!\! f(X^\ell )
\end{bmatrix}.
\end{multline} }
\end{thm}

\begin{thm}\label{thm:bidiag_k}
Let
$f\in\tclass{k}(\Omega^{(0)},\ldots,\Omega^{(k)};\ncspacej{\module{N}}{0},\ldots,\ncspacej{\module{N}}{k})$,
where $\Omega^{(k)}\subseteq\ncspacej{\module{M}}{k}$ is a right
admissible nc set, $k>0$. For every $X^0\in\Omega^{(0)}_{n_0}$,
\ldots, $X^{k-1}\in\Omega^{(k-1)}_{n_{k-1}}$,
$X^{k,0}\in\Omega^{(k)}_{n_{k,0}}$, \ldots,
$X^{k,\ell}\in\Omega^{(k)}_{n_{k,\ell}}$,
$Z^{k+1,1}\in\rmat{\module{M}_k}{n_{k,0}}{n_{k,1}}$, \ldots,
$Z^{k+1,\ell}\in\rmat{\module{M}_k}{n_{k,\ell -1}}{n_{k,\ell}}$
such that
$$\begin{bmatrix}
X^{k,0} & Z^{k+1,1} & 0 & \cdots & 0\\
0   & X^{k,1} & \ddots & \ddots & \vdots\\
\vdots & \ddots & \ddots & \ddots & 0 \\
\vdots &  & \ddots & X^{k,\ell  -  1} & Z^{k+1,\ell}\\
 0 & \cdots & \cdots & 0 & X^{k,\ell}
\end{bmatrix}\in\Omega^{(k)}_{n_{k,0}+\cdots +n_{k,\ell}},
$$
one has
\begin{multline}\label{eq:bidiag_k}
f\left(X^0,\ldots,X^{k-1},\begin{bmatrix}
X^{k,0} & Z^{k+1,1} & 0 & \cdots & 0\\
0   & X^{k,1} & \ddots & \ddots & \vdots\\
\vdots & \ddots & \ddots & \ddots & 0 \\
\vdots &  & \ddots & X^{k,\ell  -  1} & Z^{k+1,\ell}\\
 0 & \cdots & \cdots & 0 & X^{k,\ell}
\end{bmatrix}\right)\\
(Z^1,\ldots,Z^{k-1},\row\,[Z^{k,0},\ldots,Z^{k,\ell}])\\
=\row_{0\le
j\le\ell}\Big[\sum_{i=0}^j\Delta_R^{j-i}f(X^0,\ldots,X^{k-1},X^{k,i},\ldots,X^{k,j})\\
(Z^1,\ldots,Z^{k-1},Z^{k,i},Z^{k+1,i+1},\ldots,Z^{k+1,j})\Big]
\end{multline}
for all $Z^1\in\rmat{\module{N}_1}{n_{0}}{n_{1}}$, \ldots,
$Z^{k-1}\in\rmat{(\module{N}_{k-1})}{n_{k-2}}{n_{k-1}}$,
$Z^{k,0}\in\rmat{\module{N}_k}{n_{k-1}}{n_{k,0}}$, \ldots,
$Z^{k,\ell}\in\rmat{\module{N}_k}{n_{k-1}}{n_{k,\ell}}$.
\end{thm}
\begin{rem}\label{rem:bidiag_k_special}
In the special case where $Z^{k,1}=0$, \ldots, $Z^{k,\ell}=0$, the
right-hand side of \eqref{eq:bidiag_k} becomes
\begin{multline*}
  \row\left[ \right.  f(X^0,\ldots,X^{k-1},X^{k,0})(Z^1,\ldots,Z^{k-1},Z^{k,0}),\\
 \Delta_Rf(X^0,\ldots,X^{k-1},X^{k,0},X^{k,1})
(Z^1,\ldots,Z^{k-1},Z^{k,0},Z^{k+1,1}), \ldots, \\
   \left. \Delta_R^\ell
 f(X^0,\ldots,X^{k-1},X^{k,0},\ldots,X^{k,\ell})
 (Z^1,\ldots,Z^{k-1},Z^{k,0},Z^{k+1,1},\ldots,Z^{k+1,\ell})\right].
\end{multline*}
\end{rem}
For the proofs of Theorems \ref{thm:bidiag} and
\ref{thm:bidiag_k}, we will need the following lemma.
\begin{lem}\label{lem:bidiag-homog}
Let $\Omega\subseteq\ncspace{\module{M}}$ be similarity invariant,
i.e., $\widetilde{\Omega}=\Omega$; see Appendix \ref{app}. Then
for all $n_0$, \ldots, $n_\ell$, $X^0\in\Omega_{n_0}$, \ldots,
$X^{\ell}\in\Omega_{n_{\ell}}$,
$Z^1\in\rmat{\module{M}}{n_0}{n_1}$, \ldots,
$Z^\ell\in\rmat{\module{M}}{n_{\ell -1}}{n_\ell}$ one has
$$\begin{bmatrix}
X^j & Z^{j+1} & 0 & \cdots & 0\\
0   & X^{j+1} & \ddots & \ddots & \vdots\\
\vdots & \ddots & \ddots & \ddots & 0 \\
\vdots &  & \ddots & X^{\ell  -  1} & Z^\ell\\
 0 & \cdots & \cdots & 0 & X^\ell
\end{bmatrix}\in\Omega_{n_j+\cdots +n_\ell},\quad
j=0,\ldots,\ell-1.
$$
\end{lem}
\begin{proof}
Induction on $\ell$, using Proposition \ref{prop:adm-env}.
\end{proof}
\begin{proof}[Proof of Theorem \ref{thm:bidiag}]
By \eqref{eq:delta=delta_ext-k}, we may assume, without loss of
 generality,
  that $\Omega$ is similarity
 invariant, i.e., $\widetilde{\Omega}=\Omega$; see Remark
 \ref{rem:alternative} and Appendix \ref{app}.
We will prove the statement by induction on $\ell$. For $\ell=0$,
the two sides of \eqref{eq:bidiag} are identical. For $\ell=1$,
\eqref{eq:bidiag} coincides with \eqref{eq:uptr} up to notation.
Suppose now that the statement is true for $\ell = m$, with some
$m\in\mathbb{N}$. Let $X^0\in\Omega_{n_0}$, \ldots,
$X^{m+1}\in\Omega_{n_{m+1}}$,
$Z^{1}\in\rmat{\module{M}}{n_{0}}{n_{1}}$, \ldots,
$Z^{m+1}\in\rmat{\module{M}}{n_{m}}{n_{m+1}}$. Using induction
hypothesis,  Lemma \ref{lem:bidiag-homog}, and \eqref{eq:uptr-k},
we obtain {\tiny\begin{multline*} f\left({\begin{bmatrix}
X^{0} & Z^{1} & 0 & \cdots & \cdots & 0\\
0   & X^{1} & \ddots & \ddots &  &\vdots\\
\vdots & \ddots & \ddots & \ddots & \ddots & \vdots \\
\vdots &  & \ddots & X^{m-1} & Z^{m} & 0\\
 0 & \cdots & \cdots & 0 & X^{m} & Z^{m+1}\\
0 & \cdots & \cdots & \cdots & 0 & X^{m+1}
\end{bmatrix}}\right)\\
={\!\!\begin{bmatrix} f(X^0) &  \cdots & \cdots &
\!\!\!\Delta_R^m f\Big(X^0,\ldots,X^{m-1},\begin{bmatrix} X^m & Z^{m+1}\\
0 & X^{m+1}
\end{bmatrix}
\Big)(Z^1,\ldots,Z^{m-1},\row[Z^m,0])\\
0   & \ddots & & \vdots \\
\vdots & \ddots & f(X^{m-1}) &  \vdots\\
 0 & \cdots &  0 & f\Big( \begin{bmatrix} X^m & Z^{m+1}\\
0 & X^{m+1}
\end{bmatrix}\Big)
\end{bmatrix}}\\
={\!\! \begin{bmatrix} f(X^0) &
\Delta_Rf(X^{0},X^{1})(Z^{1}) & \cdots &
\Delta_R^{m+1} f\Big(X^0,\ldots,X^{m+1}\Big)(Z^1,\ldots,Z^{m+1})\\
0   & \ddots &  \ddots & \vdots\\
\vdots & \ddots & f(X^m)  &  \Delta_Rf(X^{m},X^{m+1})(Z^{m+1})\\
 0 & \cdots & \ 0 & f(X^{m+1})
\end{bmatrix}}
\end{multline*}
}

Thus, we have \eqref{eq:bidiag} for $\ell=m+1$.
\end{proof}

\begin{proof}[Proof of Theorem \ref{thm:bidiag_k}]
By \eqref{eq:delta=delta_ext-k}, we may assume, without loss of
 generality,
  that $\Omega^{(k)}$ is similarity
 invariant, i.e., $\widetilde{\Omega}^{(k)}=\Omega^{(k)}$; see Remark
 \ref{rem:alternative} and Appendix \ref{app}.
We will prove the statement by induction on $\ell$. For $\ell=0$,
the two sides of \eqref{eq:bidiag_k} are identical. For $\ell=1$,
\eqref{eq:bidiag_k} coincides with \eqref{eq:uptr-k} up to
notation. Suppose now that the statement is true for all $k>0$ and
$\ell = m$, with some $m\in\mathbb{N}$. Let $k>0$ be fixed. Let
$X^0\in\Omega^{(0)}_{n_0}$, \ldots,
$X^{k-1}\in\Omega^{(k-1)}_{n_{k-1}}$,
$X^{k,0}\in\Omega^{(k)}_{n_{k,0}}$, \ldots,
$X^{k,m+1}\in\Omega^{(k)}_{n_{k,m+1}}$,
$Z^{k+1,1}\in\rmat{\module{M}_k}{n_{k,0}}{n_{k,1}}$, \ldots,
$Z^{k+1,m+1}\in\rmat{\module{M}_k}{n_{k,m}}{n_{k,m+1}}$,
$Z^1\in\rmat{\module{N}_1}{n_{0}}{n_{1}}$, \ldots,
$Z^{k-1}\in\rmat{(\module{N}_{k-1})}{n_{k-2}}{n_{k-1}}$,
$Z^{k,0}\in\rmat{\module{N}_k}{n_{k-1}}{n_{k,0}}$, \ldots,
$Z^{k,m+1}\in\rmat{\module{N}_k}{n_{k-1}}{n_{k,m+1}}$. By Lemma
\ref{lem:bidiag-homog},
$$\begin{bmatrix}
X^{k,j} & Z^{k+1,j+1} & 0 & \cdots & 0\\
0   & X^{k,j+1} & \ddots & \ddots & \vdots\\
\vdots & \ddots & \ddots & \ddots & 0 \\
\vdots &  & \ddots & X^{k,m} &  Z^{k+1,m+1}\\
 0 & \cdots & \cdots & 0 & X^{k,m+1}
\end{bmatrix}\in\Omega^{(k)}_{n_{k,j}+\cdots +n_{k,m+1}}
$$
for each $j=0,\ldots,m$. By \eqref{eq:uptr-k}, we have
\begin{multline*} f\left(X^0,\ldots,X^{k-1},\begin{bmatrix}
X^{k,0} & Z^{k+1,1} & 0 & \cdots & 0\\
0   & X^{k,1} & \ddots & \ddots & \vdots\\
\vdots & \ddots & \ddots & \ddots & 0 \\
\vdots &  & \ddots & X^{k,m} & Z^{k+1,m+1}\\
 0 & \cdots & \cdots & 0 & X^{k,m+1}
\end{bmatrix}\right)\\
\hfill(Z^1,\ldots,Z^{k-1}, \row\,[Z^{k,0},
\ldots,Z^{k,m+1}])\\
=\row\left[f(X^0,\ldots,X^{k-1},X^{k,0})
(Z^1,\ldots,Z^{k-1},Z^{k,0}),\right. \hfill \\
\Delta_Rf\left(X^0,\ldots,X^{k-1},\begin{bmatrix}
X^{k,1} & Z^{k+1,2} & 0 & \cdots & 0\\
0   & X^{k,2} & \ddots & \ddots & \vdots\\
\vdots & \ddots & \ddots & \ddots & 0 \\
\vdots &  & \ddots & X^{k,m} & Z^{k+1,m+1}\\
 0 & \cdots & \cdots & 0 & X^{k,m+1}
\end{bmatrix}\right)\\
\hfill(Z^1,\ldots,Z^{k-1},Z^{k,0},\row\,[Z^{k+1,1},0,
\ldots,0])\\
+f\left(X^0,\ldots,X^{k-1},\begin{bmatrix}
X^{k,1} & Z^{k+1,2} & 0 & \cdots & 0\\
0   & X^{k,2} & \ddots & \ddots & \vdots\\
\vdots & \ddots & \ddots & \ddots & 0 \\
\vdots &  & \ddots & X^{k,m} & Z^{k+1,m+1}\\
 0 & \cdots & \cdots & 0 & X^{k,m+1}
\end{bmatrix}\right)\\
\hfill\left.(Z^1,\ldots,Z^{k-1},\row\,[Z^{k,1},\ldots,Z^{k,m+1}])\right].
\end{multline*}
Applying the induction assumption to the right-hand side of the
latter (see also Remark \ref{rem:bidiag_k_special}), we obtain
\begin{multline*}
f\left(X^0,\ldots,X^{k-1},\begin{bmatrix}
X^{k,0} & Z^{k+1,1} & 0 & \cdots & 0\\
0   & X^{k,1} & \ddots & \ddots & \vdots\\
\vdots & \ddots & \ddots & \ddots & 0 \\
\vdots &  & \ddots & X^{k,m} & Z^{k+1,m+1}\\
 0 & \cdots & \cdots & 0 & X^{k,m+1}
\end{bmatrix}\right)\\
\hfill(Z^1,\ldots,Z^{k-1},\row\,[Z^{k,0},
\ldots,Z^{k,m+1}])\\
=\row\Big[f(X^0,\ldots,X^{k-1},X^{k,0})
(Z^1,\ldots,Z^{k-1},Z^{k,0}),\hfill\\
\Delta_Rf(X^0,\ldots,X^{k-1},X^{k,0},X^{k,1})
(Z^1,\ldots,Z^{k-1},Z^{k,0},Z^{k+1,1})\\
+f(X^0,\ldots,X^{k-1},X^{k,1}) (Z^1,\ldots,Z^{k-1},Z^{k,1}),\\
\Delta^2_Rf(X^0,\ldots,X^{k-1},X^{k,0},X^{k,1},X^{k,2})
(Z^1,\ldots,Z^{k-1},Z^{k,0},Z^{k+1,1},Z^{k+1,2})\\
+\Delta_Rf(X^0,\ldots,X^{k-1},X^{k,1},X^{k,2})
(Z^1,\ldots,Z^{k-1},Z^{k,1},Z^{k+1,2})\\
+f(X^0,\ldots,X^{k-1},X^{k,2}) (Z^1,\ldots,Z^{k-1},Z^{k,2}),
\end{multline*}
\begin{multline*}
\ldots,\\
\Delta^{m+1}_R(X^0,\ldots,X^{k-1},X^{k,0},\ldots,X^{k,m+1})
(Z^1,\ldots,Z^{k-1},Z^{k,0},Z^{k+1,1},\ldots,Z^{k+1,m+1})\\
+\sum\limits_{i=0}^m
\Delta^{m-i}_Rf(X^0,\ldots,X^{k-1},X^{k,i+1},\ldots,X^{k,m+1})\\
\hfill(Z^1,\ldots,Z^{k-1},Z^{k,i+1},Z^{k+1,i+2},\ldots,Z^{k+1,m+1})\Big]
\end{multline*}
\begin{multline*}
=\row_{0\le j\le
m+1}\Big[\sum_{i=0}^j\Delta_R^{j-i}f(X^0,\ldots,X^{k-1},X^{k,i},\ldots,X^{k,j})\hfill
\\
(Z^1,\ldots,Z^{k-1},Z^{k,i},Z^{k+1,i+1},\ldots,Z^{k+1,j})\Big].
\end{multline*}
Thus the statement is true also for $\ell=m+1$.
\end{proof}
\begin{rem}\label{rem:delta^ell_restrict}
It follows from Theorem \ref{thm:bidiag} (and Lemma
\ref{lem:bidiag-homog}) that if $\Omega'$ a right admissible nc
subset of $\Omega$ and $X^0$, \ldots, $X^\ell \in\Omega'$, then
$\Delta_R^\ell f(X^0,\ldots,X^\ell)$ is determined by
$f|_{\Omega'}$.
\end{rem}
\begin{rem}\label{rem:delta_interpret}
Using the tensor product interpretation of the values of a nc
function $f$ of order $k$ (see Remark \ref{rem:tensor_values}), we
can view the definition of $\Delta_Rf$ as follows. Let
$n_k=n_k^\prime + n_k^{\prime\prime}$. Then we have the $2\times
2$ block matrix decomposition of $n_k\times n_k$ matrices over
$\module{N}_k^*$:
$$\mat{\module{N}_k^*}{n_k}=\mat{\module{N}_k^*}{n_k^\prime}\oplus
\rmat{\module{N}_k^*}{n_k^\prime}{n_k^{\prime\prime}}\oplus
\rmat{\module{N}_k^*}{n_k^{\prime\prime}}{n_k^{\prime}}\oplus\mat{\module{N}_k^*}{n_k^{\prime\prime}}.$$
This decomposition induces the corresponding decomposition of
$$\mat{\module{N}_0}{n_0}\otimes\mat{\module{N}_1^*}{n_1}\otimes\cdots\otimes\mat{\module{N}_k^*}{n_k},$$
and
\begin{multline*}
f\left(X^0,\ldots,X^{k-1},\begin{bmatrix}
X^{k\prime} & Z\\
0 & X^{k\prime\prime}
\end{bmatrix}\right)
=f(X^0,\ldots,X^{k-1}, X^{k\prime})\\
+\Delta_Rf(X^0,\ldots,X^{k-1},
X^{k\prime},X^{k\prime\prime})(\cdot,\ldots,\cdot,Z)
+f(X^0,\ldots,X^{k-1}, X^{k\prime\prime}),
\end{multline*}
where
$$
f(X^0,\ldots,X^{k-1}, X^{k\prime})
\in\mat{\module{N}_0}{n_0}\otimes\mat{\module{N}_1^*}{n_1}\otimes\cdots\otimes
\mat{(\module{N}_{k-1}^*)}{n_{k-1}}\otimes\mat{\module{N}_k^*}{n_{k}^\prime},$$
\begin{multline*}
\Delta_Rf(X^0,\ldots,X^{k-1},
X^{k\prime},X^{k\prime\prime})(\cdot,\ldots,\cdot,Z)\\
\hfill\in \mat{\module{N}_0}{n_0}\otimes\mat{\module{N}_1^*}{n_1}
\otimes\cdots\otimes\mat{(\module{N}_{k-1}^*)}{n_{k-1}}\otimes
\rmat{\module{N}_k^*}{n_{k}^\prime}{n_k^{\prime\prime}},
\end{multline*}
$$f(X^0,\ldots,X^{k-1}, X^{k\prime\prime})
\in\mat{\module{N}_0}{n_0}\otimes\mat{\module{N}_1^*}{n_1}\otimes\cdots\otimes
\mat{(\module{N}_{k-1}^*)}{n_{k-1}}\otimes
\mat{\module{N}_k^*}{n_{k}^{\prime\prime}}
$$
(this is equivalent to \eqref{eq:uptr-k} when using the
multilinear mapping interpretation of the values of $f$).
Furthermore, the mapping
\begin{gather}\label{eq:natural-delta}
\rmat{\module{M}_k}{n_k^\prime}{n_k^{\prime\prime}}\longrightarrow
\mat{\module{N}_0}{n_0}\otimes\mat{\module{N}_1^*}{n_1}\otimes
\cdots\otimes\mat{(\module{N}_{k-1}^*)}{n_{k-1}}\otimes
\rmat{\module{N}_k^*}{n_{k}^\prime}{n_k^{\prime\prime}},\\
Z\longmapsto\Delta_Rf(X^0,\ldots,X^{k-1},
X^{k\prime},X^{k\prime\prime})(\cdot,\ldots,\cdot,Z)\label{eq:natur_map-delta}
\end{gather}
 is linear
(this is equivalent to Propositions \ref{prop:homog-k} and
\ref{prop:add-k}). We have a special case of the natural
isomorphism \eqref{eq:natural}--\eqref{eq:natur_map},
$$\mat{\module{N}_k^*}{n_k^\prime}\otimes\mat{\module{M}_k^*}{n_k^{\prime\prime}}
\overset{\sim}{\longrightarrow}\hom_\ring\left(\rmat{\module{M}_k}{n_k^\prime}{n_k^{\prime\prime}},
\rmat{\module{N}_k^*}{n_k^\prime}{n_k^{\prime\prime}}\right)$$ (we
assume here that the module $\module{M}_k$ is free and of finite
rank, in addition to our previous assumptions that the modules
$\module{N}_1$, \ldots, $\module{N}_k$ are free and of finite rank
and that the module $\module{N}_0$ is free). Tensoring with
$\mat{\module{N}_0}{n_0}\otimes\mat{\module{N}_1^*}{n_1}
\otimes\cdots\otimes\mat{(\module{N}_{k-1}^*)}{n_{k-1}}$ yields a
natural isomorphism
\begin{multline*}
\mat{\module{N}_0}{n_0}\otimes\mat{\module{N}_1^*}{n_1}\cdots
\otimes\mat{(\module{N}_{k-1}^*)}{n_{k-1}}\otimes
\mat{\module{N}_k^*}{n_k^\prime}\otimes\mat{\module{M}_k^*}{n_k^{\prime\prime}}\\
\overset{\sim}{\longrightarrow}\hom_\ring\!\!\Big(\rmat{\module{M}_k}{n_k^\prime}{n_k^{\prime\prime}},
\mat{\module{N}_0}{n_0}\otimes\mat{\module{N}_1^*}{n_1}\otimes
\cdots\otimes\mat{(\module{N}_{k-1}^*)}{n_{k-1}}\otimes
\rmat{\module{N}_k^*}{n_{k}^\prime}{n_k^{\prime\prime}}\Big).
\end{multline*}
It follows from
\eqref{eq:natural-delta}--\eqref{eq:natur_map-delta} that
\begin{multline*}
\Delta_Rf(X^0,\ldots,X^{k-1},
X^{k\prime},X^{k\prime\prime})\\
\in\mat{\module{N}_0}{n_0}\otimes\mat{\module{N}_1^*}{n_1}\otimes
\cdots\otimes\mat{(\module{N}_{k-1}^*)}{n_{k-1}}\otimes
\mat{\module{N}_k^*}{n_k^\prime}\otimes\mat{\module{M}_k^*}{n_k^{\prime\prime}},
\end{multline*}
which is equivalent to the first part of Theorem \ref{thm:dif-op}.
\end{rem}
\begin{rem}\label{rem:delta_tensor_prod}
It follows from the definition of $\Delta_R$ that its action on a
tensor product of (possibly, higher order) nc functions (see
Remark \ref{rem:tensor_prod}) boils down to its action on the last
factor:
\begin{equation}\label{eq:delta_tensor_prod}
\Delta_R(f\otimes g)=f\otimes\Delta_Rg,
\end{equation}
where
\begin{eqnarray*}
f\in\tclass{k}(\Omega^{(0)},\ldots,\Omega^{(k)};\ncspacej{\module{N}}{0},\ldots,\ncspacej{\module{N}}{k}),\\
g\in\tclass{\ell}(\Omega^{(k+1)},\ldots,\Omega^{(k+\ell)};\ncspacej{\module{L}^*}{0},
\ncspacej{\module{L}}{1},\ldots,\ncspacej{\module{L}}{\ell}),\end{eqnarray*}
and
\begin{equation}\label{eq:delta_tensor_prod_k}
\Delta_R(f^0\otimes\cdots\otimes f^{k-1}\otimes
f^k)=f^0\otimes\cdots\otimes f^{k-1}\otimes\Delta_Rf^k,
\end{equation}
where $f^0\in\tclass{0}(\Omega^{(0)};\ncspacej{\module{N}}{0})$,
$f^1\in\tclass{0}(\Omega^{(1)};\ncspacej{\module{N}^*}{1})$,
\ldots,
$f^k\in\tclass{0}(\Omega^{(k)};\ncspacej{\module{N}^*}{k})$.
\end{rem}
\begin{rem}\label{rem:j-delta}
More generally, we can define a difference-differential operator
$_j\Delta_R$ \index{$_j\Delta_R$}
 by evaluating an nc function of order $k$ on a $(k+1)$-tuple of
 square matrices with the $(j+1)$-th argument upper triangular, $j=0,\ldots, k$.
Hence, $\Delta_R={}_k\Delta_R$. To define ${}_0\Delta_Rf$ for
$f\in\tclass{k}(\Omega^{(0)},\ldots,\Omega^{(k)};\ncspacej{\module{N}}{0},\ldots,\ncspacej{\module{N}}{k})$,
where $\Omega^{(0)}\subseteq\ncspacej{\module{M}}{0}$ is a right
admissible nc set, let $X^{0\prime}\in\Omega^{(0)}_{n_0^\prime}$,
$X^{0\prime\prime}\in\Omega^{(0)}_{n_0^{\prime\prime}}$,
$X^1\in\Omega^{(1)}_{n_1}$, \ldots,
$X^{k}\in\Omega^{(k)}_{n_{k}}$,
$Z^{1\prime}\in\rmat{\module{N}_1}{n_0^\prime}{n_1}$,
$Z^{1\prime\prime}\in\rmat{\module{N}_1}{n_0^{\prime\prime}}{n_1}$,
$Z^2\in\rmat{\module{N}_{2}}{n_{1}}{n_{2}}$ \ldots,
$Z^{k}\in\rmat{\module{N}_{k}}{n_{k-1}}{n_{k}}$. Let
$Z\in\rmat{\module{M}_0}{n_0^\prime}{n_0^{\prime\prime}}$ be
such that $\begin{bmatrix} X^{0\prime} & Z\\
0 & X^{0\prime\prime}
\end{bmatrix}\in\Omega^{(0)}_{n_0^\prime+n_0^{\prime\prime}}$. Then,
similarly to Proposition \ref{prop:delta_r-k},
\begin{multline}\label{eq:uptr_left0-k}
 f\left(\begin{bmatrix} X^{0\prime} & Z\\
0 & X^{0\prime\prime}
\end{bmatrix},X^1,\ldots,X^k\right)(\col\,
[Z^{1\prime},Z^{1\prime\prime}],Z^2,\ldots,Z^k)\\
 =\col\Big
[f(X^{0\prime},X^1,\ldots,X^k)(Z^{1\prime},Z^2,\ldots,Z^k)\hfill\\
+ {}_0\Delta_Rf(X^{0\prime},X^{0\prime\prime},X^1,\ldots,X^k)
(Z,Z^{1\prime\prime},Z^2,\ldots,Z^k),\\
 \hfill
f(X^{0\prime\prime},X^1,\ldots,X^k)
(Z^{1\prime\prime},Z^2,\ldots,Z^k)\Big].
\end{multline}
Here
${}_0\Delta_Rf(X^{0\prime},X^{0\prime\prime},X^1,\ldots,X^{k})
(Z,Z^{1\prime\prime},Z^2,\ldots,Z^k)$ is determined uniquely by
\eqref{eq:uptr_left0-k}, is independent of
$Z^{1\prime}$ and is homogeneous in $Z$ whenever $\begin{bmatrix} X^{0\prime} & Z\\
0 & X^{0\prime\prime}
\end{bmatrix}\in\Omega^{(0)}_{n_0^\prime +n_0^{\prime\prime}}$.
We can define
${}_0\Delta_Rf(X^{0\prime},X^{0\prime\prime},X^1,\ldots,X^{k})
(Z,Z^{1\prime\prime},Z^2,\ldots,Z^k)$ for all
$Z\in\rmat{\module{M}_0}{n_0^\prime}{n_0^{\prime\prime}}$ by
homogeneity analogously to \eqref{eq:def_delta_r-k}. Similarly to
Proposition \ref{prop:add-k},
 it is linear in $Z$; similarly
to Theorem \ref{thm:dif-op},
$${}_0\Delta_Rf\in\tclass{k+1}(\Omega^{(0)},\Omega^{(0)},\Omega^{(1)},\ldots,\Omega^{(k)};
\ncspacej{\module{N}}{0},\ncspacej{\module{M}}{0},\ncspacej{\module{N}}{1},\ldots,\ncspacej{\module{N}}{k}).$$

To define ${}_j\Delta_Rf$, $1\le j\le k-1$, we assume now that
$\Omega^{(j)}$ is a right admissible nc set. Let
$X^0\in\Omega^{(0)}_{n_0}$, \ldots,
$X^{j-1}\in\Omega^{(j-1)}_{n_{j-1}}$,
$X^{j\prime}\in\Omega^{(j)}_{n_j^\prime}$,
$X^{j\prime\prime}\in\Omega^{(j)}_{n_j^{\prime\prime}}$,
$X^{j+1}\in\Omega^{(j+1)}_{n_{j+1}}$, \ldots,
$X^{k}\in\Omega^{(k)}_{n_{k}}$,
$Z^1\in\rmat{\module{N}_{1}}{n_{0}}{n_{1}}$, \ldots,
$Z^{j-1}\in\rmat{(\module{N}_{j-1})}{n_{j-2}}{n_{j-1}}$,
$Z^{j\prime}\in\rmat{\module{N}_j}{n_{j-1}}{n_j^\prime}$,
$Z^{j\prime\prime}\in\rmat{\module{N}_j}{n_{j-1}}{n_j^{\prime\prime}}$,
$Z^{j+1\prime}\in\rmat{(\module{N}_{j+1})}{n_{j}^\prime}{n_{j+1}}$,
$Z^{j+1\prime\prime}\in\rmat{(\module{N}_{j+1})}{n_{j}^{\prime\prime}}{n_{j+1}}$,
$Z^{j+2}\in\rmat{(\module{N}_{j+2})}{n_{j+1}}{n_{j+2}}$, \ldots,
$Z^{k}\in\rmat{\module{N}_{k}}{n_{k-1}}{n_{k}}$. Let
$Z\in\rmat{\module{M}_j}{n_j^\prime}{n_j^{\prime\prime}}$ be
such that $\begin{bmatrix} X^{j\prime} & Z\\
0 & X^{j\prime\prime}
\end{bmatrix}\in\Omega^{(j)}_{n_j^\prime+n_j^{\prime\prime}}$. Then,
similarly to Proposition \ref{prop:delta_r-k},
\begin{multline}\label{eq:uptr_leftj-k}
 f\left(X^0,\ldots,X^{j-1},\begin{bmatrix} X^{j\prime} & Z\\
0 & X^{j\prime\prime}
\end{bmatrix},X^{j+1},\ldots,X^k\right)\\
\hfill(Z^1,\ldots,Z^{j-1},\row\,
[Z^{j\prime},Z^{j\prime\prime}],\col
\,[Z^{j+1\prime},Z^{j+1\prime\prime}],Z^{j+2},\ldots,Z^k)\\
=
f(X^0,\ldots,X^{j-1},X^{j\prime},X^{j+1},\ldots,X^k)(Z^1,\ldots,Z^{j-1},
Z^{j\prime},Z^{j+1\prime},Z^{j+2},\ldots,Z^k)\\
+
{}_j\Delta_Rf(X^0,\ldots,X^{j-1},X^{j\prime},X^{j\prime\prime},X^{j+1},\ldots,X^k)\\
\hfill (Z^1,\ldots,Z^{j-1},Z^{j\prime},Z,Z^{j+1\prime\prime},Z^{j+2},\ldots,Z^k)\\
+ f(X^0,\ldots,X^{j-1},X^{j\prime\prime},X^{j+1},\ldots,X^k)
(Z^1,\ldots,Z^{j-1},Z^{j\prime\prime},Z^{j+1\prime\prime},Z^{j+2},\ldots,Z^k).
\end{multline}
Here
\begin{multline*}
{}_j\Delta_Rf(X^0,\ldots,X^{j-1},X^{j\prime},X^{j\prime\prime},X^{j+1},\ldots,X^k)\\
\hfill(Z^1,\ldots,Z^{j-1},Z^{j\prime},Z,Z^{j+1\prime\prime},Z^{j+2},\ldots,Z^k)
\end{multline*}
is determined uniquely by \eqref{eq:uptr_leftj-k}, is independent
of
$Z^{j\prime\prime}$ and $Z^{j+1\prime}$ and is homogeneous in $Z$ whenever $\begin{bmatrix} X^{j\prime} & Z\\
0 & X^{j\prime\prime}
\end{bmatrix}\in\Omega^{(j)}_{n_j^\prime +n_j^{\prime\prime}}$.
We can define
\begin{multline*}
{}_j\Delta_Rf(X^0,\ldots,X^{j-1},X^{j\prime},X^{j\prime\prime},X^{j+1},\ldots,X^k)\\
\hfill(Z^1,\ldots,Z^{j-1},Z^{j\prime},Z,Z^{j+1\prime\prime},Z^{j+2},\ldots,Z^k)
\end{multline*}
 for all
$Z\in\rmat{\module{M}_j}{n_j^\prime}{n_j^{\prime\prime}}$ by
homogeneity analogously to \eqref{eq:def_delta_r-k}. Similarly to
Proposition \ref{prop:add-k},
 it is linear in $Z$; similarly
to Theorem \ref{thm:dif-op},
\begin{multline*}
{}_j\Delta_Rf\in\tclass{k+1}(\Omega^{(0)},\ldots,\Omega^{(j-1)},\Omega^{(j)},\Omega^{(j)},\Omega^{(j+1)},\ldots,
\Omega^{(k)}
;\\
\ncspacej{\module{N}}{0},\ldots,\ncspacej{\module{N}}{j},
\ncspacej{\module{M}}{j}, \ncspacej{\module{N}}{j+1},
\ldots,\ncspacej{\module{N}}{k}).
\end{multline*}

Using Proposition \ref{prop:higher-ncfun-ext}, we can establish a
$j$-analogue of \eqref{eq:delta=delta_ext-k},
\begin{multline}\label{eq:j_delta=j_delta_ext-k}
{}_j\Delta_Rf(X^0,\ldots,X^{j-1},X^{j\prime},X^{j\prime\prime},X^{j+1},\ldots,X^k)\\
\hfill(Z^1,\ldots,Z^{j-1},Z^{j\prime},Z,Z^{j+1\prime\prime},Z^{j+2},\ldots,Z^k)\\
={}_j\Delta_R\widetilde{f}(X^0,\ldots,X^{j-1},X^{j\prime},X^{j\prime\prime},X^{j+1},\ldots,X^k)\\
\hfill(Z^1,\ldots,Z^{j-1},Z^{j\prime},Z,Z^{j+1\prime\prime},Z^{j+2},\ldots,Z^k)
\end{multline}

One can also formulate the definition of ${}_j\Delta_Rf$, $0\le
j\le k$, equivalently, using the tensor product interpretation of
the values of the nc function $f$ of order $k$, similarly to
Remark \ref{rem:delta_interpret}.

Analogously to Remark \ref{rem:delta_tensor_prod}, it follows from
the definition of ${}_j\Delta_R$ that its action on a tensor
product of (possibly, higher order) nc functions boils down to its
action on one of the factors:
\begin{equation}\label{eq:j-delta_tensor_prod}
{}_j\Delta_R (f \otimes g) = \begin{cases}
{}_j\Delta_R(f) \otimes g, & 0 \leq j \leq k,\\
f \otimes {}_{j-k-1}\Delta_R(g), & k+1 \leq j \leq k+\ell+1,
\end{cases}
\end{equation}
where
\begin{eqnarray*}
f\in\tclass{k}(\Omega^{(0)},\ldots,\Omega^{(k)};\ncspacej{\module{N}}{0},\ldots,\ncspacej{\module{N}}{k}),\\
g\in\tclass{\ell}(\Omega^{(k+1)},\ldots,\Omega^{(k+\ell)};
\ncspacej{\module{L}}{0}^*,\ldots,\ncspacej{\module{L}}{\ell}),
\end{eqnarray*}
and
\begin{equation}\label{eq:j-delta_tensor_prod_k}
{}_j\Delta_R (f^0\otimes\cdots\otimes f^{j-1}\otimes f^j\otimes
f^{j+1}\otimes\cdots\otimes f^k)=(f^0\otimes\cdots\otimes
f^{j-1}\otimes \Delta_R f^j\otimes f^{j+1}\otimes\cdots\otimes
f^k).
\end{equation}
\end{rem}

\section{First order difference formulae for higher order nc
functions}\label{subsec:Lagrange_k} We establish now the analogue
of Theorem \ref{thm:Lagrange} for higher order
 nc functions. An iterative use of this result will lead us in
 Chapter \ref{sec:TT} to a
 nc analogue of  Brook Taylor expansion.
\begin{thm}\label{thm:Lagrange_k}
Let
$f\in\tclass{k}(\Omega^{(0)},\ldots,\Omega^{(k)};\ncspacej{\module{N}}{0},\ldots,\ncspacej{\module{N}}{k})$,
where $\Omega^{(k)}\subseteq\ncspacej{\module M}{k}$ is a
 right admissible nc set. Then for
all $n_0,\ldots,n_k\in\mathbb{N}$, and arbitrary
$X^0\in\Omega^{(0)}_{n_0}$, \ldots,
$X^{k-1}\in\Omega^{(k-1)}_{n_{k-1}}$, $X,Y\in\Omega^{(k)}_{n_k}$,
$Z^1\in\rmat{\module{N}_1}{n_0}{n_1}$, \ldots,
$Z^k\in\rmat{\module{N}_k}{n_{k-1}}{n_k}$,
\begin{multline}\label{eq:RightLagr_k}
f(X^0,\ldots,X^{k-1},X)(Z^1,\ldots,Z^k)-f(X^0,\ldots,X^{k-1},Y)
(Z^1,\ldots,Z^k)\\
= \Delta_Rf(X^0,\ldots,X^{k-1},Y,X)(Z^1,\ldots,Z^k,X-Y)
\end{multline}
and
\begin{multline}\label{eq:RightLagr_k'}
 f(X^0,\ldots,X^{k-1},X)(Z^1,\ldots,Z^k)-f(X^0,\ldots,X^{k-1},Y)
 (Z^1,\ldots,Z^k)\\
  = \Delta_Rf(X^0,\ldots,X^{k-1},X,Y)(Z^1,\ldots,Z^k,X-Y).
\end{multline}
\end{thm}

As with Theorem \ref{thm:Lagrange}, it is a consequence of the
following more general result.

\begin{thm}\label{thm:gen-Lagrange_k}
Let
$f\in\tclass{k}(\Omega^{(0)},\ldots,\Omega^{(k)};\ncspacej{\module{N}}{0},\ldots,\ncspacej{\module{N}}{k})$,
where $\Omega^{(k)}\subseteq\ncspacej{\module M}{k}$ is a
 right admissible nc set. Then for
all $n_0,\ldots,n_{k-1}\in\mathbb{N}$ and arbitrary
$X^0\in\Omega^{(0)}_{n_0}$, \ldots,
$X^{k-1}\in\Omega^{(k-1)}_{n_{k-1}}$, $X\in\Omega^{(k)}_{n}$,
$Y\in\Omega^{(k)}_{m}$, $Z^1\in\rmat{\module{N}_1}{n_0}{n_1}$,
\ldots, $Z^{k-1}\in\rmat{(\module{N}_{k-1})}{n_{k-2}}{n_{k-1}}$,
$Z^k\in\rmat{\module{N}_{k}}{n_{k-1}}{m}$,
$S\in\rmat{\ring}{m}{n}$,
\begin{multline}\label{eq:gen-RightLagr_k}
f(X^0,\ldots,X^{k-1},X)(Z^1,\ldots,Z^kS)-f(X^0,\ldots,X^{k-1},Y)
(Z^1,\ldots,Z^k)S\\
= \Delta_Rf(X^0,\ldots,X^{k-1},Y,X)(Z^1,\ldots,Z^k,SX-YS).
\end{multline}
\end{thm}

\begin{proof} By \eqref{eq:delta=delta_ext-k}, we may assume, without loss of
 generality,
  that $\Omega^{(k)}$ is similarity
 invariant, i.e., $\widetilde{\Omega}^{(k)}=\Omega^{(k)}$; see Remark
 \ref{rem:alternative} and Appendix \ref{app}. By Proposition
 \ref{prop:adm-env},
$$\begin{bmatrix} Y & SX-YS\\
0 & X
\end{bmatrix}\in\Omega^{(k)}_{m+n}.$$
We have
$$\begin{bmatrix} Y & SX-YS\\
0 & X
\end{bmatrix}\begin{bmatrix}
S\\
I_n
\end{bmatrix}=\begin{bmatrix}
S\\
I_n
\end{bmatrix}X.
$$
By \eqref{eq:3X^k},
\begin{multline*}
f\left(X^0,\ldots,X^{k-1},\begin{bmatrix} Y & SX-YS\\
0 & X
\end{bmatrix}\right)(Z^1,\ldots,Z^{k-1},\row\,[Z^k,0])\begin{bmatrix}
S\\
I_n
\end{bmatrix}\\
=f(X^0,\ldots,X^{k-1},X)(Z^1,\ldots,Z^{k-1},Z^kS).
\end{multline*}
By Proposition \ref{prop:delta_r-k}, the left-hand side is equal
to \begin{multline*}
\row\,[f(X^0,\ldots,X^{k-1},Y)(Z^1,\ldots,Z^k),\\
\Delta_Rf(X^0,\ldots,X^{k-1},Y,X)(Z^1,\ldots,Z^k, SX-YS)]
\col\,[S,I_n]\\
 =f(X^0,\ldots,X^{k-1},Y)(Z^1,\ldots,Z^k)S\hfill\\
+\Delta_Rf(X^0,\ldots,X^{k-1},Y,X)(Z^1,\ldots,Z^k,SX-YS),
\end{multline*}
hence \eqref{eq:gen-RightLagr_k} follows.
\end{proof}
\begin{rem}\label{rem:deriv_k}
Let $\ring=\mathbb{R}$ or $\ring=\mathbb{C}$. Setting $X=Y+tZ$
(with $t\in\mathbb{R}$ or $t\in\mathbb{C}$) in Theorem
\ref{thm:Lagrange_k}, we obtain
\begin{multline*}
f(X^0,\ldots,X^{k-1},Y+tZ)(Z^1,\ldots,Z^{k})-
f(X^0,\ldots,X^{k-1},Y)(Z^1,\ldots,Z^{k})\\
=t\Delta_Rf(X^0,\ldots,X^{k-1},Y,Y+tZ)(Z^1,\ldots,Z^{k},Z).
\end{multline*}
 Under appropriate continuity conditions,
it follows that
$$\Delta_Rf(X^0,\ldots,X^{k-1},Y,Y)(Z^1,\ldots,Z^{k},Z)$$
 is the directional derivative of
$$f(X^0,\ldots,X^{k-1},\cdot)(Z^1,\ldots,Z^{k})$$
  at $Y$ in
the direction $Z$, i.e.,
$$\Delta_Rf(X^0,\ldots,X^{k-1},Y,Y)(Z^1,\ldots,Z^{k},\cdot)$$
is the differential of
$$f(X^0,\ldots,X^{k-1},\cdot)(Z^1,\ldots,Z^{k})$$
  at $Y$.
\end{rem}
\begin{rem}\label{rem:j-Lagrange_k}
Using the difference-differential operators ${}_j\Delta_R$ (see Remark \ref{rem:j-delta}), we can establish first
order difference formulae for higher order nc functions (cf. Theorem \ref{thm:Lagrange_k}) with respect to the
$j$-th variable: for $j=0,\ldots,k$,
\begin{multline}\label{eq:j-RightLagr_k}
f(X^0,\ldots, X^{j-1},X,X^{j+1},\ldots,X^k)(Z^1,\ldots,Z^k)\\
-f(X^0,\ldots, X^{j-1},Y,X^{j+1},\ldots,X^k)(Z^1,\ldots,Z^k)\\
={}_j\Delta_Rf(X^0,\ldots,
X^{j-1},Y,X,X^{j+1},\ldots,X^k)(Z^1,\ldots,Z^j,X-Y,Z^{j+1},\ldots,Z^k)
\end{multline}
and
\begin{multline}\label{eq:j-RightLagr_k'}
f(X^0,\ldots, X^{j-1},X,X^{j+1},\ldots,X^k)(Z^1,\ldots,Z^k)\\
-f(X^0,\ldots, X^{j-1},Y,X^{j+1},\ldots,X^k)(Z^1,\ldots,Z^k)\\
={}_j\Delta_Rf(X^0,\ldots, X^{j-1},X,Y,X^{j+1},\ldots,X^k)(Z^1,\ldots,Z^j,X-Y,Z^{j+1},\ldots,Z^k).
\end{multline}

More generally (cf. \ref{thm:gen-Lagrange_k}), we have
\begin{multline}\label{eq:gen-0-RightLagr_k}
Sf(X,X^1,\ldots,X^k)(Z^1,\ldots,Z^k)
-f(Y,X^1,\ldots,X^k)(SZ^1,Z^2,\ldots,Z^k)\\
={}_0\Delta_Rf(Y,X,X^1,\ldots, X^k)(SX-YS,Z^1,\ldots,Z^k)
\end{multline}
and, for $0<j<k$,
\begin{multline}\label{eq:gen-j-RightLagr_k}
f(X^0,\ldots,
X^{j-1},X,X^{j+1},\ldots,X^k)(Z^1,\ldots,Z^{j-1},Z^jS,Z^{j+1},\ldots,Z^k)\\
-f(X^0,\ldots, X^{j-1},Y,X^{j+1},\ldots,X^k)(Z^1,\ldots,Z^j,SZ^{j+1},Z^{j+2},\ldots,Z^k)\\
\!=\!{}_j\Delta_Rf(X^0,\ldots,
X^{j-1},Y,X,X^{j+1},\ldots,X^k)(Z^1,\ldots,Z^{j},SX-YS,Z^{j+1},\ldots,Z^k).
\end{multline}
The proof is similar to the proof of Theorem
\ref{thm:gen-Lagrange_k}, except that we use \eqref{eq:3X^0}
together with \eqref{eq:uptr_left0-k} (in the case of
\eqref{eq:gen-0-RightLagr_k}) or \eqref{eq:3X^j} together with
\eqref{eq:uptr_leftj-k} (in the case of
\eqref{eq:gen-j-RightLagr_k}) instead of using \eqref{eq:3X^k}
together with \eqref{eq:uptr-k}. We can also use the similarity
invariant envelope $\widetilde{\Omega}^{(j)}$ of the corresponding
right admissible nc set $\Omega^{(j)}$, and the equality
\eqref{eq:j_delta=j_delta_ext-k}.
\end{rem}
\begin{rem}\label{rem:newdef_k-ncfun}
Clearly, the formulae \eqref{eq:gen-0-RightLagr_k}, \eqref{eq:gen-j-RightLagr_k}, and \eqref{eq:gen-RightLagr_k}
generalize \eqref{eq:3X^0}, \eqref{eq:3X^j}, and \eqref{eq:3X^k}, respectively.
\end{rem}

\section{NC integrability}\label{subsec:integra} We derive now
necessary conditions for a higher order nc function to be
integrable, i.e., to be the result of applying a higher order
difference-differential operator to a nc function (or to a higher
order nc function of a lower order).
\begin{thm}\label{thm:integra}
Let $f\in\tclass{k}(\Omega;\ncspace{\module{N}},\ncspace{\module{M}},\ldots,\ncspace{\module{M}})$, with
$\Omega\subseteq\ncspace{\module{M}}$ a right admissible nc set. If there exists a nc function
$g\in\tclass{0}(\Omega;\ncspace{\module{N}})$ such that $f=\Delta_R^kg$, then
\begin{equation}
{}_i\Delta_Rf={}_j\Delta_Rf, \qquad i,j=0,\ldots,k.
\end{equation}
\end{thm}
\begin{cor}\label{cor:j-delta_is_delta}
For every $g\in\tclass{0}(\Omega;\ncspace{\module{N}})$, $\ell\in\mathbb{N}$, and $j_s\in\mathbb{N}$ with $0\le
j_s\le s$, $s=1,\ldots,\ell$, one has
\begin{equation}\label{eq:equal_delta_words}
{}_{j_\ell}\Delta_R\cdots{}_{j_1}\Delta_Rg=\Delta_R^\ell g.
\end{equation}
\end{cor}
\begin{proof}[Proof of Theorem \ref{thm:integra}]
It clearly suffices to show that
${}_j\Delta_R\Delta_R^kg=\Delta_R^{k+1}g$ for $0\le j<k$. We prove
this for $0<j<k$ (the proof for the case where $j=0$ is analogous,
using \eqref{eq:uptr_left0-k} instead of \eqref{eq:uptr_leftj-k}).
By \eqref{eq:j_delta=j_delta_ext-k}, we may assume, without loss
of
 generality,
  that $\Omega$ is similarity
 invariant, i.e., $\widetilde{\Omega}=\Omega$; see Remark
 \ref{rem:alternative} and Appendix \ref{app}.
 Let $X^0\in\Omega_{n_0}$, \ldots, $X^{j-1}\in\Omega_{n_{j-1}}$,
$X^{j\prime}\in\Omega_{n_{j}^\prime}$,
$X^{j\prime\prime}\in\Omega_{n_{j}^{\prime\prime}}$,
$X^{j+1}\in\Omega_{n_{j+1}}$, \ldots, $X^{k}\in\Omega_{n_{k}}$,
$Z^{1}\in\rmat{\module{M}}{n_0}{n_1}$, \ldots,
$Z^{j-1}\in\rmat{\module{M}}{n_{j-2}}{n_{j-1}}$,
$Z^{j\prime}\in\rmat{\module{M}}{n_{j-1}}{n_{j}^\prime}$,
$Z\in\rmat{\module{M}}{n_{j}^\prime}{n_{j}^{\prime\prime}}$,
$Z^{j+1\prime\prime}\in\rmat{\module{M}}{n_{j}^{\prime\prime}}{n_{j+1}}$,
$Z^{j+2}\in\rmat{\module{M}}{n_{j+1}}{n_{j+2}}$, \ldots,
$Z^k\in\rmat{\module{M}}{n_{k-1}}{n_k}$. By Proposition
\ref{prop:adm-env}, $\begin{bmatrix} X^{j\prime} & Z\\
0 &
X^{j\prime\prime}\end{bmatrix}\in\Omega_{n_j^\prime+n_j^{\prime\prime}}.$
According to \eqref{eq:uptr_leftj-k},
\begin{multline}\label{eq:animal}
{}_j\Delta_R\Delta_R^kg(X^0,\ldots,X^{j-1},X^{j\prime},X^{j\prime\prime},X^{j+1},\ldots,X^k)\\
(Z^1,\ldots,Z^{j-1},Z^{j\prime},Z,Z^{j+1\prime\prime},Z^{j+2},\ldots,Z^k)\\
=\Delta_R^kg(X^0,\ldots,X^{j-1},\begin{bmatrix} X^{j\prime} & Z\\
0 & X^{j\prime\prime}\end{bmatrix},X^{j+1},\ldots,X^k)\\
(Z^1,\ldots,Z^{j-1},\row[Z^{j\prime}, 0],\col[0,Z^{j+1\prime\prime}],Z^{j+2},\ldots,Z^k).
\end{multline}
By \eqref{eq:bidiag}, this is equal to the $(1,k+1)$ block entry of the $(k+1)\times (k+1)$ block matrix
$$f\left(\begin{bmatrix}
X^0     & Z^1    & 0       & \cdots  & \cdots                 & \cdots & \cdots & \cdots & 0\\
0       & X^1 & \ddots  & \ddots  &                        &        &        &        & \vdots\\
\vdots  & \ddots & \ddots  & Z^{j-1} & \ddots                 &        &        &        & \vdots \\
\vdots  &        &   \ddots      & X^{j-1} & \begin{bmatrix}
                                         Z^{j\prime} & 0
                                       \end{bmatrix}          &  \ddots       &       &        & \vdots\\
\vdots        &        &        &  \ddots        & \begin{bmatrix}
                                       X^{j\prime} & Z\\
                                       0      & X^{j\prime\prime}
                                       \end{bmatrix}          & \begin{bmatrix}
                                                                 0\\
                                                                 Z^{j+1\prime\prime}
                                                                 \end{bmatrix} & \ddots &        &  \vdots    \\
\vdots &         &     &       &                          \ddots   &   X^{j+1} &    Z^{j+2} & \ddots & \vdots \\
\vdots &         &     &       &                             &   \ddots   & \ddots & \ddots & 0\\
\vdots &         &     &       &                             &            &  \ddots     &   X^{k-1}     &   Z^k\\
0      & \cdots  & \cdots & \cdots & \cdots                  &  \cdots    & \cdots       & 0           & X^k
\end{bmatrix}
\right),$$ where the evaluation of $f$ at the block bidiagonal
matrix above is well defined, since by Lemma
\ref{lem:bidiag-homog} that matrix belongs to
$\Omega_{n_0+\cdots+n_{j-1}+n_j^\prime+n_j^{\prime\prime}+n_{j+1}+\cdots+n_k}.$
We consider $\begin{bmatrix}
                                         Z^{j\prime} & 0
                                       \end{bmatrix} $, $\begin{bmatrix}
                                       X^{j\prime} & Z\\
                                       0      & X^{j\prime\prime}
                                       \end{bmatrix}$, and $\begin{bmatrix}
                                                                 0\\
                                                                 Z^{j+1\prime\prime}
\end{bmatrix}$ as single blocks. If we view this block matrix as a $(k+2)\times (k+2)$ block matrix, removing the
brackets from those blocks, then the expression in \eqref{eq:animal} becomes the $(1,k+2)$ block entry, which by
\eqref{eq:bidiag} is equal to
\begin{multline*}
\Delta_R^{k+1}g(X^0,\ldots,X^{j-1},X^{j\prime},X^{j\prime\prime},X^{j+1},\ldots,X^k)\\
(Z^1,\ldots,Z^{j-1},Z^{j\prime},Z,Z^{j+1\prime\prime},Z^{j+2},\ldots,Z^k).
\end{multline*}

\end{proof}

\label{COMULT}

\section{Higher order directional nc difference-differential
operators} \label{subsec:dir_difdif-k} Similarly to Section
\ref{subsec:dir_difdif}, we will also consider a directional
version of higher order difference-differential operators. It is
introduced by iterating directional nc differ\-ence-differential
operators on higher order nc functions.

Let $\Omega^{(j)}\subseteq\ncspacej{\module{M}}{j}$,
$j=0,\ldots,k$, be nc sets,
 and let
$\Omega^{(k)}\subseteq\ncspacej{\module{M}}{k}$ be right
admissible. Let
$f\in\tclass{k}(\Omega^{(0)},\ldots,\Omega^{(k)};\ncspacej{\module{N}}{0},\ldots,\ncspacej{\module{N}}{k})$.
Given $\mu\in\module{M}_k$, we define for all $n_0$, \ldots,
$n_{k+1}\in\mathbb{N}$, $X^0\in\Omega^{(0)}_{n_0}$, \ldots,
$X^k\in\Omega^{(k)}$, $X^{k+1}\in\Omega^{(k)}_{n_{k+1}}$ the
$(k+1)$-linear mapping
$$\Delta_{R,\mu}f(X^0,\ldots,X^{k+1})\colon
\rmat{\module{N}_1}{n_0}{n_1}\times\cdots\times\rmat{\module{N}_k}{n_{k-1}}{n_k}
\times\rmat{\ring}{n_k}{n_{k+1}}
\longrightarrow\rmat{\module{N}_0}{n_0}{n_{k+1}}$$ by
\begin{multline}\label{eq:dir_difdif_k}
\Delta_{R,\mu}f(X^0,\ldots,X^{k+1})(Z^1,\ldots,Z^{k},A)\\
=\Delta_Rf(X^0,\ldots,X^{k+1})(Z^1,\ldots,Z^{k},A\mu).
\end{multline}
Furthermore,\index{$\Delta_{R,\mu}f(X^0,\ldots,X^{k+1})$}
$\Delta_{R,\mu}f\in\tclass{k+1}(\Omega^{(0)},\ldots,\Omega^{(k)},\Omega^{(k)};\ncspacej{\module{N}}{0},
\ldots,\ncspacej{\module{N}}{k},\ncspace{\ring})$.

In the special case of $\module{M}_k=\ring^d$, we define the
\emph{$j$-th right partial nc difference-differential operator}
\index{right partial nc difference-differential operator}  by
$\Delta_{R,j}=\Delta_{R,e_j}$. By linearity, we have
\begin{multline}\label{eq:r_difdecomp-k}
\Delta_Rf(X^0,\ldots,X^{k+1})(Z^1,\ldots,Z^{k+1})\\
=\sum_{j=1}^d
\Delta_{R,j}f(X^0,\ldots,X^{k+1})(Z^1,\ldots,Z^{k},Z_j^{k+1}).
\end{multline}
It follows that the first order difference formulae of Theorem
\ref{thm:Lagrange_k} can be written in this case in terms of
partial nc difference-differential operators:
\begin{align*}
& f(X^0,\ldots,X^{k-1},X)(Z^1,\ldots,Z^k)-f(X^0,\ldots,X^{k-1},Y)
(Z^1,\ldots,Z^k)\\
&=
\sum_{j=1}^d\Delta_{R,j}f(X^0,\ldots,X^{k-1},Y,X)(Z^1,\ldots,Z^k,X_j-Y_j)\\
&=\sum_{j=1}^d\Delta_{R,j}f(X^0,\ldots,X^{k-1},X,Y)(Z^1,\ldots,Z^k,X_j-Y_j).
\end{align*}

Given $\mu^1$, \ldots, $\mu^\ell\in\module{M}_k$, we define the
corresponding \emph{$\ell$-th order directional nc
difference-differential operators} \index{higher order directional
nc difference-differential operator}
\begin{multline*}
\Delta_{R,\mu^1,\ldots,\mu^\ell}^\ell=\Delta_{R,\mu^\ell}\cdots\Delta_{R,\mu^1}\colon\
\tclass{k}(\Omega^{(0)},\ldots,\Omega^{(k)};\ncspacej{\module{N}}{0},\ldots,\ncspacej{\module{N}}{k})\\
\longrightarrow
\tclass{k+\ell}(\Omega^{(0)},\ldots,\Omega^{(k)},\underset{\ell\
\text{times}}{\underbrace{\Omega^{(k)},\ldots,
\Omega^{(k)}}};\ncspacej{\module{N}}{0},\ldots,\ncspacej{\module{N}}{k},
\underset{\ell\ \text{times}}{\underbrace{\ncspace{\ring},\ldots,
\ncspace{\ring}}}).
\end{multline*}
\index{$\Delta_{R,\mu^1,\ldots,\mu^\ell}^\ell$}More explicitly, using \eqref{eq:delta_r_k} and induction on
$\ell$, one obtains
\begin{multline*}
\Delta_{R,\mu^1,\ldots,\mu^\ell}^\ell
f(X^0,\ldots,X^{k+\ell})(Z^1,\ldots,Z^k,A^1,\ldots,A^\ell)\\
=\Delta_{R}^\ell
f(X^0,\ldots,X^{k+\ell})(Z^1,\ldots,Z^k,A^1\mu^1,\ldots,A^\ell\mu^\ell).
\end{multline*}
In the tensor product interpretation of the values of nc functions
(see Remark \ref{rem:tensor_values}), we have
\begin{multline*}
\Delta_{R}^\ell
f(X^0,\ldots,X^{k+\ell})\\
\in\mat{\module{N}_0}{n_0}\otimes\mat{\module{N}_1^*}{n_1}\otimes\cdots\otimes
\mat{\module{N}_k^*}{n_k}\otimes\mat{\module{M}_k^*}{n_{k+1}}
\otimes\cdots\otimes\mat{\module{M}_k^*}{n_{k+\ell}},
\end{multline*}
\begin{multline*}
\Delta_{R,\mu^1,\ldots,\mu^\ell}^\ell
f(X^0,\ldots,X^{k+\ell})\\
\in\mat{\module{N}_0}{n_0}\otimes\mat{\module{N}_1^*}{n_1}\otimes\cdots\otimes
\mat{\module{N}_k^*}{n_k}\otimes\mat{\ring}{n_{k+1}}\otimes\cdots\otimes\mat{\ring}{n_{k+\ell}},
\end{multline*}
and \begin{multline*} \Delta_{R,\mu^1,\ldots,\mu^\ell}^\ell
f(X^0,\ldots,X^{k+\ell})\\
=(I_{n_0}\otimes\cdots\otimes I_{n_k}\otimes
I_{n_{k+1}}\mu^1\otimes\cdots\otimes I_{n_{k+\ell}}\mu^\ell)
 \Delta_{R}^\ell
f(X^0,\ldots,X^{k+\ell}).
\end{multline*}
Setting $Z^1=A^1\mu^1$, \ldots, $Z^\ell=A^\ell\mu^\ell$ in Theorem
\ref{thm:bidiag}, we can rewrite $\eqref{eq:bidiag}$ as
{\small \begin{multline}\label{eq:dir_bidiag}
f\left(\begin{bmatrix}
X^0 & A^1\mu^1 & 0 & \cdots & 0\\
0   & X^1 & \ddots & \ddots & \vdots\\
\vdots & \ddots & \ddots & \ddots & 0 \\
\vdots &  & \ddots & X^{\ell  -  1} & A^\ell\mu^\ell\\
 0 & \cdots & \cdots & 0 & X^\ell
\end{bmatrix}\right)\\
\!=\!\!\begin{bmatrix} f(X^0) & \Delta_{R,\mu^1}f(X^0,X^1)(A^1) &
\cdots & \cdots &
\Delta_{R,\mu^1,\ldots,\mu^\ell}^\ell f(X^0,\ldots,X^\ell)(A^1,\ldots,A^\ell)\\
0   & f(X^1) & \ddots & & \vdots\\
\vdots & \ddots & \ddots & \ddots & \vdots \\
\vdots &  & \ddots &  \ddots & \Delta_{R,\mu^\ell}f(X^{\ell -1},X^\ell)(A^\ell)\\
 0 & \cdots & \cdots &  0 & f(X^\ell )
\end{bmatrix}.
\end{multline}
}There is a similar version of Theorem \ref{thm:bidiag_k} in terms
of higher order directional nc difference-differential operators.

In the special case where $\module{M}=\ring^d$, we define the
\emph{$\ell$-th order partial nc difference-differential operator}
\index{higher order partial nc difference-differential operator}
corresponding to a word $w=g_{i_1}\cdots g_{i_\ell}\in\free_d$ by
$$\Delta_R^w=\Delta_{R,e_{i_1},\ldots,e_{i_\ell}}^\ell=\Delta_{R,i_\ell}\cdots\Delta_{R,i_1};$$
\index{$\Delta_R^w$}see Section \ref{subsec:ncpoly} for the
definition of the free monoid $\free_d$ and nc multipowers.
(Notice that we can interpret $\Delta_R^w$ as a nc multipower if
we abuse the notation by letting $\Delta_R$ denote the $d$-tuple
$(\Delta_{R,1},\ldots,\Delta_{R,d})$ of first order partial
difference-differential operators as well as the first order  full
difference-differential operator.)

 By multilinearity we
have, for $f\in\tclass{k}$,
\begin{multline}\label{eq:r_difdecomp-kl}
\Delta_R^\ell f(X^0,\ldots,X^{k+\ell})(Z^1,\ldots,Z^{k+\ell})\\
=\sum_{w=g_{i_1}\cdots
g_{i_\ell}}\Delta_R^{w^\top}f(X^0,\ldots,X^{k+\ell})
(Z^1,\ldots,Z^{k},Z^{k+1}_{i_1},\ldots,Z^{k+\ell}_{i_\ell} ),
\end{multline}
where the summation is over all the words of length $\ell$ (notice
the transposition of words: for $w=g_{i_1}\cdots g_{i_\ell}$ we
use the notation $w^\top=g_{i_\ell}\cdots g_{i_1}$).

Setting $Z^1=A^1e_{i_1}$, \ldots, $Z^\ell=A^\ell e_{i_\ell}$ in
Theorem \ref{thm:bidiag}, i.e., choosing $\mu^1=e_{i_1}$, \ldots,
$\mu^\ell=e_{i_\ell}$ in
 \eqref{eq:dir_bidiag}, we obtain
{\small \begin{multline}\label{eq:part_bidiag}
f\left(\begin{bmatrix}
X^0 & A^1e_{i_1} & 0 & \cdots & 0\\
0   & X^1 & \ddots & \ddots & \vdots\\
\vdots & \ddots & \ddots & \ddots & 0 \\
\vdots &  & \ddots & X^{\ell  -  1} & A^\ell e_{i_\ell}\\
 0 & \cdots & \cdots & 0 & X^\ell
\end{bmatrix}\right)\\
\!=\!\!\small{\begin{bmatrix} f(X^0) &
\Delta_R^{g_{i_1}}f(X^0,X^1)(A^1) & \cdots & \cdots &
\Delta_R^{g_{i_\ell}\cdots g_{i_1}} f(X^0,\ldots,X^\ell)(A^1,\ldots,A^\ell)\\
0   & f(X^1) & \ddots & & \vdots\\
\vdots & \ddots & \ddots & \ddots & \vdots \\
\vdots &  & \ddots &  \ddots & \Delta_R^{g_{i_\ell}}f(X^{\ell -1},X^\ell)(A^\ell)\\
 0 & \cdots & \cdots &  0 & f(X^\ell )
\end{bmatrix}}\!.
\end{multline}
}

\chapter{The Taylor--Taylor formula}\label{sec:TT}
 We use now the calculus of higher order nc difference-differential
 operators to derive a nc analogue of the Brook Taylor expansion, which we
call the TT (Taylor--Taylor) expansion in honour of Brook Taylor
and Joseph L. Taylor.
\begin{thm}\label{thm:TT}
Let $f\in\tclass{0}(\Omega;\ncspace{\module{N}})$ with
$\Omega\subseteq\ncspace{\module{M}}$ a right admissible nc set,
$n\in\mathbb{N}$, and  $Y\in\Omega_n$. Then for each
$N\in\mathbb{N}$ and arbitrary $X\in\Omega_n$,
\begin{multline}\label{eq:TT}
f(X)=\sum_{\ell=0}^N\Delta_R^\ell f(\underset{\ell+1\
\rm{times}}{\underbrace{Y,\ldots,Y}})(\underset{\ell\
\rm{times}}{\underbrace{X-Y,\ldots,X-Y}})\\
+\Delta_R^{N+1}f(\underset{N+1\
\rm{times}}{\underbrace{Y,\ldots,Y}},X)(\underset{N+1\
\rm{times}}{\underbrace{X-Y,\ldots,X-Y}}).
\end{multline}
\end{thm}
\begin{proof}
Using Theorems \ref{thm:Lagrange} and \ref{thm:Lagrange_k}, we
obtain
\begin{align*}
f(X) &=f(Y)+\Delta_Rf(Y,X)(X-Y)\\
&= f(Y)+\Delta_Rf(Y,Y)(X-Y)+\Delta_R^2f(Y,Y,X)(X-Y,X-Y)\\
 &=\ldots \\
 &= f(Y)+\Delta_Rf(Y,Y)(X-Y)+\cdots
+\Delta_R^Nf(\underset{N+1\
\text{times}}{\underbrace{Y,\ldots,Y}})(\underset{N\
\text{times}}{\underbrace{X-Y,\ldots,X-Y}})\\
& \hspace{1.2cm} +\Delta_R^{N+1}f(\underset{N+1\
\text{times}}{\underbrace{Y,\ldots,Y}},X)(\underset{N+1\
\text{times}}{\underbrace{X-Y,\ldots,X-Y}}).
\end{align*}
\end{proof}

It is possible to obtain a version of a power expansion centered
at a matrix $Y\in\Omega_s$ valid in matrix dimensions which are
multiples of $s$.
  For
$X\in\mat{\module{M}}{n}$ and $n=sm$ we denote
$$X^{\odot_s\ell}:=\underset{\ell\
\text{times}}{\underbrace{X\odot_s\cdots\,\odot_s X}},$$
\index{$X^{\odot_s\ell}$}where we use the notation introduced
prior to Proposition \ref{prop:diag_tensa_k}. Notice that
$X^{\odot_s\ell}$ is the $\ell$-th power of $X$ viewed as a
$m\times m$ matrix over $\tensa{\mat{\module{M}}{s}}$. Here
$$\tensa{\module{L}}:=\bigoplus_{\ell=0}^\infty\module{L}^{\otimes\ell}$$
\index{$\tensa{\module{L}}$}denotes the tensor algebra over a
module $\module{L}$ over $\ring$. We also write $X^{\odot \ell}$
\index{$X^{\odot \ell}$} instead of $X^{\odot_1\ell}$, thus
omitting the subscript in case $s=1$.

\begin{thm}\label{thm:tt-power-gen}
Let $f\in\tclass{0}(\Omega;\ncspace{\module{N}})$ with
$\Omega\subseteq\ncspace{\module{M}}$ a right admissible nc set,
 and $Y\in\Omega_s$. Then for each $N\in\mathbb{N}$ and arbitrary
 $m\in\mathbb{N}$ and $X\in\Omega_{ms}$,
\begin{multline}\label{eq:tt-power-gen}
f(X)=\sum_{\ell=0}^N\Big(X-
\bigoplus_{\alpha=1}^mY\Big)^{\odot_s\ell}\,\Delta_R^\ell
f(\underset{\ell+1\
\rm{times}}{\underbrace{Y,\ldots,Y}})\\
+\Big(X- \bigoplus_{\alpha=1}^mY\Big)^{\odot_s
N+1}\,\Delta_R^{N+1}f(\underset{N+1\
\rm{times}}{\underbrace{Y,\ldots,Y}},X).
\end{multline}
In particular, if  $\mu\in\Omega_1$ then for each $N\in\mathbb{N}$
and arbitrary $m\in\mathbb{N}$ and $X\in\Omega_m$,
\begin{multline}\label{eq:tt-power-gen'}
f(X)=\sum_{\ell=0}^N(X- I_m\mu)^{\odot\ell}\,\Delta_R^\ell
f(\underset{\ell+1\
\rm{times}}{\underbrace{\mu,\ldots,\mu}})\\
+(X- I_m\mu)^{\odot N+1}\,\Delta_R^{N+1}f(\underset{N+1\
\rm{times}}{\underbrace{\mu,\ldots,\mu}},X).
\end{multline}
\end{thm}
\begin{proof}
The first statement is immediate from Theorem \ref{thm:TT} and
Proposition \ref{prop:diag_tensa_k}. The second statement is a
special case of the first one when $s=1$.
\end{proof}
\begin{rem}\label{rem:lost_abbey}
It follows from \eqref{eq:gen-0-RightLagr_k},
\eqref{eq:gen-j-RightLagr_k}, \eqref{eq:gen-RightLagr_k}, and
\eqref{eq:equal_delta_words} that the multilinear mappings
$f_\ell:=\Delta_R^\ell f(Y,\ldots,Y)\colon
\left(\mat{\module{M}}{s}\right)^\ell\to\mat{\module{N}}{s}$,
satisfy
\begin{equation}\label{eq:ncfun_coef_0}
Sf_0-f_0S=f_1(SY-YS),
\end{equation}
and for $\ell=1,\ldots$,
\begin{equation}\label{eq:ncfun-coef-ell_0}
Sf_\ell(Z^1,\ldots,Z^\ell)-f_\ell(SZ^1,Z^2,\ldots,Z^\ell)=f_{\ell+1}(SY-YS,Z^1,\ldots,Z^\ell),
\end{equation}
\begin{multline}\label{eq:ncfun-coef_ell_j}
f_\ell(Z^1,\ldots,Z^{j-1},Z^jS,Z^{j+1},\ldots,Z^\ell)-f_\ell(Z^1,\ldots,Z^j,SZ^{j+1},Z^{j+2},\ldots,Z^\ell)\\
=f_{\ell+1}(Z^1,\ldots,Z^j,SY-YS,Z^{j+1},\ldots,Z^\ell),
\end{multline}
\begin{equation}\label{eq:ncfun-coef_ell_ell}
f_\ell(Z^1,\ldots,Z^{\ell-1},Z^\ell S)-f_\ell(Z^1,\ldots,Z^\ell)S
=f_{\ell+1}(Z^1,\ldots,Z^\ell,SY-YS),
\end{equation}
for every $S\in\mat{\ring}{s}$.
\end{rem}

We next specialize Theorem \ref{thm:tt-power-gen} to the case
where $\module{M}=\ring^d$, using higher order partial nc
difference-differential operators.  For
$X\in\mattuple{\ring}{n}{d}$, $n=sm$, and $w=g_{i_1}\cdots
g_{i_\ell}$ we denote
\begin{equation}\label{eq:faux-power-semigr}
X^{\odot_sw}:=X_{i_1}\odot_s\cdots\,\odot_s X_{i_\ell}.
\end{equation}
\index{$X^{\odot_sw}$}Notice
that $X^{\odot_sw}$ is the $w$-th nc power of $X$ viewed as a
$d$-tuple of $m\times m$ matrices over $\tensa{\mat{\ring}{s}}$.
We also write $X^{\odot w}$ \index{$X^{\odot w}$} instead of
$X^{\odot_1w}$, thus omitting the subscript in case $s=1$. Notice
that
\begin{equation}\label{eq:Z^l_vs_Z^w}
Z^{\odot_s\ell}=\sum_{w=g_1\cdots
g_{i_\ell}}(Z^{\odot_sw})e_{i_1}\otimes\cdots\otimes e_{i_\ell},
\end{equation}
for every $\ell, m\in\mathbb{N}$, and
$Z\in\mattuple{\ring}{sm}{d}$,
 where $e_i$,
$i=1,\ldots,d$, is the standard basis for $\ring^d$. One also has
\begin{equation}\label{eq:delta^l_vs_delta^w_m=1}
\Delta_R^\ell f(Y,\ldots,
Y)(Z^1,\ldots,Z^\ell)=\sum_{w=g_{i_1}\ldots
g_{i_\ell}}\Delta_R^{w^\top}f(Y,\ldots,Y)(Z^1_{i_1},\ldots,Z^\ell_{i_\ell})
\end{equation}
(see \eqref{eq:r_difdecomp-kl} with $k=0$) and, similarly,
\begin{equation}\label{eq:delta^l_vs_delta^w}
Z^{\odot_s\ell}\Delta_R^\ell f(Y,\ldots,Y)=\sum_{w=g_1\cdots
g_{i_\ell}}Z^{\odot_sw}\Delta_R^{w^\top}f(Y,\ldots,Y)
\end{equation}
for a nc function $f$ on a nc set
$\Omega\subseteq\ncspaced{\ring}{d}$ with values in
$\ncspace{\module{N}}$, and every $Y\in\Omega_{s}$, $Z^1$, \ldots,
$Z^\ell\in\mattuple{\ring}{s}{d}$, and
$Z\in\mattuple{\ring}{sm}{d}$. Theorem \ref{thm:tt-power-gen}
implies the following.
\begin{cor}\label{cor:part_TT}
Let $f\in\tclass{0}(\Omega;\ncspace{\module{N}})$ with
$\Omega\subseteq\ncspaced{\ring}{d}$ a right admissible nc set,
$m\in\mathbb{N}$, and $Y\in\Omega_s$. Then for each
$N\in\mathbb{N}$ and arbitrary $X\in\Omega_{sm}$,
\begin{multline}\label{eq:part_TT}
f(X)=\sum_{\ell=0}^N\sum_{|w|=\ell}\Big(X-\bigoplus_{\alpha=1}^mY\Big)^{\odot_sw}
\Delta_R^{w^\top}f(\underset{\ell+1\
\rm{times}}{\underbrace{Y,\ldots,Y}})\\
+\sum_{|w|=N+1}\Big(X-\bigoplus_{\alpha=1}^mY\Big)^{\odot_sw}\Delta_R^{w^\top}f(\underset{N+1\
\rm{times}}{\underbrace{Y,\ldots,Y}},X).
\end{multline}
In particular, for $s=1$ and $\mu\in\Omega_1$, we obtain a genuine
nc power expansion
\begin{multline}\label{eq:part_TT'}
f(X)=\sum_{\ell=0}^N\sum_{|w|=\ell}(X-I_m\mu)^{w}
\Delta_R^{w^\top}f(\underset{\ell+1\
\rm{times}}{\underbrace{\mu,\ldots,\mu}})\\
+\sum_{|w|=N+1}(X-I_m\mu)^{w}\Delta_R^{w^\top}f(\underset{N+1\
\rm{times}}{\underbrace{\mu,\ldots,\mu}},X).
\end{multline}
Here we identify the $\ell$-linear mapping
$\Delta_R^{w^\top}f(\mu,\ldots,\mu)\colon\ring\times\cdots\times\ring\to\module{N}$
with the vector
$\Delta_R^{w^\top}f(\mu,\ldots,\mu)(1,\ldots,1)\in\module{N}$, and
the $m\times m$ matrix $\Delta_R^{w^\top}f(\mu,\ldots,\mu,X)$ of
$(N+1)$-linear mappings
$\ring\times\cdots\times\ring\to\module{N}$ (see formula
\eqref{eq:diag_tensa_k'} in Proposition \ref{prop:diag_tensa_k})
with the $m\times m$ matrix over $\module{N}$ whose $i$-th row is
equal to $\Delta_R^{w^\top}f(\mu,\ldots,\mu,X)(1,\ldots,1,
e_i^\top)$ where $e_i$ is the $i$-th standard basis vector of
$\ring^m$.
\end{cor}
\begin{rem}\label{rem:lost_abbey_semigr}
Specializing Remark \ref{rem:lost_abbey} to the case of
$\module{M}=\ring^d$ and using \eqref{eq:delta^l_vs_delta^w_m=1},
we obtain that the multilinear forms
$f_w=\Delta_R^{w^\top}f(Y,\ldots,Y)$, $w\in\free_d$, satisfy the
conditions
\begin{equation}\label{eq:ncfun_coef_empty}
S f_\emptyset-f_\emptyset S=\sum_{k=1}^d f_{g_k}(SY_k-Y_kS),
\end{equation}
and for $w=g_{i_1}\cdots g_{i_\ell} \neq \emptyset$,
\begin{equation}\label{eq:ncfun-coef-w_0}
Sf_w(A^1,\ldots,A^\ell)-f_w(SA^1,A^2,\ldots,A^\ell)=\sum_{k=1}^df_{g_kw}(SY_k-Y_kS,A^1,\ldots,A^\ell),
\end{equation}
\begin{multline}\label{eq:ncfun-coef_w_j}
f_w(A^1,\ldots,A^{j-1},A^jS,A^{j+1},\ldots,A^\ell)-f_w(A^1,\ldots,A^j,SA^{j+1},A^{j+2},\ldots,A^\ell)\\
=\sum_{k=1}^df_{g_{i_1}\cdots g_{i_j} g_k g_{i_{j+1}} \cdots
g_{i_\ell}}(A^1,\ldots,A^j,SY_k-Y_kS,A^{j+1},\ldots,A^\ell),
\end{multline}
\begin{multline}\label{eq:ncfun-coef_w_ell}
f_w(A^1,\ldots,A^{\ell-1},A^\ell S)-f_w(A^1,\ldots,A^\ell)S\\
= \sum_{k=1}^df_{wg_k}(A^1,\ldots,A^\ell,SY_k-Y_kS),
\end{multline}
for all $S,A^1,\ldots,A^\ell \in \mat{\ring}{s}$.
\end{rem}

We proceed to rewrite the TT formula \eqref{eq:part_TT} in a more
concrete fashion. Using the tensor product interpretation for the
values of higher order nc functions (see Remark
\ref{rem:tensor_values} --- we assume that the module $\module{N}$
is free),
$$\Delta_R^{w^\top}f(\underset{\ell+1\
\rm{times}}{\underbrace{Y,\ldots,Y}})\in\mat{\module{N}}{s}\otimes\underset{\ell\
\rm{times}}{\underbrace{\mat{\ring}{s}\otimes\cdots\otimes\mat{\ring}{s}}}$$
and therefore
\begin{equation}\label{eq:tt-coef}
\Delta_R^{w^\top}f(\underset{\ell+1\
\rm{times}}{\underbrace{Y,\ldots,Y}})=A_{w,(0)}\otimes
A_{w,(1)}\otimes\cdots\otimes A_{w,(\ell)}.
\end{equation}
\index{$A_{w,(0)}\otimes
A_{w,(1)}\otimes\cdots\otimes A_{w,(\ell)}$}Here we use an
analogue of the \emph{sumless Sweedler notation} \index{sumless
Sweedler notation} familiar in the theory of coalgebras
\cite{HSw},
\begin{equation}\label{eq:Sweedler}
A_{w,(0)}\otimes A_{w,(1)}\otimes\cdots\otimes
A_{w,(\ell)}=\sum_{j}A^{(j)}_{w,(0)}\otimes
A_{w,(1)}^{(j)}\otimes\cdots\otimes A^{(j)}_{w,(\ell)},
\end{equation}
with a finite sum (having at most ${(s^2)}^\ell$ terms) and
$A_{w,(0)}^{(j)}\in\mat{\module{N}}{s}$, $A_{w,(1)}^{(j)}$,
\ldots, $A_{w,(\ell)}^{(j)}\in\mat{\ring}{s}$. According to
\eqref{eq:natur_map},
\begin{multline}\label{eq:tt-term}
\Delta_R^{w^\top}f(\underset{\ell+1\
\rm{times}}{\underbrace{Y,\ldots,Y}})( X_{i_1}-Y_{i_1},\ldots,X_{i_\ell}-Y_{i_\ell})\\
=(A_{w,(0)}\otimes A_{w,(1)}\otimes\cdots\otimes
A_{w,(\ell)})\star(X-Y)^{[w]},
\end{multline}
where we introduce the notation
\begin{multline}\label{eq:Ramamurti}
(C_{w,(0)}\otimes C_{w,(1)}\otimes\cdots\otimes
C_{w,(\ell)})\star Z^{[w]}\\
=\sum_{j}C^{(j)}_{w,(0)}\left(Z_{i_1}C_{w,(1)}^{(j)}\right)\cdots
\left(Z_{i_\ell}C^{(j)}_{w,(\ell)}\right),
\end{multline} \index{$(C_{w,(0)}\otimes C_{w,(1)}\otimes\cdots\otimes
C_{w,(\ell)})\star Z^{[w]}$}where $C_{w,(0)}\otimes
C_{w,(1)}\otimes\cdots\otimes
C_{w,(\ell)}\in\mat{\module{N}}{s}\otimes\underset{\ell\
\rm{times}}{\underbrace{\mat{\ring}{s}\otimes\cdots\otimes\mat{\ring}{s}}}$
in the  sumless Sweedler notation, and
$Z\in\mattuple{\ring}{n}{d}$. Recall that for
$Z\in\mattuple{\ring}{sm}{d}$, $w=g_{i_1}\cdots
g_{i_\ell}\in\free_d$, and $1\le i,k\le m$ we have
$$(Z^{\odot_sw})_{ik}=\sum_{1\le j_1,\ldots,j_{\ell-1}\le
m}(Z_{i_1})_{ii_1}\otimes (Z_{i_2})_{j_1j_2}\otimes
\cdots\otimes(Z_{i_{\ell-1}})_{j_{\ell-2}j_{\ell-1}}\otimes
(Z_{i_\ell})_{j_{\ell-1}k}.$$ Therefore, \begin{multline*}
\Big(Z^{\odot_sw}\Delta_R^{w^\top}f(\underset{\ell+1\
\rm{times}}{\underbrace{Y,\ldots,Y}})\Big)_{ik}\\
=\sum_{1\le j_1,\ldots,j_{\ell-1}\le
m}\Delta_R^{w^\top}f(\underset{\ell+1\
\rm{times}}{\underbrace{Y,\ldots,Y}})\Big((Z_{i_1})_{ii_1},(Z_{i_2})_{j_1j_2},
\ldots,(Z_{i_{\ell-1}})_{j_{\ell-2}j_{\ell-1}},
(Z_{i_\ell})_{j_{\ell-1}k}\Big)\\
=\sum_j\sum_{1\le j_1,\ldots,j_{\ell-1}\le
m}A^{(j)}_{w,(0)}\Big((Z_{i_1})_{ii_1}A^{(j)}_{w,(1)}\Big)\Big((Z_{i_2})_{j_1j_2}A^{(j)}_{w,(2)}\Big)\\
\cdots
\Big((Z_{i_{\ell-1}})_{j_{\ell-2}j_{\ell-1}}A^{(j)}_{w,(\ell-1)}\Big)
\Big((Z_{i_\ell})_{j_{\ell-1}k}A^{(j)}_{w,(\ell)}\Big)\\
=\Bigg(\left(\Big(\bigoplus_{\alpha=1}^mA_{w,(0)}\Big)\otimes\cdots\otimes\Big(\bigoplus_{\alpha=1}^mA_{w,(\ell)}\Big)
\right) \star Z^{[w]}\Bigg)_{ik}
\end{multline*}
for every $Z\in\mattuple{\ring}{sm}{d}$,
 so that
\begin{multline}\label{eq:tt-term-gen}
\Big(X-\bigoplus_{\alpha=1}^mY\Big)^{\odot_sw}\Delta_R^{w^\top}f(\underset{\ell+1\
\rm{times}}{\underbrace{Y,\ldots,Y}})\\
=\Bigg(\Big(\bigoplus_{\alpha=1}^mA_{w,(0)}\Big)
\otimes\cdots\otimes\Big(\bigoplus_{\alpha=1}^mA_{w,(\ell)}\Big)\Bigg)
\star \Big(X-\bigoplus_{\alpha=1}^mY\Big)^{[w]}
\end{multline}
for every $X\in\mattuple{\ring}{sm}{d}$.

Analogously, for the remainder of the TT formula
\eqref{eq:part_TT} (the sum over words of length $N+1$) we have
$$\Delta_R^{w^\top}f(\underset{N+1\
\rm{times}}{\underbrace{Y,\ldots,Y}},X)\in\mat{\module{N}}{n}\otimes\underset{N+1\
\rm{times}}{\underbrace{\mat{\ring}{n}\otimes\cdots\otimes\mat{\ring}{n}}}$$
and therefore
\begin{equation}\label{eq:tt-coef-remain}
\Delta_R^{w^\top}f(\underset{N+1\
\rm{times}}{\underbrace{Y,\ldots,Y}},X)=B_{w,(0)}\otimes
B_{w,(1)}\otimes\cdots\otimes B_{w,(N+1)}
\end{equation}
and
\begin{multline}\label{eq:tt-term-remain-gen}
\Big(X-\bigoplus_{\alpha=1}^mY\Big)^{\odot_sw}\Delta_R^{w^\top}f(\underset{N+1\
\rm{times}}{\underbrace{Y,\ldots,Y}},X)\\
=\Bigg(\Big(\bigoplus_{\alpha=1}^mB_{w,(0)}\Big)
\otimes\cdots\otimes\Big(\bigoplus_{\alpha=1}^mB_{w,(N)}\Big)\otimes
B_{w,(N+1)}\Bigg) \star \Big(X-\bigoplus_{\alpha=1}^mY\Big)^{[w]}
\end{multline}
for every $X\in\mattuple{\ring}{sm}{d}$. Notice that while the
``TT coefficients" $A_{w,(0)}\otimes A_{w,(1)}\otimes\cdots\otimes
A_{w,(\ell)}$ depend only on the center point $Y$, the ``remainder
coefficients" $B_{w,(0)}\otimes B_{w,(1)}\otimes\cdots\otimes
B_{w,(N+1)}$ depend on both $Y$ and $X$.

 Thus,  we obtain the
TT formula in the form of ``nc pseudo-power" expansion,  which is
a generalization of nc power expansion \eqref{eq:part_TT'} for the
case of $s=1$ to the case of general s.
\begin{thm}\label{thm:tt-power}
Let $f\in\tclass{0}(\Omega;\ncspace{\module{N}})$ with
$\Omega\subseteq\ncspaced{\ring}{d}$ a right admissible nc set,
 and $Y\in\Omega_s$. Then for each
$N\in\mathbb{N}$, an arbitrary $m\in\mathbb{N}$, and
$X\in\Omega_{sm}$,
\begin{multline}\label{eq:tt-pseudopower-mult}
f(X)=\sum_{\ell=0}^N\sum_{|w|=\ell}
\left(\Big(\bigoplus_{\alpha=1}^mA_{w,(0)}\Big)\otimes\cdots\otimes
\Big(\bigoplus_{\alpha=1}^mA_{w,(\ell)}\Big)\right)\star
\Big(X-\bigoplus_{\alpha=1}^mY\Big)^{[w]}\\
+\sum_{|w|=N+1}\left(\Big(\bigoplus_{\alpha=1}^mB_{w,(0)}\Big)\otimes\cdots\otimes
\Big(\bigoplus_{\alpha=1}^mB_{w,(N)}\Big)\otimes
B_{w,(N+1)}\right)\star \Big(X-\bigoplus_{\alpha=1}^mY\Big)^{[w]}.
\end{multline}
 Here
$A_{w,(0)}\otimes \cdots\otimes
A_{w,(\ell)}=\Delta_R^{w^\top}f(\underset{\ell+1\
\rm{times}}{\underbrace{Y,\ldots,Y}})$ and $B_{w,(0)}\otimes
\cdots\otimes B_{w,(N)}\otimes
B_{w,(N+1)}=\Delta_R^{w^\top}f(\underset{N+1\
\rm{times}}{\underbrace{Y,\ldots,Y}},X)$
 with the tensor product
interpretation for the values of $\Delta_R^{w^\trans}f$ (Remark
\ref{rem:tensor_values}), the sumless Sweedler notation
\eqref{eq:Sweedler}, and the pseudo-power notation
\eqref{eq:Ramamurti}.
\end{thm}

\chapter{NC functions on nilpotent matrices}\label{sec:nilp}
As the first application of the TT formula, we will now describe
nc functions on nilpotent matrices as sums of their TT series.
Conversely, we will see that the sum of nc power series evaluated
on nilpotent matrices defines a nc function, provided the
coefficients of the series satisfy the conditions listed in
Remarks \ref{rem:lost_abbey} and \ref{rem:lost_abbey_semigr}.

\section{From nc functions to nc power series}\label{subsec:
from_ncfun_to_ncps} The TT formula \eqref{eq:tt-power-gen'} allows
us to describe nc functions on nilpotent matrices over
$\module{M}$ as sums of nc
 power series. Let $n,\kappa\in\mathbb{N}$. Define the set
$\text{Nilp}(\module{M};n,\kappa)$
\index{$\text{Nilp}(\module{M};n,\kappa)$} of  $n\times n$
matrices $X$ over $\module{M}$ which are \emph{nilpotent of rank
at most} $\kappa$, \index{nilpotent matrix of rank at most
$\kappa$} i.e., $X^{\odot\ell}=0$ for all $\ell\geq \kappa$.
Define also the set
$\text{Nilp}(\module{M};n)=\bigcup_{\kappa=1}^\infty\mbox{Nilp}(\module{M};n,\kappa)$
\index{$\text{Nilp}(\module{M};n)$} of nilpotent $n\times n$
matrices over $\module{M}$ and the set
$\text{Nilp}(\module{M})=\coprod_{n=1}^\infty\text{Nilp}(\module{M};n)$
\index{$\text{Nilp}(\module{M})$} of nilpotent matrices over
$\module{M}$. Notice that
$\Omega=\text{Nilp}(\module{M})\subseteq\ncspace{\module{M}}$ is a
right (as well as left) admissible nc set, and
$\Omega_n=\text{Nilp}(\module{M};n)$, $n=1,2,\ldots$.

In the special case of $\module{M}=\ring^d$, the set
  $\text{Nilp}_d(\ring;n,\kappa):=\text{Nilp}(\ring^d;n,\kappa)$
  \index{$\text{Nilp}_d(\ring;n,\kappa)$}
  consists
  of
$d$-tuples $X$ of $n\times n$ matrices over $\ring$ which are
\emph{jointly nilpotent of rank at most} $\kappa$, \index{jointly
nilpotent matrices of rank at most $\kappa$} i.e., $X^w=0$ for all
$w\in\free_d$ such that $|w|\geq \kappa$; see
\eqref{eq:Z^l_vs_Z^w}. We also have the set
$\text{Nilp}_d(\ring;n):=\text{Nilp}(\ring^d;n)$ of jointly
\index{$\text{Nilp}_d(\ring;n)$} nilpotent $d$-tuples of $n\times
n$ matrices over $\ring$ and the set
$\text{Nilp}_d(\ring):=\text{Nilp}(\ring^d)$ of
\index{$\text{Nilp}_d(\ring)$} jointly nilpotent $d$-tuples of
matrices over $\ring$.
\begin{rem}\label{rem:nilp-uptr}
In the case where $\ring=\field$ is a field, a matrix
$X\in\text{Nilp}(\module{M};n)$ is similar to a strictly upper
triangular matrix. Indeed, consider a finite-dimensional subspace
of $\module{M}$ spanned by the entries of $X$ and choose its basis
$\upsilon_1$, \ldots, $\upsilon_d$; we can then write
$X=A_1\upsilon_1+\cdots+A_d\upsilon_d$, where
$(A_1,\ldots,A_d)\in\text{Nilp}_d(\field;n)$. The Lie subalgebra
of $\mat{\field}{n}$ generated by $A_1$, \ldots, $A_d$ consists of
nilpotent matrices. By Engel's theorem (see, e.g., \cite[Sections
3.2, 3.3]{Hump}), the matrices  $A_1$, \ldots, $A_d$ are jointly
similar to strictly upper triangular matrices. Hence, $X$ is
similar to a strictly upper triangular matrix.
\end{rem}

\begin{thm}\label{thm:nc_nilp-gen}
Let $f\in\tclass{0}({\rm Nilp}(\module{M});\ncspace{\module{N}})$.
Then for all  $X\in{\rm Nilp}(\module{M})$
\begin{equation}\label{eq:tt-nilp-gen}
f(X)=\sum_{\ell=0}^\infty X^{\odot\ell}\Delta_R^\ell
f(\underset{\ell+1\ \rm{times}}{\underbrace{0,\ldots,0}}),
\end{equation}
where the sum has finitely many nonzero terms.
\end{thm}
\begin{proof} Let $X\in\text{Nilp}(\module{M};n,\kappa)$ for some
positive integers $n$ and $\kappa$. Then the TT formula
\eqref{eq:tt-power-gen'} centered at $\mu=0$ for $N>\kappa$
becomes
$$f(X)=\sum_{\ell=0}^{\kappa-1}X^{\odot\ell}\Delta_R^\ell f(\underset{\ell+1\
\rm{times}}{\underbrace{0,\ldots,0}}).$$
\end{proof}
\begin{rem}\label{rem:ncfun_nilp}
It follows from \eqref{eq:tt-nilp-gen} that
$f(X)-I_nf(0)\in\text{Nilp}(\module{N};n)$ for
$X\in\text{Nilp}(\module{M};n)$. Alternatively (cf. Remark
\ref{rem:nilp-uptr}), if $X$ is strictly upper triangular, then
$f(X)$ is upper triangular with all diagonal entries equal to
$f(0)$.
\end{rem}

In the special case of $\module{M}=\ring^d$, we can use
\eqref{eq:part_TT'} instead of \eqref{eq:tt-power-gen'} to obtain:
\begin{thm}\label{thm:nc_nilp}
Let $f\in\tclass{0}({\rm Nilp}_d(\ring);\ncspace{\module{N}})$.
Then for all  $X\in{\rm Nilp}_d(\ring)$
\begin{equation}\label{eq:tt-nilp}
f(X)=\sum_{w\in\free_d}X^w\Delta_R^{w^\top}\!\!f(\underset{|w|+1\
\rm{times}}{\underbrace{0,\ldots,0}}),
\end{equation}
where the sum has finitely many nonzero terms.
\end{thm}

We proceed now to generalize Theorems \ref{thm:nc_nilp-gen} and
\ref{thm:nc_nilp} to the case of an arbitrary center. For
$Y\in\mat{\module{M}}{s}$, we define the set
$\text{Nilp}(\module{M},Y;sm,\kappa)$
\index{$\text{Nilp}(\module{M},Y;sm,\kappa)$} of $sm\times sm$
matrices $X$ over $\module{M}$ which are \emph{nilpotent about $Y$
of rank at most} $\kappa$, \index{nilpotent matrix about $Y$ of
rank at most $\kappa$} i.e.,
$(X-\bigoplus_{\alpha=1}^mY)^{\odot_s\ell}=0$ for all $\ell\geq
\kappa$. We also define
 the set
$\text{Nilp}(\module{M},Y;sm)=\bigcup_{\kappa=1}^\infty\text{Nilp}(\module{M},Y;sm,\kappa)$
and the set \index{$\text{Nilp}(\module{M},Y;sm)$}
$\text{Nilp}(\module{M},Y)=\coprod_{m=1}^\infty\text{Nilp}(\module{M},Y;sm)$.
\index{$\text{Nilp}(\module{M},Y)$} Notice that
$\text{Nilp}(\module{M},Y)$ is a right (as well as left)
admissible nc set.

In the special case of $\module{M}=\ring^d$, the set
  $\text{Nilp}_d(\ring,Y;sm,\kappa):=\text{Nilp}(\ring^d;sm,\kappa)$
  consists \index{$\text{Nilp}_d(\ring,Y;sm,\kappa)$}
  of
$d$-tuples $X$ of $sm\times sm$ matrices over $\ring$ which are
\emph{jointly nilpotent about $Y$ of rank at most} $\kappa$,
\index{jointly nilpotent matrices about $Y$ of rank at most
$\kappa$} i.e., $(X-\bigoplus_{\alpha=1}^mY)^{\odot_sw}=0$ for all
$w\in\free_d$ such that $|w|\geq \kappa$. We also have the sets
$\text{Nilp}_d(\ring,Y;sm):=\text{Nilp}(\ring^d,Y;sm)$ and
\index{$\text{Nilp}_d(\ring,Y;sm)$}
$\text{Nilp}_d(\ring,Y):=\text{Nilp}(\ring^d,Y)$.
\index{$\text{Nilp}_d(\ring,Y)$}
\begin{rem}\label{rem:nilp-uptr-Y}
Clearly, we have $X\in\text{Nilp}(\module{M},Y;sm,\kappa)$ if and
only if
$X-\bigoplus_{\alpha=1}^mY\in\text{Nilp}(\mat{\module{M}}{s};m,\kappa)$.
It follows from Remark \ref{rem:nilp-uptr} that when
$\ring=\field$ is a field,
 $X\in\text{Nilp}(\module{M},Y;sm)$ is similar, with a similarity matrix of the form $S\otimes
I_s\in\mat{\field}{sm}$, to an $s\times s$-block upper  triangular
matrix with all diagonal blocks equal to $Y$.
\end{rem}
We have the following analogue of Theorems \ref{thm:nc_nilp-gen},
with exactly the same proof.
\begin{thm}\label{thm:nc_nilp-gen-Y}
Let $Y\in\mat{\module{M}}{s}$ and $f\in\tclass{0}({\rm
Nilp}(\module{M},Y);\ncspace{\module{N}})$. Then for all $X\in{\rm
Nilp}(\module{M},Y;sm)$
\begin{equation}\label{eq:tt-nilp-gen-Y}
f(X)=\sum_{\ell=0}^\infty
\Big(X-\bigoplus_{\alpha=1}^mY\Big)^{\odot_s\ell}\Delta_R^\ell
f(\underset{\ell+1\ \rm{times}}{\underbrace{Y,\ldots,Y}}),
\end{equation}
where the sum has finitely many nonzero terms.
\end{thm}
\begin{rem}\label{rem:ncfun_nilp-gen-Y}
It follows from \eqref{eq:tt-nilp-gen-Y} that we have
$f(X)\in\text{Nilp}(\module{N},f(Y))$ for
$X\in\text{Nilp}(\module{M},Y)$. Alternatively (cf. Remark
\ref{rem:nilp-uptr-Y}), if $X$ is $s\times s$-block upper
triangular with all diagonal blocks equal to $Y$, then $f(X)$ is
$s\times s$-block upper triangular with all diagonal entries equal
to $f(Y)$.
\end{rem}

In the special case of $\module{M}=\ring^d$, we have the following
analogue of Theorem \ref{thm:nc_nilp}.
\begin{thm}\label{thm:nc_nilp-Y}
Let $Y\in\mattuple{\ring}{s}{d}$ and $f\in\tclass{0}({\rm
Nilp}_d(\ring,Y);\ncspace{\module{N}})$. Then for all  $X\in{\rm
Nilp}_d(\ring,Y;sm)$
\begin{multline}\label{eq:tt-nilp-Y}
f(X)=\sum_{w\in\free_d}\Big(X-\bigoplus_{\alpha=1}^mY\Big)^{\odot_sw}\Delta_R^{w^\top}\!\!f(\underset{|w|+1\
\rm{times}}{\underbrace{Y,\ldots,Y}})\\
\\
=\sum_{\ell=0}^N\sum_{|w|=\ell}
\left(\Big(\bigoplus_{\alpha=1}^mA_{w,(0)}\Big)\otimes\cdots\otimes
\Big(\bigoplus_{\alpha=1}^mA_{w,(\ell)}\Big)\right)\star
\Big(X-\bigoplus_{\alpha=1}^mY\Big)^{[w]},
\end{multline}
where $A_{w,(0)}\otimes \cdots\otimes
A_{w,(\ell)}=\Delta_R^{w^\top}f(\underset{\ell+1\
\rm{times}}{\underbrace{Y,\ldots,Y}})$
 with the tensor product
interpretation for the values of $\Delta_R^{w^\trans}f$ (Remark
\ref{rem:tensor_values}), the sumless Sweedler notation
\eqref{eq:Sweedler}, and the pseudo-power notation
\eqref{eq:Ramamurti}, and the sum has finitely many nonzero terms.
\end{thm}

We establish now the uniqueness of nc power expansions for nc
functions on nilpotent matrices.
\begin{thm}\label{thm:ncps-nilp-gen-unique}
Let $Y\in\mat{\module{M}}{s}$, $f\in\tclass{0}({\rm
Nilp}(\module{M},Y);\ncspace{\module{N}})$, and assume that
$$f(X)=\sum_{\ell=0}^\infty\Big(X-\bigoplus_{\alpha=1}^mY\Big)^{\odot_s\ell}f_\ell$$
for all $m\in\mathbb{N}$, $X\in{\rm Nilp}(\module{M},Y;sm)$, and
some $\ell$-linear mappings
$f_\ell\colon\mattuple{\module{M}}{s}{\ell}\to\mat{\module{N}}{s}$,
$\ell=0,1,\ldots$ (where the sum has finitely many nonzero terms).
Then $f_\ell=\Delta_R^\ell f(Y,\ldots,Y)$ for all $\ell$.
\end{thm}
\begin{proof}
By Theorem \ref{thm:nc_nilp-gen-Y}, it suffices to show that the
coefficients $f_\ell$ are uniquely determined by the nc function
$f$. For any fixed $\ell$ and any $Z^1$, \ldots,
$Z^\ell\in\mat{\module{M}}{s}$,
\begin{multline*}
f\left(\begin{bmatrix}
Y & Z^1 & 0 & \cdots & 0\\
\vdots   & \ddots & \ddots & \ddots & \vdots\\
\vdots &  & \ddots & \ddots & 0 \\
\vdots &  &  & \ddots & Z^\ell\\
 0 & \cdots & \cdots & \cdots & Y
\end{bmatrix} \right)
=\begin{bmatrix} f_0 & f_1(Z^1) & \cdots & \cdots  &
f_\ell(Z^1,\ldots, Z^\ell) \\
0   & f_0 & \ddots & & \vdots\\
\vdots & \ddots & \ddots & \ddots & \vdots \\
\vdots &  & \ddots &  \ddots & f_1(Z^\ell)\\
 0 & \cdots & \cdots &  0 & f_0
\end{bmatrix}.
\end{multline*}
The $(1,\ell+1)$ entry of the matrix on the right-hand side, i.e.,
the value of $f_\ell$ on arbitrary $\ell$ matrices in
$\mat{\module{M}}{s}$ is uniquely determined by $f$. (To argue
somewhat differently, we see directly that $f_\ell=\Delta_R^\ell
f(Y,\ldots,Y)$ because of \eqref{eq:bidiag}.)
\end{proof}
There is also a version of Theorem \ref{thm:ncps-nilp-gen-unique}
in the special case of $\module{M}=\ring^d$ for the coefficients
of nc power series along $\free_d$, as in \eqref{eq:tt-nilp-Y}.
\begin{thm}\label{thm:ncps-nilp-unique}
Let $Y\in\mat{\module{M}}{s}$, $f\in\tclass{0}({\rm
Nilp}_d(\ring,Y);\ncspace{\module{N}})$, and assume that
$$f(X)=\sum_{w\in\free_d}\Big(X-\bigoplus_{\alpha=1}^mY\Big)^{\odot_sw}f_w$$
for all $m\in\mathbb{N}$, $X\in{\rm Nilp}_d(\ring,Y;sm)$, and some
$\ell$-linear mappings
$f_w\colon\mattuple{\ring}{s}{|w|}\to\mat{\module{N}}{s}$,
$w\in\free_d$ (where the sum has finitely many nonzero terms).
Then $f_w=\Delta_R^{w^\top} f(Y,\ldots,Y)$ for all $w$.
\end{thm}
\begin{proof}
We define $\ell$-linear mappings
$f_\ell\colon\left(\mattuple{\ring}{s}{d}\right)^\ell\to\mat{\module{N}}{s}$
 by
\begin{equation}\label{eq:lw-forms}
f_{\ell}(Z^1,\ldots,Z^\ell)=\sum_{w=g_{i_1}\cdots
g_{i_\ell}\in\free_d}f_w(Z^1_{i_1},\ldots, Z^\ell_{i_\ell}),
\end{equation}
 so that for every $m\in\mathbb{N}$ and $Z\in\mattuple{\ring}{sm}{d}$,
\begin{equation}\label{eq:lw-powers}
Z^{\odot_s\ell}f_\ell =\sum_{|w|=\ell}Z^{\odot_sw}f_w.
\end{equation}
Clearly,
\begin{equation}\label{eq:f_w_via_f_l}
f_w(A^1,\ldots,A^\ell)=f_\ell(A^1e_{i_1},\ldots,A^\ell
e_{i_\ell})\quad {\rm for\ every}\ w=g_{i_1}\cdots g_{i_\ell}.
\end{equation}
 The result now follows from Theorem
\ref{thm:ncps-nilp-gen-unique} and the definition of
$\Delta_R^{w^\top}f$.
\end{proof}

\begin{rem} \label{rem:ncfun_1var}
A simple corollary of Theorem \ref{thm:nc_nilp-gen-Y} or \ref{thm:nc_nilp-Y} is a fairly explicit
description of nc functions from $\ncspace{\field}$ to $\ncspace{\vecspace{W}}$ for a field $\field$
and a vector space $\vecspace{W}$ over $\field$.
More precisely, let $\Lambda_n$ be the set of all $X \in \mat{\field}{n}$
such that all the roots of the characteristic polynomial of $X$ are in $\field$,
let $\Lambda=\coprod_{n=1}^\infty\Lambda_n$ (clearly, $\Lambda \subseteq \ncspace{\field}$ is a nc set),
and let $f \colon \Lambda \to \ncspace{\vecspace{W}}$ be a nc function.
For any $X \in \Lambda_n$, there exists an invertible matrix $T \in \mat{\field}{n}$
so that
\begin{equation} \label{eq:1var_decomp}
X  = T^{-1}\left(\bigoplus_\lambda X_\lambda\right)T,
\end{equation}
where the summation is over all the distinct eigenvalues $\lambda$ of $X$ with algebraic multiplicities $n_\lambda$
and $X_\lambda \in \text{Nilp}_1(\field,\lambda;n_\lambda)$
(if we want we can take $X_\lambda$ to be a direct sum of Jordan cells with eigenvalue $\lambda$).
It follows that
\begin{equation} \label{eq:1var_fun}
f(X) = T^{-1}\left(\bigoplus_\lambda \sum_{\ell=0}^{n_\lambda} (X_\lambda - I_{n_\lambda} \lambda)^\ell
\Delta_R^\ell f(\lambda,\ldots,\lambda)\right)T.
\end{equation}
Conversely, given any functions $f_\ell \colon \field \to \vecspace{W}$, $\ell=0,1,\ldots$, we can define
for any matrix $X \in \Lambda_n$ as above
\begin{equation} \label{eq:1var_ps}
f(X) = T^{-1}\left(\bigoplus_\lambda \sum_{\ell=0}^{n_\lambda} (X_\lambda - I_{n_\lambda} \lambda)^\ell
f_\ell(\lambda)\right)T.
\end{equation}
It is seen immediately that the right hand side is well defined
independently of the representation \eqref{eq:1var_decomp} of $X$,
and that if $XS=SY$ then $f(X)S=Sf(Y)$;
therefore \eqref{eq:1var_ps} defines a nc function $f \colon \Lambda \to \ncspace{\vecspace{W}}$,
and it follows from Theorem \ref{thm:ncps-nilp-gen-unique} or \ref{thm:ncps-nilp-unique}
that $\Delta_R^\ell f(\lambda,\ldots,\lambda)=f_\ell(\lambda)$.
\end{rem}

\section{From nc power series to nc
functions}\label{subsec:from_ncps_to_ncfun} In this chapter, we
consider the sum of power series of the form
\begin{equation}\label{eq:power-series-nilp}
\sum_{\ell=0}^\infty
\Big(X-\bigoplus_{\alpha=1}^mY\Big)^{\odot_s\ell}f_\ell
\end{equation}
on nilpotent matrices $X$ about $Y$ (so that the sum has finitely
many nonzero terms).
 Here $Y\in\mat{\module{M}}{s}$ is a given center; $X\in\text{Nilp}(\module{M},Y;sm)$,
$m=1,2,\ldots$;
$f_\ell\colon(\mat{\module{M}}{s})^{\ell}\to\mat{\module{N}}{s}$,
$\ell=0,1,\ldots$, is a given sequence of $\ell$-linear mappings,
i.e., a linear mapping
$f\colon\tensa{\mat{\module{M}}{s}}\to\mat{\module{N}}{s}$;
 and conditions
\eqref{eq:ncfun_coef_0}--\eqref{eq:ncfun-coef_ell_ell} are
satisfied. Notice that the latter conditions are trivially
satisfied for any sequence $f_\ell$ in the case of $s=1$.

 The necessity of conditions  \eqref{eq:ncfun_coef_0}--\eqref{eq:ncfun-coef_ell_ell} on the
coefficients $f_\ell$ for the sum of the series to be a nc
function on $\text{Nilp}(\module{M},Y)$ follows from Theorem
\ref{thm:nc_nilp-gen-Y} and Remark \ref{rem:lost_abbey}. It turns
out that these conditions are also sufficient; this will
essentially follow from the following lemma.

\begin{lem}\label{lem:intertw-poly}
Let $\module{M}$ and $\module{N}$ be modules over a ring $\ring$.
Let $Y\in\mat{\module{M}}{s}$, and let
$f_\ell\colon(\mat{\module{M}}{s})^{\ell}\to\mat{\module{N}}{s}$,
$\ell=0,1,\ldots$, be a sequence of $\ell$-linear mappings
satisfying \eqref{eq:ncfun_coef_0}--\eqref{eq:ncfun-coef_ell_ell}.
Let $X\in\mat{\module{M}}{sm}$,
$\widetilde{X}\in\mat{\module{M}}{s\widetilde{m}}$, and
$S\in\rmat{\ring}{s\widetilde{m}}{sm}$ satisfy
$SX=\widetilde{X}S$. Then
\begin{multline}\label{eq:intertw-poly}
S \Bigg(\sum_{\ell=0}^N \Big(X - \bigoplus_{\nu=1}^m
Y\Big)^{\odot_s \ell} f_\ell\Bigg)- \Bigg(\sum_{\ell=0}^N
\Big(\widetilde{X} - \bigoplus_{\nu=1}^{\widetilde{m}} Y\Big)^{\odot_s \ell} f_\ell\Bigg) S \\
\!=\! \sum_{k=0}^N \Bigg( \Big(\widetilde{X} -
\bigoplus_{\nu=1}^{\widetilde{m}} Y\Big)^{\odot_s k} \odot_s
\Big(S \bigoplus_{\nu=1}^{m} Y - \bigoplus_{\nu=1}^{\widetilde{m}}
Y S\Big) \odot_s \Big( X - \bigoplus_{\nu=1}^{m} Y\Big)^{\odot_s
(N-k)}\Bigg)\! f_{N+1}.
\end{multline}
\end{lem}
\begin{proof}
In what follows, we view $X$, $\widetilde{X}$, and $S$ as block
matrices of sizes $m\times m$, $\widetilde{m}\times\widetilde{m}$,
and $\widetilde{m}\times {m}$, respectively, with $s\times s$
blocks. Since $SX=\widetilde{X}S$, we have
$$\sum_{\beta=1}^mS_{\alpha\beta}X_{\beta\gamma}=
\sum_{\widetilde{\beta}=1}^{\widetilde{m}}\widetilde{X}_{\alpha\widetilde{\beta}}S_{\widetilde{\beta}\gamma},\quad
\alpha=1,\ldots,\widetilde{m},\ \gamma=1,\ldots,m.
$$
Denote
$$L=S\Bigg(\sum_{\ell=0}^N \Big(X - \bigoplus_{\nu=1}^m Y\Big)^{\odot_s \ell} f_\ell\Bigg),\qquad \widetilde{L}=
\Bigg(\sum_{\ell=0}^N \Big(\widetilde{X} -
\bigoplus_{\nu=1}^{\widetilde{m}} Y\Big)^{\odot_s \ell}
f_\ell\Bigg) S.$$
 Then
\begin{multline*}
L_{\alpha\gamma}
=\sum_{\beta_0=1}^mS_{\alpha\beta_0}\sum_{\ell=0}^N\Big[\Big
(X-\bigoplus_{\nu=1}^mY\Big)^{\odot_s\ell}f_\ell\Big]_{\beta_0\gamma}\\
 =S_{\alpha\gamma}f_0
+\sum_{\ell=1}^N\sum_{\beta_0,\beta_1,\ldots,\beta_{\ell-1}=1}^m
S_{\alpha\beta_0}f_\ell
(X_{\beta_0\beta_1}-\delta_{\beta_0\beta_1}Y,\ldots,X_{\beta_{\ell-2}\beta_{\ell-1}}-
\delta_{\beta_{\ell-2}\beta_{\ell-1}}Y,\\
X_{\beta_{\ell-1}\gamma}-\delta_{\beta_{\ell-1}\gamma}Y),
\end{multline*}
\begin{multline*}
\widetilde{L}_{\alpha\gamma}
=\sum_{\widetilde{\beta}_0=1}^{\widetilde{m}}\sum_{\ell=0}^N\Big[\Big
(\widetilde{X}-\bigoplus_{\nu=1}^{\widetilde{m}}Y\Big)^{\odot_s\ell}f_\ell\Big]_{\alpha\widetilde{\beta}_0}
S_{\widetilde{\beta}_0\gamma}\\
 =f_0S_{\alpha\gamma}
+\sum_{\ell=1}^N\sum_{\widetilde{\beta}_0,\widetilde{\beta}_1,\ldots,\widetilde{\beta}_{\ell-1}=1}^{\widetilde{m}}
f_\ell
(\widetilde{X}_{\alpha\widetilde{\beta}_0}-\delta_{\alpha\widetilde{\beta}_0}Y,
\widetilde{X}_{\widetilde{\beta}_0\widetilde{\beta}_1}-\delta_{\widetilde{\beta}_0\widetilde{\beta}_1}Y,\ldots,\\
\widetilde{X}_{\widetilde{\beta}_{\ell-2}\widetilde{\beta}_{\ell-1}}-
\delta_{\widetilde{\beta}_{\ell-2}\widetilde{\beta}_{\ell-1}}Y)S_{\widetilde{\beta}_{\ell-1}\gamma},
\end{multline*}
where $\delta_{ij}$ is the Kronecker symbol and
$f_0\in\mat{\vecspace{W}}{s}$. Applying \eqref{eq:ncfun_coef_0}
and \eqref{eq:ncfun-coef-ell_0}, we obtain that

\begin{multline*}
L_{\alpha\gamma}=
f_0S_{\alpha\gamma}+f_1(S_{\alpha\gamma}Y-YS_{\alpha\gamma})
+\sum_{\beta_0=1}^m f_1
(S_{\alpha\beta_0}X_{\beta_0\gamma}-S_{\alpha\beta_0}\delta_{\beta_0\gamma}Y)\\
+\sum_{\ell=2}^N\sum_{\beta_0,\beta_1,\ldots,\beta_{\ell-1}=1}^m
f_{\ell}
(S_{\alpha\beta_0}X_{\beta_0\beta_1}-S_{\alpha\beta_0}\delta_{\beta_0\beta_1}Y,
X_{\beta_1\beta_2}-\delta_{\beta_1\beta_2}Y,
 \ldots,\\
X_{\beta_{\ell-2}\beta_{\ell-1}}-
\delta_{\beta_{\ell-2}\beta_{\ell-1}}Y,
X_{\beta_{\ell-1}\gamma}-\delta_{\beta_{\ell-1}\gamma}Y)\\
+\sum_{\ell=1}^N\sum_{\beta_0,\beta_1,\ldots,\beta_{\ell-1}=1}^m
f_{\ell+1}
(S_{\alpha\beta_0}Y-YS_{\alpha\beta_0},X_{\beta_0\beta_1}-\delta_{\beta_0\beta_1}Y,
 \ldots,\\
X_{\beta_{\ell-2}\beta_{\ell-1}}-
\delta_{\beta_{\ell-2}\beta_{\ell-1}}Y,
X_{\beta_{\ell-1}\gamma}-\delta_{\beta_{\ell-1}\gamma}Y).
\end{multline*}
Using the assumption that $SX=\widetilde{X}S$ and the
multilinearity of the mappings $f_\ell$, we obtain that
\begin{multline*}
L_{\alpha\gamma}=
f_0S_{\alpha\gamma}+f_1(S_{\alpha\gamma}Y)-f_1(YS_{\alpha\gamma})+\sum_{\widetilde{\beta}_0=1}^{\widetilde{m}}
f_1(\widetilde{X}_{\alpha\widetilde{\beta}_0}S_{\widetilde{\beta}_0\gamma})-f_1(S_{\alpha\gamma}Y)\\
+\sum_{\ell=2}^N\sum_{\widetilde{\beta}_0=1}^{\widetilde{m}}\sum_{\beta_1,\ldots,\beta_{\ell-1}=1}^m
f_\ell
(\widetilde{X}_{\alpha\widetilde{\beta}_0}S_{\widetilde{\beta}_0\beta_1},
X_{\beta_1\beta_2}-\delta_{\beta_1\beta_2}Y,
\ldots, X_{\beta_{\ell-2}\beta_{\ell-1}}- \delta_{\beta_{\ell-2}\beta_{\ell-1}}Y,\\
X_{\beta_{\ell-1}\gamma}-\delta_{\beta_{\ell-1}\gamma}Y)\\
 - \sum_{\ell=2}^N\sum_{\beta_1,\ldots,\beta_{\ell-1}=1}^m f_\ell (S_{\alpha\beta_1}Y,
X_{\beta_1\beta_2}-\delta_{\beta_1\beta_2}Y, \ldots,
 X_{\beta_{\ell-2}\beta_{\ell-1}}-
\delta_{\beta_{\ell-2}\beta_{\ell-1}}Y,\\
X_{\beta_{\ell-1}\gamma}-\delta_{\beta_{\ell-1}\gamma}Y)\\
+\sum_{\ell=1}^N\sum_{\beta_0,\beta_1,\ldots,\beta_{\ell-1}=1}^m
f_{\ell+1}
(S_{\alpha\beta_0}Y,X_{\beta_0\beta_1}-\delta_{\beta_0\beta_1}Y,
 \ldots,
X_{\beta_{\ell-2}\beta_{\ell-1}}- \delta_{\beta_{\ell-2}\beta_{\ell-1}}Y,\\
X_{\beta_{\ell-1}\gamma}-\delta_{\beta_{\ell-1}\gamma}Y)\\
-\sum_{\ell=1}^N\sum_{\beta_0,\beta_1,\ldots,\beta_{\ell-1}=1}^m
f_{\ell+1}
(YS_{\alpha\beta_0},X_{\beta_0\beta_1}-\delta_{\beta_0\beta_1}Y,
 \ldots,
X_{\beta_{\ell-2}\beta_{\ell-1}}- \delta_{\beta_{\ell-2}\beta_{\ell-1}}Y,\\
X_{\beta_{\ell-1}\gamma}-\delta_{\beta_{\ell-1}\gamma}Y).
\end{multline*}
After obvious cancellations and using the multilinearity of
$f_\ell$ again, we obtain

\begin{multline*}
L_{\alpha\gamma}= f_0S_{\alpha\gamma}
+\sum_{\widetilde{\beta}_0=1}^{\widetilde{m}}f_1((\widetilde{X}_{\alpha\widetilde{\beta}_0}-
\delta_{\alpha\widetilde{\beta}_0}Y)S_{\widetilde{\beta}_0\gamma})
\\
+
\sum_{\ell=2}^N\sum_{\widetilde{\beta}_0=1}^{\widetilde{m}}\sum_{\beta_1,\ldots,\beta_{\ell-1}=1}^m
f_\ell
((\widetilde{X}_{\alpha\widetilde{\beta}_0}-\delta_{\alpha\widetilde{\beta}_0}Y)S_{\widetilde{\beta}_0\beta_1},
X_{\beta_1\beta_2}-\delta_{\beta_1\beta_2}Y, \ldots,\\
X_{\beta_{\ell-2}\beta_{\ell-1}}-
\delta_{\beta_{\ell-2}\beta_{\ell-1}}Y,
X_{\beta_{\ell-1}\gamma}-\delta_{\beta_{\ell-1}\gamma}Y)
\\
\end{multline*}
\begin{multline*}
+\sum_{\beta_0,\beta_1,\ldots,\beta_{N-1}=1}^m \!\! f_{N+1}
(S_{\alpha\beta_0}Y-YS_{\alpha\beta_0},X_{\beta_0\beta_1}-\delta_{\beta_0\beta_1}Y,
 \ldots,
X_{\beta_{N-2}\beta_{N-1}}- \delta_{\beta_{N-2}\beta_{N-1}}Y,\\
X_{\beta_{N-1}\gamma}-\delta_{\beta_{N-1}\gamma}Y).
\end{multline*}
Using \eqref{eq:ncfun-coef_ell_ell} and
\eqref{eq:ncfun-coef_ell_j}, we obtain
\begin{multline*}
L_{\alpha\gamma}= f_0S_{\alpha\gamma}\\
+\sum_{\widetilde{\beta}_0=1}^{\widetilde{m}}f_1(\widetilde{X}_{\alpha\widetilde{\beta}_0}-
\delta_{\alpha\widetilde{\beta}_0}Y)S_{\widetilde{\beta}_0\gamma}
+\sum_{\widetilde{\beta}_0=1}^{\widetilde{m}}f_2(\widetilde{X}_{\alpha\widetilde{\beta}_0}-
\delta_{\alpha\widetilde{\beta}_0}Y,S_{\widetilde{\beta}_0\gamma}Y-YS_{\widetilde{\beta}_0\gamma})\\
+\sum_{\widetilde{\beta}_0=1}^{\widetilde{m}}\sum_{\beta_1=1}^m
f_2(\widetilde{X}_{\alpha\widetilde{\beta}_0}-\delta_{\alpha\widetilde{\beta}_0}Y,S_{\widetilde{\beta}_0\beta_1}
(X_{\beta_1\gamma}-\delta_{\beta_1\gamma}Y))\\
+\sum_{\ell=3}^N\sum_{\widetilde{\beta}_0=1}^{\widetilde{m}}\sum_{\beta_1,\ldots,\beta_{\ell-1}=1}^m
\!\! f_\ell
(\widetilde{X}_{\alpha\widetilde{\beta}_0}-\delta_{\alpha\widetilde{\beta}_0}Y,S_{\widetilde{\beta}_0\beta_1}
(X_{\beta_1\beta_2}-\delta_{\beta_1\beta_2}Y),\\
X_{\beta_2\beta_3}-\delta_{\beta_2\beta_3}Y,
\ldots,X_{\beta_{\ell-2}\beta_{\ell-1}}-
\delta_{\beta_{\ell-2}\beta_{\ell-1}}Y,
X_{\beta_{\ell-1}\gamma}-\delta_{\beta_{\ell-1}\gamma}Y)\\
+\sum_{\ell=2}^N\sum_{\widetilde{\beta}_0=1}^{\widetilde{m}}\sum_{\beta_1,\ldots,\beta_{\ell-1}=1}^m
f_{\ell+1}(\widetilde{X}_{\alpha\widetilde{\beta}_0}-\delta_{\alpha\widetilde{\beta}_0}Y,
S_{\widetilde{\beta}_0\beta_1}Y-YS_{\widetilde{\beta}_0\beta_1},X_{\beta_1\beta_2}-\delta_{\beta_1\beta_2}Y,\ldots,\\
X_{\beta_{\ell-2}\beta_{\ell-1}}-
\delta_{\beta_{\ell-2}\beta_{\ell-1}}Y,
X_{\beta_{\ell-1}\gamma}-\delta_{\beta_{\ell-1}\gamma}Y)\\
+\sum_{\beta_0,\beta_1,\ldots,\beta_{N-1}=1}^m f_{N+1}
(S_{\alpha\beta_0}Y-YS_{\alpha\beta_0},X_{\beta_0\beta_1}-\delta_{\beta_0\beta_1}Y,
 \ldots,\\
X_{\beta_{N-2}\beta_{N-1}}- \delta_{\beta_{N-2}\beta_{N-1}}Y,
X_{\beta_{N-1}\gamma}-\delta_{\beta_{N-1}\gamma}Y).
\end{multline*}
Using $SX=\widetilde{X}S$ and the multilinearity of the mappings
$f_\ell$, we obtain that
\begin{multline*}
L_{\alpha\gamma}= f_0S_{\alpha\gamma}
+\sum_{\widetilde{\beta}_0=1}^{\widetilde{m}}f_1(\widetilde{X}_{\alpha\widetilde{\beta}_0}-
\delta_{\alpha\widetilde{\beta}_0}Y)S_{\widetilde{\beta}_0\gamma}\\
+\sum_{\widetilde{\beta}_0=1}^{\widetilde{m}}f_2(\widetilde{X}_{\alpha\widetilde{\beta}_0}-
\delta_{\alpha\widetilde{\beta}_0}Y,S_{\widetilde{\beta}_0\gamma}Y)-\sum_{\widetilde{\beta}_0=1}^{\widetilde{m}}
f_2(\widetilde{X}_{\alpha\widetilde{\beta}_0}-
\delta_{\alpha\widetilde{\beta}_0}Y,YS_{\widetilde{\beta}_0\gamma})\\
+\sum_{\widetilde{\beta}_0,\widetilde{\beta}_1=1}^{\widetilde{m}}
f_2(\widetilde{X}_{\alpha\widetilde{\beta}_0}-\delta_{\alpha\widetilde{\beta}_0}Y,
\widetilde{X}_{\widetilde{\beta}_0\widetilde{\beta}_1}
S_{\widetilde{\beta}_1\gamma})-
\sum_{\widetilde{\beta}_0=1}^{\widetilde{m}}f_2(
\widetilde{X}_{\alpha\widetilde{\beta}_0}-\delta_{\alpha\widetilde{\beta}_0}Y,S_{\widetilde{\beta}_0\gamma}Y)\\
+\sum_{\ell=3}^N\sum_{\widetilde{\beta}_0,\widetilde{\beta}_1=1}^{\widetilde{m}}
\sum_{\beta_2,\ldots,\beta_{\ell-1}=1}^m \!\! f_\ell
(\widetilde{X}_{\alpha\widetilde{\beta}_0}-\delta_{\alpha\widetilde{\beta}_0}Y,
\widetilde{X}_{\widetilde{\beta}_0\widetilde{\beta}_1}S_{\widetilde{\beta}_1\beta_2},
X_{\beta_2\beta_3}-\delta_{\beta_2\beta_3}Y,
\ldots,\\
X_{\beta_{\ell-2}\beta_{\ell-1}}-
\delta_{\beta_{\ell-2}\beta_{\ell-1}}Y,
X_{\beta_{\ell-1}\gamma}-\delta_{\beta_{\ell-1}\gamma}Y)\\
-\sum_{\ell=3}^N\sum_{\widetilde{\beta}_0=1}^{\widetilde{m}}\sum_{\beta_2,\ldots,\beta_{\ell-1}=1}^m
f_\ell
(\widetilde{X}_{\alpha\widetilde{\beta}_0}-\delta_{\alpha\widetilde{\beta}_0}Y,S_{\widetilde{\beta}_0\beta_2}
Y, X_{\beta_2\beta_3}-\delta_{\beta_2\beta_3}Y,
\ldots,\\
X_{\beta_{\ell-2}\beta_{\ell-1}}-
\delta_{\beta_{\ell-2}\beta_{\ell-1}}Y,
X_{\beta_{\ell-1}\gamma}-\delta_{\beta_{\ell-1}\gamma}Y)\\
\end{multline*}
\begin{multline*}
+\sum_{\ell=2}^N\sum_{\widetilde{\beta}_0=1}^{\widetilde{m}}\sum_{\beta_1,\ldots,\beta_{\ell-1}=1}^m
f_{\ell+1}(\widetilde{X}_{\alpha\widetilde{\beta}_0}-\delta_{\alpha\widetilde{\beta}_0}Y,
S_{\widetilde{\beta}_0\beta_1}Y,X_{\beta_1\beta_2}-\delta_{\beta_1\beta_2}Y,\ldots,\\
X_{\beta_{\ell-2}\beta_{\ell-1}}-
\delta_{\beta_{\ell-2}\beta_{\ell-1}}Y,
X_{\beta_{\ell-1}\gamma}-\delta_{\beta_{\ell-1}\gamma}Y)\\
-\sum_{\ell=2}^N\sum_{\widetilde{\beta}_0=1}^{\widetilde{m}}\sum_{\beta_1,\ldots,\beta_{\ell-1}=1}^m
f_{\ell+1}(\widetilde{X}_{\alpha\widetilde{\beta}_0}-\delta_{\alpha\widetilde{\beta}_0}Y,
YS_{\widetilde{\beta}_0\beta_1},X_{\beta_1\beta_2}-\delta_{\beta_1\beta_2}Y,\ldots,\\
X_{\beta_{\ell-2}\beta_{\ell-1}}-
\delta_{\beta_{\ell-2}\beta_{\ell-1}}Y,
X_{\beta_{\ell-1}\gamma}-\delta_{\beta_{\ell-1}\gamma}Y)\\
+\sum_{\beta_0,\beta_1,\ldots,\beta_{N-1}=1}^m f_{N+1}
(S_{\alpha\beta_0}Y-YS_{\alpha\beta_0},X_{\beta_0\beta_1}-\delta_{\beta_0\beta_1}Y,
 \ldots,\\
X_{\beta_{N-2}\beta_{N-1}}- \delta_{\beta_{N-2}\beta_{N-1}}Y,
X_{\beta_{N-1}\gamma}-\delta_{\beta_{N-1}\gamma}Y).
\end{multline*}
After obvious cancellations and using the multilinearity of
$f_\ell$ again, we obtain \begin{multline*} L_{\alpha\gamma}=
f_0S_{\alpha\gamma}
+\sum_{\widetilde{\beta}_0=1}^{\widetilde{m}}f_1(\widetilde{X}_{\alpha\widetilde{\beta}_0}-
\delta_{\alpha\widetilde{\beta}_0}Y)S_{\widetilde{\beta}_0\gamma}
\\
+\sum_{\widetilde{\beta}_0,\widetilde{\beta}_1=1}^{\widetilde{m}}
f_2(\widetilde{X}_{\alpha\widetilde{\beta}_0}-\delta_{\alpha\widetilde{\beta}_0}Y,
(\widetilde{X}_{\widetilde{\beta}_0\widetilde{\beta}_1}-\delta_{\widetilde{\beta}_0\widetilde{\beta}_1}Y)
S_{\widetilde{\beta}_1\gamma})\\
+\sum_{\ell=3}^N\sum_{\widetilde{\beta}_0,\widetilde{\beta}_1=1}^{\widetilde{m}}
\sum_{\beta_2,\ldots,\beta_{\ell-1}=1}^m f_\ell
(\widetilde{X}_{\alpha\widetilde{\beta}_0}-\delta_{\alpha\widetilde{\beta}_0}Y,
(\widetilde{X}_{\widetilde{\beta}_0\widetilde{\beta}_1}-\delta_{\widetilde{\beta}_0\widetilde{\beta}_1}Y)
S_{\widetilde{\beta}_1\beta_2},
X_{\beta_2\beta_3}-\delta_{\beta_2\beta_3}Y,
\\
\ldots,X_{\beta_{\ell-2}\beta_{\ell-1}}-
\delta_{\beta_{\ell-2}\beta_{\ell-1}}Y,
X_{\beta_{\ell-1}\gamma}-\delta_{\beta_{\ell-1}\gamma}Y) \\
+\sum_{\beta_0,\beta_1,\ldots,\beta_{N-1}=1}^m f_{N+1}
(S_{\alpha\beta_0}Y-YS_{\alpha\beta_0},X_{\beta_0\beta_1}-\delta_{\beta_0\beta_1}Y,
 \ldots,
X_{\beta_{N-2}\beta_{N-1}}- \delta_{\beta_{N-2}\beta_{N-1}}Y,\\
X_{\beta_{N-1}\gamma}-\delta_{\beta_{N-1}\gamma}Y)\\
+\sum_{\widetilde{\beta}_0=1}^{\widetilde{m}}\sum_{\beta_1,\ldots,\beta_{N-1}=1}^m
f_{N+1}(\widetilde{X}_{\alpha\widetilde{\beta}_0}-\delta_{\alpha\widetilde{\beta}_0}Y,
S_{\widetilde{\beta}_0\beta_1}Y-YS_{\widetilde{\beta}_0\beta_1},X_{\beta_1\beta_2}-\delta_{\beta_1\beta_2}Y,\ldots,\\
X_{\beta_{N-2}\beta_{N-1}}- \delta_{\beta_{N-2}\beta_{N-1}}Y,
X_{\beta_{N-1}\gamma}-\delta_{\beta_{N-1}\gamma}Y).
\end{multline*}
Continuing this argument, we obtain that
\begin{multline*}
L_{\alpha\gamma}=f_0S_{\alpha\gamma}
+\sum_{\ell=1}^N\sum_{\widetilde{\beta}_0,\widetilde{\beta}_1,\ldots,\widetilde{\beta}_{\ell-1}=1}^{\widetilde{m}}
f_\ell
(\widetilde{X}_{\alpha\widetilde{\beta}_0}-\delta_{\alpha\widetilde{\beta}_0}Y,
\widetilde{X}_{\widetilde{\beta}_0\widetilde{\beta}_1}-\delta_{\widetilde{\beta}_0\widetilde{\beta}_1}Y,\ldots,\\
\widetilde{X}_{\widetilde{\beta}_{\ell-2}\widetilde{\beta}_{\ell-1}}-
\delta_{\widetilde{\beta}_{\ell-2}\widetilde{\beta}_{\ell-1}}Y)S_{\widetilde{\beta}_{\ell-1}\gamma}\\
+\sum_{k=0}^N\sum_{\widetilde{\beta}_0,\ldots,\widetilde{\beta}_{k-1}=1}^{\widetilde{m}}
\sum_{\beta_k,\ldots,\beta_{N-1}=1}^m
f_{N+1}(\widetilde{X}_{\alpha\widetilde{\beta}_0}-\delta_{\alpha\widetilde{\beta}_0}Y,
\widetilde{X}_{\widetilde{\beta}_0\widetilde{\beta}_1}-
\delta_{\widetilde{\beta}_0\widetilde{\beta}_1}Y,\ldots,\\
\widetilde{X}_{\widetilde{\beta}_{k-2}\widetilde{\beta}_{k-1}}-
\delta_{\widetilde{\beta}_{k-2}\widetilde{\beta}_{k-1}}Y,
S_{\widetilde{\beta}_{k-1}\beta_k}Y-YS_{\widetilde{\beta}_{k-1}\beta_k},X_{\beta_k\beta_{k+1}}-
\delta_{\beta_k\beta_{k+1}}Y,\ldots,\\
X_{\beta_{N-2}\beta_{N-1}}- \delta_{\beta_{N-2}\beta_{N-1}}Y,
X_{\beta_{N-1}\gamma}-\delta_{\beta_{N-1}\gamma}Y)\\
\end{multline*}
\begin{multline*}
=\Big[\widetilde{L}+\sum_{k=0}^N \Big(\widetilde{X} -
\bigoplus_{\nu=1}^{\widetilde{m}} Y\Big)^{\odot_s k}\! \odot_s
\!\Big(S \bigoplus_{\nu=1}^{m} Y -
\bigoplus_{\nu=1}^{\widetilde{m}} Y S\Big) \odot_s \Big( X -
\bigoplus_{\nu=1}^{m} Y\Big)^{\odot_s N-k} \!\!
f_{N+1}\Big]_{\alpha\gamma}.
\end{multline*}
\end{proof}
\begin{thm}\label{thm:sufficiency_of_LAC}
The sum $f$ of the series \eqref{eq:power-series-nilp} is a nc
function on ${\rm Nilp}(\module{M},Y)$ with values in
$\ncspace{\module{N}}$.
\end{thm}
\begin{rem}\label{rem:identify_f}
By Theorem \ref{thm:ncps-nilp-gen-unique}, the coefficients
$f_\ell$ (i.e., the linear mapping
$f\colon\tensa{\mat{\module{M}}{s}}\to\mat{\module{N}}{s}$) are
uniquely determined by the sum of the corresponding series, which
is why we use the same notation for both the nc function and the
linear mapping.
\end{rem}
\begin{proof}[Proof of Theorem \ref{thm:sufficiency_of_LAC}]
Let $$X\in\text{Nilp}(\module{M},Y;sm,\kappa),\quad
\widetilde{X}\in\text{Nilp}(\module{M},Y;s\widetilde{m},\widetilde{\kappa}),$$
and $S\in\rmat{\ring}{s\widetilde{m}}{sm}$. Let
$N\ge\kappa+\widetilde{\kappa}$. Then
$$f(X)=\sum_{\ell=0}^N\Big(X-\bigoplus_{\alpha=1}^mY\Big)^{\odot_s\ell}f_\ell,\quad
f(\widetilde{X})=\sum_{\ell=0}^N\Big(\widetilde{X}-\bigoplus_{\alpha=1}^{\widetilde{m}}Y\Big)^{\odot_s\ell}f_\ell.$$
By Lemma \ref{lem:intertw-poly}, $Sf(X)-f(X)S=0$, since for every
$k\colon 0\le k\le N$ either $k\ge\widetilde{\kappa}$ or
$N-k\ge\kappa$, hence the right-hand side of
\eqref{eq:intertw-poly} vanishes. Thus, $f$ respects
intertwinings, i.e., is a nc function.
\end{proof}

In the special case $\module{M}=\ring^d$, we can consider instead
of \eqref{eq:power-series-nilp} a nc power series along $\free_d$,
i.e., a series of the form
\begin{equation}\label{eq:power-series-nilp-semigr}
\sum_{w\in\free_d}
\Big(X-\bigoplus_{\alpha=1}^mY\Big)^{\odot_sw}f_w
\end{equation}
on jointly nilpotent $d$-tuples of matrices $X$ about $Y$ (so that
the sum has finitely many nonzero terms).
 Here $Y\in\mattuple{\ring}{s}{d}$ is a given center; $X\in\text{Nilp}_d(\ring,Y;sm)$,
$m=1,2,\ldots$;
$f_w\colon(\mat{\ring}{s})^{|w|}\to\mat{\module{N}}{s}$,
$w\in\free_d$, is a given sequence of $|w|$-linear mappings;
 and conditions
\eqref{eq:ncfun_coef_empty}--\eqref{eq:ncfun-coef_w_ell} are
satisfied.   The series \eqref{eq:power-series-nilp-semigr} can be
alternatively written as
\begin{equation*}
\sum_{w\in\free_d}\left(\Big(\bigoplus_{\alpha=1}^mA_{w,(0)}\Big)\otimes
\Big(\bigoplus_{\alpha=1}^mA_{w,(1)}\Big)
\otimes\cdots\otimes\Big(\bigoplus_{\alpha=1}^mA_{w,(|w|)}\Big)\right)\star
\Big(X-\bigoplus_{\alpha=1}^mY\Big)^{[w]}.
\end{equation*}
Here
$$f_w=A_{w,(0)}\otimes
A_{w,(1)}\otimes\cdots\otimes
A_{w,(|w|)}\in\mat{\module{N}}{s}\otimes\underset{|w|\ {\rm
times}}{\underbrace{\mat{\ring}{s}\otimes
\cdots\otimes\mat{\ring}{s}}},$$ with the tensor product
interpretation for the multilinear mappings $f_w$ (Remark
\ref{rem:natur_map}), the sumless Sweedler notation
\eqref{eq:Sweedler}, and the pseudo-power notation
\eqref{eq:Ramamurti}. Notice that conditions
\eqref{eq:ncfun_coef_empty}--\eqref{eq:ncfun-coef_w_ell} are
trivially satisfied for any sequence $f_w$ in the case of $s=1$.
In this case, \eqref{eq:power-series-nilp-semigr} is a genuine nc
power series
$$\sum_{w\in\free}(X-I_m\mu)^wf_w,$$
where $Y=\mu\in\ring^d$ and $f_w\in\module{N}$.

We define a sequence of $\ell$-linear mappings
$f_\ell\colon\left(\mattuple{\ring}{s}{d}\right)^\ell\to\mat{\module{N}}{s}$,
$\ell=0,1,\ldots$, by \eqref{eq:lw-forms}, so that
\eqref{eq:lw-powers} and \eqref{eq:f_w_via_f_l} hold. Hence, the
sum of a series of the form \eqref{eq:power-series-nilp-semigr}
coincides with the sum of the corresponding series of the form
\eqref{eq:power-series-nilp}. Similarly to Remark
\ref{rem:lost_abbey_semigr}, conditions
\eqref{eq:ncfun_coef_empty}--\eqref{eq:ncfun-coef_w_ell} on the
multilinear mappings $f_w$ are equivalent to conditions
\eqref{eq:ncfun_coef_0}--\eqref{eq:ncfun-coef_ell_ell} on the
multilinear mappings $f_\ell$. Therefore,  Theorem
\ref{thm:sufficiency_of_LAC} implies that:
\begin{thm}\label{thm:sufficiency_of_LAC-semigr}
The sum $f$ of the series \eqref{eq:power-series-nilp-semigr} is a
nc function on ${\rm Nilp}_d(\ring,Y)$ with values in
$\ncspace{\module{N}}$.
\end{thm}

\chapter{NC polynomials vs. polynomials in matrix
entries}\label{sec:alg}

Let \begin{equation}\label{eq:nc-poly-d} p=\sum_{w\in\free_d\colon
|w|\le L}p_wx^w\in\module{N}\langle x_1,\ldots,x_d\rangle
\end{equation}
\index{$\module{N}\langle x_1,\ldots,x_d\rangle$}be a nc
polynomial of degree $L$ in $x_1$, \ldots, $x_d$ with coefficients
in a module $\module{N}$ over a ring $\ring$. Evaluating $p$ on
$d$-tuples $X=(X_1,\ldots,X_d)\in(\mat{\ring}{n})^d$,
\begin{equation}\label{eq:ncfun-poly-d}
p(X)=\sum_{w\in\free_d\colon |w|\le
L}X^wp_w\in\mat{\module{N}}{n},\quad n=1,2,\ldots,
\end{equation}
 defines a nc
function on $\ncspaced{\ring}{d}$ with values in
$\ncspace{\module{N}}$. Notice that the nc polynomial $p$ is
determined by this nc function uniquely: at least when the module
$\module{N}$ is free, this follows from the absence of identities
for $n\times n$ matrices over $\ring$ (i.e., nc polynomials
vanishing on $n\times n$ matrices  over $\ring$) for all
$n=1,2,\ldots$ --- see \cite[Example 1.4.4 and Theorem 1.4.5, page
22]{Row80}, and also follows from Theorem
\ref{thm:ncps-nilp-unique} by restricting a nc polynomial to ${\rm
Nilp}_d(\ring)$.  We will often identify the two nc objects,
\eqref{eq:nc-poly-d} and \eqref{eq:ncfun-poly-d}, and use the same
notation for both. We observe that for every $n$, the nc function
$p$ is a polynomial in $dn^2$ commuting variables $(X_i)_{jk}$,
$i=1,\ldots,d$; $j,k=1,\ldots,n$, of degree at most $L$ with
values in $\mat{\module{N}}{n}$.
\begin{thm}\label{thm:king-poly}
Let $f$ be a nc function on $\ncspaced{\mathbb{K}}{d}$, where
$\mathbb{K}$ is an infinite field, with values in
$\ncspace{\module{N}}$ (so that $\module{N}$ is a vector space
over $\mathbb{K}$). Assume that for each $n$, $f(X_1,\ldots,X_d)$
is a polynomial function in $dn^2$ commuting variables
$(X_i)_{jk}$, $i=1,\ldots,d$; $j,k=1,\ldots,n$, with values in
$\mat{\module{N}}{n}$. Assume also that the degrees of these
polynomial functions of the commuting variables $(X_i)_{jk}$ are
bounded. Then $f$ is a nc polynomial with coefficients in
$\module{N}$.
\end{thm}
\begin{proof}
First, we observe that for every $n\in\mathbb{N}$ and
$\ell\in\mathbb{N}$, the function
$$\Delta_R^\ell f(X^0,\ldots,X^\ell)(Z^1,\ldots,Z^\ell)$$
is a polynomial in the commuting variables
$$(X^\alpha_i)_{jk}, (Z^\beta_i)_{jk}\colon
\alpha=0,\ldots,\ell,\ \beta=1,\ldots,\ell;\ i=1,\ldots,d;\
j,k=1,\ldots,n,$$ with the coefficients in $\mat{\module{N}}{n}$.
This follows immediately from \eqref{eq:bidiag}.

Second, the TT formula \eqref{eq:TT} implies that
\begin{multline}\label{eq:TT'}
f(Y+tZ)=\sum_{\ell=0}^Nt^\ell\Delta_R^\ell
f(Y,\ldots,Y)(Z,\ldots,Z)\\
+t^{N+1}\Delta_R^{N+1}f(Y,\ldots,Y,Y+tZ)(Z,\ldots,Z)
\end{multline}
 for
$n\in\mathbb{N}$, $Y,Z\in (\mat{\mathbb{K}}{n})^d$,
$t\in\mathbb{K}$, and for $N=0,1,\ldots$. For fixed $n$, $Y$, and
$Z$, both $f_{Y,Z}(t):=f(Y+tZ)$ and
$\Delta_R^{N+1}(Y,\ldots,Y,Y+tZ)(Z,\ldots,Z)$ are polynomials in
$t$ with coefficients in $\mat{\module{N}}{n}$.

Let $L$ be the maximal degree of $f(X_1,\ldots,X_d)$ as a
polynomial function in $dn^2$ commuting variables $(X_i)_{jk}$,
$i=1,\ldots,d$; $j,k=1,\ldots,n$, as $n$ varies from $1$ to
$\infty$. Then for any fixed $n$, $Y$, and $Z$, the degree of the
polynomial $f_{Y,Z}(t)$ is at most $L$. Therefore, for $N=L$ the
remainder term in \eqref{eq:TT'} vanishes identically as a
function of $t$, and
$$f_{Y,Z}(t)=\sum_{\ell=0}^Lt^\ell\Delta_R^\ell
f(Y,\ldots,Y)(Z,\ldots,Z).$$ Setting $t=1$ and $X=Y+Z$ (so that
$f(X)=f_{Y,Z}(1)$), we obtain
\begin{equation}\label{eq:finite_TT}
f(X)=\sum_{\ell=0}^L\Delta_R^\ell f(\underset{\ell+1\
\rm{times}}{\underbrace{Y,\ldots,Y}})(\underset{\ell\
\rm{times}}{\underbrace{X-Y,\ldots,X-Y}}).
\end{equation}
Notice that this is a (uniformly in $n$) finite version of the TT
formula \eqref{eq:TT}. In particular, setting $Y=0$ and using
\eqref{eq:delta^l_vs_delta^w_m=1}, we obtain
\begin{equation}\label{eq:finite_TT-power}
f(X)=\sum_{|w|\le
L}X^w\,\Delta_R^{w^\top}\!\!f(\!\!\underset{|w|+1\
\rm{times}}{\underbrace{0,\ldots,0}}\!\!)
\end{equation}
(cf. \eqref{eq:tt-nilp}).
 Therefore, $f$ is a nc polynomial with coefficients in
$\module{N}$.
\end{proof}
\begin{rem}\label{rem:dir-der}
Notice that if $\field$ is a field of characteristic zero, then by
\eqref{eq:TT'} and the classical single-variable Taylor formula,
for all $\ell\colon  0\le\ell\le ML$ one has
\begin{equation}\label{eq:dir-der}
\Delta_R^\ell f(Y,\ldots,Y)(Z,\ldots,Z)=\frac{1}{\ell
!}\left.\frac{d^\ell f_{Y,Z}}{dt^\ell}\right|_{t=0}.
\end{equation}
\end{rem}

The following example shows that the assumption of boundedness of
degrees in Theorem \ref{thm:king-poly} cannot be omitted.
\begin{ex}\label{ex:degrees_unbdd}
Suppose that we have a sequence of homogeneous polynomials
$p_n\in\mathbb{K}\langle x_1,x_2\rangle$ of strictly increasing
degrees $\alpha_n$ such that $p_n$ vanishes on
$\left(\mat{\mathbb{K}}{n}\right)^2$. For example, we may take
$$p_n=\sum_{\pi\in
S_{n+1}}\operatorname{sign}(\pi)\,x_1^{\pi(1)-1}x_2\cdots
x_1^{\pi(n+1)-1}x_2$$ as in \cite{Row80} or \cite{Form}, where
$S_{n+1}$ \index{$S_{n}$} is the symmetric group on $n+1$
elements, and $\operatorname{sign}(\pi)$ is the sign of the
permutation $\pi$; clearly, $\alpha_n=\frac{(n+1)(n+2)}{2}$ in
this case. Define a nc function $f\colon
\ncspaced{\mathbb{K}}{2}\to\ncspace{\mathbb{K}}$ by
$$f(X_1,X_2)=\sum_{k=1}^{\infty}p_k(X_1,X_2),$$
where the sum is well defined: for $n\in\mathbb{N}$ and
$(X_1,X_2)\in\left(\mat{\mathbb{K}}{n}\right)^2$ the terms
$p_k(X_1,X_2)$, $k\ge n$, all vanish. Clearly, for all $n$,
$f(X_1,X_2)$ is a polynomial function in $2n^2$ commuting
variables $(X_1)_{jk}$, $(X_2)_{jk}$, $j,k=1,\ldots,n$, with
coefficients in $\mat{\mathbb{K}}{n}$. However, $f$ is not a nc
polynomial. Indeed, suppose that $f=q$, where $q\in
\mathbb{K}\langle x_1,x_2\rangle$. Let $\deg q=L$, and let
$$q=\sum_{j=0}^Lq_j,$$
with $q_j\in\mathbb{K}\langle x_1,x_2\rangle$ homogeneous of
degree $j$. Define a sequence of homogeneous polynomials
$$h_j=\begin{cases} p_k,& j=\alpha_k \text{ for some } k\in\mathbb{N},\\
0,& \text{otherwise,} \end{cases}$$ of degree $j$. Then we can
write
$$f(X_1,X_2)=\sum_{j=0}^\infty h_j(X_1,X_2)=\sum_{j=0}^{\alpha_{n-1}}
h_j(X_1,X_2)$$ for
$(X_1,X_2)\in\left(\mat{\mathbb{K}}{n}\right)^2$. Comparing
$$f(tX_1,tX_2)=\sum_{j=0}^{\infty}h_j(X_1,X_2)\,t^j$$ with
$$q(tX_1,tX_2)=\sum_{j=0}^Lq_j(X_1,X_2)\,t^j$$
 as polynomials in
$t\in\mathbb{K}$ for fixed $(X_1,X_2)$, we conclude that
$h_j(X_1,X_2)=0$ for all $j>L$. By the absence of polynomial
identities for matrices of all sizes, $h_j=0$ for all $j>L$. Since
the sequence $\{\alpha_n\}$ is strictly increasing, this is a
contradiction.
\end{ex}

We will show now that Example \ref{ex:degrees_unbdd} reflects the
general situation: any nc function on $\ncspaced{\mathbb{K}}{d}$
which, for each matrix size, is a polynomial in matrix entries,
can be written as a sum of a nc polynomial and an infinite series
of nc polynomials vanishing on matrices of increasing sizes.
\begin{thm}\label{thm:queen-poly}
Let $f$ be a nc function on $\ncspaced{\mathbb{K}}{d}$, where
$\mathbb{K}$ is an infinite field, with values in
$\ncspace{\module{N}}$ (so that $\module{N}$ is a vector space
over $\mathbb{K}$). Assume that for each $n$, $f(X_1,\ldots,X_d)$
is a polynomial function of degree $L_n$ in $dn^2$ commuting
variables $(X_i)_{jk}$, $i=1,\ldots,d$; $j,k=1,\ldots,n$, with
values in $\mat{\module{N}}{n}$. Then there exists a unique
sequence of homogeneous nc polynomials $f_j\in\module{N}\langle
x_1,\ldots,x_d\rangle$ of degree $j$, $j=0$, $1$, \ldots, such
that for all $n\in\mathbb{N}$,
\begin{itemize}
    \item $f_j$ vanishes on $\left(\mat{\mathbb{K}}{n}\right)^d$ for all
    $j>L_n$,
    \item if $L_n> -\infty$ (where we use the usual convention $\deg 0=-\infty$) then
     $f_{L_n}$ does not vanish identically on
    $\left(\mat{\mathbb{K}}{n}\right)^d$,
\end{itemize}
and for $X\in\left(\mat{\mathbb{K}}{n}\right)^d,$
\begin{equation}\label{eq:queen-expansion}
f(X)=\sum_{j=0}^\infty f_j(X)=\sum_{j=0}^{L_n} f_j(X).
\end{equation}
\end{thm}
\begin{rem}\label{rem:degrees_increase}
Since $n\times n$ matrices can be embedded into $(n+1)\times(n+1)$
matrices by adding zero row and column, the sequence of degrees
$\{L_n\}$ is non-decreasing. It eventually stabilizes at some $L$
if and only if it is bounded, in which case we see that $f_j$,
$j>L$, vanish on $\left(\mat{\mathbb{K}}{n}\right)^d$ for all
  $n$; this means that $f$ is a nc polynomial, i.e., Theorem
  \ref{thm:king-poly} follows from Theorem \ref{thm:queen-poly}.
\end{rem}
\begin{rem}\label{rem:queen-expansion'}
The expansion \eqref{eq:queen-expansion} can be written in the
form
\begin{equation}\label{eq:queen-expansion'}
f(X)=\sum_{j=0}^{L_1}f_j(X)+\sum_{k=1}^{\infty}
\sum_{j=L_k+1}^{L_{k+1}}f_j(X).
\end{equation}
Here the $k$-th term in the infinite series is a nc polynomial of
degree $L_{k+1}$ which vanishes on
$\left(\mat{\mathbb{K}}{k}\right)^d$ and does not vanish
identically on $\left(\mat{\mathbb{K}}{(k+1)}\right)^d$.
\end{rem}
\begin{rem}\label{rem:cf_ex}
Applying Theorem \ref{thm:queen-poly} to the nc function $f$ in
Example \ref{ex:degrees_unbdd}, we see that $f_j=h_j$, $j=0$, $1$,
\ldots. If the polynomials $p_n$ are chosen so that for each $n$,
$p_n$ does not vanish identically on
$\left(\mat{\mathbb{K}}{(n+1)}\right)^d$ then $L_n=\alpha_{n-1}$,
$n=2$, $3$, \ldots (and $L_1=-\infty$).
\end{rem}
\begin{proof}[Proof of Theorem \ref{thm:queen-poly}]
The proof is similar to that of Theorem \ref{thm:king-poly}. We
first obtain, exactly as in the proof of Theorem
\ref{thm:king-poly}, that for any fixed $n$, $Y$, and $Z$, the
degree of the polynomial $f_{Y,Z}(t)$ is at most $L_n$. Therefore,
we have
\begin{equation*}
f_{Y,Z}(t)= \sum_{\ell=0}^{L_n}t^\ell \Delta_R^\ell
f(Y,\ldots,Y)(Z,\ldots,Z).
\end{equation*}
Setting $t=1$ and $X=Y+Z$ (so that $f(X)=f_{Y,Z}(1)$), we obtain
\begin{equation}\label{eq:locally_finite_TT}
f(X)=\sum_{\ell=0}^{L_n}\Delta_R^\ell f(\underset{\ell+1\
\rm{times}}{\underbrace{Y,\ldots,Y}})(\underset{\ell\
\rm{times}}{\underbrace{X-Y,\ldots,X-Y}}).
\end{equation}
Notice that for each $n$, this is a finite version of the TT
formula \eqref{eq:TT}. In particular, setting $Y=0$  and using
\eqref{eq:delta^l_vs_delta^w_m=1}, we obtain
 that
\begin{equation}\label{eq:locally_finite_TT-power}
f(X)=\sum_{\ell=0}^{L_n}\Delta_R^\ell f(\underset{\ell+1\
\rm{times}}{\underbrace{0,\ldots,0}})(\underset{\ell\
\rm{times}}{\underbrace{X,\ldots,X}})=\sum_{\ell=0}^{L_n}
\sum_{|w|=\ell}X^w\,\Delta_R^{w^\top}\!\!f(\underset{l+1\
\rm{times}}{\underbrace{0,\ldots,0}}).
\end{equation}
This yields the expansion \eqref{eq:queen-expansion} with
\begin{equation}\label{eq:f_j}
f_j=\sum_{|w|=j}\Delta_R^{w^\top}\!\!f(\underset{j+1\
\rm{times}}{\underbrace{0,\ldots,0}})\,x^w.
\end{equation}
 The fact that $f_j$
vanishes on $\left(\mat{\mathbb{K}}{n}\right)^d$ for all
    $j>L_n$ and the  fact that $f_{LM_n}$ does not
vanish identically on $\left(\mat{\mathbb{K}}{n}\right)^d$
     follow from the assumption on the degree of $f(X_1,\ldots, X_d)$ as a
    polynomial function of matrix entries $(X_i)_{ab}$, $i=1,\ldots,d$;\ $a,b=1,\ldots,n$.
   Finally, nc polynomials $f_j$ in the
expansion \eqref{eq:queen-expansion} are determined uniquely,
because their evaluations $f_j(X)$ are the homogeneous components
of degree $j$ of  $f(X)$ viewed as a polynomial in the entries of
the matrices $X_i$.
\end{proof}

A nc function $f\colon
\ncspaced{\mathbb{K}}{d}\to\ncspace{\module{N}}$, for $\module{N}$
a vector space over a field $\mathbb{K}$, will be called
\emph{locally polynomial} \index{locally polynomial nc function}
if for each $n$, $f(X_1,\ldots,X_d)$ is a polynomial function in
$dn^2$ commuting variables $(X_i)_{jk}$, $i=1,\ldots,d$;
$j,k=1,\ldots,n$, with values in $\mat{\module{N}}{n}$. We denote
the $\mathbb{K}$-vector space of locally polynomial nc functions
from $\ncspace{\mathbb{K}^d}$ to $\ncspace{\module{N}}$ by
$\module{N}_{\operatorname{loc}}\langle x_1,\ldots,x_d\rangle$;
\index{$\module{N}_{\operatorname{loc}}\langle
x_1,\ldots,x_d\rangle$} it is obviously a module over the
$\mathbb{K}$-algebra $\mathbb{K}_{\operatorname{loc}}\langle
x_1,\ldots,x_d\rangle$. For each $n$, we define a submodule
$\module{J}_\module{N}^n\subseteq\module{N}_{\operatorname{loc}}\langle
x_1,\ldots,x_d\rangle$ \index{$\module{J}_\module{N}^n$}
consisting of locally polynomial nc functions vanishing on
$\left(\mat{\mathbb{K}}{n}\right)^d$. The submodules
$\module{J}_\module{N}^n$, $n=1,2,\ldots$, form a decreasing
sequence with zero intersection. In particular, we obtain a
decreasing sequence of ideals $\module{J}_\mathbb{K}^n$,
$n=1,2,\ldots$, with zero intersection in the algebra
$\mathbb{K}_{\operatorname{loc}}\langle x_1,\ldots,x_d\rangle$.
Taking the ideals $\module{J}_\mathbb{K}^n$ as a fundamental
system of neighbourhoods of zero \label{ADIC} makes
$\mathbb{K}_{\operatorname{loc}}\langle x_1,\ldots,x_d\rangle$
into a topological ring. Analogously, taking the submodules
$\module{J}_\module{N}^n$ as a fundamental system of
neighbourhoods of zero  makes
$\module{N}_{\operatorname{loc}}\langle x_1,\ldots,x_d\rangle$
into a topological module over the topological ring
$\mathbb{K}_{\operatorname{loc}}\langle x_1,\ldots,x_d\rangle$.
Formulae \eqref{eq:queen-expansion} and \eqref{eq:f_j} mean that
the TT series of $f\in\module{N}_{\operatorname{loc}}\langle
x_1,\ldots,x_d\rangle$ converges to $f$ in this topology as a
series of homogeneous polynomials; therefore, the space
$\module{N}\langle x_1,\ldots,x_d\rangle$ of nc polynomials is
dense in the space $\module{N}_{\operatorname{loc}}\langle
x_1,\ldots,x_d\rangle$ of locally polynomial nc functions. In
fact, $\module{N}_{\operatorname{loc}}\langle
x_1,\ldots,x_d\rangle$ is complete as a topological module --- it
is a completion of $\module{N}\langle x_1,\ldots,x_d\rangle$; we
leave the proof to the reader as an easy exercise.

We can extend Theorem \ref{thm:king-poly} to the setting of
 arbitrary (possibly, infinite-dimensional) vector
spaces over an infinite field $\field$. Let $\module{M}$ and
$\module{N}$ be vector spaces over $\field$. A nc function
$f\colon\ncspace{\module{M}}\to\ncspace{\module{N}}$ is said to be
\emph{polynomial on slices} \index{polynomial on slices nc
function} if $f|_{\mat{\module{M}}{n}}$ is polynomial on slices of
some degree $L_n$ for every $n=1,2,\ldots$; i.e., for every
$Y,Z\in\mat{\module{M}}{n}$, $f_{Y,Z}(t)=f(Y+tZ)$ is a
single-variable polynomial with coefficients in
$\mat{\module{N}}{n}$ of degree at most $L_n$, and $\deg
f_{Y,Z}=L_n$ for some $Y,Z\in\mat{\module{M}}{n}$ (for the case
$\field=\mathbb{C}$, see \cite[Section 26.2]{HiPh} or
\cite[Section 1]{Mu}).

A nc polynomial over $\module{M}$ with coefficients in
$\module{N}$ is a finite formal sum
\begin{equation}\label{eq:nc-poly}
p=\sum_{\ell=0}^Lp_\ell,
\end{equation}
where $p_\ell\colon\module{M}\times\cdots\times
\module{M}\to\module{N}$ are $\ell$-linear mappings; we say that
$p$ is of degree $L$ if $p_L\neq 0$. We can evaluate $p$ on
$X\in\mat{\module{M}}{n}$ by
\begin{equation}\label{eq:ncfun-poly}
p(X)=\sum_{\ell=0}^LX^{\odot\ell}p_\ell\in\mat{\module{N}}{n},\quad
n=1,2,\ldots,
\end{equation}
yielding a nc function on $\ncspace{\module{M}}$ with values in
$\ncspace{\module{N}}$. Restricting this nc function to ${\rm
Nilp}(\module{M})$, we deduce from Theorem
\ref{thm:ncps-nilp-gen-unique} that the coefficients $p_\ell$ are
uniquely determined by $p$, so that we can identify the nc
polynomial \eqref{eq:nc-poly} with the nc function
\eqref{eq:ncfun-poly} and use the same notation for both. Notice
that a nc polynomial over $\field^d$ can be identified with a nc
polynomial \eqref{eq:nc-poly-d} in $d$ noncommuting
indeterminates, as in \eqref{eq:Z^l_vs_Z^w},
\eqref{eq:lw-forms}--\eqref{eq:f_w_via_f_l} (with $s=1$).

If $p$ is a nc polynomial over $\module{M}$ of degree $L$, then
the nc function $p$ on $\ncspace{\module{M}}$ is polynomial on
slices, with $\max_{n\in\mathbb{N}}L_n=L$. As in the special case
$\module{M}=\field^d$, the converse is also true; the proof is
identical, up to notations, to that of Theorem
\ref{thm:king-poly}.
\begin{thm}\label{thm:king-poly-gen}
Let $f\colon\ncspace{\module{M}}\to\ncspace{\module{N}}$ be a nc
function which is polynomial on slices. Assume that the degrees
$L_n$ of $f|_{\mat{\module{M}}{n}}$, $n=1,2,\ldots$, are bounded.
 Then $f$ is a nc polynomial over $\module{M}$ with coefficients in
$\module{N}$.
\end{thm}

There is an analogous generalization of Theorem
\ref{thm:queen-poly} --- we leave the details to the reader.

\chapter{NC analyticity and convergence of TT series}\label{sec:conv}
In this section we discuss analyticity of nc functions and
convergence of TT series. It turns out that very mild boundedness
assumptions imply analyticity. There are two essentially different
settings: one where the boundedness assumptions hold in every
matrix dimension separately, and one where the boundedness
assumptions hold uniformly for all matrix dimensions. We assume in
this section that the ground ring $\ring$ is the field
$\mathbb{C}$; many results have analogues for the field
$\mathbb{R}$, see Remarks \ref{rem:real_queen},
\ref{rem:real_king}, and \ref{rem:real-higher}.

\section{Analytic nc functions}\label{subsec:analytic}
In this section we will use the theory of analytic functions from
vectors to vectors; our basic reference is \cite[Chapter III,
3.16--3.19 and Chapter XXVI]{HiPh}, see also \cite{Mu}.

Let $\vecspace{V}$ be a vector space over $\mathbb{C}$. A set
$\Omega\subseteq\ncspace{\vecspace{V}}$ is called \emph{finitely
open} \index{finitely open set} if for every $n\in\mathbb{N}$ the
set $\Omega_n\subseteq\mat{\vecspace{V}}{n}$ is finitely open,
i.e., the intersection of $\Omega_n$ with any finite-dimensional
subspace $\vecspace{U}$ of $\mat{\vecspace{V}}{n}$ is open in the
Euclidean  topology of $\vecspace{U}$. We notice that a finitely
open nc set $\Omega$ is right and left admissible. Indeed, if
$X\in\Omega_n$, $Y\in\Omega_m$, $Z\in\rmat{\vecspace{V}}{n}{m}$,
and $W\in\rmat{\vecspace{V}}{m}{n}$, then the sets
$$\Omega_{n+m}\cap\spa\left\{\begin{bmatrix}
X & 0\\
0 & Y
\end{bmatrix},\begin{bmatrix}
0 & Z\\
0 & 0
\end{bmatrix}\right\}$$
and
$$\Omega_{n+m}\cap\spa\left\{\begin{bmatrix}
X & 0\\
0 & Y
\end{bmatrix},\begin{bmatrix}
0 & 0\\
W & 0
\end{bmatrix}\right\}$$
are open in the corresponding two-dimensional vector spaces and
contain $\begin{bmatrix}
X & 0\\
0 & Y
\end{bmatrix}$, thus $\begin{bmatrix}
X & tZ\\
0 & Y
\end{bmatrix}$, $\begin{bmatrix}
X & 0\\
tW & Y
\end{bmatrix}\in\Omega_{n+m}$ for all sufficiently small $t\in\mathbb{C}$.

Let $\vecspace{W}$ be a Banach space over $\mathbb{C}$. An
\emph{admissible system of matrix norms over $\vecspace{W}$}
\index{admissible system of matrix norms} is a sequence of norms
$\|\cdot\|_n$ \index{$\Vert\cdot\Vert_n$} on
$\mat{\vecspace{W}}{n}$, $n=1,2,\ldots$, satisfying the following
two conditions:
\begin{itemize}
    \item For every $n,m\in\mathbb{N}$ there exist $C_1(n,m)$, $C_1'(n,m)>0$
    such that for all $X\in\mat{\vecspace{W}}{n}$ and
    $Y\in\mat{\vecspace{W}}{m}$,
    \begin{multline}\label{eq:dirsums-norms}
C_1(n,m)^{-1}\max\{\|X\|_n,\| Y\|_m\}\le\| X\oplus Y\|_{n+m}\\
\le
C_1'(n,m)\max\{\|X\|_n,\| Y\|_m\}.
    \end{multline}
    \item For every $n\in\mathbb{N}$ there exists $C_2(n)>0$ such that for all
    $X\in\mat{\vecspace{W}}{n}$ and
    $S,T\in\mat{\mathbb{C}}{n}$,
    \begin{equation}\label{eq:simprod-norms}
\| SXT\|_{n}\le C_2(n)\|S\|\,\|X\|_n\|T\|,
    \end{equation}
   where $\|\cdot\|$ denotes the operator norm of
   $\mat{\mathbb{C}}{n}$ with respect to the standard Euclidean
   norm of $\mathbb{C}^n$.
\end{itemize}
\begin{prop}\label{prop:inj-proj}
The conditions \eqref{eq:dirsums-norms}--\eqref{eq:simprod-norms}
are equivalent to the boundedness of the \emph{injections}
\index{injection} \index{$\iota_{ij}$}
\begin{equation}
\label{eq:inj-bd} \iota_{ij}\colon
\vecspace{W}\longrightarrow\mat{\vecspace{W}}{n},\quad
w\longmapsto E_{ij}w,
\end{equation}
and of the \emph{projections} \index{projection}
\begin{equation}\label{eq:proj-bd}
\pi_{ij}\colon\mat{\vecspace{W}}{n}\longrightarrow\vecspace{W},\quad
W=[w_{i'j'}]_{i',j'=1,\ldots,n}\longmapsto w_{ij},
\end{equation}
for all $n\in\mathbb{N}$ and all $i,j=1,\ldots,n$;  here
\index{$\pi_{ij}$} $E_{ij}\in\mat{\mathbb{C}}{n}$ \index{$E_{ij}$}
has the $(i,j)$-th entry $1$ and all the other entries $0$.
Moreover, the conditions
\eqref{eq:dirsums-norms}--\eqref{eq:simprod-norms} imply that for
any $s\in\mathbb{N}$, the \emph{block injections} \index{block
injection} \index{$\iota_{ij}^s$}
\begin{equation}\label{eq:block-inj-bd}
\iota_{ij}^s\colon
\mat{\vecspace{W}}{s}\longrightarrow\mat{\vecspace{W}}{ns}\cong
\mat{\left(\mat{\vecspace{W}}{s}\right)}{n},\quad
W\longmapsto E_{ij}\otimes W,
\end{equation}
and the block projections \index{block projection} \index{$
\pi_{ij}^s$}
\begin{gather}\label{eq:block-proj-bd} \pi_{ij}^s\colon
\mat{\vecspace{W}}{ns}\cong\mat{\left(\mat{\vecspace{W}}{s}\right)}{n}
\longrightarrow\mat{\vecspace{W}}{s},\\
W=[([w_{ab}]_{a,b=1,\ldots,s})_{i'j'}]_{i',j'=1,\ldots,n}\longmapsto
([w_{ab}]_{a,b=1,\ldots,s})_{ij}=W_{ij},\nonumber
\end{gather}
 are bounded.
\end{prop}
\begin{proof}
Let $s,n\in\mathbb{N}$ and $W\in\mat{\vecspace{W}}{s}$. Using
\eqref{eq:dirsums-norms} and \eqref{eq:simprod-norms}, we obtain
\begin{multline*}
\|\iota_{ij}^s(W)\|_{ns}=\|E_{ij}\otimes W
\|_{ns}=\|(P_{i1}\otimes I_s)(E_{11}\otimes W)(P_{1j}\otimes
I_s\|_{ns}\\
\le C_2(ns)\|E_{11}\otimes W\|_{ns}=C_2(ns)\|W\oplus
0_{(ns-s)\times(ns-s)}\|_{ns}\\
\le C_2(ns)C_1'(s,ns-s)\|W\|_s.
\end{multline*}
Here
$P_{ij}=I_n-E_{ii}-E_{jj}+E_{ij}+E_{ji}\in\mat{\mathbb{C}}{n}$ is
a permutation matrix ($i\ne j$). \index{$P_{ij}$} Therefore the
block injections $\iota_{ij}^s$ (in particular, the injections
$\iota_{ij}$) are bounded.

Next, let $s,n\in\mathbb{N}$ and
$W\in\mat{\vecspace{W}}{ns}\cong\mat{(\mat{\vecspace{W}}{s})}{n}$.
Using \eqref{eq:dirsums-norms} and \eqref{eq:simprod-norms}, we
obtain
\begin{multline*}
\|(\pi_{ij}^s(W)\|_{s}=\|W_{ij}\|_{s}\le C_1(s,ns-s)\|W_{ij}\oplus
0_{(ns-s)\times(ns-s)}\|_{ns}\\
=C_1(s,ns-s)\|(E_{1i}\otimes
I_{\mat{\vecspace{W}}{s}})W(E_{j1}\otimes
I_{\mat{\vecspace{W}}{s}})\|_{ns}\\
\le
C_1(s,ns-s)C_2(ns)\|W\|_{ns}.
\end{multline*}
Therefore the block projections $\pi_{ij}^s$ (in particular, the
projections $\pi_{ij}$) are bounded.

Assume now that the injections \eqref{eq:inj-bd} and the
projections \eqref{eq:proj-bd} are all bounded. Let
$n,m\in\mathbb{N}$ and $X\in\mat{\vecspace{W}}{n}$,
$Y\in\mat{\vecspace{W}}{m}$. Then we obtain \begin{multline*}
\|X\oplus
Y\|_{n+m}=\Big\|\sum_{i,j=1}^n\iota_{ij}(x_{ij})+\sum_{i,j=1}^{m}\iota_{n+i,n+j}(y_{ij})\Big\|_{n+m}\\
\le
\sum_{i,j=1}^n\|\iota_{ij}(x_{ij})\|_{n+m}+\sum_{i,j=1}^{m}\|\iota_{n+i,n+j}(y_{ij})\|_{n+m}\\
\le \sum_{i,j=1}^n\|\iota_{ij}\|\cdot\|x_{ij}\|_{1}+
\sum_{i,j=1}^{m}\|\iota_{n+i,n+j}\|\cdot\|y_{ij}\|_{1}\\
= \sum_{i,j=1}^n\|\iota_{ij}\|\cdot\|\pi_{ij}(X)\|_{1}+
\sum_{i,j=1}^{m}\|\iota_{n+i,n+j}\|\cdot\|\pi_{ij}(Y)\|_{1}\\
\le\Big(\sum_{i,j=1}^n\|\iota_{ij}\|\cdot\|\pi_{ij}\|+
\sum_{i,j=1}^{m}\|\iota_{n+i,n+j}\|\cdot\|\pi_{ij}\|\Big)
\max\{\|X\|_n,\|Y\|_m\}
\end{multline*}
and
\begin{multline*}
\max\{\|X\|_n,\|Y\|_m\}\le\max\Big\{\Big\|\sum_{i,j=1}^n\iota_{ij}(x_{ij})\Big\|_n,
\Big\|\sum_{i,j=1}^{m}\iota_{ij}(y_{ij})\Big\|_{m}\Big\}\\
\le\max\Big\{\sum_{i,j=1}^n\|\iota_{ij}\|\cdot\|x_{ij}\|_1,
\sum_{i,j=1}^{m}\|\iota_{ij}\|\cdot\|y_{ij}\|_{1}\Big\}\\
=\max\Big\{\sum_{i,j=1}^n\|\iota_{ij}\|\cdot\|\pi_{ij}(X\oplus
Y)\|_1,
\sum_{i,j=1}^{m}\|\iota_{ij}\|\cdot\|\pi_{n+i,n+j}(X\oplus Y)\|_{1}\Big\}\\
\le\max\Big\{\sum_{i,j=1}^n\|\iota_{ij}\|\cdot\|\pi_{ij}\|,
\sum_{i,j=1}^{m}\|\iota_{ij}\|\cdot\|\pi_{n+i,n+j}\|\Big\}\|X\oplus
Y\|_{n+m}.
\end{multline*}
Therefore, \eqref{eq:dirsums-norms} follows.

Assume again that the injections \eqref{eq:inj-bd} and the
projections \eqref{eq:proj-bd} are all bounded. Let
$n\in\mathbb{N}$, $X\in\mat{\vecspace{W}}{n}$, and
$S,T\in\mat{\mathbb{C}}{n}$. Then
\begin{multline*}
\|SXT\|_n=\Big\|\sum_{i,j=1}^n\iota_{ij}((SXT)_{ij})\Big\|_n
=\Big\|\sum_{i,j=1}^n\iota_{ij}\Big(\sum_{k,\ell=1}^ns_{ik}x_{k\ell}t_{\ell
j}\Big)\Big\|_n\\
\le\sum_{i,j,k,\ell=1}^n\|\iota_{ij}\|\cdot
|s_{ik}|\cdot\|x_{k\ell}\|_1\cdot |t_{\ell
j}|=\sum_{i,j,k,\ell=1}^n\|\iota_{ij}\|\cdot
|s_{ik}|\cdot\|\pi_{k\ell}(X)\|_1\cdot |t_{\ell j}|\\
\le
\Big(\sum_{i,j=1}^n\|\iota_{ij}\|\Big)\Big(\sum_{k,\ell=1}^n\|\pi_{k\ell}\|\Big)
\Big(\max_{i,k=1,\ldots,n}|s_{ik}|\Big)\|X\|_n\Big(\max_{\ell,j=1,\ldots,n}|t_{\ell
j}|\Big)\\
\le
\Big(\sum_{i,j=1}^n\|\iota_{ij}\|\Big)\Big(\sum_{k,\ell=1}^n\|\pi_{k\ell}\|\Big)
\|S\|\|X\|_n\|T\|,
\end{multline*}
i.e., \eqref{eq:simprod-norms} holds.
\end{proof}

Proposition \ref{prop:inj-proj} means that a sequence of norms
$\|\cdot\|_n$ on $\mat{\vecspace{W}}{n}$ is an admissible system
of matrix norms over $\vecspace{W}$ if and only if
$\mat{\vecspace{W}}{n}$  is homeomorphic to the direct product of
$n^2$ copies of $\vecspace{W}$ for each $n$. In particular, any
two admissible systems of matrix norms over $\vecspace{W}$ yield
equivalent norms on $\mat{\vecspace{W}}{n}$ for each $n$, and for
any admissible system of matrix norms over $\vecspace{W}$ the
spaces $\mat{\vecspace{W}}{n}$ are all complete. \label{POLYBDD}

Let $\vecspace{V}$ be a vector space, let
$\Omega\subseteq\ncspace{\vecspace{V}}$ be a finitely open nc set,
and let $\vecspace{W}$ be a Banach space with an admissible system
of matrix norms over $\vecspace{W}$. A nc function
$f\colon\Omega\to\ncspace{\vecspace{W}}$ is called \emph{locally
bounded on slices} \index{locally bounded on slices nc function}
if for every $n\in\mathbb{N}$ the function $f|_{\Omega_n}$ is
locally bounded on slices, i.e., for every $X\in\Omega_n$ and
$Z\in\mat{\vecspace{V}}{n}$ there exists $\epsilon>0$ such that
$f(X+tZ)$ is bounded for $|t|<\epsilon$. A nc function
$f\colon\Omega\to\ncspace{\vecspace{W}}$ is called
\emph{G\^{a}teaux (G-) differentiable} \index{G\^{a}teaux (G-)
differentiable nc dunction} if for every $n\in\mathbb{N}$ the
function $f|_{\Omega_n}$ is G-differentiable, i.e., for every
$X\in\Omega_n$ and $Z\in\mat{\vecspace{V}}{n}$ the G-derivative of
$f$ at $X$ in direction $Z$, \index{$\delta f(X)(Z)$}
$$\delta f(X)(Z)=\lim_{t\to
0}\frac{f(X+tZ)-f(X)}{t}=\frac{d}{dt}f(X+tZ)\Big|_{t=0},$$ exists.
It follows that $f$ is \emph{analytic on slices}, \index{analytic
on slices nc function} i.e., for every $X\in\Omega_n$ and
$Z\in\mat{\vecspace{V}}{n}$, $f(X+tZ)$ is an analytic function of
$t$ in a neighbourhood of $0$. By Hartogs' theorem \cite[Page
28]{Sh},  $f$ is analytic on $\vecspace{U}\cap\Omega_n$ as a
function of several complex variables for every $n$ and every
finite-dimensional subspace $\vecspace{U}$ of
$\mat{\vecspace{V}}{n}$. We also note that $\delta
f(X)\colon\mat{\vecspace{V}}{n}\to\mat{\vecspace{W}}{n}$ is a
linear operator \cite[Theorem 26.3.2]{HiPh}.

Let $\vecspace{X}$ be a vector space, and let $Y\in\vecspace{X}$.
A set $\Upsilon\subseteq\vecspace{X}$ is called \emph{complete
circular} \index{complete circular set (or c-star) about $Y$} (or
a \emph{c-star about $Y$} in the terminology of \cite[Definition
3.16.1]{HiPh}) if for every $X\in\Upsilon$ we have that
$Y+t(X-Y)\in\Upsilon$ for all $t\in\mathbb{C}$ with $|t|\le 1$.
\begin{thm}\label{thm:g-queen}
Let a nc function $f\colon\Omega\to\ncspace{\vecspace{W}}$ be
locally bounded on slices. Then
\begin{enumerate}
    \item $f$ is G-differentiable.
    \item For every $n\in\mathbb{N}$, $Y\in\Omega_n$,
    $Z\in\mat{\vecspace{V}}{n}$, and each $N\in\mathbb{N}$,
    $$\frac{1}{N!}\frac{d^N}{dt^N}f(Y+tZ)\Big|_{t=0}=\Delta_R^Nf(\underset{N+1\
    {\rm times}}{\underbrace{Y,\ldots,Y}})(\underset{N\
    {\rm times}}{\underbrace{Z,\ldots,Z}}).$$
    \item For every $n\in\mathbb{N}$ and $Y\in\Omega_n$, let
    $$\Upsilon(Y)=\{X\in\Omega_n\colon
     Y+t(X-Y)\in\Omega_n\ \mathrm{for\ all\ } t\in\mathbb{C}\ \mathrm{with}\ |t|\le 1\}.$$
  Then $\Upsilon(Y)$ \index{$\Upsilon(Y)$} is the maximal complete circular set about
  $Y$ contained in $\Omega_n$ and is finitely open. For every
  $X\in\Upsilon(Y)$,
\begin{equation}\label{eq:tt-g-queen}
f(X)=\sum_{\ell=0}^\infty\Delta_R^\ell f(\underset{\ell+1\
\rm{times}}{\underbrace{Y,\ldots,Y}})(\underset{\ell\
\rm{times}}{\underbrace{X-Y,\ldots,X-Y}}),
\end{equation}
where the series converges absolutely. \item The series in
\eqref{eq:tt-g-queen} converges uniformly on compact subsets of
$\vecspace{U}\cap\Upsilon(Y)$ for every finite-dimensional
subspace $\vecspace{U}$ of $\mat{\vecspace{V}}{n}$. Moreover, for
every such  compact set $K$ there exists a finitely open complete
circular set $\Upsilon_K$ about $Y$,
$K\subseteq\Upsilon_K\subseteq\Upsilon(Y)$ such that
\begin{equation}\label{eq:tt-g-queen-normal}
\sum_{\ell=0}^\infty\sup_{X\in\Upsilon_K}\|\Delta_R^\ell
f(\underset{\ell+1\
\rm{times}}{\underbrace{Y,\ldots,Y}})(\underset{\ell\
\rm{times}}{\underbrace{X-Y,\ldots,X-Y}})\|_n<\infty.
\end{equation}
\end{enumerate}
\end{thm}
\begin{proof} (1),(2).
Let $n,N\in\mathbb{N}$, $Y\in\Omega_n$, and
$Z\in\mat{\vecspace{V}}{n}$ be fixed. Since the nc set $\Omega$ is
finitely open, $\underset{N+2\
    \rm{times}}{\underbrace{Y\oplus\ldots\oplus Y}}\in\Omega_{n(N+2)}$ and
there exists $r>0$ such that the $(N+2)\times (N+2)$ matrix over
$\mat{\vecspace{V}}{n}$,
$$\begin{bmatrix}
Y & rZ & 0 & \ldots & 0\\
0 & \ddots & \ddots & \ddots & \vdots\\
\vdots & \ddots  & \ddots & \ddots & 0 \\
\vdots &  & \ddots & \ddots & rZ\\
0 & \ldots & \ldots & 0 & Y
 \end{bmatrix},$$
belongs to $\Omega_{n(N+2)}$. Since $\Omega_{n(N+2)}$ is finitely
open and $f$ is locally bounded on slices, there exists
$\epsilon_1>0$ such that
$$\begin{bmatrix}
Y & rZ & 0 & \ldots & 0\\
0 & \ddots & \ddots & \ddots & \vdots\\
\vdots & \ddots  & \ddots & \ddots & 0 \\
\vdots &  & \ddots & Y & rZ\\
0 & \ldots & \ldots & 0 & Y+trZ
 \end{bmatrix}\in\Omega_{n(N+2)}$$
 and $$f\left(\begin{bmatrix}
Y & rZ & 0 & \ldots & 0\\
0 & \ddots & \ddots & \ddots & \vdots\\
\vdots & \ddots  & \ddots & \ddots & 0 \\
\vdots &  & \ddots & Y & rZ\\
0 & \ldots & \ldots & 0 & Y+trZ
 \end{bmatrix}\right)$$
 is bounded for $t\in\mathbb{C}\colon |t|<\epsilon_1$. Since the Banach space $\vecspace{W}$
 is equipped with
 an admissible system of matrix norms over $\vecspace{W}$, by
 Theorem \ref{thm:bidiag},
 $\Delta_R^{N+1}f(Y,\ldots,Y,Y+trZ)(rZ,\ldots,rZ)$ is bounded for  $t\in\mathbb{C}\colon
 |t|<\epsilon_1$. Next, since $\Omega_n$ is finitely open, there
 exists $\epsilon_2>0$ such that $Y+trZ\in\Omega_n$ for all
 $t\in\mathbb{C}\colon |t|<\epsilon_2$. According to the TT
 formula \eqref{eq:TT}, we have for
 $\epsilon=\min\{\epsilon_1,\epsilon_2\}$ and $|t|<\epsilon$ that
 \begin{multline*}
f(Y+trZ)=\sum_{\ell=0}^N\Delta_R^\ell
f(Y,\ldots,Y)(trZ,\ldots,trZ)\\
+\Delta_R^{N+1}f(Y,\ldots,Y,Y+trZ)(trZ,\ldots,trZ),
 \end{multline*}
 or, for $|t|<\epsilon r$, that
 \begin{multline}\label{eq:TT-t}
f(Y+tZ)=\sum_{\ell=0}^Nt^\ell\Delta_R^\ell
 f(Y,\ldots,Y)(Z,\ldots,Z)\\
+t^{N+1}\Delta_R^{N+1}f(Y,\ldots,Y,Y+tZ)(Z,\ldots,Z).
 \end{multline}
 Since $\Delta_R^{N+1}f(Y,\ldots,Y,Y+tZ)(Z,\ldots,Z)$ is bounded for $|t|<\epsilon
 r$,
 $f(Y+tZ)$ is $N$ times differentiable as a function of $t$ at
 $t=0$, and
 the equality in part (2) holds.
Clearly, part (1) follows as the
 special case of (2) where $N=1$.

 (3). The fact that $\Upsilon(Y)$ is the maximal complete circular
 set about $Y$ contained in $\Omega_n$ is immediate from the
 definition of $\Upsilon(Y)$.

 The fact that $\Upsilon(Y)$ is finitely open is obtained as
 follows. Let $\vecspace{U}$ be a finite-dimensional subspace of
 $\mat{\vecspace{V}}{n}$. Then
 $$\vecspace{U}\cap\Upsilon(Y)=
 \{X\in\vecspace{U}\cap\Omega_n\colon Y+t(X-Y)\in\Omega_n\ \mathrm{for\ all\ }
  t\in\mathbb{C}\ \mathrm{with}\ |t|\le 1\}.$$
 Suppose that we are given a norm $\|\cdot\|$ on the
 space $\spa\{Y\}+\vecspace{U}$. (Since this space is finite-dimensional,
 all norms on it are
 equivalent.) Let $X\in\vecspace{U}\cap\Upsilon(Y)$. Then for
 every $t\in\overline{\mathbb{D}}$ there is an open ball
 $B_t$  centered at $Y+t(X-Y)$ in the space
 $\spa\{Y\}+\vecspace{U}$ which is contained in the open set
 $(\spa\{Y\}+\vecspace{U})\cap\Omega_n$. Since the set
 $S:=\{Y+t(X-Y)\colon |t|\le 1\}$ is compact in
 $\spa\{Y\}+\vecspace{U}$,
 there exists a finite number of balls $B_{t_k}$,  $k=1,\ldots,m$,
 which covers
 $S$. Let $R_k$ be the radius of $B_{t_k}$. Since for every
 $t\in\overline{\mathbb{D}}$ there is a $k\colon 1\le k\le m$ such
 that $Y+t(X-Y)\in B_{t_k}$, we have that $$\phi(t):=\max_{1\le k\le
 m}\{R_k-|t-t_k|\|X-Y\|\}>0.$$
Clearly, the function $\phi$ is continuous on
$\overline{\mathbb{D}}$, hence
$$R:=\min_{t\in\overline{\mathbb{D}}}\phi(t)>0.$$
Then
 for every $t\in\overline{\mathbb{D}}$ there is
  $k_0\colon 1\le k_0\le m$ such that
  $$R_{k_0}-|t-t_{k_0}|\|X-Y\|=\phi(t).$$ Therefore, we have
 $$Y+t(X+Z-Y)=Y+t_{k_0}(X-Y)+(t-t_{k_0})(X-Y)+tZ\in(\spa\{Y\}+\vecspace{U})\cap\Omega_n$$
for every $Z\in\spa\{Y\}+\vecspace{U}$ with $\|Z\|<R$, since
$$\|(t-t_{k_0})(X-Y)+tZ\|\le |t-t_{k_0}|\|X-Y\|+|t|\|Z\|<R_{k_0}$$
and $Y+t(X+Z-Y)\in B_{t_{k_0}}$. In particular,
$Y+t(X+Z-Y)\in(\spa\{Y\}+\vecspace{U})\cap\Omega_n$ is
 true for every $t\in\overline{\mathbb{D}}$ and $Z\in\vecspace{U}$ with $\|Z\|<R$. (Clearly, the induced
 norm on $\vecspace{U}$ is equivalent to any other norm
 on $\vecspace{U}$.)
 Therefore, $X+Z\in\vecspace{U}\cap\Upsilon(Y)$. We conclude that the set
 $\vecspace{U}\cap\Upsilon(Y)$ is open in $\vecspace{U}$ and thus
 the set $\Upsilon(Y)$ is finitely open.

 It follows from parts (1) and (2) that for every
 $X\in\Upsilon(Y)$ the function $f(Y+t(X-Y))$ is analytic in $t$
 on the closed disk $\{t\in\mathbb{C}\colon |t|\le 1\}$ (see the
 remark on the analyticity on slices preceding this theorem), and
 $$f(Y+t(X-Y))=\sum_{\ell=0}^\infty t^\ell\Delta_R^\ell
f(Y,\ldots,Y)(X-Y,\ldots,X-Y),\quad |t|\le 1,$$  where the series
converges absolutely. In particular, this is true for $t=1$, which
yields the last statement in part (3).

(4).  Let $\vecspace{U}$ be a finite-dimensional subspace of
$\mat{\vecspace{V}}{n}$. Then, clearly,
$(\spa\{Y\}+\vecspace{U})\cap\Upsilon(Y)$ is a complete circular
set about $Y$, which is open in $\spa\{Y\}+\vecspace{U}$. By part
(1), $f$ is G-differentiable in $\Omega$, and in particular in
$(\spa\{Y\}+\vecspace{U})\cap\Upsilon(Y)$. By the remark preceding
this theorem, $f$ is analytic in
$(\spa\{Y\}+\vecspace{U})\cap\Upsilon(Y)$ as a function of several
complex variables. More precisely, if we pick up some basis
$\{E_j\}_{j=1,\ldots,d}$ in $\spa\{Y\}+\vecspace{U}$ and write
$X-Y=\sum_{j=1}^dz_jE_j$, then $g(z):=f(X)$ is an analytic
function in several complex variables, $z=(z_1, \ldots, z_d)$, in
a complete circular domain with center $z=0$. Since $\Delta_R^\ell
f(Y,\ldots,Y)$ is a $\ell$-linear form, the series in
\eqref{eq:tt-g-queen} in $z$-coordinates coincides with the series
of homogeneous polynomials in the Taylor expansion of $g(z)$. By
\cite[Theorem I.3.3]{Sh}, this series converges uniformly on any
compact subset $K_Y$ of $(\spa\{Y\}+\vecspace{U})\cap\Upsilon(Y)$.
In particular, it converges uniformly on the compact set
$K:=K_Y\cap\vecspace{U}$ (unless the latter set is empty). On the
other hand, given any compact subset $K$ of
$\vecspace{U}\cap\Upsilon(Y)$, there exists $\delta>0$ such that
the set $$K_Y:=\{X+sY\colon X\in K,\ s\in\mathbb{C},
|s|\le\delta\}$$ is contained in
$(\spa\{Y\}+\vecspace{U})\cap\Upsilon(Y)$. Clearly, $K_Y$ is
compact and $K=K_Y\cap\vecspace{U}$. Therefore we can conclude
that the series in \eqref{eq:tt-g-queen} converges uniformly on
compact subsets of $\vecspace{U}\cap\Upsilon(Y)$.

The second statement in part (4) follows from \cite[Theorem
26.3.8]{HiPh}.
\end{proof}

\begin{rem}\label{rem:polylin}
The fact that the \emph{G-differential} \index{G\^{a}teaux (G-)
differential} $\delta
f(X)\colon\mat{\vecspace{V}}{n}\to\mat{\vecspace{W}}{n}$ is a
linear operator follows also from Theorem \ref{thm:g-queen}(2) for
$N=1$, since $\delta f(X)\!=\!\Delta_Rf(X,X)$. More generally, one
defines the \emph{$N$-th G-differential} \index{higher order
G\^{a}teaux (G-) differential} $\delta^N
f(X)\colon\mat{\vecspace{V}}{n}\to\mat{\vecspace{W}}{n}$ by
\index{$\delta^N f(X)$}
$$\delta^Nf(X)(Z)=\frac{d^N}{dt^N}f(X+tZ)\Big|_{t=0}.$$
It is known \cite[Theorem 26.3.5]{HiPh} that $\delta^Nf(X)(Z)$ is
a homogeneous polynomial of degree $N$ in $Z$, that is,
$\delta^Nf(X)(\alpha Z+\beta W)$ is a homogeneous polynomial of
degree $N$ in two complex variables $\alpha$ and $\beta$ for any
$Z$ and $W$. We observe that this also follows from Theorem
\ref{thm:g-queen}(2), since
\begin{equation}\label{eq:N-deltas}
\frac{1}{N!}\delta^N f(X)=\Delta_R^Nf(\underset{N+1\
    \rm{times}}{\underbrace{X,\ldots,X}}).
\end{equation}
\end{rem}

Let $\vecspace{V}$ be a Banach space equipped with an admissible
system of matrix norms over $\vecspace{V}$.  We will say that a
set $\Omega\subseteq\ncspace{\vecspace{V}}$ is \emph{open}
\index{open set in a nc space} if for every $n\in\mathbb{N}$ and
$Y\in\Omega_n$ there exists $\delta_n>0$ such that the open ball
$$B(Y,\delta_n):=\{X\in\mat{\vecspace{V}}{n}\colon
\|X-Y\|_n<\delta_n\}$$ \index{$B(Y,\delta)$}is contained in
$\Omega_n$. Clearly, an open set is finitely open.

Let $\vecspace{V}$ and $\vecspace{W}$ be Banach spaces equipped
with admissible systems of matrix norms over $\vecspace{V}$ and
over $\vecspace{W}$, and let
$\Omega\subseteq\ncspace{\vecspace{V}}$ be an open nc set. A nc
function $f\colon\Omega\to\ncspace{\vecspace{W}}$ is called
\emph{locally bounded} \index{locally bounded nc function} if for
every $n\in\mathbb{N}$ the function $f|_{\Omega_n}$ is locally
bounded, i.e., for every $Y\in\Omega_n$ there exists $\delta_n>0$
such that $f$ is bounded on $B(Y,\delta_n)$. Clearly, a locally
bounded nc function is locally bounded on slices.

Let $\Omega\subseteq\ncspace{\vecspace{V}}$ be an open nc set. A
nc function $f\colon\Omega\to\ncspace{\vecspace{W}}$ is called
\emph{Fr\'{e}chet (F-) differentiable} \index{Fr\'{e}chet (F-)
differentiable nc function} if for every $n\in\mathbb{N}$ the
function $f|_{\Omega_n}$ is F-differentiable, i.e.,
$f|_{\Omega_n}$ is G-differentiable and for any $X\in\Omega_n$ the
linear operator $\delta f(X)\colon
\mat{\vecspace{V}}{n}\to\mat{\vecspace{W}}{n}$ is bounded. It was
shown by Zorn in \cite{Z} that in this case
$$\lim_{\|Z\|_n\to 0}\frac{\|f(X+Z)-f(X)-\delta
f(X)(Z)\|_n}{\|Z\|_n}=0$$ for all $X\in\Omega_n$. A nc function
$f\colon\Omega\to\ncspace{\vecspace{W}}$ is called \emph{analytic}
\index{analytic nc function} if for every $n\in\mathbb{N}$ the
function $f|_{\Omega_n}$ is analytic, i.e., $f|_{\Omega_n}$ is
locally bounded and G-differentiable. By \cite[Theorem
3.17.1]{HiPh}, in this case $f|_{\Omega_n}$ is also continuous and
F-differentiable. It follows that an analytic nc function
$f\colon\Omega\to\ncspace{\vecspace{W}}$ is continuous (i.e., all
its restrictions $f|_{\Omega_n}$ are continuous) and
F-differentiable.

\begin{thm}\label{thm:f-queen}
Let a nc function $f\colon\Omega\to\ncspace{\vecspace{W}}$ be
locally bounded. Then, in addition to the conclusions of Theorem
\ref{thm:g-queen}, $f$ is analytic. Let $n\in\mathbb{N}$,
$Y\in\Omega_n$, and let $\Upsilon$ be an open complete circular
set about $Y$ such that $\Upsilon\subseteq\Omega_n$ and $f$ is
bounded on $\Upsilon$. For every $\epsilon>0$, the TT series in
\eqref{eq:tt-g-queen} converges uniformly on the set
$$\Upsilon_\epsilon:=\{X\in\Upsilon\colon
Y+(1+\epsilon)(X-Y)\in\Upsilon\}.$$
\index{$\Upsilon_\epsilon$}Moreover,
\begin{equation}\label{eq:tt-f-queen-normal}
\sum_{\ell=0}^\infty\sup_{X\in \Upsilon_\epsilon}\|\Delta_R^\ell
f(\underset{\ell+1\
\rm{times}}{\underbrace{Y,\ldots,Y}})(\underset{\ell\
\rm{times}}{\underbrace{X-Y,\ldots,X-Y}})\|_n<\infty.
\end{equation}
\end{thm}
\begin{proof}
Since $f$ is locally bounded, $f$ is also locally bounded on
slices. By Theorem \ref{thm:g-queen}, $f$ is G-differentiable.
Therefore $f$ is analytic.

Assume that $\|f(X)\|_n\le M$ for some $M>0$ and all $X\in
\Upsilon$. Then for any $Z\in \mat{\vecspace{V}}{n}$ with
$Y+Z\in\Upsilon$ the function $f(Y+tZ)$ is analytic in $t$ on the
closed disk $\{t\in\mathbb{C}\colon |t|\le 1\}$, and the Cauchy
inequalities
$$\frac{1}{N!}\left\|\frac{d^N}{dt^N}f(Y+tZ)\Big|_{t=0}\right\|_n\le
M$$ hold for every $N\in\mathbb{N}$. By Theorem
\ref{thm:g-queen}(2) and the homogeneity of degree $N$ in $Z$ of
$\Delta_R^Nf(Y,\ldots,Y)(Z,\ldots,Z)$, we obtain the following
estimate for all $X\in \Upsilon_\epsilon$:
$$\|\Delta_R^Nf(Y,\ldots,Y)(X-Y,\ldots,X-Y)\|_n\le
\frac{M}{(1+\epsilon)^N}.$$ This implies the convergence of the
series in \eqref{eq:tt-f-queen-normal}, and hence the uniform
convergence of the TT series in \eqref{eq:tt-g-queen}.
\end{proof}
We notice an important special case of Theorem \ref{thm:f-queen}.
\begin{cor}\label{cor:balls}
Let a nc function $f\colon\Omega\to\ncspace{\vecspace{W}}$ be
locally bounded. For every $n\in\mathbb{N}$ and $Y\in\Omega_n$,
let $\delta_n=\sup\{r>0\colon f\ {\rm is\ bounded\ on\ }
{B}(Y,r)\}.$
 Then the TT
series in \eqref{eq:tt-g-queen} converges absolutely and uniformly
on every open ball ${B}(Y,r)\subsetneq B(Y,\delta_n)$. Moreover,
\begin{equation}\label{eq:tt-f-queen-normal'}
\sum_{\ell=0}^\infty\sup_{X\in B(Y,r)}\|\Delta_R^\ell
f(\underset{\ell+1\
\rm{times}}{\underbrace{Y,\ldots,Y}})(\underset{\ell\
\rm{times}}{\underbrace{X-Y,\ldots,X-Y}})\|_n<\infty.
\end{equation}
\end{cor}
\begin{cor}\label{cor:nc-analytic}
Let $\Omega\subseteq\ncspace{\vecspace{V}}$ be an open nc set.
Then a nc function $f\!\colon\!\Omega\to\ncspace{\vecspace{W}}$ is
locally bounded if and only if $f$ is continuous if and only if
$f$ is F-differenti\-able if and only if $f$ is analytic.
\end{cor}
\begin{rem}\label{rem:polybdd}
The fact that $\delta
f(X)\colon\mat{\vecspace{V}}{n}\to\mat{\vecspace{W}}{n}$ is a
bounded linear operator for any $X\in\Omega_n$ (and thus $f$ is
F-differentiable) follows also from the equality $\delta
f(X)=\Delta_Rf(X,X)$. Indeed, $\Delta_Rf(X,X)(Z)$ is the $(1,2)$
block entry of the matrix
$f\left(\begin{bmatrix} X & Z\\
0 & X\end{bmatrix}\right)$. Since $f$ is locally bounded, there
exists $\delta>0$ such that $f\left(\begin{bmatrix} X & Z\\
0 & X\end{bmatrix}\right)$ is bounded when $\|Z\|_n<\delta$. But
then so is $\Delta_Rf(X,X)(Z)$, since the system of matrix norms
over $\vecspace{W}$ is admissible. More generally, it is known
\cite[Theorems 26.3.5 and 26.3.6]{HiPh}  that the $N$-th
G-differential
$\delta^Nf(X)\colon\mat{\vecspace{V}}{n}\to\mat{\vecspace{W}}{n}$
is a bounded homogeneous polynomial of degree $N$, i.e.,
$$\|\delta^Nf(X)(Z)\|_n\le C\|Z\|_n^N$$
for some $C>0$. This also follows from \eqref{eq:N-deltas}, since
$\Delta_R^Nf(X,\ldots,X)(Z,\ldots,Z)$ is the $(1,N+1)$ block entry
of the $(N+1)\times (N+1)$ matrix
$$f\left(\begin{bmatrix}
X & Z & 0 & \ldots & 0\\
0 & \ddots & \ddots & \ddots & \vdots\\
\vdots & \ddots  & \ddots & \ddots & 0 \\
\vdots &  & \ddots & \ddots & Z\\
0 & \ldots & \ldots & 0 & X
 \end{bmatrix}\right),$$
see Theorem \ref{thm:bidiag}. It follows from the proof of Theorem
\ref{thm:f-queen} that, moreover, there exist positive constants
$K$ and $\rho$ such that
\begin{equation}\label{eq:N-Delta-bdd}
\|\Delta_R^Nf(X,\ldots,X)(Z,\ldots,Z)\|_n\le K\rho^N\|Z\|_n^N.
\end{equation}
\end{rem}

The convergence results, Theorems \ref{thm:g-queen} and
\ref{thm:f-queen}, allow us to write infinite nc power expansions
centered at a matrix $Y\in\Omega_s$ valid in matrix dimensions
which are multiples of $s$, as in Theorem \ref{thm:tt-power-gen}.
\begin{thm}\label{thm:tt-series-gen}
Suppose that $\vecspace{V}$ is a vector space over $\mathbb{C}$,
 $\Omega\subseteq\ncspace{\vecspace{V}}$ is a finitely open nc set, and
$\vecspace{W}$ is a Banach space with an admissible system of
matrix norms over $\vecspace{W}$. Let
$f\colon\Omega\to\ncspace{\vecspace{W}}$ be a nc function which is
locally bounded on slices,
 and let $Y\in\Omega_s$. Then for an arbitrary
 $m\in\mathbb{N}$,
\begin{equation}\label{eq:tt-series-gen}
f(X)=\sum_{\ell=0}^\infty\Big(X-
\bigoplus_{\alpha=1}^mY\Big)^{\odot_s\ell}\,\Delta_R^\ell
f(\underset{\ell+1\ \rm{times}}{\underbrace{Y,\ldots,Y}}),
\end{equation}
where the series converges absolutely and uniformly on compact
subsets of $\vecspace{U}\cap\Upsilon(\bigoplus_{\alpha=1}^mY)$ for
any finite-dimensional subspace $\vecspace{U}$ of
$\mat{\vecspace{V}}{sm}$. Here $\Upsilon(\bigoplus_{\alpha=1}^mY)$
is the maximal complete circular set about
$\bigoplus_{\alpha=1}^mY$ contained in $\Omega_{sm}$, as in
Theorem \ref{thm:tt-power-gen}(3). (Notice that
$\coprod_{m=1}^\infty\Upsilon\Big(\bigoplus_{\alpha=1}^mY\Big)$ is
a nc set.)

If, furthermore, $\vecspace{V}$ is a Banach space with an
admissible system of matrix norms over $\vecspace{V}$, $\Omega$ is
an open nc set, and $f$ is locally bounded, then the TT series in
\eqref{eq:tt-series-gen} converges uniformly on
$\Upsilon_\epsilon$ for every $\epsilon>0$ and every $\Upsilon$ an
open complete circular set about $\bigoplus_{\alpha=1}^mY$
contained in $\Omega_{sm}$ and such that $f$ is bounded on
$\Upsilon$, where $\Upsilon_\epsilon$ is defined as in Theorem
\ref{thm:f-queen}. As a special case, the TT series converges
uniformly on every open ball
$B(\bigoplus_{\alpha=1}^mY,r)\subsetneq
B(\bigoplus_{\alpha=1}^mY,\delta_{sm})$, where $
B(\bigoplus_{\alpha=1}^mY,\delta_{sm})$ is defined as in Corollary
\ref{cor:balls}.

In particular, if  $\mu\in\Omega_1$ then \eqref{eq:tt-series-gen}
becomes
\begin{equation}\label{eq:tt-series-gen'}
f(X)=\sum_{\ell=0}^\infty (X- I_m\mu)^{\odot\ell}\,\Delta_R^\ell
f(\underset{\ell+1\ \rm{times}}{\underbrace{\mu,\ldots,\mu}}).
\end{equation}
\end{thm}
We have the following uniqueness theorem for convergent nc power
expansion of a nc function.
\begin{thm}\label{thm:ncps-unique}
Suppose that $\vecspace{V}$ is a vector space over $\mathbb{C}$,
 $\Omega\subseteq\ncspace{\vecspace{V}}$ is a finitely open nc set, and
$\vecspace{W}$ is a Banach space with an admissible system of
matrix norms over $\vecspace{W}$. Let
$f\colon\Omega\to\ncspace{\vecspace{W}}$ be a G-differentiable nc
function
 and  $Y\in\Omega_s$. Suppose that
\begin{equation*}
f(X)=\sum_{\ell=0}^\infty\Big(X-
\bigoplus_{\alpha=1}^mY\Big)^{\odot_s\ell}f_\ell,
\end{equation*}
for $X\in\Gamma_{sm}$, $m=1,2,\ldots$, where $\Gamma$ is a
finitely open subset of $\Omega$ which contains
$\bigoplus_{\alpha=1}^mY$ for every $m=1,2,\ldots$, and
$f_\ell\colon\mattuple{\vecspace{V}}{s}{\ell}\to\mat{\vecspace{W}}{s}$
is a $\ell$-linear mapping, $\ell=0,1,\ldots$. Then
$$f_\ell=\Delta_R^\ell f(\underset{\ell+1\
\rm{times}}{\underbrace{Y,\ldots,Y}}).$$
\end{thm}
\begin{proof}
Since $\Gamma$ is finitely open, for every $Z^1$, \ldots,
$Z^\ell\in\mat{\vecspace{V}}{s}$ there exists a real $r>0$ such
that $$X=\begin{bmatrix}
Y & rZ^1 & 0 & \ldots & 0\\
0 & \ddots & \ddots & \ddots & \vdots\\
\vdots & \ddots &\ddots & \ddots & 0\\
\vdots &        &\ddots & \ddots & rZ^\ell\\
0      & \ldots & \ldots & 0 & Y
\end{bmatrix}\in\Gamma_{\ell+1}.$$
As in the proof of Theorem \ref{thm:ncps-nilp-gen-unique}, we see
that $f_\ell(rZ^1,\ldots, rZ^\ell)$, and thus
$$f_\ell(Z^1,\ldots,Z^\ell)=r^{-\ell}f_\ell(rZ^1,\ldots,
rZ^\ell),$$ is uniquely determined by $f(X)$. On the other hand,
if $r>0$ is sufficiently small, we have
$X\in\Upsilon(\bigoplus_{\alpha=1}^{\ell+1}Y)$ and the series
\eqref{eq:tt-series-gen} converges---see Theorem \ref{thm:g-queen}
and Theorem \ref{thm:tt-series-gen}. Then we have by
\eqref{eq:bidiag} that
$$f_\ell(Z^1,\ldots,Z^\ell)=\Delta_R^\ell
f(Y,\ldots,Y)(Z^1,\ldots,Z^\ell).$$
\end{proof}

In the case where the space $\vecspace{V}$ is finite-dimensional,
so that we may assume without loss of generality that
$\vecspace{V}=\mathbb{C}^d$ (equipped with any admissible system
of matrix norms), we can expand each term of the convergent TT
series using higher order partial nc difference-differential
operators, as in Corollary \ref{cor:part_TT} and Theorem
\ref{thm:tt-power}. We notice that in this case a finitely open nc
set is the same as an open nc set, and a G-differentiable nc
function is the same as a F-differentiable nc function.
\begin{thm}\label{thm:tt-q-fin-dim}
Suppose that $\Omega\subseteq\ncspaced{\mathbb{C}}{d}$ is an open
nc set, and $\vecspace{W}$ is a Banach space with an admissible
system of matrix norms over $\vecspace{W}$. Let
$f\colon\Omega\to\ncspace{\vecspace{W}}$ be a nc function which is
locally bounded on slices. Then $f$ is analytic, and for any
$s\in\mathbb{N}$, $Y\in\Omega_s$, and arbitrary $m\in\mathbb{N}$,
\begin{multline}\label{eq:tt-pseudoseries}
f(X)= \sum_{\ell=0}^\infty\Big(X-
\bigoplus_{\alpha=1}^mY\Big)^{\odot_s\ell}\Delta_R^{\ell}f(\underset{\ell+1\ \rm{times}}{\underbrace{Y,\ldots,Y}})\\
=\sum_{\ell=0}^\infty\left(\sum_{|w|=\ell}\left(\bigoplus_{\alpha=1}^mA_{w,(0)}\otimes
\cdots\otimes \bigoplus_{\alpha=1}^mA_{w,(\ell)}\right)\star
\left(X-\bigoplus_{\alpha=1}^mY\right)^{[w]}\right),
\end{multline}
where the series converges absolutely and uniformly on compact
subsets of the set $\Upsilon\left(\bigoplus_{\alpha=1}^mY\right)$. Here
$\Upsilon\left(\bigoplus_{\alpha=1}^mY\right)$ is the maximal
complete circular set about $\bigoplus_{\alpha=1}^mY$ contained in
$\Omega_{sm}$, as in Theorem \ref{thm:tt-power-gen}(3),
$$A_{w,(0)}\otimes\cdots\otimes
A_{w,(\ell)}=\Delta_R^{w^\trans}f(\underset{\ell +1\
\rm{times}}{\underbrace{Y,\ldots,Y}}),$$ with the tensor product
interpretation for the values of $\Delta_R^{w^\trans}f$ (Remark
\ref{rem:tensor_values}), the sumless Sweedler notation
\eqref{eq:Sweedler}, and the pseudo-power notation
\eqref{eq:Ramamurti}.

In particular, if $\mu\in\Omega_1$ and $Y=\mu$, then
\eqref{eq:tt-pseudoseries} becomes
\begin{equation}\label{eq:tt-series}
f(X)=\sum_{\ell=0}^\infty\Big(\sum_{|w|=\ell}(X-
I_m\mu)^w\,\Delta_R^{w^\top}\!\!f(\underset{\ell+1\
\rm{times}}{\underbrace{\mu,\ldots,\mu}})\Big).
\end{equation}
\end{thm}
\begin{rem}\label{rem:real_queen}
The results of this section are of a mixed nature: noncommutative
and complex analytic. Therefore, they admit only partial analogues
in the real case. Let $\vecspace{V}$ be a vector space over
$\mathbb{R}$, let $\Omega\subseteq\ncspace{\vecspace{V}}$ be a
finitely open nc set, and let $\vecspace{W}$ be a Banach space
over $\mathbb{R}$ with an admissible system of matrix norms over
$\vecspace{W}$. (All the notions above are defined exactly as in
the complex case.) If a nc function
$f\colon\Omega\to\ncspace{\vecspace{W}}$ is locally bounded on
slices then, similar to Theorem \ref{thm:g-queen}(1), $f$ is
G-differentiable. Moreover, similar to Theorem
\ref{thm:g-queen}(2), $f(Y+tZ)$ is infinitely many times
differentiable at $0$ as a function of $t\in\mathbb{R}$ for every
$n\in\mathbb{N}$, $Y\in\Omega_n$, and $Z\in\mat{\vecspace{V}}{n}$,
and
$$\frac{1}{N!}\frac{d^N}{dt^N}f(Y+tZ)\Big|_{t=0}=\Delta_R^Nf(Y,\ldots,Y)(Z,\ldots,Z),
\qquad N=1,2,\ldots,$$ holds.

It does not follow that $f(Y+tZ)$ is analytic at $0$ as a function of $t\in\mathbb{R}$, i.e., one
cannot guarantee
the convergence of the TT
series of $f$ as in \eqref{eq:tt-g-queen}. It also does not follow that $f$ is
infinitely many times differentiable  on $\mathcal{U}\cap\Omega_n$ as a function of several real
variables, for a
 finite-dimensional subspace $\mathcal{U}$ of $\mat{\vecspace{V}}{n}$.

If, in addition, $f$ is \emph{locally bounded in $\Omega$ on
affine finite-dimensional subspaces}, i.e., if for every
$n,M\in\mathbb{N}$, $Y\in\Omega_n$,
 and $Z_1$,
\ldots,
 $Z_M\in\mat{\vecspace{V}}{n}$ there exists
 $\epsilon>0$ such that $f(Y+t_1Z_1+\cdots +t_MZ_M)$ is bounded
on the set $\{t\in\mathbb{R}\colon\max_{1\le j\le
M}|t_j|<\epsilon\}$, then $f(Y+t_1Z_1+\cdots
 +t_MZ_M)$ is infinitely many times
differentiable at $0$ as a function of
$t_1,\ldots,t_M\in\mathbb{R}$, and the identity
\begin{multline}\label{eq:partial}
\frac{1}{k!}\frac{\partial^{|k|}}{\partial t_1^{k_1}\cdots
\partial t_M^{k_M}}f(Y+t_1Z_1+\cdots
 +t_MZ_M)\Big|_{t_1=\cdots
 =t_M=0}\\
 =\sum_{\pi}\Delta_R^{|k|}f(Y,\ldots,Y)(Z_{\pi(1)},\ldots,Z_{\pi(N)})
 \end{multline}
 holds. Here $k=(k_1,\ldots,k_M)$ is an arbitrary $M$-tuple of nonnegative
 integers, $$|k|:=k_1+\cdots +k_M,\quad k!:=k_1!\cdots k_M!,$$ and
  $\pi$ runs over the set of all permutations with repetitions of the set
   $\{1,\ldots, M\}$, where the element
  $j$ appears exactly $k_j$ times. The identity \eqref{eq:partial}
  is of course also true in the complex case, and in
 both real and complex cases its proof is an obvious modification
 of the proof of part (2) of Theorem \ref{thm:g-queen}.

 If  $f$ is \emph{real analytic on slices}, \index{real analytic on slices nc function} i.e., if $f(Y+tZ)$
 is analytic at $0$ as a function of $t\in\mathbb{R}$ for every
$n\in\mathbb{N}$, $Y\in\Omega_n$, and $Z\in\mat{\vecspace{V}}{n}$
then the TT series of $f$, as in \eqref{eq:tt-g-queen} or as in
\eqref{eq:tt-series-gen} (in particular,
\eqref{eq:tt-series-gen'}), converges absolutely in some
neighborhood of $0$ in every slice. In the case where
$\vecspace{V}=\mathbb{R}^d$, if $f$ is \emph{real analytic},
\index{real analytic nc function} i.e., $f|_{\Omega_n}$ is real
analytic for each $n\in\mathbb{N}$, one can write the TT series
expansions of $f$ about $Y\in\Omega_n$, as in
\eqref{eq:tt-pseudoseries} (in particular, \eqref{eq:tt-series}),
with the series converging absolutely and uniformly on compact
subsets of $\Upsilon(Y)=\{X\in\Omega_n\colon Y+t(X-Y)\in\Omega_n\
{\rm for\ all\ } t\in [-1,1]\}$.
\end{rem}

\section{Uniformly-open topology over an operator
space}\label{subsec:unif-open-top} Recall that  a Banach space
$\vecspace{W}$  over $\mathbb{C}$ equipped with an admissible
system of matrix norms such that \eqref{eq:dirsums-norms} and
\eqref{eq:simprod-norms} hold with
 $C_1=C_1'=C_2=1$ independent of $n$ and $m$, is an
\emph{operator space}; see \cite{ER,Pi,Pa}. Let $\vecspace{W}$ be
an operator space. For $Y\in\mat{\vecspace{W}}{s}$ and $r>0$,
define a \emph{nc ball centered at $Y$ of radius $r$} \index{nc
ball} as \index{$B_{\mathrm{nc}}(Y,r)$}
\begin{equation*}
B_{\mathrm{nc}}(Y,r)=\coprod_{m=1}^\infty
B\Big(\bigoplus_{\alpha=1}^mY,r\Big) =\coprod_{m=1}^\infty\Big\{
X\in\mat{\vecspace{W}}{sm}\colon \Big\|
X-\bigoplus_{\alpha=1}^mY\Big\|_{sm}<r\Big\}.
\end{equation*}
\begin{prop}\label{prop:nc_top}
Let $Y\in\mat{\vecspace{W}}{s}$ and $r>0$. For any $X\in
B_{\mathrm{nc}}(Y,r)$ there is a $\rho>0$ such that
$B_{\mathrm{nc}}(X,\rho)\subseteq B_{\mathrm{nc}}(Y,r)$. Hence, nc
balls form a basis for a topology on $\ncspace{\vecspace{W}}$.
This topology will be called the \emph{uniformly-open topology}.
\index{uniformly-open topology}
\end{prop}
\begin{proof}
Let $X\in B_{\rm nc}(Y,r)_{sm}=B(\bigoplus_{\alpha=1}^mY,r)$,
i.e., $\delta:=\| X-\bigoplus_{\alpha=1}^mY\|_{sm}<r$. Set
$\rho:=r-\delta$. Let $Z\in
B_{\mathrm{nc}}(X,\rho)_{smn}=B(\bigoplus_{\beta=1}^{n}X,\rho)$.
Then, by \eqref{eq:dirsums-norms} with $C_1'=1$, we have
\begin{multline*}
\Big\|Z-\bigoplus_{\alpha=1}^{mn}Y\Big\|_{smn}\le\Big\|Z-\bigoplus_{\beta=1}^{n}X\Big\|_{smn}+
\Big\|\bigoplus_{\beta=1}^{n}X-\bigoplus_{\alpha=1}^{mn}Y\Big\|_{smn}\\
=\Big\|Z-\bigoplus_{\beta=1}^{n}X\big\|_{smn}+
\Big\|\bigoplus_{\beta=1}^{n}\Big(X-\bigoplus_{\alpha=1}^{m}Y\Big)\Big\|_{smn}\\
\le\Big\|Z-\bigoplus_{\beta=1}^{n}X\Big\|_{smn}+\Big\|X-\bigoplus_{\alpha=1}^{m}Y\Big\|_{sm}
<\rho+\delta=r,
\end{multline*}
i.e., $Z\in B_{\mathrm{nc}}(Y,r)_{smn}$, which proves the first
statement.

Let $Y^1\in\mat{\vecspace{W}}{s_1}$,
$Y^2\in\mat{\vecspace{W}}{s_2}$, and $r_1,r_2>0$ be such that
$B_{\mathrm{nc}}(Y^1,r_1)\cap
B_{\mathrm{nc}}(Y^2,r_2)\neq\emptyset$. Let $m$ be the smallest
common multiple of $s_1$ and $s_2$ for which
$B_{\mathrm{nc}}(Y^1,r_1)_m\cap
B_{\mathrm{nc}}(Y^2,r_2)_m\neq\emptyset$. If $X\in
B_{\mathrm{nc}}(Y^1,r_1)_m\cap B_{\mathrm{nc}}(Y^2,r_2)_m$.
Setting
$$\rho:=\min\Big\{r_1-\Big\|X-\bigoplus_{\alpha=1}^{m/s_1}Y^1\Big\|_m,\
r_2-\Big\|X-\bigoplus_{\alpha=1}^{m/s_2}Y^2\Big\|_m\Big\}$$ and
using the same argument as in the preceding paragraph, we obtain
that
$$B_{\mathrm{nc}}(X,\rho)\subseteq B_{\mathrm{nc}}(Y^1,r_1)\cap
B_{\mathrm{nc}}(Y^2,r_2).$$ Thus, nc balls form a basis for a
topology on $\ncspace{\vecspace{W}}$.
\end{proof}
\begin{rem}\label{rem:unif-mult-cont}
Notice that multiplication by a scalar as a mapping
$\mathbb{C}\times\ncspace{\vecspace{W}}\to\ncspace{\vecspace{W}}$
is continuous in the uniformly-open topology .
\end{rem}
Open sets in the uniformly-open topology on
$\ncspace{\vecspace{W}}$ will be called \emph{uniformly open}.
\index{uniformly open set} The uniformly-open topology is never
Hausdorff: the points $\bigoplus_{\alpha=1}^mY$ and
$\bigoplus_{\beta=1}^{m'}Y$, $m\neq m'$, cannot be separated by nc
balls. However, the following Hausdorff-like property holds: if
$X\in\mat{\vecspace{W}}{n}$ and $X'\in\mat{\vecspace{W}}{n'}$ are
such that there does not exist $Y\in\ncspace{\vecspace{W}}$ for
which $X=\bigoplus_{\alpha=1}^mY$ and
$X'=\bigoplus_{\beta=1}^{m'}Y$ for some $m$ and $m'$, then there
exist $\delta$ and $\delta'$ such that
$B_{\mathrm{nc}}(X,\delta)\cap
B_{\mathrm{nc}}(X',\delta')=\emptyset$. One can choose $\delta$
and $\delta'$ such that
$\delta+\delta'<\Big\|\bigoplus_{\alpha=1}^MX-\bigoplus_{\beta=1}^{M'}X'\Big\|_{Mn},$
where $Mn=M'n'$ is the least common multiple of $n$ and $n'$. We
only have to show that the norm above is non-zero in this case.
\begin{prop}\label{prop:Hausdorff-like}
For $X\in\mat{\vecspace{W}}{n}$ and $X'\in\mat{\vecspace{W}}{n'}$,
one has
$$\Big\|\bigoplus_{\alpha=1}^MX-\bigoplus_{\beta=1}^{M'}X'\Big\|_{Mn}=0$$
if and only if there exists $Y\in\ncspace{\vecspace{W}}$ for which
$X=\bigoplus_{\alpha=1}^mY$ and $X'=\bigoplus_{\beta=1}^{m'}Y$.
\end{prop}
\begin{proof}
The ``if" part is obvious. Suppose
$$(\overline{X}:=)\bigoplus_{\alpha=1}^MX=\bigoplus_{\beta=1}^{M'}X'.$$
For the purpose of this proof, let us use the convention that the
rows and columns of a $N\times N$ matrix $Z$ over $\vecspace{W}$
are enumerated from $0$ to $N-1$. Define the diagonal shift  $S$
which acts on such matrices as follows:
$$(SZ)_{ij} = Z_{(i-1)\bmod N, (j-1)\bmod N}, \quad  i,j = 0, \ldots, N-1.$$
Clearly, the inverse shift is given by
$$ (S^{-1}Z)_{ij} = Z_{(i+1)\bmod N, (j+1)\bmod N}, \quad  i,j = 0, \ldots, N-1.$$
Clearly, we have $S^{\pm n}\overline{X} =\overline{X}$ and $S^{\pm
n'}\overline{X}=\overline{X}$, where $N=Mn$. Let $d={\rm
gcd}(n,n')$. Since there exist $k,k'\in \mathbb{Z}$ such that $ d=
kn + k'n'$ (see, e.g., \cite[Problem 1.1]{Vinograd}), we obtain
that $S^{d}\overline{X} = \overline{X}$. Therefore, all $d\times
d$ blocks on the main block diagonal of the matrix $\overline{X}$
are equal, say $Y$. Since every other $d\times d$-block NW-to-SE
diagonal of $\overline{X}$ has at least one zero block, it must
have all zero blocks. We conclude that
$\overline{X}=\bigoplus_{\gamma=1}^{Mn/d}Y$, and therefore
$X=\bigoplus_{\alpha=1}^mY$ and $X'=\bigoplus_{\beta=1}^{m'}Y$,
with $m=n/d$ and $m'=n'/d$. The proof is complete.
\end{proof}

Alternatively, the closure of $X\in\mat{\vecspace{W}}{n}$,
$\clos\{X\}$, consists of all $X'=\bigoplus_{\beta=1}^{m'}Y$,
$m'=1,2,\ldots$, where $Y\in\mat{\vecspace{W}}{s}$ is such that
$n=sm$ and $X=\bigoplus_{\alpha=1}^mY$, with $Y$ not representable
as a direct sum of matrices over $\vecspace{W}$ --- such $Y$ is
unique by Proposition \ref{prop:Hausdorff-like}; the quotient
topology on
$\widehat{\ncspace{\vecspace{W}}}:=\ncspace{\vecspace{W}}/\!\!\sim$
where $X\sim X'$ whenever $X=\bigoplus_{\alpha=1}^mY$ and
$X'=\bigoplus_{\beta=1}^{m'}Y$ for some
$Y\in\ncspace{\vecspace{W}}$ is Hausdorff.

Let $\module{M}$ be a module. For a set
$\Omega\subseteq\ncspace{\module{M}}$, the \emph{radical of
$\Omega$} \index{radical of a set} is the set
$$\rad\Omega:=\Big\{Y\in\ncspace{\module{M}}\colon
\bigoplus_{\alpha=1}^mY\in\Omega\ {\rm for\ some\ }
m\in\mathbb{N}\Big\}.$$ \index{$\rad\Omega$}We will say that
$\Omega\subseteq\ncspace{\module{M}}$ is a \emph{radical set}
\index{radical set} if $\rad\Omega=\Omega$.
\begin{prop}\label{prop:rad}
For a set $\Omega\subseteq\ncspace{\module{M}}$, one has
\begin{equation}\label{eq:rad-incl}
\widetilde{\rad\Omega}\subseteq\rad\widetilde{\Omega},
\end{equation}
 where
$\widetilde{\Omega}$ and $\widetilde{\rad\Omega}$ are defined by
\eqref{eq:sim-inv}. In particular, if $\Omega$ is similarity
invariant, then so is $\rad\Omega$. If $\Omega$ is furthermore a
nc set, then so is $\rad\Omega$. If $\Omega$ is furthermore right
admissible, then so is $\rad\Omega$.
\end{prop}
\begin{proof}
If $Y\in\widetilde{\rad\Omega}$, then $SYS^{-1}\in\rad\Omega$ for
every invertible matrix over $\ring$ of the same size as $Y$,
i.e.,
$$\bigoplus_{\alpha=1}^mSYS^{-1}=
\Big(\bigoplus_{\alpha=1}^mS\Big)\Big(\bigoplus_{\alpha=1}^mY\Big)\Big(\bigoplus_{\alpha=1}^mS\Big)^{-1}\in\Omega$$
for some $m\in\mathbb{N}$. Therefore,
$\bigoplus_{\alpha=1}^mY\in\widetilde{\Omega}$ and thus
$Y\in\rad\widetilde{\Omega}$. The inclusion \eqref{eq:rad-incl}
follows.

In particular, if $\widetilde{\Omega}=\Omega$, we obtain
$\widetilde{\rad\Omega}\subseteq\rad\Omega$. Together with the
obvious inclusion $\widetilde{\rad\Omega}\supseteq\rad\Omega$,
this implies the equality $\widetilde{\rad\Omega}=\rad\Omega$,
i.e., $\rad\Omega$ is similarity invariant.

Let $Y,Y'\in\rad\Omega$. Then $\bigoplus_{\alpha=1}^mY\in\Omega$
and $\bigoplus_{\alpha=1}^{m'}Y'\in\Omega$ for some $m$ and $m'$.
If $\Omega$ is a nc set, then $\bigoplus_{\alpha=1}^{{\rm lcm}
(m,m')}Y\oplus\bigoplus_{\beta=1}^{{\rm lcm}(mm')}Y'\in\Omega$.
Since $\Omega$ is similarity invariant, the permutation of blocks
gives that $\bigoplus_{\alpha=1}^{{\rm lcm}(m,m')}(Y\oplus
Y')\in\Omega$, and hence $Y\oplus Y'\in\rad\Omega$. We conclude
that $\rad\Omega$ is a nc set in this case.

Suppose that $\Omega$ is a right admissible similarity invariant
nc set. Let $Y\in(\rad\Omega)_n$, $Y'\in(\rad\Omega)_{n'}$,
$Z\in\rmat{\ring}{n}{n'}$. Then
$\bigoplus_{\alpha=1}^mY\in\Omega_{nm}$ and
$\bigoplus_{\alpha=1}^{m'}Y'\in\Omega_{n'm'}$ for some $m$ and
$m'$. Moreover, $\bigoplus_{\alpha=1}^{{\rm
lcm}(m,m')}Y\in\Omega_{{\rm lcm}(m,m')n}$ and
$\bigoplus_{\alpha=1}^{{\rm lcm}(m,'m')'}Y'\in\Omega_{{\rm
lcm}(m,m')n'}$. By Proposition \ref{prop:adm-env},
$$\begin{bmatrix}
\bigoplus\limits_{\alpha=1}^{{\rm lcm}(m,m')}Y &
\bigoplus\limits_{\alpha=1}^{{\rm lcm}(m,m')}Z \\
 0 &
\bigoplus\limits_{\alpha=1}^{{\rm lcm}(m,m')}Y'
\end{bmatrix}\in\Omega_{(n+n'){\rm lcm}(m,m')}.
$$
Since $\Omega$ is similarity invariant, the permutation of blocks
gives that $$\bigoplus_{\alpha=1}^{{\rm lcm}(m,m')}\begin{bmatrix}
Y & Z\\
0 & Y'
\end{bmatrix}\in\Omega_{(n+n'){\rm lcm}(m,m')}.$$
Therefore, $\begin{bmatrix}
Y & Z\\
0 & Y'
\end{bmatrix}\in\rad\Omega$. This implies that the nc set $\rad\Omega$ is
right admissible.
\end{proof}

We note that inclusion \eqref{eq:rad-incl} can be proper, even in
the case of nc sets. This example also shows that
$\widetilde{\Omega}$ is not necessarily radical, even if $\Omega$
is.
\begin{ex}\label{ex:proper-rad-incl}
Let $\Omega\subseteq\ncspace{\mathbb{C}}$ be a nc set consisting
of matrices $\bigoplus_{\alpha=1}^mY$, $m=1,2,\ldots$, with
$Y=\diag[0,1,1,0]\in\mat{\mathbb{C}}{4}$. Clearly, the set
$\Omega$ is radical. Its similarity invariant envelope,
$\widetilde{\Omega}$, consists of diagonalizable square matrices
over $\mathbb{C}$ with eigenvalues $0$ and $1$ of the same even
multiplicity. In particular, $\widetilde{\Omega}$ contains the
matrix $\diag[0,1,0,1]\in\mat{\mathbb{C}}{4}$, so that the matrix
$\diag[0,1]\in\mat{\mathbb{C}}{2}$ belongs to
$\rad\widetilde{\Omega}$. Thus,
$\widetilde{\rad\Omega}=\widetilde{\Omega}$ is a proper subset of
$\rad\widetilde{\Omega}$.
\end{ex}

We also note that $\rad\Omega$ is not necessarily a nc set, even
if $\Omega$ is.
\begin{ex}\label{eq:rad-not-nc}
Let $\Omega\subseteq\ncspace{\mathbb{C}}$ be a nc set consisting
of matrices  $0_{2\times 2}, I_2 \in\Omega_2$ and all their direct
sums, in any possible order. Then $1\times 1$ matrices $0$ and $1$
belong to $\rad\Omega$, however their direct sum, $0\oplus
1=\diag[0,1]$ does not.
\end{ex}

Let $\vecspace{W}$ be an operator space. We now define another
topology on $\ncspace{\vecspace{W}}$ by means of the pseudometric
\index{pseudometric}
$\tau\colon\ncspace{\vecspace{W}}\times\ncspace{\vecspace{W}}\to\mathbb{R}_+$:
\begin{equation}\label{eq:tau}
\tau(X,X'):=\Big\|\bigoplus_{\alpha=1}^{{\rm
lcm}(n,n')/n}X-\bigoplus_{\beta=1}^{{\rm
lcm}(n,n')/n'}X'\Big\|_{{\rm lcm}(n,n')}\end{equation} for every
$n,n'\in\mathbb{N}$, $X\in\mat{\vecspace{W}}{n}$, and
$X'\in\mat{\vecspace{W}}{n'}$. We will call this topology a
\emph{$\tau$-topology}, \index{$\tau$-topology} and the
corresponding open sets \emph{$\tau$-open}. \index{$\tau$-open}
\emph{$\tau$-balls}\index{$\tau$-balls} centered at points
$Y\in\ncspace{\vecspace{W}}$ of radii $\delta>0$,
\index{$B_\tau(Y,\delta)$}
$$B_\tau(Y,\delta):=\{X\in\ncspace{\vecspace{W}}\colon\tau(X,Y)<\delta\},$$
form a base for the $\tau$-topology. Notice that $\tau$-balls, and
thus $\tau$-open sets, are radical.
\begin{prop}\label{prop:rad-properties} Let $\vecspace{W}$ be an
operator space.
\begin{enumerate}
    \item For every $s\in\mathbb{N}$,
 $Y\in\mat{\vecspace{W}}{s}$, and $\delta>0$, one has
\begin{equation}\label{eq:nc-and-tau-balls}
B_{\rm
nc}(Y,\delta)=B_{\tau}(Y,\delta)\cap\coprod_{m=1}^\infty\mat{\vecspace{W}}{sm}.
\end{equation}
 \item If $\Omega\subseteq\ncspace{\vecspace{W}}$ is a
 uniformly-open set, then $\rad\,\Omega$ is the smallest
 $\tau$-open set that contains $\Omega$.
    \item $\Omega\subseteq\ncspace{\vecspace{W}}$ is a $\tau$-open
    set if and only if $\Omega$ is uniformly-open and radical.
\end{enumerate}
\end{prop}
\begin{proof}
(1) follows from the definition of $\tau$-balls.

(2) It follows from \eqref{eq:nc-and-tau-balls} that $\rad\,
B_{\rm nc}(Y,\delta)=B_\tau(Y,\delta)$. Then
$$\rad\,\Omega=\rad\,\bigcup_{Y\in\Omega}B_{\rm
nc}(Y,\delta_Y)=\bigcup_{Y\in\Omega}\rad\,B_{\rm
nc}(Y,\delta_Y)=\bigcup_{Y\in\Omega}B_\tau(Y,\delta_Y),$$ where
$\delta_Y$ is any positive real number such that $B_{\rm
nc}(Y,\delta_Y)\subseteq\Omega$. Therefore, $\rad\,\Omega$ is
$\tau$-open. Since any $\tau$-open set is radical, if such a set
contains $\Omega$, then it must contain $\rad\,\Omega$ as well. We
conclude that $\rad\,\Omega$ is the smallest $\tau$-open set that
contains $\Omega$.

(3) Part (2) implies that if $\Omega$ is uniformly-open and
radical, i.e., $\rad\,\Omega=\Omega$, then $\Omega$ is
$\tau$-open. The converse is obvious.
\end{proof}
By part (1) of Proposition \ref{prop:rad-properties} that the
$\tau$-topology on $\ncspace{\vecspace{W}}$ is weaker than the
uniformly-open topology. However, two points of
$\ncspace{\vecspace{W}}$ cannot be separated by two nc balls if
and only if they cannot be separated by two $\tau$-balls. Indeed,
 Proposition \ref{prop:Hausdorff-like} says
that $X\sim X'$ if and only if $\tau(X,X')=0$. Therefore, the
quotient spaces of $\ncspace{\vecspace{W}}$ with respect to the
equivalence $\sim$ and the $\tau$-equivalence coincide as sets.
Since the pre-image of any set in the quotient set
$\widehat{\ncspace{\vecspace{W}}}$ under the quotient map is
radical, the corresponding quotient topologies on
$\widehat{\ncspace{\vecspace{W}}}$ coincide. In fact,
 $\widehat{\ncspace{\vecspace{W}}}$ is a metric
space with respect to the quotient metric $\widehat{\tau}$ which
is defined by
$\widehat{\tau}(\widehat{X},\widehat{X'}):=\tau(X,X')$ for any
representatives $X$ and $X'$ of the cosets $\widehat{X}$ and
$\widehat{X'}$, respectively (the definition is correct due to the
properties of the system of norms $\|\cdot\|_n$ on matrices over
an operator space). The quotient topology on
$\widehat{\ncspace{\vecspace{W}}}$ is the metric topology with
respect to the metric $\widehat{\tau}$.

\section{Uniformly analytic nc
functions}\label{subsec:u-analytic}
 Let $\vecspace{V}$ and $\vecspace{W}$ be Banach
spaces equipped with admissible systems of matrix norms. A linear
mapping $\phi\colon\vecspace{V}\to\vecspace{W}$ is called
\emph{completely bounded} \index{completely bounded linear
mapping} if \index{$\Vert\phi\Vert_{\cb}$}
$$\|\phi\|_{\cb}:=\sup_{n\in\mathbb{N}}\|\id_{\mat{\mathbb{C}}{n}}\otimes\phi
\|_{\mat{\vecspace{V}}{n}\to\mat{\vecspace{W}}{n}}<\infty,$$
\emph{completely contractive} \index{completely contractive linear
mapping} if $\|\phi\|_{\cb}\le 1$, and \emph{completely isometric}
\index{completely isometric linear mapping} if
$\id_{\mat{\mathbb{C}}{n}}\otimes\phi$ is an isometry for every
$n\in\mathbb{N}$.

\begin{prop}\label{prop:ucb-inj-proj}
If $\vecspace{W}$ is a Banach space equipped with an admissible
system of matrix norms $\|\cdot\|_n$ such that
\eqref{eq:dirsums-norms} and \eqref{eq:simprod-norms} hold with
$C_1$, $C_1'$, and $C_2$ independent of $n$ and $m$, then the
injections \eqref{eq:inj-bd} and the projections
\eqref{eq:proj-bd} are uniformly completely bounded, i.e., for
every $n\in\mathbb{N}$ and $i,j=1,\ldots,n$, the mappings
$\iota_{ij}$ and $\pi_{ij}$ are completely bounded. Moreover,
$$\|\iota_{ij}\|_{\cb}\le C_1^\prime C_2^2,\quad \|\pi_{ij}\|_{\cb}\le C_1C_2^2.$$
If $\vecspace{W}$ is
 an operator space, then all $\iota_{ij}$ are complete isometries and all
$\pi_{ij}$ are complete coisometries.
\end{prop}
\begin{proof}
It follows from the proof of Proposition \ref{prop:inj-proj} that
$\|\iota_{ij}^s\|\le C_1'C_2$ and $\|\pi_{ij}^s\|\le C_1C_2$ for
all $n,i,j,s$. We have, for any $n,i,j,s$ and
$W\in\mat{\vecspace{W}}{s}$, that
\begin{multline*}
\|(\id_{\mat{\mathbb{C}}{s}}\otimes\iota_{ij})(W)\|_{sn}=
\Big\|\Big(\id_{\mat{\mathbb{C}}{s}}\otimes\iota_{ij}\Big)
\Big(\sum_{a,b=1}^sE_{ab}
w_{ab}\Big)\Big\|_{sn}\\
=\Big\|\sum_{a,b=1}^sE_{ab}\otimes
\iota_{ij}(w_{ab})\Big\|_{sn}=\Big\|\sum_{a,b=1}^sE_{ab}\otimes
(E_{ij}w_{ab})\Big\|_{sn}\\
=\Big\|\sum_{a,b=1}^s(E_{ab}\otimes
E_{ij})w_{ab}\Big\|_{sn}=\Big\|\sum_{a,b=1}^sP(s,n)(E_{ij}\otimes
E_{ab})P(s,n)^{\mtrans}w_{ab}\Big\|_{sn}\\
=\Big\|\sum_{a,b=1}^sP(s,n)(E_{ij}\otimes E_{ab}
w_{ab})P(s,n)^{\mtrans}\Big\|_{sn}=\|P(s,n)(
E_{ij}\otimes W)P(s,n)^{\mtrans}\|_{sn}\\
\le C_2\| E_{ij}\otimes W\|_{sn}=C_2\|\iota_{ij}^s(W)\|_{sn}\le
C_2\|\iota_{ij}^s\|\|W\|_s.
\end{multline*}
Here we use the $sn\times sn$ permutation matrices
$P(s,n)=[E_{ij}^T]_{i=1,\ldots,s;\,j=1,\ldots,n}$ which allow us
to change the order of factors in tensor products:
\begin{equation}\label{eq:perm}
A\otimes B=P(r,p) (B\otimes A)P(m,q)^{\mtrans}
\end{equation}
for any $A\in\rmat{\mathbb{C}}{r}{m}$ and
$B\in\rmat{\mathbb{C}}{p}{q}$. See \cite[Pages 259--261]{HJ}.  We
conclude from the calculation above that
$$\|\iota_{ij}\|_{\cb}\le C_2\sup_{s\in\mathbb{N}}\|\iota_{ij}^s\|\le
C_1'C_2^2.$$

Suppose now that
$$W=[([w_{i'j'}]_{i',j'=1,\ldots,n})_{ab}]_{a,b=1,\ldots,s}\in
\mat{\left(\mat{\vecspace{W}}{n}\right)}{s}\cong
\mat{\mathbb{C}}{s}\otimes \mat{\vecspace{W}}{n}.$$ Then
\begin{multline*}
\widetilde{W}:=P(n,s)WP(n,s)^{\mtrans}=
[([\widetilde{w}_{ab}]_{a,b=1,\ldots,s})_{i'j'}]_{i',j'=1,\ldots,n}\\
\in\mat{\left(\mat{\vecspace{W}}{s}\right)}{n}\cong
\mat{\mathbb{C}}{n}\otimes \mat{\vecspace{W}}{s},
\end{multline*}
where
$$(\widetilde{w}_{ab})_{i'j'}=(w_{i'j'})_{ab},\quad
a,b=1\ldots,s;\ i',j'=1,\ldots,n.$$ We have
\begin{multline*}
\|(\id_{\mat{\mathbb{C}}{s}}\otimes\pi_{ij})(W)\|_s=
\Big\|\Big(\id_{\mat{\mathbb{C}}{s}}\otimes\pi_{ij}\Big)
\Big(\sum_{a,b=1}^s\sum_{i',j'=1}^nE_{ab}\otimes
E_{i'j'}(w_{i'j'})_{ab}
\Big)\Big\|_s\\
=\Big\| \sum_{a,b=1}^sE_{ab}(w_{ij})_{ab} \Big\|_s=\Big\|
\sum_{a,b=1}^sE_{ab}(\widetilde{w}_{ab})_{ij}) \Big\|_s = \|
\widetilde{W}_{ij}\|_s
=\|\pi_{ij}^s(\widetilde{W})\|_s\\
\le \|\pi_{ij}^s\|\,\|\widetilde{W}\|_{sn}\le
C_2\|\pi_{ij}^s\|\,\|W\|_{sn}.
\end{multline*}
Therefore,
$$\|\pi_{ij}\|_{\cb}\le
C_2\sup_{s\in\mathbb{N}}\|\pi_{ij}^s\|\le C_1C_2^2.$$  In the case
of operator spaces, it follows that all the operators $\iota_{ij}$
and $\pi_{ij}$ are complete contractions. Moreover, since
$\pi_{ij}\iota_{ij}=I_{\vecspace{W}}$, we conclude that all
$\iota_{ij}$ are complete isometries, and all $\pi_{ij}$ are
complete coisometries.
\end{proof}

Let $\vecspace{V},\vecspace{W}$ be operator spaces, and let
$\Omega\subseteq\ncspace{\vecspace{V}}$ be a uniformly open nc
set. A nc function $f\colon\Omega\to\ncspace{W}$ is called
\emph{uniformly locally bounded} \index{uniformly locally bounded
nc function} if for any $s\in\mathbb{N}$ and $Y\in\Omega_s$ there
exists a $r>0$ such that $B_{\mathrm{nc}}(Y,r)\subseteq\Omega$ and
$f$ is bounded on $B_{\mathrm{nc}}(Y,r)$, i.e., there is a $M>0$
such that $\|f(X)\|_{sm}\le M$ for all $m\in\mathbb{N}$ and $X\in
B_{\mathrm{nc}}(Y,r)_{sm}$. A function
$f\colon\Omega\to\ncspace{W}$ is called \emph{uniformly analytic}
\index{uniformly analytic function in a nc space} if $f$ is
uniformly locally bounded and analytic, i.e., $f$ is uniformly
locally bounded and G-differentiable. Later on (see Example
\ref{ex:unif-normal_neq_unif}), we will present an example of
analytic (in fact, entire) nc function which is not uniformly
analytic in any neighborhood of $0_{1\times 1}$.
\begin{prop}\label{prop:cont-bdd}
Let $\Omega\subseteq\ncspace{\vecspace{V}}$ be a uniformly open nc
set. If a nc function $f\colon\Omega\to\ncspace{\vecspace{W}}$ is
continuous with respect to the uniformly-open topologies on
$\ncspace{\vecspace{V}}$ and $\ncspace{\vecspace{W}}$, then $f$ is
uniformly locally bounded.
\end{prop}
\begin{proof}
Given $s\in\mathbb{N}$ and $Y\in\Omega_s$, there exists a nc ball
$B_{\mathrm{nc}}(Y,r)\subseteq\Omega$ such that
$\|f(X)-f(\bigoplus_{\alpha=1}^mY)\|_{sm}<1$ for all
$m\in\mathbb{N}$ and $X\in B_{\mathrm{nc}}(Y,r)_{sm}$. Then
$$\|f(X)\|_{sm}<\Big\|f\Big(\bigoplus_{\alpha=1}^mY\Big)\Big\|_{sm}+1=
\Big\|\bigoplus_{\alpha=1}^mf(Y)\Big\|_{sm}+1\le \|f(Y)\|_s+1.$$
\end{proof}

Let $l\colon\vecspace{V}\to\vecspace{W}$ be a linear mapping of
operator spaces. We extend $l$ to a nc function
$\widetilde{l}\colon\ncspace{\vecspace{V}}\to\ncspace{\vecspace{W}}$
such that $l_n:=\widetilde{l}|_{\mat{\vecspace{V}}{n}}$ is the
linear mapping from $\mat{\vecspace{V}}{n}$ to
$\mat{\vecspace{W}}{n}$ defined by $l_n([v_{ij}])=[l(v_{ij})]$,
$n\in\mathbb{N}$, i.e., $l_n=\id_{\mat{\mathbb{C}}{n}}\otimes l$.
Then it is clear that $\widetilde{l}$ is uniformly locally bounded
if and only if $\widetilde{l}$ is continuous with respect to the
uniformly-open topologies if and only if $l$ is completely
bounded. We will see shortly (Corollary \ref{cor:nc-u-analytic})
that the first equivalence is true for arbitrary nc functions.

We will present two versions of the main result on the uniform
convergence of the TT series in the uniformly-open topology on two
 different types of domains. The proof of the first theorem (like that of Theorem
\ref{thm:f-queen}) is of a mixed nature: noncommutative and
complex analytic. The second theorem admits a purely
noncommutative proof; in particular, it admits an analogue in the
real case, see Remark \ref{rem:real_king} below.

Let $\vecspace{X}$ be a vector space over $\mathbb{C}$ and let
$Y\in\mat{\vecspace{X}}{s}$. A nc set
$\Upsilon_{\mathrm{nc}}\subseteq\ncspace{\vecspace{X}}$ is called
\emph{complete circular about $Y$} \index{complete circular set
about $Y$} if for every $m\in\mathbb{N}$ the set
$(\Upsilon_{\mathrm{nc}})_{sm}=\Upsilon_{\mathrm{nc}}\cap\mat{\vecspace{X}}{sm}$
is complete circular about $\bigoplus_{\alpha=1}^mY$ (see the
paragraph preceding Theorem \ref{thm:g-queen}). For any
$\epsilon>0$, we define
\begin{equation}\label{eq:ups-eps}
\Upsilon_{\mathrm{nc},\epsilon}:=\coprod_{m=1}^\infty\Big\{X\in(\Upsilon_{\mathrm{nc}})_{sm}\colon
\bigoplus_{\alpha=1}^mY+(1+\epsilon)\Big(X-\bigoplus_{\alpha=1}^mY\Big)
\in(\Upsilon_{\mathrm{nc}})_{sm} \Big\}.
\end{equation}
\index{$\Upsilon_{\mathrm{nc},\epsilon}$}It is easily seen that if $\Upsilon_{\mathrm{nc}}$ is a uniformly
open complete circular nc set about $Y$, then so is
$\Upsilon_{\mathrm{nc},\epsilon}$, and
$\bigcup_{\epsilon>0}\Upsilon_{\mathrm{nc},\epsilon}=\Upsilon_\mathrm{nc}$.
\begin{thm}\label{thm:king-conv}
Let a nc function $f\colon\Omega\to\ncspace{\vecspace{W}}$ be
uniformly locally bounded. Let $s\in\mathbb{N}$, $Y\in\Omega_s$,
and let $\Upsilon_{\mathrm{nc}}$ be a uniformly open complete
circular nc set about $Y$ such that
$\Upsilon_{\mathrm{nc}}\subseteq\Omega$ and $f$ is bounded
$\Upsilon_{\mathrm{nc}}$. Then, for every $\epsilon>0$,
$m\in\mathbb{N}$, and
$X\in(\Upsilon_{\mathrm{nc},\epsilon})_{sm}$,
\begin{equation}\label{eq:tt-king-absolute}
f(X)=\sum_{\ell=0}^\infty\Big(X-\bigoplus_{\alpha=1}^mY\Big)^{\odot_s\ell}\Delta_R^\ell
f \Big(\underset{\ell+1\
\rm{times}}{\underbrace{Y,\ldots,Y}}\Big),
\end{equation}
where the series converges absolutely and uniformly on
$\Upsilon_{\mathrm{nc},\epsilon}$. Moreover,
\begin{equation}\label{eq:tt-king-normal}
\sum_{\ell=0}^\infty\
\sup_{m\in\mathbb{N},\,X\in(\Upsilon_{\mathrm{nc},\epsilon})_{sm}}\Big\|
\Big(X-\bigoplus_{\alpha=1}^mY\Big)^{\odot_s\ell}\Delta_R^\ell
f\Big(\underset{\ell+1\
\rm{times}}{\underbrace{Y,\ldots,Y}}\Big)\Big\|_{sm}<\infty.
\end{equation}
\end{thm}
\begin{proof}
Let $s\in\mathbb{N}$ and $Y\in\Omega_s$ be fixed. Let
$\Upsilon_{\mathrm{nc}}$ be a uniformly open complete circular set
about $Y$ on which $f$ is bounded, say by $M$. Given $\epsilon>0$,
$m\in\mathbb{N}$, and
$X\in(\Upsilon_{\mathrm{nc},\epsilon})_{sm}$, as in the proof of
Theorem \ref{thm:f-queen} we have that
\begin{multline}\label{eq:delta-king-conv}
\Big\|
\Big(X-\bigoplus_{\alpha=1}^mY\Big)^{\odot_s\ell}\Delta_R^\ell
f\Big(\underset{\ell+1\
\rm{times}}{\underbrace{Y,\ldots,Y}}\Big)\Big\|_{sm}\\
=
\Big\|\Delta_R^\ell f\Big(\underset{\ell+1\
\rm{times}}{\underbrace{\bigoplus_{\alpha=1}^mY,\ldots,\bigoplus_{\alpha=1}^mY}}\Big)
\Big(\underset{\ell\
\rm{times}}{\underbrace{X-\bigoplus_{\alpha=1}^mY,\ldots,X-\bigoplus_{\alpha=1}^mY}}\Big)
\Big\|_{sm}\le \frac{M}{(1+\epsilon)^\ell},
\end{multline}
but in this case the constant $M$ is independent on $m$. Therefore
\eqref{eq:tt-king-normal} holds and the TT series in
\eqref{eq:tt-king-absolute} converges absolutely and uniformly on
$\Upsilon_{\mathrm{nc},\epsilon}$. The fact that the sum of the
series equals $f(X)$ follows from Theorem \ref{thm:g-queen}.
\end{proof}

Let $\vecspace{X}$ be a vector space over $\mathbb{C}$ and let
$Y\in\mat{\vecspace{X}}{s}$. A nc set
$\Upsilon_{\mathrm{nc}}\subseteq\ncspace{\vecspace{X}}$ is called
\emph{matrix circular about $Y$} \index{matrix circular set about
$Y$} if for every $m\in\mathbb{N}$, $X\in
(\Upsilon_{\mathrm{nc}})_{sm}$, and unitary matrices
$U,V\in\mat{\mathbb{C}}{m}$ we have
$$\bigoplus_{\alpha=1}^mY+(U\otimes
I_s)\Big(X-\bigoplus_{\alpha=1}^mY\Big)(V\otimes
I_s)\in(\Upsilon_{\mathrm{nc}})_{sm}.$$ For any $\epsilon>0$, we
define $\Upsilon_{\mathrm{nc},\epsilon}$ as in \eqref{eq:ups-eps}.
It is easily seen that if $\Upsilon_{\mathrm{nc}}$ is a uniformly
open matrix circular nc set about $Y$, then so is
$\Upsilon_{\mathrm{nc},\epsilon}$.
\begin{lem}\label{lem:hvacts}
Let $\Upsilon_{\mathrm{nc}}\subseteq \ncspace{\vecspace{X}}$ be a
matrix circular nc set about $Y\in\mat{\vecspace{X}}{s}$. Then,
for every $m,N\in\mathbb{N}$, $X\in(\Upsilon_{\mathrm{nc}})_{sm}$,
and $\zeta,\eta\in\mathbb{C}$ with $|\zeta|^2+|\eta|^2=1$,
\begin{equation}\label{eq:hvact_A}
A:=\begin{bmatrix}
\bigoplus\limits_{\alpha=1}^mY & X-\bigoplus\limits_{\alpha=1}^mY & 0 & \ldots & 0\\
0 & \ddots & \ddots & \ddots & \vdots\\
\vdots & \ddots  & \ddots & \ddots & 0 \\
\vdots &  & \ddots & \ddots & X-\bigoplus\limits_{\alpha=1}^mY\\
0 & \ldots & \ldots & 0 & \bigoplus\limits_{\alpha=1}^mY
 \end{bmatrix}\in(\Upsilon_{\mathrm{nc}})_{sm(N+1)}
 \end{equation}
and
{\small \begin{equation}\label{eq:hvact_B}
B:=\begin{bmatrix}
\bigoplus\limits_{\alpha=1}^mY & X-\bigoplus\limits_{\alpha=1}^mY & 0 & \ldots & 0\\
0 & \ddots & \ddots & \ddots & \vdots\\
\vdots & \ddots  & \ddots & X-\bigoplus\limits_{\alpha=1}^mY & 0 \\
\vdots &  & \ddots & \bigoplus\limits_{\alpha=1}^mY & \zeta
\left(X-\bigoplus\limits_{\alpha=1}^mY\right)\\
0 & \ldots & \ldots & 0 &
\bigoplus\limits_{\alpha=1}^mY+\eta\left(X-\bigoplus\limits_{\alpha=1}^mY\right)
 \end{bmatrix}\in(\Upsilon_{\mathrm{nc}})_{sm(N+2)}.
\end{equation}
}
\end{lem}
\begin{proof}
We observe that \begin{multline*}\left( \begin{bmatrix}
0 & \ldots & \ldots & 0 & I_m\\
I_m & \ddots &  & \vdots & 0\\
0 & \ddots  & \ddots & \vdots & \vdots \\
\vdots & \ddots & \ddots & 0 & \vdots\\
0 & \ldots & 0 & I_m & 0
 \end{bmatrix}\otimes I_s\right)\left(A-\bigoplus\limits_{\alpha=1}^{m(N+1)}Y\right)\\
 =\begin{bmatrix}
0 & 0 & \ldots & 0 \\
0 & X-\bigoplus\limits_{\alpha=1}^mY & \ddots & \vdots \\
\vdots & \ddots & \ddots  & 0\\
0 & \ldots & 0 & X-\bigoplus\limits_{\alpha=1}^mY
 \end{bmatrix}
 \end{multline*}
 and
\begin{multline*}
\left(\begin{bmatrix}
0 & \ldots & \ldots & 0 & -\eta I_m& \zeta I_m\\
I_m & \ddots &  & \vdots & 0 & 0\\
0 & \ddots  & \ddots & \vdots & \vdots & \vdots \\
\vdots & \ddots & \ddots & 0 & \vdots & \vdots\\
\vdots &  & \ddots & I_m & 0 & 0\\
0 & \ldots & \ldots & 0 & \overline{\zeta}I_m & \overline{\eta}I_m
 \end{bmatrix}\otimes I_s\right)\left(B-\bigoplus\limits_{\alpha=1}^{m(N+2)}Y\right)\\
 =\begin{bmatrix}
0 & 0 & \ldots & 0 \\
0 & X-\bigoplus\limits_{\alpha=1}^mY & \ddots & \vdots \\
\vdots & \ddots & \ddots  & 0\\
0 & \ldots & 0 & X-\bigoplus\limits_{\alpha=1}^mY
 \end{bmatrix},
 \end{multline*}
 with square matrices of size $sm(N+1)$ in the first equality and
 $sm(N+2)$ in the second equality. Hence the result follows.
\end{proof}
\begin{thm}\label{thm:king-conv-nc}
Let a nc function $f\colon\Omega\to\ncspace{\vecspace{W}}$ be
uniformly locally bounded. Let $s\in\mathbb{N}$, $Y\in\Omega_s$,
and let $\Upsilon_{\mathrm{nc}}$ be a uniformly open matrix
circular nc set about $Y$ that is contained in $\Omega$ and on
which $f$ is bounded. Then, for every $\epsilon>0$,
\eqref{eq:tt-king-absolute} is true, with the TT series
converging absolutely and uniformly, on
$\Upsilon_{\mathrm{nc},\epsilon}$ and \eqref{eq:tt-king-normal}
holds.
\end{thm}
\begin{proof}
 Let $\Upsilon_{\mathrm{nc}}$ be a uniformly open matrix circular nc set about
 $Y$, with the norm of
$f$ bounded on $\Upsilon_{\mathrm{nc}}$, say by $M>0$. For
$\epsilon>0$, $m\in\mathbb{N}$, and $X\in
(\Upsilon_{\mathrm{nc},\epsilon})_{sm}$, we have
{\small\begin{multline}\label{eq:delta-king-conv-nc}
\Big\|\Delta_R^\ell f\Big(\underset{\ell+1\
\rm{times}}{\underbrace{\bigoplus_{\alpha=1}^mY,\ldots,\bigoplus_{\alpha=1}^mY}}\Big)
\Big(\underset{\ell\
\rm{times}}{\underbrace{X-\bigoplus_{\alpha=1}^mY,\ldots,X-\bigoplus_{\alpha=1}^mY}}\Big)
\Big\|_{sm}\\
=\!\!(1+\epsilon)^{-\ell}\!\Big\|\Delta_R^\ell
f\Big(\underset{\ell+1\
\rm{times}}{\underbrace{\bigoplus_{\alpha=1}^mY,\ldots,\bigoplus_{\alpha=1}^mY}}\Big)
\Big(\underset{\ell\
\rm{times}}{\underbrace{(1+\epsilon)\Big(X-\bigoplus_{\alpha=1}^mY\Big),\ldots,
(1+\epsilon)\ \Big(X-\bigoplus_{\alpha=1}^mY\Big)}}\Big)\Big\|_{sm}\\
=\!\!(1+\epsilon)^{-\ell}\!\!\left\|\pi_{1,\ell+1}f\left(\begin{bmatrix}
\bigoplus\limits_{\alpha=1}^mY & (1+\epsilon)(X-\bigoplus\limits_{\alpha=1}^mY) & 0 & \ldots & 0\\
0 & \ddots & \ddots & \ddots & \vdots\\
\vdots & \ddots  & \ddots & \ddots & 0 \\
\vdots &  & \ddots & \ddots & (1+\epsilon)(X-\bigoplus\limits_{\alpha=1}^mY)\\
0 & \ldots & \ldots & 0 & \bigoplus\limits_{\alpha=1}^mY
 \end{bmatrix}\right)\right\|_{sm}\\
 \le \frac{M}{(1+\epsilon)^\ell},
\end{multline} }
which implies that \eqref{eq:tt-king-normal} holds and the TT
series in \eqref{eq:tt-king-absolute} converges absolutely and
uniformly on $\Upsilon_{\mathrm{nc},\epsilon}$. Here we used the
fact that
{\small \begin{equation}\label{eq:hvact_eps_A} \begin{bmatrix}
\bigoplus\limits_{\alpha=1}^mY & (1+\epsilon)(X-\bigoplus\limits_{\alpha=1}^mY) & 0 & \ldots & 0\\
0 & \ddots & \ddots & \ddots & \vdots\\
\vdots & \ddots  & \ddots & \ddots & 0 \\
\vdots &  & \ddots & \ddots & (1+\epsilon)(X-\bigoplus\limits_{\alpha=1}^mY)\\
0 & \ldots & \ldots & 0 & \bigoplus\limits_{\alpha=1}^mY
 \end{bmatrix}\in(\Upsilon_{\mathrm{nc}})_{sm(\ell+1)}
 \end{equation}
}
by  \eqref{eq:hvact_A}, with $X$ replaced by
$\bigoplus\limits_{\alpha=1}^mY+
(1+\epsilon)(X-\bigoplus\limits_{\alpha=1}^mY)$ and $N$ replaced
by $\ell$.

It remains to prove that the sum of the series equals $f(X)$,
i.e., that the remainder term in the TT formula (see \eqref{eq:TT}
and \eqref{eq:tt-power-gen}),
\begin{multline*} \Delta_R^{N+1}f\Big(\underset{N+1\
\rm{times}}{\underbrace{\bigoplus_{\alpha=1}^mY,\ldots,\bigoplus_{\alpha=1}^mY}},X\Big)
\Bigg(\underset{N+1\
\rm{times}}{\underbrace{\Big(X-\bigoplus_{\alpha=1}^mY\Big),\ldots,
\Big(X-\bigoplus_{\alpha=1}^mY\Big)}}\Bigg)\\
=\Big(X-\bigoplus_{\alpha=1}^mY\Big)^{\odot_s(N+1)}\Delta_R^{N+1}f\Big(\underset{N+1\
\rm{times}}{\underbrace{Y,\ldots,Y},X}\Big),
\end{multline*}
tends to $0$ as $N\to\infty$. We have {\small\begin{multline*}
\Big\|\Delta_R^{N+1}f\Big(\underset{N+1\
\rm{times}}{\underbrace{\bigoplus_{\alpha=1}^mY,\ldots,\bigoplus_{\alpha=1}^mY}},X\Big)
\Big(\underset{N+1\
\rm{times}}{\underbrace{\big(X-\bigoplus_{\alpha=1}^mY\big),\ldots,
\big(X-\bigoplus_{\alpha=1}^mY\big)}}\Big)\Big\|_{sm}\\
=\frac{1}{\delta(1+\epsilon)^N}
\Big\|\Delta_R^{N+1}f\Big(\underset{N+1\
\rm{times}}{\underbrace{\bigoplus_{\alpha=1}^mY,\ldots,\bigoplus_{\alpha=1}^mY}},X\Big)\Big. \hfill\\
\hfill\Big.\Big(\underset{N\
\rm{times}}{\underbrace{(1+\epsilon)\big(X-\bigoplus_{\alpha=1}^mY\big),\ldots,
(1+\epsilon)\big(X-\bigoplus_{\alpha=1}^mY\big)}},
\delta\big(X-\bigoplus_{\alpha=1}^mY\big)\Big)\Big\|_{sm}\\
=\frac{1}{\delta(1+\epsilon)^N}\hfill\\
\cdot\left\|\pi_{1,N+2}f\!\!\left(
\begin{bmatrix}
\bigoplus\limits_{\alpha=1}^mY & \!\!(1+\epsilon)(X-\bigoplus\limits_{\alpha=1}^mY) & 0 & \ldots & \!\! 0\\
0 & \ddots & \ddots & \ddots & \!\!\vdots\\
\vdots & \ddots  & \ddots & \!\!(1+\epsilon)(X-\bigoplus\limits_{\alpha=1}^mY) & \!\! 0 \\
\vdots &  & \ddots & \bigoplus\limits_{\alpha=1}^mY & \!\!\delta(X-\bigoplus\limits_{\alpha=1}^mY)\\
0 & \ldots & \ldots & 0 & \!\! X
 \end{bmatrix}\right)\right\|_{sm}\\
 \le \frac{M}{\delta(1+\epsilon)^N}\longrightarrow 0
 \end{multline*}
 }
 as $N\to \infty$. The last inequality holds with
 $\delta=\sqrt{\epsilon^2+2\epsilon}$; indeed, for this choice of
 $\delta$ we have
{\small\begin{equation}\label{eq:hvact_eps_B}
\begin{bmatrix}
\bigoplus\limits_{\alpha=1}^mY & (1+\epsilon)(X-\bigoplus\limits_{\alpha=1}^mY) & 0 & \ldots & 0\\
0 & \ddots & \ddots & \ddots & \vdots\\
\vdots & \ddots  & \ddots & (1+\epsilon)(X-\bigoplus\limits_{\alpha=1}^mY) & 0 \\
\vdots &  & \ddots & \bigoplus\limits_{\alpha=1}^mY & \delta(X-\bigoplus\limits_{\alpha=1}^mY)\\
0 & \ldots & \ldots & 0 & X
 \end{bmatrix}\in(\Upsilon_{\mathrm{nc}})_{sm(N+2)}
\end{equation}
} by  \eqref{eq:hvact_B}  with $X$ replaced by
$\bigoplus\limits_{\alpha=1}^mY+
(1+\epsilon)(X-\bigoplus\limits_{\alpha=1}^mY)$,
$\zeta=\frac{\delta}{1+\epsilon}$, $\eta=\frac{1}{1+\epsilon}$.
\end{proof}
\begin{rem}\label{rem:hvacts}
Notice that Theorem \ref{thm:king-conv-nc} holds under a weaker
assumption that $\Upsilon_{\mathrm{nc}}$ is a uniformly open nc
set satisfying \eqref{eq:hvact_A} and \eqref{eq:hvact_B}.
\end{rem}
\begin{rem}\label{rem:hvacts-eps}
Notice that in the proof of Theorem \ref{thm:king-conv-nc} we
actually established the following result. Let
$\Upsilon_{\mathrm{nc}}$ be a uniformly open nc set, which
contains $Y$, is contained in $\Omega$, and on which $f$ is
bounded. Then, for every $\epsilon>0$, \eqref{eq:tt-king-absolute}
holds, with the series converging absolutely and uniformly, on the
set of all $X\in\Omega$ such that
\begin{itemize}
\item \eqref{eq:hvact_eps_A} holds for all $\ell\in\mathbb{N}$;
\item there exists $\delta>0$ such that \eqref{eq:hvact_eps_B}
holds for all $N\in\mathbb{N}$.
\end{itemize}
This gives even weaker assumptions on a domain where the TT series
in \eqref{eq:tt-king-absolute} converges absolutely and uniformly.
\end{rem}

We notice an important special case of the two previous theorems
 for the uniform convergence on open nc balls --- clearly, nc
 balls are both complete circular and matrix circular.
\begin{cor}\label{cor:nc-balls}
Let a nc function $f\colon\Omega\to\ncspace{\vecspace{W}}$ be
uniformly locally bounded. For every $s\in\mathbb{N}$,
$Y\in\Omega_s$, let $\delta:=\sup\{r>0\colon f\ \mathrm{is\
bounded\ on\ } B_{\mathrm{nc}}(Y,r)\}$. Then
\eqref{eq:tt-king-absolute} holds, with the TT series converging
absolutely and uniformly, on every open nc ball
$B_{\mathrm{nc}}(Y,r)$ with $r<\delta$. Moreover,
\begin{equation}\label{eq:tt-king-normal'}
\sum_{\ell=0}^\infty\ \sup_{m\in\mathbb{N},\,X\in
B_{\mathrm{nc}}(Y,r)_{sm}}\Big\|
\Big(X-\bigoplus_{\alpha=1}^mY\Big)^{\odot_s\ell}\Delta_R^\ell
f\Big(\underset{\ell+1\
\rm{times}}{\underbrace{Y,\ldots,Y}}\Big)\Big\|_{sm}<\infty.
\end{equation}
\end{cor}
\begin{rem}\label{rem:king-tt-gen'}
In the case where $Y=\mu\in\Omega_1$, Theorems \ref{thm:king-conv}
and \ref{thm:king-conv-nc}, and Corollary \ref{cor:nc-balls}, give
results on the absolute and uniform convergence of the TT series
\eqref{eq:tt-series-gen'} centered at a scalar point.
\end{rem}
\begin{cor}\label{cor:nc-u-analytic}
Let $\Omega\subseteq\ncspace{\vecspace{V}}$ be a uniformly open nc
set. Then a nc function $f\colon\Omega\to\ncspace{\vecspace{W}}$
is uniformly locally bounded if and only if $f$ is continuous with
respect to the uniformly-open topologies on
$\ncspace{\vecspace{V}}$ and $\ncspace{\vecspace{W}}$ if and only
if $f$ is uniformly analytic.
\end{cor}
\begin{proof}
If $f$ is uniformly locally bounded, then $f$ is also locally
bounded on slices, thus by Theorem \ref{thm:g-queen} $f$ is
G-differentiable. Therefore $f$ is uniformly analytic. Taking into
account Proposition \ref{prop:cont-bdd}, it remains to show that
if $f$ is uniformly locally bounded then $f$ is continuous with
respect to the uniformly-open topologies.

Let the norm of $f$ be bounded by $M>0$ on a nc ball
$B_{\mathrm{nc}}(Y,r_1)$ centered at $Y\in\mat{\vecspace{V}}{s}$.
Then, using the estimate \eqref{eq:delta-king-conv} or
\eqref{eq:delta-king-conv-nc} with $\epsilon=\frac{r_1-r}{r}$ for
$r<r_1$, we obtain for $X\in B_{\mathrm{nc}}(Y,r)_{sm}$,
$m\in\mathbb{N}$, that
$$\Big\|f(X)-f\Big(\bigoplus_{\alpha=1}^mY\Big)\Big\|_{sm}\le
M\sum_{\ell=1}^\infty\Big(\frac{r}{r_1}\Big)^\ell=\frac{Mr}{r_1-r}.$$
Hence, for any given $\eta>0$ there exists  $r<r_1$ such that for
every $m\in\mathbb{N}$ and $X\in B_{\mathrm{nc}}(Y,r)_{sm}$ we
have
$$\Big\|f(X)-f\Big(\bigoplus_{\alpha=1}^mY\Big)\Big\|_{sm}<\eta.$$
In other words, $f$ is continuous at $Y$ with respect to the
uniformly-open topologies.
\end{proof}

We consider now the case where the space $\vecspace{V}$ is finite
dimensional, so that we can assume without loss of generality that
$\vecspace{V}=\mathbb{C}^d$ with some operator space structure.
Notice that any two operator space structures on $\mathbb{C}^d$
are completely isomorphic, i.e., the identity mapping of
$\mathbb{C}^d$ is completely bounded both ways (see, e.g.,
\cite[Proposition 1.10(iii)]{Pi}, thus the uniformly-open topology
on $\ncspaced{\mathbb{C}}{d}$ is uniquely determined. We can
expand each term of the convergent TT series (of polylinear forms)
using higher order nc difference-differential operators as in
\eqref{eq:part_TT}, \eqref{eq:tt-pseudoseries}, and
\eqref{eq:tt-series}, and consider the question of the convergence
along the free monoid $\free_d$. This is a stronger notion of
convergence, since we do not group together all the terms
corresponding to the words of a given length. To define this
convergence, we need to specify an order on $\free_d$; however, we
will mostly deal with absolute convergence along $\free_d$, so
that the choice of the order does not matter.

Let $\Upsilon_{\mathrm{nc}}\subseteq\ncspaced{\mathbb{C}}{d}$ be a
matrix circular nc set about $Y=(Y_1,\ldots,
Y_d)\in\mattuple{\mathbb{C}}{s}{d}$, and let
$\brho=(\rho_1,\ldots,\rho_d)$ be a $d$-tuple of positive real
numbers with $\sum_{j=1}^d\rho_j<1$. We define
\begin{align}\label{eq:ups-brho}
\Upsilon_{\mathrm{nc},\brho}  := &\coprod_{m=1}^\infty \!\!\Big\{
X\in(\Upsilon_{\mathrm{nc}})_{sm}\colon
\!\bigoplus_{\alpha=1}^mY+\frac{1}{\rho_j}\Big(0,\ldots,0,\underset{j-\mathrm{th\
place}}{\underbrace{X_j-\bigoplus_{\alpha=1}^mY_j}},0,\ldots,0\Big)
\in(\Upsilon_{\mathrm{nc}})_{sm}\\
&  {\rm for\ all}\ j=1,\ldots,d\Big\}.
\end{align}
\index{$\Upsilon_{\mathrm{nc},\brho}$}If $\Upsilon_{\mathrm{nc}}$ is a uniformly open matrix circular nc
set about $Y$, then so is $\Upsilon_{\mathrm{nc},\brho}$; the fact
that $\Upsilon_{\mathrm{nc},\brho}$ is uniformly open follows from
the fact that the linear mapping $Z\mapsto Z_je_j$ on
$\mattuple{\mathbb{C}}{s}{d}$ is completely bounded as a
finite-rank mapping of operator spaces, see \cite[Proposition
1.10(iii)]{Pi}.
\begin{thm}\label{thm:king-semigr}
Suppose that $\Omega\subseteq \ncspaced{\mathbb{C}}{d}$ is a
uniformly open nc set, $\vecspace{W}$ is an operator space,
$f\colon\Omega\to\ncspace{\vecspace{W}}$ is a nc function which is
uniformly locally bounded, $s\in\mathbb{N}$, $Y\in\Omega_s$. Let
$\Upsilon_{\mathrm{nc}}$ be a uniformly open matrix circular nc
set about $Y$ that is contained in $\Omega$ and on which $f$ is
bounded. Then, for every
$\brho=(\rho_1,\ldots,\rho_d)\in(\mathbb{R}_+\setminus\{0\})^d$
with $\sum_{j=1}^d\rho_j<1$, $m\in\mathbb{N}$, and
$X\in(\Upsilon_{\mathrm{nc},\brho})_{sm}$,
\begin{multline}\label{eq:king-semigr}
f(X)=\sum_{w\in\free_d}\Big(X-\bigoplus_{\alpha=1}^mY\Big)\Delta_R^{w^\top}
f(\underset{|w|+1\ {\rm times}}{\underbrace{Y,\ldots,Y}})\\
=
\sum_{w\in\free_d}\Bigg(\Big(\bigoplus_{\alpha=1}^mA_{w,(0)}\Big)\otimes
\Big(\bigoplus_{\alpha=1}^mA_{w,(1)}\Big)\otimes\cdots\otimes
\Big(\bigoplus_{\alpha=1}^mA_{w,(|w|)}\Big)\Bigg)\star
\Big(X-\bigoplus_{\alpha=1}^mY\Big)^{[w]},
\end{multline}
where the series converges absolutely and uniformly on
$\Upsilon_{\mathrm{nc},\brho}$. Moreover,
\begin{multline}\label{eq:king-semigr-normal}
\sum_{w\in\free_d}\sup_{m\in\mathbb{N},\,X\in(\Upsilon_{\mathrm{nc},\brho})_{sm}}
\Bigg\|\Bigg(\Big(\bigoplus_{\alpha=1}^mA_{w,(0)}\Big)\otimes
\Big(\bigoplus_{\alpha=1}^mA_{w,(1)}\Big)\otimes\cdots\otimes
\Big(\bigoplus_{\alpha=1}^mA_{w,(|w|)}\Big)\Bigg)\\
\star
\Big(X-\bigoplus_{\alpha=1}^mY\Big)^{[w]}\Bigg\|_{sm}<\infty.
\end{multline}
 Here
$A_{w,(0)}\otimes A_{w,(1)}\otimes\cdots\otimes
A_{w,(|w|)}=\Delta_R^{w^\top}f(\underset{|w|+1\
\rm{times}}{\underbrace{Y,\ldots,Y}}),$ with the tensor product
interpretation for the values of $\Delta_R^{w^\trans}f$ (Remark
\ref{rem:tensor_values}), the sumless Sweedler notation
\eqref{eq:Sweedler}, and the pseudo-power notation
\eqref{eq:Ramamurti}.
\end{thm}
\begin{proof}
Let $\Upsilon_{\mathrm{nc}}$ be a uniformly open matrix circular
nc set about
 $Y$ such that the norm of
$f$ is bounded, say by $M>0$,  on $\Upsilon_{\mathrm{nc}}$. For
$\brho=(\rho_1,\ldots,\rho_d)\!\in\!(\mathbb{R}_+\setminus\{0\})^d$
with $\sum_{j=1}^d\rho_j<1$, $m\in\mathbb{N}$, and
$X\in(\Upsilon_{\mathrm{nc},\brho})_{sm}$, we obtain, using
\eqref{eq:part_bidiag}, that {\small\begin{multline*}
\sum_{|w|=\ell}\Big\|\Delta_R^{w^\top}f\Big(\underset{\ell+1\
\rm{times}}{\underbrace{\bigoplus_{\alpha=1}^mY,\ldots,\bigoplus_{\alpha=1}^mY}}\Big)\Big(
X_{i_1}-\bigoplus_{\alpha=1}^mY_{i_1},\ldots,X_{i_\ell}-\bigoplus_{\alpha=1}^mY_{i_\ell}
\Big)\Big\|_{sm} \\
=\sum_{|w|=\ell}\brho^w\Big\|\Delta_R^{w^\top}f\Big(\underset{\ell+1\
\rm{times}}{\underbrace{\bigoplus_{\alpha=1}^mY,\ldots,\bigoplus_{\alpha=1}^mY}}\Big)\Big(
\frac{1}{\rho_{i_1}}\Big(X_{i_1}-\bigoplus_{\alpha=1}^mY_{i_1}\Big),
\ldots,\frac{1}{\rho_{i_\ell}}\Big(X_{i_\ell}-\bigoplus_{\alpha=1}^mY_{i_\ell}\Big)\Big)
\Big\|_{sm} \\
=\sum_{|w|=\ell}\brho^w\\
\cdot\left\|\pi_{1,\ell+1}f\left(\begin{bmatrix}
\bigoplus\limits_{\alpha=1}^mY & \frac{1}{\rho_{i_1}}(X_{i_1}-
\bigoplus\limits_{\alpha=1}^mY_{i_1})e_{i_1}
& 0 & \ldots & 0\\
0 & \ddots & \ddots & \ddots & \vdots\\
\vdots & \ddots  & \ddots & \ddots & 0 \\
\vdots &  & \ddots & \ddots & \frac{1}{\rho_{i_\ell}}(X_{i_\ell}-
\bigoplus\limits_{\alpha=1}^mY_{i_\ell})e_{i_\ell}\\
0 & \ldots & \ldots & 0 & \bigoplus\limits_{\alpha=1}^mY
 \end{bmatrix}\right)\right\|_{sm} \\
 \le M\sum_{|w|=\ell}\brho^w=M\Big(\sum_{j=1}^d\rho_j\Big)^\ell,
\end{multline*}
}which implies that \eqref{eq:king-semigr-normal} holds and the
TT series in \eqref{eq:king-semigr} converges absolutely and
uniformly on $\Upsilon_{\mathrm{nc},\brho}$. Here we used the fact
that
\begin{equation}\label{eq:hvact_rho_A}
\begin{bmatrix}
\bigoplus\limits_{\alpha=1}^mY & \frac{1}{\rho_{i_1}}(X_{i_1}-
\bigoplus\limits_{\alpha=1}^mY_{i_1})e_{i_1} &
 0 & \ldots & 0\\
0 & \ddots & \ddots & \ddots & \vdots\\
\vdots & \ddots  & \ddots & \ddots & 0 \\
\vdots &  & \ddots & \ddots & \frac{1}{\rho_{i_\ell}}(X_{i_\ell}-
\bigoplus\limits_{\alpha=1}^mY_{i_\ell})e_{i_\ell}\\
0 & \ldots & \ldots & 0 & \bigoplus\limits_{\alpha=1}^mY
 \end{bmatrix}\in(\Upsilon_{\mathrm{nc}})_{sm(\ell+1)},
 \end{equation}
which is obtained using the same argument as in the proof of
\eqref{eq:hvact_A} in Lemma \ref{lem:hvacts}.

In order to show that the sum  of the (absolutely convergent) TT
series in \eqref{eq:king-semigr} equals $f(X)$, we observe that
its sum must coincide with the sum of the TT series in
\eqref{eq:tt-king-absolute} which is obtained by grouping together
the terms corresponding to the words of the same length. But the
remark made in the paragraph preceding Lemma \ref{lem:hvacts},
together with Theorem \ref{thm:king-conv-nc}, means that the TT
series \eqref{eq:tt-king-absolute} converges to $f(X)$ everywhere
in $\Upsilon_{\mathrm{nc}}$.
\end{proof}
\begin{rem}\label{rem:hvacts-rho}
Notice that in the proof of Theorem \ref{thm:king-semigr} we
actually established the following result. Let
$\Upsilon_{\mathrm{nc}}$ be a uniformly open nc set, which
contains $Y$, is contained in $\Omega$, and on which $f$ is
bounded. Then, for every
$\brho=(\rho_1,\ldots,\rho_d)\in(\mathbb{R}_+\setminus\{0\})^d$
with $\sum_{j=1}^d\rho_j<1$,  \eqref{eq:king-semigr} holds, with
the series converging absolutely and uniformly, on the set of all
$X\in\Omega$ such that \eqref{eq:hvact_rho_A} holds for all
$\ell\in\mathbb{N}$ and such that the TT series in
\eqref{eq:tt-king-absolute}, which is obtained by grouping
together the terms corresponding to the words of the same length,
converges to $f(X)$. This gives a weaker assumption on a domain
where the TT series in \eqref{eq:king-semigr} converges absolutely
and uniformly.
\end{rem}

Unlike in the setting of Theorem \ref{thm:king-conv} or of Theorem
\ref{thm:king-conv-nc}, it may happen that for a uniformly open
matrix circular nc set $\Upsilon_{\mathrm{nc}}$ we have
$\bigcup_{\brho}\Upsilon_{\mathrm{nc},\brho}\subsetneq\Upsilon_{\mathrm{nc}}$,
i.e., we cannot guarantee the absolute convergence along $\free_d$
of the TT series in \eqref{eq:king-semigr} on all of
$\Upsilon_{\mathrm{nc}}$. We elucidate the situation in the
special case of open nc balls.
\begin{cor}\label{cor:nc-balls-semigr}
Let a nc function $f\colon\Omega\to\ncspace{\vecspace{W}}$ be
uniformly locally bounded. For every $s\in\mathbb{N}$,
$Y\in\Omega_s$, let $\delta:=\sup\{r>0\colon f\ \mathrm{is\
bounded\ on\ } B_{\mathrm{nc}}(Y,r)\}$. Then
\eqref{eq:king-semigr} holds, with the TT series converging
absolutely and uniformly, on every \emph{open nc diamond about
$Y$} \index{nc diamond} \index{$\diamondsuit_{\mathrm{nc}}(Y,r)$}
\begin{equation}\label{eq:diamond}
\diamondsuit_{\mathrm{nc}}(Y,r):=\coprod_{m=1}^\infty\Big\{X\in
\Omega_{sm}\colon
\sum_{j=1}^d\|e_j\|_1\,\Big\|X_j-\bigoplus_{\alpha=1}^mY_j\Big\|<r\Big\}
\end{equation}
with $r<\delta$. Moreover,
\begin{multline}\label{eq:nc-balls-normal}
\sum_{w\in\free_d}\sup_{m\in\mathbb{N},\,X\in{\diamondsuit_{\mathrm{nc}}(Y,r)}_{sm}}
\Bigg\|\Bigg(\Big(\bigoplus_{\alpha=1}^mA_{w,(0)}\Big)\otimes
\Big(\bigoplus_{\alpha=1}^mA_{w,(1)}\Big)\otimes\cdots\otimes
\Big(\bigoplus_{\alpha=1}^mA_{w,(|w|)}\Big)\Bigg)\\
\star
\Big(X-\bigoplus_{\alpha=1}^mY\Big)^{[w]}\Bigg\|_{sm}<\infty.
\end{multline}
\end{cor}
\begin{proof}
We first observe that for any $n\in\mathbb{N}$,
$A\in\mat{\mathbb{C}}{n}$, and $j=1$, \ldots, $d$, we have
$\|Ae_j\|_n=\|e_j\|_1\|A\|$, since $\spa\{ e_j\}$ is a
one-dimensional subspace in $\mathbb{C}^d$ and therefore there is
only one operator space structure on $\spa{e_j}$, up to a constant
factor $\|e_j\|_1$; see \cite[Proposition 1.10(ii)]{Pi}.

Let $r$ be given such that $0<r<\delta$. Let $r'$ be a real number
with $r<r'<\delta$, and let the norm of $f$ be bounded by $M>0$ on
$B_{\mathrm{nc}}(Y,r')$. Given $X\in
{\diamondsuit_{\mathrm{nc}}(Y,r)}_{sm}$, one can find
$\brho=(\rho_1,\ldots,\rho_d)$ with
$$\rho_j>\frac{1}{r'}\|e_j\|_1\Big\|X_j-\bigoplus_{\alpha=1}^mY\Big\|,\quad j=1,\ldots,d,$$
such that $\sum_{j=1}^d\rho_j<\frac{r}{r'}$. Since open nc balls
are uniformly open and matrix circular nc sets, so is
$\Upsilon_{\mathrm{nc}}:=B_{\mathrm{nc}}(Y,r')$. For $X$ and
$\brho$ as above, we observe that
$X\in\Upsilon_{\mathrm{nc},\brho}$. Therefore,
\eqref{eq:king-semigr} holds with the series absolutely
convergent. Using the estimate in the proof of Theorem
\ref{thm:king-semigr} we obtain that
$$\sum_{|w|=\ell}\Big\|\Delta_R^{w^\top}f\Big(\underset{\ell+1\
\rm{times}}{\underbrace{\bigoplus_{\alpha=1}^mY,\ldots,\bigoplus_{\alpha=1}^mY}}\Big)\Big(
X_{i_1}-\bigoplus_{\alpha=1}^mY_{i_1},\ldots,X_{i_\ell}-
\bigoplus_{\alpha=1}^mY_{i_\ell}\Big)\Big\|_{sm}\le
M\Big(\frac{r}{r'}\Big)^\ell.$$ Since the right-hand side is
independent of $\brho$, $m$, and $X$, \eqref{eq:nc-balls-normal}
follows.
\end{proof}
\begin{rem}\label{rem:abs-conv-diamond}
Since
$\bigcup_{0<r<\delta}\diamondsuit_{\mathrm{nc}}(Y,r)=\diamondsuit_{\mathrm{nc}}(Y,\delta)$,
the TT series in \eqref{eq:king-semigr} converges absolutely along
$\free_d$ on the nc diamond
$\diamondsuit_{\mathrm{nc}}(Y,\delta)$.
\end{rem}
\begin{rem}\label{rem:exhaust}
It is clear that open nc diamonds are uniformly open and that, for
all $Y$ and $r$, $\diamondsuit_{\mathrm{nc}}(Y,r)\subseteq
B_{\mathrm{nc}}(Y,r)$. This inclusion is always proper. The
equality would occur if and only if for all $m\in\mathbb{N}$ and
$X\in\mattuple{\mathbb{C}}{sm}{d}$ one has
$\|X\|_{sm}=\sum_{j=1}^d\|e_j\|_1\,\|X_j\|$. Indeed, we have
$X=\sum_{j=1}^dX_je_j$ and
$\|X\|_{sm}\le\sum_{j=1}^d\|X_je_j\|_{sm}=\sum_{j=1}^d\|e_j\|_1\,\|X_j\|$
for all $m$ and $X$, so that the inequality is always an equality
if and only if $\diamondsuit_{\mathrm{nc}}(Y,r)=
B_{\mathrm{nc}}(Y,r)$ for some, and hence for all, $Y\in
(\mat{\mathbb{C}}{s})^d$ and $r$. However
$\sum_{j=1}^d\alpha_j\,\|X_j\|$ with some positive coefficients
$\alpha_j$ is not an operator space norm on
$\ncspace{(\mat{(\mathbb{C}^d)}{s})}$.
\end{rem}

\begin{rem}\label{rem:uniform-const}
Instead of considering operator spaces, we can consider Banach
spaces equipped with admissible systems of matrix norms such that
\eqref{eq:dirsums-norms} and \eqref{eq:simprod-norms} hold with
constants $C_1$, $C_1'$, and $C_2$ independent of $n$ and $m$, as
in Proposition \ref{prop:ucb-inj-proj}.  We have to require
$C_1'=1$ to ensure that Proposition \ref{prop:nc_top} holds, i.e.,
 the nc balls form a
base for a topology on the corresponding nc space. In this case,
we can still consider uniformly locally bounded nc functions and
Proposition \ref{prop:cont-bdd}, Theorem \ref{thm:king-conv},
Corollary \ref{cor:nc-balls}, and Corollary
\ref{cor:nc-u-analytic} remain true. The condition $C_1=1$ is
natural to ensure the Hausdorff-like property of the
uniformly-open topology, see Proposition \ref{prop:Hausdorff-like}
together with the paragraph that precedes it. The condition
$C_2=1$ is essential for the proofs of Theorem
\ref{thm:king-conv-nc}, Theorem \ref{thm:king-semigr}, and
Corollary \ref{cor:nc-balls-semigr} (in particular, it guarantees
that nc balls are matrix circular nc sets).
\end{rem}

\begin{rem}\label{rem:real_king}
The results of this section, unlike those of Section
\ref{subsec:analytic} --- see Remark \ref{rem:real_queen}, are of
a purely noncommutative nature (except for Theorem
\ref{thm:king-conv}). Therefore, they admit full analogues in the
real case. An operator space over $\mathbb{R}$ is a Banach space
over $\mathbb{R}$ equipped with an admissible system of matrix
norms such that the real analogues of \eqref{eq:dirsums-norms} and
\eqref{eq:simprod-norms} hold with $C_1=C_1'=C_2=1$ independent of
$n$ and $m$ (see \cite{Ru}). Theorem \ref{thm:king-conv-nc} and
Corollaries \ref{cor:nc-balls}, \ref{cor:nc-u-analytic} are true
for the case where $\vecspace{V}$ and $\vecspace{W}$ are operator
spaces over $\mathbb{R}$, with all the notions defined exactly as
in the complex case. Also, Theorem \ref{thm:king-semigr} and
Corollary \ref{cor:nc-balls-semigr} are true with $\mathbb{C}^d$
replaced by $\mathbb{R}^d$ and $\vecspace{W}$ an operator space
over $\mathbb{R}$.
\end{rem}

\section{Analytic and uniformly analytic higher order nc
functions}\label{subsec:analytic-higher} We will discuss now
analyticity of higher order nc functions in various settings. Our
main goal here is to prove analogues of Theorem \ref{thm:dif-op}
for the corresponding classes of higher order nc functions.

Let $\vecspace{W}$ be a Banach space over $\mathbb{C}$. An
\emph{admissible system of rectangular matrix norms over
$\vecspace{W}$} \index{admissible system of rectangular matrix
norms} is a doubly-indexed sequence of norms $\|\cdot\|_{n,m}$
\index{$\Vert\cdot\Vert_{n,m}$} on $\rmat{\vecspace{W}}{n}{m}$,
$n,m=1,2,\ldots$, satisfying the following condition: for every
$n,p,q,m\in\mathbb{N}$ there exists $C(n,p,q,m)>0$ such that for
all $X\in\rmat{\vecspace{W}}{p}{q}$,
$S\in\rmat{\mathbb{C}}{n}{p}$, and $T\in\rmat{\mathbb{C}}{q}{m}$,
\begin{equation}\label{eq:rect-simprod-norms}
\|SXT\|_{n,m}\le C(n,p,q,m)\|S\|\,\|X\|_{p,q}\|T\|,
\end{equation}
where $\|\cdot\|$ denotes the operator norm of rectangular
matrices over $\mathbb{C}$ with respect to the standard Euclidean
norms; cf. \eqref{eq:simprod-norms}. Next we show that
\eqref{eq:rect-simprod-norms} implies the analogue of
\eqref{eq:dirsums-norms} for rectangular matrices,
\begin{multline}\label{eq:rect-dirsums-norms}
C_1(n,m,r,s)^{-1}\max\{\|X\|_{n,m},\|Y\|_{r,s}\}\le\|X\oplus
Y\|_{n+r,m+s}\\
\le C_1'(n,m,r,s)\max\{\|X\|_{n,m},\|Y\|_{r,s}\},
\end{multline}
for every $n,m,r,s\in\mathbb{N}$, $X\in\rmat{\vecspace{W}}{n}{m}$,
$Y\in\rmat{\vecspace{W}}{r}{s}$ and some positive constants
$C_1(n,m,r,s)$ and $C_1'(n,m,r,s)$.
\begin{prop}\label{prop:rect-dirsums-norms}
Let $\vecspace{W}$ be a Banach space equipped with an admissible
system of rectangular matrix norms over $\vecspace{W}$, i.e.,
\eqref{eq:rect-simprod-norms} holds. Then
\eqref{eq:rect-dirsums-norms} holds for every
$n,m,r,s\in\mathbb{N}$, $X\in\rmat{\vecspace{W}}{n}{m}$, and
$Y\in\rmat{\vecspace{W}}{r}{s}$ with
$$C_1(n,m,r,s)=\max\{C(n,n+r,m+s,m),C(r,n+r,m+s,s)\},$$
$$C_1'(n,m,r,s)=C(n+r,n,m,m+s)+C(n+r,r,s,m+s).$$
\end{prop}
\begin{proof}
Let $X\in\rmat{\vecspace{W}}{n}{m}$ and
$Y\in\rmat{\vecspace{W}}{r}{s}$. Then
\begin{multline*}
\|X\oplus Y\|_{n+r,m+s}=\|(X\oplus 0_{r\times s})+(0_{n\times
m}\oplus Y)\|_{n+r,m+s}\\
\le\Big\|\begin{bmatrix} I_n\\
0_{r\times n}
\end{bmatrix}X\begin{bmatrix}
I_m & 0_{m\times s}\end{bmatrix}+\begin{bmatrix} 0_{n\times r}\\
I_r \end{bmatrix}Y\begin{bmatrix} 0_{s\times m} & I_s
\end{bmatrix}\Big\|
\\
\le\Big(C(n+r,n,m,m+s)+C(n+r,r,s,m+s)\Big)\max\{\|X\|_{n,m},\|Y\|_{r,s}\}
\end{multline*}
and
\begin{multline*}
\max\{\|X\|_{n,m},\|Y\|_{r,s}\}\\
=\max\Big\{\Big\|\begin{bmatrix}
I_n & 0_{n\times r}\end{bmatrix}\begin{bmatrix} X & 0_{n\times
s}\\
0_{r\times m} & Y\end{bmatrix}\begin{bmatrix} I_m\\
 0_{s\times m}\end{bmatrix}\Big\|_{n,m},\\
 \Big\|\begin{bmatrix}
0_{r\times n} & I_r\end{bmatrix}\begin{bmatrix} X & 0_{n\times
s}\\
0_{r\times m} & Y\end{bmatrix}\begin{bmatrix} 0_{m\times s}\\
I_s \end{bmatrix}\Big\|_{r,s}\Big\}\\
\le\max\{C(n,n+r,m+s,m),C(r,n+r,m+s,s)\}\|X\oplus Y\|_{n+r,m+s}.
\end{multline*}
\end{proof}

 Clearly, an
admissible system of rectangular matrix norms over $\vecspace{W}$
induces an admissible system of (square) matrix norms over
$\vecspace{W}$. Notice that if $\vecspace{W}$ is an operator
space, then an admissible system of rectangular matrix norms over
$\vecspace{W}$ is unique and one can choose
$$C(n,p,q,m)=C_1(n,p,q,m)=C_1'(n,p,q,m)=1$$ for every
$n,p,q,m\in\mathbb{N}$ \cite[Exercises 13.1 and 13.2]{Pa}.

We can also introduce rectangular block injections and projections
as follows. Let  $\ring$ be a unital commutative ring, let
$\module{M}$ be a module over $\ring$, and let $s$, $s_1$, \ldots,
$s_m\in\mathbb{N}$ be such that $s=s_1+\cdots+s_m$. Set
\index{$E_{i}$}
$$E_{i}:=\col\,[0_{s_1\times s_i},\ldots, 0_{s_{i-1\,}\times
s_{i}},I_{s_i},0_{s_{i+1}\times s_{i+1}},\cdots,0_{s_m\times
s_i}]\in\rmat{\ring}{s}{s_i},\quad i=1,\ldots,m.$$ The
\emph{rectangular block injections} \index{rectangular block
injection} are defined by \index{$\iota_{ij}^{(s_1,\ldots,s_m)}$}
\begin{equation}\label{eq:rect-inj}
\iota_{ij}^{(s_1,\ldots,s_m)}\colon\rmat{\module{M}}{s_i}{s_j}\to\mat{\module{M}}{s},\quad
\iota_{ij}^{(s_1,\ldots,s_m)}\colon Z\mapsto E_{i}ZE_j^\top,\quad
i,j=1,\ldots,m,
\end{equation}
and the \emph{rectangular block projections} \index{rectangular
block projection} are defined by
\index{$\pi_{ij}^{(s_1,\ldots,s_m)}$}
\begin{equation}\label{eq:rect-proj}
\pi_{ij}^{(s_1,\ldots,s_m)}\colon\mat{\module{M}}{s}\to\rmat{\module{M}}{s_i}{s_j},\quad
\pi_{ij}^{(s_1,\ldots,s_m)}\colon Z\mapsto E_{i}^\top ZE_j,\quad
i,j=1,\ldots,m,
\end{equation}
The following analogue of Proposition \ref{prop:inj-proj}, in
conjunction with Proposition \ref{prop:rect-dirsums-norms}, holds.
\begin{prop}\label{prop:rect-inj-proj}
Condition \eqref{eq:rect-simprod-norms} (valid for all $n$, $p$,
$q$, $m\in\mathbb{N}$, $X\in\rmat{\vecspace{W}}{p}{q}$,
$S\in\rmat{\mathbb{C}}{n}{p}$, and $T\in\rmat{\mathbb{C}}{q}{m}$)
is equivalent to the boundedness of the rectangular block
injections $\iota^{(s_1,\ldots,s_m)}_{ij}$ and rectangular block
projections $\pi^{(s_1,\ldots,s_m)}_{ij}$, $m\in\mathbb{N}$,
$s_1$, \ldots, $s_m\in\mathbb{N}$, $i,j=1,\ldots,m$. In
particular,
\begin{equation}\label{eq:rect-inj-proj-estimates}
\|\iota^{(s_1,\ldots,s_m)}_{ij}\|\le C(s,s_i,s_j,s),\quad
\|\pi^{(s_1,\ldots,s_m)}_{ij}\|\le C(s_i,s,s,s_j).
\end{equation}
\end{prop}
\begin{proof}
If \eqref{eq:rect-simprod-norms} holds for all $n$, $p$, $q$,
$m\in\mathbb{N}$, $X\in\rmat{\vecspace{W}}{p}{q}$,
$S\in\rmat{\mathbb{C}}{n}{p}$, and $T\in\rmat{\mathbb{C}}{q}{m}$,
then \eqref{eq:rect-inj-proj-estimates} follows immediately from
\eqref{eq:rect-inj} and \eqref{eq:rect-proj}, thus
$\iota^{(s_1,\ldots,s_m)}_{ij}$ and $\pi^{(s_1,\ldots,s_m)}_{ij}$
are bounded.

 The converse direction can be shown analogously to that of Proposition
\ref{prop:inj-proj}.
\end{proof}

We also have the following analogue of Proposition
\ref{prop:ucb-inj-proj}.
\begin{prop}\label{prop:rect-ucb-inj-proj}
If $\vecspace{W}$ is a Banach space equipped with an admissible
system of rectangular matrix norms over $\vecspace{W}$ such that
\eqref{eq:rect-simprod-norms} holds with $C$ independent of $n$,
$p$, $q$, and $m$, then the rectangular block injections
 \eqref{eq:rect-inj} and the rectangular block projections
\eqref{eq:rect-proj} are uniformly completely bounded, i.e., for
every $m\in\mathbb{N}$ and $s_1,\ldots,s_m\in\mathbb{N}$, the
mappings $\iota_{ij}^{(s_1,\ldots,s_m)}$ and
$\pi_{ij}^{(s_1,\ldots,s_m)}$ are completely bounded. Moreover,
\begin{equation}\label{eq:rect-inj-proj-cb}
\|\iota_{ij}^{(s_1,\ldots,s_m)}\|_{\cb}\le C,\quad
\|\pi_{ij}^{(s_1,\ldots,s_m)}\|_{\cb}\le C.
\end{equation}
If $\vecspace{W}$ is
 an operator space, then all $\iota_{ij}^{(s_1,\ldots,s_m)}$ are complete isometries and all
$\pi_{ij}^{(s_1,\ldots,s_m)}$ are complete coisometries.
\end{prop}
\begin{proof}
Let $n,m\in\mathbb{N}$, $1\le i,j\le m$,
$s_1,\ldots,s_m\in\mathbb{N}$, $W=(w_{ab})_{a,b=1,\ldots,n}\in
\mat{(\rmat{\vecspace{W}}{s_i}{s_j})}{n}\cong\rmat{\vecspace{W}}{s_{i}n}{s_{j}n}$.
For $a,b=1,\ldots,n$, let $E_{ab}\in\mat{\mathbb{C}}{n}$ be
defined by $(E_{ab})_{\alpha\beta}=\delta_{(a,b),(\alpha,\beta)}$,
$\alpha,\beta=1,\ldots,n$. Then
\begin{multline*}
\|(\id_{\mat{\mathbb{C}}{n}}\otimes\iota_{ij}^{(s_1,\ldots,s_m)})(W)\|_{(s_1+\cdots+s_m)n}\\
=
\Big\|(\id_{\mat{\mathbb{C}}{n}}\otimes\iota_{ij}^{(s_1,\ldots,s_m)})
\Big(\sum_{a,b=1}^nE_{ab}\otimes
w_{ab}\Big)\Big\|_{(s_1+\cdots+s_m)n}\\
=\Big\|\sum_{a,b=1}^nE_{ab}\otimes
\iota_{ij}^{(s_1,\ldots,s_m)}(w_{ab})\Big\|_{(s_1+\cdots+s_m)n}=\Big\|\sum_{a,b=1}^nE_{ab}\otimes
E_{i}w_{ab}E_j^\top\Big\|_{(s_1+\cdots+s_m)n}\\
=\|(I_n\otimes E_i)W(I_n\otimes E_j^\top)\|_{(s_1+\cdots+s_m)n}
\le C\|W\|_{s_{i}n,s_{j}n}.
\end{multline*}
This proves the first inequality in \eqref{eq:rect-inj-proj-cb}.

 Suppose now that
$$W\!=\!(w_{ab})_{a,b=1,\ldots,n}\in\!
\mat{\left(\mat{\vecspace{W}}{(s_1+\cdots+s_m)}\right)}{n}\!\cong
\mat{\vecspace{W}}{(s_1+\cdots+s_m)n}.$$ Then
\begin{multline*}
\|(\id_{\mat{\mathbb{C}}{n}}\otimes\pi_{ij}^{(s_1,\ldots,s_m)})(W)\|_{s_{i}n,s_{j}n}\\
=
\Big\|(\id_{\mat{\mathbb{C}}{n}}\otimes\pi_{ij}^{(s_1,\ldots,s_m)})
\Big(\sum_{a,b=1}^nE_{ab}\otimes w_{ab}
\Big)\Big\|_{s_{i}n,s_{j}n}\\
=\Big\|
\sum_{a,b=1}^nE_{ab}\otimes\pi_{ij}^{(s_1,\ldots,s_m)}(w_{ab})\Big\|_{s_{i}n,s_{j}n}=
\Big\| \sum_{a,b=1}^nE_{ab}\otimes E_i^\top
w_{ab}E_j\Big\|_{s_{i}n,s_{j}n}\\
=\|(I_n\otimes E_i^\top) W(I_n\otimes E_j)\|_{s_{i}n,s_{j}n}\le
C\|W\|_{(s_1+\cdots+s_m)n}.
\end{multline*}
This proves the second inequality in \eqref{eq:rect-inj-proj-cb}.

If $\vecspace{W}$ is an operator space, we obtain that the
operators $\iota_{ij}^{(s_1,\ldots,s_m)}$ and
$\pi_{ij}^{(s_1,\ldots,s_m)}$ are complete contractions. Moreover,
since
$\pi_{ij}^{(s_1,\ldots,s_m)}\iota_{ij}^{(s_1,\ldots,s_m)}=\id_{\rmat{\vecspace{W}}{s_i}{s_j}}$,
all $\iota_{ij}^{(s_1,\ldots,s_m)}$ are complete isometries and
all $\pi_{ij}^{(s_1,\ldots,s_m)}$ are complete coisometries.
\end{proof}

Let $\vecspace{V}_0$, \ldots, $\vecspace{V}_k$, $\vecspace{W}_1$,
\ldots, $\vecspace{W}_k$ be vector spaces over $\mathbb{C}$, let
$\Omega^{(j)}\subseteq\ncspacej{\vecspace{V}}{j}$, $j=0,\ldots,k$,
be finitely open nc sets, and let $\vecspace{W}_0$ be a Banach
space with an admissible system of rectangular matrix norms over
$\vecspace{W}_0$. A higher order nc function
$f\in\tclass{k}(\Omega^{(0)},\ldots,\Omega^{(k)};\ncspacej{\vecspace{W}}{0},\ldots,\ncspacej{\vecspace{W}}{k})$
is called \emph{$W$-locally bounded on slices} \index{W
@$W$-locally bounded on slices higher order nc function} if for
every $n_0,\ldots,n_k\in\mathbb{N}$, $Y^0\in\Omega^{(0)}_{n_0}$,
\ldots, $Y^k\in\Omega^{(k)}_{n_k}$,
$Z^0\in\mat{\vecspace{V}_0}{n_0}$, \ldots,
$Z^k\in\mat{\vecspace{V}_k}{n_k}$,
$W^1\in\rmat{\vecspace{W}_1}{n_0}{n_1}$, \ldots,
$W^k\in\rmat{\vecspace{W}_k}{n_{k-1}}{n_k}$ there exists
$\epsilon>0$ such that
$f(Y^0+tZ^0,\ldots,Y^k+tZ^k)(W^1,\ldots,W^k)$ is bounded in
$\rmat{\vecspace{W}_0}{n_0}{n_k}$ for $|t|<\epsilon$, i.e.,  the
function
$f|_{\Omega^{(0)}_{n_0}\times\cdots\times\Omega^{(k)}_{n_k}}$ is
locally bounded on slices for every fixed $W^1$, \ldots, $W^k$. A
higher order nc function
$f\in\tclass{k}(\Omega^{(0)},\ldots,\Omega^{(k)};\ncspacej{\vecspace{W}}{0},\ldots,\ncspacej{\vecspace{W}}{k})$
is called \emph{$W$-G\^{a}teaux (G$_W$-) differentiable} \index{W
@$W$-G\^{a}teaux (G$_W$-) differentiable higher order nc function}
if for every $n_0,\ldots,n_k\in\mathbb{N}$,
$Y^0\in\Omega^{(0)}_{n_0}$, \ldots, $Y^k\in\Omega^{(k)}_{n_k}$,
$Z^0\in\mat{\vecspace{V}_0}{n_0}$, \ldots,
$Z^k\in\mat{\vecspace{V}_k}{n_k}$,
$W^1\in\rmat{\vecspace{W}_1}{n_0}{n_1}$, \ldots,
$W^k\in\rmat{\vecspace{W}_k}{n_{k-1}}{n_k}$ the G-derivative,
\begin{multline}\label{eq:w-g-der-k}
\lim_{t\to 0}\frac{f(Y^0+tZ^0,\ldots,Y^k+tZ^k)(W^1,\ldots,W^k)-
f(Y^0,\ldots,Y^k)(W^1,\ldots,W^k)}{t}\\
=\frac{d}{dt}f(Y^0+tZ^0,\ldots,Y^k+tZ^k)(W^1,\ldots,W^k)\Big|_{t=0},
\end{multline}
exists, i.e., the function
$f|_{\Omega^{(0)}_{n_0}\times\cdots\times\Omega^{(k)}_{n_k}}$ is
G-differentiable for every fixed $W^1$, \ldots, $W^k$. It follows
that $f$ is \emph{$W$-analytic on slices}, \index{W @$W$-analytic
on slices higher order nc function} i.e., for every
$Y^0\in\Omega^{(0)}_{n_0}$, \ldots, $Y^k\in\Omega^{(k)}_{n_k}$,
$Z^0\in\mat{\vecspace{V}_0}{n_0}$, \ldots,
$Z^k\in\mat{\vecspace{V}_k}{n_k}$,
$W^1\in\rmat{\vecspace{W}_1}{n_0}{n_1}$, \ldots,
$W^k\in\rmat{\vecspace{W}_k}{n_{k-1}}{n_k}$,
$f(Y^0+tZ^0,\ldots,Y^k+tZ^k)(W^1,\ldots,W^k)$ is an analytic
function of $t$ in a neighbourhood of $0$. By Hartogs' theorem
\cite[Page 28]{Sh}, $f$ is analytic on
$$(\vecspace{U}_0\cap\Omega^{(0)}_{n_0})\times\cdots\times(\vecspace{U}_k\cap\Omega^{(k)}_{n_k})\times
\vecspace{Y}_1\times\cdots\times\vecspace{Y}_k$$ as a function of
several complex variables for every $n_0,\ldots,n_k\in\mathbb{N}$
and for all finite-dimensional subspaces
$\vecspace{U}_0\subseteq\mat{\vecspace{V}_0}{n_0}$, \ldots,
$\vecspace{U}_k\subseteq\mat{\vecspace{V}_k}{n_k}$,
$\vecspace{Y}_1\subseteq\rmat{\vecspace{W}_1}{n_0}{n_1}$, \ldots,
$\vecspace{Y}_k\subseteq\rmat{\vecspace{W}_k}{n_{k-1}}{n_k}$.
\begin{thm}\label{thm:lbs-dif-op}
Suppose that
$f\in\tclass{k}(\Omega^{(0)},\ldots,\Omega^{(k)};\ncspacej{\vecspace{W}}{0},\ldots,\ncspacej{\vecspace{W}}{k})$
is $W$-locally bounded on slices. Then so is
$$\Delta_Rf\in\tclass{k+1}(\Omega^{(0)},\ldots,\Omega^{(k)},\Omega^{(k)};
\ncspacej{\vecspace{W}}{0},\ldots,\ncspacej{\vecspace{W}}{k},
\ncspacej{\vecspace{V}}{k}).$$
\end{thm}
\begin{proof}
Let $n_0,\ldots,n_{k+1}\in\mathbb{N}$, $Y^0\in\Omega^{(0)}_{n_0}$,
\ldots, $Y^k\in\Omega^{(k)}_{n_k}$,
$Y^{k+1}\in\Omega^{(k)}_{n_{k+1}}$,
$Z^0\in\mat{\vecspace{V}_0}{n_0}$, \ldots,
$Z^k\in\mat{\vecspace{V}_k}{n_k}$,
$Z^{k+1}\in\mat{\vecspace{V}_k}{n_{k+1}}$,
$W^1\in\rmat{\vecspace{W}_1}{n_0}{n_1}$, \ldots,
$W^{k+1}\in\rmat{(\vecspace{W}_{k+1})}{n_{k}}{n_{k+1}}$. Since
$\Omega^{(k)}\subseteq\ncspacej{\vecspace{V}}{k}$ is a right
admissible nc
set, there exists $r>0$ such that $\begin{bmatrix} Y^k & rW^{k+1}\\
0 & Y^{k+1}\end{bmatrix}\in\Omega^{(k)}_{n_k+n_{k+1}}$. Since
$\Omega^{(k)}$
is finitely open, there exists $\epsilon_1>0$ such that $\begin{bmatrix} Y^k+tZ^k & rW^{k+1}\\
0 & Y^{k+1}+tZ^{k+1}\end{bmatrix}\in\Omega_{n_k+n_{k+1}}$ for all
$t$ with $|t|<\epsilon_1$. Since $f$ is $W$-locally bounded on
slices, there exists $\epsilon_2>0$ such that
\begin{multline*}
f\left(Y^0+tZ^0,\ldots,Y^{k-1}+tZ^{k-1},\begin{bmatrix} Y^k+tZ^k & rW^{k+1}\\
0 &
Y^{k+1}+tZ^{k+1}\end{bmatrix}\right)\\
(W^1,\ldots,W^{k-1},\row[W^k, 0]) \end{multline*} is bounded for
$|t|<\epsilon_2$. Therefore, using \eqref{eq:uptr-k}, we obtain
that
\begin{multline*}
\Delta_Rf(Y^0+tZ^0,\ldots,Y^{k+1}+tZ^{k+1})(W^1,\ldots,W^{k+1})\\
=r^{-1}f\left(Y^0+tZ^0,\ldots,Y^{k-1}+tZ^{k-1},\begin{bmatrix} Y^k+tZ^k & rW^{k+1}\\
0 &
Y^{k+1}+tZ^{k+1}\end{bmatrix}\right)\\
(W^1,\ldots,W^{k-1},\row[W^k,
0])\begin{bmatrix} 0\\
I_{n_{k+1}}\end{bmatrix}
\end{multline*}
is bounded for $|t|<\min\{\epsilon_1,\epsilon_2\}$. We conclude
that $\Delta_Rf$ is $W$-locally bounded on slices. Notice that we
used here the fact that $\ncspacej{\vecspace{W}}{0}$ is equipped
with an admissible system of rectangular matrix norms (more
precisely, the fact that the block projection implemented by the
multiplication with $\begin{bmatrix} 0\\
I_{n_{k+1}}\end{bmatrix}$ on the right is bounded).
\end{proof}
\begin{rem}\label{rem:lbs-j-dif-op}
An analogue of Theorem \ref{thm:lbs-dif-op} also holds, with an
analogous proof, for ${}_j\Delta_Rf$, see Remark
\ref{rem:j-delta}.
\end{rem}
\begin{thm}\label{thm:g-queen-k}
Suppose that
$f\in\tclass{k}(\Omega^{(0)},\ldots,\Omega^{(k)};\ncspacej{\vecspace{W}}{0},\ldots,\ncspacej{\vecspace{W}}{k})$
is $W$-locally bounded on slices. Then $f$ is
G$_W$-differentiable.
\end{thm}
\begin{proof}
 We will first show
that $f$ is \emph{$W$-continuous on slices}, \index{W
@$W$-continuous on slices higher order nc function} i.e.,
$$\lim_{t\to
0}f(Y^0+tZ^0,\ldots,Y^{k}+tZ^{k})(W^1,\ldots,W^{k})=f(Y^0,\ldots,Y^{k})(W^1,\ldots,W^{k}),$$
for all $n_0,\ldots,n_k\in\mathbb{N}$, $Y^0\in\Omega^{(0)}_{n_0}$,
\ldots, $Y^k\in\Omega^{(k)}_{n_k}$,
$Z^0\in\mat{\vecspace{V}_0}{n_0}$, \ldots,
$Z^k\in\mat{\vecspace{V}_k}{n_k}$,
$W^1\in\rmat{\vecspace{W}_1}{n_0}{n_1}$, \ldots,
$W^k\in\rmat{\vecspace{W}_k}{n_{k-1}}{n_k}$. We have, by
\eqref{eq:j-RightLagr_k'},
\begin{multline}\label{eq:slice-Lagrange}
f(Y^0+tZ^0,\ldots,Y^{k}+tZ^{k})(W^1,\ldots,W^{k})-f(Y^0,\ldots,Y^{k})(W^1,\ldots,W^{k})\\
=\sum_{j=0}^k\left(f(Y^0+tZ^0,\ldots,Y^{j}+tZ^{j},Y^{j+1},\ldots,Y^k)(W^1,\ldots,W^{k})\right.\\
\left.-
f(Y^0+tZ^0,\ldots,Y^{j-1}+tZ^{j-1},Y^j,\ldots,Y^{k})(W^1,\ldots,W^{k})\right)\\
=t\sum_{j=0}^k {}_j\Delta_Rf(Y^0+tZ^0,\ldots,Y^{j-1}+tZ^{j-1},Y^{j}+tZ^{j},Y^j, Y^{j+1},\ldots,Y^k)\\
(W^1,\ldots,W^{j-1},Z^j,W^j,W^{j+1},\ldots,W^k)\to 0
\end{multline}
as $t\to 0$ because the latter sum is bounded for small $t$, see
Remark \ref{rem:lbs-j-dif-op}.

Since ${}_j\Delta_Rf$ is $W$-continuous on slices by Remark
\ref{rem:lbs-j-dif-op} again and the paragraph above, it follows
from \eqref{eq:slice-Lagrange} that $f$ is G$_W$-differentiable
and
\begin{multline}\label{eq:slice-der-k}
\frac{d}{dt}f(Y^0+tZ^0,\ldots,Y^{k}+tZ^{k})(W^1,\ldots,W^{k})\Big|_{t=0}\\
=\sum_{j=0}^k{}_j\Delta_Rf
(Y^0,\ldots,Y^{j-1},Y^j,Y^j,Y^{j+1},\ldots,Y^k)\\
(W^1,\ldots,W^{j-1},W^j,Z^j,W^{j+1},\ldots,W^k).
\end{multline}
\end{proof}
The converse of Theorem \ref{thm:g-queen-k} is clear. We denote by
$$\tclass{k}_{G_W}=\tclass{k}_{G_W}(\Omega^{(0)},\ldots,\Omega^{(k)};
\ncspacej{\vecspace{W}}{0},\ldots,\ncspacej{\vecspace{W}}{k})$$
the subclass of
$\tclass{k}(\Omega^{(0)},\ldots,\Omega^{(k)};\ncspacej{\vecspace{W}}{0},\ldots,\ncspacej{\vecspace{W}}{k})$
consisting of nc functions of order $k$ that are
G$_W$-differentiable.

The following statement is an obvious corollary of Theorems
\ref{thm:lbs-dif-op} and \ref{thm:g-queen-k}.
\begin{cor}\label{cor:g-dif-op}
Let
$f\in\tclass{k}_{G_W}(\Omega^{(0)},\ldots,\Omega^{(k)};\ncspacej{\vecspace{W}}{0},\ldots,\ncspacej{\vecspace{W}}{k})$.
Then
$$\Delta_Rf\in\tclass{k+1}_{G_W}(\Omega^{(0)},\ldots,\Omega^{(k)},\Omega^{(k)};\ncspacej{\vecspace{W}}{0},\ldots,
\ncspacej{\vecspace{W}}{k},\ncspacej{\vecspace{V}}{k}).$$
\end{cor}
\begin{rem}\label{rem:g-j-dif-op}
An analogue of Corollary \ref{cor:g-dif-op} also holds for
${}_j\Delta_Rf$, see Remarks \ref{rem:j-delta} and
\ref{rem:lbs-j-dif-op}.
\end{rem}

We recall that a $k$-linear mapping
$\omega\colon\vecspace{W}_1\times\cdots\times\vecspace{W}_k\to\vecspace{W}_0$,
where $\vecspace{W}_0$, \ldots, $\vecspace{W}_k$ are Banach
spaces, is called \emph{bounded} \index{bounded $k$-linear
mapping} if \index{$\Vert\omega\Vert_{\mathcal{L}^k}$}
$$\|\omega\|_{\mathcal{L}^k}:=\sup_{\|w^1\|=\ldots=\|w^k\|=1}\|\omega(w^1,\ldots,w^k)\|<\infty;$$
we denote by
$\mathcal{L}^k(\vecspace{W}_1\times\cdots\times\vecspace{W}_k,\vecspace{W}_0)$
the Banach space of such bounded $k$-linear mappings. Let
$\vecspace{V}_0$, \ldots, $\vecspace{V}_k$, $\vecspace{W}_0$,
\ldots, $\vecspace{W}_k$ be Banach spaces over $\mathbb{C}$
equipped with admissible systems of rectangular matrix norms, and
let $\Omega^{(j)}\subseteq\ncspacej{\vecspace{V}}{j}$,
$j=0,\ldots,k$, be open nc sets. A higher order nc function
$f\in\tclass{k}(\Omega^{(0)},\ldots,\Omega^{(k)};\ncspacej{\vecspace{W}}{0},\ldots,\ncspacej{\vecspace{W}}{k})$
is called \emph{locally bounded} \index{locally bounded higher
order nc function} if for every $n_0,\ldots,n_k\in\mathbb{N}$ the
function
$f|_{\Omega^{(0)}_{n_0}\times\cdots\times\Omega^{(k)}_{n_k}}$ has
values in
$\mathcal{L}^k(\rmat{\vecspace{W}_1}{n_0}{n_1}\times\cdots\times\rmat{\vecspace{W}_k}{n_{k-1}}{n_k},
\rmat{\vecspace{W}_0}{n_0}{n_k})$ and is locally bounded, i.e.,
 for every
$Y^0\in\Omega^{(0)}_{n_0}$, \ldots, $Y^k\in\Omega^{(k)}_{n_k}$
there exists $\delta>0$ such that
$\|f(X^0,\ldots,X^{k})\|_{\mathcal{L}^k}$ is bounded for $X^0\in
B(Y^0,\delta)$, \ldots, $X^k\in B(Y^k,\delta)$. Clearly, a locally
bounded higher order nc function is $W$-locally bounded on slices.

 A higher order nc function
$f\in\tclass{k}(\Omega^{(0)},\ldots,\Omega^{(k)};\ncspacej{\vecspace{W}}{0},\ldots,\ncspacej{\vecspace{W}}{k})$
is called \emph{G-differentiable} \index{G\^{a}teaux (G-)
differentiable higher order nc function} if for every
$n_0,\ldots,n_k\in\mathbb{N}$ the function
$f|_{\Omega^{(0)}_{n_0}\times\cdots\times\Omega^{(k)}_{n_k}}$ has
values in
$\mathcal{L}^k(\rmat{\vecspace{W}_1}{n_0}{n_1}\times\cdots\times\rmat{\vecspace{W}_k}{n_{k-1}}{n_k},
\rmat{\vecspace{W}_0}{n_0}{n_k})$ and is G-differentiable, i.e.,
for every $Y^0\in\Omega^{(0)}_{n_0}$, \ldots,
$Y^k\in\Omega^{(k)}_{n_k}$, $Z^0\in\mat{\vecspace{V}_0}{n_0}$,
\ldots, $Z^k\in\mat{\vecspace{V})k}{n_k}$, the G-derivative
\begin{multline}\label{eq:g-der-k}
\delta f(Y^0,\ldots,Y^k)(Z^0,\ldots,Z^k)=\lim_{t\to
0}\frac{f(Y^0+tZ^0,\ldots,Y^k+tZ^k)-f(Y^0,\ldots,Y^k)}{t}\\
=\frac{d}{dt}f(Y^0+tZ^0,\ldots,Y^k+tZ^k)\Big|_{t=0}
\end{multline}
exists in the norm of
$\mathcal{L}^k(\rmat{\vecspace{W}_1}{n_0}{n_1}\times\cdots\times\rmat{\vecspace{W}_k}{n_{k-1}}{n_k},
\rmat{\vecspace{W}_0}{n_0}{n_k})$. Clearly, a G-differentiable
higher order nc function is G$_W$-differentiable --- compare
\eqref{eq:g-der-k} and \eqref{eq:w-g-der-k}.

A higher order nc function
$f\in\tclass{k}(\Omega^{(0)},\ldots,\Omega^{(k)};\ncspacej{\vecspace{W}}{0},\ldots,\ncspacej{\vecspace{W}}{k})$
is called \emph{Fr\'{e}chet (F-) differentiable}
\index{Fr\'{e}chet (F-) differentiable higher order nc function}
if for every $n_0,\ldots,n_k\in\mathbb{N}$ the function
$f|_{\Omega^{(0)}_{n_0}\times\cdots\times\Omega^{(k)}_{n_k}}$ has
values in
$\mathcal{L}^k(\rmat{\vecspace{W}_1}{n_0}{n_1}\times\cdots\times\rmat{\vecspace{W}_k}{n_{k-1}}{n_k},
\rmat{\vecspace{W}_0}{n_0}{n_k})$ and  is F-differentiable, i.e.,
$f|_{\Omega^{(0)}_{n_0}\times\cdots\times\Omega^{(k)}_{n_k}}$ is
G-differentiable and for any $Y^0\in\Omega^{(0)}_{n_0}$, \ldots,
$Y^k\in\Omega^{(k)}_{n_k}$ the linear operator
\begin{multline*}\delta f(Y^0,\ldots,Y^k)\colon
\mat{\vecspace{V}_0}{n_0}\oplus\cdots\oplus\mat{\vecspace{V}_k}{n_k}\\
\to
\mathcal{L}^k(\rmat{\vecspace{W}_1}{n_0}{n_1}\times\cdots\times\rmat{\vecspace{W}_k}{n_{k-1}}{n_k},
\rmat{\vecspace{W}_0}{n_0}{n_k})\end{multline*} is bounded.

A higher order nc function
$f\in\tclass{k}(\Omega^{(0)},\ldots,\Omega^{(k)};\ncspacej{\vecspace{W}}{0},\ldots,\ncspacej{\vecspace{W}}{k})$
is called \emph{analytic} \index{analytic higher order nc
function} if $f$ is locally bounded and G-differentiable. Clearly,
an analytic higher order nc function is $W$-analytic. If
$f\in\tclass{k}(\Omega^{(0)},\ldots,\Omega^{(k)};\ncspacej{\vecspace{W}}{0},\ldots,\ncspacej{\vecspace{W}}{k})$
is analytic, then by \cite[Theorem 3.17.1]{HiPh}
$f|_{\Omega^{(0)}_{n_0}\times\cdots\times\Omega^{(k)}_{n_k}}$ is
also continuous and F-differentiable. It follows that an analytic
higher order nc function $f$ is continuous (i.e., all its
restrictions
$f|_{\Omega^{(0)}_{n_0}\times\cdots\times\Omega^{(k)}_{n_k}}$ are
continuous) and F-differentiable.

\begin{thm}\label{thm:lb-dif-op}
If
$f\in\tclass{k}(\Omega^{(0)},\ldots,\Omega^{(k)};\ncspacej{\vecspace{W}}{0},\ldots,\ncspacej{\vecspace{W}}{k})$
is locally bounded, then so is
$\Delta_Rf\in\tclass{k+1}(\Omega^{(0)},\ldots,\Omega^{(k)},\Omega^{(k)};\ncspacej{\vecspace{W}}{0},\ldots,
\ncspacej{\vecspace{W}}{k},\ncspacej{\vecspace{V}}{k})$.
\end{thm}
\begin{proof}
Let $n_0,\ldots,n_{k+1}\in\mathbb{N}$, and let
$Y^0\in\Omega^{(0)}_{n_0}$, \ldots, $Y^k\in\Omega^{(k)}_{n_k}$,
$Y^{k+1}\in\Omega^{(k)}_{n_{k+1}}$. Since $f$ is locally bounded,
there exists $\delta>0$ such that
\begin{multline*}
\|f(X^0,\ldots,X^{k-1},X)\|_{\mathcal{L}^k}\\
=\sup_{\|W^1\|_{n_0,n_1}=\ldots
=\|W^{k-1}\|_{n_{k-2},n_{k-1}}=\|W\|_{n_{k-1},n_k+n_{k+1}}=1}
\|f(X^0,\ldots,X^{k-1},X)\\
(W^1,\ldots,W^{k-1},W)\|_{n_0,n_k+n_{k+1}}
\end{multline*}
is bounded for all $X^0\in B(Y^0,\delta)$, \ldots, $X^{k-1}\in
B(Y^{k-1},\delta)$, $X\in B(Y^k\oplus Y^{k+1},\delta)$. Since
$$\begin{bmatrix} X^k & \frac{\delta}{2}W^{k+1}\\
0 & X^{k+1}
\end{bmatrix}\in B(Y^k\oplus Y^{k+1},\delta)$$ whenever
$X^k\in B(Y^k,\frac{\delta}{2})$, $X^{k+1}\in
B(Y^{k+1},\frac{\delta}{2})$, and
$W^{k+1}\in\rmat{\vecspace{V}_k}{n_k}{n_{k+1}}$ with
$\|W^{k+1}\|_{n_k,n_{k+1}}=1$, we obtain that
$$\left\|f\left(X^0,\ldots,X^{k-1},\begin{bmatrix} X^k & \frac{\delta}{2}W^{k+1}\\
0 & X^{k+1}
\end{bmatrix}\right)(W^1,\ldots,W^{k-1},\row[W^k,0])\right\|_{n_0,n_k+n_{k+1}}$$
is bounded for $X^0\in B(Y^0,\frac{\delta}{2})$, \ldots,
$X^{k+1}\in B(Y^{k+1},\frac{\delta}{2})$,
$\|W^1\|_{n_0,n_1}=\ldots=\|W^{k+1}\|_{n_k,n_{k+1}}=1$. Here we
used the boundedness of the block injection
$$W^k\in\rmat{\vecspace{W}_k}{n_{k-1}}{n_k}\longmapsto\row[W^k,0]\in
\rmat{\vecspace{W}_k}{n_{k-1}}{(n_k+n_{k+1})}.$$
Therefore,
\begin{multline*}
\|\Delta_Rf(X^0,\ldots,X^{k+1})\|_{\mathcal{L}^{k+1}}\\
=\sup_{\|W^1\|_{n_0,n_1}=\ldots =\|W^{k+1}\|_{n_k,n_{k+1}}=1}
\|\Delta_Rf(X^0,\ldots,X^{k+1})(W^1,\ldots,W^{k+1})\|_{n_0,n_{k+1}}\\
=\frac{2}{\delta}\sup_{\|W^1\|_{n_0,n_1}=\ldots
=\|W^{k+1}\|_{n_k,n_{k+1}}=1}
\Big\|f\left(X^0,\ldots,X^{k-1},\begin{bmatrix} X^k & \frac{\delta}{2}W^{k+1}\\
0 & X^{k+1}
\end{bmatrix}\right)\\
(W^1,\ldots,W^{k-1},\row[W^k,0])
\begin{bmatrix}
0\\
I_{n_{k+1}}
\end{bmatrix}\Big\|_{n_0,n_{k+1}}
\end{multline*}
is bounded for $X^0\in B(Y^0,\frac{\delta}{2})$, \ldots,
$X^{k+1}\in B(Y^{k+1},\frac{\delta}{2})$. Here we used the
boundedness of the block projection implemented by the
multiplication with $\begin{bmatrix} 0\\
I_{n_{k+1}}\end{bmatrix}$ on the right.
\end{proof}
\begin{rem}\label{rem:lb-j-dif-op}
An analogue of Theorem \ref{thm:lb-dif-op} also holds, with an
analogous proof, for ${}_j\Delta_Rf$, see Remark
\ref{rem:j-delta}.
\end{rem}
\begin{thm}\label{thm:f-queen-k}
If
$f\in\tclass{k}(\Omega^{(0)},\ldots,\Omega^{(k)};\ncspacej{\vecspace{W}}{0},\ldots,\ncspacej{\vecspace{W}}{k})$
is locally bounded, then $f$ is analytic.
\end{thm}
\begin{proof}
 We will first show
that $f$ is \emph{continuous on slices}, \index{continuous on
slices higher order nc function} i.e.,
$$\lim_{t\to
0}\|f(Y^0+tZ^0,\ldots,Y^{k}+tZ^{k})-f(Y^0,\ldots,Y^{k})\|_{\mathcal{L}^k}=0$$
for all $n_0,\ldots,n_k\in\mathbb{N}$, $Y^0\in\Omega^{(0)}_{n_0}$,
\ldots, $Y^k\in\Omega^{(k)}_{n_k}$,
$Z^0\in\mat{\vecspace{V}_0}{n_0}$, \ldots,
$Z^k\in\mat{\vecspace{V}_k}{n_k}$. Indeed, the convergence in
\eqref{eq:slice-Lagrange} is uniform on the set of
$W^1\in\rmat{\vecspace{W}_1}{n_0}{n_1}$, \ldots,
$W^k\in\rmat{\vecspace{W}_k}{n_{k-1}}{n_k}$ satisfying
$\|W^1\|_{n_0,n_1}=\ldots =\|W^k\|_{n_{k-1},n_k}=1$ because the
last sum in \eqref{eq:slice-Lagrange}  is uniformly bounded on
this set for small $t$, see Remark \ref{rem:lb-j-dif-op}.

Since ${}_j\Delta_Rf$ is continuous on slices by Remark
\ref{rem:lb-j-dif-op} again and the paragraph above,
\eqref{eq:slice-Lagrange} implies that $f$ is G-differentiable and
$\delta f(Y^0,\ldots,Y^k)(Z^0,\ldots,Z^k)$ is calculated by
\eqref{eq:slice-der-k} where the limit is uniform on the set of
$W^1\in\rmat{\vecspace{W}_1}{n_0}{n_1}$, \ldots,
$W^k\in\rmat{\vecspace{W}_k}{n_{k-1}}{n_k}$ satisfying
$\|W^1\|_{n_0,n_1}=\ldots =\|W^k\|_{n_{k-1},n_k}=1$. We conclude
that $f$ is analytic.
\end{proof}

It follows from Theorem \ref{thm:f-queen-k} that a higher order nc
function $f$ is locally bounded if and only if $f$ is continuous
if and only if $f$ is F-differentiable if and only if $f$ is
analytic. We denote by $\tclass{k}_{\rm an}=\tclass{k}_{\rm
an}(\Omega^{(0)},\ldots,\Omega^{(k)};\ncspacej{\vecspace{W}}{0},\ldots,\ncspacej{\vecspace{W}}{k})$
the subclass of \index{$\tclass{k}_{\rm
an}(\Omega^{(0)},\ldots,\Omega^{(k)};\ncspacej{\vecspace{W}}{0},\ldots,\ncspacej{\vecspace{W}}{k})$}
$\tclass{k}(\Omega^{(0)},\ldots,\Omega^{(k)};\ncspacej{\vecspace{W}}{0},\ldots,\ncspacej{\vecspace{W}}{k})$
consisting of nc functions of order $k$ that are analytic.

The following statement is an obvious corollary of Theorems
\ref{thm:lb-dif-op} and \ref{thm:f-queen-k}.
\begin{cor}\label{cor:an-dif-op}
Let $f\in\tclass{k}_{\rm
an}(\Omega^{(0)},\ldots,\Omega^{(k)};\ncspacej{\vecspace{W}}{0},\ldots,\ncspacej{\vecspace{W}}{k})$.
Then
$$\Delta_Rf\in\tclass{k+1}_{\rm an}(\Omega^{(0)},\ldots,\Omega^{(k)},\Omega^{(k)};\ncspacej{\vecspace{W}}{0},\ldots,
\ncspacej{\vecspace{W}}{k},\ncspacej{\vecspace{V}}{k}).$$
\end{cor}
\begin{rem}\label{rem:an-j-dif-op}
An analogue of Corollary \ref{cor:an-dif-op} also holds for
${}_j\Delta_Rf$, see Remarks \ref{rem:j-delta} and
\ref{rem:lb-j-dif-op}.
\end{rem}

Let $\vecspace{W}_0$, \ldots, $\vecspace{W}_k$ be Banach spaces
equipped with admissible systems of rectangular matrix norms, and
let
$\omega\colon\vecspace{W}_1\times\cdots\times\vecspace{W}_k\to\vecspace{W}_0$
be a $k$-linear mapping. For every $n_0$, $n_1$, \ldots,
$n_k\in\mathbb{N}$ we define a $k$-linear mapping
$\omega^{(n_0,\ldots,n_k)}\colon\rmat{\vecspace{W}_1}{n_0}{n_1}\times\cdots
\times\rmat{\vecspace{W}_k}{n_{k-1}}{n_k}\to\rmat{\vecspace{W}_0}{n_0}{n_k}$
by
\begin{equation}\label{eq:k-CS}
\omega^{(n_0,\ldots,n_k)}(W^1,\ldots,W^k)=(W^1\odot\cdots\odot
W^k)\omega,
\end{equation}
where we use the notation \eqref{eq:tensa_prod} and $\omega$ acts
on the matrix $$W^1\odot\cdots\odot
W^k\in\rmat{(\vecspace{W}_1\otimes\cdots\otimes\vecspace{W}_k)}{n_0}{n_k}$$
on the right entrywise, cf. \eqref{eq:diag_tensa_k}. The
$k$-linear mapping $\omega$ is called \emph{completely bounded}
\index{completely bounded $k$-linear mapping} (in the sense of
Christensen and Sinclair, see, e.g., \cite[Chapter 17]{Pa}) if
\index{$\Vert\omega\Vert_{\mathcal{L}^k_\mathrm{cb}}$}
\begin{equation*}
\|\omega\|_{\mathcal{L}^k_\mathrm{cb}}:=
\sup_{n_0,\ldots,n_k\in\mathbb{N}}\|\omega^{(n_0,\ldots,n_k)}\|_{\mathcal{L}^k}<\infty.
\end{equation*}
  We denote by $\mathcal{L}^k_{\mathrm
cb}(\vecspace{W}_1\times\cdots\times\vecspace{W}_k,\vecspace{W}_0)$
\index{$\mathcal{L}^k_{\mathrm
cb}(\vecspace{W}_1\times\cdots\times\vecspace{W}_k,\vecspace{W}_0)$}
the Banach space of such completely bounded $k$-linear mappings.

Recall \cite[Page 185, Exercise 13.2]{Pa} that an operator space
 $\ncspace{\vecspace{W}}$ is equipped with a
unique admissible system of rectangular matrix norms such that the
analogues for rectangular matrices of \eqref{eq:dirsums-norms} and
of \eqref{eq:simprod-norms} (i.e., \eqref{eq:rect-simprod-norms})
hold with all the constants equal to $1$ independently of matrix
sizes. We notice  (see, e.g., \cite[Chapter 17]{Pa} or
\cite[Chapter 5]{Pi}) that if $\vecspace{W}_0$, \ldots,
$\vecspace{W}_k$ are operator spaces, then
$\|\omega\|_{\mathcal{L}^k_\mathrm{cb}}$ coincides with the
completely bounded norm of $\omega$ viewed as a linear mapping
$\vecspace{W}_1\otimes\cdots\otimes\vecspace{W}_k\to\vecspace{W}_0$,
where the vector spaces
$\rmat{(\vecspace{W}_1\otimes\cdots\otimes\vecspace{W}_k)}{n_0}{n_k}$,
$n_0$, $n_k\in\mathbb{N}$,  are endowed with the \emph{Haagerup
norm, $\|\cdot\|_{H,n_0,n_k}$}: \index{Haagerup norm}
\index{$\Vert W\Vert_{H,n_0,n_k}$}
\begin{multline}\label{eq:Haagerup}
\|W\|_{H,n_0,n_k}\\
=\inf\{\|W^1\|_{n_0,n_1}\cdots\|W^k\|_{n_{k-1},n_k}\colon
n_1,\ldots,n_{k-1}\in\mathbb{N},\ W=W^1\odot\cdots\odot W^k\}.
\end{multline}
\begin{prop}\label{prop:k-cb-norm}
Let $\vecspace{V}$ and $\vecspace{W}$ be operator spaces, and let
$\omega\colon\vecspace{V}^k\to\vecspace{W}$ be a $k$-linear
mapping. Then
\begin{equation}\label{eq:k-cb-norm}
\|\omega\|_{\mathcal{L}^k_\mathrm{cb}}=\sup\{\|W^{\odot
k}\omega\|_n \colon n\in\mathbb{N}, W\in\mat{\vecspace{V}}{n},\
\|W\|_{n}=1\}.
\end{equation}
\end{prop} \index{$\Vert\omega\Vert_{\mathcal{L}^k_\mathrm{cb}}$}
\begin{proof}
 The inequality ``$\ge$" is obvious. In order to show that
``$\le$" holds as well, let $n_0$, \ldots, $n_k\in\mathbb{N}$,
$W^1\in\rmat{\vecspace{V}}{n_0}{n_1}$, \ldots,
$W^k\in\rmat{\vecspace{V}}{n_{k-1}}{n_k}$ be arbitrary. Set
$$\widetilde{W}=\begin{bmatrix}
0_{n_0\times n_0} & W^1     & 0_{n_0\times n_2}& \ldots & 0_{n_0\times n_k}\\
\vdots            & \ddots  & \ddots  & \ddots &  \vdots \\
\vdots            &         & \ddots  & \ddots & 0_{n_{k-2}\times n_k}\\
\vdots            &         &         & \ddots &   W^k\\
0_{n_k\times n_0} & \ldots  & \ldots  & \ldots & 0_{n_k\times n_k}
\end{bmatrix}\in\mat{\vecspace{V}}{(n_0+\cdots +n_k)}.$$
Then
\begin{multline*}
\|(W^1\odot\cdots\odot W^k)\omega\|_{n_0\times n_k}\\
=\left\|\begin{bmatrix}
0_{n_0\times n_0} & \ldots  & 0_{n_0\times n_{k-1}} & (W^1\odot\cdots\odot W^k)\omega \\
\vdots            & \ddots  &                       &  0_{n_1\times n_k} \\
\vdots            &         & \ddots              & \vdots \\
0_{n_k\times n_0} & \ldots  & \ldots  &      0_{n_k\times n_k}\\
\end{bmatrix}\right\|_{n_0+\cdots +n_k}\\
=\|\widetilde{W}^{\odot k}\omega \|_{n_0+\cdots +n_k}
\end{multline*}
is not greater than the right-hand side of \eqref{eq:k-cb-norm}.
Taking the supremum of $\|(W^1\odot\cdots\odot
W^k)\omega\|_{n_0\times n_k}$ over all $n_0$, \ldots,
$n_k\in\mathbb{N}$, and all $W^1\in\rmat{\vecspace{V}}{n_0}{n_1}$,
\ldots, $W^k\in\rmat{\vecspace{V}}{n_{k-1}}{n_k}$ of norm 1, we
obtain the desired inequality ``$\le$" in \eqref{eq:k-cb-norm}.
\end{proof}
\begin{rem}\label{rem:k-cb-norm}
Proposition \ref{prop:k-cb-norm} tells us that the evaluation on
the diagonal is an isometric isomorphism from the space
$\mathcal{L}^k_\mathrm{cb}(\vecspace{V}\times\cdots\times\vecspace{V},\vecspace{W})$
to the space of bounded nc homogeneous polynomials of degree $k$
on $\ncspace{\vecspace{V}}$ with values in
$\ncspace{\vecspace{W}}$ and the supremum norm on $B_{\rm
nc}(0,1)$, see the comments in  the end of Chapter \ref{sec:alg}.
This is a nc analogue of the well-known relation between bounded
homogeneous polynomials and bounded symmetric multilinear forms on
Banach spaces, where the isomorphism however is not isometric, see
\cite[Section 26.2]{HiPh} and \cite[Section 2]{Mu}.
\end{rem}

Let $\vecspace{V}_0, \ldots, \vecspace{V}_k, \vecspace{W}_0,
\ldots, \vecspace{W}_k$ be operator spaces, and let
$\Omega^{(j)}\subseteq\ncspacej{\vecspace{V}}{j}$, $j=0,\ldots,k$,
be uniformly open nc sets. We call a higher order nc function
$f\in\tclass{k}(\Omega^{(0)},\ldots,\Omega^{(k)};\ncspacej{\vecspace{W}}{0},\ldots,\ncspacej{\vecspace{W}}{k})$
 \emph{uniformly locally completely bounded}
\index{uniformly locally completely bounded higher order nc
function} if for every $n_0,\ldots,n_k\in\mathbb{N}$ the function
$f|_{\Omega^{(0)}_{n_0}\times\cdots\times\Omega^{(k)}_{n_k}}$ has
values in
$$\mathcal{L}^k_{\mathrm
cb}(\rmat{\vecspace{W}_1}{n_0}{n_1}\times\cdots\times\rmat{\vecspace{W}_k}{n_{k-1}}{n_k},
\rmat{\vecspace{W}_0}{n_0}{n_k})$$ and $f$ is locally bounded in
the uniformly open topology, i.e.,
 for every
$Y^0\in\Omega^{(0)}_{n_0}$, \ldots, $Y^k\in\Omega^{(k)}_{n_k}$
there exists $\delta>0$ such that
$\|f(X^0,\ldots,X^{k})\|_{\mathcal{L}^k_{\mathrm cb}}$ is bounded
for $X^0\in B_{\mathrm nc}(Y^0,\delta)$, \ldots, $X^k\in
B_{\mathrm nc}(Y^k,\delta)$. (Notice that $p\times q$ matrices
over an operator space $\vecspace{W}$ form an operator space in
its own right.)

For $\phi\in\mathcal{L}^k_{\mathrm
cb}(\rmat{\vecspace{W}_1}{s_0}{s_1}\times\cdots\times\rmat{\vecspace{W}_k}{s_{k-1}}{s_k},
\rmat{\vecspace{W}_0}{s_0}{s_k})$ and $r>0$, define a \emph{nc
ball centered at $\phi$ of radius $r$} \index{nc ball} as
\index{$B_{\mathrm nc}(\phi,r)$}
\begin{multline*}
 B_{\mathrm nc}(\phi,r)\\
=\coprod_{m_0,\ldots,m_k=1}^\infty \Big\{
\psi\in\mathcal{L}^k_{\mathrm
cb}(\rmat{\vecspace{W}_1}{s_0m_0}{s_1m_1}\times\cdots\times\rmat{\vecspace{W}_k}{s_{k-1}m_{k-1}}{s_km_k},
\rmat{\vecspace{W}_0}{s_0m_0}{s_km_k}) \colon \\
\|\psi-\phi^{(m_0,\ldots,m_k)}\|_{\mathcal{L}^k_{\mathrm
cb}}<r\Big\}.
\end{multline*}
Notice that here, for
$W^1\in\rmat{\vecspace{W}_1}{s_0m_0}{s_1m_1}$, \ldots,
$W^k\in\rmat{\vecspace{W}_k}{s_{k-1}m_{k-1}}{s_km_k}$,
$$\phi^{(m_0,\ldots,m_k)}(W^1,\ldots,W^k):=(W^1_{s_0,s_2}\odot_{s_1}\cdots_{s_{k-2},s_k}
\odot_{s_{k-1}}W^k)\phi,$$
\index{$\phi^{(m_0,\ldots,m_k)}(W^1,\ldots,W^k)$}see
\eqref{eq:tensa_prod} and \eqref{eq:k-CS}.
  Analogously to
Proposition \ref{prop:nc_top}, nc balls form a basis for a
topology on
$$\coprod_{n_0,\ldots,n_k=1}^\infty\mathcal{L}^k_{\mathrm
cb}(\rmat{\vecspace{W}_1}{n_0}{n_1}\times\cdots\times\rmat{\vecspace{W}_k}{n_{k-1}}{n_k},
\rmat{\vecspace{W}_0}{n_0}{n_k}),$$ which will be called the
\emph{uniformly-open topology}. \index{uniformly-open topology}
Analogously to Proposition \ref{prop:cont-bdd}, if a higher order
nc function
$f\in\tclass{k}(\Omega^{(0)},\ldots,\Omega^{(k)};\ncspacej{\vecspace{W}}{0},\ldots,\ncspacej{\vecspace{W}}{k})$
is continuous with respect to the product topology on
$$\underset{k+1\
\mathrm{times}}{\underbrace{\ncspacej{\vecspace{V}}{0}\times\cdots\times\ncspacej{\vecspace{V}}{k}}}$$
(of the uniformly-open topologies on $\ncspacej{\vecspace{V}}{j}$)
and the uniformly-open topology on
$$\coprod_{n_0,\ldots,n_k=1}^\infty\mathcal{L}^k_{\mathrm
cb}(\rmat{\vecspace{W}_1}{n_0}{n_1}\times\cdots\times\rmat{\vecspace{W}_k}{n_{k-1}}{n_k},
\rmat{\vecspace{W}_0}{n_0}{n_k}),$$ then $f$ is uniformly locally
completely bounded.

 A higher order nc function
 $f\in\tclass{k}(\Omega^{(0)},\ldots,\Omega^{(k)};\ncspacej{\vecspace{W}}{0},\ldots,\ncspacej{\vecspace{W}}{k})$
 is called \emph{completely bounded G\^{a}teaux}
 \emph{(G$_{\mathrm{cb}}$-) differentiable} \index{completely bounded G\^{a}teaux
(G$_{\mathrm{cb}}$-) differentiable higher order nc function} if
for every
 $n_0,\ldots,n_k\in\mathbb{N}$ the function
 $f|_{\Omega^{(0)}_{n_0}\times\cdots\times\Omega^{(k)}_{n_k}}$ has values in
 $\mathcal{L}^k_{\mathrm
 cb}(\rmat{\vecspace{W}_1}{n_0}{n_1}\times\cdots\times\rmat{\vecspace{W}_k}{n_{k-1}}{n_k},
 \rmat{\vecspace{W}_0}{n_0}{n_k})$ and is G-differentiable, i.e.,
 for every $Y^0\in\Omega^{(0)}_{n_0}$, \ldots, $Y^k\in\Omega^{(k)}_{n_k}$,
 $Z^0\in\mat{\vecspace{V}_0}{n_0}$, \ldots,
 $Z^k\in\mat{\vecspace{V}_k}{n_k}$, the G-derivative
 \eqref{eq:g-der-k} exists in the norm of
 $$\mathcal{L}^k_{\mathrm{cb}}(\rmat{\vecspace{W}_1}{n_0}{n_1}\times
 \cdots\times\rmat{\vecspace{W}_k}{n_{k-1}}{n_k},
 \rmat{\vecspace{W}_0}{n_0}{n_k}).$$


A higher order nc function
$f\in\tclass{k}(\Omega;\ncspacej{\vecspace{W}}{0},\ldots,\ncspacej{\vecspace{W}}{k})$
is called \emph{uniformly completely bounded (uniformly cb-)
analytic} \index{uniformly completely bounded analytic nc
function} if $f$ is uniformly locally completely bounded and
G$_\mathrm{cb}$-differentiable.


\begin{thm}\label{thm:ulcb-dif-op}
If
$f\in\tclass{k}(\Omega^{(0)},\ldots,\Omega^{(k)};\ncspacej{\vecspace{W}}{0},\ldots,\ncspacej{\vecspace{W}}{k})$
is uniformly locally completely bounded, then so is
$$\Delta_Rf\in\tclass{k+1}(\Omega^{(0)},\ldots,\Omega^{(k)},\Omega^{(k)};\ncspacej{\vecspace{W}}{0},\ldots,
\ncspacej{\vecspace{W}}{k},\ncspacej{\vecspace{V}}{k}).$$
\end{thm}
\begin{proof}
Let $s_0,\ldots,s_{k+1}\in\mathbb{N}$, and let
$Y^0\in\Omega^{(0)}_{s_0}$, \ldots, $Y^k\in\Omega^{(k)}_{s_k}$,
$Y^{k+1}\in\Omega^{(k)}_{s_{k+1}}$. Since $f$ is uniformly locally
completely bounded, there exists $\delta>0$ such that
\begin{equation}\label{eq:zveryuga}\|
(f(X^0,\ldots,X^{k-1},X))^{(n_0,\ldots,n_{k-1},n)}
(W^1,\ldots,W^{k-1},W)\|_{s_0m_0n_0,(s_k+s_{k+1})mn}
\end{equation}
is bounded for all $m_0$, \ldots, $m_{k-1}$, $m$, $n_0$, \ldots,
$n_{k-1}$, $n\in\mathbb{N}$, $X^0\in B_{\mathrm
nc}(Y^0,\delta)_{s_0m_0}$, \ldots, $X^{k-1}\in B_{\mathrm
nc}(Y^{k-1},\delta)_{s_{k-1}m_{k-1}}$, $X\in B_{\mathrm
nc}(Y^k\oplus Y^{k+1},\delta)_{(s_k+s_{k+1})m}$,
$W^1\in\rmat{\vecspace{W}_1}{s_0m_0n_0}{s_1m_1n_1}$, \ldots,
$W^{k-1}\in\rmat{(\vecspace{W}_{k-1})}{s_{k-2}m_{k-2}n_{k-2}}{s_{k-1}m_{k-1}n_{k-1}}$,
and\\
$W\in\rmat{\vecspace{W}_k}{s_{k-1}m_{k-1}n_{k-1}}{(s_k+s_{k+1})mn}$,
with
\begin{multline*}
\|W^1\|_{s_0m_0n_0,s_1m_1n_1}=\ldots =
\|W^{k-1}\|_{s_{k-2}m_{k-2}n_{k-2},s_{k-1}m_{k-1}n_{k-1}}\\
=
\|W\|_{s_{k-1}m_{k-1}n_{k-1},(s_k+s_{k+1})mn}=1.
\end{multline*}
By definition, \eqref{eq:zveryuga} is equal to
\begin{multline*}
\|(W^1_{s_0m_0,s_2m_2}\odot_{s_1m_1}\cdots_{s_{k-3}m_{k-3},s_{k-1}m_{k-1}}
\odot_{s_{k-2}m_{k-2}}W^{k-1}\\
{}_{s_{k-2}m_{k-2},(s_k+s_{k+1})m}
\odot_{s_{k-1}m_{k-1}}W)f(X^0,\ldots,
X^{k-1},X)\|_{s_0m_0n_0,(s_k+s_{k+1})mn}.
\end{multline*}
Using \eqref{eq:diag_tensa_k}, we rewrite this as
\begin{equation*}
\Big\|f\Big(\bigoplus_{\alpha=1}^{n_0}X^0,\ldots,
\bigoplus_{\alpha=1}^{n_{k-1}}X^{k-1},\bigoplus_{\alpha=1}^{n}X\Big)(W^1,\ldots,W^{k-1},W)
\Big\|_{s_0m_0n_0,(s_k+s_{k+1})mn}.
\end{equation*}

 Since
\begin{multline*}
\begin{bmatrix} \bigoplus\limits_{\alpha=1}^{n_k}X^k & \frac{\delta}{2}W^{k+1} & 0\\
0 & \bigoplus\limits_{\alpha=1}^{n_{k+1}}X^{k+1} & 0\\
0 & 0 & \bigoplus\limits_{\alpha=1}^{|m_{k+1}n_{k+1}-m_kn_k|}Y
\end{bmatrix}\\
\in B_\mathrm{nc}(Y^k\oplus Y^{k+1},\delta)_{(s_k+s_{k+1})\max\{
m_kn_k,m_{k+1}n_{k+1}\}}
\end{multline*}
after a permutation of block rows and columns
 whenever
$X^k\in B_\mathrm{nc}(Y^k,\frac{\delta}{2})_{s_km_k}$, $X^{k+1}\in
B_{\mathrm nc}(Y^{k+1},\frac{\delta}{2})_{s_{k+1}m_{k+1}}$, and
$W^{k+1}\in\rmat{\vecspace{V}}{s_km_kn_k}{s_{k+1}m_{k+1}n_{k+1}}$
with \\ $\|W^{k+1}\|_{s_km_kn_k,s_{k+1}m_{k+1}n_{k+1}}=1$, where
\begin{equation}\label{eq:Y}
Y=\left\{\begin{matrix}
Y^k & \mathrm{if\ \ } m_kn_k<m_{k+1}n_{k+1},\\
Y^{k+1} & \mathrm{if\ \ } m_kn_k>m_{k+1}n_{k+1},\\
\end{matrix}\right.
\end{equation}
 we obtain that
\begin{multline*}
\Big\|f\left(\bigoplus\limits_{\alpha=1}^{n_0}X^0,\ldots,\bigoplus\limits_{\alpha=1}^{n_{k-1}}X^{k-1},
\begin{bmatrix} \bigoplus\limits_{\alpha=1}^{n_k}X^k & \frac{\delta}{2}W^{k+1} & 0\\
0 & \bigoplus\limits_{\alpha=1}^{n_{k+1}}X^{k+1} & 0\\
0 & 0 & \bigoplus\limits_{\alpha=1}^{|m_{k+1}n_{k+1}-m_kn_k|}Y
\end{bmatrix}\right)\\
(W^1,\ldots,W^{k-1},\row[W^k,0,0])\Big\|_{s_0m_0n_0,(s_k+s_{k+1})\max\{
m_kn_k,m_{k+1}n_{k+1}\}}
\end{multline*}
is bounded for all $m_0$, \ldots, $m_{k+1}$, $n_0$, \ldots,
$n_{k+1}\in\mathbb{N}$,  $X^0\in
B_\mathrm{nc}(Y^0,\frac{\delta}{2})_{s_0m_0}$, \ldots, $X^{k+1}\in
B_{\mathrm nc}(Y^{k+1},\frac{\delta}{2})_{s_{k+1}m_{k+1}}$, and
$W^1\in\rmat{\vecspace{W}_1}{s_0m_0n_0}{s_1m_1n_1}$, \ldots,
$W^{k+1}\in\rmat{(\vecspace{W}_{k+1})}{s_{k}m_{k}n_{k}}{s_{k+1}m_{k+1}n_{k+1}}$,
with
$$\|W^1\|_{s_0m_0n_0,s_1m_1n_1}=\ldots =
\|W^{k+1}\|_{s_{k}m_{k}n_{k},s_{k+1}m_{k+1}n_{k+1}}=1.
$$
Here we used the isometricity of the block injection
\begin{multline*}
W^k\in\rmat{\vecspace{W}_k}{s_{k-1}m_{k-1}n_{k-1}}{s_km_kn_k}\longmapsto
\row[W^k,0,0]\\
\in\rmat{\vecspace{W}_k}{s_{k-1}m_{k-1}n_{k-1}}{(s_k+s_{k+1})\max\{
m_kn_k,m_{k+1}n_{k+1}\}}.
\end{multline*}
 Therefore,
\begin{multline*}
\Big\|\Delta_Rf\Big(\bigoplus_{\alpha=1}^{n_0}X^0,\ldots,\bigoplus_{\alpha=1}^{n_{k+1}}X^{k+1}\Big)
(W^1,\ldots,W^{k+1})\Big\|_{s_0m_0n_0,s_{k+1}m_{k+1}n_{k+1}}\\
=\frac{2}{\delta}\Big\|f\left(\bigoplus\limits_{\alpha=1}^{n_0}X^0,\ldots,
\bigoplus\limits_{\alpha=1}^{n_{k-1}}X^{k-1},
\begin{bmatrix} \bigoplus\limits_{\alpha=1}^{n_k}X^k & \frac{\delta}{2}W^{k+1} & 0\\
0 & \bigoplus\limits_{\alpha=1}^{n_{k+1}}X^{k+1} & 0\\
0 & 0 & \bigoplus\limits_{\alpha=1}^{|m_{k+1}n_{k+1}-m_kn_k|}Y
\end{bmatrix}\right)\\
(W^1,\ldots,W^{k-1},\row[W^k,0,0])
\begin{bmatrix}
0\\
I_{s_{k+1}m_{k+1}n_{k+1}}\\
0
\end{bmatrix}\Big\|_{s_0m_0n_0,s_{k+1}m_{k+1}n_{k+1}}
\end{multline*}
is bounded for all $m_0$, \ldots, $m_{k+1}$, $n_0$, \ldots,
$n_{k+1}\in\mathbb{N}$,  $X^0\in
B_\mathrm{nc}(Y^0,\frac{\delta}{2})_{s_0m_0}$, \ldots, $X^{k+1}\in
B_{\mathrm nc}(Y^{k+1},\frac{\delta}{2})_{s_{k+1}m_{k+1}}$, and
$W^1\in\rmat{\vecspace{W}_1}{s_0m_0n_0}{s_1m_1n_1}$, \ldots,
$W^{k+1}\in\rmat{(\vecspace{W}_{k+1})}{s_{k}m_{k}n_{k}}{s_{k+1}m_{k+1}n_{k+1}}$,
with
$$\|W^1\|_{s_0m_0n_0,s_1m_1n_1}=\ldots =
\|W^{k+1}\|_{s_{k}m_{k}n_{k},s_{k+1}m_{k+1}n_{k+1}}=1.
$$ Here we used the contractivity of the block projection
implemented by the multiplication with $$\begin{bmatrix}
0\\
I_{s_{k+1}m_{k+1}n_{k+1}}\\
0
\end{bmatrix}$$ on the right.

Applying the same argument as in the beginning of the proof, for
$\Delta_Rf$ in the place of $f$, we obtain
\begin{multline*}
\|(\Delta_Rf(X^0,\ldots,X^{k+1}))^{(n_0,\ldots,n_{k+1})}
(W^1,\ldots,W^{k+1})\|_{s_0m_0n_0,s_{k+1}m_{k+1}n_{k+1}}\\
=\Big\|\Delta_Rf\Big(\bigoplus_{\alpha=1}^{n_0}X^0,\ldots,\bigoplus_{\alpha=1}^{n_{k+1}}X^{k+1}\Big)
(W^1,\ldots,W^{k+1})\Big\|_{s_0m_0n_0,s_{k+1}m_{k+1}n_{k+1}},
\end{multline*}
which by the result of the preceding paragraph implies the uniform
local complete boundedness of $\Delta_Rf$.
\end{proof}
\begin{rem}\label{rem:ulcb-j-dif-op}
An analogue of Theorem \ref{thm:ulcb-dif-op} also holds, with an
analogous proof, for ${}_j\Delta_Rf$, see Remark
\ref{rem:j-delta}.
\end{rem}
\begin{thm}\label{thm:king-k}
If
$f\in\tclass{k}(\Omega^{(0)},\ldots,\Omega^{(k)};\ncspacej{\vecspace{W}}{0},\ldots,\ncspacej{\vecspace{W}}{k})$
is uniformly locally completely bounded, then $f$ is continuous
with respect to the uniformly-open topologies and is uniformly cb
analytic.
\end{thm}
\begin{proof}
Let $s_0,\ldots,s_{k+1}\in\mathbb{N}$, and let
$Y^0\in\Omega^{(0)}_{s_0}$, \ldots, $Y^k\in\Omega^{(k)}_{s_k}$,
$Y^{k+1}\in\Omega^{(k)}_{s_{k+1}}$. Since $f$ is uniformly locally
completely bounded, according to Remark \ref{rem:ulcb-j-dif-op} so
are ${}_j\Delta_Rf$, $j=0$, \ldots, $k$. Therefore, there exist
$\delta>0$ and $C>0$ such that
\begin{multline*}
\|({}_j\Delta_Rf(X^0,\ldots,X^{k+1}))^{(n_0,\ldots,n_{k+1})}
(W^1,\ldots,W^{k+1})\|_{s_0m_0n_0,s_{k+1}m_{k+1}n_{k+1}}\\
=\Big\|{}_j\Delta_Rf\Big(\bigoplus_{\alpha=1}^{n_0}X^0,\ldots,
\bigoplus_{\alpha=1}^{n_{k+1}}X^{k+1}\Big)
(W^1,\ldots,W^{k+1})\Big\|_{s_0m_0n_0,s_{k+1}m_{k+1}n_{k+1}}\le C
\end{multline*}
for $j=0$, \ldots, $k$, and for all  $m_0$, \ldots, $m_{k+1}$,
$n_0$, \ldots, $n_{k+1}\in\mathbb{N}$,  $X^0\in
B_\mathrm{nc}(Y^0,\delta)_{s_0m_0}$, \ldots, $X^{k+1}\in
B_{\mathrm nc}(Y^{k+1},\delta)_{s_{k+1}m_{k+1}}$, and
$W^1\in\rmat{\vecspace{W}_1}{s_0m_0n_0}{s_1m_1n_1}$, \ldots,
$W^{k+1}\in\rmat{(\vecspace{W}_{k+1})}{s_{k}m_{k}n_{k}}{s_{k+1}m_{k+1}n_{k+1}}$,
with
$$\|W^1\|_{s_0m_0n_0,s_1m_1n_1}=\ldots =
\|W^{k+1}\|_{s_{k}m_{k}n_{k},s_{k+1}m_{k+1}n_{k+1}}=1.
$$
We have by \eqref{eq:j-RightLagr_k} (cf. \eqref{eq:slice-Lagrange})
\begin{multline*}
\Big\|f\Big(\bigoplus_{\alpha=1}^{n_0}X^0,\ldots,\bigoplus_{\alpha=1}^{n_k}X^{k}\Big)
(W^1,\ldots,W^{k})\\
-f\Big(\bigoplus_{\alpha=1}^{m_0n_0}Y^0,\ldots,\bigoplus_{\alpha=1}^{m_kn_k}
Y^{k}\Big)(W^1,\ldots,W^{k})\Big\|_{s_0m_0n_0,s_km_kn_k}\\
=\Big\|\sum_{j=0}^k\Bigg(f\Big(\bigoplus_{\alpha=1}^{n_0}X^0,\ldots,\bigoplus_{\alpha=1}^{n_j}X^{j},
\bigoplus_{\alpha=1}^{m_{j+1}n_{j+1}}Y^{j+1},\ldots,
\bigoplus_{\alpha=1}^{m_kn_k}Y^k\Big)(W^1,\ldots,W^{k})\hfill\\
-
f\Big(\bigoplus_{\alpha=1}^{n_0}X^0,\ldots,\bigoplus_{\alpha=1}^{n_{j-1}}X^{j-1},
\bigoplus_{\alpha=1}^{m_jn_j}Y^j,\ldots,\bigoplus_{\alpha=1}^{m_kn_k}Y^{k}\Big)
(W^1,\ldots,W^{k})\Bigg)\Big\|_{s_0m_0n_0,s_km_kn_k}\\
=\Big\|\sum_{j=0}^k
{}_j\Delta_Rf\Big(\bigoplus_{\alpha=1}^{n_0}X^0,\ldots,
\bigoplus_{\alpha=1}^{n_j}X^{j},\bigoplus_{\alpha=1}^{m_jn_j}Y^j,
\ldots,\bigoplus_{\alpha=1}^{m_kn_k}Y^k\Big)\hfill\\
\Big(W^1,\ldots,W^{j-1},\bigoplus_{\alpha=1}^{n_j}X^j-\bigoplus_{\alpha=1}^{m_jn_j}Y^j,W^j,
\ldots,W^k\Big)\Big\|_{s_0m_0n_0,s_km_kn_k}\\
\le
C\sum_{j=0}^k\Big\|X^j-\bigoplus_{\alpha=1}^{m_j}Y^j\|_{s_jm_j}
\end{multline*}
for all  $m_0$, \ldots, $m_{k}$, $n_0$, \ldots,
$n_{k}\in\mathbb{N}$,  $X^0\in
B_\mathrm{nc}(Y^0,\delta)_{s_0m_0}$, \ldots, $X^{k}\!\in\!
B_{\mathrm nc}(Y^{k},\delta)_{s_{k}m_{k}}$, and
$W^1\in\rmat{\vecspace{W}_1}{s_0m_0n_0}{s_1m_1n_1}$, \ldots,
$W^{k}\in\rmat{(\vecspace{W}_{k})}{s_{k-1}m_{k-1}n_{k-1}}{s_{k}m_{k}n_{k}}$,
with
$$\|W^1\|_{s_0m_0n_0,s_1m_1n_1}=\ldots =
\|W^{k}\|_{s_{k-1}m_{k-1}n_{k-1},s_{k}m_{k}n_{k}}=1.
$$
Given $\epsilon>0$, we set
$\eta=\min\{\delta,\frac{\epsilon}{C(k+1)}\}$ and obtain that
\begin{multline*}
\Big\|f\Big(\bigoplus_{\alpha=1}^{n_0}X^0,\ldots,\bigoplus_{\alpha=1}^{n_k}X^{k}\Big)
(W^1,\ldots,W^{k})\\
-f\Big(\bigoplus_{\alpha=1}^{m_0n_0}Y^0,\ldots,\bigoplus_{\alpha=1}^{m_kn_k}
Y^{k}\Big)(W^1,\ldots,W^{k})\Big\|_{s_0m_0n_0,s_km_kn_k}<\epsilon
\end{multline*}
for all $m_0$, \ldots, $m_{k}$, $n_0$, \ldots,
$n_{k}\in\mathbb{N}$,  $X^0\in B_\mathrm{nc}(Y^0,\eta)_{s_0m_0}$,
\ldots, $X^{k}\!\in\! B_{\mathrm nc}(Y^{k},\eta)_{s_{k}m_{k}}$,
and $W^1\in\rmat{\vecspace{W}_1}{s_0m_0n_0}{s_1m_1n_1}$, \ldots,
$W^{k}\in\rmat{\vecspace{W}_{k}}{s_{k-1}m_{k-1}n_{k-1}}{s_{k}m_{k}n_{k}}$,
with
$$\|W^1\|_{s_0m_0n_0,s_1m_1n_1}=\ldots =
\|W^{k}\|_{s_{k-1}m_{k-1}n_{k-1},s_{k}m_{k}n_{k}}=1.
$$
Thus $f$ is continuous with respect to the uniformly-open
topologies.

By Remark \ref{rem:ulcb-j-dif-op}, ${}_j\Delta_Rf$ is uniformly
locally completely bounded for $j=0$, \ldots, $k$. Applying the
first part of this theorem to ${}_j\Delta_Rf$, we obtain that
${}_j\Delta_Rf$ is continuous with respect to the uniformly-open
topologies. Let $s_0,\ldots,s_{k+1}\in\mathbb{N}$, and let
$Y^0\in\Omega^{(0)}_{s_0}$, \ldots,
$Y^{k}\in\Omega^{(k)}_{s_{k}}$. By \eqref{eq:slice-Lagrange} we
have
\begin{multline}\label{eq:cb-Lagrange}
f\Big(\bigoplus_{\alpha=1}^{n_0}Y^0+t\bigoplus_{\alpha=1}^{n_0}Z^0,\ldots,\bigoplus_{\alpha=1}^{n_k}Y^{k}
+t\bigoplus_{\alpha=1}^{n_k}Z^{k}\Big)(W^1,\ldots,W^{k})\\
\hfill -f\Big(\bigoplus_{\alpha=1}^{n_0}Y^0,\ldots,
\bigoplus_{\alpha=1}^{n_k}Y^{k}\Big)(W^1,\ldots,W^{k})\\
=t\sum_{j=0}^k
{}_j\Delta_Rf\Big(\bigoplus_{\alpha=1}^{n_0}Y^0+t\bigoplus_{\alpha=1}^{n_0}Z^0,
\ldots,
\bigoplus_{\alpha=1}^{n_j}Y^{j}+t\bigoplus_{\alpha=1}^{n_j}Z^{j},\bigoplus_{\alpha=1}^{n_j}Y^j,
\ldots,\bigoplus_{\alpha=1}^{n_k}Y^k\Big)\\
\Big(W^1,\ldots,W^j,\bigoplus_{\alpha=1}^{n_j}Z^j,W^{j+1},\ldots,W^k\Big)
\end{multline}
for any $n_0$, \ldots, $n_k\in\mathbb{N}$,
$Z^0\in\mat{\vecspace{V}_0}{s_0}$ \ldots,
$Z^k\in\mat{\vecspace{V}_k}{s_k}$, and
$W^1\in\rmat{\vecspace{W}_1}{s_0n_0}{s_1n_1}$, \ldots,
$W^{k}\in\rmat{\vecspace{W}_{k}}{s_{k-1}n_{k-1}}{s_{k}n_{k}}$,
with
$$\|W^1\|_{s_0n_0,s_1n_1}=\ldots =
\|W^{k}\|_{s_{k-1}n_{k-1},s_{k}n_{k}}=1,
$$
and $t$ small enough. Dividing both parts of
\eqref{eq:cb-Lagrange} by $t$ and letting $t\to 0$, we obtain (cf.
\eqref{eq:slice-der-k})
\begin{multline*}
\frac{d}{dt}f\Big(\bigoplus_{\alpha=1}^{n_0}Y^0+t\bigoplus_{\alpha=1}^{n_0}Z^0,\ldots,
\bigoplus_{\alpha=1}^{n_k}Y^{k}+t\bigoplus_{\alpha=1}^{n_k}Z^{k}\Big)(W^1,\ldots,W^{k})\Big|_{t=0}\\
=\sum_{j=0}^k{}_j\Delta_Rf
\Big(\bigoplus_{\alpha=1}^{n_0}Y^0,\ldots,\bigoplus_{\alpha=1}^{n_{j-1}}Y^{j-1},
\bigoplus_{\alpha=1}^{n_j}Y^j,
\bigoplus_{\alpha=1}^{n_j}Y^j,\bigoplus_{\alpha=1}^{n_{j+1}}Y^{j+1},\ldots,
\bigoplus_{\alpha=1}^{n_k}Y^k\Big)\\
\Big(W^1,\ldots,W^{j-1},\bigoplus_{\alpha=1}^{n_j}Z^j,W^{j+1},\ldots,W^k\Big)
\end{multline*}
where the limit is uniform in $n_0$, \ldots, $n_k\in\mathbb{N}$,
and in $W^1\in\rmat{\vecspace{W}_1}{s_0n_0}{s_1n_1}$, \ldots,
$W^{k}\in\rmat{\vecspace{W}_{k}}{s_{k-1}n_{k-1}}{s_{k}n_{k}}$,
with
$$\|W^1\|_{s_0n_0,s_1n_1}=\ldots =
\|W^{k}\|_{s_{k-1}n_{k-1},s_{k}n_{k}}=1.
$$
Therefore, the G-derivative
$$\delta f(Y^0,\ldots, Y^k)(Z^0,\ldots, Z^k)=\frac{d}{dt}f\left(Y^0+tZ^0,\ldots,
Y^{k}+tZ^{k}\right)\Big|_{t=0}$$ (see \eqref{eq:g-der-k}) exists
in the norm of
$$\mathcal{L}^k_{\mathrm{cb}}(\rmat{\vecspace{W}_1}{s_0}{s_1}
\times\cdots\times\rmat{\vecspace{W}_k}{s_{k-1}}{s_k},
 \rmat{\vecspace{W}_0}{s_0}{s_k}),$$
 i.e., $f$ is G$_\mathrm{cb}$-differentiable.
Since $f$ is also uniformly locally completely bounded, we
conclude that $f$ is uniformly cb analytic.
\end{proof}

It follows from Theorem \ref{thm:king-k} that
 a higher order nc function
$f$ is uniformly locally completely bounded if and only if $f$ is
continuous with respect to the uniformly-open topologies if and
only if $f$ is uniformly cb analytic. We denote by
$$\tclass{k}_{\mathrm{uan}}=\tclass{k}_{\mathrm{uan}}(\Omega^{(0)},\ldots,\Omega^{(k)};\ncspacej{\vecspace{W}}{0},
\ldots,\ncspacej{\vecspace{W}}{k})$$
\index{$\tclass{k}_{\mathrm{uan}}(\Omega^{(0)},\ldots,\Omega^{(k)};\ncspacej{\vecspace{W}}{0},
\ldots,\ncspacej{\vecspace{W}}{k})$}the subclass of
$\tclass{k}(\Omega^{(0)},\ldots,\Omega^{(k)};\ncspacej{\vecspace{W}}{0},\ldots,\ncspacej{\vecspace{W}}{k})$
consisting of nc functions of order $k$ that are uniformly cb
analytic.

The following statement is an obvious corollary of Theorems
\ref{thm:ulcb-dif-op} and \ref{thm:king-k}.
\begin{cor}\label{cor:ucb-an-dif-op}
Let
$f\in\tclass{k}_{\mathrm{uan}}(\Omega^{(0)},\ldots,\Omega^{(k)};\ncspacej{\vecspace{W}}{0},\ldots,
\ncspacej{\vecspace{W}}{k})$. Then
$$\Delta_Rf\in\tclass{k+1}_{\mathrm{uan}}(\Omega^{(0)},\ldots,\Omega^{(k)},\Omega^{(k)};\ncspacej{\vecspace{W}}{0},
\ldots,\ncspacej{\vecspace{W}}{k},\ncspacej{\vecspace{V}}{k}).$$
\end{cor}
\begin{rem}\label{rem:ucb-an-j-dif-op}
An analogue of Corollary \ref{cor:ucb-an-dif-op} also holds for
${}_j\Delta_Rf$, see Remarks \ref{rem:j-delta} and
\ref{rem:ulcb-j-dif-op}.
\end{rem}

\begin{rem}\label{rem:real-higher}
The results of this section admit analogues in the real case ---
compare Remarks \ref{rem:real_queen} and \ref{rem:real_king}.
Theorem \ref{thm:lbs-dif-op}, Remark \ref{rem:lbs-j-dif-op},
Theorem \ref{thm:g-queen-k}, Corollary \ref{cor:g-dif-op}, and
Remark \ref{rem:g-j-dif-op} are true for the case where
$\vecspace{V}_0$, \ldots, $\vecspace{V}_k$, $\vecspace{W}_1$,
\ldots, $\vecspace{W}_k$ are vector spaces over $\mathbb{R}$ and
$\vecspace{W}_0$ is a Banach space over $\mathbb{R}$ equipped with
an admissible system of rectangular matrix norms (all the notions
are defined exactly as in the complex case). If
$f\in\tclass{k}(\Omega^{(0)},\ldots,\Omega^{(k)};\ncspacej{\vecspace{W}}{0},\ldots,\ncspacej{\vecspace{W}}{k})$
is G$_W$-differentiable, it does not necessarily follow that $f$
is $W$-analytic on slices, nor that $f$ is infinitely many times
differentiable on
\begin{equation}\label{eq:affine-higher}
(\vecspace{U}_0\cap\Omega^{(0)}_{n_0})\times\cdots\times(\vecspace{U}_k\cap\Omega^{(k)}_{n_k})\times
\vecspace{Y}_1\times\cdots\times\vecspace{Y}_k
\end{equation}
 as a
function of several real variables for every
$n_0,\ldots,n_k\in\mathbb{N}$ and for all finite-dimensional
subspaces $\vecspace{U}_0\subseteq\mat{\vecspace{V}_0}{n_0}$,
\ldots, $\vecspace{U}_k\subseteq\mat{\vecspace{V}_k}{n_k}$,
$\vecspace{Y}_1\subseteq\rmat{\vecspace{W}}{n_0}{n_1}$, \ldots,
$\vecspace{Y}_k\subseteq\rmat{\vecspace{W}}{n_{k-1}}{n_k}$. Assume
in addition that $f$ is \emph{$W$-locally bounded on affine
finite-dimensional subspaces}, \index{W @$W$-locally bounded on
affine finite-dimensional subspaces higher order nc function}
i.e., that
$f|_{\Omega^{(0)}_{n_0}\times\cdots\times\Omega^{(k)}_{n_k}}$ is
locally bounded on affine finite-dimensional subspaces, similarly
to Remark \ref{rem:real_queen}, for every fixed $W_1$, \ldots,
$W_k$. Then
$\Delta_Rf\in\tclass{k+1}(\Omega^{(0)},\ldots,\Omega^{(k)},\Omega^{(k)};\ncspacej{\vecspace{W}}{0},
\ldots,\ncspacej{\vecspace{W}}{k},\ncspacej{\vecspace{V}}{k})$ is
also $W$-locally bounded on affine finite-dimensi\-onal subspaces
and $f$ is infinitely many times differentiable on the sets
\eqref{eq:affine-higher} --- the proof is similar to the proofs of
Theorems \ref{thm:lbs-dif-op} and \ref{thm:g-queen-k}.

Theorem \ref{thm:lb-dif-op} and Remark \ref{rem:lb-j-dif-op} are
true for the case where $\vecspace{V}_0$, \ldots,
$\vecspace{V}_k$, $\vecspace{W}_0$, \ldots, $\vecspace{W}_k$ are
Banach spaces over $\mathbb{R}$ equipped with an admissible system
of rectangular matrix norms. The analogue of Theorem
\ref{thm:g-queen-k} is the following: if
$f\in\tclass{k}(\Omega^{(0)},\ldots,\Omega^{(k)};\ncspacej{\vecspace{W}}{0},\ldots,\ncspacej{\vecspace{W}}{k})$
is locally bounded, then $f$ is $G$-differentiable, and
furthermore, $f$ is infinitely many times differentiable on
$(\vecspace{U}_0\cap\Omega^{(0)}_{n_0})\times\cdots\times(\vecspace{U}_k\cap\Omega^{(k)}_{n_k})$
 as a
function of several real variables for every
$n_0,\ldots,n_k\in\mathbb{N}$ and for all finite-dimensional
subspaces $\vecspace{U}_0\subseteq\mat{\vecspace{V}_0}{n_0}$,
\ldots, $\vecspace{U}_k\subseteq\mat{\vecspace{V}_k}{n_k}$, the
derivative being taken in the norm of
$\mathcal{L}^k(\rmat{\vecspace{W}_1}{n_0}{n_1}\times\cdots\times\rmat{\vecspace{W}_k}{n_{k-1}}{n_k},
\rmat{\vecspace{W}_0}{n_0}{n_k})$.

Theorem \ref{thm:ulcb-dif-op}, Remark \ref{rem:ulcb-j-dif-op},
Theorem \ref{thm:king-k}, Corollary \ref{cor:ucb-an-dif-op}, and
Remark \ref{rem:ucb-an-j-dif-op} are true for the case where
$\vecspace{V}_0$, \ldots, $\vecspace{V}_k$, $\vecspace{W}_0$,
\ldots, $\vecspace{W}_k$ are operator spaces over $\mathbb{R}$.
\end{rem}

\label{K-TOPOLOGIES}

\chapter{Convergence of nc power series}\label{sec:nc-ps}
In this chapter, we shall study the convergence of power series of
the form
\begin{equation}\label{eq:power-series}
\sum_{\ell=0}^\infty
\Big(X-\bigoplus_{\alpha=1}^mY\Big)^{\odot_s\ell}f_\ell
\end{equation}
to a nc function. Here $\vecspace{V}$ and $\vecspace{W}$ are
vector spaces over $\mathbb{C}$; $Y\in\mat{\vecspace{V}}{s}$ is a
given center; $X\in\mat{\vecspace{V}}{sm}$, $m=1,2,\ldots$;
$f_\ell\colon(\mat{\vecspace{V}}{s})^{\ell}\to\mat{\vecspace{W}}{s}$,
$\ell=0,1,\ldots$, is a given sequence of $\ell$-linear mappings,
i.e., a linear mapping
$f\colon\tensa{\mat{\vecspace{V}}{s}}\to\mat{\vecspace{W}}{s}$;
 and conditions
\eqref{eq:ncfun_coef_0}--\eqref{eq:ncfun-coef_ell_ell} are
satisfied. Notice that for $s=1$ these conditions are trivially
satisfied for any sequence $f_\ell$ of linear mappings. We shall
be interested in the growth conditions on the coefficients
$f_\ell$ under which the sum in \eqref{eq:power-series} is an
analytic or a uniformly analytic nc function. The results of this
section are the converse of the results of Sections
\ref{subsec:analytic} and \ref{subsec:u-analytic}, much as the
results of Section \ref{subsec:from_ncps_to_ncfun} is the converse
of the results of Section \ref{subsec: from_ncfun_to_ncps} where
we evaluate power series on nilpotent matrices, so that no
questions of convergence arise.

\section{Finitely open topology}\label{subsec:conv-fin-open}
Let $\vecspace{V}$ be a vector space over $\mathbb{C}$, let
$\vecspace{W}$ be a Banach space over $\mathbb{C}$ equipped with
an admissible system of matrix norms (see Section
\ref{subsec:analytic}), and let
$f\colon\tensa{\mat{\vecspace{V}}{s}}\to\mat{\vecspace{W}}{s}$ be
a linear mapping. Notice that, for $Z\in\mat{\vecspace{V}}{sm}$,
$$Z^{\odot_s\ell}f_\ell=f_\ell^{(m,\ldots,m)}(Z,\ldots,Z)$$
(see \eqref{eq:k-CS}) is a homogeneous polynomial of degree $\ell$ on $\mat{\vecspace{V}}{sm}$ with values in
$\mat{\vecspace{W}}{sm}$ and the series $\eqref{eq:power-series}$ is a power series on $\mat{\vecspace{V}}{sm}$
 centered at
$\bigoplus_{\alpha=1}^mY$  with values in $\mat{\vecspace{W}}{sm}$ as in \cite[Chapter 26]{HiPh} and
\cite[Chapter 1]{Mu}.

If the series \eqref{eq:power-series} converges on a finitely open
nc set containing $Y$, then it follows from Theorem
\ref{thm:ncps-unique} and Remark \ref{rem:lost_abbey} that
conditions \eqref{eq:ncfun_coef_0}--\eqref{eq:ncfun-coef_ell_ell}
are necessary for the sum of the series \eqref{eq:power-series} to
be a nc function. We will show in Theorem \ref{thm:g-series} below
that these conditions are also sufficient, using the following
lemma.

\begin{lem}\label{lem:triangle-poly}
Let $\module{M}$ be a module over a ring $\ring$,
$A\in\mat{\module{M}}{m}$,
$B\in\rmat{\module{M}}{m}{\widetilde{m}}$,
$\widetilde{A}\in\mat{\module{M}}{\widetilde{m}}$. Then for any
$\ell\in\mathbb{N}$,
\begin{equation}\label{eq:triangle-power}
\begin{bmatrix} A & B \\ 0 & \widetilde{A} \end{bmatrix}^{\odot \ell} = \begin{bmatrix} A^{\odot \ell} &
\sum_{k=0}^{\ell-1} A^{\odot k}\odot B\odot C^{\odot (\ell-1-k)} \\
0 & C^{\odot \ell} \end{bmatrix}.
\end{equation}
\end{lem}
\begin{proof}
A trivial induction.
\end{proof}

For all $m=1,2,\ldots$ and all $Z\in\mat{\vecspace{V}}{sm}$, we
introduce the quantities \index{$\mu(Z)$} \index{$\rho(Z)$}
\begin{equation}\label{eq:z-mu-rho}
\mu(Z)=\limsup_{\ell\to\infty}\sqrt[\ell]{\|Z^{\odot_s\ell}f_\ell\|_{sm}},\qquad
\rho(Z)=\mu(Z)^{-1},
\end{equation}
and \index{$\rho_{\rm fin}(Z)$}
\begin{equation}\label{eq:z-rho-fin}
\rho_{\rm fin}(Z)=\liminf_{W\overset{\rm fin}{\to} Z}\rho(W),
\end{equation}
where $W\overset{\rm fin}{\to} Z$ \index{$W\overset{\rm fin}{\to}
Z$} means that the lower limit is taken in the finitely open
topology on $\mat{\vecspace{V}}{sm}$. It is clear that the
function $\rho(\cdot)$ is \emph{positive-homogeneous of degree
$-1$}, \index{positive-homogeneous function} i.e., satisfies
$\rho(\lambda Z)=\frac{1}{|\lambda|}\rho(Z)$ for every nonzero
$\lambda\in\mathbb{C}$. Since the multiplication by $\lambda$ is
continuous in the finitely open topology, $\rho_{\rm fin}(\cdot)$
is positive-homogeneous of degree $-1$ as well.

\begin{thm}\label{thm:g-series} 1. The convergence set, $\Upsilon_{\rm nc}\subseteq\ncspace{\vecspace{V}}$, for
the series \eqref{eq:power-series} is a nc set which is complete circular about $Y$, so that the sum of the
series \eqref{eq:power-series}, $f\colon\Upsilon_{\rm nc}\to\ncspace{\vecspace{W}}$ is a well defined mapping. We
have
\begin{multline}\label{eq:conv_set_incl}
\coprod_{m=1}^\infty\Big\{X\in\mat{\vecspace{V}}{sm}\colon\rho\Big(X-\bigoplus_{\alpha=1}^mY\Big)>1\Big\}\subseteq
\Upsilon_{\rm
nc}\\
\subseteq\coprod_{m=1}^\infty\Big\{X\in\mat{\vecspace{V}}{sm}\colon\rho\Big(X-\bigoplus_{\alpha=1}^mY\Big)\ge
1\Big\}.
\end{multline}
 The set
\begin{equation}\label{eq:conv-set-int-fin}
\Upsilon_{\rm nc,
fin}=\coprod_{m=1}^\infty\Big\{X\in\mat{\vecspace{V}}{sm}\colon\rho_{\rm
fin}\Big(X-\bigoplus_{\alpha=1}^mY\Big)>1\Big\}
\end{equation}
\index{$\Upsilon_{\rm nc, fin}$}is the interior of $\Upsilon_{\rm nc}$ in the finitely open
topology on $\ncspace{\vecspace{V}}$.

2. For each $m$, ${(\Upsilon_{\rm nc, fin})}_{sm}$ is nonempty if
and only if it contains
 $\bigoplus_{\alpha=1}^mY$. In this case,
\begin{itemize}
    \item $\rho_{\rm  fin}(0_{sm\times sm})=\infty$ and $\rho_{\rm
    fin}(Z)>0$ for every $Z\in\mat{\vecspace{V}}{sm}$,
    \item $(\Upsilon_{\rm nc, fin})_{sm}$ is
 a complete circular set about $\bigoplus_{\alpha=1}^mY$,
    \item the series \eqref{eq:power-series} converges absolutely on
$(\Upsilon_{\rm nc, fin})_{sm}$, and \item $f|_{(\Upsilon_{\rm nc,
fin})_{sm}}$ is G-differentiable.
\end{itemize}

3. For every $m\in\mathbb{N}$ and $X\in(\Upsilon_{\rm nc, fin})_{sm}$ there exists a finitely open complete
circular set $\Upsilon_{{\rm fin}, X}$ about $\bigoplus_{\alpha=1}^mY$ which contains $X$ and such that
\begin{equation}\label{eq:normal-g-conv}
\sum_{\ell=0}^\infty\sup_{W\in\Upsilon_{{\rm fin},
X}}\Big\|\Big(W-\bigoplus_{\alpha=1}^mY\Big)^{\odot_s\ell}
f_\ell\Big\|_{sm}<\infty.
\end{equation}

4. For any finitely open nc set $\Gamma\subseteq\Upsilon_{\rm nc, fin}$, $f|_{\Gamma}$ is
 a G-differentiable nc
function. Let \index{$\Upsilon_{\rm nc, fin-G}$}
\begin{equation}\label{eq:ups-nc-fin-g}
\Upsilon_{\rm nc, fin-G}:=\{X\in\Upsilon_{\rm nc, fin}\colon\exists\Gamma,\ {\rm a\ finitely\ open\ nc\ set,\
such\ that\ } X\in\Gamma\subseteq\Upsilon_{\rm nc, fin}\}.
\end{equation}
Then $\Upsilon_{\rm nc, fin-G}$ is nonempty if and only if there exists a nonempty finitely open nc set
$U\subseteq\ncspace{\vecspace{V}}$ such that $\inf\limits_{Z\in U}\rho(Z)>0$. We have $Y\in\Upsilon_{\rm nc,
fin-G}$ if and only if there exists a finitely open nc neighborhood $U$ of $Y$ such that $\inf\limits_{X\in
U}\rho(X-\bigoplus_{\alpha=1}^{m_X}Y)>0$; here $m_X$ is the size of a block matrix $X$ with $s\times s$ blocks.
\end{thm}
\begin{rem}\label{rem:series-uniqueness}
It follows that if $Y\in\Upsilon_{\rm nc,fin-G}$, then the
restriction of $f$ to a finitely open nc neighborhood of $Y$ is a
G-differentiable nc function, and according to Theorem
\ref{thm:ncps-unique}
 one has $$f_\ell=\Delta_R^\ell f(\underset{\ell+1\
\mbox{times}}{\underbrace{Y,\ldots, Y}}),\quad \ell=0,1,\ldots.$$
In particular, the coefficients $f_\ell$ (i.e., the original
linear mapping $f$) are uniquely determined by the corresponding
nc function, which is why we used the same notation for both the
nc function and the linear mapping.
\end{rem}

\begin{proof}[Proof of Theorem \ref{thm:g-series}]
1. The statement on $\Upsilon_{\rm nc}$ is straightforward. The inclusions in \eqref{eq:conv_set_incl} follow
from the classical Cauchy--Hadamard theorem  since, for a fixed $Z$, $\rho(Z)$ is the radius of convergence of
the series
\begin{equation}\label{eq:lambda-power-series}
\sum_{\ell=0}^\infty Z^{\odot_s\ell}f_\ell\lambda^\ell
\end{equation}
as a series in $\lambda\in\mathbb{C}$. Next we observe that, by
the construction, $\rho_{\rm fin}(Z)$ is a lower semicontinuous
function of $Z$ in the finitely open topology on
$\ncspace{\vecspace{V}}$. This implies that $\Upsilon_{\rm nc,
fin}$ is a finitely open set. Clearly, $\rho_{\rm
fin}(Z)\le\rho(Z)$, therefore, again by the classical
Cauchy--Hadamard theorem, the series \eqref{eq:power-series}
converges absolutely on $\Upsilon_{\rm nc, fin}$. On the other
hand, if $Z\notin\Upsilon_{\rm nc, fin}$, i.e., if $\rho_{\rm
fin}(Z)\le 1$, then by the positive-homogeneity of degree $-1$ of
$\rho_{\rm fin}(\cdot )$ we have that $\rho_{\rm
fin}((1+\epsilon)Z)\le\frac{1}{1+\epsilon}$ for an arbitrary
$\epsilon>0$. Therefore, in every finitely open neighborhood of
$(1+\epsilon)Z$ there is a point $W$ such that $\rho(W)< 1$. This
means that $(1+\epsilon)Z$ is not in the interior in the finitely
open topology on $\ncspace{\vecspace{V}}$ of the convergence set
for the series \eqref{eq:power-series}, and then so is $Z$. Thus
we obtain that $\Upsilon_{\rm nc, fin}$ is
 the interior of $\Upsilon_{\rm nc}$ in the finitely open topology on
$\ncspace{\vecspace{V}}$.

2. Let $m\in\mathbb{N}$ be fixed. If
$\bigoplus_{\alpha=1}^mY\in(\Upsilon_{\rm nc, fin})_{sm}$, then
clearly $(\Upsilon_{\rm nc, fin})_{sm}\neq\emptyset$. Conversely,
if $(\Upsilon_{\rm nc, fin})_{sm}$, which is the interior of
$(\Upsilon_{\rm nc})_{sm}$ in the finitely open topology on
$\mat{\vecspace{V}}{sm}$, is nonempty, then by \cite[Theorem
26.5.9]{HiPh} (see also \cite[Theorem 26.5.2]{HiPh})
$\bigoplus_{\alpha=1}^mY$ is an interior point of $(\Upsilon_{\rm
nc})_{sm}$ in this topology, i.e., $\rho_{\rm fin}(0_{sm\times
sm})>1$. Moreover, by the positive-homogeneity of degree $-1$ and
lower semicontinuity of $\rho_{\rm fin}(\cdot)$ we have that
$\rho_{\rm fin}(Z)>0$ for every $Z\in\mat{\vecspace{V}}{sm}$ and
that $\rho_{\rm fin}(0_{sm\times sm})=\infty$. From now on, assume
$(\Upsilon_{\rm nc, fin})_{sm}\neq \emptyset$. By the
positive-homogeneity of $\rho_{\rm fin}(\cdot)$ the set
${(\Upsilon_{\rm nc, fin})}_{sm}$ is a complete circular set about
$\bigoplus_{\alpha=1}^mY$, for every $m\in\mathbb{N}$ (see also
\cite[Theorem 26.5.9]{HiPh}). By \cite[Theorem 26.5.6]{HiPh} (see
also \cite[Theorem 26.5.3 and Theorem 26.5.4]{HiPh}) the sum of
the series \eqref{eq:power-series} is G-differentiable on
$(\Upsilon_{\rm nc, fin})_{sm}$. The absolute convergence of the
series \eqref{eq:power-series} on $(\Upsilon_{\rm nc, fin})_{sm}$
has been proved in the preceding paragraph. We also can prove this
using Hartogs' theorem on the convergence in homogeneous
polynomials \cite[Theorem I.3.3]{Sh}.

3. This statement for any fixed $m$ and $X$ follows from \cite[Theorem 26.3.8]{HiPh} (cf. Theorem
\ref{thm:g-queen}(4)).

4. It is clear that $f((\Upsilon_{\rm nc})_{sm})\subseteq \mat{\vecspace{W}}{sm}$. Let $X\in(\Upsilon_{\rm
nc})_{sm}$, $\widetilde{X}\in(\Upsilon_{\rm nc})_{s\widetilde{m}}$, and
$S\in\rmat{\mathbb{C}}{s\widetilde{m}}{sm}$ be such that $SX=\widetilde{X}S$. Taking the limit as $N\to \infty$
in \eqref{eq:intertw-poly} (see Lemma \ref{lem:intertw-poly}) and using Lemma \ref{lem:triangle-poly}, we see
that $Sf(X)=f(\widetilde{X})S$ if and only if the general term of the series \eqref{eq:power-series} evaluated at
$\begin{bmatrix} \widetilde{X} & S\bigoplus_{\nu=1}^mY-\bigoplus_{\widetilde{\nu}=1}^{\widetilde{m}}YS\\
0 & X
\end{bmatrix}$ converges to $0$. For $X,\widetilde{X}\in\Gamma$ where $\Gamma\subseteq
\Upsilon_{\rm nc, fin}$ is a finitely open nc set, we have that
$\begin{bmatrix} \widetilde{X} & \epsilon (S\bigoplus_{\nu=1}^mY-\bigoplus_{\widetilde{\nu}=1}^{\widetilde{m}}YS)\\
0 & X
\end{bmatrix}\in\Gamma$  for $\epsilon>0$  small enough, so that the series \eqref{eq:power-series} converges
at this point and a fortiori its general term converges to $0$; therefore $Sf(X)=f(\widetilde{X})S$. We conclude
that $f|_{\Gamma}$ is a nc function. It is a G-differentiable nc function by part 2 of the theorem.

 Next, if $\Upsilon_{\rm nc, fin-G}\neq\emptyset$, then
there exists a nonempty finitely open nc set $\Gamma\subseteq\Upsilon_{\rm nc, fin}$, i.e.,  $\rho_{\rm
fin}(W)>1$ for every $W\in\Gamma$. We have
$$\inf_{W\in\Gamma}\rho(W)\ge\inf_{W\in\Gamma}\rho_{\rm
fin}(W)\ge 1>0.$$  Conversely, if $U\subseteq
\ncspace{\vecspace{V}}$ is a nonempty finitely open nc set so that
$\rho_0:=\inf\limits_{W\in U}\rho(W)>0$, then for any positive
$\epsilon<\rho_0$ the set $(\rho_0-\epsilon) U\subseteq
\ncspace{\vecspace{V}}$ is also a nonempty finitely open nc set
and $\inf\limits_{W\in (\rho_0-\epsilon)
U}\rho(W)=\frac{\rho_0}{\rho_0-\epsilon}$. Therefore for any
$Z\in(\rho_0-\epsilon) U$ we have
$$\rho_{\rm fin}(Z)=\liminf_{\underset{W\to Z}{\rm fin}}\rho(W)\ge\frac{\rho_0}{\rho_0-\epsilon}>1,$$ and thus the
finitely open nc set $(\rho_0-\epsilon) U$ is contained in $\Upsilon_{\rm nc, fin}$. We conclude that
$\Upsilon_{\rm nc, fin-G}\neq\emptyset$.

If $Y\in\Upsilon_{\rm nc, fin-G}$, then there exists a finitely open nc neighborhood $\Gamma$ of $Y$ such that
$\Gamma\subseteq\Upsilon_{\rm nc, fin}$. Then, as we have shown in the preceding paragraph,
$\inf\limits_{X\in\Gamma}\rho(X-\bigoplus_{\alpha=1}^{m_X}Y)>0.$ Conversely, if there exists a finitely open nc
neighborhood $U$ of $Y$ such that $\rho_0:=\inf\limits_{X\in U}\rho(X-\bigoplus_{\alpha=1}^{m_X}Y)>0$, then
arguing as in the preceding paragraph we obtain that, for $\epsilon>0$ small enough, the set $\Gamma$ consisting
of $X\in\Upsilon_{\rm nc}$ for which there is a $\widetilde{X}\in U$ satisfying
$X-\bigoplus_{\alpha=1}^{m_X}Y=(\rho_0-\epsilon)(\widetilde{X}-\bigoplus_{\alpha=1}^{m_X}Y)$, is a subset of
  $\Upsilon_{\rm nc, fin}$. Since $\Gamma$ is a finitely open nc neighborhood of $Y$, we conclude that
$Y\in\Upsilon_{\rm nc, fin-G}$.
\end{proof}
\begin{rem}\label{rem:s=1}
In the case $s=1$, $f$ is a nc function on $\Upsilon_{\rm nc}$. This is clear from the form of the series
\eqref{eq:power-series}. It also follows from the proof of part 4 of Theorem \ref{thm:g-series} since in this
case $S\bigoplus_{\nu=1}^mY-\bigoplus_{\widetilde{\nu}=1}^{\widetilde{m}}YS=0$.
\end{rem}

 \label{QUESTIONS}

The following example demonstrates the possibility that
$Y=0_{1\times 1}\!\in\!\Upsilon_{\rm nc,
fin-G}\!\subsetneq\!\Upsilon_{\rm nc, fin}$.
\begin{ex}\label{ex:nc-conv-set}  Consider the
series
$$\sum_{\ell=0}^\infty Z_1^\ell(Z_1Z_2-Z_2Z_1).$$
Here $\vecspace{V}=\mathbb{C}^2$, $\vecspace{W}=\mathbb{C}$, with
the Euclidean topologies on
$\mat{\vecspace{V}}{n}=\mat{(\mathbb{C}^2)}{n}\cong\mattuple{\mathbb{C}}{n}{2}$
and $\mat{\vecspace{W}}{n}=\mat{\mathbb{C}}{n}$. If
$Z=(Z_1,Z_2)\in\mattuple{\mathbb{C}}{n}{2}$ satisfies
$Z_1Z_2=Z_2Z_1$ (in particular, it is always the case when $n=1$),
then $\rho(Z)=\infty$. If
$z=(z_1,z_2)\in\mat{\vecspace{V}}{1}=\mathbb{C}^2$, then we have
$\rho_{\rm fin}(z)=\rho(z)=\infty$ and therefore $(\Upsilon_{\rm
nc, fin})_1=\mathbb{C}^2$. If $Z=z\oplus
z\in\mat{\vecspace{V}}{2}$ and $z_1\neq 0$, then of course we have
$\rho(Z)=\infty$, however $\rho_{\rm fin}(Z)=\frac{1}{|z_1|}$. To
show the latter, we first calculate, for any fixed $\epsilon>0$,
\begin{multline*}
\mu\left(\begin{bmatrix} z_1 & \epsilon\\
0 & z_1
\end{bmatrix},\begin{bmatrix} z_2 & 0\\
\epsilon & z_2
\end{bmatrix}\right)=\limsup_{\ell\to\infty}\sqrt[\ell+2]{\left\|\begin{bmatrix} z_1 & \epsilon\\
0 & z_1
\end{bmatrix}^\ell\begin{bmatrix} \epsilon^2 & 0\\
0 & -\epsilon^2
\end{bmatrix}\right\|}\\
=\limsup_{\ell\to\infty}\sqrt[\ell+2]{\left\|\begin{bmatrix} \epsilon^2z_1^\ell &
 -\ell\epsilon^3z_1^{\ell-1}\\
0 & -\epsilon^2z_1^\ell
\end{bmatrix}\right\|}
=\limsup_{\ell\to\infty}\sqrt[\ell+2]{\epsilon^2|z_1|^{\ell-1}\left\|\begin{bmatrix} z_1 & -\ell\epsilon\\
0 & -z_1
\end{bmatrix}\right\|}=|z_1|.
\end{multline*}
Hence $$\rho_{\rm fin}(Z)\le\lim_{\epsilon\to 0}\rho\left(\begin{bmatrix} z_1 & \epsilon\\
0 & z_1
\end{bmatrix},\begin{bmatrix} z_2 & 0\\
\epsilon & z_2
\end{bmatrix}\right)=\frac{1}{|z_1|}.$$
On the other hand, for any $M,N\in\mat{\mathbb{C}}{2}$ we have
\begin{multline*}
\mu(z_1I_2+M,z_2I_2+N)=\limsup_{\ell\to\infty}\sqrt[\ell+2]{\|(z_1I_2+M)^\ell
(MN-NM)\|}\\
\le \|z_1I_2+M\|\lim_{\ell\to 0}\sqrt[\ell+2]{\|MN-NM\|} \le
\|z_1I_2+M\|,
\end{multline*}
and for $M$ small enough we have
$\rho(z_1I_2+M,z_2I_2+N)\ge\frac{1}{\|z_1I_2+M\|}$. Hence
 $$\rho_{\rm fin}(Z)\ge\frac{1}{|z_1|}.$$
We conclude that $\rho_{\rm fin}(Z)=\frac{1}{|z_1|}$. If
$|z_1|>1$, we have that $z=(z_1,z_2)\in\Upsilon_{\rm nc, fin}$,
however $Z=z\oplus z\notin\Upsilon_{\rm nc, fin}$. Thus the set
$\Upsilon_{\rm nc, fin}$ is not a nc set. On the other hand, the
subset $\Gamma\subsetneq\Upsilon_{\rm nc, fin}$ consisting of
pairs of matrices $Z=(Z_1,Z_2)$ satisfying $\|Z_1\|<1$ (with the
operator norm $\|\cdot\|$) is clearly a uniformly-open (and
therefore finitely open) nc set, the series $\sum_{\ell=0}^\infty
Z_1^\ell(Z_1Z_2-Z_2Z_1)$ converges absolutely on $\Gamma$, and its
sum $f\colon\Gamma\to\ncspace{\mathbb{C}}$ is a G-differentiable
nc function.
\end{ex}

The next example shows that it is also possible that $(\Upsilon_{\rm nc, fin})_{1\cdot m}\neq\emptyset$, i.e.,
$\bigoplus_{\alpha=1}^mY=0_{1\cdot m\times 1\cdot m}\in\Upsilon_{\rm nc, fin}$ for all $m\in\mathbb{N}$, however
$Y=0_{1\times 1}\notin\Upsilon_{\rm nc, fin-G}$ (in fact, we have $0_{1\cdot m\times 1\cdot m}\notin\Upsilon_{\rm
nc, fin-G}$ for all $m\in\mathbb{N}$).
\begin{ex}\label{ex:no-nc-conv-set}
Let $p_n$ be the homogeneous polynomial in two noncommuting
variables of degree $\alpha_n=\frac{(n+1)(n+2)}{2}$,
$n=1,2,\ldots$, from Example \ref{ex:degrees_unbdd}. We remind the
reader that $p_n$ vanishes on all pairs of $n\times n$ matrices.
 Consider the series
$$\sum_{j=2}^\infty\sum_{k=1}^{j-1} (kZ_1)^{\alpha_j-\alpha_k-j+k}p_k(Z_1,Z_2).$$
We note that the total degrees of the homogeneous polynomials in
this series strictly increase. Indeed, for any fixed $j$ we have
$$\deg \left(
(kx_1)^{\alpha_j-\alpha_k-j+k}p_k\right)=\alpha_j-j+k$$ increasing
as $k$ changes from $1$ to $j-1$. On the other hand, for any fixed
$j$ we have
$$\alpha_{j}-j+(j-1)=\alpha_j-1<\alpha_{j+1}-j=\alpha_{j+1}-(j+1)+1,$$
i.e., the degree increases when we pass from the last term in the
$j$-th sum to the first term in the $j+1$-th sum. Similarly to
Example \ref{ex:nc-conv-set}, we have $\vecspace{V}=\mathbb{C}^2$,
$\vecspace{W}=\mathbb{C}$, with the Euclidean topologies on
$\mat{\vecspace{V}}{n}=\mat{(\mathbb{C}^2)}{n}\cong\mattuple{\mathbb{C}}{n}{2}$
and $\mat{\vecspace{W}}{n}=\mat{\mathbb{C}}{n}$, and
$(\Upsilon_{\rm nc, fin})_1=\mathbb{C}^2$. Let
$z=(z_1,z_2)\in\mat{\vecspace{V}}{1}=\mathbb{C}^2$ be such that
$z_1\neq 0$. Define
$Z(k)=\bigoplus_{j=1}^kz\in\mat{\vecspace{V}}{k}$. Arguing as in
Example \ref{ex:nc-conv-set} and using the subsequence of
homogeneous polynomials $(kx_1)^{\alpha_j-\alpha_k-j+k}p_k$,
$j=2,3,\ldots$, we can easily show that $\rho_{\rm
fin}(Z(k))=\frac{1}{k|z_1|}$. Therefore, there is no finitely open
nc neighborhood of $0_{m\times m}$, for any $m$, where the series
converges (otherwise $|z_1|<\frac{1}{k}$ for all $k$, which is
impossible).
\end{ex}

\section{Norm topology}\label{subsec:conv-norm} Let
$\vecspace{V}$ and $\vecspace{W}$ be Banach spaces over
$\mathbb{C}$ equipped with admissible systems of matrix norms, and
let
$f_\ell\colon\left(\mat{\vecspace{V}}{s}\right)^\ell\to\mat{\vecspace{W}}{s}$,
$\ell=0,1,\ldots$, be a sequence of bounded $\ell$-linear
mappings. It follows that, for all $m_0$, \ldots,
$m_k\in\mathbb{N}$, $f_\ell^{(m_0,\ldots,m_k)}$ is also a sequence
of bounded $\ell$-linear mappings. Therefore, for
$Z\in\mat{\vecspace{V}}{sm}$,
$$Z^{\odot_s\ell}f_\ell=f_\ell^{(m,\ldots,m)}(Z,\ldots,Z)$$
 is a bounded homogeneous polynomial of degree $\ell$ on
$\mat{\vecspace{V}}{sm}$ with values in $\mat{\vecspace{W}}{sm}$ and the series $\eqref{eq:power-series}$ is a
F-power series (a power series of bounded homogeneous polynomials) on $\mat{\vecspace{V}}{sm}$ centered at
$\bigoplus_{\alpha=1}^mY$ with values in $\mat{\vecspace{W}}{sm}$ as in \cite[Chapter 26]{HiPh} and \cite[Chapter
1]{Mu}.

We introduce the quantities \index{$\rho_{\rm norm}(Z)$}
\begin{equation}\label{eq:z-rho-norm}
\rho_{\rm norm}(Z)=\liminf_{W\to Z}\rho(W)
\end{equation}
(cf. \eqref{eq:z-rho-fin}), \index{$\mu_m$} \index{$\rho_m$}
\begin{equation}\label{eq:mu-rho}
\mu_m=\limsup_{\ell\to\infty}\sqrt[\ell]{\sup_{\|Z\|_{sm}=1}\|
Z^{\odot_s\ell}f_\ell\|_{sm}},\qquad \rho_m=\mu_m^{-1}.
\end{equation}
Analogously to $\rho_{\rm fin}(\cdot)$, the function $\rho_{\rm
norm}(\cdot)$ is positive-homogeneous of degree $-1$. We notice
that the expression under the $\ell$-th root is the norm of the
homogeneous polynomial $Z^{\odot_s\ell}f_\ell$, see
\cite[Definition 26.2.4]{HiPh} and \cite[Definition 2.1]{Mu}.
\begin{prop}\label{prop:mu_m-vs-coef-norms}
The following inequalities hold:
\begin{equation}\label{eq:mu_norm-vs-coef-norms}
\frac{1}{e}\limsup_{\ell\to\infty}\sqrt[\ell]{\|f_\ell\|}\le\mu_1\le
K_1\mu_m\le mK_1K_2\limsup_{\ell\to\infty}\sqrt[\ell]{\|f_\ell\|},
\end{equation}
where $$K_1=\min_{1\le i\le m}\|\iota^s_{ii}\|,\quad
K_2=\max_{1\le i,j\le m}\|\pi^s_{ij}\|;$$  block injections,
$\iota^s_{ij}$, and block projections, $\pi^s_{ij}$, are defined
in \eqref{eq:block-inj-bd} and \eqref{eq:block-proj-bd}.
\end{prop}
Notice that we also have estimates for $K_1$ and $K_2$ from the
proof of Proposition  \ref{prop:inj-proj} via the constants in
\eqref{eq:dirsums-norms} and \eqref{eq:simprod-norms} which
determine the admissible system of norms on
$\mat{\vecspace{V}}{n}$ and $\mat{\vecspace{W}}{n}$,
$n=1,2,\ldots$:
$$K_1\le C_2(sm)C_1^\prime(s,sm-s),\quad K_2\le
C_1(s,sm-s)C_2(sm),$$ and
 that in the case where $\vecspace{V}$ and $\vecspace{W}$
are operator spaces one has $K_1=K_2=1$.
\begin{proof}[Proof of
Proposition \ref{prop:mu_m-vs-coef-norms}] The leftmost inequality
in \eqref{eq:mu_norm-vs-coef-norms} follows from the estimate
$$\|f_\ell\|=\sup_{\|Z_1\|_s=\cdots=\|Z_\ell\|_s=1}\|(Z_1\otimes\cdots\otimes Z_\ell)f_\ell\|_s\le
\frac{\ell^\ell}{\ell!}\sup_{\|Z\|_{s}=1}\|Z^{\otimes\ell}f_\ell\|_s$$
(see \cite[Theorem 2.2]{Mu}) by taking the $\ell$-th root of both
parts, then $\limsup_{\ell\to\infty}$, and then applying the
Stirling formula. The second inequality in
\eqref{eq:mu_norm-vs-coef-norms} follows from the estimate
$$\sup_{\|Z\|_{s}=1}\|Z^{\otimes\ell}f_\ell\|_s=
\sup_{\|Z\|_{s}=1}\|\pi^s_{ii}\left((\iota^s_{ii}Z)^{\odot_s\ell}f_\ell\right)\|_s
\le\|\pi^s_{ii}\|\|\iota^s_{ii}\|^\ell\sup_{\|W\|_{sm}=1}\|W^{\odot_s\ell}f_\ell\|_{sm}
$$
by taking the $\ell$-th root of both parts, then
$\limsup_{\ell\to\infty}$, and then $\min_{1\le i\le m}$. The last
inequality in \eqref{eq:mu_norm-vs-coef-norms} follows from the
estimate
\begin{multline*}
\sup_{\|Z\|_{sm}=1}\|Z^{\odot_s\ell}f_\ell\|_{sm}=\sup_{\|Z\|_{sm}=1}\Big\|\sum_{1\le
i,j\le m}\iota^s_{ij}(Z^{\odot_s\ell}f_\ell)_{ij}\Big\|_{sm}\\
\le\sum_{1\le
i,j\le m}\|\iota^s_{ij}\|\sup_{\|Z\|_{sm}=1}\|(Z^{\odot_s\ell}f_\ell)_{ij}\|_{s}\\
=\sum_{1\le i,k_1,\ldots,k_{\ell-1},j\le
m}\|\iota^s_{ij}\|\sup_{\|Z\|_{sm}=1}\|(\pi^s_{ik_1}Z\otimes\pi^s_{k_1k_2}Z\otimes\cdots\otimes
\pi^s_{k_{\ell-2}k_{\ell-1}}Z\otimes \pi^s_{k_{\ell-1}j}Z) f_\ell\|_{s}\\
\le\sum_{1\le i,k_1,\ldots,k_{\ell-1},j\le
m}\|\iota^s_{ij}\|\|\pi^s_{ik_1}\|\|\pi^s_{k_1k_2}\|\cdots\|\pi^s_{k_{\ell-2}k_{\ell-1}}\|\|\pi^s_{k_{\ell-1}j}\|
\|f_\ell\|_{s}\\
\le m^{\ell+1}K_2^\ell\sup_{1\le i,j\le
m}\|\iota^s_{ij}\|\,\|f_\ell\|
\end{multline*}
by taking the $\ell$-th root, then $\limsup_{\ell\to\infty}$, and
then multiplying by $K_1$.
\end{proof}

We recall that $\Upsilon_{\rm nc}\subseteq\ncspace{\vecspace{V}}$
denotes the convergence set for the series
\eqref{eq:power-series}, and $f$ denote the sum of the series.
\begin{thm}\label{thm:f-series} 1.  The set \index{$\Upsilon_{\rm nc,
norm}$}
\begin{equation}\label{eq:conv-set-int-norm}
\Upsilon_{\rm nc,
norm}=\coprod_{m=1}^\infty\Big\{X\in\mat{\vecspace{V}}{sm}\colon\rho_{\rm
norm}\Big(X-\bigoplus_{\alpha=1}^mY\Big)>1\Big\}
\end{equation}
is contained in $\Upsilon_{\rm nc, fin}$ and is the interior of
$\Upsilon_{\rm nc}$ in the norm topology on
$\ncspace{\vecspace{V}}$, i.e., in the topology on
$\ncspace{\vecspace{V}}$ defined in the paragraph next to Remark
\ref{rem:polylin}.

2. For each $m$, ${(\Upsilon_{\rm nc, norm})}_{sm}$ is nonempty if
and only if it contains
 $\bigoplus_{\alpha=1}^mY$. In this case,
\begin{itemize}
    \item $\rho_{\rm norm}(0_{sm\times sm})=\infty$ and $\rho_{\rm
 norm}(Z)>0$ for every $Z\in\mat{\vecspace{V}}{sm}$,
    \item $(\Upsilon_{\rm nc, norm})_{sm}$ is
 a complete circular set about $\bigoplus_{\alpha=1}^mY$, and
    \item  $f|_{(\Upsilon_{\rm
nc, norm})_{sm}}$ is analytic.
\end{itemize}

3. For every $m\in\mathbb{N}$ and $X\in(\Upsilon_{\rm nc, norm})_{sm}$ there exists an open complete circular set
$\Upsilon_{{\rm norm}, X}$ about $\bigoplus_{\alpha=1}^mY$ which contains $X$ and such that
\begin{equation}\label{eq:normal-f-conv}
\sum_{\ell=0}^\infty\sup_{W\in\Upsilon_{{\rm norm},
X}}\Big\|\Big(W-\bigoplus_{\alpha=1}^mY\Big)^{\odot_s\ell}
f_\ell\Big\|_{sm}<\infty.
\end{equation}

4. For any open nc set $\Gamma\subseteq\Upsilon_{\rm nc, norm}$, $f|_\Gamma$ is an analytic nc function. Let
\begin{equation}\label{eq:ups-nc-norm-an}
\Upsilon_{\rm nc, norm-an}:=\{X\in\Upsilon_{\rm nc, norm}\colon\exists\Gamma,\ {\rm an\ open\ nc\ set,\ such\
that\ } X\in\Gamma\subseteq\Upsilon_{\rm nc, norm}\}.
\end{equation}
\index{$\Upsilon_{\rm nc, norm-an}$}Then $\Upsilon_{\rm nc, norm-an}$ is nonempty if and only if there
exists a nonempty open nc set $U\subseteq\ncspace{\vecspace{V}}$
such that $\inf\limits_{Z\in U}\rho(Z)>0$. We have
$Y\in\Upsilon_{\rm nc, norm-an}$ if and only if there exists an
open nc neighborhood of $Y$ such that $\inf\limits_{X\in
U}\rho(X-\bigoplus_{\alpha=1}^{m_X}Y)>0$; here $m_X$ is the size
of a block matrix $X$ with $s\times s$ blocks.

5. The series \eqref{eq:power-series} converges uniformly on every ball $B(\bigoplus_{\alpha=1}^{m}Y,\delta)$
with $\delta<\rho_m$, moreover
\begin{equation}\label{eq:normal-f-conv-balls}
\sum_{\ell=0}^\infty\sup_{X\in
B(\bigoplus_{\alpha=1}^{m}Y,\delta)}\Big\|\Big(X-\bigoplus_{\alpha=1}^{m}Y\Big)^{\odot_s\ell}
f_\ell\Big\|_{sm}<\infty.
\end{equation}
The series \eqref{eq:power-series} fails to converge uniformly on every ball
$B(\bigoplus_{\alpha=1}^{m}Y,\delta)$ with $\delta>\rho_m$. The set $(\Upsilon_{\rm nc, norm})_{sm}$ is nonempty
if and only if $\rho_m>0$ if and only if $\inf\{ \rho(W)\colon W\in\mat{\vecspace{V}}{sm}, \|W\|_{sm}=1\}>0$.
\end{thm}

\begin{proof}
1. We first observe that the finitely open topology on
$\ncspace{\vecspace{V}}$ is stronger than the norm topology, hence
 $\rho_{\rm norm}(Z)\le\rho_{\rm fin}(Z)$ for any $Z$, and
 $\Upsilon_{\rm nc,
norm}\subseteq\Upsilon_{\rm nc,fin}$. Next we observe that the
function $\rho_{\rm norm}(\cdot)$ is lower semicontinuous in the
norm topology and recall that $\rho_{\rm norm}(\cdot)$ is also
positive-homogeneous of degree $-1$. The rest of the proof of part
1 is analogous to the proof of part 1 of Theorem
\ref{thm:g-series}.

2. Let $m\in\mathbb{N}$ be fixed. If $\bigoplus_{\alpha=1}^mY\in(\Upsilon_{\rm nc, norm})_{sm}$, then clearly
$(\Upsilon_{\rm nc, norm})_{sm}\neq\emptyset$. Conversely, if $(\Upsilon_{\rm nc, norm})_{sm}$, which is the
interior of $(\Upsilon_{\rm nc})_{sm}$ in the norm topology on $\mat{\vecspace{V}}{sm}$, is nonempty, then by
\cite[Theorem 26.6.1]{HiPh} $\bigoplus_{\alpha=1}^mY$ is an interior point of $(\Upsilon_{\rm nc})_{sm}$ in this
topology and $(\Upsilon_{\rm nc, norm})_{sm}$ is a complete circular set about $\bigoplus_{\alpha=1}^mY$. Then
$\rho_{\rm norm}(0_{sm\times sm})>1$. Moreover, by the positive-homogeneity of degree $-1$ and lower
semicontinuiuty of $\rho_{\rm norm}(\cdot)$ we have that $\rho_{\rm norm}(Z)>0$ for every
$Z\in\mat{\vecspace{V}}{sm}$ and that $\rho_{\rm norm}(0_{sm\times sm})=\infty$. The analyticity of
$f|(\Upsilon_{\rm norm})_{sm}$ follows from \cite[Theorem 26.6.4]{HiPh}.

3. This statement for any fixed $m$ and $X$ follows from \cite[Theorem 26.6.5]{HiPh} (cf. Theorem
\ref{thm:f-queen}).

4. Since an open set is finitely open, $f|\Gamma$ is a nc function by Theorem \ref{thm:g-series}, part 4. It is
an analytic nc function by part 2 of the present theorem. The remaining statements can be proved analogously to
the proof of part 4 of Theorem \ref{thm:g-series}.

5. Let $\delta<\rho_{m}$. Then for any $\epsilon>0$ there exists
$L>0$ such that $$\sup_{\|W\|_{sm}=1}\|
W^{\odot_s\ell}f_\ell\|_{sm}<(\mu_m+\epsilon)^\ell$$ for all $\ell
> L$. Choosing $\epsilon<\delta^{-1}-\mu_m$, we obtain that,
for all $\ell>L$,
$$\sup_{W\in B(0_{sm\times sm},\delta)}\|
W^{\odot_s\ell}f_\ell\|_{sm}<\delta^\ell(\mu_m+\epsilon)^\ell.$$
Since $\delta(\mu_m+\epsilon)<1$, the series
\eqref{eq:power-series} converges absolutely and uniformly on the
ball $ B(0_{sm\times sm},\delta)$, and
\begin{multline*}
\sum_{\ell=0}^\infty\sup_{W\in B(0_{sm\times
sm},\delta)}\|W^{\odot_s\ell}
f_\ell\|_{sm}\\
<\sum_{\ell=0}^L\sup_{W\in B(0_{sm\times
sm},\delta)}\|W^{\odot_s\ell} f_\ell\|_{sm}+
\sum_{\ell=L+1}^\infty\delta^\ell(\mu_m+\epsilon)^\ell<\infty,
\end{multline*}
i.e., \eqref{eq:normal-f-conv-balls} holds by setting $X=W+\bigoplus_{\alpha=1}^mY$.

Now let $\delta>\rho_m$ be fixed. Choose any $\epsilon$ with
$0<\epsilon<\mu_m-\delta^{-1}$. Then for any $L>0$ there exists
$\ell>L$ such that
$$\sup_{\|W\|_{sm}=1}\|
W^{\odot_s\ell}f_\ell\|_{sm}>(\mu_m-\epsilon)^\ell.$$ For any such
$\ell$ there exists $Z\in\mat{\vecspace{V}}{sm}$ with
$\|Z\|_{sm}=(\mu_m-\epsilon)^{-1}<\delta$ such that
$$\|Z^{\odot_s\ell}f_\ell\|_{sm}>(\mu_m-\epsilon)^\ell\|Z\|_{sm}^\ell=1.$$
Hence the sequence $Z^{\odot_s\ell}f_\ell$, $\ell=1,2,\ldots$,
does not converge to $0_{sm\times sm}$ uniformly on the ball
$B(0_{sm\times sm},\delta)$. We conclude that the series
\eqref{eq:power-series} does not converge uniformly on
 $B(\bigoplus_{\alpha=1}^mY,\delta)$.

If $(\Upsilon_{\rm nc,norm})_{sm}\neq\emptyset$, i.e., the interior of $\Upsilon_{\rm nc}$ in the norm topology
is nonempty, then by \cite[Theorem 26.6.2]{HiPh} the series \eqref{eq:power-series} converges uniformly on
$B(\bigoplus_{\alpha=1}^mY,\delta)$ for sufficiently small $\delta$, hence $\rho_m>0$. Since $\rho_m\le\rho(W)$
for every $W\in\mat{\vecspace{V}}{sm}$ with $\|W\|_{sm}=1$, we have that $\rho_m>0$ implies $\inf\{ \rho(W)\colon
W\in\mat{\vecspace{V}}{sm}, \|W\|_{sm}=1\}>0$. Finally, if $\alpha=\inf\{ \rho(W)\colon
W\in\mat{\vecspace{V}}{sm}, \|W\|_{sm}=1\}>0$, then by positive-homogeneity of degree $-1$ of both $\rho(\cdot)$
and $\rho_{\rm norm}(\cdot)$ we have that $\rho_{\rm norm}(Z)\ge\alpha$ for any $Z\in\mat{\vecspace{V}}{sm}$ with
$\|Z\|_{sm}=1$ and that $\rho_{\rm norm}(0_{sm\times sm})=\infty$. The latter implies by part 2 that
$(\Upsilon_{\rm nc,norm})_{sm}\neq\emptyset$. This completes the proof of the last statement of the theorem.
\end{proof}

\begin{rem}\label{rem:nc-conv-open-nbhd} Since the finitely open and the norm topologies
coincide for finite-dimensional spaces, Examples \ref{ex:nc-conv-set} and \ref{ex:no-nc-conv-set} also show that
$\Upsilon_{\rm nc, norm}$ is not necessarily a nc set, and that a (norm-open) nc neighborhood of $Y$, which is
contained in $\Upsilon_{\rm nc, norm}$, may or may not exist.
\end{rem}

\label{RHOS}

\section{Uniformly-open topology}\label{subsec:conv-unif-open}
Let  $\vecspace{V}$ and $\vecspace{W}$ be operator spaces over
$\mathbb{C}$, and let
$f_\ell\colon\left(\mat{\vecspace{V}}{s}\right)^\ell\to\mat{\vecspace{W}}{s}$,
$\ell=0,1,\ldots$, be a sequence of completely bounded
$\ell$-linear mappings. We introduce the quantities
\index{$\rho_{\rm unif}(Z)$}
\begin{equation}\label{eq:z-rho-unif}
\rho_{\rm unif}(Z)=\liminf_{W\overset{\rm unif}{\to} Z}\rho(W)
\end{equation}
(cf. \eqref{eq:z-rho-fin} and \eqref{eq:z-rho-norm}), where
$W\overset{\rm unif}{\to} Z$ \index{$W\overset{\rm unif}{\to} Z$}
means that the lower limit is taken in the uniformly-open topology
on $\ncspace{\vecspace{V}}$,
\begin{equation}\label{eq:z-rho-normal-unif}
\begin{split}
\mu_{\rm unif-normal}(Z) &=\lim_{\delta\to
0}\limsup_{\ell\to\infty}\sqrt[\ell]{\sup_{W\in B_{\rm
nc}(Z,\delta)}\|W^{\odot_s\ell}f_\ell\|},\\
 \rho_{\rm
unif-normal}(Z) &=\mu_{\rm unif-normal}(Z)^{-1},
\end{split}
\end{equation}
\index{$\mu_{\rm unif-normal}(Z)$}\index{$\rho_{\rm
unif-normal}(Z)$}where
$\|W^{\odot_s\ell}f_\ell\|=\|W^{\odot_s\ell}f_\ell\|_{sm}$ for
$W\in\mat{\vecspace{V}}{sm}$, and \index{$\mu_{\rm cb}$}
\index{$\rho_{\rm cb}$}
\begin{equation}\label{eq:mu-rho-cb}
\mu_{\rm
cb}=\limsup_{\ell\to\infty}\sqrt[\ell]{\|f_\ell\|_{\mathcal{L}^\ell_{\rm
cb}}},\qquad \rho_{\rm cb}=\mu_{\rm cb}^{-1}.
\end{equation}
We notice that the limit $\lim\limits_{\delta\to 0}$ in
\eqref{eq:z-rho-normal-unif} exists as a limit of a monotone
function (decreasing as $\delta\downarrow 0$) bounded from below
by $0$. Analogously to $\rho_{\rm fin}(\cdot)$ and $\rho_{\rm
norm}(\cdot)$, the functions $\rho_{\rm unif}(\cdot)$ and
$\rho_{\rm unif-normal}(\cdot)$ are positive-homogeneous of degree
$-1$, cf. Remark \ref{rem:unif-mult-cont}.

We will need the following lemma \cite[Lemma 26.5.1]{HiPh}; we
provide its proof for completeness.
\begin{lem}\label{lem:HiPh_lemma}
Let $\vecspace{V}$ be a vector space over $\mathbb{C}$,
$\vecspace{W}$ a Banach space over $\mathbb{C}$, and let
$\Gamma\subseteq\vecspace{V}$ be a complete circular set about
$0$. Suppose that $\omega\colon\vecspace{V}^k\to\vecspace{W}$ is a
$k$-linear mapping and that $X\in\vecspace{V}$. Then
 $$\sup_{W\in \bigcup\limits_{\zeta\in\overline{\mathbb{D}}}(\zeta X+\Gamma)}\|\omega(W,\ldots,W)\|
 =\sup_{W\in X+\Gamma}\|\omega(W,\ldots,W)\|,$$
where $\overline{\mathbb{D}}=\{\zeta\in\mathbb{C}\colon|\zeta|\le
1\}$ is the closed unit disk, $X+\Gamma:=\{X+Z\colon Z\in\Gamma\}$
is a complete circular set about $X$ and
$\bigcup\limits_{\zeta\in\overline{\mathbb{D}}}(\zeta X+\Gamma)$
is a complete circular set about $0$.
\end{lem}
\begin{proof}
Fix an arbitrary $Z_0\in\Gamma$. Consider the function
$$\phi\colon\overline{\mathbb{D}}\to\vecspace{W},\quad
\phi\colon \zeta\mapsto\omega(\zeta X+Z_0,\ldots,\zeta X+Z_0).$$
 Clearly, $\phi$ is a polynomial and hence $\phi$ is
analytic. By the maximum principle, there exists $\zeta_0$ with
$|\zeta_0|=1$ such that
$\|\phi(\zeta_0)\|=\max_{\zeta\in\overline{\mathbb{D}}}\|\phi(\zeta)\|$.
By the homogeneity (of degree $k$) of $\omega$,
\begin{multline*}
\|\omega(\zeta_0 X+Z_0,\ldots,\zeta_0 X+Z_0)\|=\|\omega(
X+\zeta_0^{-1}Z_0,\ldots,
X+\zeta_0^{-1}Z_0)\|\\
\le\sup_{Z\in\Gamma}\|\omega(X+Z,\ldots,X+Z)\|
\end{multline*}
since $\Gamma$ is completely circular about $0$. This proves the
lemma.
\end{proof}

\begin{thm}\label{thm:king-series}
1. The set \index{$\Upsilon_{\rm nc, unif}$}
\begin{equation}\label{eq:conv-set-int-unif}
\Upsilon_{\rm nc,
unif}=\coprod_{m=1}^\infty\Big\{X\in\mat{\vecspace{V}}{sm}\colon\rho_{\rm
unif}\Big(X-\bigoplus_{\alpha=1}^mY\Big)>1\Big\}
\end{equation}
is contained in $\Upsilon_{\rm nc, norm}\subseteq\Upsilon_{\rm nc, fin}$ and is the interior of $\Upsilon_{\rm
nc}$ in the uniformly-open topology on $\ncspace{\vecspace{V}}$.

2. The set \begin{equation}\label{eq:conv-set-normal-unif}
\Upsilon_{\rm nc,
unif-normal}=\coprod_{m=1}^\infty\Big\{X\in\mat{\vecspace{V}}{sm}\colon\rho_{\rm
unif-normal}\Big(X-\bigoplus_{\alpha=1}^mY\Big)>1\Big\}
\end{equation} \index{$\Upsilon_{\rm nc,
unif-normal}$}is a uniformly-open subset of $\Upsilon_{\rm nc,
unif}\subseteq\Upsilon_{\rm nc, norm}\subseteq\Upsilon_{\rm nc,
fin}$ and
\begin{align}
\nonumber \Upsilon_{\rm nc, unif-normal} &=\Big\{X\in\Upsilon_{\rm
nc,
unif}\colon \exists\delta>0\\
 &  \label{eq:ups-unif-bdd} {\rm such\ that\ } \sup_{\ell\in\mathbb{N}}\sup_{W\in
B_{\rm nc}(X,\delta)}
\Big\|\Big(W-\bigoplus_{\alpha=1}^{m_W}Y\Big)^{\odot_s\ell}f_\ell\Big\|_{sm_W}<\infty\Big\}\\
\nonumber &=\Big\{X\in\Upsilon_{\rm nc, unif}\colon
\exists\delta>0\\
& \label{eq:ups-unif-normal} {\rm such\ that\ }
\sum_{\ell=0}^\infty \sup_{W\in B_{\rm nc}(X,\delta)}
\Big\|\Big(W-\bigoplus_{\alpha=1}^{m_W}Y\Big)^{\odot_s\ell}f_\ell\Big\|_{sm_W}<\infty\Big\};
\end{align}
here $m_W$ \index{$m_W$} is the size of a block matrix $W$ with $s
\times s$ blocks. $\Upsilon_{\rm nc, unif-normal}$ is nonempty if
and only if
 it contains $Y$. In this
case,
\begin{itemize}
    \item $\rho_{\rm unif-normal}(0_{s\times
 s})=\infty$ and $\rho_{\rm unif-normal}(Z)>0$ for every
 $Z\in\ncspace{(\mat{\vecspace{V}}{s})}$,
    \item $\Upsilon_{\rm nc, unif-normal}$ is complete
 circular about $Y$ (i.e.,
$(\Upsilon_{\rm nc, unif-normal})_{sm}$ is
 a complete circular set about $\bigoplus_{\alpha=1}^mY$ for every
 $m\in\mathbb{N}$).
    \end{itemize}

3. For every $m\in\mathbb{N}$ and $X\in(\Upsilon_{\rm nc, unif-normal})_{sm}$ there exists a uniformly-open
complete circular set $\Upsilon_{{\rm nc}, X}$ about $\bigoplus_{\alpha=1}^mY$, which contains $X$ and such that
\begin{equation}\label{eq:normal-king-conv}
\sum_{\ell=0}^\infty \sup_{W\in \Upsilon_{{\rm nc}, X}}
\Big\|\Big(W-\bigoplus_{\alpha=1}^{m_W}Y\Big)^{\odot_s\ell}f_\ell\Big\|_{sm_W}<\infty.
\end{equation}

4. For any uniformly-open nc set $\Gamma \subseteq \Upsilon_{\rm nc, unif-normal}$, $f|_\Gamma$ is a uniformly
analytic nc function.

5. The series \eqref{eq:power-series} converges uniformly on every nc ball $B_{\rm nc}(Y,\delta)$ with
$\delta<\rho_{\rm cb}$, moreover
\begin{equation}\label{eq:normal-king-conv-balls}
\sum_{\ell=0}^\infty \sup_{W\in B_{\rm nc}(Y,\delta)}
\Big\|\Big(W-\bigoplus_{\alpha=1}^{m_W}Y\Big)^{\odot_s\ell}f_\ell\Big\|_{sm_W}<\infty.
\end{equation}
The series \eqref{eq:power-series} fails to converge uniformly on every nc ball $B_{\rm nc}(Y,\delta)$ with
$\delta>\rho_{\rm cb}$. The set $\Upsilon_{\rm nc, unif-normal}$ is nonempty if and only if $\rho_{\rm cb}>0$.
\end{thm}

\begin{rem}\label{rem:unif-question}
It is an open question whether $\Upsilon_{\rm nc, unif}$ is
nonempty if and only if
 $\rho_{\rm unif}(0_{s\times s})>0$. If $\rho_{\rm unif}(0_{s\times s})>0$, then similarly to
 Theorems \ref{thm:g-series} and \ref{thm:f-series}, $\rho_{\rm unif}(0_{s\times
 s})=\infty$ and
$\Upsilon_{\rm nc, unif}$ is
 a complete circular set about $Y$.
\end{rem}
\begin{proof}[Proof of Theorem \ref{thm:king-series}]
1. We first observe that the norm topology on
$\ncspace{\vecspace{V}}$ is stronger than the uniformly-open
topology, hence $\rho_{\rm unif}(Z)\le\rho_{\rm norm}(Z)$ for any
$Z$ and $\Upsilon_{\rm nc, unif}\subseteq\Upsilon_{\rm nc, norm}$.
The proof of part 1 is analogous to the proof of part 1 of
Theorems \ref{thm:g-series} and \ref{thm:f-series}.

$2\& 3$. We first show that $\rho_{\rm unif-normal}(\cdot)$ is
lower semicontinuous in the uniformly-open topology on
$\ncspace{\vecspace{V}}$. Let
$$\kappa:=\limsup_{W\overset{\rm unif}{\to} Z}\mu_{\rm unif-normal}(W)=\limsup_{W\overset{\rm
unif}{\to} Z}\lim_{\delta\to
0}\limsup_{\ell\to\infty}\sqrt[\ell]{\sup_{\widetilde{W}\in B_{\rm
nc}(W,\delta)}\|\widetilde{W}^{\odot_s\ell}f_\ell\|}.$$ Then for
an arbitrary $\epsilon>0$ there is a sequence $W_n\overset{\rm
unif}{\to} Z$ such that
$$\lim_{\delta\to
0}\limsup_{\ell\to\infty}\sqrt[\ell]{\sup_{\widetilde{W}\in B_{\rm
nc}(W_n,\delta)}\|\widetilde{W}^{\odot_s\ell}f_\ell\|}>\kappa-\epsilon,\quad n=1,2,\ldots .$$ Clearly, there
exists $N\in\mathbb{N}$ such that $W_n\in B_{\rm nc}(Z,\delta)$ for all $n\ge N$. Therefore,
\begin{multline*}
\mu_{\rm unif-normal}(Z)=\lim_{\delta\to
0}\limsup_{\ell\to\infty}\sqrt[\ell]{\sup_{\widetilde{W}\in B_{\rm
nc}(Z,2\delta)}\|\widetilde{W}^{\odot_s\ell}f_\ell\|}\\
\ge\lim_{\delta\to 0}\limsup_{\ell\to\infty}\sqrt[\ell]{\sup_{\widetilde{W}\in B_{\rm
nc}(W_n,\delta)}\|\widetilde{W}^{\odot_s\ell}f_\ell\|}>\kappa-\epsilon,\quad n\ge N.
\end{multline*}
Since $\epsilon$ can be chosen arbitrarily small, we obtain that
$$\mu_{\rm
unif-normal}(Z)\ge\kappa=\limsup_{W\overset{\rm unif}{\to} Z}\mu_{\rm unif-normal}(W),$$ i.e., $\mu_{\rm
unif-normal}(\cdot)$ is upper semicontinuous, and thus $\rho_{\rm unif-normal}(\cdot)$ is lower semicontinuous,
in the uniformly-open topology on $\ncspace{\vecspace{V}}$. This implies also that the set $\Upsilon_{\rm nc,
unif-normal}$ is uniformly-open.

Since clearly $\rho_{\rm unif-normal}(Z)\le \rho(Z)$ for any $Z$,
by the lower semicontinuity of $\rho_{\rm unif-normal}(\cdot)$ we
obtain that $\rho_{\rm unif-normal}(Z)\le \rho_{\rm unif}(Z)$ for
any $Z$ and hence $\Upsilon_{\rm nc,
unif-normal}\subseteq\Upsilon_{\rm nc, unif}$.

Let $X\in\Upsilon_{\rm nc, unif-normal}$. Then $$\mu_{\rm
unif-normal}\Big(X-\bigoplus_{\alpha=1}^{m_X}Y\Big)\!=\!\lim_{\delta\to
0}\limsup_{\ell\to\infty}\sqrt[\ell]{\sup_{W\in B_{\rm
nc}(X,\delta)}\Big\|\Big(W-\bigoplus_{\alpha=1}^{m_W}Y\Big)^{\odot_s\ell}f_\ell\Big\|_{sm_W}}<1.$$
Hence, there exists $\delta>0$ such that
$$\limsup_{\ell\to\infty}\sqrt[\ell]{\sup_{W\in B_{\rm
nc}(X,\delta)}\Big\|\Big(W-\bigoplus_{\alpha=1}^{m_W}Y\Big)^{\odot_s\ell}f_\ell\Big\|_{sm_W}}<1.$$
This implies that the sequence
$$\sup_{W\in B_{\rm nc}(X,\delta)}\Big\|\Big(W-\bigoplus_{\alpha=1}^{m_W}Y\Big)^{\odot_s\ell}f_\ell\Big\|_{sm_W},
\quad \ell=1,2,\ldots,$$ is bounded, moreover there
exists $C>0$ and $\eta\colon 0<\eta<1$ such that
$$\sup_{W\in B_{\rm
nc}(X,\delta)}\Big\|\Big(W-\bigoplus_{\alpha=1}^{m_W}Y\Big)^{\odot_s\ell}f_\ell\Big\|_{sm_W}\le
C\eta^\ell,\quad \ell=1,2,\ldots.$$ Consequently,
$$\sum_{\ell=0}^\infty\sup_{W\in B_{\rm
nc}(X,\delta)}\Big\|\Big(W-\bigoplus_{\alpha=1}^{m_W}Y\Big)^{\odot_s\ell}f_\ell\Big\|_{sm_W}<\infty.$$
Conversely, if $X\in\Upsilon_{\rm nc, unif}$ is such that for some
$\delta>0$ the series
$$\sum_{\ell=0}^\infty\sup_{W\in B_{\rm
nc}(X,\delta)}\left\|\left(W-\bigoplus_{\alpha=1}^{m_W}Y\right)^{\odot_s\ell}f_\ell\right\|_{sm_W}$$ converges,
then the series
$$\sum_{\ell=0}^\infty\lambda^\ell\sup_{W\in B_{\rm
nc}(X,\delta)}\Big\|\Big(W-\bigoplus_{\alpha=1}^{m_W}Y\Big)^{\odot_s\ell}f_\ell\Big\|_{sm_W}$$
as a series in $\lambda$ converges uniformly and absolutely on the
closed unit disk $\overline{\mathbb{D}}$, and thus its radius of
convergence
$$\left(\limsup_{\ell\to\infty}\sqrt[\ell]{\sup_{W\in B_{\rm
nc}(X,\delta)}\Big\|\Big(W-\bigoplus_{\alpha=1}^{m_W}Y\Big)^{\odot_s\ell}f_\ell\Big\|_{sm_W}}\right)^{-1}>1.$$
Since the left-hand side is increasing as $\delta\downarrow 0$, we
have that \begin{multline*} \rho_{\rm
unif-normal}\Big(X-\bigoplus_{\alpha=1}^{m_X}Y\Big)\\
= \left(\lim_{\delta\to
0}\limsup_{\ell\to\infty}\sqrt[\ell]{\sup_{W\in B_{\rm
nc}(X,\delta)}\Big\|\Big(W-\bigoplus_{\alpha=1}^{m_W}Y\Big)^{\odot_s\ell}f_\ell\Big\|_{sm_W}}\right)^{-1}>1,
\end{multline*}
i.e., $X\in\Upsilon_{\rm nc, unif-normal}$. This proves
equalities \eqref{eq:ups-unif-bdd} and \eqref{eq:ups-unif-normal}.

Next, if $Y\in\Upsilon_{\rm nc, unif-normal}$, then clearly
$\Upsilon_{\rm nc, unif-normal}\neq\emptyset$. Conversely, assume
that  $X\in(\Upsilon_{\rm nc, unif-normal})_{sm}$ for some
$m\in\mathbb{N}$. Then by \eqref{eq:ups-unif-normal} there exists
a nc ball $B_{\rm nc}(X,\delta)$ such that
$$\sum_{\ell=0}^\infty \sup_{W\in B_{\rm
nc}(X,\delta)}\Big\|\Big(W-\bigoplus_{\alpha=1}^{m_W}Y\Big)^{\odot_s\ell}f_\ell\Big\|_{sm_W}<\infty.$$
By Lemma \ref{lem:HiPh_lemma} this inequality holds with the
suprema over $B_{\rm nc}(X,\delta)$ replaced by the suprema over
$$\Upsilon_{{\rm
nc},X}:=\bigcup_{\zeta\in\overline{\mathbb{D}}}B_{\rm
nc}\left(\bigoplus_{\alpha=1}^mY+\zeta\Big(X-\bigoplus_{\alpha=1}^mY\Big),\delta\right).$$
This implies part 3 of the theorem, with $\Upsilon_{{\rm nc}, X}$
as above. In particular, we obtain that
$\bigoplus_{\alpha=1}^mY\in\Upsilon_{\rm nc, unif-normal}$. We
claim $\rho_{\rm unif-normal}(\bigoplus_{j=1}^mW)=\rho_{\rm
unif-normal}(W)$ for every $k\in\mathbb{N}$ and
$W\in\mat{\vecspace{V}}{sk}$ which yields $Y\in\Upsilon_{\rm nc,
unif-normal}$. Indeed, this follows from the equality
$$\Big\|\Big(\bigoplus_{j=1}^mW\Big)^{\odot_s\ell}f_\ell\Big\|_{skm}=\|W^{\odot_s\ell}f_\ell\|_{sk}$$
and the definition of $\rho_{\rm unif-normal}(\cdot)$.

The remaining statements of part 2 (i.e., the ones under the bullets) follow from the positive-homogeneity and
lower semicontinuity of $\rho_{\rm unif-normal}(\cdot)$.

4. Since a uniformly-open nc set is a finitely open nc set, $f|_\Gamma$ is a G-differentiable nc function by
Theorem \ref{thm:g-series}, part 4. It follows from \eqref{eq:ups-unif-normal} that $f|_\Gamma$ is locally
uniformly bounded. We conclude that $f|_\Gamma$ is a uniformly analytic nc function.

5. If $\rho_{\rm cb}>\delta>0$, then we have $$\mu_{\rm
cb}=\limsup_{\ell\to\infty}\sqrt[\ell]{\|f_\ell\|_{\mathcal{L}_{\rm cb}^\ell}}<\infty$$ and there exists $C>0$
such that
$$\|f_\ell\|_{\mathcal{L}_{\rm
cb}^\ell}\le C\mu_{\rm cb}^\ell,\quad \ell=0,1,\ldots.$$ Therefore
\begin{multline*}
\sum_{\ell=0}^\infty\sup_{W\in B_{\rm nc}(0_{s\times
s},\delta)}\|W^{\odot_s\ell}f_\ell\|=\sum_{\ell=0}^\infty\sup_{\|W\|=\delta}\|W^{\odot_s\ell}f_\ell\|\\
=\sum_{\ell=0}^\infty\delta^\ell\sup_{\|W\|=1}\|W^{\odot_s\ell}f_\ell\|=\sum_{\ell=0}^\infty\delta^\ell
\|f_\ell\|_{\mathcal{L}_{\rm cb}^\ell} \le
C\sum_{\ell=0}^\infty\delta^\ell\mu_{\rm cb}^\ell<\infty,
\end{multline*}
which is equivalent to \eqref{eq:normal-king-conv-balls}; here we used Proposition \ref{prop:k-cb-norm} and the
inequality $\delta\mu_{\rm cb}<1$. In particular, the series \eqref{eq:power-series} converges absolutely and
uniformly on $B_{\rm nc}(Y,\delta)$, so that $Y\in\Upsilon_{\rm nc, unif-normal}$ and thus $\Upsilon_{\rm nc,
unif-normal}\neq\emptyset$.

Now suppose $\delta>\rho_{\rm cb}$. Then for any $\epsilon\colon
0<\epsilon<\mu_{\rm cb}-\delta^{-1}$ and $L>0$ there exists
$\ell>L$ such that $\|f_\ell\|_{\mathcal{L}_{\rm
cb}^\ell}>(\mu_{\rm cb}-\epsilon)^\ell$. For any such $\ell$ there
exists $W\in\ncspace{(\mat{\vecspace{V}}{s})}$ with
$\|W\|=(\mu_{\rm cb}-\epsilon)^{-1}<\delta$ such that
$$\|W^{\odot_s\ell} f_\ell\|>(\mu_{\rm
cb}-\epsilon)^\ell\|W\|=1.$$ Hence the sequence $\|W^{\odot_s\ell} f_\ell\|$, $\ell=0,1,\ldots$, does not
converge to $0$ uniformly on $B_{\rm nc}(0_{s\times s},\delta)$, and the series \eqref{eq:power-series} does not
converge uniformly on $B_{\rm nc}(Y,\delta)$.

Finally, if $\Upsilon_{\rm nc, unif-normal}\neq\emptyset$, then by part 2 we have that $Y\in\Upsilon_{\rm nc,
unif-normal}$. Therefore, again by part 2, there exists $\delta>0$ such that the series \eqref{eq:power-series}
converges absolutely and uniformly on $B_{\rm nc}(Y,\delta)$. This implies that $\rho_{\rm cb}\ge\delta>0$. The
proof is complete.
\end{proof}

The following example demonstrates a possibility that $0=\rho_{\rm
unif}(Z)<\rho_{\rm norm}(Z)$ for some $Z$ and thus $\Upsilon_{\rm
nc, unif}\subsetneq\Upsilon_{\rm nc,norm}$. We also have in this
example that $0=\rho_{\rm cb}<\rho_m$ for every fixed
$m\in\mathbb{N}$.
\begin{ex}\label{ex:unif_neq_norm}
It was shown in \cite{Pa1} (see also \cite[Chapter 3]{Pi}) that
for every $p=3,4,\ldots$ there exist $k_p\in\mathbb{N}$ (with
$k_p$ increasing as $p$ increases) and a linear mapping
$\phi_p\colon\mathbb{C}^p\to\mat{\mathbb{C}}{k_p}$ such that
$\max_{\zeta\in\overline{\mathbb{D}^p}}\|\phi_p(\zeta)\|=1$ and
\begin{multline}
\sup_{n\in\mathbb{N},Z_{p1},\ldots,Z_{pp}\in\mat{\mathbb{C}}{n}\colon\|Z_{p1}\|=\ldots=\|Z_{pp}\|=1}
\|(\id_{\mat{\mathbb{C}}{n}}\otimes\phi_p)(Z_{p1},\ldots,Z_{pp})\|\\
=
\|(\id_{\mat{\mathbb{C}}{k_p}}\otimes\phi_p)(Z^0_{p1},\ldots,Z^0_{pp})\|=\sqrt{\frac{p}{2}}
\label{eq:pau1}
\end{multline}
for some $Z^0_{p1}$, \ldots, $Z^0_{pp}\in\mat{\mathbb{C}}{k_p}$ of
norm 1, where $\|\cdot\|$ denotes the operator norm with respect
to the standard Euclidean norms. Moreover,
\begin{equation}\label{eq:pau2}
\|(\id_{\mat{\mathbb{C}}{k_p}}\otimes\phi_p)(Z^0_{p1},\ldots,Z^0_{pp})^\ell\|=
\left(\sqrt{\frac{p}{2}}\right)^\ell,
\quad \ell=1,2,\ldots.
\end{equation}
 Let $\vecspace{V}$ be a Banach space
consisting of bounded double-indexed sequences of complex numbers
$z=(z_{pq})_{p=3,4,\ldots;\, q=1,\ldots,p}$ with the norm
$\|z\|_{\infty}=\sup_{p,q}|z_{pq}|$. One can identify $n\times n$
matrices over $\vecspace{V}$ with sequences of $n\times n$
matrices over $\mathbb{C}$, $Z=(Z_{pq})_{p=3,4,\ldots;\,
q=1,\ldots,p}$. Then multiplication of such a $Z$ on the right (on
the left) with a $n\times n$ matrix over $\mathbb{C}$ is defined
as simultaneous right (left) multiplication with this matrix for
all components $Z_{pq}$. It is easy to see that the norms on
$\mat{\vecspace{V}}{n}$ defined by
$\|Z\|_{\infty,n}=\sup_{p,q}\|Z_{pq}\|$ for $n=1,2,\ldots$
determine an operator space structure on $\vecspace{V}$. Let
$\vecspace{W}$ be a Banach space consisting of block-diagonal
infinite matrices (infinite direct sums) $W=\bigoplus_{p=3}^\infty
W_p$ with $W_p\in\mat{\mathbb{C}}{k_p}$ and with the norm
$\|W\|_\infty\!=\!\sup_p\|W_p\|$. One can identify $n\times n$
matrices over $\vecspace{W}$ with direct sums of $n\times n$
matrices over $\mat{\mathbb{C}}{k_p}$, $W=\bigoplus_{p=3}^\infty
W_{p}$. Multiplication of such a $W$ on the right (on the left)
with a $n\times n$ matrix $S$ over $\mathbb{C}$ is then reduced to
the simultaneous right (left) multiplication with this matrix for
all diagonal (block) entries $W_{p}$ as follows:
$(WS)_p=W_p(S\otimes I_{k_p})$ (resp., $(SW)_p=(S\otimes
I_{k_p})W_p$. It is easy to see that the norms on
$\mat{\vecspace{W}}{n}$ defined by
$\|W\|_{\infty,n}=\sup_{p}\|W_{p}\|$ for $n=1,2,\ldots$ determine
an operator space structure on $\vecspace{W}$. For each
$\ell=3,4,\ldots$ we define a linear mapping
$\phi^{(\ell)}\colon\vecspace{V}\to\vecspace{W}$ by
$$\phi^{(\ell)}(z)=\phi_3(z_{31},z_{32},z_{33})\oplus\cdots\oplus\phi_\ell(z_{\ell
1},\ldots, z_{\ell\ell})\oplus 0_{k_{\ell+1}\times
k_{\ell+1}}\oplus 0_{k_{\ell+2}\times k_{\ell+2}}\oplus\cdots$$
 and a $\ell$-linear mapping
 $f_\ell\colon\vecspace{V}^\ell\to\vecspace{W}$ by
 $$f_\ell(z^1,\ldots,z^\ell)=\phi^{(\ell)}(z^1)\cdots\phi^{(\ell)}(z^\ell),$$
 so that
 $$f_\ell(z,\ldots,z)=\phi^{(\ell)}(z)^\ell.$$
 The mappings $\phi^{(\ell)}$ and $f_\ell$ are extended to
 $n\times n$ matrices over $\vecspace{V}$ in the standard way.
 Then by Proposition \ref{prop:k-cb-norm} and the equalities
 \eqref{eq:pau1}, \eqref{eq:pau2} we have
 \begin{multline*}
\|f_\ell\|_{\mathcal{L}_{\rm
cb}^\ell}=\sup_{n\in\mathbb{N},Z\in\mat{\vecspace{V}}{n}\colon
\|Z\|_{\infty,n}=1}\|Z^{\odot\ell}f_\ell\|_{\infty,n}\\
=\sup_{n\in\mathbb{N},Z\in\mat{\vecspace{V}}{n}\colon\|Z\|_{\infty,n}=1}\|f_\ell(Z,\ldots,
Z)\|_{\infty,n}
\end{multline*}
\begin{multline*}
 =
\sup_{n\in\mathbb{N},Z\in\mat{\vecspace{V}}{n}\colon\|Z\|_{\infty,n}=1}
\|\phi^{(\ell)}(Z)^\ell\|_{\infty,n}\\
=\sup_{n\in\mathbb{N},Z_{\ell 1},\ldots,Z_{\ell
\ell}\in\mat{\mathbb{C}}{n}\colon\|Z_{\ell
1}\|=\ldots=\|Z_{\ell\ell}\|=1}
\|(\id_{\mat{\mathbb{C}}{n}}\otimes\phi_\ell)(Z_{\ell
1},\ldots,Z_{\ell\ell})^\ell\|\\
=\left(\sqrt{\frac{\ell}{2}}\right)^\ell,
 \end{multline*}
hence the $\ell$-linear mappings $f_\ell$, $\ell=3,4,\ldots$, are
completely bounded. Now, consider the power series
$$\sum_{\ell=3}^\infty Z^{\odot\ell}f_\ell.$$
We have for every fixed $m\in\mathbb{N}$ and
$Z=(Z_{pq})_{p=3,4,\ldots;q=1,\ldots,p}\in\mat{\vecspace{V}}{m}$
with $\|Z\|_{\infty,m}=1$ that
\begin{multline*}
\mu(Z)=\limsup_{\ell\to\infty}\sqrt[\ell]{\| Z^{\odot\ell}
f_\ell\|_{\infty,m}}\le\limsup_{\ell\to\infty}\|\phi^{(\ell)}(Z)\|\\
\le m^2\limsup_{\ell\to\infty}\max_{p=3,\ldots,\ell}\max_{1\le
i,j\le m}\|\phi_p(Z_{p 1}^{ij},\ldots,Z_{pp}^{ij})\|\le m^2,
\end{multline*}
where $Z_{pq}^{ij}$ is the $(i,j)$-th entry of the matrix
$Z_{pq}$. Therefore $\rho(Z)\ge m^{-2}$ for every
$Z\in\mat{\vecspace{V}}{m}$ with $\|Z\|_{\infty,m}=1$. It is easy
to see that by the homogeneity of degree $-1$ we also have
$\rho_{\rm norm}(Z)\ge m^{-2}$ for every
$Z\in\mat{\vecspace{V}}{m}$ with $\|Z\|_{\infty,m}=1$. This
implies that $\rho_{\rm norm}(Z)>0$ for every
$Z\in\ncspace{\vecspace{V}}$ and hence $\rho_{\rm norm}(0_{m\times
m})=\infty$ for every $m\in\mathbb{N}$. We also obtain that
$\rho_m\ge m^{-2}$ for every $m\in\mathbb{N}$.

Define the sequence
$Z^j=(Z^j_{pq})_{p=3,4,\ldots;q=1,\ldots,p}\in\mat{\vecspace{V}}{k_j}$,
$j=3,4,\ldots$, by
$$Z^j_{pq}=\left\{\begin{array}{lcl}
0_{k_j\times k_j} & {\rm for} & p\neq j,\\
\frac{1}{\sqrt[4]{j}}Z^0_{jq} & {\rm for} & p=j,\ 1\le q\le j.
\end{array}\right.$$
Then
$$\lim_{j\to\infty}\|Z^j\|_{\infty,k_j}=\lim_{j\to\infty}\frac{1}{\sqrt[4]{j}}=0,$$
however
\begin{multline*}
\lim_{j\to\infty}\mu(Z^j)=\lim_{j\to\infty}\limsup_{\ell\to\infty}
\sqrt[\ell]{\|(Z^j)^{\odot\ell}f_\ell\|_{\infty,k_j}}\\
=\lim_{j\to\infty}\frac{1}{\sqrt[4]{j}}\limsup_{\ell\to\infty}
\sqrt[\ell]{\|(\id_{\mat{\mathbb{C}}{k_j}}\otimes\phi_j)
(Z^0_{j1},\ldots,Z^0_{jj})^\ell\|_{\infty,k_j}}=\lim_{j\to\infty}
\frac{1}{\sqrt[4]{j}}\sqrt{\frac{j}{2}}=\infty.
\end{multline*}
Thus $$\rho_{\rm unif}(0_{1\times
1})=\lim_{j\to\infty}\rho(Z^j)=0.$$ Note that we can show
analogously that $\rho_{\rm unif}(Z)=0$ when $Z\neq 0$. E.g., for
$z=(z_{pq})_{p=3,4,\ldots;q=1,\ldots,p}$ with any $z_{3,q}\neq 0$,
$q=1,2,3$, and $z_{pq}=0$ for all $p>3$ and all $q$, we can define
a sequence $z\otimes I_{k_j}+Z^j$ with $Z^j$ as above and verify
that $$\rho_{\rm unif}(z)=\lim_{j\to\infty}\rho(z\otimes
I_{k_j}+Z^j)=0.$$

Finally, we observe that $$\rho_{\rm
cb}=\left(\limsup_{\ell\to\infty}\sqrt[\ell]{\|f_\ell\|_{\mathcal{L}_{\rm
cb}^\ell}}\right)^{-1}=\left(\limsup_{\ell\to\infty}\sqrt{\frac{\ell}{2}}\right)^{-1}=0.$$
Thus all the statements formulated in the paragraph preceding this
example are true.
\end{ex}

In the next exampl we demonstrate the possibility that
$0=\rho_{\rm unif-normal}(0_{1\times 1})\!<\!\rho_{\rm
unif}(0_{1\times 1})$ and thus $\emptyset=\Upsilon_{\rm nc,
unif-normal}\subsetneq\Upsilon_{\rm nc, unif}$. The sum of the nc
power series in this example is an entire (i.e., analytic
everywhere and bounded on every ball of $n\times n$ matrices,
$n=1,2,\ldots$) nc function which is not uniformly analytic at
$0_{1\times 1}$, that is, not uniformly analytic in any
uniformly-open nc neighborhood of $0_{1\times 1}$.
\begin{ex}\label{ex:unif-normal_neq_unif}
Let $\vecspace{V}=\mathbb{C}^2$ with the norm
$\|(z_1,z_2)\|=\max\{|z_1|,|z_2|\}$ and $\vecspace{W}=\mathbb{C}$.
We identify $n\times n$ matrices over $\vecspace{V}=\mathbb{C}^2$
with pairs of matrices over $\mathbb{C}$. It is clear that the
norms on $\mat{\vecspace{V}}{n}$, $n=1,2,\ldots$, defined by
$\|(Z_1,Z_2)\|_n=\max\{\|Z_1\|,\|Z_2\|\}$ and componentwise
multiplication of pairs $(Z_1,Z_2)$ of matrices over $\mathbb{C}$
on the left (on the right) with a matrix over $\mathbb{C}$
determine the operator space structure on $\vecspace{V}$. The
canonical (i.e., with $\|1\|_1=1$) operator space structure on
$\vecspace{W}=\mathbb{C}$ is uniquely determined. Consider the
power series
$$\sum_{k=1}^\infty Z^{\odot \alpha_k}f_{\alpha_k}=\sum_{k=1}^\infty c_kp_{k}(Z_1,Z_2)=
\sum_{k=1}^\infty c_k\sum_{\pi\in
\mathcal{S}_{k+1}}{\rm sign}(\pi)Z_1^{\pi(1)-1}Z_2\cdots
Z_1^{\pi(k+1)-1}Z_2$$ with positive constants $c_k$ which will be
defined later and $\alpha_k=\deg p_k=\frac{(k+1)(k+2)}{2}$ (see
Example \ref{ex:degrees_unbdd} for more details on the polynomials
$p_k$). The $\alpha_k$-linear mappings $f_{\alpha_k}$ are
completely bounded. Indeed, by Proposition \ref{prop:k-cb-norm} we
have
\begin{multline*}
\|f_{\alpha_k}\|_{\mathcal{L}_{\rm
cb}^{\alpha_k}}=\sup_{n\in\mathbb{N},\
Z=(Z_1,Z_2)\in(\mat{\mathbb{C}}{n})^2\colon
\max\{\|Z_1\|,\|Z_2\|\}=1}c_k\|p_k(Z_1,Z_2)\|_n\\
\le c_k(k+1)!<\infty.
\end{multline*}
Since the polynomials $p_k$ vanish identically on pairs of
$k\times k$ matrices, for every fixed $n\in\mathbb{N}$ and any
$Z=(Z_1,Z_2)\in (\mat{\mathbb{C}}{n})^2$ the series contains only
the powers smaller than $\alpha_n$ and thus converges uniformly on
$(\mat{\mathbb{C}}{n})^2$. Therefore, $\rho(Z)=\infty$ for every
$Z\in\ncspace{\vecspace{V}}$. We also have $\rho_{\rm
unif}(Z)=\infty$ for every $Z\in\ncspace{\vecspace{V}}$ and
$\Upsilon_{\rm nc, unif}=\ncspace{\vecspace{V}}$. On the other
hand, any nonzero polynomial $p$ in two noncommuting
indeterminates is not identically zero on pairs of square matrices
of sufficiently large size due to the absence of polynomial
identities for matrices of all sizes --- see \cite[Example 1.4.4
and Theorem 1.4.5, page 22]{Row80}. Therefore
$$\sup_{n\in\mathbb{N},\
Z=(Z_1,Z_2)\in(\mat{\mathbb{C}}{n})^2\colon
\max\{\|Z_1\|,\|Z_2\|\}=1}\|p_k(Z_1,Z_2)\|_n\neq 0.$$ Define
$$c_k=k^{\alpha_k}\Big(\sup_{n\in\mathbb{N},\
Z=(Z_1,Z_2)\in(\mat{\mathbb{C}}{n})^2\colon
\max\{\|Z_1\|,\|Z_2\|\}=1}\|p_k(Z_1,Z_2)\|_n\Big)^{-1}.$$ Then
$$\|f_{\alpha_k}\|_{\mathcal{L}_{\rm
cb}^{\alpha_k}}=k^{\alpha_k}.$$ For any $\delta>0$ we have
\begin{multline*}
\sup_{\ell\in\mathbb{N}}\sup_{W\in B_{\rm nc}(0_{1\times
1},\delta)}
\|W^{\odot\ell}f_\ell\|\\
=\sup_{k\in\mathbb{N}}c_k\delta^{\alpha_k}k^{\alpha_k}\sup_{n\in\mathbb{N},\
Z=(Z_1,Z_2)\in(\mat{\mathbb{C}}{n})^2\colon
\max\{\|Z_1\|,\|Z_2\|\}=1}\|p_k(Z_1,Z_2)\|_n\\
=\sup_{k\in\mathbb{N}}(\delta k)^{\alpha_k}=\infty.
\end{multline*}
Then by part 2 of Theorem \ref{thm:king-series} we have that
$0_{1\times 1}\notin \Upsilon_{\rm nc, unif-normal}$ and thus
$\Upsilon_{\rm nc, unif-normal}=\emptyset$. Alternatively, we can
compute
$$
\mu_{\rm unif-normal}(0_{1\times 1})=\lim_{\delta\to
0}\limsup_{\ell\to 0}\sqrt[\ell]{\sup_{W\in B_{\rm nc}(0_{1\times
1},\delta)} \|W^{\odot\ell}f_\ell\|}=\lim_{\delta\to
0}\lim_{k\to\infty}\delta k=\infty,
$$
and thus $\rho_{\rm unif-normal}(0_{1\times 1})=0$.  By Remark
\ref{rem:series-uniqueness} and Theorem \ref{thm:king-conv}, the
sum of the series, $f(Z)=\sum_{k=1}^\infty
Z^{\odot\alpha_k}f_{\alpha_k}$ is not uniformly analytic at
$0_{1\times 1}$.
\end{ex}

\label{OPEN-QUESTION}

\begin{rem}\label{rem:unif-normal_neq_unif}
It is not too difficult to modify Example \ref{ex:unif-normal_neq_unif} by changing the coefficients $c_k$ so
that $\Upsilon_{\rm nc, unif-normal}\subsetneq \Upsilon_{\rm nc, unif}=\ncspaced{\mathbb{C}}{2}$ contains a nc
ball centered at $0_{1\times 1}$ of a given finite radius.
\end{rem}

We consider now the case where $\vecspace{V}=\nspace{C}{d}$ with some operator space structure and we expand each
term of the power series \eqref{eq:power-series} obtaining a series along the free monoid $\free_d$. Our
results will be the converse of Theorem \ref{thm:king-semigr}. In other words, we shall consider the convergence
of power series of the form
\begin{equation}\label{eq:semigr-power-series}
\sum_{w\in\free_d}
\Big(X-\bigoplus_{\alpha=1}^mY\Big)^{\odot_sw}f_w.
\end{equation}
Here $Y=(Y_1,\ldots,Y_d)\in(\mat{\mathbb{C}}{s})^d$ is a given
center; $X=(X_1,\ldots,X_d)\in(\mat{\mathbb{C}}{sm})^d$,
$m=1,2,\ldots$; $Z^{\odot_sw}$ for $w=g_{i_1}\cdots g_{i_\ell}$
and $Z=(Z_1,\ldots,Z_d)\in(\mat{\mathbb{C}}{sm})^d$ is defined in
\eqref{eq:faux-power-semigr};
$f_w\colon(\mat{\mathbb{C}}{s})^\ell\to\mat{\vecspace{W}}{s}$,
$w\in\free_d$, $\ell=|w|$, is a given sequence of $\ell$-linear
mappings (here $\mathbb{C}$ and, hence, $\mat{\mathbb{C}}{s}$ are
equipped with the canonical operator space structure); and
conditions
\eqref{eq:ncfun_coef_empty}--\eqref{eq:ncfun-coef_w_ell} are
satisfied.
 The series
\eqref{eq:semigr-power-series} can be alternatively written as
\begin{equation*}
\sum_{w\in\free_d}\left(\Big(\bigoplus_{\alpha=1}^mA_{w,(0)}\Big)\otimes
\Big(\bigoplus_{\alpha=1}^mA_{w,(1)}\Big)
\otimes\cdots\otimes\Big(\bigoplus_{\alpha=1}^mA_{w,(|w|)}\Big)\right)\star
Z^{[w]}.
\end{equation*}
Here
$$f_w=A_{w,(0)}\otimes
A_{w,(1)}\otimes\cdots\otimes A_{w,(|w|)}\in\mat{\vecspace{W}}{s}\otimes\underset{|w|\ {\rm
times}}{\underbrace{\mat{\mathbb{C}}{s}\otimes \cdots\otimes\mat{\mathbb{C}}{s}}},$$ with the tensor product
interpretation for the multilinear mappings $f_w$ (Remark \ref{rem:natur_map}), the sumless Sweedler notation
\eqref{eq:Sweedler}, and the pseudo-power notation \eqref{eq:Ramamurti}. In the case $s=1$, the series
\eqref{eq:semigr-power-series} becomes
$$\sum_{w\in\free_d}Z^wf_w,$$
with $f_w\in\vecspace{W}$; more precisely, we identify the $\ell$-linear mapping
$f_w\colon\mathbb{C}^\ell\to\vecspace{W}$ with the vector $f_w(1,\ldots, 1)\in\vecspace{W}$, and
\begin{equation}\label{eq:s=1}
\|Z^wf_w\|_m=\|Z^w\|\|f_w\|_1
\end{equation}
for every $m\in\mathbb{N}$ and $Z\in\mattuple{\mathbb{C}}{m}{d}$ (see \cite[Proposition 1.10(ii)]{Pi}), so that
\begin{equation}\label{eq:s=1_cb_norm}
\|f_w\|_{\mathcal{L}_{\rm cb}^\ell}=\|f_w\|_1.
\end{equation}

We associate with the series \eqref{eq:semigr-power-series} the
series \eqref{eq:power-series} where the $\ell$-linear mappings
$f_\ell\colon(\mattuple{\mathbb{C}}{s}{d})^\ell\to\mat{\vecspace{W}}{s}$
are defined by \eqref{eq:lw-forms}, so that \eqref{eq:lw-powers}
and \eqref{eq:f_w_via_f_l} hold. Similarly to Remark
\ref{rem:lost_abbey_semigr}, conditions
\eqref{eq:ncfun_coef_empty}--\eqref{eq:ncfun-coef_w_ell} on the
linear mappings $f_w$ are easily seen to be equivalent to
conditions \eqref{eq:ncfun_coef_0}--\eqref{eq:ncfun-coef_ell_ell}
on the linear mappings $f_\ell$. Notice that the completely
bounded norms of $f_\ell$ and of $f_w$ for all $w$ with $|w|=\ell$
are related by
\begin{equation}\label{eq:lw-cb-norms-1}
\|f_\ell\|_{\mathcal{L}_{\rm cb}^\ell}\le\sum_{|w|=\ell}\mathbf{C}_1^w\|f_w\|_{\mathcal{L}_{\rm cb}^\ell},
\end{equation}
where $\mathbf{C}_1=(\|\pi_1\|_{\rm cb},\ldots,\|\pi_d\|_{\rm cb})$ for the $i$-th coordinate mapping
$\pi_i\colon\mathbb{C}^d\to\mathbb{C}$, and
\begin{equation}\label{eq:lw-cb-norms-2}
\|f_w\|_{\mathcal{L}_{\rm cb}^\ell}\le\mathbf{C}_2^w\|f_\ell\|_{\mathcal{L}_{\rm cb}^\ell},
\end{equation}
where $\mathbf{C}_2=(\|e_1\|_1,\ldots,\|e_d\|_1)$.

Let $\vecspace{X}$ be a vector space over $\mathbb{C}$ with a
direct sum decomposition
$\vecspace{X}=\bigoplus_{j=1}^d\vecspace{X}_j$, and let
$Y=(Y_1,\ldots,Y_d)\in\vecspace{X}$. A set
$\Upsilon\subseteq\vecspace{X}$ is called \emph{complete
Reinhardt-like about $Y$} \index{complete Reinhardt-like set about
$Y$} if for every $X=(X_1,\ldots,X_d)\in\Upsilon$ and every
$t=(t_1,\ldots,t_d)\in\mathbb{C}^d$ with $|t_j|\le 1$ for all $j$,
one has $Y+t\circ (X-Y)\in\Upsilon$, where $t\circ
(X-Y):=(t_1(X_1-Y_1),\ldots,t_d(X_d-Y_d))$. Clearly, a complete
Reinhardt-like set is complete circular. In the case of $d=1$ we
recover the definition of complete circular set. In the case where
$\vecspace{X}=\mathbb{C}^d$ and $\vecspace{X}_j=\mathbb{C}e_j$,
$j=1,\ldots,d$, we recover the usual definition of complete
Reinhardt set \cite[Page 8]{Sh}.

 The following is an analogue of Lemma
\ref{lem:HiPh_lemma} for complete Reinhardt-like sets.
\begin{lem}\label{lem:HiPh-Reinhardt}
Let $\vecspace{X}$ be a vector space over $\mathbb{C}$ with the
direct sum decomposition
$\vecspace{X}=\bigoplus_{j=1}^d\vecspace{X}_j$, let $\vecspace{W}$
be a Banach space over $\mathbb{C}$, and let
$\Gamma\subseteq\vecspace{X}$ be a complete Reinhardt-like set
about $0$. Suppose that
$\omega\colon\vecspace{X}^k\to\vecspace{W}$ is a $k$-linear
mapping whose restriction to the set
$$\diag[\vecspace{X}^k]=\{(\underset{k\ {\rm
times}}{\underbrace{W,\ldots, W}})\colon
W=(W_1,\ldots,W_d)\in\vecspace{X}\}$$ is homogeneous in $W_j$ of
degree $k_j$ (so that $k_1+\cdots +k_d=k$), and suppose that
$X=(X_1,\ldots,X_d)\in\vecspace{X}$. Then
 $$\sup_{W\in \bigcup\limits_{t\in\overline{\mathbb{D}^d}}(t\circ X+\Gamma)}\|\omega(W,\ldots,W)\|
 =\sup_{W\in X+\Gamma}\|\omega(W,\ldots,W)\|,$$
where
$\overline{\mathbb{D}^d}=\overline{\mathbb{D}}^d=\{t=(t_1,\ldots,t_d)\in\mathbb{C}^d\colon|t_j|\le
1\}$ is the closed unit polydisk, $X+\Gamma:=\{X+Z\colon
Z\in\Gamma\}$ is a complete Reinhardt-like set about $X$ and
$\bigcup\limits_{t\in\overline{\mathbb{D}^d}}(t\circ X+\Gamma)$ is
a complete Reinhardt-like set about $0$.
\end{lem}
\begin{proof}
Fix an arbitrary $Z^0\in\Gamma$. Consider the function
$$\phi\colon\overline{\mathbb{D}^d}\to\vecspace{W},\quad
\phi\colon t\mapsto\omega(t\circ X+Z^0,\ldots,t\circ X+Z^0).$$
 Clearly, $\phi$ is a polynomial in $t_1$, \ldots, $t_d$ and hence $\phi$ is
analytic. By the maximum principle, there exists $t^0$ with
$|t^0_j|=1$, $j=1,\ldots,d$, such that
$\|\phi(t^0)\|=\max_{t\in\overline{\mathbb{D}^d}}\|\phi(t)\|$. By
the simultaneous homogeneity of $\omega$,
\begin{multline*}
\|\omega(t^0\circ X+Z^0,\ldots,t^0\circ X+Z^0)\|=\|\omega(
X+(t^0)^{-1}\circ Z^0,\ldots,
X+(t^0)^{-1}\circ Z^0)\|\\
\le\sup_{Z\in\Gamma}\|\omega(X+Z,\ldots,X+Z)\|,
\end{multline*} where
$(t^0)^{-1}=((t^0_1)^{-1},\ldots,(t^0_d)^{-1})$, since $\Gamma$ is
complete Reinhardt-like about $0$. This proves the lemma.
\end{proof}

We shall be mostly interested in the case where
$\vecspace{X}=(\mat{\mathbb{C}}{n})^d$ and
$\vecspace{X}_j=\mat{\mathbb{C}}{n}e_j$. A nc set $\Upsilon_{\rm
nc}\subseteq\ncspaced{\mathbb{C}}{d}$ is called \emph{complete
Reinhardt-like about $Y\in\mattuple{\mathbb{C}}{s}{d}$}
\index{complete Reinhardt-like nc set about $Y$} if
$(\Upsilon_{\rm nc})_{sm}\subseteq \mattuple{\mathbb{C}}{sm}{d}$
is complete Reinhardt-like about $\bigoplus_{\alpha=1}^mY$ for
every $m\in\mathbb{N}$.

For $Y=(Y_1,\ldots,Y_d)\in\mattuple{\mathbb{C}}{s}{d}$ and
$\mathbf{r}=(r_1,\ldots,r_d)$ a $d$-tuple of positive extended
real numbers  (i.e., $0< r_j\le\infty$, $j=1$, \ldots, $d$),
define the \emph{nc polydisk centered at $Y$ of multiradius
$\mathbf{r}$} \index{nc polydisk} \index{$(\mathbb{D}^d)_{\rm
nc}(Y,\mathbf{r})$} \begin{multline*} (\mathbb{D}^d)_{\rm
nc}(Y,\mathbf{r})\\
=\Big\{Z=(Z_1,\ldots,Z_d)\in\coprod_{m=1}^\infty\mattuple{\mathbb{C}}{sm}{d}\colon
\Big\|Z_j-\bigoplus_{\alpha=1}^mY_j\Big\|_{sm}<r_j,\ j=1,\ldots,d\Big\}.
\end{multline*}
Clearly, $(\mathbb{D}^d)_{\rm nc}(Y,\mathbf{r})$ is a complete
Reinhardt-like nc set about $Y$.

For $\mathbf{r}=(r_1,\ldots,r_d)$ a $d$-tuple of nonnegative real
numbers, we define \index{$\mu(\mathbf{r})$}
\begin{equation}\label{eq:br-mu}
\mu(\mathbf{r})=\limsup_{\ell\to\infty}\sqrt[\ell]{\sum_{w\in\free_d\colon
|w|=\ell}\mathbf{r}^w\|f_w\|_{\mathcal{L}_{\rm cb}^\ell}}.
\end{equation}
\begin{lem}\label{lem:mu}
The function $\mu(\cdot)$ is continuous on $d$-tuples of strictly
positive real numbers.
\end{lem}
\begin{proof}
Let $\mathbf{r}^0=(r_1^0,\ldots,r_d^0)\in\mathbb{R}^d_+$ with
$r_j^0>0$ for all $j=1,\ldots,d$. Let $\epsilon>0$. Then for all
$d$-tuples $\mathbf{r}$ of positive real numbers satisfying
$1-\epsilon<\frac{r_j^0}{r_j}<1+\epsilon$ for all $j=1,\ldots,d$
we have
$$(1-\epsilon)\mu(\mathbf{r})=\mu((1-\epsilon)\mathbf{r})\le\mu(\mathbf{r}^0)\le\mu((1+\epsilon)\mathbf{r})=
(1+\epsilon)\mu(\mathbf{r}),$$ since the function $\mu(\cdot)$ is
nondecreasing in each argument and homogeneous of degree one.
Therefore
$\lim_{\mathbf{r}\to\mathbf{r}^0}\mu(\mathbf{r})=\mu(\mathbf{r}^0)$,
i.e., $\mu(\cdot)$ is continuous at $\mathbf{r}^0$.
\end{proof}

We call nonnegative extended real numbers $\rho_1$, \ldots,
$\rho_d$
 \emph{associated radii of normal convergence of the
series} \index{associated radii of normal convergence of a series}
\eqref{eq:semigr-power-series} if $\mu(\mathbf{r})\le 1$ for every
$\mathbf{r}=(r_1,\ldots,r_d)$ with $r_j\le\rho_j$, $j=1,\ldots,d$,
and for every $d$-tuple $\mathbf{r}'=(r_1',\ldots,r_d')$ of
nonnegative extended real numbers with $r_j'\ge \rho_j$ so that
strict inequality occurs for at least one $j$, there exists a
$d$-tuple $\mathbf{r}=(r_1,\ldots,r_d)$ of nonnegative real
numbers $r_j\le r_j'$, $j=1,\ldots,d$, such that
$\mu(\mathbf{r})>1$.

\begin{thm}\label{thm:free-series}
1. The set of absolute convergence of the series
\eqref{eq:semigr-power-series}, \index{$\Upsilon_{\rm free, abs}$}
\begin{multline}\label{eq:free-abs-conv-set}
\Upsilon_{\rm free,
abs}=\coprod_{m=1}^\infty\Big\{X\in\mattuple{\mathbb{C}}{sm}{d}\colon
\sum_{w\in\free_d}\Big\|\Big(X-\bigoplus_{\alpha=1}^mY\Big)^{\odot_sw}f_w\Big\|_{sm}<\infty\Big\}\\
\subseteq\Upsilon_{\rm nc}\subseteq\ncspaced{\mathbb{C}}{d},
\end{multline}
is a nc set which is complete Reinhardt-like about $Y$. Here $\Upsilon_{\rm nc}$ is the convergence set (see
Theorem \ref{thm:g-series}) of the corresponding series \eqref{eq:power-series}; see \eqref{eq:lw-forms} and
\eqref{eq:lw-powers}.

2. The set \index{$\Upsilon_{\rm free, unif-normal}$}
\begin{multline}\label{eq:free-normal} \Upsilon_{\rm free,
unif-normal} =\Big\{X\in\Upsilon_{\rm free,
abs}\colon \exists\delta>0 {\ \rm such\ that\ }\\
\sum_{w\in\free_d} \sup_{W\in B_{\rm nc}(X,\delta)}
\Big\|\Big(W-\bigoplus_{\alpha=1}^{m_W}Y\Big)^{\odot_sw}f_w\Big\|_{sm_W}<\infty\Big\}
\subseteq\Upsilon_{\rm nc,unif-normal}
\end{multline}
(see Theorem \ref{thm:king-series} for the definition of the set $\Upsilon_{\rm nc,unif-normal}$; here $m_W$ is
the size of a block matrix $W$ with $s \times s$ blocks) is a uniformly-open subset of $\Upsilon_{\rm free,abs}$
which is nonempty if and only if
 it contains $Y$. In this
case, $\Upsilon_{\rm free, unif-normal}$ is complete Reinhardt-like about $Y$ (i.e., $(\Upsilon_{\rm free,
unif-normal})_{sm}$ is
 a complete Reinhardt-like set about $\bigoplus_{\alpha=1}^mY$ for every
 $m\in\mathbb{N}$).

3. For every $m\in\mathbb{N}$ and $X\in(\Upsilon_{\rm free, unif-normal})_{sm}$ there exists a uniformly-open
complete Reinhardt-like set $\Upsilon_{{\rm free}, X}$ about $\bigoplus_{\alpha=1}^mY$, which contains $X$ and
such that
\begin{equation}\label{eq:normal-free-conv}
\sum_{w\in\free_d} \sup_{W\in \Upsilon_{{\rm free}, X}}
\Big\|\Big(W-\bigoplus_{\alpha=1}^{m_W}Y\Big)^{\odot_sw}f_w\Big\|_{sm_W}<\infty.
\end{equation}

4. Let $\brho=(\rho_1,\ldots,\rho_d)$ be a $d$-tuple of associated radii of normal convergence of the series
\eqref{eq:semigr-power-series}. Then the series converges absolutely on every nc polydisk $(\mathbb{D}^d)_{\rm
nc}(Y,\mathbf{r})$ with $r_j<\rho_j$ for all $j=1,\ldots,d$, moreover
\begin{equation}\label{eq:normal-free-conv-poly}
\sum_{w\in\free_d} \sup_{W \in (\mathbb{D}^d)_{\rm
nc}(Y,\mathbf{r})}
\Big\|\Big(W-\bigoplus_{\alpha=1}^{m_W}Y\Big)^{\odot_sw}f_w\Big\|_{sm_W}<\infty.
\end{equation}
For every $\mathbf{r}$ with $r_j\ge \rho_j$ for all $j=1,\ldots,d$, so that strict inequality occurs for at least
one $j$, \eqref{eq:normal-free-conv-poly} fails. In the case $s=1$, there exists $Z\in(\mathbb{D}^d)_{\rm
nc}(Y,\mathbf{r})$ such that the series \eqref{eq:semigr-power-series} does not converge absolutely at $Z$. The
set $\Upsilon_{\rm free, unif-normal}$ is nonempty if and only if there exists a $d$-tuple of strictly positive
associated radii of normal convergence of the series \eqref{eq:semigr-power-series}.
\end{thm}

\begin{proof}[Proof of Theorem \ref{thm:free-series}]
1. The statements in this part are straightforward.

2 $\&$ 3. The inclusion in \eqref{eq:free-normal} and the fact
that $\Upsilon_{\rm free, unif-normal}$ is a uniformly-open subset
of $\Upsilon_{\rm free, abs}$ are obvious.

Clearly, if $Y\in\Upsilon_{\rm free, unif-normal}$, then $\Upsilon_{\rm free, unif-normal}$ is nonempty.
Conversely, suppose there exist $m\in\mathbb{N}$ and $X\in(\Upsilon_{\rm free, unif-normal})_{sm}$. Then there
exists $\delta>0$ such that
\begin{equation}\label{eq:free-normal-ball}
 \sum_{w\in\free_d} \sup_{W\in B_{\rm nc}(X,\delta)}
\Big\|\Big(W-\bigoplus_{\alpha=1}^{m_W}Y\Big)^{\odot_sw}f_w\Big\|_{sm_W}<\infty.
\end{equation}
 There
exists a nc polydisk $(\mathbb{D}^d)_{\rm nc}(X,\brho)$ contained
in $B_{\rm nc}(X,\delta)$, e.g., one can set
$\rho_j=\frac{\delta}{d\|e_j\|_1}$, $j=1,\ldots,d$, and then for
$W\in(\mathbb{D}^d)_{\rm nc}(X,\brho)_{smk}$ one has
$$\Big\|W-\bigoplus_{\alpha=1}^kX\Big\|_{smk}\le\sum_{j=1}^d\Big\|W_j-\bigoplus_{\alpha=1}^kX_j\Big\|_{smk}
\|e_j\|_1<\delta,$$ i.e., $W\in B_{\rm nc}(X,\delta)_{smk}$. Since
$\Gamma=(\mathbb{D}^d)_{\rm
nc}\left(\bigoplus_{\alpha=1}^mY,\brho\right)$ is a complete
Reinhardt-like nc set about $\bigoplus_{\alpha=1}^mY$, it follows
from Lemma \ref{lem:HiPh-Reinhardt} that the inequality in
\eqref{eq:free-normal-ball} holds with the suprema over $B_{\rm
nc}(X,\delta)$ replaced by the suprema over $$\Upsilon_{\rm
free,X}:=\bigcup\limits_{t\in\overline{\mathbb{D}^d}}(\mathbb{D}^d)_{\rm
nc}\!\left(\bigoplus_{\alpha=1}^mY+t\circ
\Big(X-\bigoplus_{\alpha=1}^mY\Big),\brho\right).$$ This implies
part 3 of the theorem with $\Upsilon_{\rm free,X}$ as above. In
particular, $\bigoplus_{\alpha=1}^mY\in\Upsilon_{\rm free,
unif-normal}$ and
$$\sum_{w\in\free_d} \sup_{W\in B_{\rm nc}(\bigoplus_{\alpha=1}^mY,\delta)}
\Big\|\Big(W-\bigoplus_{\alpha=1}^{m_W}Y\Big)^{\odot_sw}f_w\Big\|_{sm_W}<\infty.$$
By picking up only $W\in B_{\rm
nc}\left(\bigoplus_{\alpha=1}^mY,\delta\right)$ of the form
$W=\bigoplus_{\alpha=1}^{m}V$, where $V\in B_{\rm nc}(Y,\delta)$,
we obtain that
$$\sum_{w\in\free_d} \sup_{V\in B_{\rm nc}(Y,\delta)}
\Big\|\Big(V-\bigoplus_{\alpha=1}^{m_V}Y\Big)^{\odot_sw}f_w\Big\|_{sm_V}<\infty,$$
i.e., $Y\in\Upsilon_{\rm free, unif-normal}$.

 Since $Y\in\Upsilon_{\rm free, unif-normal}$ and since the
union of complete Reinhardt-like sets about
$\bigoplus_{\alpha=1}^mY$ for all $m\in\mathbb{N}$ is a complete
Reinhardt-like set about $Y$, the set $$\Upsilon_{\rm free,
unif-normal}=\bigcup_{X\in\Upsilon_{\rm free,
unif-normal}}\Upsilon_{\rm free,X}$$ is a complete Reinhardt-like
set about $Y$. This completes the proof of part 2.

4. Let $\brho=(\rho_1,\ldots,\rho_d)$ be a $d$-tuple of associated
radii of absolute convergence of the series
\eqref{eq:semigr-power-series}, and let $\mathbf{r}=(r_1,\ldots,
r_d)$ be a $d$-tuple of positive real numbers such that
$r_j<\rho_j$. Then there exists $\epsilon>0$ small enough, so that
$\widetilde{r}_j=r_j(1+\epsilon)<\rho_j$, $j=1,\ldots,d$. Then for
$\widetilde{\mathbf{r}}=(\widetilde{r}_1,\ldots,\widetilde{r}_d)$
we have $\mu(\widetilde{\mathbf{r}})\le 1$. Since the function
$\mu(\cdot)$ is homogeneous of degree $1$, we have
$$\mu(\mathbf{r})=\frac{\mu(\widetilde{\mathbf{r}})}{1+\epsilon}\le\frac{1}{1+\epsilon}.$$
Hence there exists $\ell_0\in\mathbb{N}$ such that
$$\sqrt[\ell]{\sum_{|w|=\ell}\mathbf{r}^w\|f_w\|_{\mathcal{L}_{\rm
cb}^\ell}}<\frac{1+0.5\epsilon}{1+\epsilon}<1$$ for all
$\ell>\ell_0$. This implies that the series
\eqref{eq:semigr-power-series} converges absolutely on the nc
polydisk $(\mathbb{D}^d)_{\rm nc}(Y,\mathbf{r})$, moreover
\begin{multline*}
\sum_{w\in\free_d} \sup_{W\in (\mathbb{D}^d)_{\rm
nc}(Y,\mathbf{r})}
\Big\|\Big(W-\bigoplus_{\alpha=1}^{m_W}Y\Big)^{\odot_sw}f_w\Big\|_{sm_W
}=
\sum_{\ell=0}^\infty\sum_{|w|=\ell}\mathbf{r}^w\|f_w\|_{\mathcal{L}_{\rm
cb}^\ell}\\
<\sum_{\ell=0}^{\ell_0}\sum_{|w|=\ell}\mathbf{r}^w\|f_w\|_{\mathcal{L}_{\rm
cb}^\ell}+\sum_{\ell=\ell_0+1}^\infty\left(\frac{1+0.5\epsilon}{1+\epsilon}\right)^\ell<\infty.
\end{multline*}

Now suppose that $\mathbf{r}=(r_1,\ldots,r_d)$ is a $d$-tuple of
positive extended real numbers with $r_j\ge\rho_j$ for all $j$ and
$r_i>\rho_i$ for some $i$. Then there exists a $d$-tuple
$\widetilde{\mathbf{r}}=(\widetilde{r}_1,\ldots,\widetilde{r}_d)$
of positive real numbers with $\widetilde{r}_j < r_j$ such that
$\mu(\widetilde{\mathbf{r}})>1$. Then there exists $\epsilon>0$
such that $\mu(\widetilde{\mathbf{r}})>1+\epsilon$. Hence
\begin{equation}\label{eq:normal-diverge}
\sum_{|w|=\ell}\widetilde{\mathbf{r}}^w\|f_w\|_{\mathcal{L}_{\rm
cb}^\ell}>(1+0.5\epsilon)^\ell>1+0.5\ell\epsilon
\end{equation}
 for infinitely
many $\ell\in\mathbb{N}$. Let $S\subseteq\mathbb{N}$ be the
(infinite) set of all such $\ell$'s. For any $\ell\in S$ and
$w\in\free_d$ with $|w|=\ell$, using an argument similar to the
one in Proposition \ref{prop:k-cb-norm} we obtain that
\begin{align}\label{eq:w-cb-norm}
\|f_w\|_{\mathcal{L}_{\rm cb}^\ell}
&=\sup\{\|W^{\odot_sw}f_w\|_{sm} \colon  \\
\nonumber & m\in\mathbb{N},\ W\in(\mat{\mathbb{C}}{sm})^d,\
\|W_j\|_{sm}\le 1,\ j=1,\ldots,d\}.
\end{align}
Therefore, there exist $m_w\in\mathbb{N}$ and
$W_w\in(\mathbb{D}^d)_{\rm nc}(0_{s\times
s},\widetilde{\mathbf{r}})_{sm_w}$ such that
$$\|W_w^{\odot_s
w}f_w\|_{sm_w}>\widetilde{\mathbf{r}}^w\|f_w\|_{\mathcal{L}_{\rm
cb}^\ell}-\frac{0.5\ell\epsilon}{d^\ell}.$$ Set
$m(\ell)=\sum_{|w|=\ell}m_w$ and
$W(\ell)=\bigoplus_{|w|=\ell}\Big(W_w+\bigoplus_{\alpha=1}^{m_w}Y\Big)$.
Then $W(\ell)\in(\mathbb{D}^d)_{\rm
nc}(Y,\widetilde{\mathbf{r}})_{sm(\ell)}$ and
$$\sum_{|w|=\ell}\Big\|\Big(W(\ell)-\bigoplus_{\alpha=1}^{m(\ell)}Y\Big)^{\odot_sw}f_w\|_{sm(\ell)}>
\sum_{|w|=\ell}\widetilde{\mathbf{r}}^w\|f_w\|_{\mathcal{L}_{\rm
cb}^\ell}-0.5\ell\epsilon>1.$$ Thus for $\widetilde{\mathbf{r}}$,
and hence for $\mathbf{r}$, \eqref{eq:normal-free-conv-poly}
fails.

Now, let $s=1$. Choosing $\widetilde{\mathbf{r}}\in(\mathbb{D}^d)_{\rm nc}(0_{1\times 1},\mathbf{r})_{1}$  as in
the preceding paragraph, we obtain from \eqref{eq:s=1},  \eqref{eq:s=1_cb_norm}, and \eqref{eq:normal-diverge},
that
$$\sum_{w\in\free_d}\|\widetilde{\mathbf{r}}^wf_w\|_1=
\sum_{w\in\free_d}\widetilde{\mathbf{r}}^w\|f_w\|_{\mathcal{L}_{\rm
cb}^{|w|}}=\infty,$$ i.e., the series
\eqref{eq:semigr-power-series} does not converge absolutely at
$Z=\widetilde{\mathbf{r}}+Y$.

For the proof of the last statement, first suppose that
$\brho=(\rho_1,\ldots,\rho_d)$ is a $d$-tuple of strictly positive
associated radii of normal convergence of the series
\eqref{eq:semigr-power-series}. We have $B_{\rm
nc}(Y,\delta)\subseteq(\mathbb{D}^d)_{\rm
nc}(Y,(\delta,\ldots,\delta))\subseteq(\mathbb{D}^d)_{\rm
nc}(Y,\brho)$ for any positive real $\delta<\min\limits_{1\le j\le
d}\rho_j$, where for the first inclusion we used the fact that the
block projections $\pi^s_{ij}$ are completely contractive. By the
(already proven) first statement of part 4,
\eqref{eq:normal-free-conv-poly} holds with
$\mathbf{r}=(\delta,\ldots,\delta)$. Then
\begin{equation}\label{eq:normal-delta}
\sum_{w\in\free_d} \sup_{W\in B_{\rm nc}(Y,\delta)}
\Big\|\Big(W-\bigoplus_{\alpha=1}^{m_W}Y\Big)^{\odot_sw}f_w\Big\|_{sm_W}<\infty,
\end{equation}
 i.e., $Y\in\Upsilon_{\rm
free,unif-normal}$. Conversely, if $\Upsilon_{\rm
free,unif-normal}$ is nonempty, then by part 2 of the theorem we
have that $Y\in\Upsilon_{\rm free,unif-normal}$. The latter means
that \eqref{eq:normal-delta} holds for some $\delta>0$. Then there
exists a polydisk $(\mathbb{D}^d)_{\rm nc}(Y,\mathbf{r}^0)$
contained in $B_{\rm nc}(Y,\delta)$, say for
$r^0_j=\frac{\delta}{d\|e_j\|_1}$, $j=1,\ldots,d$, so that
\eqref{eq:normal-free-conv-poly} holds with $\mathbf{r}^0$ in the
place of $\mathbf{r}$. Clearly, in this case $\mu(\mathbf{r}^0)\le
1$. Let $$\rho_1=\sup\{r_1\ge r_1^0\colon
\mu(r_1,r_2^0,\ldots,r_d^0)\le 1\}.$$ Then for any positive real
number $r_1\le \rho_1$ we have $\mu(r_1,r_2^0,\ldots,r_d^0)\le 1$
(in the case of $\rho_1<\infty$ this is a consequence of Lemma
\ref{lem:mu}). Let $$\rho_2= \sup\{r_2\ge r_2^0\colon
\mu(r_1,r_2,r_3^0,\ldots,r_d^0)\le 1\ {\rm for\ every\ }
r_1<\rho_1\}.$$ Then for any positive real numbers $r_1\!\le\!
\rho_1$ and $r_2\!\le\!\rho_2$ we have
$\mu(r_1,r_2,r_3^0,\ldots,r_d^0)\!\le\! 1$ (here we use Lemma
\ref{lem:mu} again in the case of $\rho_1<\infty$ and/or
$\rho_2<\infty$). Continuing this construction in an obvious way,
we define $\rho_3,\ldots,\rho_{d-1}$, and finally,
$$\rho_d=\sup\{r_d\ge r_d^0\colon \mu(r_1,\ldots,r_{d-1},r_d)\le
1\ {\rm for\ every\ } r_1<\rho_1,\ldots,r_{d-1}<\rho_{d-1}\}.$$
Then for any $d$-tuple $\mathbf{r}=(r_1,\ldots,r_d)$ of positive
real numbers $r_j\le\rho_j$ we have $\mu(\mathbf{r})\le 1$ (we use
Lemma \ref{lem:mu} again when $\rho_j$ are finite). It is clear
that $\brho=(\rho_1,\ldots,\rho_d)$ is a $d$-tuple of strictly
positive associated radii of normal convergence of the series
\eqref{eq:semigr-power-series}.
\end{proof}
\begin{rem}\label{rem:trich}
It follows from the homogeneity and monotonicity of $\mu(\cdot)$
that we have the following trichotomy:\begin{enumerate}
    \item $\mu(\mathbf{r})=\infty$ for some, and hence for all, $d$-tuples $\mathbf{r}$ of strictly positive real
    numbers; thus there is no strictly positive associated radii
    of normal convergence of the series
\eqref{eq:semigr-power-series}.
    \item $0<\mu(\mathbf{r})<\infty$ for some, and hence for all, $d$-tuples $\mathbf{r}$ of strictly positive real
    numbers; the argument in the last paragraph of the proof of
    Theorem \ref{thm:free-series} shows that there exists a
    $d$-tuple of strictly positive associated radii
    of normal convergence of
\eqref{eq:semigr-power-series} not all of which equal $\infty$.
    \item $\mu(\mathbf{r})=0$ for some, and hence for all, $d$-tuples $\mathbf{r}$ of strictly positive real
    numbers; thus the only $d$-tuple of associated radii of normal
    convergence of \eqref{eq:semigr-power-series} is
    $(\infty,\ldots,\infty)$.
\end{enumerate}
\end{rem}
\begin{rem}\label{rem:ass-rad-vs-rho-cb}
There exists a $d$-tuple of strictly positive associated radii of normal convergence of the series
\eqref{eq:semigr-power-series} if and only if $\rho_{\rm cb}>0$ --- see \eqref{eq:mu-rho-cb}. One direction is
obvious since $\Upsilon_{\rm free, unif-normal}\subseteq\Upsilon_{\rm nc, unif-normal}$. The other direction
follows by applying Theorem \ref{thm:king-semigr} (or Corollary \ref{cor:nc-balls-semigr}) to the nc function $f$
which is uniformly analytic on $\Upsilon_{\rm nc, unif-normal}$. The statement also follows by direct estimates.
We obtain from \eqref{eq:lw-cb-norms-1}, \eqref{eq:lw-cb-norms-2}, \eqref{eq:br-mu}, and \eqref{eq:mu-rho-cb}
that
\begin{align}\label{eq:dir-est-1}
\mu(\mathbf{r}) &\le  (\mathbf{r}\cdot\mathbf{C}_2)\mu_{\rm
cb}=\Big(\sum_{j=1}^dr_j\|e_j\|_1\Big)\mu_{\rm cb},\\
 \label{eq:dir-est-2} \mu_{\rm cb} &\le \mu(\mathbf{C}_1).
\end{align}
These inequalities imply that $\mu_{\rm cb}<\infty$ if and only if
$\mu(\mathbf{r})<\infty$ for some $d$-tuple $\mathbf{r}$ of
strictly positive real numbers. Clearly, $\mu_{\rm cb}<\infty$ is
equivalent to $\rho_{\rm cb}>0$, and the conclusion follows from
Remark \ref{rem:trich}.
\end{rem}

For power series in several complex variables, the classical Abel
lemma states that if the terms of the series at a point
$\mu=(\mu_1,\ldots,\mu_d)$ are bounded,  then the series converges
absolutely and locally normally in the polydisk with multiradius
$(|\mu_1|,\ldots,|\mu_d|)$ --- see \cite[Page 19]{Sh}. We proceed
to establish a noncommutative counterpart of this statement.

For $Y=(Y_1,\ldots,Y_d)\in\mattuple{\mathbb{C}}{s}{d}$ and
$\mathbf{r}=(r_1,\ldots,r_d)$ a $d$-tuple of positive extended
real numbers, define the \emph{nc diamond centered at $Y$ of
multiradius $\mathbf{r}$} \index{nc diamond}
\index{$(\diamondsuit^d)_{\rm nc}(Y,\mathbf{r})$}
$$(\diamondsuit^d)_{\rm nc}(Y,\mathbf{r}):=\coprod_{m=1}^\infty
\Big\{Z=(Z_1,\ldots,Z_d)\in\mattuple{\mathbb{C}}{sm}{d}\colon\sum_{j=1}^d
\frac{\|Z_j-\bigoplus_{\alpha=1}^mY_j\|}{r_j}<1\Big\}.$$ Notice
that $\diamondsuit_{\rm nc}(Y,r)=(\diamondsuit^d)_{\rm
nc}(Y,\mathbf{r})$ with
$\mathbf{r}=\left(\frac{r}{\|e_1\|_1},\ldots,\frac{r}{\|e_d\|_1}\right)$
--- see \eqref{eq:diamond}.

\begin{thm}\label{thm:abel}
1. Suppose there exist $\mu=(\mu_1,\ldots,\mu_d)\in\mathbb{C}^d$
with all $\mu_j\neq 0$ and $C>0$ such that
$|\mu^w|\|f_w\|_{\mathcal{L}_{\rm cb}^{|w|}}\le C$ for every
$w\in\free_d$ (in the case $s=1$ this means that the terms of the
series \eqref{eq:semigr-power-series} at $Y+\mu$ are bounded).
Then the series \eqref{eq:semigr-power-series} converges
absolutely on every nc diamond $(\diamondsuit^d)_{\rm
nc}(Y,\mathbf{r})$ with $r_j<|\mu_j|$ for all $j=1,\ldots,d$.
Moreover,
\begin{equation}\label{eq:normal-free-conv-diam}
\sum_{w\in\free_d} \sup_{W\in (\diamondsuit^d)_{\rm
nc}(Y,\mathbf{r})}
\Big\|\Big(W-\bigoplus_{\alpha=1}^{m_W}Y\Big)^{\odot_sw}f_w\Big\|_{sm_W}<\infty;
\end{equation}
here $m_W$ is the size of a block matrix $W$ with $s \times s$
blocks.

2. Assume there exist $m\in\mathbb{N}$ and
$X=(X_1,\ldots,X_d)\in\mattuple{\mathbb{C}}{sm}{d}$ such that the
terms of the series \eqref{eq:semigr-power-series} are bounded on
some uniformly-open neighborhood of $X$, i.e., there exist
$\delta>0$ and $C>0$ such that
$\|(W-\bigoplus_{\alpha=1}^{m_W}Y)^{\odot_sw}f_w\|_{sm_W}\le C$
for every $W\in B_{\rm nc}(X,\delta)$ and $w\in\free_d$. Then the
set $\Upsilon_{\rm free,unif-normal}$ is nonempty.
\end{thm}
\begin{proof}
1. We may assume that $0<r_j<(1-\epsilon)|\mu_j|$ for some
$\epsilon>0$ and for all $j=1,\ldots, d$. Then
\begin{multline*}
\sum_{w\in\free_d} \sup_{W\in (\diamondsuit^d)_{\rm nc}(0_{s\times
s},\mathbf{r})}\|W^{\odot_sw}f_w\|=\sum_{\ell=0}^\infty\sum_{|w|=\ell}
\sup_{W\in (\diamondsuit^d)_{\rm nc}(0_{s\times s},\mathbf{r})}
\|W^{\odot_sw}f_w\|\\
=\sum_{\ell=0}^\infty\sum_{|w|=\ell} \sup_{W\in
(\diamondsuit^d)_{\rm nc}(0_{s\times s},\mathbf{r})}
\left(\frac{\|W_1\|}{r_1},\ldots,\frac{\|W_d\|}{r_d}\right)^w\mathbf{r}^w\|f_w\|_{\mathcal{L}_{\rm
cb}^\ell}\\
<\sum_{\ell=0}^\infty(1-\epsilon)^\ell\sum_{|w|=\ell} \sup_{W\in
(\diamondsuit^d)_{\rm nc}(0_{s\times s},\mathbf{r})}
\left(\frac{\|W_1\|}{r_1},\ldots,\frac{\|W_d\|}{r_d}\right)^w|\mu^w|\|f_w\|_{\mathcal{L}_{\rm
cb}^\ell}\\
\le C\sum_{\ell=0}^\infty(1-\epsilon)^\ell\sum_{|w|=\ell}
\sup_{W\in (\diamondsuit^d)_{\rm nc}(0_{s\times s},\mathbf{r})}
\left(\frac{\|W_1\|}{r_1},\ldots,\frac{\|W_d\|}{r_d}\right)^w.
\end{multline*}
For any fixed $\ell$ and any $w\in\free_d$ with $|w|=\ell$ there
exists $W_w\in(\diamondsuit^d)_{\rm nc}(0_{s\times s},\mathbf{r})$
such that
$$\sup_{W\in (\diamondsuit^d)_{\rm nc}(0_{s\times s},\mathbf{r})}
\left(\frac{\|W_1\|}{r_1},\ldots,\frac{\|W_d\|}{r_d}\right)^w\le
(1+\epsilon)^\ell\left(\frac{\|(W_{w})_1\|}{r_1},\ldots,\frac{\|(W_w)_d\|}{r_d}\right)^w.$$
Define $W(\ell)=\bigoplus_{|w|=\ell}W_w\in(\diamondsuit^d)_{\rm
nc}(0_{s\times s},\mathbf{r})$. Clearly,
$$\left(\frac{\|(W_{w})_1\|}{r_1},\ldots,\frac{\|(W_w)_d\|}{r_d}\right)^w
\le\left(\frac{\|W(\ell)_1\|}{r_1},\ldots,\frac{\|W(\ell)_d\|}{r_d}\right)^w.$$
Therefore,
\begin{multline*}
\sum_{w\in\free_d} \sup_{W\in (\diamondsuit^d)_{\rm nc}(0_{s\times
s},\mathbf{r})}\|W^{\odot_sw}f_w\|\\
\le C\sum_{\ell=0}^\infty(1-\epsilon^2)^\ell\sum_{|w|=\ell}
\left(\frac{\|W(\ell)_1\|}{r_1},\ldots,\frac{\|W(\ell)_d\|}{r_d}\right)^w\\
=C\sum_{\ell=0}^\infty(1-\epsilon^2)^\ell\Big(\sum_{j=1}^d\frac{\|W(\ell)_j\|}{r_j}\Big)^\ell
<C\sum_{\ell=0}^\infty(1-\epsilon^2)^\ell=\frac{C}{\epsilon^2}<\infty,
\end{multline*}
i.e., \eqref{eq:normal-free-conv-diam} holds. In the case $s=1$ we have by \eqref{eq:s=1} and
\eqref{eq:s=1_cb_norm} that
$$\|f_w(\mu_1,\ldots,\mu_d)\|_1=|\mu^w|\|f_w\|_{\mathcal{L}_{\rm
cb}^{|w|}}\le C,\quad w\in\free_d.$$

2. By Lemma \ref{lem:HiPh_lemma}, the inequality
$$\sup_{W\in B_{\rm
nc}(X,\delta)}\Big\|\Big(W-\bigoplus_{\alpha=1}^{m_W}Y\Big)^{\odot_sw}f_w\Big\|_{sm_W}\le
C,\quad w\in\free_d,$$ holds with the suprema over $ B_{\rm
nc}(X,\delta)$ replaced by the suprema over
$$\bigcup_{\zeta\in\overline{\mathbb{D}}}B_{\rm
nc}\left(\bigoplus_{\alpha=1}^mY+\zeta\Big(X-\bigoplus_{\alpha=1}^mY\Big),\delta\right).$$
 In
particular, taking $\zeta=0$ we obtain that
$$\sup_{W\in B_{\rm nc}(\bigoplus_{\alpha=1}^mY,\delta)}\Big\|\Big(W-\bigoplus_{\alpha=1}^{m_W}Y\Big)^{\odot_sw}
f_w\Big\|_{sm_W}\le C,\quad w\in\free_d.$$ By picking up only
$W\in B_{\rm nc}\Big(\bigoplus_{\alpha=1}^mY,\delta\Big)$ of the
form $W=\bigoplus_{\alpha=1}^mV$, where $V\in B_{\rm
nc}(Y,\delta)$, we obtain that
 $$\sup_{V\in B_{\rm nc}(Y,\delta)}\Big\|\Big(V-\bigoplus_{\alpha=1}^{m_V}\Big)^{\odot_sw}f_w\Big\|_{sm_V}\le C,
 \quad w\in\free_d.$$
 As in
the proof of Theorem \ref{thm:free-series}, we can find a nc
polydisk $(\mathbb{D}^d)_{\rm nc}(Y,\brho)$ which is contained in
$B_{\rm nc}(Y,\delta)$, with $\rho_j=\frac{\delta}{d\|e_j\|_1}$,
$j=1,\ldots,d$. Then
\begin{multline*} \brho^w\|f_w\|_{\mathcal{L}_{\rm
cb}^{|w|}}= \sup_{W\in (\mathbb{D}^d)_{\rm
nc}(Y,\brho)}\Big(\Big\|V_1-\bigoplus_{\alpha=1}^{m_V}Y_1\Big\|_{sm_V},\ldots,
\Big\|V_d-\bigoplus_{\alpha=1}^{m_V}Y_d\Big\|_{sm_V}\Big)^w\|f_w\|_{\mathcal{L}_{\rm
cb}^{|w|}}\\
=\sup_{V\in (\mathbb{D}^d)_{\rm
nc}(Y,\brho)}\Big\|\Big(V-\bigoplus_{\alpha=1}^{m_V}Y\Big)^{\odot_sw}f_w\|_{sm_V}
\le C,\quad w\in\free_d.\end{multline*} By part 1 of the theorem
for
 $\mathbf{\mu}=\brho$ and every  $\mathbf{r}=(r_1,\ldots,r_d)$ with
$0<r_j<\mu_j$, $j=1,\ldots, d$, \eqref{eq:normal-free-conv-diam}
holds. Clearly, the nc diamond $(\diamondsuit^d)_{\rm
nc}(Y,\mathbf{r})$ is open in the uniformly-open topology on
$\ncspaced{\mathbb{C}}{d}$. In particular, $Y$ is its interior
point. Indeed, $B_{\rm nc}(Y,\frac{1}{d}\min_{1\le j\le
d}r_j)\subseteq(\diamondsuit^d)_{\rm nc}(Y,\mathbf{r})$.
Therefore, we obtain that
$$\sum_{w\in\free_d}\sup_{W\in B_{\rm nc}(Y,\frac{1}{d}\min_{1\le j\le d}r_j)}\Big\|
\Big(W-\bigoplus_{\alpha=1}^{m_W}Y\Big)^{\odot_sw}f_w\Big\|_{sm_W}<\infty,$$
i.e., $Y\in\Upsilon_{\rm free,unif-normal}$.
\end{proof}
\begin{rem}\label{rem:diamond-of-conv}
It follows from Theorem \ref{thm:abel}, part 1, that one can
formulate an analogue of Theorem \ref{thm:free-series}, part 4,
for nc diamonds rather than nc polydisks, replacing
\eqref{eq:br-mu} by \index{$\mu_\diamondsuit(\mathbf{r})$}
\begin{equation}\label{eq:br-mu-diamond}
\mu_\diamondsuit(\mathbf{r})=\limsup_{\ell\to\infty}\sqrt[\ell]{\max_{w\in\free_d\colon
|w|=\ell}\mathbf{r}^w\|f_w\|_{\mathcal{L}_{\rm cb}^\ell}}
\end{equation}
and defining the corresponding associated radii of normal
convergence.
\end{rem}

For power series in several complex variables, the union of
polydiscs of (absolute) convergence equals the interior of the set
of locally normal convergence equals the interior of the set of
absolute convergence equals the interior of the set of
convergence. On the other hand, the interior of the set of
convergence of a power series in several complex variables may be
strictly contained in the interior of the set of convergence of
the corresponding series of homogeneous polynomials --- consider,
e.g., the power series $\sum_{n=0}^\infty(z_1+z_2)^n.$

In the noncommutative setting we have, for the series
\eqref{eq:semigr-power-series}, the following sequence of
inclusions:
\begin{equation} \label{eq:free-inclusions}
\bigcup_{\brho}(\mathbb{D}^d)_{\rm nc}(0_{s\times s},\brho)
\subseteq \Upsilon_{\rm free, unif-normal} \subseteq \Int
\Upsilon_{\rm free, abs} \subseteq \Int \Upsilon_{{\rm
free},\prec} \subseteq \Upsilon_{\rm nc, unif}.
\end{equation}
 Here the union  is taken over all $d$-tuples
$\brho=(\rho_1,\ldots,\rho_d)$ of strictly positive associated
radii of normal convergence, $\Int$ denotes the interior in the
uniformly-open topology, $\Upsilon_{{\rm free},\prec}$ is the
\index{$\Upsilon_{{\rm free},\prec}$} convergence set of the
series \eqref{eq:semigr-power-series} with respect to some order
$\prec$ on the free monoid $\free_d$ which is compatible with
the word length (i.e., $w \prec w' \Longrightarrow |w| < |w'|$),
and $\Upsilon_{\rm nc, unif} = \Int \Upsilon_{\rm nc}$ (see
Theorem \ref{thm:king-series}).

We will calculate various sets appearing in
\eqref{eq:free-inclusions} for several examples. We need the
following lemma (where $\specrad(Z)$ \index{$\specrad(Z)$} denotes
the spectral radius of a matrix $Z$).
\begin{lem}\label{lem:specrad}
$$\lim\limits_{\delta\to
0}\limsup\limits_{\ell\to\infty}\sqrt[\ell]{\sup_{W\in B_{\rm
nc}(Z,\delta)}\|W^\ell\|} =\lim\limits_{\delta\to
0}\liminf\limits_{\ell\to\infty}\sqrt[\ell]{\sup_{W\in B_{\rm
nc}(Z,\delta)}\|W^\ell\|}= \specrad(Z).$$
\end{lem}
\begin{proof}
 Suppose $Z\in\mat{\mathbb{C}}{s}$ and let $\epsilon>0$. Using the lower semicontinuity of the spectral radius
 for elements
of $\mathcal{L}(\ell^2)$  \cite[page 167]{HiPh}, we obtain that
there exists $\delta>0$ such that $$r_{\rm spec}(X)< r_{\rm
spec}\Big(\bigoplus_{\alpha=1}^\infty Z\Big)+\epsilon=r_{\rm
spec}(Z)+\epsilon$$ for every $X\in \mathcal{L}(\ell^2)$ with
$\|X-\bigoplus_{\alpha=1}^\infty Z\|<\delta$. It follows that
$r_{\rm spec}(W)< r_{\rm spec}(Z)+\epsilon$ for every $W\in B_{\rm
nc}(Z,\delta)$ (by taking $X=\bigoplus_{\beta=1}^\infty W$). Set
$$M_\ell:=\sup_{W\in B_{\rm nc}(Z,\delta)}\|W^\ell\|,\quad
\ell\in\mathbb{N}.$$ Then for any fixed $\ell$ we have
$M_\ell\le(\|Z\|+\delta)^\ell<\infty$. Next, for every fixed
$\ell$ one can find $W_\ell\in B_{\rm nc}(Z,\delta)$ such that
$\|W_\ell^\ell\|\ge \frac{M_\ell}{2}$. Set
$X:=\bigoplus_{\ell=1}^\infty W_\ell\in\mathcal{L}(\ell^2)$. Since
$\specrad(X)=\sup_{\ell}r_{\rm spec}(W_\ell)$ and
$\specrad(W_\ell)<\specrad(Z)+\epsilon$, we have
\begin{multline*}
\specrad(Z)+\epsilon\ge\specrad(X)=\lim_{\ell\to\infty}\sqrt[\ell]{\|X^\ell\|}\ge
\limsup_{\ell\to\infty}\sqrt[\ell]{\|W_\ell^\ell\|}
\ge\limsup_{\ell\to\infty}\sqrt[\ell]{\frac{M_\ell}{2}}\\
=\limsup_{\ell\to\infty}\sqrt[\ell]{M_\ell}
=\limsup_{\ell\to\infty}\sqrt[\ell]{\sup_{W\in B_{\rm
nc}(Z,\delta)}\|W^\ell\|}\ge\lim\limits_{\delta\to
0}\liminf\limits_{\ell\to\infty}\sqrt[\ell]{\sup_{W\in B_{\rm
nc}(Z,\delta)}\|W^\ell\|}\\ \ge
\lim_{\ell\to\infty}\sqrt[\ell]{\|Z^\ell\|}=\specrad(Z).
\end{multline*}
We conclude that
$$\lim_{\delta\to 0}\limsup_{\ell\to\infty}\sqrt[\ell]{\sup_{W\in B_{\rm
nc}(Z,\delta)}\|W^\ell\|}=\lim\limits_{\delta\to
0}\liminf\limits_{\ell\to\infty}\sqrt[\ell]{\sup_{W\in B_{\rm
nc}(Z,\delta)}\|W^\ell\|}=\specrad(Z)$$ as required.
\end{proof}

The following two examples show that the first inclusion in
\eqref{eq:free-inclusions} can be proper.
\begin{ex} \label{ex:free_scalar}
Let $d=1$ and $s=1$, i.e., \eqref{eq:power-series} and
\eqref{eq:semigr-power-series} both collapse to
\begin{equation} \label{eq:1d-1s-power-series}
\sum_{\ell=0}^\infty Z^{\ell}  f_\ell,
\end{equation}
where $Z \in \mat{\mathbb{C}}{m}$, $m=1,2,\ldots$, for a given
sequence of vectors $f_\ell \in \vecspace{W}$. Let $\brho =
\left(\limsup_{\ell \to \infty}
\sqrt[\ell]{\|f_\ell\|}\right)^{-1}$ be the (unique associated)
radius of normal convergence; of course, $\brho = \rho_{\rm cb} =
\rho_m$ for all $m$. Then clearly
\begin{equation*}
\rho(Z) = \frac{\brho}{\specrad(Z)},
\end{equation*}
with the convention that the fraction equals $\infty$ when the
denominator is zero, irrespectively of the numerator. It follows
from Lemma \ref{lem:specrad} that if $\brho$ and $\specrad(Z)$ are
not both zero, then we have also
$$\rho_{\rm unif-normal}(Z)=\frac{\brho}{\specrad(Z)}.$$
Indeed,
\begin{multline*}
\rho(Z)\ge\rho_{\rm
unif-normal}(Z)=\frac{1}{\lim\limits_{\delta\to
0}\limsup\limits_{\ell\to\infty}\sqrt[\ell]{\sup\limits_{W\in
B_{\rm
nc}(Z,\delta)}\|W^\ell f_\ell\|}}\\
=\frac{1}{\lim\limits_{\delta\to
0}\limsup\limits_{\ell\to\infty}\sqrt[\ell]{\sup\limits_{W\in
B_{\rm nc}(Z,\delta)}\|W^\ell\| \|f_\ell\|}}\\
\ge\frac{1}{\lim\limits_{\delta\to
0}\limsup\limits_{\ell\to\infty}\sqrt[\ell]{\sup\limits_{W\in
B_{\rm
nc}(Z,\delta)}\|W^\ell\|}\limsup\limits_{\ell\to\infty}\sqrt[\ell]{
\|f_\ell\|}}=\frac{\brho}{\specrad(Z)}=\rho(Z),
\end{multline*}
which implies the required equality. It follows that
\begin{equation*}
\Upsilon_{\rm nc, unif-normal} = \Upsilon_{\rm nc, unif} =
\Upsilon_{\rm nc, norm} = \left\{Z \in\ncspace{\mathbb{C}} \colon
\specrad(Z) < \brho\right\}.
\end{equation*}
Of course, $\Upsilon_{\rm free, unif-normal}=\Upsilon_{\rm nc,
unif-normal}$; except for the cases $\brho=0$ and $\brho=\infty$
this is a much bigger set than
\begin{equation*}
(\mathbb{D}^1)_{\rm nc}(0_{1\times 1},\brho) = \left\{Z \in
\ncspace{\mathbb{C}} \colon \|Z\| < \brho\right\}.
\end{equation*}
\end{ex}

\label{QUESTIONS*}

\begin{ex}\label{ex:incl-1-proper}
Consider the series
$$\sum_{n=0}^\infty(Z_1Z_2)^n.$$
We have
$$\mu(\mathbf{r})=\limsup_{n\to\infty}\sqrt[2n]{r_1^nr_2^n}=\sqrt{r_1r_2}.$$
It follows that a pair $\brho=(\rho_1,\rho_2)$ of strictly
positive extended real numbers is a pair of associated radii of
normal convergence of the series if and only if $\rho_1\rho_2=1$,
and
$$\bigcup_{\brho} (\mathbb{D}^2)_{\rm nc}(0_{1\times
1},\brho)=\left\{(Z_1,Z_2)\in\ncspaced{\mathbb{C}}{2}\colon
\|Z_1Z_2\|<1\right\}.$$ It is clear that the convergence set of
this series is
$$\Upsilon_{\rm nc}=\{(Z_1,Z_2)\in\ncspaced{\mathbb{C}}{2}\colon
r_{\rm spec}(Z_1Z_2)<1\},$$ and $\rho(Z)=(\specrad(Z_1Z_2))^{-1}$.
It follows from Lemma \ref{lem:specrad} that
\begin{multline*}\rho(Z)\ge\rho_{\rm unif-normal}(Z)=\Big(\lim\limits_{\delta\to
0}\limsup\limits_{\ell\to\infty}\sqrt[\ell]{\sup\limits_{(W_1,W_2)\in
B_{\rm nc}(Z,\delta)}\|(W_1W_2)^\ell \|}\Big)^{-1}\\
\ge\Big(\lim\limits_{\delta\to
0}\limsup\limits_{\ell\to\infty}\sqrt[\ell]{\sup\limits_{(W)\in
B_{\rm nc}(Z_1Z_2,\delta)}\|W^\ell
\|}\Big)^{-1}=(\specrad(Z_1Z_2))^{-1}=\rho(Z),
\end{multline*}
and therefore
\begin{equation*}
\Upsilon_{\rm nc, unif-normal}=\Upsilon_{\rm nc, unif} =
\Upsilon_{\rm nc, norm} =
\{(Z_1,Z_2)\in\ncspaced{\mathbb{C}}{2}\colon r_{\rm
spec}(Z_1Z_2)<1\}.
\end{equation*}
 Since each homogeneous
polynomial  in the series has only one term, the convergence of
this series along the free monoid is the same as its
convergence as a series of homogeneous polynomials, so that
$\Upsilon_{\rm free,unif-normal}=\Upsilon_{\rm nc, unif-normal}$.
Similarly to the previous example, the second set in
\eqref{eq:free-inclusions} is much bigger than the first one.
\end{ex}

The following two examples show that the last inclusion in
\eqref{eq:free-inclusions} may be proper. The first example also
shows that, unlike in the classical Abel lemma, the convergence of
the series \eqref{eq:semigr-power-series} at a point
$\mu=(\mu_1,\ldots,\mu_d)\in\mathbb{C}^d$ does not imply its
convergence on the nc polydisk of multiradius
$(|\mu_1|,\ldots,|\mu_d|)$. The second example also shows that the
TT series along $\free_d$ of a uniformly analytic nc function does
not necessarily converge on a nc polydisk when the nc function is
bounded there --- cf. Theorem \ref{thm:king-semigr} and Corollary
\ref{cor:nc-balls-semigr}.
\begin{ex}\label{ex:free_ncgeomseries}
Let $d>1$ and let us consider ``nc geometric series''
\begin{equation} \label{eq:ncgeomseries}
\sum_{w \in \free_d} Z^w,
\end{equation}
where $Z \in \mat{\mathbb{C}}{m}$, $m=1,2,\ldots$ (i.e., $s=1$,
$\vecspace{W}=\mathbb{C}$, and $f_w=1$ for all $w \in \free_d$).
Notice that $$\sum_{w \in \free_d,\, |w|=\ell} Z^w = (Z_1 + \cdots
+ Z_d)^\ell,$$ therefore $$\Upsilon_{\rm nc}=\left\{(Z_1,\ldots,
Z_d)\in\ncspaced{\mathbb{C}}{d}\colon \specrad(Z_1+\cdots
+Z_d)<1\right\}$$ and $\rho(Z)=(\specrad(Z_1+\cdots+Z_d))^{-1}$.
Similarly to Example \ref{ex:incl-1-proper}, it follows from Lemma
\ref{lem:specrad} that \begin{multline}\label{eq:set-above}
\Upsilon_{\rm
nc,unif-normal}=\Upsilon_{\rm nc,unif}=\Upsilon_{\rm nc,norm}\\
=\left\{(Z_1,\ldots, Z_d)\in\ncspaced{\mathbb{C}}{d}\colon
\specrad(Z_1+\cdots +Z_d)<1\right\} \end{multline}
 as well. Notice that the set $\Upsilon_{\rm free,prec}$ is
 properly contained in \eqref{eq:set-above} for any order $\prec$ on
 $\free_d$ which is compatible with the word length ---
 because $\specrad(Z^w)<1$ for every $w\in\free_d$ is a necessary
 condition
 for the convergence of the series \eqref{eq:ncgeomseries}, and there are plenty of $d$-tuples of matrices in
 \eqref{eq:set-above}
  which violate this condition.

  Next, for a $d$-tuple
 $\mathbf{r}=(r_1,\ldots,r_d)$ of nonnegative real numbers,
 $\mu(\mathbf{r})=r_1+\cdots +r_d$. Therefore a $d$-tuple $\brho=(\rho_1,\ldots,\rho_d)$
 of strictly positive extended real numbers is a $d$-tuple of
 associated radii of normal convergence if and only if
 $\rho_1+\cdots +\rho_d=1$, so that
\begin{multline}\label{eq:union-of-polydisks}
\bigcup_{\brho}(\mathbb{D}^d)_{\rm nc}(0_{1\times
 1},\brho)=\left\{Z\in\ncspaced{\mathbb{C}}{d}\colon
 \|Z_1\|+\cdots +\|Z_d\|<1\right\}\\
 =(\diamondsuit^d)_{\rm nc}(0_{1\times
 1},(1,\ldots,1)).
 \end{multline}
On the other hand, for a $d$-tuple
 $\mathbf{r}=(r_1,\ldots,r_d)$ of nonnegative real numbers,
 $\mu_\diamondsuit(\mathbf{r})=\max_{j}r_j$. Therefore a $d$-tuple $\brho=(\rho_1,\ldots,\rho_d)$
 of strictly positive extended real numbers is a $d$-tuple of
 associated radii of normal convergence in the sense of diamonds --- see Remark \ref{rem:diamond-of-conv} ---
 if and only if
 $\rho_1=\cdots =\rho_d=1$, so that
the set \eqref{eq:union-of-polydisks} coincides with
$\bigcup_{\brho}(\diamondsuit^d)_{\rm nc}(0_{1\times
 1},\brho)$ where the union is taken over all $d$-tuples
 $\brho=(\rho_1,\ldots,\rho_d)$ of  of strictly positive associated
 radii of normal convergence in the sense of diamonds.

\label{COINCIDENCE}

Clearly, $Z\in\Upsilon_{\rm free,\prec}$ if and only if
$Z\in\Upsilon_{\rm nc}$, i.e., $\specrad(Z_1+\cdots +Z_d)<1$, and
$$\lim_{w_1\preceq w_2, |w_1|=|w_2|\to\infty}\sum_{w_1\preceq w\preceq
w_2}Z^w=0.$$
 Let $d=2$ and let $\zeta \in \mathbb{C}$, $\dfrac{1}{2} < |\zeta| <
1$. It folows that $(\zeta,-\zeta) \in \Upsilon_{\rm free,\prec}$
if and only if
\begin{equation}\label{eq:pm-conv}
\lim_{w_1\preceq w_2,
|w_1|=|w_2|\to\infty}\zeta^{|w_1|}\sum_{w_1\preceq w\preceq
w_2}(-1)^{|w|_2}=0,
\end{equation}
 where $|w|_2$ denotes the number of times the
letter $g_2$ appears in the word $w$. Notice that the number of
words of length $\ell$ with $|w|_2$ even is the same as the number
of words of length $\ell$ with $|w|_2$ odd and equals
$2^{\ell-1}$. Thus if we choose the order $\prec$ so that all the
words of a given length with $|w|_2$ even preceed all the words of
the same length with $|w|_2$ odd, then \eqref{eq:pm-conv} fails.
On the other hand, if we choose the order $\prec$ so that words
with $|w|_2$ even and words with $|w|_2$ odd interlace, then
\eqref{eq:pm-conv} holds. Therefore, $(\zeta,-\zeta)\notin
\Upsilon_{\rm free,\prec}$ or $(\zeta,-\zeta)\in \Upsilon_{\rm
free,\prec}$, depending on the choice of $\prec$. Furthemore, in
the case where $(\zeta,-\zeta)\in \Upsilon_{\rm free,\prec}$ the
series does not converge on $(\mathbb{D}^2)_{\rm nc}(0_{1\times
1},(|\zeta|,|\zeta|))$, e.g., it does not converge at the point
$(\epsilon\zeta,\epsilon\zeta)$ for
$\frac{1}{2|\zeta|}<\epsilon<1$.
\end{ex}

\label{EXAMPLE-QUESTIONS}

\begin{ex}\label{ex:resolvent}
Consider
$$f(Z)=(I-\mathcal{Z}(Z)\otimes
A)^{-1}=(I-Z_1(P_1A)-Z_2(P_2A))^{-1},$$ where
$$A=\begin{bmatrix}
a_{11} & a_{12}\\
a_{21} & a_{22}
\end{bmatrix}, \quad P_1=\begin{bmatrix}
1 & 0\\
0 & 0
\end{bmatrix},\quad  P_2=\begin{bmatrix}
0 & 0\\
0 & 1
\end{bmatrix}
$$
are $2\times 2$ matrices over $\mathbb{C}$, and
$\mathcal{Z}(Z)=\begin{bmatrix}
Z_1 & 0\\
0 & Z_2
\end{bmatrix}=Z_1\otimes P_1+ Z_2\otimes P_2$ is a $2n\times 2n$
matrix over $\mathbb{C}$ for every $n\in\mathbb{N}$ and
$Z\in\mat{\mathbb{C}}{n}$. It is easy to see that $f$ is a
uniformly analytic nc function on the uniformly-open nc set
$$\left\{
Z\in\ncspaced{\mathbb{C}}{2}\colon\det(I-\mathcal{Z}(Z)\otimes
A)\neq 0\right\}.$$

From now on, we assume that $\|A\|<1$, so that $f$ is bounded
 by
$(1-\|A\|)^{-1}$ on $(\mathbb{D}^2)_{\rm nc}(0_{1\times
1},(1,\ldots,1))$. The TT series of $f$ at $0_{1\times 1}$ is
$$\sum_{\ell=0}^\infty (\mathcal{Z}(Z)\otimes A)^\ell.$$
For this series, it is clear that
$$\Upsilon_{\rm nc}=\left\{Z\in\ncspaced{\mathbb{C}}{2}\colon
\specrad(\mathcal{Z}(Z)\otimes A)<1\right\}$$ and
$\rho(Z)=(\specrad(\mathcal{Z}(Z)\otimes A))^{-1}$. Using Lemma
\ref{lem:specrad}, we obtain that
$$
\Upsilon_{\rm nc, unif-normal}=\Upsilon_{\rm
nc,unif}=\Upsilon_{\rm nc, norm}
=\left\{Z\in\ncspaced{\mathbb{C}}{2}\colon
\specrad(\mathcal{Z}(Z)\otimes A)<1\right\}
$$
as well.

We consider now the corresponding series along $\free_2$. The
$(i,j)$-th entry of the sum is given by
$$\delta_{ij}+\sum_{\ell=0}^\infty\sum_{1\le
i_1,\ldots,i_{\ell-1}\le 2}a_{ii_1}a_{i_1i_2}\cdots
a_{i_{\ell-1}j}Z_iZ_{i_1}\cdots Z_{i_{\ell-1}},
$$
for all $i,j=1,2$. For $(z_1,z_2)\in\mathbb{C}^2$, the series
converges absolutely if and only if the series
\begin{equation}\label{eq:abs-series}
\sum_{\ell=0}^\infty\left(\begin{bmatrix}
|z_1| & 0\\
0 & |z_2|
\end{bmatrix}\begin{bmatrix}
|a_{11}| & |a_{12}|\\
|a_{21}| & |a_{22}|
\end{bmatrix}\right)^\ell
\end{equation}
converges. We choose $A$ with $\|A\|<1$ and
$\specrad\begin{bmatrix}
|a_{11}| & |a_{12}|\\
|a_{21}| & |a_{22}|
\end{bmatrix}>1$, e.g.,
$$A=\frac{1-\epsilon}{\sqrt{2}}\begin{bmatrix}
1 & -1\\
1 & 1
\end{bmatrix},$$
with $\epsilon>0$ small enough. Then for $|z_1|,|z_2|<1$
sufficiently close to $1$, the series \eqref{eq:abs-series}
diverges.
\end{ex}

\chapter{Direct summands extensions of nc sets and nc functions}\label{sec:dirsum-ext}
In this chapter, we consider a natural extension of a similarity
invariant nc set $\Omega$ to a larger nc set, $\Omega_{\rm
d.s.e.}$ \index{$\Omega_{\rm d.s.e.}$} which contains all direct
summands of matrices from $\Omega$, and the corresponding
extension of nc functions on $\Omega$ to nc functions on
$\Omega_{\rm d.s.e.}$. Clearly, $\Omega_{\rm d.s.e.}$ includes not
only $\Omega$, but also the set $\rad\Omega$ which was introduced
in Section \ref{subsec:unif-open-top}---see the paragraph
preceding Proposition \ref{prop:rad}.

Let $\module{M}$ be a module over a commutative unital ring
$\ring$, and let $\Omega\subseteq\ncspace{\module{M}}$ be a
similarity invariant nc set. We define the \emph{direct summands
extension} \index{direct summands extension of a nc set} of
$\Omega$ by
\begin{equation}\label{eq:dirsum-ext}
\Omega_{\rm d.s.e.}:=\{X\in\ncspace{\module{M}}\colon X\oplus
Y\in\Omega\ {\rm for\ some}\ Y\in\ncspace{\module{M}}\}.
\end{equation}
Since $\Omega$ is invariant under similarities, in particular,
when a similarity matrix is a permutation of rows (columns),
$X\in\Omega_{\rm d.s.e.}$ if and only if $Y^1\oplus X\oplus
Y^2\in\Omega$ for some $Y^1$, $Y^2\in\ncspace{\module{M}}$.
\begin{prop}\label{prop:dirsum-ext-properties}
Let $\Omega\subseteq\ncspace{\module{M}}$ be a similarity
invariant nc set. Then
\begin{enumerate}
    \item $\Omega_{\rm d.s.e.}$ is a radical similarity invariant nc set;
    \item If $\Omega$ is right (resp., left) admissible, then so
    is $\Omega_{\rm d.s.e.}$;
    \item If $\Omega$ is finitely open  (resp., open,
    uniformly-open), then so is $\Omega_{\rm d.s.e.}$;
    moreover, if $\Omega$ is uniformly-open, then $\Omega_{\rm
    d.s.e.}$ is $\tau$-open (see Section
    \ref{subsec:unif-open-top} for the definition of the pseudometric
    $\tau$).
\end{enumerate}
\end{prop}
\begin{proof}
(1) If $X^1$, $X^2\in\Omega_{\rm d.s.e.}$, then $X^1\oplus
Y^1\in\Omega$ and $X^2\oplus Y^2\in\Omega$ for some $Y^1,Y^2
\in\ncspace{\module{M}}$. Since $\Omega$ is a nc set, $X^1\oplus
Y^1\oplus X^2\oplus Y^2\in\Omega$. Since $\Omega$ is similarity
invariant, and the matrix $X^1\oplus Y^1\oplus X^2\oplus Y^2$ is
similar, with the similarity matrix being a permutation, to
$X^1\oplus X^2\oplus Y^1\oplus Y^2$, the latter belongs to
$\Omega$. Therefore, $X^1\oplus X^2\in\Omega_{\rm d.s.e.}$, and we
conclude that $\Omega_{\rm d.s.e.}$ is a nc set.

 Let $X\in(\Omega_{\rm d.s.e.})_n$ and let $S\in\mat{\ring}{n}$
be invertible. Then $X\oplus Y\in\Omega_{n+m}$ for some
$m\in\mathbb{N}$ and $Y\in\mat{\module{M}}{m}$, and since $\Omega$
is similarity invariant,
$$SXS^{-1}\oplus Y=(S\oplus I_m)(X\oplus Y)(S\oplus
I_m)^{-1}\in\Omega_{n+m}.$$ Therefore, $SXS^{-1}\in(\Omega_{\rm
d.s.e.})_n$, and we conclude that $\Omega_{\rm d.s.e.}$ is
similarity invariant.

Suppose that $Y\in\rad\,\Omega_{\rm d.s.e.}$. Then
$X=\bigoplus_{\alpha=1}^mY\in\Omega_{\rm d.s.e.}$ for some
$m\in\mathbb{N}$. Therefore, $X\oplus Z\in\Omega$ for some
$Z\in\ncspace{\module{M}}$, i.e.,
$$Y\oplus\Big(\bigoplus_{\alpha=1}^{m-1}Y\Big)\oplus Z\in\Omega.$$
Hence, $Y\in\Omega_{\rm d.s.e.}$. This proves the inclusion
$\rad\,\Omega_{\rm d.s.e.}\subseteq\Omega_{\rm d.s.e.}$. Together
with the obvious inclusion $\rad\,\Omega_{\rm
d.s.e.}\supseteq\Omega_{\rm d.s.e.}$, this implies that
$\rad\,\Omega_{\rm d.s.e.}=\Omega_{\rm d.s.e.}$, i.e., the set
$\Omega_{\rm d.s.e.}$ is radical.

(2) Suppose that $\Omega$ is right admissible. Let
$X^1\in(\Omega_{\rm d.s.e.})_{n_1}$, $X^2\in(\Omega_{\rm
d.s.e.})_{n_2}$, and $Z\in\rmat{\module{M}}{n_1}{n_2}$. Then there
exist $m_1,m_2\in\mathbb{N}$ and $Y^1\in\mat{\module{M}}{m_1}$,
$Y^2\in\mat{\module{M}}{m_2}$ such that $X^1\oplus
Y^1\in\Omega_{n_1+m_1}$ and $X^2\oplus Y^2\in\Omega_{n_2+m_2}$.
Since $\Omega$ is a similarity invariant right admissible nc set,
by Proposition \ref{prop:adm-env},
$$\begin{bmatrix}
X^1 & 0 & Z & 0\\
0 & Y^1 & 0 & 0\\
0 & 0 & X^2 & 0\\
0 & 0 & 0 & Y^2
\end{bmatrix}\in\Omega_{n_1+m_1+n_2+m_2}.$$
The latter matrix is similar, with the similarity matrix being a
permutation, to
$$\begin{bmatrix}
X^1 & Z & 0 & 0\\
0 & X^2 & 0 & 0\\
0 & 0 & Y^1 & 0\\
0 & 0 & 0 & Y^2
\end{bmatrix},$$
which therefore also belongs to
$\Omega_{n_1+m_1+n_2+m_2}=\Omega_{n_1+n_2+m_1+m_2}$. Hence,
$$\begin{bmatrix} X^1 & Z\\
0 & X^2
\end{bmatrix}\in(\Omega_{\rm d.s.e.})_{n_1+n_2},$$ and we conclude
that $\Omega_{\rm d.s.e.}$ is right admissible.

The statement for a left admissible $\Omega$ is proved
analogously.

(3) Let $\module{M}=\vecspace{V}$ be a vector space over
$\mathbb{C}$, and let the nc set
$\Omega\subseteq\ncspace{\vecspace{V}}$ be finitely open. Let
$X\in(\Omega_{\rm d.s.e.})_n$. Then there exist $m\in\mathbb{N}$
and $Y\in\mat{\vecspace{V}}{m}$ such that $X\oplus
Y\in\Omega_{n+m}$. Let $\vecspace{U}$ be a finite-dimensional
subspace of $\mat{\vecspace{V}}{n}$ which contains $X$. Then the
space $\vecspace{U}\oplus\spa\{Y\}:=\{U\oplus\alpha Y\colon
U\in\vecspace{U},\alpha\in\mathbb{C}\}$ is finite-dimensional and
contains $Y$. Since $\Omega_{n+m}$ is finitely open, there exists
an open neighborhood $\Gamma$ of $X\oplus Y$ in
$\vecspace{U}\oplus\spa\{Y\}$ which is contained in
$\Omega_{n+m}$. Moreover, one can choose such a neighborhood of
the form $\Gamma=\Phi\oplus\Psi$ where $\Phi$ is an open
neighborhood of $X$ in $\vecspace{U}$, and $\Psi$ is an open
neighborhood of $Y$ in $\spa\{Y\}$. Since for every $U\in\Phi$ and
$W\in\Psi$ one has $U\oplus W\in\Gamma\subseteq\Omega_{n+m}$, we
have that $\Phi\subseteq(\Omega_{\rm d.s.e.})_n$. We conclude that
$\Omega_{\rm d.s.e.}$ is finitely open.

Let $\module{M}=\vecspace{V}$ be a Banach space equipped with an
admissible system of matrix norms over $\vecspace{V}$, and let the
nc set $\Omega\subseteq\ncspace{\vecspace{V}}$ be open. Let
$X\in(\Omega_{\rm d.s.e.})_n$. Then there exist $m\in\mathbb{N}$
and $Y\in\mat{\vecspace{V}}{m}$ such that $X\oplus
Y\in\Omega_{n+m}$. Let $\epsilon>0$ be such that $B(X\oplus
Y,\epsilon)\subseteq\Omega_{n+m}$. Set
$\delta:=C_1'(n,m)^{-1}\epsilon$ (see \eqref{eq:dirsums-norms}
where $C_1'(n,m)$ is introduced). Then for every $U\in
B(X,\delta)$ one has $U\oplus Y\in B(X\oplus
Y,\epsilon)\subseteq\Omega_{n+m}$, hence $U\in(\Omega_{\rm
d.s.e.})_n$. We conclude that $\Omega_{\rm d.s.e.}$ is open.

Let $\module{M}=\vecspace{V}$ be an operator space, and let the nc
set $\Omega\subseteq\ncspace{\vecspace{V}}$ be uniformly-open. Let
$X\in(\Omega_{\rm d.s.e.})_n$. Then there exist $m\in\mathbb{N}$
and $Y\in\mat{\vecspace{V}}{m}$ such that $X\oplus
Y\in\Omega_{n+m}$. Let $\epsilon>0$ be such that $B_{\rm
nc}(X\oplus Y,\epsilon)\subseteq\Omega$. Let $Z\in B_{\rm
nc}(X,\epsilon)_{nk}$, and let $U\in\mat{\ring}{(n+m)k}$ be a
permutation matrix such that \begin{equation}\label{eq:u}
U\Big(\bigoplus_{\alpha=1}^k(X\oplus
Y)\Big)U^{-1}=\Big(\bigoplus_{\alpha=1}^kX\Big)\oplus\Big(\bigoplus_{\beta=1}^kY\Big).
\end{equation}
Since $\Omega$ is a similarity invariant nc set, this matrix
belongs to $\Omega_{(n+m)k}$. Since $U$ is unitary,
\begin{equation}\label{eq:z}
U^{-1}\Big(Z\oplus\bigoplus_{\beta=1}^kY\Big)U\in B_{\rm
nc}(X\oplus Y,\epsilon)\subseteq\Omega_{(n+m)k}, \end{equation}
and since $\Omega$ is similarity invariant,
$Z\oplus\bigoplus_{\beta=1}^kY\in \Omega_{(n+m)k}$. Therefore,
$Z\in(\Omega_{\rm d.s.e.})_{nk}$. This proves that $B_{\rm
nc}(X,\epsilon)\subseteq\Omega_{\rm d.s.e.}$, and we conclude that
$\Omega_{\rm d.s.e.}$ is uniformly-open.
\end{proof}

We define next a direct summands extension $f_{\rm d.s.e.}$
\index{$f_{\rm d.s.e.}$} of a nc function $f$.

\begin{prop}\label{prop:ncfun-dirsum-ext}
Let $\module{M}$, $\module{N}$ be modules over a unital
commutative ring $\ring$, let
$\Omega\subseteq\ncspace{\module{M}}$ be a similarity invariant nc
set, and let $f\colon\Omega\to\ncspace{\module{N}}$ be a nc
function. Then
\begin{enumerate}
    \item There exists a uniquely determined \emph{direct summands extension} \index{direct summands extension of a
    nc function} of $f$, i.e., a nc function $f_{\rm
    d.s.e.}\colon\Omega_{\rm d.s.e.}\to\ncspace{\module{N}}$ such
    that $f_{\rm d.s.e.}|_\Omega =f$;
    \item If $\ring=\field$ is an infinite field, $\Omega=\coprod_{m=1}^\infty \mat{\module{M}}{sm}$
     and $f|_{\mat{\module{M}}{sm}}$ is polynomial on slices for every $m\in\mathbb{N}$,
    then $\Omega_{\rm d.s.e.}=\ncspace{\module{M}}$ and  $f_{\rm d.s.e.}$ is polynomial on
    slices; if, furthermore, the degrees of
    $f|_{\mat{\module{M}}{sm}}$, $m=1,2,\ldots$, are bounded, then so are the
    degrees of $f_{\rm d.s.e.}|_{\mat{\module{M}}{n}}$,
    $n=1,2,\ldots$, so that $f_{\rm d.s.e.}$ is a nc polynomial over $\module{M}$;
    \item If $\module{M}=\vecspace{V}$ is a vector space over $\mathbb{C}$,
    $\module{N}=\vecspace{W}$ is a Banach space equipped with an admissible system
    of matrix norms over $\vecspace{W}$, $\Omega\subseteq\ncspace{\vecspace{V}}$ is
 a finitely open nc set, and $f$ is
    G-differentiable on $\Omega$, then $f_{\rm d.s.e.}$ is G-differentiable on $\Omega_{\rm
    d.s.e.}$; moreover, if $X\in(\Omega_{\rm d.s.e.})_n$, $Z\in\mat{\vecspace{V}}{n}$, and
    $X\oplus Y\in\Omega_{n+m}$, then
    $$\delta f(X\oplus Y)(Z\oplus 0_{m\times m})=\delta f_{\rm
    d.s.e.}(X)(Z)\oplus f_{\rm d.s.e.}(Y);$$
    \item If $\module{M}=\vecspace{V}$ and $\module{N}=\vecspace{W}$ are Banach spaces equipped with
    admissible systems of matrix norms over $\vecspace{V}$ and over $\vecspace{W}$,
 $\Omega\subseteq\ncspace{\vecspace{V}}$ is an open nc set, and $f$ is
    analytic on $\Omega$, then  $f_{\rm d.s.e.}$ is analytic on $\Omega_{\rm d.s.e.}$;
     \item If $\module{M}=\vecspace{V}$ and
    $\module{N}=\vecspace{W}$ are operator spaces, $\Omega\subseteq\ncspace{\vecspace{V}}$
     is a uniformly-open nc set, and $f$ is
    uniformly analytic on $\Omega$, then  $f_{\rm d.s.e.}$ is uniformly analytic on $\Omega_{\rm d.s.e.}$.
\end{enumerate}
\end{prop}
\begin{proof}
(1) Let $X\in(\Omega_{\rm d.s.e.})_n$. Then there exist
$m\in\mathbb{N}$ and $Y\in\mat{\module{M}}{m}$ such that $X\oplus
Y\in\Omega_{n+m}$. The matrix $\mu I_n\oplus\nu
I_m\in\mat{\ring}{(n+m)}$, with arbitrary distinct
$\mu,\nu\in\ring$, commutes with $X\oplus Y$. Hence, $\mu
I_n\oplus\nu I_m$ commutes with $f(X\oplus Y)$. This is possible
only if $f(X\oplus Y)=A\oplus B$, with some
$A\in\mat{\module{N}}{n}$ and $B\in\mat{\module{N}}{m}$. We will
show now that $A$ is independent of $Y$, i.e., is determined by
$f$ and $X$ only. Suppose $X\oplus Y'\in\Omega_{n+m'}$ for some
$m'\in\mathbb{N}$ and $Y'\in\mat{\module{M}}{m'}$. Then, using the
same argument as above, we conclude that $f(X\oplus Y')=A'\oplus
B'$, for some $A'\in\mat{\module{N}}{n}$ and
$B'\in\mat{\module{N}}{m'}$. Since
$$(X\oplus Y)(I_n\oplus 0_{m\times m'})=(I_n\oplus 0_{m\times
m'})(X\oplus Y'),$$ we must have
$$(A\oplus B)(I_n\oplus 0_{m\times m'})=(I_n\oplus 0_{m\times
m'})(A'\oplus B'),$$ so that $A=A'$. We define $f_{\rm
d.s.e.}(X):=A$. On the other hand, since $\Omega$ is similarity
invariant, we have that $Y\in(\Omega_{\rm d.s.e.})_{m}$. An
analogous argument then shows that $B$ is uniquely determined by
$f$ and $Y$, and we define $f_{\rm d.s.e.}(Y):=B$.

The mapping $f_{\rm d.s.e.}\colon\Omega_{\rm
d.s.e.}\to\ncspace{N}$ is correctly defined. Indeed, by the
similarity invariance of $\Omega$,  if $X\oplus Y\in\Omega$ and
$f(X\oplus Y)=A\oplus B$, then $Y\oplus X\in\Omega$ and $f(Y\oplus
X)=B\oplus A$.

Clearly, $f_{\rm d.s.e.}((\Omega_{\rm
d.s.e.})_n)\subseteq\mat{\module{N}}{n}$ for every
$n\in\mathbb{N}$. Since, by Proposition
\ref{prop:dirsum-ext-properties}, $\Omega_{\rm d.s.e.}$ is a nc
set, we have for arbitrary $X\in(\Omega_{\rm d.s.e.})_n$ and
$X'\in(\Omega_{\rm d.s.e.})_{n'}$ that $X\oplus X'\in(\Omega_{\rm
d.s.e.})_{n+n'}$ and $X\oplus X'\oplus Y\in\Omega_{n+n'+m}$, with
some $m\in\mathbb{N}$ and $Y\in\mat{\module{M}}{m}$. Using the
similarity invariance of $\Omega$, we observe that $$f(X\oplus
X'\oplus Y)=f_{\rm d.s.e.}(X)\oplus f_{\rm d.s.e.}(X')\oplus
f_{\rm d.s.e.}(Y)=f_{\rm d.s.e.}(X\oplus X')\oplus f_{\rm
d.s.e.}(Y),$$ thus $f_{\rm d.s.e.}(X\oplus X')=f_{\rm
d.s.e.}(X)\oplus f_{\rm d.s.e.}(X')$, i.e., $f_{\rm d.s.e.}$
respects direct sums.

If $X\in(\Omega_{\rm d.s.e.})_n$ and $S\in\mat{\ring}{n}$ is
invertible, then, by Proposition \ref{prop:dirsum-ext-properties},
$SXS^{-1}\in(\Omega_{\rm d.s.e.})_n$. We have $X\oplus
Y\in\Omega_{n+m}$ for some $m\in\mathbb{N}$ and
$Y\in\mat{\module{M}}{m}$. Since $\Omega$ is similarity invariant,
we also have that $$SXS^{-1}\oplus Y=(S\oplus I_m)(X\oplus
Y)(S\oplus I_m)^{-1}\in\Omega_{n+m}.$$ Then
\begin{multline*}
f_{\rm d.s.e.}(SXS^{-1})\oplus f_{\rm d.s.e.}(Y)=f(SXS^{-1}\oplus
Y)\\
=f\Big((S\oplus I_m)(X\oplus Y)(S\oplus I_m)^{-1} \Big)=(S\oplus
I_m)f(X\oplus Y)(S\oplus I_m)^{-1}\\ = (S\oplus I_m)(f_{\rm
d.s.e.}(X)\oplus f_{\rm d.s.e.}(Y))(S\oplus I_m)^{-1}=Sf_{\rm
d.s.e.}(X)S^{-1}\oplus f_{\rm d.s.e.}(Y),
\end{multline*}
thus $f_{\rm d.s.e.}(SXS^{-1})=Sf_{\rm d.s.e.}(X)S^{-1}$, i.e.,
$f_{\rm d.s.e.}$ respects similarities.

We conclude that $f_{\rm d.s.e.}$ is a nc function. By the
construction, $f_{\rm d.s.e.}$ extends $f$ to $\Omega_{\rm
d.s.e.}$ and such an extended nc function is unique.

(2) It is obvious that $\Omega_{\rm d.s.e.}=\ncspace{\module{M}}$.
Let $X\in\mat{\module{M}}{n}$. Then $X\oplus
Y\in\mat{\module{M}}{sm}$ for some $m\in\mathbb{N}$ and
$Y\in\mat{\module{M}}{(sm-n)}$. Since $f|_{\mat{\module{M}}{sm}}$
is polynomial on slices of some finite degree $M_{sm}$, one has
that, for an arbitrary $Z\in\mat{\module{M}}{n}$,
$$f_{X\oplus Y,Z\oplus 0_{(sm-n)\times (sm-n)}}(t)=f((X+tZ)\oplus Y)=f_{\rm
d.s.e.}(X+tZ)\oplus f_{\rm d.s.e.}(Y)$$ is a polynomial in $t$ of
degree at most $M_{sm}$. Then so is ${(f_{\rm
d.s.e.})}_{X,Z}(t)=f_{\rm d.s.e.}(X+tZ)$, and therefore, $f_{\rm
d.s.e.}$ is polynomial on slices. If, in addition, the degrees
$M_{sm}$ of $f|_{\mat{\module{M}}{sm}}$, $m=1,2,\ldots$, are
bounded, then the degrees of $f_{\rm
d.s.e.}|_{\mat{\module{M}}{n}}$, $n=1,2,\ldots$, are bounded by
the same constant. By Theorem \ref{thm:king-poly-gen}, $f|_{\rm
d.s.e.}$ is a nc polynomial over $\module{M}$ with coefficients in
$\module{N}$.

(3) Let $X\in(\Omega_{\rm d.s.e.})_n$ and
$Z\in\mat{\vecspace{V}}{n}$. Then $X\oplus Y\in\Omega_{n+m}$ for
some $m\in\mathbb{N}$ and $Y\in\mat{\vecspace{V}}{m}$. Since
$\Omega$ is finitely open, so is $\Omega_{\rm d.s.e.}$ by
Proposition \ref{prop:dirsum-ext-properties}. Then both
$X+tZ\in(\Omega_{\rm d.s.e.})_n$ and $(X+tZ)\oplus
Y\in\Omega_{n+m}$ for a sufficiently small scalar $t$. Since $f$
is G-differentiable at $X\oplus Y$, we have
\begin{multline*}
\delta f(X\oplus Y)(Z\oplus 0_{m\times m})=\lim_{t\to
0}\frac{f((X\oplus Y)+t(Z\oplus 0_{m\times m}))-f(X\oplus Y)}{t}\\
=\lim_{t\to
0}\frac{f((X+tZ)\oplus Y)-f(X\oplus Y)}{t}\\
=\lim_{t\to 0}\frac{(f_{\rm d.s.e.}(X+tZ)\oplus f_{\rm
d.s.e.}(Y))-(f_{\rm d.s.e.}(X)\oplus f_{\rm d.s.e.}(Y))}{t}\\
=\lim_{t\to 0}\frac{(f_{\rm d.s.e.}(X+tZ)-f_{\rm d.s.e.}(X))\oplus
f_{\rm
d.s.e.}(Y)}{t}\\
=\lim_{t\to 0}\frac{f_{\rm d.s.e.}(X+tZ)-f_{\rm
d.s.e.}(X)}{t}\oplus f_{\rm
d.s.e.}(Y)\\
=\delta f_{\rm d.s.e.}(X)(Z)\oplus f_{\rm d.s.e.}(Y),
\end{multline*}
i.e., $\delta f_{\rm d.s.e.}(X)(Z)$ exists and is determined by
$\delta f(X\oplus Y)(Z\oplus 0_{m\times m})$. Clearly, $\delta
f_{\rm d.s.e.}(X)(Z)$ is independent of the choice of $Y$, since
it is unique as a limit. We conclude that $f_{\rm d.s.e.}$ is
G-differentiable on $\Omega_{\rm d.s.e.}$.

(4) It suffices to show, either by the result of part (4) and the
definition of an analytic nc function or by Theorem
\ref{thm:f-queen}, that $f_{\rm d.s.e.}$ is locally bounded on
$\Omega_{\rm d.s.e.}$. Let $X\in(\Omega_{\rm d.s.e.})_n$. Then
$X\oplus Y\in\Omega_{n+m}$ for some $m\in\mathbb{N}$ and
$Y\in\mat{\vecspace{V}}{m}$. Since $\Omega_{n+m}$ is open and $f$
is locally bounded on $\Omega_{n+m}$, there exists a constant
$K>0$ such that $\|f(W)\|_{n+m}\le K$ for every $W$ in some open
ball $B(X\oplus Y,\epsilon)\subseteq\Omega_{n+m}$. By Proposition
\ref{prop:dirsum-ext-properties},  $(\Omega_{\rm d.s.e.})_n$ is
open. Let $\delta<C_1'(n,m)^{-1}\epsilon$ be such that
$B(X,\delta)\subseteq(\Omega_{\rm d.s.e.})_n$, where
$C_1'(n,m)=C_1'(\vecspace{V};n,m)$ is the constant on the
right-hand side of \eqref{eq:dirsums-norms} for $\vecspace{V}$.
Let $Z\in B(X,\delta)$ be arbitrary. Then $Z\oplus Y\in B(X\oplus
Y,\epsilon)$ and
\begin{multline*}
\|f_{\rm d.s.e.}(Z)\|_n\le\max\{\|f_{\rm d.s.e.}(Z)\|_n,\|f_{\rm
d.s.e.}(Y)\|_m\}\\
\le C_1(n,m)\|f_{\rm d.s.e.}(Z)\oplus f_{\rm
d.s.e.}(Y)\|_{n+m}\\
=C_1(n,m)\|f(Z\oplus Y)\|_{n+m}\le C_1(n,m)K,
\end{multline*}
where $C_1(n,m)=C_1(\vecspace{W};n,m)$ is the constant on the
left-hand side of \eqref{eq:dirsums-norms} for $\vecspace{W}$.
Thus, $f_{\rm d.s.e.}$ is locally bounded on $\Omega_{\rm
d.s.e.}$.

(5) It suffices to show, either by the result of part (4) and the
definition of a uniformly analytic nc function or by Corollary
\ref{cor:nc-u-analytic}, that $f_{\rm d.s.e.}$ is uniformly
locally bounded on $\Omega_{\rm d.s.e.}$. Let $X\in(\Omega_{\rm
d.s.e.})_n$. Then $X\oplus Y\in\Omega_{n+m}$ for some
$m\in\mathbb{N}$ and $Y\in\mat{\vecspace{V}}{m}$. Since $\Omega$
is uniformly-open and $f$ is uniformly locally bounded on
$\Omega$, there exists a constant $K>0$ such that
$\|f(W)\|_{n_W}\le K$ for every $W$ in some open nc ball $B_{\rm
nc}(X\oplus Y,\epsilon)\subseteq\Omega$ (here $n_W$ is the size of
the matrix $W$). By Proposition \ref{prop:dirsum-ext-properties},
$\Omega_{\rm d.s.e.}$ is uniformly-open. Moreover, the argument in
the proof of Proposition \ref{prop:dirsum-ext-properties} implies
that $B_{\rm nc}(X,\epsilon)\subseteq\Omega_{\rm d.s.e.}$: for an
arbitrary $Z\in B_{\rm nc}(X,\epsilon)_{nk}$, one has \eqref{eq:z}
with the permutation matrix $U$ defined by \eqref{eq:u}. By the
properties of operator space norms,
\begin{multline*}
\|f_{\rm d.s.e.}(Z)\|_{nk}\le\Big\|f_{\rm
d.s.e.}(Z)\oplus\bigoplus_{\beta=1}^kf_{\rm d.s.e.}(Y)\Big\|_{nk}
=\Big\|f\Big(Z\oplus\bigoplus_{\beta=1}^kY\Big)\Big\|_{nk}\\
=\Big\|U^{-1}f\Big(Z\oplus\bigoplus_{\beta=1}^kY\Big)U\Big\|_{nk}
=\Big\|f\Big(U^{-1}\Big(Z\oplus\bigoplus_{\beta=1}^kY\Big)U\Big)\Big\|_{nk}
\le K.
\end{multline*}
We conclude that $f_{\rm d.s.e.}$ is uniformly locally bounded on
$\Omega_{\rm d.s.e.}$.
\end{proof}

Proposition \ref{prop:ncfun-dirsum-ext} can be extended to higher
order nc functions as follows.

\begin{prop}\label{prop:higher-ncfun-dirsum-ext}
Let $\module{M}_0$, \ldots, $\module{M}_k$, $\module{N}_0$,
\ldots, $\module{N}_k$ be modules over a commutative unital ring
$\ring$, let $\Omega^{(0)}\subseteq\ncspacej{\module{M}}{0}$,
\ldots, $\Omega^{(k)}\subseteq\ncspacej{\module{M}}{k}$ be
similarity invariant nc sets, and let
$f\in\tclass{k}(\Omega^{(0)},\ldots,\Omega^{(k)};\ncspacej{\module{N}}{0},\ldots,\ncspacej{\module{N}}{k})$.
Then

\begin{enumerate}
    \item There exists a uniquely determined \emph{direct summands extension} \index{direct summands extension
    of a higher order nc function} of $f$, i.e., a nc function
    of order $k$,
${f}_{\rm d.s.e.}\in\tclass{k}({\Omega}^{(0)}_{\rm
d.s.e.},\ldots,{\Omega}^{(k)}_{\rm d.s.e.};
\ncspacej{\module{N}}{0},\ldots,\ncspacej{\module{N}}{k})$, such
that ${f}_{\rm
d.s.e.}|_{\Omega^{(0)}\times\cdots\times\Omega^{(k)}}=f$;
    \item If $\module{M}_0=\vecspace{V}_0$, \ldots, $\module{M}_k=\vecspace{V}_k$ are vector spaces over
    $\mathbb{C}$,
    $\module{N}_0=\vecspace{W}_0$, \ldots, $\module{N}_k=\vecspace{W}_k$ are Banach spaces equipped with
    admissible systems
    of rectangular matrix norms over $\vecspace{W}_0$, \ldots,
    $\vecspace{W}_k$,
    $\Omega^{(0)}\subseteq\ncspacej{\vecspace{V}}{0}$,
    \ldots, $\Omega^{(k)}\subseteq\ncspacej{\vecspace{V}}{k}$ are
    finitely open nc sets, and $f$ is
    $G_W$-differentiable on $\Omega^{(0)}\times\cdots\times\Omega^{(k)}$, then
    $f_{\rm d.s.e.}$ is $G_W$-differentiable on $\Omega^{(0)}_{\rm
    d.s.e.}\times\cdots\times\Omega^{(k)}_{\rm d.s.e.}$;
    moreover, if $X^0\in(\Omega^{(0)}_{\rm d.s.e.})_{n_0}$,
    \ldots, $X^k\in(\Omega^{(k)}_{\rm d.s.e.})_{n_k}$,
     $Z^0\in\mat{\vecspace{V}_0}{n_0}$, \ldots, $Z^k\in\mat{\vecspace{V}_k}{n_k}$,
$W^1\in\rmat{\vecspace{W}_1}{n_0}{n_1}$, \ldots,
$W^k\in\rmat{\vecspace{W}_k}{n_{k-1}}{n_k}$,
     and
    $X^0\oplus Y^0\in\Omega^{(0)}_{n_0+m_0}$, \ldots, $X^k\oplus Y^k\in\Omega^{(k)}_{n_k+m_k}$,  then
    \begin{multline*}
\frac{d}{dt} f\Big((X^0\oplus Y^0)+t(Z^0\oplus 0_{m_0\times
m_0}),\ldots,(X^k\oplus Y^k)+t(Z^k\oplus 0_{m_k\times
m_k})\Big)\\
(W^1\oplus 0_{m_0\times m_1},\ldots,W^k\oplus
0_{m_{k-1}\times m_k})\Big|_{t=0}\\
= \frac{d}{dt} f_{\rm
    d.s.e.}(X^0+tZ^0,\ldots,X^k+tZ^k)(W^1,\ldots,W^k)\Big|_{t=0}\oplus 0_{m_0\times m_k};
    \end{multline*}
    \item If $\module{M}_0=\vecspace{V}_0$, \ldots, $\module{M}_k=\vecspace{V}_k$,
    $\module{N}_0=\vecspace{W}_0$, \ldots, $\module{N}_k=\vecspace{W}_k$ are Banach spaces equipped with
    admissible systems
    of rectangular matrix norms over $\vecspace{V}_0$, \ldots,
    $\vecspace{V}_k$, $\vecspace{W}_0$, \ldots,
    $\vecspace{W}_k$,  $\Omega^{(0)}\subseteq\ncspacej{\vecspace{V}}{0}$,
    \ldots, $\Omega^{(k)}\subseteq\ncspacej{\vecspace{V}}{k}$ are
    open nc sets, and $f$ is analytic on $\Omega^{(0)}\times\cdots\times\Omega^{(k)}$, then
    $f_{\rm d.s.e.}$ is analytic on $\Omega^{(0)}_{\rm
    d.s.e.}\times\cdots\times\Omega^{(k)}_{\rm d.s.e.}$;
\item If $\module{M}_0=\vecspace{V}_0$, \ldots,
$\module{M}_k=\vecspace{V}_k$,
    $\module{N}_0=\vecspace{W}_0$, \ldots, $\module{N}_k=\vecspace{W}_k$ are operator spaces,
      $\Omega^{(0)}\subseteq\ncspacej{\vecspace{V}}{0}$,
    \ldots, $\Omega^{(k)}\subseteq\ncspacej{\vecspace{V}}{k}$ are
   uniformly-open nc sets, and $f$ is uniformly analytic on $\Omega^{(0)}\times\cdots\times\Omega^{(k)}$, then
    $f_{\rm d.s.e.}$ is uniformly analytic on $\Omega^{(0)}_{\rm
    d.s.e.}\times\cdots\times\Omega^{(k)}_{\rm d.s.e.}$.
\end{enumerate}
\end{prop}
We omit the proof which is essentially similar to that of
Proposition \ref{prop:ncfun-dirsum-ext}.

Our next goal is to define  the direct summands extension of a
sequence of $\ell$-linear mappings
$f_\ell\colon\mattuple{\module{M}}{s}{\ell}\to\mat{\module{N}}{s}$,
$\ell=0,1,\ldots$ satisfying conditions
\eqref{eq:ncfun_coef_0}--\eqref{eq:ncfun-coef_ell_ell}.
\begin{prop}\label{prop:LAC-dirsum-ext}
Let $\module{M}$ and $\module{N}$ be modules over a commutative
unital ring $\ring$, let $s_1$, \ldots, $s_m\in\mathbb{N}$ be such
that $s=s_1+\cdots+s_m$, let $Y=Y_{11}\oplus\cdots\oplus
Y_{mm}\in\mat{\module{M}}{s}$, with
$Y_{ii}\in\mat{\module{M}}{s_i}$, $i=1,\ldots,m$, and let a
sequence of $\ell$-linear mappings
$f_\ell\colon\mattuple{\module{M}}{s}{\ell}\to\mat{\module{N}}{s}$,
$\ell=0,1,\ldots$ satisfy
\eqref{eq:ncfun_coef_0}--\eqref{eq:ncfun-coef_ell_ell}.  Then
\begin{enumerate}
\item With respect to the corresponding block decomposition of
$s\times s$ matrices,
\begin{equation}\label{eq:f_0-dirsum-decomp}
f^{\rm
d.s.e.}_{0;\alpha,\beta}:=(f_0)_{\alpha\beta}\in\rmat{\module{N}}{s_\alpha}{s_\beta},\quad
 \alpha,\beta=1,\ldots,m,
\end{equation}
satisfy $f^{\rm d.s.e.}_{0;\alpha,\beta}=0$ if $\alpha\neq \beta$,
\begin{multline}\label{eq:f_ell-dirsum-decomp}
f_\ell(Z^1,\ldots,Z^\ell)_{\alpha\beta}=
\sum_{1\le\alpha_1,\ldots,\alpha_{\ell-1}\le m}f^{\rm
d.s.e.}_{\ell;\alpha_0,\ldots,\alpha_\ell}(Z^1_{\alpha_0\alpha_1},\ldots,Z^\ell_{\alpha_{\ell-1}\alpha_\ell}),\\
{\text where}\ \alpha_0=\alpha,\ \alpha_\ell=\beta,\,
\alpha,\beta=1,\ldots,m,\ \ell=1,2,\ldots,\
\end{multline}
with uniquely determined $\ell$-linear mappings
$$f^{\rm
d.s.e.}_{\ell;\alpha_0,\ldots,\alpha_\ell}\colon\rmat{\module{M}}{s_{\alpha_0}}{s_{\alpha_1}}\times\cdots\times
\rmat{\module{M}}{s_{\alpha_{\ell-1}}}{s_{\alpha_\ell}}\to\rmat{\module{N}}{s_{\alpha_0}}{s_{\alpha_\ell}}$$
satisfying
\begin{equation}\label{eq:ncfun-dirsum_coef_0}
Sf^{\rm d.s.e.}_{0;\beta,\beta}-f^{\rm
d.s.e.}_{0;\alpha,\alpha}S=f^{\rm
d.s.e.}_{1;\alpha,\beta}(SY_{\beta \beta}-Y_{\alpha
\alpha}S),\quad S\in\rmat{\ring}{s_\alpha}{s_\beta},
\end{equation}
 and for
$\ell=1,2,\ldots$,
\begin{multline}\label{eq:ncfun-dirsum-coef-ell_0}
Sf^{\rm
d.s.e.}_{\ell;\gamma,\alpha_1,\ldots,\alpha_\ell}(W^1,\ldots,W^\ell)-
f^{\rm
d.s.e.}_{\ell;\alpha_0,\ldots,\alpha_\ell}(SW^1,W^2,\ldots,W^\ell)\\
=f^{\rm
d.s.e.}_{\ell+1;\alpha_0,\gamma,\alpha_1,\ldots,\alpha_\ell}(SY_{\gamma\gamma}-Y_{\alpha_0\alpha_0}S,W^1,\ldots,W^\ell),
\quad S\in\rmat{\ring}{s_{\alpha_0}}{\gamma},
\end{multline}
\begin{multline}\label{eq:ncfun-dirsum-coef_ell_j}
f^{\rm
d.s.e.}_{\ell;\alpha_0,\ldots,\alpha_{j-1},\gamma,\alpha_{j+1},\ldots,\alpha_\ell}
(W^1,\ldots,W^{j-1},W^jS,W^{j+1},\ldots,W^\ell)\\
-f^{\rm d.s.e.}_{\ell;\alpha_0,\ldots,\alpha_\ell}
(W^1,\ldots,W^j,SW^{j+1},W^{j+2},\ldots,W^\ell)\\
=f^{\rm
d.s.e.}_{\ell+1;\alpha_0,\ldots,\alpha_j,\gamma,\alpha_{j+1},\ldots,\alpha_\ell}
(W^1,\ldots,W^j,SY_{\gamma\gamma}-Y_{\alpha_j\alpha_j}S,W^{j+1},\ldots,W^\ell),\\
S\in\rmat{\ring}{s_{\alpha_j}}{\gamma},
\end{multline}
\begin{multline}\label{eq:ncfun-dirsum-coef_ell_ell}
f^{\rm
d.s.e.}_{\ell;\alpha_0,\ldots,\alpha_{\ell-1},\gamma}(W^1,\ldots,W^{\ell-1},W^\ell
S)-f^{\rm
d.s.e.}_{\ell;\alpha_0,\ldots,\alpha_\ell}(W^1,\ldots,W^\ell)S\\
=f^{\rm
d.s.e.}_{\ell+1;\alpha_0,\ldots,\alpha_{\ell},\gamma}(W^1,\ldots,W^\ell,SY_{\gamma\gamma}-Y_{\alpha_\ell\alpha_\ell}S),
\quad S\in\rmat{\ring}{\alpha_\ell}{\gamma}.
\end{multline}
In particular, conditions
\eqref{eq:ncfun-dirsum_coef_0}--\eqref{eq:ncfun-dirsum-coef_ell_ell}
for $f^{\rm d.s.e.}_{\ell;\alpha,\ldots,\alpha}$, \index{$f^{\rm
d.s.e.}_{\ell;\alpha_0,\ldots,\alpha_\ell}$} $\ell=0,1,\ldots$,
with a fixed $\alpha$, coincide with
\eqref{eq:ncfun_coef_0}--\eqref{eq:ncfun-coef_ell_ell} for $s$
replaced by $s_\alpha$ and $Y$ replaced by $Y_{\alpha\alpha}$;
\item If $s_1=\ldots=s_m$ and $Y_{11}=\cdots=Y_{mm}$, then the
mappings $f^{\rm d.s.e.}_{\ell;\alpha_0,\ldots,\alpha_\ell}$,
$1\le\alpha_0$, \ldots, $\alpha_\ell\le m$, coincide
 for every fixed $\ell$;
 \item If $\ring=\mathbb{C}$, $\module{M}=\vecspace{V}$ and
$\module{N}=\vecspace{W}$ are Banach spaces equipped with
admissible systems of rectangular matrix norms over $\vecspace{V}$
and $\vecspace{W}$, and the $\ell$-linear mappings $f_\ell$ are
bounded, then so are $f^{\rm
d.s.e.}_{\ell;\alpha_0,\ldots,\alpha_\ell}$;
 \item If $\ring=\mathbb{C}$, $\module{M}=\vecspace{V}$ and
$\module{N}=\vecspace{W}$ are Banach spaces equipped with
admissible systems of rectangular matrix norms over $\vecspace{V}$
and $\vecspace{W}$, \eqref{eq:rect-simprod-norms} hold for
$\vecspace{V}$ and $\vecspace{W}$ with the constants
$C^{\vecspace{V}}$ and $C^{\vecspace{W}}$ which are independent of
$n,p,q$, and $m$, and the $\ell$-linear mappings $f_\ell$ are
completely bounded,  then so are $f^{\rm
d.s.e.}_{\ell;\alpha_0,\ldots,\alpha_\ell}$, and
\begin{equation}\label{eq:f_ell-dirsum-cb-norms}
\|f^{\rm
d.s.e.}_{\ell;\alpha_0,\ldots,\alpha_\ell}\|_{\mathcal{L}^\ell_{\rm
cb}}\le
C^{\vecspace{W}}(C^{\vecspace{V}})^\ell\|f_\ell\|_{\mathcal{L}^\ell_{\rm
cb}};
\end{equation}
in particular, if $\vecspace{V}$ and $\vecspace{W}$ are operator
spaces, then
\begin{equation}\label{eq:f_ell-dirsum-cb-norms-os}
\|f^{\rm
d.s.e.}_{\ell;\alpha_0,\ldots,\alpha_\ell}\|_{\mathcal{L}^\ell_{\rm
cb}}\le \|f_\ell\|_{\mathcal{L}^\ell_{\rm cb}}.
\end{equation}
\end{enumerate}
\end{prop}
\begin{proof}
(1) Let $S=E_\alpha E_\alpha^\top\in\mat{\ring}{s}$. Since
$SY=YS$, we obtain from \eqref{eq:ncfun_coef_0} that
${(Sf_0)}_{\alpha\beta}={(f_0S)}_{\alpha\beta}$, i.e.,
${(f_0)}_{\alpha\beta}={(f_0)}_{\alpha\alpha}\delta_{\alpha\beta}$,
hence ${(f_0)}_{\alpha\beta}=0$ if $\alpha\neq\beta$.

For any $\alpha,\beta=1,\ldots,m$, and for an arbitrary
$S\in\rmat{\ring}{s_\alpha}{s_\beta}$, we set $T=E_\alpha
SE_\beta^\top$. Define the linear mapping $f^{\rm
d.s.e.}_{1;\alpha,\beta}\colon\rmat{\module{M}}{s_\alpha}{s_\beta}\to\rmat{\module{N}}{s_\alpha}{s_\beta}$
by $$f^{\rm d.s.e.}_{1;\alpha,\beta}(W):=f_1(E_\alpha
WE_\beta^\top)_{\alpha\beta}.$$ Then \eqref{eq:ncfun_coef_0} with
$T$ in the place of $S$ implies
\begin{multline*}
Sf^{\rm d.s.e.}_{0;\beta,\beta}-f^{\rm
d.s.e.}_{0;\alpha,\alpha}S=S(f_0)_{\beta\beta}-(f_0)_{\alpha\alpha}S
=(Tf_0-f_0T)_{\alpha\beta}=f_1(TY-YT)_{\alpha\beta}\\
=f_1(E_\alpha
(SY_{\beta\beta}-Y_{\alpha\alpha}S)E_\beta^\top)_{\alpha\beta}=f^{\rm
d.s.e.}_{1;\alpha,\beta}(SY_{\beta\beta}-Y_{\alpha\alpha}S),
\end{multline*}
i.e., \eqref{eq:ncfun-dirsum_coef_0} holds. Similarly, for every
$\ell=2,3,\ldots$, $\alpha_0, \ldots, \alpha_\ell=1,\ldots,m$, one
defines the $\ell$-linear mapping $$f^{\rm
d.s.e.}_{\ell;\alpha_0,\ldots,\alpha_\ell}\colon\rmat{\module{M}}{\alpha_0}{\alpha_1}\times\cdots\times
\rmat{\module{M}}{\alpha_{\ell-1}}{\alpha_\ell}\to\rmat{\module{N}}{\alpha_0}{\alpha_\ell}$$
by
\begin{equation}\label{eq:def-f_ell-dirsum}
f^{\rm
d.s.e.}_{\ell;\alpha_0,\ldots,\alpha_\ell}(W^1,\ldots,W^\ell):=f_\ell(E_{\alpha_0}W^1E_{\alpha_1}^\top,\ldots,
E_{\alpha_{\ell-1}}W^\ell
E_{\alpha_\ell}^\top)_{\alpha_0\alpha_\ell}
\end{equation}
 and checks that
\eqref{eq:ncfun-dirsum-coef-ell_0}--\eqref{eq:ncfun-dirsum-coef_ell_ell}
hold.

Let $1\le\alpha,\beta\le m$ be arbitrary. Since $E_i E_i^\top$
commutes with $Y$ for every $i=1,\ldots,m$, it follows by
linearity from
\eqref{eq:ncfun-coef-ell_0}--\eqref{eq:ncfun-coef_ell_ell} that
\begin{multline*}
f_\ell(Z^1,\ldots,Z^\ell)_{\alpha\beta}=f_\ell\Big(\sum_{1\le\gamma_0,\alpha_1\le
m}E_{\gamma_0}Z^1_{\gamma_0\alpha_1}E_{\alpha_1}^\top,\ldots,\\
\sum_{1\le\gamma_{\ell-2},\alpha_{\ell-1}\le
m}E_{\gamma_{\ell-2}}Z^{\ell-1}_{\gamma_{\ell-2}\alpha_{\ell-1}}E_{\alpha_{\ell-1}}^\top,
\sum_{1\le\gamma_{\ell-1},\gamma_{\ell}\le
m}E_{\gamma_{\ell-1}}Z^{\ell}_{\gamma_{\ell-1}\gamma_{\ell}}E_{\gamma_{\ell}}^\top\Big)_{\alpha\beta}\\
=\Bigg(E_\alpha E_\alpha^\top
f_\ell\Big(\sum_{1\le\gamma_0,\alpha_1\le
m}E_{\gamma_0}Z^1_{\gamma_0\alpha_1}E_{\alpha_1}^\top
E_{\alpha_1}E_{\alpha_1}^\top,
\ldots,\\
\sum_{1\le\gamma_{\ell-2},\alpha_{\ell-1}\le
m}E_{\gamma_{\ell-2}}Z^{\ell-1}_{\gamma_{\ell-2}\alpha_{\ell-1}}E_{\alpha_{\ell-1}}^\top
E_{\alpha_{\ell-1}} E_{\alpha_{\ell-1}}^\top,\\
\sum_{1\le\gamma_{\ell-1},\gamma_{\ell}\le
m}E_{\gamma_{\ell-1}}Z^{\ell}_{\gamma_{\ell-1}\gamma_{\ell}}
E_{\gamma_{\ell}}^\top\Big)E_\beta E_\beta^\top\Bigg)_{\alpha\beta}\\
= f_\ell\Big(E_\alpha E_\alpha^\top\sum_{1\le\gamma_0,\alpha_1\le
m}E_{\gamma_0}Z^1_{\gamma_0\alpha_1}E_{\alpha_1}^\top,\
E_{\alpha_1}E_{\alpha_1}^\top\sum_{1\le\gamma_1,\alpha_2\le
m}E_{\gamma_1}Z^1_{\gamma_1\alpha_2}E_{\alpha_2}^\top, \ldots,\\
E_{\alpha_{\ell-2}}E_{\alpha_{\ell-2}}^\top\sum_{1\le\gamma_{\ell-2},\alpha_{\ell-1}\le
m}E_{\gamma_{\ell-2}}Z^{\ell-1}_{\gamma_{\ell-2}\alpha_{\ell-1}}E_{\alpha_{\ell-1}}^\top,\\
E_{\alpha_{\ell-1}}
E_{\alpha_{\ell-1}}^\top\sum_{1\le\gamma_{\ell-1},\gamma_{\ell}\le
m}E_{\gamma_{\ell-1}}Z^{\ell}_{\gamma_{\ell-1}\gamma_{\ell}}
E_{\gamma_{\ell}}^\top E_\beta E_\beta^\top\Big)_{\alpha\beta}\\
=\sum_{1\le\alpha_1,\ldots,\alpha_{\ell-1}\le
m}f_\ell(E_{\alpha}Z^1_{\alpha\alpha_1}E_{\alpha_1}^\top,
E_{\alpha_1}Z^2_{\alpha_1\alpha_2}E_{\alpha_2}^\top,\ldots,\\
E_{\alpha_{\ell-2}}
Z^{\ell-1}_{\alpha_{\ell-2}\alpha_{\ell-1}}E_{\alpha_{\ell-1}}^\top,E_{\alpha_{\ell-1}}
Z^{\ell}_{\alpha_{\ell-1}\beta}E_{\beta}^\top)_{\alpha\beta}
\\
=\sum_{1\le\alpha_1,\ldots,\alpha_{\ell-1}\le m}f^{\rm
d.s.e.}_{\ell;\alpha,\alpha_1,\ldots,\alpha_{\ell-1},\alpha_\ell}(Z^1_{\alpha\alpha_1},
Z^2_{\alpha_1\alpha_2},\ldots,
Z^{\ell-1}_{\alpha_{\ell-2}\alpha_{\ell-1}},
Z^{\ell}_{\alpha_{\ell-1}\beta}),
\end{multline*}
i.e., \eqref{eq:f_ell-dirsum-decomp} holds. On the other hand, if
\eqref{eq:f_ell-dirsum-decomp} is satisfied with some
$\ell$-linear mappings $$f^{\rm
d.s.e.}_{\ell;\alpha_0,\ldots,\alpha_\ell}\colon\rmat{\module{M}}{s_{\alpha_0}}{s_{\alpha_1}}\times\cdots\times
\rmat{\module{M}}{s_{\alpha_{\ell-1}}}{s_{\alpha_\ell}}\to\rmat{\module{N}}{s_{\alpha_0}}{s_{\alpha_\ell}},$$
then
\begin{multline*}
f_\ell(E_{\alpha_0}Z^1_{\alpha_0\alpha_1}E_{\alpha_1}^\top,
\ldots,E_{\alpha_{\ell-1}}
Z^{\ell}_{\alpha_{\ell-1}\alpha_\ell}E_{\alpha_\ell}^\top)_{\alpha_0\alpha_\ell}\\
=f^{\rm
d.s.e.}_{\ell;\alpha_0,\ldots,\alpha_\ell}\Big((E_{\alpha_0}Z^1_{\alpha_0\alpha_1}E_{\alpha_1}^\top)_{\alpha_0\alpha_1},
\ldots,(E_{\alpha_{\ell-1}}
Z^{\ell}_{\alpha_{\ell-1}\alpha_\ell}E_{\alpha_\ell}^\top)_{\alpha_{\ell-1}\alpha_\ell}\Big)\\
=f^{\rm
d.s.e.}_{\ell;\alpha_0,\ldots,\alpha_\ell}(Z^1_{\alpha_0\alpha_1},
\ldots, Z^{\ell}_{\alpha_{\ell-1}\alpha_\ell}),
\end{multline*}
which coincides with our original definition of $f^{\rm
d.s.e.}_{\ell;\alpha_0,\ldots,\alpha_\ell}$ as in
\eqref{eq:def-f_ell-dirsum}.

(2) follows from
\eqref{eq:ncfun-dirsum_coef_0}--\eqref{eq:ncfun-dirsum-coef_ell_ell}
with $S=I_{s/m}$.

(3) By \eqref{eq:def-f_ell-dirsum} and
\eqref{eq:rect-inj-proj-estimates}, we have for
$W^1\in\rmat{\vecspace{V}}{s_{\alpha_0}}{s_{\alpha_1}}$, \ldots,
$W^\ell\in\rmat{\vecspace{V}}{s_{\alpha_{\ell-1}}}{s_{\alpha_\ell}}$:
\begin{multline*}
\|f^{\rm
d.s.e.}_{\ell;\alpha_0,\ldots,\alpha_\ell}(W^1,\ldots,W^\ell)\|_{s_{\alpha_0},s_{\alpha_1}}\\
=
\|\pi_{\alpha_0,\alpha_\ell}^{(s_1,\ldots,s_m),\vecspace{W}}f^\ell(E_{\alpha_0}W^1E_{\alpha_1}^\top,\ldots,
E_{\alpha_{\ell-1}}W^\ell
E_{\alpha_\ell}^\top)\|_{s_{\alpha_0},s_{\alpha_1}}
\\
 \le
C^{\vecspace{W}}(s_{\alpha_0},s,s,s_{\alpha_\ell})\|f_\ell\|\|E_{\alpha_0}W^1
E_{\alpha_{1}}^\top\|_s\cdots\|E_{\alpha_{\ell-1}}W^\ell
E_{\alpha_\ell}^\top\|_s\\
\le
C^{\vecspace{W}}(s_{\alpha_0},s,s,s_{\alpha_\ell})C^{\vecspace{V}}(s,s_{\alpha_0},s_{\alpha_1},s)\cdots
C^{\vecspace{V}}(s,s_{\alpha_0},s_{\alpha_1},s)\\
\cdot\|f_\ell\|\,\|W^1\|_{s_{\alpha_0},s_{\alpha_1}}\cdots
\|W^\ell\|_{s_{\alpha_{\ell-1}},s_{\alpha_\ell}},
\end{multline*}
where we indicate the corresponding Banach space in the
superscript, for the corresponding rectangular block projections
and constants. Thus the $\ell$-linear mappings $f^{\rm
d.s.e.}_{\ell;\alpha_0,\ldots,\alpha_\ell}$ are bounded.

(4)  By \eqref{eq:def-f_ell-dirsum} and
\eqref{eq:rect-inj-proj-estimates}, we have for every $n_0$,
\ldots, $n_\ell\in\mathbb{N}$ and for every
$W^1\in\rmat{\left(\rmat{\vecspace{V}}{s_{\alpha_0}}{s_{\alpha_1}}\right)}{n_0}{n_1}\cong
\rmat{\vecspace{V}}{s_{\alpha_0}n_0}{s_{\alpha_1}n_1}$, \ldots,
$W^\ell\in\rmat{\left(\rmat{\vecspace{V}}{s_{\alpha_{\ell-1}}}{s_{\alpha_\ell}}\right)}{n_{\ell-1}}{n_\ell}
\cong
\rmat{\vecspace{V}}{s_{\alpha_{\ell-1}}n_{\ell-1}}{s_{\alpha_\ell}n_\ell}$:
\begin{multline*}
\|f^{{\rm
d.s.e.}(n_0,\ldots,n_\ell)}_{\ell;\alpha_0,\ldots,\alpha_\ell}(W^1,\ldots,W^\ell)\|_{s_{\alpha_0}n_0,s_{\alpha_1}n_1}\\
=
\|f_\ell^{(n_0,\ldots,n_\ell)}((\id_{\mat{\mathbb{C}}{n_0}}\otimes
E_{\alpha_0})W^1(\id_{\mat{\mathbb{C}}{n_1}}\otimes
E_{\alpha_1}^\top),\ldots,\hfill \\
(\id_{\mat{\mathbb{C}}{n_{\ell-1}}}\otimes
E_{\alpha_{\ell-1}})W^\ell (\id_{\mat{\mathbb{C}}{n_\ell}}\otimes
E_{\alpha_\ell}^\top))_{s_{\alpha_0}n_0,s_{\alpha_1}n_1}\|_{s_{\alpha_0}n_0,s_{\alpha_1}n_1}
\\
 \le
C^{\vecspace{W}}\|f_\ell\|_{\mathcal{L}^\ell_{\rm
cb}}\|(\id_{\mat{\mathbb{C}}{n_0}}\otimes
E_{\alpha_0})W^1(\id_{\mat{\mathbb{C}}{n_1}}\otimes
E_{\alpha_1}^\top)\|_{sn_0,sn_1}
\cdots\\
\cdot\|(\id_{\mat{\mathbb{C}}{n_{\ell-1}}}\otimes
E_{\alpha_{\ell-1}})W^\ell (\id_{\mat{\mathbb{C}}{n_\ell}}\otimes
E_{\alpha_\ell}^\top)\|_{sn_{\ell-1},sn_\ell}\\
\le
C^{\vecspace{W}}(C^{\vecspace{V}})^\ell\|f_\ell\|_{\mathcal{L}^\ell_{\rm
cb}} \,\|W^1\|_{s_{\alpha_0}n_0,s_{\alpha_1}n_1}\cdots
\|W^\ell\|_{s_{\alpha_{\ell-1}n_{\ell-1}},s_{\alpha_\ell}n_\ell},
\end{multline*}
which implies \eqref{eq:f_ell-dirsum-cb-norms}. Clearly, in the
case of operator spaces, $C^{\vecspace{V}}=C^{\vecspace{W}}=1$,
and  \eqref{eq:f_ell-dirsum-cb-norms} becomes
\eqref{eq:f_ell-dirsum-cb-norms-os}.
\end{proof}

\begin{rem}\label{rem:LAC-dirsum-ext}
Let $\Omega\subseteq\ncspace{\module{M}}$ be a similarity
invariant nc set, let $f\colon\Omega\to\ncspace{\module{N}}$ be a
nc function, and let $Y=\bigoplus_{i=1}^mY_{ii}\in\Omega_s$, with
$Y_{ii}\in\mat{\module{M}}{s_i}$. Then the sequence of
$\ell$-linear mappings $f_\ell:=\Delta_R^\ell
f(Y,\ldots,Y)\colon\mattuple{\module{M}}{s}{\ell}\to\mat{\module{N}}{s}$,
$\ell=0,1,\ldots$,
 satisfies
\eqref{eq:ncfun_coef_0}--\eqref{eq:ncfun-coef_ell_ell}; see Remark
\ref{rem:lost_abbey}. By Proposition
\ref{prop:dirsum-ext-properties}, $\Omega_{\rm
d.s.e.}\subseteq\ncspace{\module{M}}$ is a similarity invariant nc
set extending $\Omega$, and by Proposition
\ref{prop:ncfun-dirsum-ext}, there exists a nc function $f_{\rm
d.s.e.}\colon\Omega_{\rm d.s.e.}\to\ncspace{\module{N}}$ extending
$f$. We then have \eqref{eq:f_entries-mult} for $k=\ell$ and
$X^0=\cdots=X^\ell=Y$, with $f$ in the left-hand side replaced by
$f_\ell$, and $f$ in the
 right-hand side replaced by
$\Delta_R^\ell f_{\rm d.s.e.}$. Comparing this equality with
\eqref{eq:f_ell-dirsum-decomp}, we conclude that
\begin{equation}\label{eq:LAC-dirsum-vs-delta}
f^{\rm d.s.e.}_{\ell;\alpha_0,\ldots,\alpha_\ell}=\Delta_R^\ell
f_{\rm
d.s.e.}(Y_{\alpha_0\alpha_0},\ldots,Y_{\alpha_\ell\alpha_\ell}).
\end{equation}
On the other hand, given a sequence of $\ell$-linear mappings
$f_\ell\colon\mattuple{\module{M}}{s}{\ell}\to\mat{\module{N}}{s}$,
$\ell=0,1,\ldots$, satisfying conditions
\eqref{eq:ncfun_coef_0}--\eqref{eq:ncfun-coef_ell_ell} for
$Y=\bigoplus_{i=1}^mY_{ii}\in\mat{\module{M}}{s}$, by Theorem
\ref{thm:sufficiency_of_LAC} one can define a nc function
\begin{equation}\label{eq:ncps-on-nilp}
f(X)=\sum_{\ell=0}^\infty
\Big(X-\bigoplus_{\alpha=1}^mY\Big)^{\odot_s\ell}f_\ell
\end{equation}
 on ${\rm Nilp}(\module{M},Y)$ with values in
$\ncspace{\module{N}}$. Moreover, by Theorem
\ref{thm:ncps-nilp-gen-unique}, $f_\ell=\Delta_R^\ell f(Y,\ldots,
Y)$. Extending $f$ to ${\rm Nilp}(\module{M},Y)_{\rm d.s.e.}$ and
writing the TT expansion for $f_{\rm d.s.e.}$ (which is a finite
sum at every point of ${\rm Nilp}(\module{M},Y)_{\rm d.s.e.}$), we
can define the direct summands extensions $f^{\rm
d.s.e.}_{\ell;\alpha_0,\ldots,\alpha_\ell}$ of the $\ell$-linear
mappings $f_\ell$ by \eqref{eq:LAC-dirsum-vs-delta}. This gives an
alternative way of proving parts (1) and (2) of Proposition
\ref{prop:LAC-dirsum-ext} using Proposition
\ref{prop:ncfun-dirsum-ext}. Parts (3) and (4) of Proposition
\ref{prop:LAC-dirsum-ext} can also be derived from Proposition
\ref{prop:ncfun-dirsum-ext} provided the the series in
\eqref{eq:ncps-on-nilp} defines an analytic (resp., uniformly
analytic) nc function $f$ on some open (resp., uniformly open)
neighborhood of $Y$.
\end{rem}

In a similar way, we can define  the direct summands extension of
a sequence of $\ell$-linear mappings
$f_w\colon\mattuple{\module{M}}{s}{\ell}\to\mat{\module{N}}{s}$,
$w\in\free_d$, satisfying conditions
\eqref{eq:ncfun_coef_empty}--\eqref{eq:ncfun-coef_w_ell}. The
 next proposition can be proved similarly to Proposition
\ref{prop:LAC-dirsum-ext}, or can be derived from the result of
Proposition \ref{prop:LAC-dirsum-ext} via relations
\eqref{eq:lw-forms}, \eqref{eq:f_w_via_f_l}.
\begin{prop}\label{prop:LAC-dirsum-ext-semigr}
Let $\ring$ be a commutative unital ring, let $\module{N}$ be a
module over $\ring$, let $s_1$, \ldots, $s_m\in\mathbb{N}$ be such
that $s=s_1+\cdots+s_m$, let
$Y=(Y_1,\ldots,Y_d)\in\mattuple{\ring}{s}{d}$ and
$Y_k=(Y_k)_{11}\oplus\cdots\oplus (Y_k)_{mm}\in\mat{\ring}{s}$,
with $(Y_k)_{ii}\in\mat{\ring}{s_i}$, $k=1,\ldots,d$,
$i=1,\ldots,m$, and let a sequence of $\ell$-linear mappings
$f_w\colon\mattuple{\ring}{s}{\ell}\to\mat{\module{N}}{s}$,
$w\in\free_d\colon |w|=\ell$, $\ell=0,1,\ldots$, satisfy
\eqref{eq:ncfun_coef_empty}--\eqref{eq:ncfun-coef_w_ell}.  Then
\begin{enumerate}
\item With respect to the corresponding block decomposition of
$s\times s$ matrices,
\begin{equation}\label{eq:f_empty-dirsum-decomp}
f^{\rm
d.s.e.}_{\emptyset;\alpha,\beta}:=(f_\emptyset)_{\alpha\beta}\in\rmat{\module{N}}{s_\alpha}{s_\beta},\quad
 \alpha,\beta=1,\ldots,m,
\end{equation}
satisfy $f^{\rm d.s.e.}_{\emptyset;\alpha,\beta}=0$ if $\alpha\neq
\beta$,
\begin{multline}\label{eq:f_w-dirsum-decomp}
f_w(A^1,\ldots,A^\ell)_{\alpha\beta}=
\sum_{1\le\alpha_1,\ldots,\alpha_{\ell-1}\le m}f^{\rm
d.s.e.}_{w;\alpha_0,\ldots,\alpha_\ell}(A^1_{\alpha_0\alpha_1},\ldots,A^\ell_{\alpha_{\ell-1}\alpha_\ell}),\\
{\text where}\ |w|=\ell,\ \alpha_0=\alpha,\ \alpha_\ell=\beta,\,
\alpha,\beta=1,\ldots,m,\ \ell=1,2,\ldots,\
\end{multline}
with uniquely determined $\ell$-linear mappings
$$f^{\rm
d.s.e.}_{w;\alpha_0,\ldots,\alpha_\ell}\colon\rmat{\ring}{s_{\alpha_0}}{s_{\alpha_1}}\times\cdots\times
\rmat{\ring}{s_{\alpha_{\ell-1}}}{s_{\alpha_\ell}}\to\rmat{\module{N}}{s_{\alpha_0}}{s_{\alpha_\ell}}$$
satisfying
\begin{equation}\label{eq:ncfun-dirsum_coef_empty}
Sf^{\rm d.s.e.}_{\emptyset;\beta,\beta}-f^{\rm
d.s.e.}_{\emptyset;\alpha,\alpha}S=\sum_{k=1}^df^{\rm
d.s.e.}_{g_k;\alpha,\beta}(S(Y_k)_{\beta \beta}-(Y_k)_{\alpha
\alpha}S),\quad S\in\rmat{\ring}{s_\alpha}{s_\beta},
\end{equation}
 and for \index{$f^{\rm
d.s.e.}_{w;\gamma,\alpha_1,\ldots,\alpha_\ell}$}
$\ell=1,2,\ldots$,
\begin{multline}\label{eq:ncfun-dirsum-coef-w_0}
Sf^{\rm
d.s.e.}_{w;\gamma,\alpha_1,\ldots,\alpha_\ell}(A^1,\ldots,A^\ell)-
f^{\rm
d.s.e.}_{w;\alpha_0,\ldots,\alpha_\ell}(SA^1,A^2,\ldots,A^\ell)\\
=\sum_{k=1}^df^{\rm
d.s.e.}_{g_kw;\alpha_0,\gamma,\alpha_1,\ldots,\alpha_\ell}(S(Y_k)_{\gamma\gamma}-
(Y_k)_{\alpha_0\alpha_0}S,A^1,\ldots,A^\ell), \quad
S\in\rmat{\ring}{s_{\alpha_0}}{\gamma},
\end{multline}
\begin{multline}\label{eq:ncfun-dirsum-coef_w_j}
f^{\rm
d.s.e.}_{w;\alpha_0,\ldots,\alpha_{j-1},\gamma,\alpha_{j+1},\ldots,\alpha_\ell}
(A^1,\ldots,A^{j-1},A^jS,A^{j+1},\ldots,A^\ell)\\
-f^{\rm d.s.e.}_{w;\alpha_0,\ldots,\alpha_\ell}
(A^1,\ldots,A^j,SA^{j+1},A^{j+2},\ldots,A^\ell)\\
=\sum_{k=1}^df^{\rm d.s.e.}_{g_{i_1}\cdots
g_{i_j}g_kg_{i_{j+1}}\cdots
g_{i_\ell};\alpha_0,\ldots,\alpha_j,\gamma,\alpha_{j+1},\ldots,\alpha_\ell}
(A^1,\ldots,A^j,\\
S(Y_k)_{\gamma\gamma}-(Y_k)_{\alpha_j\alpha_j}S,
A^{j+1},\ldots,A^\ell),\\ w=g_{i_1}\cdots g_{i_\ell}\in\free_d,\
S\in\rmat{\ring}{s_{\alpha_j}}{\gamma},
\end{multline}
\begin{multline}\label{eq:ncfun-dirsum-coef_w_ell}
f^{\rm
d.s.e.}_{w;\alpha_0,\ldots,\alpha_{\ell-1},\gamma}(A^1,\ldots,A^{\ell-1},A^\ell
S)-f^{\rm
d.s.e.}_{w;\alpha_0,\ldots,\alpha_\ell}(A^1,\ldots,A^\ell)S\\
=\sum_{k=1}^df^{\rm
d.s.e.}_{wg_k;\alpha_0,\ldots,\alpha_{\ell},\gamma}(A^1,\ldots,A^\ell,S(Y_k)_{\gamma\gamma}-
(Y_k)_{\alpha_\ell\alpha_\ell}S), \quad
S\in\rmat{\ring}{\alpha_\ell}{\gamma}.
\end{multline}
In particular, conditions
\eqref{eq:ncfun-dirsum_coef_empty}--\eqref{eq:ncfun-dirsum-coef_w_ell}
for $f^{\rm d.s.e.}_{w;\alpha,\ldots,\alpha}$, $w\in\free_d$, with
a fixed $\alpha$, coincide with
\eqref{eq:ncfun_coef_empty}--\eqref{eq:ncfun-coef_w_ell} for $s$
replaced by $s_\alpha$ and $Y$ replaced by $Y_{\alpha\alpha}$;
\item If $s_1=\ldots=s_m$ and $Y_{11}=\cdots=Y_{mm}$, then the
mappings $f^{\rm d.s.e.}_{w;\alpha_0,\ldots,\alpha_\ell}$,
$1\le\alpha_0$, \ldots, $\alpha_\ell\le m$, coincide
 for every fixed $w$;
 \item If $\ring=\mathbb{C}$, $\module{M}=\vecspace{V}$ and
$\module{N}=\vecspace{W}$ are Banach spaces equipped with
admissible systems of rectangular matrix norms over $\vecspace{V}$
and $\vecspace{W}$, and the $\ell$-linear mappings $f_w$ are
bounded, then so are $f^{\rm
d.s.e.}_{w;\alpha_0,\ldots,\alpha_\ell}$;
 \item If $\ring=\mathbb{C}$, $\module{M}=\vecspace{V}$ and
$\module{N}=\vecspace{W}$ are Banach spaces equipped with
admissible systems of rectangular matrix norms over $\vecspace{V}$
and $\vecspace{W}$, \eqref{eq:rect-simprod-norms} hold for
$\vecspace{V}$ and $\vecspace{W}$ with the constants
$C^{\vecspace{V}}$ and $C^{\vecspace{W}}$ which are independent of
$n,p,q$, and $m$, and the $\ell$-linear mappings $f_w$ are
completely bounded,  then so are $f^{\rm
d.s.e.}_{w;\alpha_0,\ldots,\alpha_\ell}$, and
\begin{equation}\label{eq:f_w-dirsum-cb-norms}
\|f^{\rm
d.s.e.}_{w;\alpha_0,\ldots,\alpha_\ell}\|_{\mathcal{L}^\ell_{\rm
cb}}\le
C^{\vecspace{W}}(C^{\vecspace{V}})^\ell\|f_w\|_{\mathcal{L}^\ell_{\rm
cb}};
\end{equation}
in particular, if $\vecspace{V}$ and $\vecspace{W}$ are operator
spaces, then
\begin{equation}\label{eq:f_w-dirsum-cb-norms-os}
\|f^{\rm
d.s.e.}_{w;\alpha_0,\ldots,\alpha_\ell}\|_{\mathcal{L}^\ell_{\rm
cb}}\le \|f_w\|_{\mathcal{L}^\ell_{\rm cb}}.
\end{equation}
\end{enumerate}
\end{prop}

\chapter*{(Some) earlier work on nc functions} \label{sec:revprev}
We concentrate here on the works most directly related to our main theme,
namely the theory of nc functions viewed as functions on square matrices of all sizes,
and how they compare to our results; we apologize in advance for any omissions.

Taylor \cite[Section 6]{T1} introduced the algebra that he denoted by ${\mathfrak D}(U)$,
where $U$ is (in our terminology) an open nc subset of $\ncspaced{{\mathbb C}}{n}$,
that consists of functions $f \colon U \to \ncspace{{\mathbb C}}$ that preserve matrix size
and respect intertwining, and are {\em assumed} to be analytic when restricted to $U \cap \mattuple{{\mathbb C}}{k}{n}$
for all $k$. He proved, using the analyticity assumption,
the existence of the nc difference-differential $\Delta_R f$ (to be specific, the analyticity assumption is used
to show that $\Delta_R f(X,Y)$ is a linear mapping between appropriate matrix spaces),
and established the generalized finite difference formula
(our \eqref{eq:gen-fin-dif} in the Introduction and Section \ref{subsec:Lagrange})
as well as the properties of $\Delta_R f(X,Y)$ as a function of $X$ and $Y$ (our Section \ref{subsec:proper}).
He also introduced explicitly the corresponding space ${\mathfrak D}(U,U)$ of (in our terminology) first order
nc functions on $U \times U$ --- again, with an analyticity assumption
when restricted to each (pair) of matrix sizes, and showed the canonical isomorphism
${\mathfrak D}(U) \widehat{\otimes} {\mathfrak D}(U) \cong {\mathfrak D}(U,U)$
(compare our Remark \ref{rem:tensor_prod});
here all the function spaces are given the topology of uniform convergence on compacta in every matrix size,
and $\widehat{\otimes}$ denotes the completed projective tensor product.

Taylor also considered in \cite[Section 6]{T1} the algebra of nc power series with a given multiradius
of convergence. He noticed that it is not a nuclear Frechet space, whereas ${\mathfrak D}(U)$ is always nuclear,
implying (in our terminology and without an explicit example, compare our Example \ref{ex:unif-normal_neq_unif})
that not every analytic nc function on a nc polydisc is uniformly analytic.

In \cite{T2}, Taylor considered quite generally a locally convex topological algebra $A$
together with a unital algebra homomorphism from
the free associative algebra ${\mathbb C}\langle z_1,\ldots,z_n \rangle$ to $A$,
equipped with continuous linear maps
$\Delta_i \colon A \to A \widehat{\otimes} A$\footnote{
To be precise, Taylor assumes that the multiplication on $A$ is separately continuous and
looks at the completed inductive tensor product; this makes no difference if $A$ is a Fr\' echet space.
}, $i=1,\ldots,n$ --- playing the role of partial nc difference-differential
operators --- that satisfy the Leibnitz rule (compare our Section \ref{subsub:prod})
and a version of the generalized finite difference formula.
Taylor called ${\mathbb C}\langle z_1,\ldots,z_n \rangle \to A$ a localization
(he introduced this notion in \cite{T1} in a more general context from certain homological considerations).
By considering representations of $A$, he showed that each element of $A$ yields a function on a class of $n$-tuples
of operators on Banach spaces that respects intertwining, much like nc functions except that
infinite-dimensional Banach spaces may be allowed.
With respect to such an infinite-dimensional setting, we notice also the work of
Hadwin \cite{Ha78} and Hadwin--Kaonga--Mathes \cite{HaKaMa03}
for infinite-dimensional Hilbert spaces and functions respecting orthogonal direct sums and unitary similarity,
and the work of Muhly--Solel \cite{MS,MS3}
for representations of a $W^*$-correspondence.

Taylor showed in \cite[Section 4]{T2} that for any localization ${\mathbb C}\langle z_1,\ldots,z_n \rangle \to A$,
$(1 \otimes \Delta) \circ \Delta = (\Delta \otimes 1) \circ \Delta$, where
$\Delta \colon A \to \bigoplus_{i=1}^n A \widehat{\otimes} A$ with the coordinates $\Delta_i$
(in particular, each $\Delta_i$ is a coassociative comultiplication),
compare our Section \ref{subsec:integra}.
He used this to define $\Delta^k \colon A \to \bigoplus_{w\in\free_n,\,|w|=k} A^{\widehat{\otimes}k}$
(here we use our notation for the free monoid, see Section \ref{subsec:ncpoly}) iteratively,
compare our formulae \eqref{eq:delta_tensor_prod_k} and \eqref{eq:j-delta_tensor_prod_k}.
He then derived (in our terminology) a finite TT expansion as in our Theorem \ref{thm:TT},
and showed that it converges under some additional conditions on the localization
${\mathbb C}\langle z_1,\ldots,z_n \rangle \to A$
(essentially, refined continuity requirements for $\Delta_i$)
that he referred to as a free analytic function algebra on $n$ generators
and that he verified for the algebra of nc power series with a given multiradius
of convergence.

In this connection, Taylor also considered
nc power series with coefficients in a Banach algebra
(in the case of a matrix algebra
these are our nc power series \eqref{eq:series} and\eqref{eq:power-series},
though Taylor never considers their evaluations)
and posed the question of when they form a free analytic function algebra;
while he was unable to answer this question, he did derive a version
of our conditions
\eqref{eq:LAC_0}--\eqref{eq:LAC_ell_ell} and \eqref{eq:ncfun_coef_0}--\eqref{eq:ncfun-coef_ell_ell}.
Among other noteworthy results in \cite{T2},
we mention a version of the implicit function theorem.

Voiculescu introduced in \cite{Voi04} (which is a continuation of \cite{Voi00})
what he called fully matricial sets and fully matricial functions.
A fully matricial $G$-set $\Omega$, for $G$ a Banach space, is in our terminology a nc set in the nc space over $G$
that is similarity invariant and contains all direct summands of matrices from $\Omega$
(so that $\Omega$ coincides with its direct summands extension as in Chapter \ref{sec:dirsum-ext}).
A fully matricial $H$-valued function on $\Omega$, for $H$ another Banach space,
is in our terminology a nc function on $\Omega$ with values in the nc space over $H$.
Voiculescu called such a function continuous or analytic if for all $n$, $f|_{\Omega_n}$ is a continuous or analytic function
on $\Omega_n \subseteq \mat{G}{n}$ (assumed to be open) with values in $\mat{H}{n}$.
He introduced $p$-variable fully matricial analytic functions ---
in our terminology, these are nc functions of order $p-1$
that are assumed to be analytic when restricted to fixed matrix sizes,
and established the existence of
(in our terminology, see Sections \ref{subsec:dir_difdif} and \ref{subsec:dir_difdif-k})
directional nc difference-differential operator, that he denoted by $\partial$,
from scalar fully matricial analytic functions to $2$-variable scalar fully matricial analytic functions
\footnote{
To be more precise, Voiculescu considered the case where $G$ is an operator system
hence has a distinguished vector $1_G$, and looked at the nc difference-differential operator in the direction $1_G$.
};
the analyticity assumption is used essentially in the proof.
Voiculescu proved that $\partial$ satisfies the Leibnitz rule (see our Section \ref{subsub:prod})
and a coassociativity property (see our Section \ref{subsec:integra}),
so that the algebra of scalar fully matricial analytic functions becomes a topological version
of what he called a generalized difference quotient ring
(an infinitesimal bialgebra in the terminology of \cite{Rota,Agui}).
One of the main results of \cite{Voi04} is a duality transform relating
the generalized difference quotient ring $B \langle Y \rangle$ generated freely by an element $Y$
over a Banach algebra $B$ with a unit
and the generalized difference quotient ring of scalar fully matricial analytic functions
on what Voiculescu called the full $B$-resolvent set of $Y$
(which is
$\coprod_{n=1}^\infty \left\{
b \in \mat{B}{n} \colon \bigoplus_{\alpha=1}^n Y - b \text{ is invertible in } \mat{E}{n}\right\}$,
where $E$ is a Banach algebra containing $B$ and $Y$ and some additional technical assumptions are satisfied).

In \cite[Section 11]{Voi09}, Voiculescu introduced what he called stably matricial sets
and stably matricial functions by relaxing the similarity invariance property for a fully matricial set
(so that, using our terminology, nc balls in the nc space over an operator space and naturally defined
nc half planes in the nc space over an operator system are included),
and showed that an analytic stably matricial function naturally extends to an analytic fully matricial function
(compare our canonical extension of a nc function to the similarity invariant envelope of the original domain
of definition in the Appendix, and the direct summands extension in Chapter \ref{sec:dirsum-ext}).
We mention also an inverse function theorem in \cite[Section 11.5]{Voi09}.
In \cite[Section 13]{Voi09}, Voiculescu established that an analytic fully matricial function on a fully matricial set
containing the origin admits,
in our terminology, a nc power series expansion at the origin
(see our \eqref{eq:tt-series-gen'} in Theorem \ref{thm:tt-series-gen},
except for the difference of notation for nc power series and the fact that Voiculescu does not identify
the multlinear forms appearing as coefficients
as the values at the origin of the higher order nc difference-differential operator applied to the function).
It is interesting to notice that unlike Taylor's approach in \cite{T2} mentioned above
and our approach in Chapters \ref{sec:TT} and \ref{sec:conv}, Voiculescu does not obtain the nc power series expansion
via the nc difference-differential calculus;
he rather starts with the usual power series expansion of $f|_{\Omega_n}$ at $0_{n \times n}$ for every matrix size $n$,
and uses the fact that $f$ respects direct sums and similarities to deduce that for all $N$
(using our notation, compare Remark \ref{rem:polylin})
$\delta^N(f|_{\Omega_n})(0_{n \times n})(Z) = Z^{\odot N} \alpha_N$ for some
bounded $N$-linear mapping $\alpha_N$.
In \cite[Sections 14--16]{Voi09}, Voiculescu combines the nc power series expansion with asymptotic freeness results
for random matrices to study boundary values for (using our terminology)
bounded nc functions on the open nc unit ball in $\ncspaced{{\mathbb C}}{k \times k}$
with the operator space structure $\|X\|_n = \|X\|$ for
$X \in \mat{\left(\mat{{\mathbb C}}{k}\right)}{n} \cong \mat{{\mathbb C}}{nk}$
and in $\ncspaced{{\mathbb C}}{k}$
with the operator space structure $\|X\|_n = \max\left\{\|X_1\|,\ldots,\|X_k\|\right\}$ for
$X=(X_1,\ldots,X_k) \in \mattuple{{\mathbb C}}{n}{k}$
(here $\|\cdot\|$ denotes the usual Euclidean norm of a complex matrix),
and to obtain asymptotic integral formulae for the coefficients
in terms of integration over the corresponding distinguished boundary.

The main thrust of \cite{Voi09} is to develop
nc function theory on the Grassmanian completion of (using our terminology) the nc space over a Banach algebra $B$,
which is to play the role of the Riemann sphere completion of the complex plane in the case $B={\mathbb C}$.

We finish this review with the mention of the two papers
\cite{Sch} (see also \cite{SchZa}) and \cite{Nie} dealing with functions of a single complex or real nc variable.
More precisely, Schanuel considered in \cite{Sch} a family $f$ of functions $f_A \colon A \to A$,
where $A$ runs over all finite-dimensional $\field$-algebras with unit in the case $\field={\mathbb C}$
and over all finite-dimensional $\field$-algebras with unit that are embedded in upper triangular
matrices in the case $\field={\mathbb R}$,
and the functions $f_A$ respect algebra homomorphisms.
Schanuel called $f$ a global function;
$f$ is essentially the same as (in our terminology)
a nc function from $\ncspace{{\mathbb C}}$ to $\ncspace{{\mathbb C}}$
in the case $\field={\mathbb C}$,
and as a nc function from $\Lambda$ to $\ncspace{{\mathbb R}}$
in the case $\field={\mathbb R}$,
where $\Lambda \subseteq \ncspace{{\mathbb R}}$ is the nc set consisting of matrices with only real eigenvalues
(notice that $\Lambda$ is not an open nc set).
Schanuel remarked that every entire function in the case $\field={\mathbb C}$ and
every $C^\infty({\mathbb R})$ function in the case $\field={\mathbb R}$ extend naturally to a global function.
He proved that if
\begin{itemize}
\item
either for every $A$ and every $\lambda \in \field$, $f_A$ is bounded on an open neighbourhood of $\lambda$ in $A$,
\item
or there exists an integer $n \geq 0$ so that for every $A$ and every $a,b \in A$ the function
$t \mapsto t^n f_A(a+tb)$ from $\field$ to $A$ is bounded in a neignhourhood of $0$,
\end{itemize}
then for every $A$, $f_A \colon A \to A$ is an analytic function in the case $\field={\mathbb C}$
and a $C^\infty$ function in the case $\field={\mathbb R}$,
and the global function $f$ is obtained from $f_{\field}$ by the natural extension
(in the case $\field={\mathbb C}$, compare our Section \ref{subsec:analytic} ---
notice that for nc functions on $\ncspace{\mathbb C}$ local boundedness at scalar points
implies local boundedness everywhere because of the decomposition \eqref{eq:1var_decomp}).
The proof uses evaluation on upper triangular matrices and
a formal power series expansion for a global function in terms of a sequence of scalar functions
(see our \eqref{eq:1var_fun}--\eqref{eq:1var_ps} --- notice that the desired conclusion is that
$f_0$ is an entire function in the case $\field={\mathbb C}$
and a $C^\infty$ function in the case $\field={\mathbb R}$, and $f_\ell'=f_{\ell+1}$ for $\ell=0,1,\ldots$).

Niemiec \cite{Nie} considered the sets that he denoted by ${\mathcal D}_k(\Omega)$ ---
of $k \times k$ diagonalizable matrices with all the eigenvalues in $\Omega$ ---
and by ${\mathcal M}_k(\Omega)$ ---
of all $k \times k$ matrices with all the eigenvalues in $\Omega$;
here $\Omega \subseteq {\mathbb R}$ or $\Omega \subseteq {\mathbb C}$
and the matrices are over ${\mathbb R}$ or over ${\mathbb C}$, respectively.
For any function $f \colon \Omega \to {\mathbb C}$, there is a natural extension
to (in our terminolgy) nc function from the nc set
$\coprod_{k=1}^\infty {\mathcal D}_k(\Omega)$ to $\ncspace{\mathbb C}$
(here we view ${\mathbb C}$ as a ${\mathbb R}$-vector space in the case $\Omega \subseteq {\mathbb R}$)
that Niemiec denoted by $f_{\text{op}}$.
He showed that for an open set $\Omega \subseteq {\mathbb C}$ and for any $k \geq 2$,
$f_{\text{op}}$ is continuous on ${\mathcal D}_k(\Omega)$ if and only if
$f_{\text{op}}$ extends to to a continuous function from ${\mathcal M}_k(\Omega)$ to ${\mathcal M}_k({\mathbb C})$
if and only if $f$ is analytic.
He also showed that for any set $\Omega \subseteq {\mathbb C}$,
$f_{\text{op}}$ is bounded on the intersection of some nc ball around $\lambda$ (in our terminology)
  with $\coprod_{k=1}^\infty {\mathcal D}_k(\Omega)$
for all $\lambda \in \Omega$ if and only if
for all $\lambda \in \Omega$ and all $\epsilon>0$
there exists a $\delta>0$ such that
$\|f_{\text{op}}(X)-I_k f(\lambda)\|<\epsilon$ for all $k=1,2,\ldots$ and all
$X \in {\mathcal D}_k(\Omega)$: $\|X-I_k \lambda\|<\delta$, if and only if
$f$ extends to an analytic function on an open set in ${\mathbb C}$ containing $\Omega$.
For both of these results, compare our Section \ref{subsec:analytic}.
Niemiec also proved
that for an open set (or an interval) $\Omega \subseteq {\mathbb R}$ and for any $k \geq 3$,
$f_{\text{op}}$ is continuous on ${\mathcal D}_k(\Omega)$ if and only if
$f$ is of class $C^{k-2}$ and $f^{(k-2)}$ is locally Lipschitz,
while $f_{\text{op}}$ extends to to a continuous function from ${\mathcal M}_k(\Omega)$ to ${\mathcal M}_k({\mathbb C})$
if and only if $f$ is of class $C^{k-1}$.
The proofs use Taylor's formula with remainder and the calculus of divided differences.

\appendix
\chapter[Similarity invariant envelopes and extension of
nc functions]{Similarity invariant envelopes and extension of nc
functions} \label{app}

\centerline{\Large (by Shibananda Biswas\footnote{
Department of Mathematics,
Ben-Gurion University of the Negev,
 Beer-Sheva, Israel, 84105;
current address: Department of Mathematics and Statistics,
Indian Institute of Science Education and Research Kolkata,
Mohanpur -- 741246, Nadia, West Bengal, India; e-mail address: shibananda@gmail.com}, Dmitry S.
Kaliuzhnyi-Verbovetskyi,}

\centerline{\Large and Victor Vinnikov)}
\bigskip

\begin{prop}\label{prop:sim-env}
Let $\module{M}$ be a module over a commutative unital ring
$\ring$, and let $\Omega\subseteq\ncspace{\module{M}}$ be a nc
set. Then
\begin{equation*}\label{eq:sim-inv} \widetilde{\Omega}:=\{SXS^{-1}\colon X\in\Omega_n, S\in
\mat{\ring}{n}\ {\rm invertible},n\in\mathbb{N}\}
\end{equation*}
is a nc set. \index{$\widetilde{\Omega}$}
\end{prop}
Clearly, $\widetilde{\Omega}$ is the smallest nc set that contains
$\Omega$ and that is invariant under similarities. We shall call
$\widetilde{\Omega}$ the \emph{similarity invariant envelope of
the nc set $\Omega$}. \index{similarity invariant envelope}
\begin{proof}[Proof of Proposition \ref{prop:sim-env}]
Let $X\in\Omega_n$, $Y\in\Omega_m$, and let $S\in \mat{\ring}{n}$,
$T\in\mat{\ring}{m}$ be invertible, $n,m\in\mathbb{N}$. Then
$S\oplus T\in \mat{\ring}{(n+m)}$ is invertible, $X\oplus
Y\in\Omega_{n+m}$, and
$$SXS^{-1}\oplus TYT^{-1}=(S\oplus T)(X\oplus Y)(S\oplus T)^{-1}\in\widetilde{\Omega}_{n+m}.$$
\end{proof}
\begin{prop}\label{prop:adm-env}
If $\Omega\subseteq\ncspace{\vecspace{M}}$ is a right (resp.,
left) admissible nc set, then so is its similarity invariant
envelope $\widetilde{\Omega}$. Moreover, for any
$\widetilde{X}\in\widetilde{\Omega}_n$,
$\widetilde{Y}\in\widetilde{\Omega}_m$ and
$Z\in\rmat{\module{M}}{n}{m}$, one has $$\begin{bmatrix} \widetilde{X} & Z\\
0 & \widetilde{Y}
\end{bmatrix}\in\widetilde{\Omega}_{n+m},$$
and, respectively, for any $\widetilde{X}\in\widetilde{\Omega}_n$,
$\widetilde{Y}\in\widetilde{\Omega}_m$ and
$Z\in\rmat{\module{M}}{m}{n}$, one has $$\begin{bmatrix} \widetilde{X} & 0\\
Z & \widetilde{Y}
\end{bmatrix}\in\widetilde{\Omega}_{n+m},$$
\end{prop}
\begin{proof}
We shall give the proof for a right admissible nc set $\Omega$. A
similar argument works for a left admissible nc set $\Omega$. Let
$\widetilde{X}\in\widetilde{\Omega}_n$,
$\widetilde{Y}\in\widetilde{\Omega}_m$ and
$Z\in\rmat{\module{M}}{n}{m}$. Then there exist $X\in\Omega_n$,
$Y\in\Omega_m$, and invertible $S\in \mat{\ring}{n}$, $T\in
\mat{\ring}{m}$, and $r\in\ring$ such that
$\widetilde{X}=SXS^{-1}$, $\widetilde{Y}=TYT^{-1}$, and
$$
\begin{bmatrix}
X & rS^{-1}ZT\\
0 & Y
\end{bmatrix}\in\Omega_{n+m}.$$
Then $$\begin{bmatrix}
\widetilde{X} & Z\\
0 & \widetilde{Y}
\end{bmatrix}=\begin{bmatrix}
r^{-1}S & 0\\
0 & T
\end{bmatrix}\begin{bmatrix}
X & rS^{-1}ZT\\
0 & Y
\end{bmatrix}\begin{bmatrix}
r^{-1}S & 0\\
0 & T
\end{bmatrix}^{-1}\\
\in\widetilde{\Omega}_{n+m}.$$
\end{proof}

\begin{prop}\label{prop:ncfun-ext}
Let $\module{M}$, $\module{N}$ be modules over a commutative
unital ring $\ring$, let $\Omega\subseteq\ncspace{\module{M}}$ be
a nc set, and let
 $f\colon\Omega\to\ncspace{\module{N}}$
be a nc function. Then there exists a unique nc function
$\widetilde{f}\colon\widetilde{\Omega}\to\ncspace{\module{N}}$
such that $\widetilde{f}\,|_{\Omega}=f$.
\end{prop}
\begin{proof}
For any $n\in\mathbb{N}$ and
$\widetilde{X}\in\widetilde{\Omega}_n$, there exist $X\in\Omega_n$
and an invertible $S\in \mat{\ring}{n}$ such that
$\widetilde{X}=SXS^{-1}$, and we set
$\widetilde{f}(\widetilde{X})=Sf(X)S^{-1}$. This definition is
correct, since $SXS^{-1}=TYT^{-1}$ for $X,Y\in\Omega_n$ and for
invertible $S,T\in \mat{\ring}{n}$ implies that
$Y=(T^{-1}S)X(T^{-1}S)^{-1}$, hence
$f(Y)=(T^{-1}S)f(X)(T^{-1}S)^{-1}$ and $Tf(Y)T^{-1}=Sf(X)S^{-1}$.
It is straightforward to check that
$\widetilde{f}\colon\widetilde{\Omega}\to\ncspace{\module{N}}$ is
a nc function and that $\widetilde{f}\,|_{\Omega}=f$. If
$\widetilde{g}\colon\widetilde{\Omega}\to\ncspace{\module{N}}$ is
another nc function such that $\widetilde{g}\,|_{\Omega}=f$, then
for any $X\in\Omega_n$ and invertible $S\in \mat{\ring}{n}$ we
have
$$\widetilde{g}(SXS^{-1})=S\widetilde{g}(X)S^{-1}=Sf(X)S^{-1}=\widetilde{f}(SXS^{-1}),$$ i.e., we have necessarily
$\widetilde{g}=\widetilde{f}$.
\end{proof}
We will call the nc function
$\widetilde{f}\colon\widetilde{\Omega}\to\ncspace{\module{N}}$
from Proposition \ref{prop:ncfun-ext} the \emph{canonical
extension of the nc function \index{$\widetilde{f}$}
$f\colon\Omega\to\ncspace{\module{N}}$}. \index{canonical
extension of a nc function}

\begin{ex}
{\rm Let $\vecspace{V}=\vecspace{W}=\mathbb{C}$, with the
canonical operator space structure. The nc unit ball
$\Omega=B_{\rm nc}(0_{1\times 1},1)$ is a nc set. Then its
similarity invariant envelope is
$$\widetilde{\Omega}=\{Z\in\ncspace{\mathbb{C}}\colon \specrad(Z)<1\}.$$
Indeed, let $Z\in\mat{\mathbb{C}}{n}$ be such that
$\specrad(Z)<1$. We need to show that $Z=SXS^{-1}$ for some $X\in
\mat{\mathbb{C}}{n}$ with $\|X\|<1$  and some invertible
$S\in\mat{\mathbb{C}}{n}$. It suffices to show this for $Z$ a
Jordan cell, i.e., for $$Z=\begin{bmatrix}
\lambda & 1       & & \\
        & \ddots & \ddots & \\
        &         & \ddots & 1 \\
        &         &        & \lambda
\end{bmatrix}$$
with $|\lambda|<1$. Let $\epsilon>0$ be such that $\|X\|<1$, where
$$X=\begin{bmatrix}
\lambda & \epsilon       & & \\
        & \ddots & \ddots & \\
        &         & \ddots & \epsilon \\
        &         &        & \lambda
\end{bmatrix}.$$
Then $Z=SXS^{-1}$ with
$S=\diag(\epsilon^{-1},\ldots,\epsilon^{-n})$.

Clearly the nc function $f\colon\Omega\to\ncspace{\vecspace{W}}$
defined by $f(Z)=(I-Z)^{-1}$ can be canonically extended to a nc
function on $\widetilde{\Omega}$ using the same formula.
 }
\end{ex}

Proposition \ref{prop:ncfun-ext} can be extended to higher order
nc functions as follows.

\begin{prop}\label{prop:higher-ncfun-ext}
Let $\module{M}_0$, \ldots, $\module{M}_k$, $\module{N}_0$,
\ldots, $\module{N}_k$ be modules over a commutative unital ring
$\ring$, let $\Omega^{(0)}\subseteq\ncspacej{\module{M}}{0}$,
\ldots, $\Omega^{(k)}\subseteq\ncspacej{\module{M}}{k}$ be  nc
sets, and let \index{$\widetilde{f}$}
$f\in\tclass{k}(\Omega^{(0)},\ldots,\Omega^{(k)};\ncspacej{\module{N}}{0},\ldots,\ncspacej{\module{N}}{k})$.
Then there exists a unique nc function of order $k$,
$\widetilde{f}\in\tclass{k}(\widetilde{\Omega}^{(0)},\ldots,\widetilde{\Omega}^{(k)};
\ncspacej{\module{N}}{0},\ldots,\ncspacej{\module{N}}{k})$, such
that
$\widetilde{f}|_{\Omega^{(0)}\times\cdots\times\Omega^{(k)}}=f$.
\end{prop}
\begin{proof}
For any $n_0,\ldots,n_k\in\mathbb{N}$ and
$\widetilde{X}^j\in\widetilde{\Omega}^{(j)}_{n_j}$, there exist
$X^j\in\Omega^{(j)}_{n_j}$ and invertible $S_j\in\mat{\ring}{n_j}$
such that $\widetilde{X}^j=S_jX^jS_j^{-1}$, $j=0$, \ldots, $k$. We
set, for any $Z^j\in\rmat{\module{N}_j}{n_{j-1}}{n_j}$, $j=1$,
\ldots, $k$,
\begin{equation*}
\widetilde{f}(\widetilde{X}^0,\ldots,\widetilde{X}^k)(Z^1,\ldots,Z^k)=S_0f(X^0,\ldots,X^k)(S_0^{-1}Z^1S_1,\ldots,
S_{k-1}^{-1}Z^kS_k)S_k^{-1}.
\end{equation*}
This definition is correct, since the equalities
$S_jX^jS_j^{-1}=T_jY^jT_j^{-1}$ for $X^j,Y^j\in\Omega^{(j)}_{n_j}$
and for invertible $S_j,T_j\in\mat{\ring}{n_{j}}$ imply
$Y^j=(T_j^{-1}S_j)X^j(T_j^{-1}S_j)^{-1}$, hence  by
\eqref{eq:sim_k}
\begin{multline*}
f(Y^0,\ldots,Y^k)(T_0^{-1}S_0Z^1S_1^{-1}T_1,\ldots,
T_{k-1}^{-1}S_{k-1}Z^kS_k^{-1}T_k)\\
=T_0^{-1}S_0f(X^0,\ldots,X^k)(Z^1,\ldots,Z^k)S_k^{-1}T_k,
\end{multline*}
 and for
$W_j=S_{j-1}Z^jS_j^{-1}$ we have
\begin{multline*}
T_0f(Y^0,\ldots,Y^k)(T_0^{-1}W^1T_1,\ldots,
T_{k-1}^{-1}W^kT_k)T_k^{-1}\\
=S_0f(X^0,\ldots,X^k)(S_0^{-1}W^jS_1,\ldots,S_{k-1}^{-1}W^1S_k)S_k^{-1}.
\end{multline*}
It is obvious that $\widetilde{f}$ respects similarities and that
$\widetilde{f}\,|_{\Omega^{(0)}\times\cdots\times\Omega^{(k)}}=f$.
We also have
\begin{multline*}
\widetilde{f}(\widetilde{X}^{0\prime}\oplus
\widetilde{X}^{0\prime\prime},\ldots,\widetilde{X}^{k\prime}\oplus
\widetilde{X}^{k\prime\prime})\left(\begin{bmatrix}
Z^{1\prime,\prime} &
Z^{1\prime,\prime\prime}\\
Z^{1\prime\prime,\prime} & Z^{1\prime\prime,\prime\prime}
\end{bmatrix},\ldots,\begin{bmatrix} Z^{k\prime,\prime} &
Z^{k\prime,\prime\prime}\\
Z^{k\prime\prime,\prime} & Z^{k\prime\prime,\prime\prime}
\end{bmatrix}\right)\\
=(S_0^{\prime}\oplus S_0^{\prime\prime})f({X}^{0\prime}\oplus
{X}^{0\prime\prime},\ldots,{X}^{k\prime}\oplus
{X}^{k\prime\prime})\left(\begin{bmatrix}
(S_0^\prime)^{-1}Z^{1\prime,\prime}S_1^{\prime} &
(S_0^\prime)^{-1}Z^{1\prime,\prime\prime}S_1^{\prime\prime}\\
(S_0^{\prime\prime})^{-1}Z^{1\prime\prime,\prime}S_1^\prime &
(S_0^{\prime\prime})^{-1}Z^{1\prime\prime,\prime\prime}S_1^{\prime\prime}
\end{bmatrix},\right.\\
\left.\ldots,\begin{bmatrix}
(S_{k-1}^\prime)^{-1}Z^{k\prime,\prime}S_k^\prime &
(S_{k-1}^\prime)^{-1}Z^{k\prime,\prime\prime}S_k^{\prime\prime}\\
(S_{k-1}^{\prime\prime})^{-1}Z^{k\prime\prime,\prime}S_k^\prime &
(S_k^{\prime\prime})^{-1}Z^{k\prime\prime,\prime\prime}S_k^{\prime\prime}
\end{bmatrix}\right)(S_k^\prime\oplus S_k^{\prime\prime})^{-1}\\
=\begin{bmatrix} S_0^\prime f^{\prime,\prime}(S_k^\prime)^{-1} &
S_0^\prime f^{\prime,\prime\prime}(S_k^{\prime\prime})^{-1}\\
S_0^{\prime\prime}f^{\prime\prime,\prime}(S_k^\prime)^{-1} &
S_0^{\prime\prime}f^{\prime\prime,\prime\prime}(S_k^{\prime\prime})^{-1}
\end{bmatrix},
\end{multline*}
where according to \eqref{eq:dirsums_k} and \eqref{eq:f_entries},
for $\alpha,\beta\in\{^\prime,^{\prime\prime}\}$,
\begin{multline*}
f^{\alpha,\beta}=\sum_{
  \alpha_0,\ldots,\alpha_k\in \{^\prime,^{\prime\prime}\}\colon
  \alpha_0=\alpha, \alpha_k=\beta}
f(X^{0\alpha_0},\ldots,X^{k\alpha_k})\\
\left((S_0^{\alpha_0})^{-1}Z^{1\alpha_0,\alpha_1}S_1^{\alpha_1},\ldots,(S_{k-1}^{\alpha_{k-1}})^{-1}
Z^{k\alpha_{k-1},\alpha_k}S_k^{\alpha_k}\right).
\end{multline*}
Here $n_j^\prime\in\mathbb{N}$,
$n_j^{\prime\prime}\in\mathbb{Z}_+$,
$X^{j\alpha}\in\Omega^{(j)}_{n_j^\alpha}$,
$\widetilde{X}^{j\alpha}\in\widetilde{\Omega}^{(j)}_{n_j^\alpha}$
for $j=0,\ldots,k$ are such that
$\widetilde{X}^{j\alpha}=S_j^{\alpha}X^{j\alpha}(S_j^{\alpha})^{-1}$
for an invertible $S_j\in\mat{\ring}{n_j^\alpha}$,
$\alpha\in\{^\prime,^{\prime\prime}\}$;
$Z^{j\alpha,\beta}\in\rmat{\module{N}_j}{n_{j-1}^\alpha}{n_j^\beta}$
for $j=1,\ldots,k$, $\alpha,\beta\in\{^\prime,^{\prime\prime}\}$,
the block entry $f^{\alpha,\beta}$ is void if either $n_0^\alpha$
or $n_k^\beta$ is $0$, and a summand in the right-hand side sum is
$0$ if at least one of $n_j^{\alpha_j}$, $j=1,\ldots,k-1$, is $0$.
Therefore,
\begin{multline*}
\widetilde{f}(\widetilde{X}^{0\prime}\oplus
\widetilde{X}^{0\prime\prime},\ldots,\widetilde{X}^{k\prime}\oplus
\widetilde{X}^{k\prime\prime})\left(\begin{bmatrix}
Z^{1\prime,\prime} &
Z^{1\prime,\prime\prime}\\
Z^{1\prime\prime,\prime} & Z^{1\prime\prime,\prime\prime}
\end{bmatrix},\ldots,\begin{bmatrix} Z^{k\prime,\prime} &
Z^{k\prime,\prime\prime}\\
Z^{k\prime\prime,\prime} & Z^{k\prime\prime,\prime\prime}
\end{bmatrix}\right)\\
=\begin{bmatrix} \widetilde{f}^{\prime,\prime} &
 \widetilde{f}^{\prime,\prime\prime}\\
\widetilde{f}^{\prime\prime,\prime} &
\widetilde{f}^{\prime\prime,\prime\prime}
\end{bmatrix},
\end{multline*}
with
\begin{multline*}
\widetilde{f}^{\alpha,\beta}=S_0^\alpha
f^{\alpha,\beta}(S_k^{\beta})^{-1}=\sum_{
  \alpha_0,\ldots,\alpha_k\in \{^\prime,^{\prime\prime}\}\colon
  \alpha_0=\alpha, \alpha_k=\beta}
S_0^\alpha f\left(X^{0\alpha_0},\ldots,X^{k\alpha_k}\right)\\
\hfill\left((S_0^{\alpha_0})^{-1}Z^{1\alpha_0,\alpha_1}S_1^{\alpha_1},\ldots,(S_{k-1}^{\alpha_{k-1}})^{-1}
Z^{k\alpha_{k-1},\alpha_k}S_k^{\alpha_k}\right)(S_k^\beta)^{-1}\\
=\sum_{
  \alpha_0,\ldots,\alpha_k\in \{^\prime,^{\prime\prime}\}\colon
  \alpha_0=\alpha, \alpha_k=\beta}
\widetilde{f}(\widetilde{X}^{0\alpha_0},\ldots,\widetilde{X}^{k\alpha_k})
(Z^{1\alpha_0,\alpha_1},\ldots, Z^{k\alpha_{k-1},\alpha_k}).
\end{multline*}
Thus, $\widetilde{f}$ respects direct sums. Next, if
$\widetilde{g}\in\tclass{k}(\widetilde{\Omega}^{(0)},\ldots,\Omega^{(k)};
\ncspacej{\module{N}}{0},\ldots,\ncspacej{\module{N}}{k})$ is
another nc function of order $k$ such that
$\widetilde{g}\,|_{\Omega^{(0)}\times\cdots\times\Omega^{(k)}}=f$,
then for any $X^j\in\Omega^{(j)}_{n_j}$ and invertible
$S_j\in\mat{\ring}{n_j}$, $j=0,\ldots,k$, we have
\begin{multline*}
\widetilde{g}(S_0X^0S_0^{-1},\ldots,S_kX^kS_k^{-1})(S_0Z^1S_1^{-1},\ldots,S_{k-1}Z^kS_k^{-1})\\
=S_0\widetilde{g}(X^0,\ldots,X^k)(Z^1,\ldots,Z^k)S_k^{-1}
=S_0f(X^0,\ldots,X^k)(Z^1,\ldots,Z^k)S_k^{-1}\\
=\widetilde{f}(S_0X^0S_0^{-1},\ldots,S_kX^kS_k^{-1})(S_0Z^1S_1^{-1},\ldots,S_{k-1}Z^kS_k^{-1}),
\end{multline*}
i.e., we have necessarily $\widetilde{g}=\widetilde{f}$.
\end{proof}
We will call the nc function
$\widetilde{f}\in\tclass{k}(\widetilde{\Omega}^{(0)},\ldots,\widetilde{\Omega}^{(k)};
\ncspacej{\module{N}}{0},\ldots,\ncspacej{\module{N}}{k})$ the
\emph{canonical extension of the nc function
${f}\in\tclass{k}({\Omega^{(0)}},\ldots,\Omega^{(k)};\ncspacej{\module{N}}{0},\ldots,\ncspacej{\module{N}}{k})$}.
\index{canonical extension of a higher order nc function}

\backmatter

\printindex


\begin{thebibliography}{99}

\bibitem{AKV}
G. Abduvalieva and D. S. Kaliuzhnyi-Verbovetskyi. Fixed point
theorems for noncommutative functions. {\em J. Math. Anal. Appl.}
401 (2013), no. 1, 436--446.

\bibitem{AA}
V. M.~Adamjan  and D. Z.~Arov.
On unitary coupling of semiunitary operators. Russian, {\em Mat. Issled.} \textbf{1} (1966), vyp. 2, 3--64.
English translation:  {\em Amer. Math. Soc. Transl. (2)}
(95) (1970), 75--129.

\bibitem{Agui}
M. Aguiar.
Infinitesimal Hopf algebras. {\em New trends in Hopf algebra theory (La Falda, 1999)}, 1-–29,
{\em Contemp. Math.}, 267, Amer. Math. Soc., Providence, RI, 2000.

\bibitem{Ag}
J.~Agler. On the representation of certain holomorphic functions
defined
  on a polydisc. {\em Oper. Theory Adv. Appl.} 48,
  pp.~47--66, Birkh{\" a}user Verlag, Basel, 1990.

\bibitem{AgMcC1}
J. Agler and J. E. McCarthy.
Global holomorphic functions in several non-commuting variables.
Preprint, arXiv:1305.1636.

\bibitem{AgMcC2}
J. Agler and J. E. McCarthy.
Pick Interpolation for free holomorphic functions.
Preprint, arXiv:1308.3730.

\bibitem{AgY}
J. Agler and N. J. Young.
Symmetric functions of two noncommuting variables.
Preprint, arXiv:1307.1588.

\bibitem{AK0}
D. Alpay and D. S. Kaliuzhnyi-Verbovetskyi. On the intersection of
null spaces for matrix substitutions in a non-commutative rational
formal power series. {\em C. R. Math. Acad. Sci. Paris} 339
(2004), no. 8, 533--538.

\bibitem{AK}
D. Alpay and D. S. Kaliuzhnyi-Verbovetskyi.   Matrix-J-unitary
non-commutative rational formal power series.  {\em Oper. Theory
Adv. Appl.} 161, Birkh\"{a}user, Basel, 2006, 49--113.

\bibitem{AT}
C.-G.Ambrozie and D. Timotin. A von Neumann type inequality for
certain domains in $C^n$. {\em Proc. Amer. Math. Soc.}
131:859--869, 2003.

\bibitem{Am66}
S.~A. Amitsur.
\newblock Rational identities and applications to algebra and
geometry.
\newblock {\em J. Algebra}, 3:304--359, 1966.

\bibitem{BBFtH} J. A.~Ball, A.~Biswas, Q.~Fang, and S.~ter Horst.
    Multivariable generalizations of the Schur class:  Positive kernel
    characterization and transfer function realization. {\em Recent
    Advances in Operator Theory and Applicaions}, pp. 17--79,
    OT 187, Birkh\"auser, Basel, 2008.

\bibitem{BB}
J.~A. Ball and V. Bolotnikov. Realization and interpolation for
Schur-Agler-class functions on domains with matrix polynomial
defining function in $C^n$. {\em J. Funct. Anal.} 213:45--87,
2004.

\bibitem{BC}
J.A.~Ball and N.~Cohen.
de Branges--Rovnyak operator models and systems theory: a survey.
{\em Topics in Matrix and Operator Theory} (ed. by H. Bart, I. Gohberg, and M.A. Kaashoek),
pp. 93--136, {\em Operator Theory: Adv. Appl.} 50, Birkh\"auser-Verlag, Boston, 1991.

\bibitem{BGM1}
J.~A. Ball, G.~Groenewald, and T.~Malakorn.
\newblock Structured noncommutative multidimensional linear systems.
\newblock {\em SIAM J. Control Optim.}, 44(4):1474--1528 (electronic), 2005.

\bibitem{BGM3}
J.~A. Ball, G.~Groenewald, and T.~Malakorn.
\newblock Bounded real lemma for structured noncommutative multidimensional
  linear systems and robust control.
\newblock {\em Multidimens. Syst. Signal Process.}, 17(2-3):119--150, 2006.

\bibitem{BGM2}
J.~A. Ball, G.~Groenewald, and T.~Malakorn.
\newblock Conservative structured noncommutative multidimensional linear
  systems.
\newblock In {\em The state space method generalizations and applications},
  volume 161 of {\em Oper. Theory Adv. Appl.}, pages 179--223. Birkh\"auser,
  Basel, 2006.

\bibitem{BK-V}
J.~A. Ball and D.~S. Kaliuzhnyi-Verbovetskyi.
\newblock Conservative dilations of dissipative multidimensional systems: The
  commutative and non-commutative settings.
\newblock {\em Multidim. Syst. Signal Processing}, 19:79--122, 2008.

\bibitem{BT}
J. A. Ball and T. T. Trent. {Unitary colligations, reproducing
kernel {H}ilbert spaces, and
  {N}evanlinna-{P}ick interpolation in several variables}. {\em
J. Funct. Anal.} 157:1--61, 1998.


\bibitem{Cuntz2} J. A.~Ball and V.~Vinnikov, Lax-Phillips scattering
and conservative linear systems: a Cuntz-algebra multidimensional
setting, {\em Memoirs Amer. Math. Soc.} 178, no. 837, 2005.




\bibitem{BPV1}
S. T. Belinschi, M. Popa, and V. Vinnikov. Infinite divisibility
and a non-commutative Boolean-to-free Bercovici-Pata bijection.
{\em J. Funct. Anal.} 262 (2012), no. 1, 94--123.

\bibitem{BPV2}
S. T. Belinschi, M. Popa, and V. Vinnikov. On the operator-valued
analogues of the semicircle, arcsine and Bernoulli laws. {em J. Oper. Theory} 70 (2013),
239–-258.

\bibitem{BeVoi93}
H. Bercovici and D.-V. Voiculescu. {Free convolutions of measures
with unbounded support.}  {\em Indiana Univ. Math. J.}
42:733--773, 1993.

\bibitem{Be70}
G.~M. Bergman.
\newblock Skew fields of noncommutative rational functions, after {A}mitsur.
\newblock In {\em S\' eminaire Sch\" utzenberger--Lentin--Nivat, Ann\' ee
1969/70, No. 16}. Paris, 1970.

\bibitem{BerDicks}
G. M. Bergman and W. Dicks. On Universal Derivations.
{\em J. Algebra} 36 (1975), 193--211.

\bibitem{BR}
J.~Berstel and C.~Reutenauer.
\newblock {\em Rational series and their languages}, volume~12 of {\em EATCS
  Monographs on Theoretical Computer Science}.
\newblock Springer-Verlag, Berlin, 1988.

\bibitem{BJ} O. Bratelli and P. E. T.~Jorgensen, Iterated function
systems and permutation representations of the Cuntz algebra, {\em
Memoirs Amer. Math. Soc.} no. 139, 1999.

\bibitem{Co71a}
P.~M. Cohn.
\newblock The embedding of firs in skew fields.
\newblock {\em Proc. London Math. Soc.}, 23:193--213, 1971.

\bibitem{Co71}
P.~M. Cohn.
\newblock {\em Free rings and their relations}.
\newblock Academic Press, London, 1971.
\newblock London Mathematical Society Monographs, No. 2.

\bibitem{Co72}
P.~M. Cohn.
\newblock Universal skew fields of fractions.
\newblock {\em Symposia Math.}, 8:135--148, 1972.

\bibitem{Co06}
P.~M. Cohn.
\newblock {\em Free ideal rings and localization in general rings}.
\newblock Cambridge University Press, Cambridge, 2006.
\newblock New Mathematical Monographs 3.

\bibitem{DP} K. R.~Davidson and D. R.~Pitts, Invariant subspaces and
hyper-reflexivity for the free semigroup algebras, {\em
Proc.~London Math.~Soc.} 78:401--430, 1999.

\bibitem{DicksLew}
W. Dicks and J. Lewin.
A Jacobian conjecture for free associative algebras.
{\em Comm. Algebra} 10 (1982), no. 12, 1285–-1306.

   \bibitem{DMMcC}  M. A.~Dritschel, S.~Marcantognini, and S.~McCullough.
    Interpolation in semigroupoid algebras. {\em J.~Reine
    Angew.~Math.} 606:1--40, 2007.

   \bibitem{DMcC} M. A.~Dritschel and S.~McCullough.  Test functions,
    kernels, realizations and interpolation. {\em Operator Theory,
    Structured Matrices, and Dilation}, pp. 153--179, Theta Series in
    Advanced Mathematics, Bucharest, 2007.

\bibitem{Dyn}
E. B. Dynkin. Calculation of the coefficients in the
Campbell-Hausdorff formula.   {\em Doklady Akad. Nauk SSSR (N.S.)}
57:323--326, 1947 (Russian).

\bibitem{ER}
E.~G.~Effros and Zh.-J. Ruan. {\em Operator spaces}. London
Mathematical Society Monographs. New Series, \textbf{23}. The
Clarendon Press, Oxford University Press, New York, 2000.



\bibitem{Fliess70}
M.~Fliess.
\newblock Sur le plongement de l'alg\` ebre des s\' eries rationnelles non
  commutatives dans un corps gauche.
\newblock {\em C. R. Acad. Sci. Paris, Ser. A}, 271:926--927, 1970.

\bibitem{Fliess74a}
M.~Fliess.
\newblock Matrices de {H}ankel.
\newblock {\em J. Math. Pures Appl.}, 53(9):197--222, 1974.

\bibitem{Fliess74b}
M.~Fliess.
\newblock Sur divers produits de s\'eries formelles.
\newblock {\em Bull. Soc. Math. France}, 102:181--191, 1974.

\bibitem{Form}
E.~Formanek.
\newblock \emph{The polynomial identities and invariants of $n\times n$ matrices}.
\newblock CBMS Regional Conference Series in Mathematics, \textbf{78}.
 The American Mathematical Society, Providence,
RI, 1991.


\bibitem{GGRetW}
I. Gelfand, S. Gelfand, V. Retakh, and R. L. Wilson.
Quasideterminants. \emph{Adv. Math.} 193(1):56--141, 2005.

\bibitem{GKLLRT}
I. M. Gelfand, D. Krob, A. Lascoux, B. Leclerc, V. S. Retakh,  and
J.-Y. Thibon. Noncommutative symmetric functions. \emph{Adv.
Math.} 112(2):218--348, 1995.

\bibitem{GRet1}
I. M. Gelfand and V. S. Retakh. Determinants of matrices over
noncommutative rings. \emph{Funktsional. Anal. i Prilozhen.})
25(2):13--25, 96, 1991 (Russian);  translation in  {\em Funct.
Anal. Appl.} 25(2):91--102, 1991.

\bibitem{GRet2}
I. M. Gelfand and V. S. Retakh. Theory of noncommutative
determinants, and characteristic functions of graphs.
\emph{Funktsional. Anal. i Prilozhen.} 26(4):1--20, 96, 1992
(Russian); translation in {\em Funct. Anal. Appl.} 26(4):231--246,
1992.

\bibitem{Ha78}
D. W. Hadwin.
Continuous functions of operators; a functional calculus.
{\em Indiana Univ. Math. J.} 27 (1978), no. 1, 113--125.

\bibitem{HaKaMa03}
D. W. Hadwin, L. Kaonga, and B. Mathes.
Noncommutative continuous functions.
{\em J. Korean Math. Soc.} 40 (2003), no. 5, 789–-830.


\bibitem{Helton-scat}
J. W.~Helton.
Discrete time systems, operator models, and scattering theory.
{\em J. Funct. Anal.}
(16) (1974),  15--38.

\bibitem{HSOS} J. W. Helton, ``Positive'' noncommutative polynomials are sums of squares.
 \emph{Ann. of Math.} (2) 156 (2002), no. 2, 675--694.

\bibitem{H03}
J.~W. Helton.
\newblock Manipulating matrix inequalities automatically.
\newblock In {\em Mathematical systems theory in biology, communications,
  computation, and finance (Notre Dame, IN, 2002)}, volume 134 of {\em IMA Vol.
  Math. Appl.}, pages 237--256. Springer, New York, 2003.


\bibitem{HKMcC1} J. W. Helton, I. Klep, and S. McCullough.
Analytic mappings between noncommutative pencil balls.  {\em J.
Math. Anal. Appl.} 376 (2011), no. 2, 407--428.

\bibitem{HKMcC2} J. W. Helton, I. Klep, and S. McCullough.
Proper Analytic Free Maps.  {\em J. Funct. Anal.} 260 (2011), no.
5, 1476--1490.

\bibitem{HKMcC3} J. W. Helton, I. Klep, and S. McCullough.
Free analysis, convexity and LMI domains.
{\em Mathematical methods in systems, optimization and control} (ed. by H. Dym, M. de Oliveira, and M. Putinar),
pp. 195--219,
{\em Operator Theory: Adv. Appl.} 222, Birkhauser, 2012.

\bibitem{HKMcC-freecon}
J. W. Helton, I. Klep, and S. McCullough.
Free Convex Algebraic Geometry.
{\em Semidefinite Optimization and Convex Algebraic Geometry}
(ed. by G. Blekherman, P. Parrilo, and R. Thomas), pp. 341--405, SIAM, 2013.

\bibitem{HKMcCS}
J. W. Helton, I. Klep, S. McCullough, and N. Slinglend. Non-commutative ball maps.
{\em J. Funct. Anal.} 257(2009), 47-–87.

\bibitem{HMcCDeg2} J. W. Helton and S. McCullough.
 Convex noncommutative polynomials have degree two or less.
{\em SIAM J. Matrix Anal. Appl.} 25 (2004), no. 4, 1124--1139.

\bibitem{HMcC-last} J. W. Helton and S. McCullough.
 Free convex sets defined by rational expressions have LMI
representations. Preprint,  arXiv:1209.3412.

\bibitem{HMcCPV}
J.~W. Helton, S.~A. McCullough, M.~Putinar, and V.~Vinnikov.
\newblock Convex Matrix Inequalities versus Linear Matrix Inequalities,
\newblock {\em IEEE Trans. Aut. Control}, 54(5):952--964, 2009.

\bibitem{HMcCV}
J.~W. Helton, S.~A. McCullough, and V.~Vinnikov.
\newblock Noncommutative convexity arises from linear matrix inequalities.
\newblock {\em J. Funct. Anal.}, 240(1):105--191, 2006.

\bibitem{HSw}
R.~G.~Heyneman and M.~E.~Sweedler.
\newblock Affine Hopf
algebras. I.
\newblock {\em J. Algebra},
\textbf{13}:192--241, 1969.

\bibitem{HiPh}
E.~Hille and R.~S.~Phillips.
\newblock {\em Functional analysis and
semi-groups.} Third printing of the revised edition of 1957.
American Mathematical Society Colloquium Publications, Vol. XXXI.
American Mathematical Society, Providence, R. I., 1974.

\bibitem{HJ}
R.~A.~Horn and C.~R.~Johnson. {\em Topics in matrix analysis.}
Corrected reprint of the 1991 original. Cambridge University
Press, Cambridge, 1994.


\bibitem{Hump}
J. E. Humphreys. {\em Introduction to Lie algebras and
representation theory.} Springer-Verlag,
New-York--Heidelberg--Berlin, 1972.


\bibitem{KVV2}
D.~S. Kaliuzhnyi-Verbovetskyi and V.~Vinnikov.
\newblock Singularities of Noncommutative Rational
Functions and Minimal Factorizations.
\newblock {\em Lin. Alg. Appl.} 430:869--889, 2009.

\bibitem{KVV3}
D.~S. Kaliuzhnyi-Verbovetskyi and V.~Vinnikov. Noncommutative
rational functions, their difference-differential calculus and
realizations. {\em Multidimens. Syst. Signal Process.} 23 (2012),
no. 1--2, 49--77.

\bibitem{Kle}
S.~C. Kleene.
\newblock Representation of events in nerve nets and finite automata.
\newblock In {\em Automata studies}, Annals of mathematics studies, no. 34,
  pages 3--41. Princeton University Press, Princeton, N. J., 1956.

\bibitem{Lew}
J. Lewin. A Matrix Representation for Associative Algebras I, II.
{\em Trans. Amer. Math. Soc.} 188 (1977), 293--308, 309--317.

\bibitem{MS1}  P. S.~Muhly and B.~Solel. Hardy algebras,
    $W^{*}$-correspondences and interpolation theory. {\em
    Math.~Ann.} 330:353--415, 2004.

\bibitem{MS} P. S. Muhly and B. Solel. Progress in noncommutative function theory.
 {\em Sci. China Math.} 54 (2011), no. 11, 2275--2294.

\bibitem{MS2} P. S.~Muhly and B.~Solel. Schur class operator functions
    and automorphisms of Hardy algbras. {\em Doc.~Math.~} 13:365--411, 2008.

\bibitem{MS3} P. S. Muhly and B. Solel. Tensorial function theory: From Berezin transforms to Taylor's Taylor
series and back. Integral Equations Operator Theory 76 (2013), no. 4, 463–-508.

\bibitem{Mu}
J.~Mujica. {\em Complex analysis in Banach spaces. Holomorphic
functions and domains of holomorphy in finite and infinite
dimensions.} North-Holland Mathematics Studies, \textbf{120}.
Notas de Matema'tica [Mathematical Notes], \textbf{107}.
North-Holland Publishing Co., Amsterdam, 1986.

\bibitem{N06}
A.~Nemirovskii. Advances in convex optimization: conic
programming. {\em Plenary Lecture,  International Congress of
Mathematicians (ICM)}, Madrid, Spain, 2006.

\bibitem{NN}
Yu. Nesterov and A.~Nemirovskii.
\newblock {\em Interior-point polynomial algorithms in convex programming},
  volume~13 of {\em SIAM Studies in Applied Mathematics}.
\newblock Society for Industrial and Applied Mathematics (SIAM), Philadelphia,
  PA, 1994.

\bibitem{NS}
A. Nica and R. Speicher. \emph{Lectures on combinatorics of free
probability}, London Math. Soc. Lecture Note Series 335, Cambridge
Univ. Press, 2006.

\bibitem{Nie}
P. Niemiec.
Functional calculus for diagonalizable matrices.
{\em Linear Multilinear Algebra} 62 (2014), no. 3, 297-–321.

\bibitem{Pa1}
V.~I.~Paulsen. Representations of function algebras, abstract
operator spaces, and Banach space geometry.
\newblock {\em J. Funct. Anal.}  \textbf{109}(1):113--129, 1992.


\bibitem{Pa}
V.~Paulsen. {\em Completely bounded maps and operator algebras.}
Cambridge Studies in Advanced Mathematics, \textbf{78}. Cambridge
University Press, Cambridge, 2002.

\bibitem{Pi}
G.~Pisier.
\newblock {\em Introduction to operator space theory.}
 London Mathematical Society Lecture Note Series, \textbf{294}. Cambridge
University Press, Cambridge, 2003.

\bibitem{PV}
M. Popa and V. Vinnikov. Non-Commutative Functions and
Non-Commutative Free Levy-Hincin Formula. {\em Adv. Math.} 236 (2013), 131–-157.

\bibitem{Popescu-model1} G.~Popescu, Models for infinite sequences of
noncommuting operators, {\it Acta Sci. Math.} 53:355-368, 1989.



\bibitem{Popescu-model2} G.~Popescu, Characteristic functions for
infinite sequences of noncommuting operators, {\it J. Operator
 Theory} 22:51-71, 1989.



\bibitem{Popescu-CLT1} G.~Popescu, Isometric dilations for infinite
sequences of noncommuting operators, {\it Trans. Amer. Math. Soc.}
316:523-536, 1989. 


\bibitem{Popescu-CLT2} G.~Popescu, Multi-analytic operators on Fock
spaces, {\it Math.~Ann.} 303:31-46, 1995.

\bibitem{Po06} G. Popescu. Free holomorphic functions on the unit ball of $B(H)^n$. {\em J. Funct. Anal.} 241:268–-333,
2006.

\bibitem{Po10} G. Popescu. Free holomorphic automorphisms of the unit ball of $B(H)^n$. {\em J. reine angew.
Math.} 638:119–-168, 2010.

\bibitem{Popescu-OTncdom} G. Popescu.
Operator theory on noncommutative domains. {\em Mem. Amer. Math. Soc.} 205 (2010), no. 964.

\bibitem{Po12} G. Popescu.
Free biholomorphic functions and operator model theory. {\em J. Funct. Anal.} 262 (2012), no. 7, 3240-–3308.

\bibitem{Po13} G. Popescu.
Free biholomorphic functions and operator model theory, II. {\em J. Funct. Anal.} 265 (2013), no. 5, 786–-836.

\bibitem{Rota}
S. A. Joni and G.-C. Rota.
Coalgebras and bialgebras in combinatorics.
{\em Stud. Appl. Math.} 61 (1979), no. 2,
93–-139.

\bibitem{Row80}
L.~H. Rowen.
\newblock {\em Polynomial identities in ring theory}, volume~\textbf{84} of {\em Pure
  and Applied Mathematics}.
\newblock Academic Press Inc. [Harcourt Brace Jovanovich Publishers], New York,
  1980.

\bibitem{Ru}
Zh.-J.~Ruan.
\newblock On real operator spaces.
 International Workshop on Operator Algebra and Operator Theory (Linfen, 2001).
 {\em Acta Math. Sin. (Engl. Ser.)}, \textbf{19}(3):485--496, 2003.



\bibitem{Sch}
S. H. Schanuel.
Continuous extrapolation to triangular matrices characterizes smooth functions.
{\em J. Pure Appl. Algebra} 24 (1982), no. 1, 59–-71.

\bibitem{SchZa}
S. H. Schanuel and W. R. Zame.
Naturality of the functional calculus.
{\em Bull. Lond. Math. Soc.} 14 (1982), 218--220.

\bibitem{Schutz61}
M.~P. Sch{\"u}tzenberger.
\newblock On the definition of a family of automata.
\newblock {\em Information and Control}, 4:245--270, 1961.

\bibitem{Schutz62b}
M.~P. Sch{\"u}tzenberger.
\newblock Certain elementary families of automata.
\newblock In {\em Proc. Sympos. Math. Theory of Automata (New York, 1962)},
  pages 139--153. Polytechnic Press of Polytechnic Inst. of Brooklyn, Brooklyn,
  New York, 1963.

\bibitem{Sh}
 B. V.~Shabat. \emph{Introduction to Complex Analysis. {P}art {II}}.
 Translations of Mathematical Monographs 110, Functions of several
 variables, Translated from the third (1985) Russian edition by J.
 S. Joel, AMS, Providence, RI, 1992.

\bibitem{SIG97}
R.~E. Skelton, T.~Iwasaki, and K.~M. Grigoriadis.
\newblock {\em A Unified Algebraic Approach to Linear Control Design}.
\newblock Taylor \& Francis, 1997.

\bibitem{T1}
J.~L.~Taylor.
\newblock
A general framework for a multi-operator functional calculus.
\newblock {\em Advances in Math.},
\textbf{9}:183--252, 1972.

\bibitem{T2}
J.~L.~Taylor.
\newblock Functions of several noncommuting variables.
\newblock {\em Bull. Amer. Math. Soc.},
\textbf{79}:1--34, 1973.

\bibitem{Vinograd}  I. M. Vinogradov, Elements of Number Theory, Mineola, NY: Dover Publications, 2003,
reprint of the 1954 edition.

\bibitem{Voi85}
D.-V. Voiculescu. Symmetries of some reduced free product
$C^*$-algebras. {\em Operator algebras and their connections with
topology and ergodic theory (Bus¸teni, 1983)}, Lecture Notes in
Math. 1132, Springer, Berlin, 1985, pp. 556–-588.

\bibitem{Voi86}
D.-V. Voiculescu. { Addition of certain non-commutative random
variables}. {\em J. Funct. Anal.} 66:323--346, 1986.

\bibitem{Voi95}
D.-V. Voiculescu. Operations on certain non-commutative
operator-valued random variables. {\em Ast\'erisque} 232:243--275,
1995.

\bibitem{Voi00}
D.-V. Voiculescu.
The coalgebra of the free di¤erence quotient and free probability.
{\em Internat. Math. Res. Not.} 2000 (2000), no. 2, 79-–106.

\bibitem{Voi04} D.-V. Voiculescu. { Free Analysis Questions I: Duality Transform for the Coalgebra of $\partial_{X:B}$}
{\em International Math. Res. Notices} 16:793--822, 2004.

\bibitem{Voi09} D.-V. Voiculescu.
Free analysis questions. II: The Grassmannian completion and the
series expansion at the origin.  {\em J. Reine Angew. Math.} 645
(2010), 155--236.

\bibitem{VoiDyNi}
D.-V. Voiculescu, K. J. Dykema, and A. Nica. Free random
variables. A noncommutative probability approach to free products
with applications to random matrices, operator algebras and
harmonic analysis on free groups. CRM Monograph Series, 1.
American Mathematical Society, Providence, RI, 1992. vi+70 pp.

\bibitem{Wed}
J. H. M. Wedderburn. On continued fractions in non-commutative
quantities.  {\em Ann. of Math.}   15:101--105, 1913.


\bibitem{Z}
M.~A.~Zorn. Derivatives and Fr\'{e}chet differentials. {\em Bull.
Amer. Math. Soc.}, \textbf{52}:133-137, 1946.

\end{thebibliography}
\end{document}